%% file: 1factrevision.tex
\documentclass{memo-l}

\usepackage{amsmath,amssymb}
\usepackage{psfrag, chngcntr}
\usepackage{enumitem}
 \usepackage{hyperref}
\usepackage{graphicx}
\usepackage{tikz,epsfig,color,graphicx,subfig}
\usepackage{epstopdf}
\usetikzlibrary{positioning,chains,fit,shapes,calc, fit}
\usetikzlibrary{decorations.markings}

\newtheorem{firstthm}{Proposition}[chapter]
\newtheorem{thm}[firstthm]{Theorem}
\newtheorem{prop}[firstthm]{Proposition}
\newtheorem{fact}[firstthm]{Fact}
\newtheorem{lemma}[firstthm]{Lemma}
\newtheorem{cor}[firstthm]{Corollary}

\newtheorem{conjecture}[firstthm]{Conjecture}

\def\COMMENT#1{}
\def\TASK#1{}

\def\noproof{{\unskip\nobreak\hfill\penalty50\hskip2em\hbox{}\nobreak\hfill%
        $\square$\parfillskip=0pt\finalhyphendemerits=0\par}\goodbreak}
\def\endproof{\noproof\bigskip}
\newdimen\margin   
\def\textno#1&#2\par{%
    \margin=\hsize
    \advance\margin by -4\parindent
           \setbox1=\hbox{\sl#1}%
    \ifdim\wd1 < \margin
       $$\box1\eqno#2$$%
    \else
       \bigbreak
       \hbox to \hsize{\indent$\vcenter{\advance\hsize by -3\parindent
       \sl\noindent#1}\hfil#2$}%
       \bigbreak
    \fi}
\def\proof{\removelastskip\penalty55\medskip\noindent{\bf Proof. }}

\def\cC{\mathcal{C}}

\def\cM{\mathcal{M}}
\def\cP{\mathcal{P}}
\def\eps{\varepsilon}
\def\ex{\mathbb{E}}
\def\prob{\mathbb{P}}
\def\pr{\mathbb{P}}

\def\i{(i_1,i_2,i_3,i_4)}
\def\I{i_1,i_2,i_3,i_4}
\def\cJ{\mathcal{J}}
\def\epszero{\eps_0}

\numberwithin{section}{chapter}
\numberwithin{equation}{section}
\numberwithin{firstthm}{section}
\numberwithin{figure}{section}

\makeindex

\begin{document}

\frontmatter

\title[Proof of the $1$-factorization \& Hamilton Decomposition Conjectures]{Proof of the $1$-factorization and Hamilton Decomposition Conjectures}


\author{B\'ela Csaba}
\address{Bolyai Institute,
University of Szeged,
H-6720 Szeged, Aradi \break v\'ertan\'uk tere 1.
Hungary}
\email{bcsaba@math.u-szeged.hu}

\author{Daniela K\"uhn}
\address{School of Mathematics
University of Birmingham
Edgbaston
Birmingham 
B15 2TT
UK}
\email{d.kuhn@bham.ac.uk}

\author{Allan Lo}
\address{School of Mathematics
University of Birmingham
Edgbaston
Birmingham 
B15 2TT
UK}
\email{s.a.lo@bham.ac.uk}

\author{Deryk Osthus}
\address{School of Mathematics
University of Birmingham
Edgbaston
Birmingham 
B15 2TT
UK}
\email{d.osthus@bham.ac.uk}

\author{Andrew Treglown}
\address{School of Mathematics
University of Birmingham
Edgbaston
Birmingham 
B15 2TT
UK}
\email{a.treglown@bham.ac.uk}
\thanks{The research leading to these results was partially supported by the  European Research Council under the European Union's Seventh Framework Programme (FP/2007--2013) / ERC Grant Agreement no. 258345 (B.~Csaba, D.~K\"uhn and A.~Lo), 306349 (D.~Osthus) and 259385 (A.~Treglown). The research was also partially supported by the EPSRC, grant no. EP/J008087/1 (D.~K\"uhn and D.~Osthus).}

\date{}




\begin{abstract}
In this paper we prove the following results (via a unified approach) for all sufficiently large $n$:
\begin{itemize}
\item[(i)] [\emph{$1$-factorization conjecture}]
Suppose  that $n$ is even and $D\geq 2\lceil n/4\rceil -1$. 
Then every $D$-regular graph $G$ on $n$ vertices has a decomposition into perfect matchings.
Equivalently, $\chi'(G)=D$.

\item[(ii)] [\emph{Hamilton decomposition conjecture}]
Suppose that $D \ge   \lfloor n/2 \rfloor $.
Then every $D$-regular graph $G$ on $n$ vertices has a decomposition
into Hamilton cycles and at most one perfect matching.

\item[(iii)] [\emph{Optimal packings of Hamilton cycles}] Suppose that $G$ is a graph on $n$ vertices with
minimum degree $\delta\ge n/2$.
Then $G$ contains at least ${\rm reg}_{\rm even}(n,\delta)/2 \ge (n-2)/8$ edge-disjoint Hamilton cycles.
Here $\textnormal{reg}_{\textnormal{even}}(n,\delta)$ denotes the degree of the 
largest even-regular spanning subgraph one can guarantee in a graph on $n$ vertices
with minimum degree~$\delta$.
\end{itemize}
(i) was first  explicitly stated by Chetwynd and Hilton.
(ii) and the special case $\delta= \lceil n/2 \rceil$ of (iii) answer questions of Nash-Williams from 1970.
All of the above bounds are best possible. 
\end{abstract}

\maketitle

\tableofcontents

\mainmatter

\chapter{Introduction}

\section{Introduction}\label{intro}

In this paper we provide a unified approach towards proving three long-standing conjectures
for all sufficiently large graphs. Firstly, the $1$-factorization conjecture, which can be formulated as an edge-colouring problem;
secondly, the Hamilton decomposition conjecture, which provides a far-reaching generalization of Walecki's result~\cite{lucas} that
every complete graph of odd order has a Hamilton decomposition and  
thirdly, a best possible result on packing edge-disjoint 
Hamilton cycles in Dirac graphs. The latter two problems were raised by Nash-Williams~~\cite{initconj,Diracext,decompconj} in 1970. 


\subsection{The $1$-factorization Conjecture}

Vizing's theorem states that for any graph~$G$ of maximum degree $\Delta$, its edge-chromatic number
$\chi'(G)$ is either $\Delta$ or $\Delta+1$. 
However, the problem of determining the precise value of $\chi '(G)$
for an arbitrary graph $G$ is NP-complete~\cite{NP}. 
Thus, it is of interest to determine classes of graphs $G$ that attain the (trivial) lower bound~$\Delta$
-- much of the recent book~\cite{stiebitz} is devoted to the subject.
For regular graphs $G$, $\chi'(G)=\Delta(G)$ is equivalent to the existence of a $1$-factorization:
a \emph{$1$-factorization} of a graph~$G$ consists of a set of edge-disjoint perfect matchings covering all edges of~$G$.
The long-standing $1$-factorization conjecture states that every regular graph of sufficiently high degree has a 
$1$-factorization. It was first stated explicitly by Chetwynd and Hilton~\cite{1factorization,CH} (who also proved partial results).
However, they state that
according to Dirac, it was already discussed in the 1950s. Here we prove the conjecture for large graphs.

\begin{thm}\label{1factthm}
There exists an $n_0 \in \mathbb N$ such that the following holds.
Let $ n,D \in \mathbb N$ be such that $n\geq n_0$ is even and $D\geq 2\lceil n/4\rceil -1$. 
Then every $D$-regular graph $G$ on $n$ vertices has a $1$-factorization.%
    \COMMENT{So this means that $D\ge n/2$ if $n = 2 \pmod 4$ and $D\ge n/2-1$ if $n = 0 \pmod 4$.}
    Equivalently, $\chi'(G)=D$.
\end{thm}
The bound on the minimum degree in Theorem~\ref{1factthm} is best possible. To see this, suppose first that $n = 2 \pmod{4}$.
Consider the graph which is the disjoint union of two cliques of order $n/2$ (which is odd). If $n = 0 \pmod 4$,
consider the graph obtained from the disjoint union of cliques of orders $n/2-1$ and $n/2+1$ (both odd)
by deleting a Hamilton cycle in the larger clique.

Note that Theorem~\ref{1factthm} implies that for every regular graph~$G$ on an even number of vertices, 
either~$G$ or its complement has a $1$-factorization.
Also, Theorem~\ref{1factthm} has an interpretation in terms of scheduling round-robin tournaments
(where $n$ players play all of each other in $n-1$ rounds):
one can schedule the first half of the rounds arbitrarily before one needs to plan the remainder of the tournament.

The best previous result towards Theorem~\ref{1factthm} is due to Perkovic and Reed~\cite{reed}, who proved an approximate version,
i.e.~they assumed that $D \ge n/2+\eps n$.
This was generalized by Vaughan~\cite{vaughan} to multigraphs of bounded multiplicity.
Indeed, he proved an approximate form of the following multigraph version of the $1$-factorization conjecture which was raised by
Plantholt and Tipnis~\cite{PT}:
Let $G$ be a regular multigraph of even order $n$ with multiplicity at most $r$.
If the degree of $G$ is at least $rn/2$ then $G$ is $1$-factorizable.

In 1986, Chetwynd and Hilton~\cite{overfull} made the following `overfull subgraph' conjecture.%
\COMMENT{Deryk changed this}
Roughly speaking, this says that a dense graph satisfies
$\chi'(G)=\Delta(G)$ unless there is a trivial obstruction in the form of a dense subgraph $H$ on an odd number of vertices.
Formally, we say that a subgraph $H$ of $G$ is \emph{overfull} if  
$e(H) > \Delta (G) \lfloor |H|/2 \rfloor$ (note this requires $|H|$ to be odd). 
\begin{conjecture}
A graph $G$ on $n$ vertices with $\Delta (G) \geq n/3$ satisfies $\chi'(G)=\Delta(G)$ if and only if $G$ contains no overfull subgraph. 
\end{conjecture}
It is easy to see that this generalizes the $1$-factorization conjecture (see e.g.~\cite{car} for the details).
The overfull subgraph conjecture is still wide open -- partial results are discussed in~\cite{stiebitz}, which also discusses further results
and questions related to the $1$-factorization conjecture.

\subsection{The Hamilton Decomposition Conjecture}

Rather than asking for a $1$-factorization, Nash-Williams~\cite{initconj,decompconj} raised the
more difficult problem of finding a Hamilton decomposition 
in an even-regular graph.
Here, a \emph{Hamilton decomposition} of a graph~$G$ consists of a set of edge-disjoint Hamilton cycles covering all edges of~$G$.
A natural extension of this to regular graphs $G$ of odd degree is to ask for a decomposition into 
Hamilton cycles and one perfect matching (i.e.~one perfect matching $M$ in $G$ together with a 
Hamilton decomposition of $G-M$). The following result solves the problem of Nash-Williams for all large graphs.
\begin{thm} \label{HCDthm} 
There exists an $n_0 \in \mathbb N$ such that the following holds.
Let $ n,D \in \mathbb N$ be such that $n \geq n_0$ and
$D \ge   \lfloor n/2 \rfloor $.
Then every $D$-regular graph $G$ on $n$ vertices has a decomposition into Hamilton cycles and 
at most one perfect matching.
\end{thm}
Again, the bound on the degree in Theorem~\ref{HCDthm} is best possible. 
Indeed, Proposition~\ref{bestposs} shows that a smaller degree bound would not even ensure connectivity.
Previous results include the following:
Nash-Williams~\cite{NWreg} showed that the degree bound in Theorem~\ref{HCDthm} ensures a single Hamilton cycle.
Jackson~\cite{Jackson79} showed that one can ensure close to $D/2-n/6$ edge-disjoint Hamilton cycles.
Christofides, K\"uhn and Osthus~\cite{CKO} obtained an approximate decomposition under the assumption that $D \ge n/2 +\eps n$.
Under the same assumption, K\"uhn and Osthus~\cite{KellyII} obtained an exact decomposition
(as a consequence of the main result in~\cite{Kelly} on Hamilton decompositions of robustly expanding graphs).

Note that Theorem~\ref{HCDthm} does not quite imply Theorem~\ref{1factthm},
as the degree threshold in the former result is slightly higher.

A natural question is whether one can extend Theorem~\ref{HCDthm} to sparser (quasi)-random
graphs. Indeed, for random regular graphs of bounded degree this was proved by Kim and Wormald~\cite{KW01}
and for (quasi-)random regular graphs of linear degree this was proved in~\cite{KellyII} as a consequence of
the main result in~\cite{Kelly}. However, the intermediate range remains open.

\subsection{Packing Hamilton Cycles in Graphs of Large Minimum Degree}

Although Dirac's theorem is best possible in the sense that the minimum degree condition $\delta \ge n/2$ is best possible, 
the conclusion can be strengthened considerably:
a remarkable result of Nash-Williams~\cite{Diracext} states that every graph $G$ on $n$ vertices with minimum degree $\delta(G) \ge n/2$
contains $\lfloor 5n/224 \rfloor$ edge-disjoint Hamilton cycles.
He raised the question of finding the best possible bound, which we answer in Corollary~\ref{NWmindegcor} below.

We actually answer a more general form of this question: what is the number of edge-disjoint Hamilton cycles one can guarantee in a graph $G$ of minimum degree $\delta$?

A natural upper bound is obtained by considering the largest degree of an even-regular spanning subgraph of $G$.
Let $\textnormal{reg}_{\textnormal{even}}(G)$
be the largest degree of an even-regular spanning subgraph of $G$. Then let
\[
\textnormal{reg}_{\textnormal{even}}(n,\delta):=\min\{\textnormal{reg}_{\textnormal{even}}(G):|G|=n,\ \delta(G)=\delta\}.
\]
Clearly, in general we cannot guarantee more than $\textnormal{reg}_{\textnormal{even}}(n,\delta)/2$
edge-disjoint Hamilton cycles in a graph of order $n$ and minimum
degree $\delta$. The next result shows that this bound is best possible (if $\delta < n/2$, then $\textnormal{reg}_{\textnormal{even}}(n,\delta)=0$).%
\COMMENT{Andy: changed this so it is the same as the sentence added in paper II
by Deryk}
\begin{thm}\label{NWmindeg}
There exists an $n_0 \in \mathbb N$ such that the following holds. Suppose that $G$ is a graph on $n\ge n_0$ vertices with
minimum degree $\delta\ge n/2$. Then $G$ contains at least ${\rm reg}_{\rm even}(n,\delta)/2$ edge-disjoint Hamilton cycles.
\end{thm}

The main result of K\"uhn, Lapinskas and Osthus~\cite{KLOmindeg} proves Theorem~\ref{NWmindeg} unless $G$ is close to one of the extremal graphs for
Dirac's theorem.
This will allow us to restrict our attention to the latter situation (i.e.~when $G$ is close to
the complete balanced bipartite graph or close to the union of two disjoint copies of a clique).

An approximate version of Theorem~\ref{NWmindeg} for $\delta \ge n/2+\eps n$
was obtained earlier by Christofides, K\"uhn and Osthus~\cite{CKO}.
Hartke and Seacrest~\cite{HartkeHCs} gave a simpler argument with improved error bounds.

Precise estimates for ${\rm reg}_{\rm even}(n,\delta)$ 
(which yield either one or two possible values for any $n$, $\delta$)
are proved in~\cite{CKO,Hartkefactors} using Tutte's theorem:
Suppose that $n,\delta\in\mathbb{N}$ and $n/2\le\delta<n$.
Then the bounds in~\cite{Hartkefactors} imply that
\begin{equation} \label{regbound}
\frac{\delta+\sqrt{n(2\delta-n)+8}}{2}-\eps\le\textnormal{reg}_{\textnormal{even}}(n,\delta)
\le\frac{\delta+\sqrt{n(2\delta-n)}}{2}+1, 
\end{equation}
where $0<\eps  \le 2$ is chosen to make the left hand side of~\eqref{regbound} an even
integer. Note that~\eqref{regbound} determines $\textnormal{reg}_{\textnormal{even}}(n,n/2)$
exactly (the upper bound in this case was already proved by Katerinis~\cite{Katerinis}).
Moreover, \eqref{regbound} implies that if $\delta \geq n/2$ then ${\rm reg}_{\rm even}(n,\delta) \ge (n-2)/4$.%
\COMMENT{let $r$ be the lower bound in~\ref{regbound}. 
for $n$ odd, there is room to spare, as then $\delta>n/2$ and so $r \ge n/4+\Omega(\sqrt{n})$.
For $n=8k$, we have $\delta=4k$ and so their formula gives $r = 2k+\sqrt{2}-\eps=2k=n/4$.
For $n=8k+2$, we have $\delta=4k+1$ and so $r = 2k+1/2+\sqrt{2}-\eps=2k=(n-2)/4$.
For  $n=8k+4$, we have $\delta=4k+2$ and so $r = 2k+1+\sqrt{2}-\eps=2k+2=(n+4)/4$.
For $n=8k+6$, we have $\delta=4k+3$ and so $r = 2k+3/2+\sqrt{2}-\eps=2k+2=(n+2)/4$.}
So we obtain the following immediate corollary of Theorem~\ref{NWmindeg}, 
which answers a question of Nash-Williams~\cite{initconj,Diracext,decompconj}.
\begin{cor}\label{NWmindegcor}
There exists an $n_0 \in \mathbb N$ such that the following holds. Suppose that $G$ is a graph on $n\ge n_0$ vertices with
minimum degree $\delta\ge n/2$. Then $G$ contains at least $(n-2)/8$ edge-disjoint Hamilton cycles.
\end{cor}
The following construction (which is based on a construction of Babai, see~\cite{initconj}) shows that the bound in 
Corollary~\ref{NWmindegcor} is best possible for $n=8k+2$, where $k \in \mathbb{N}$. 
Consider the graph $G$ consisting of one empty
vertex class $A$ of size $4k$, one vertex class $B$ of size $4k+2$ containing
a perfect matching and no other edges, and all possible edges between $A$ and $B$.
Thus $G$ has order $n=8k+2$ and minimum degree $4k+1=n/2$. 
Any Hamilton cycle in $G$ must contain at least two edges of
the perfect matching in~$B$, so $G$ contains at most $\lfloor |B|/4 \rfloor =k=(n-2)/8$ edge-disjoint
Hamilton cycles. The lower bound on ${\rm reg}_{\rm even}(n,\delta)$ 
in~\eqref{regbound} follows from a generalization of this construction.

The following conjecture from~\cite{KLOmindeg} would be a common generalization of both Theorems~\ref{HCDthm} and~\ref{NWmindeg}
(apart from the fact that the degree threshold in Theorem~\ref{HCDthm}  is slightly lower).
It would provide a result which is best possible for every graph~$G$
(rather than the class of graphs with minimum degree at least $\delta$).

\begin{conjecture} \label{con:betterconj}
Suppose that $G$ is a graph on $n$ vertices
with minimum degree $\delta(G) \ge n/2$. Then $G$ contains
$\textnormal{reg}_{\textnormal{even}}(G)/2$ edge-disjoint
Hamilton cycles.
\end{conjecture}
For $\delta\ge(2-\sqrt{2}+\eps)n$, this conjecture was proved in~\cite{KellyII}, based on the main result of~\cite{Kelly}.
Recently, Ferber, Krivelevich and Sudakov~\cite{FKS} were able to obtain an approximate version of Conjecture~\ref{con:betterconj},
i.e.~a set of $(1-\eps)\textnormal{reg}_{\textnormal{even}}(G)/2$ edge-disjoint Hamilton cycles under the assumption that
$\delta(G) \ge (1+\eps)n/2$.
It also makes sense to consider a directed version of Conjecture~\ref{con:betterconj}.
Some related questions for digraphs are discussed in~\cite{KellyII}.

It is natural to ask for which other graphs one can obtain similar results. 
One such instance is the binomial random graph $G_{n,p}$:
for any $p$, asymptotically almost surely it contains $\lfloor \delta(G_{n,p})/2 \rfloor$ edge-disjoint 
Hamilton cycles, which is clearly optimal.
This follows from the main result of Krivelevich and Samotij~\cite{KrS} combined with that of Knox, K\"uhn and Osthus~\cite{KnoxKO} 
(which builds on a number of previous results). 
The problem of packing edge-disjoint Hamilton cycles in hypergraphs has been considered in~\cite{ferber}.
Further questions in the area are discussed in the recent survey~\cite{ICM}.

\subsection{Overall Structure of the Argument}

For all three of our main results, we split the argument according to the structure of the graph $G$ under consideration:
\begin{enumerate}
\item[{\rm (i)}] $G$ is  close to the complete balanced bipartite graph $K_{n/2,n/2}$;
\item[{\rm (ii)}] $G$ is close to the union of two disjoint copies of a clique $K_{n/2}$;
\item[{\rm (iii)}] $G$ is a `robust expander'.
\end{enumerate}
Roughly speaking, $G$ is a robust expander if for every set $S$ of vertices, its neighbourhood is at least a little larger than $|S|$,
even if we delete a small proportion of the vertices and edges%
   \COMMENT{Daniela: included vertices here}
of $G$.
The main result of~\cite{Kelly} states that every dense regular robust expander
has a Hamilton decomposition (see Theorem~\ref{undir_decomp}).
This immediately implies Theorems~\ref{1factthm} and~\ref{HCDthm} in Case~(iii).
For Theorem~\ref{NWmindeg}, Case (iii) is proved in~\cite{KLOmindeg} using a more involved argument, but also based on
the main result of~\cite{Kelly} (see Theorem~\ref{NWmindegrob}).

Case~(i) is proved in~Chapter~\ref{paper2} whilst Chapter~\ref{paper1} tackles Case~(ii).
We defer the proof of some of the key lemmas needed for Case~(ii) until Chapter~\ref{paper4}.
(These lemmas provide a suitable decomposition of the set of `exceptional edges' -- these include the edges between the two almost 
complete graphs induced by~$G$.)
Case~(ii) is by far the hardest case for Theorems~\ref{1factthm} and~\ref{HCDthm}, as the extremal examples are all close to the union of two 
cliques. On the other hand, the proof of Theorem~\ref{NWmindeg} is comparatively simple in this case, as for this result, 
the extremal construction is close to the complete balanced bipartite graph.

The arguments in Cases~(i) and (ii) 
make use of an `approximate' decomposition result.
We defer the proof of this result until Chapter~\ref{paper3}.
The arguments for both (i) and (ii) use the main lemma from~\cite{Kelly} (the `robust decomposition lemma') when transforming this approximate
decomposition into an exact one.

In Section~\ref{split}, we derive Theorems~\ref{1factthm},~\ref{HCDthm} and~\ref{NWmindeg} from the structural results covering 
Cases (i)--(iii). 

The main proof in~\cite{Kelly} (but not the proof of the robust decomposition lemma) makes use of Szemer\'edi's regularity lemma.
So due to Case (iii) the bounds on $n_0$ in our results are very large (of tower type).
However, the case of Theorem~\ref{1factthm} when both $\delta \ge n/2$ and (iii) hold was proved by Perkovic and Reed~\cite{reed}
using `elementary' methods, i.e.~with a much better bound on $n_0$. 
Since the arguments for Cases (i) and (ii) do not rely on the regularity lemma, this means that if we assume that $\delta \ge n/2$,
we get much better bounds on $n_0$ in our $1$-factorization result (Theorem~\ref{1factthm}).

\section{Notation}\label{notation}

Unless stated otherwise, all the graphs and digraphs considered in this paper are simple and do not contain loops. So in a digraph $G$, we allow up to
two edges between any two vertices, at most one edge in each direction.
Given a graph or digraph $G$, we write $V(G)$ for its vertex set, $E(G)$ for its edge set, $e(G):=|E(G)|$ for
the number of edges in $G$ and $|G|:=|V(G)|$ for the number of vertices in $G$. We denote the complement
of $G$ by $\overline{G}$.

Suppose that $G$ is an undirected graph. We write $\delta(G)$ for the minimum degree of $G$, $\Delta(G)$ for its maximum degree
and $\chi'(G)$ for the edge-chromatic number of~$G$.
Given a vertex $v$ of $G$, we write $N_G(v)$ for the set of all neighbours of~$v$ in~$G$.
Given  a set $A\subseteq V(G)$,
we write $d_G(v,A)$ for the number of  neighbours of $v$ in $G$ which lie in~$A$. Given $A,B\subseteq V(G)$,
we write $E_G(A)$ for the set of  edges of $G$ which have both endvertices in $A$ and $E_G(A,B)$ for the set of
 edges of $G$ which have one endvertex in $A$ and its other endvertex in $B$. We also call the edges in $E_G(A,B)$
\emph{$AB$-edges} of $G$. We let $e_G(A):=|E_G(A)|$ and $e_G(A,B):=|E_G(A,B)|$.
We denote by $G[A]$ the subgraph of $G$ with vertex set $A$ and edge set $E_G(A)$.
If $A\cap B=\emptyset$, we denote by $G[A,B]$ the bipartite subgraph of $G$
with vertex classes $A$ and $B$ and edge set $E_G(A,B)$. 
If $A=B$ we
 define $G[A,B]:=G[A]$.
 We often omit the index $G$ if the graph $G$ is clear from the context.
An \emph{$AB$-path} in $G$ is a path with one endpoint in $A$ and the other in $B$.
A spanning subgraph $H$ of $G$ is an \emph{$r$-factor} of $G$ if the degree of every vertex of $H$ is~$r$.

Given a vertex set $V$ and two multigraphs%
   \COMMENT{Have to define this for multigraphs instead of graph since if we take $G+H+F$ then $G+H$ might already be a multigraph.}
$G$ and $H$ with $V(G),V(H)\subseteq V$, we write $G+H$ for the multigraph whose vertex
set is $V(G)\cup V(H)$ and in which the multiplicity of $xy$ in $G+H$ is the sum of the multiplicities of $xy$ in $G$ and in~$H$
(for all $x,y\in V(G)\cup V(H)$). Similarly, if $\mathcal{H}:=\{H_1,\dots,H_\ell\}$ is a set of graphs, we define
$G+\mathcal{H}:=G+H_1+\dots+H_\ell$.
If $G$ and $H$ are simple graphs, we write $G\cup H$ for the (simple) graph whose vertex set is
$V(G)\cup V(H)$ and whose edge set is $E(G)\cup E(H)$.%
   \COMMENT{I think this is a good convention. It has the advantage that if we talk about
$G+J^*$ then it is clear that we treat the edges in $J^*$ as being distinct from those in $G$. However, this means that we cannot
talk about $G+EPS^*$ (we would need to write $G+{\rm Fict}(EPS^*)$ instead). Have to make sure that this fits with the almost decomposition section/paper.}
We write $G-H$ for the subgraph of $G$ which is obtained from $G$
by deleting all the edges in $E(G)\cap E(H)$.%
    \COMMENT{So we don't require that $H\subseteq G$ when using this notation.} 
Given $A\subseteq V(G)$, we write $G-A$ for the graph obtained from $G$ by deleting all vertices in~$A$.

We say that a graph or digraph $G$ has a \emph{decomposition} into $H_1,\dots,H_r$ if $G=H_1+\dots +H_r$ and the $H_i$ are pairwise edge-disjoint.

A \emph{path system} is a graph $Q$ which is the union of vertex-disjoint paths (some of them might be trivial).
We say that $P$ is a \emph{path in Q} if $P$ is a component of $Q$ and, abusing the notation, sometimes write $P\in Q$ for this.
A \emph{path sequence} is a digraph which is the union of vertex-disjoint directed paths (some of them might be trivial).
We often view a matching $M$ as a graph (in which every vertex has degree precisely one).%
    \COMMENT{This is different to eg the bipartite paper where a matching us a set of edges. CHECK whether we always use this
def in this paper, ie whether we write $e(M)$ for the number of edges and not $|M|$.}

If $G$ is a digraph, we write $xy$ for an edge directed from $x$ to $y$. If $xy\in E(G)$, we say that $y$ is an
\emph{outneighbour} of~$x$ and $x$ is an \emph{inneighbour} of~$y$.
 A digraph $G$ is an \emph{oriented graph} if there are no $x,y\in V(G)$ such that $xy, yx\in E(G)$.
Unless stated otherwise, when we
refer to paths and cycles in digraphs, we mean directed paths and cycles, i.e.~the edges on these paths/cycles are oriented consistently.
If $x$ is a vertex of a digraph $G$, then $N^+_G(x)$ denotes the \emph{outneighbourhood} of $x$, i.e.~the
set of all those vertices $y$ for which $xy\in E(G)$. Similarly, $N^-_G(x)$ denotes the \emph{inneighbourhood} of $x$, i.e.~the
set of all those vertices $y$ for which $yx\in E(G)$. The \emph{outdegree} of $x$ is $d^+_G(x):=|N^+_G(x)|$ and the
\emph{indegree} of $x$ is $d^-_G(x):=|N^-_G(x)|$. 
We write $d^+_G(x,A)$ for the number of%
\COMMENT{Andy: deleted `all'} 
outneighbours of $x$ lying inside $A$ and define $d^-_G(x,A)$ similarly.
We denote the minimum outdegree of $G$ by $\delta^+(G)$ and the minimum indegree by $\delta^-(G)$.
We write $\delta(G)$ and $\Delta(G)$ for the minimum and maximum degrees of the underlying simple%
   \COMMENT{Daniela: added simple}
undirected graph of $G$ respectively.%
\COMMENT{AL: added definitions of $\delta(G)$ and $\Delta(G)$. use this convention in e.g. the proof of Lemma~\ref{lma:EF-bipartite}}

Given a digraph $G$ and $A,B\subseteq V(G)$, an \emph{$AB$-edge} is an edge with initial vertex in $A$ and final vertex in $B$,
and $e_G(A,B)$ denotes the number of these edges in~$G$. If $A\cap B=\emptyset$, we denote by $G[A,B]$ the bipartite subdigraph of $G$
whose vertex classes are $A$ and $B$ and whose edges are all $AB$-edges of $G$.
By a bipartite digraph $G=G[A,B]$ we mean a digraph which only contains $AB$-edges.
A spanning subdigraph $H$ of $G$ is an \emph{$r$-factor} of $G$ if the outdegree and the indegree of every vertex of $H$ is~$r$.

If $P$ is a path and $x,y\in V(P)$, we write $xPy$ for the subpath of $P$ whose endvertices are $x$ and $y$.
We define $xPy$ similarly if $P$ is a directed path and $x$ precedes $y$ on~$P$.%
   \COMMENT{Daniela added these 2 sentences}

Let $V_1,\dots,V_k$ be pairwise disjoint sets of vertices and let $C=V_1\dots V_k$ be a directed cycle on these sets.
We say that an edge $xy$ of a digraph $R$ \emph{winds around $C$} if there is some $i$ such that $x\in V_i$ and $y\in V_{i+1}$.
In particular, we say that $R$ \emph{winds around $C$} if all edges of $R$ wind around~$C$.

In order to simplify the presentation, we omit floors and ceilings and treat large numbers as integers whenever this does
not affect the argument. The constants in the hierarchies used to state our results have to be chosen from right to left.
More precisely, if we claim that a result holds whenever $0<1/n\ll a\ll b\ll c\le 1$ (where $n$ is the order of the graph or digraph),
then this means that
there are non-decreasing functions $f:(0,1]\to (0,1]$, $g:(0,1]\to (0,1]$ and $h:(0,1]\to (0,1]$ such that the result holds
for all $0<a,b,c\le 1$ and all $n\in \mathbb{N}$ with $b\le f(c)$, $a\le g(b)$ and $1/n\le h(a)$. 
We will not calculate these functions explicitly. Hierarchies with more constants are defined in a similar way.
We will write $a = b \pm c$ as shorthand for $ b - c \le a \le b+c$.


\section[Derivation of Theorems~$\text{\ref{1factthm}}$,~$\text{\ref{HCDthm}}$,~$\text{\ref{NWmindeg}}$ from  Main Structural Results]{Derivation of Theorems~$\text{\ref{1factthm}}$,~$\text{\ref{HCDthm}}$,~$\text{\ref{NWmindeg}}$ from the Main Structural Results}
\label{split}

In this section, we combine the main auxiliary results of this paper (together with  results from~~\cite{KellyII} and~\cite{KLOmindeg})
to derive Theorems~\ref{1factthm},~\ref{HCDthm} and~\ref{NWmindeg}.
Before this, we first show that the bound on the minimum degree in Theorem~\ref{HCDthm} is best possible.

\begin{prop}\label{bestposs}
For every $n \ge 6$,%
\COMMENT{Need $n \geq 6$ here as we need $\lceil n/2\rceil+1 \geq 4$}
 let $D^*:= \lfloor n/2 \rfloor -1$. Unless both $D^*$ and $n$ are odd, 
there is a disconnected $D^*$-regular graph $G$ on $n$ vertices. 
If both $D^*$ and $n$ are odd, there is a disconnected $(D^*-1)$-regular graph $G$ on $n$ vertices.
\end{prop}
Note that if both $D^*$ and $n$ are odd, no $D^*$-regular graph exists.
\proof
If $n$ is even,
take $G$ to be the disjoint union of two cliques of order $n/2$.
Suppose that $n$ is odd and $D^*$ is even. This implies $n=3 \pmod 4$.
Let $G$ be the graph obtained from the disjoint union of cliques of orders
$\lfloor n/2\rfloor$ and $\lceil n/2\rceil$ by deleting a perfect matching in the bigger clique.%
\COMMENT{This can be done since $\lceil n/2\rceil$ is even.}
Finally, suppose that $n$ and $D^*$ are both odd.
This implies that $n = 1 \pmod 4$. 
In this case, take $G$ to be the graph obtained from the disjoint union of cliques of orders $\lfloor n/2\rfloor-1$
and $\lceil n/2\rceil+1$ by deleting a $3$-factor in the bigger clique.%
   \COMMENT{This can be done since $\lceil n/2\rceil+1$ is even. Note $G$ is regular of degree $\lfloor n/2\rfloor-2$.}
\endproof

\subsection{Deriving Theorems~\ref{1factthm} and~\ref{HCDthm}}

As indicated in Section~\ref{intro}, in the proofs of our main results
we will distinguish the cases when our given graph $G$ is close to the union of two disjoint copies of $K_{n/2}$,
close to a complete bipartite graph $K_{n/2,n/2}$ or a robust expander. We will start by defining these concepts.

We say that a graph $G$ on $n$ vertices%
   \COMMENT{Daniela}
is \emph{$\eps$-close to the union of two disjoint copies of $K_{n/2}$}
if there exists $A\subseteq V(G)$ with $|A|=\lfloor n/2\rfloor$ and such that $e(A,V(G)\setminus A)\le\eps n^{2}$.
We say that $G$ is \emph{$\eps$-close to $K_{n/2,n/2}$} if there exists $A\subseteq V(G)$ with $|A|=\lfloor n/2\rfloor$
and such that $e(A)\le \eps n^{2}$. We say that $G$ is \emph{$\eps$-bipartite} if there exists $A\subseteq V(G)$ with $|A|=\lfloor n/2\rfloor$
such that $e(A), e(V(G)\setminus A) \le \eps n^{2}$. So every $\eps$-bipartite graph is $\eps$-close to $K_{n/2,n/2}$.
Conversely, if $1/n\ll \eps$ and $G$ is a regular graph on $n$ vertices which
$\eps$-close to $K_{n/2,n/2}$, then $G$ is $2\eps$-bipartite.

Given $0<\nu \leq \tau<1$, we say that a graph $G$ on $n$ vertices is a \emph{robust $(\nu, \tau)$-expander},
if for all $S\subseteq V(G)$ with $\tau n\le |S|\le (1-\tau)n$ the number of vertices that have at least $\nu n$
neighbours in $S$ is at least $|S|+\nu n$. 

The following observation from~\cite{KLOmindeg} implies that we can split the proofs of Theorems~\ref{1factthm} and~\ref{HCDthm}
into three cases.

\begin{lemma}\label{lem:characteriseexpanders}
Suppose that $0<1/n\ll\kappa\ll\nu\ll\tau,\eps<1$.
Let $G$ be a graph on $n$ vertices of minimum degree $\delta:=\delta(G)\ge(1/2-\kappa)n$.
Then $G$ satisfies one of the following properties:
\begin{enumerate}
\item[{\rm (i)}] $G$ is $\eps$-close to $K_{n/2,n/2}$;
\item[{\rm (ii)}] $G$ is $\eps$-close to the union of two disjoint copies of $K_{n/2}$;
\item[{\rm (iii)}] $G$ is a robust $(\nu,\tau)$-expander.
\end{enumerate}
\end{lemma}

Recall that in Chapter~\ref{paper1} we prove Theorems~\ref{1factthm} and~\ref{HCDthm} in Case (ii) when our given graph~$G$
is $\eps$-close to the union of two disjoint copies of $K_{n/2}$. 
The following result is sufficiently general to imply both Theorems~\ref{1factthm} and~\ref{HCDthm}
in this case. We will prove it in Section~\ref{sec:1factstrong}.
\begin{thm}\label{1factstrong}
For every $\eps_{\rm ex} > 0$ there exists an $n_0\in\mathbb{N}$ such that the following holds for all $n\ge n_0$.
Suppose that $D \ge n - 2 \lfloor n/4 \rfloor -1$ and that $G$ is a $D$-regular graph on $n$ vertices which
is $\eps_{\rm ex}$-close to the union of two disjoint copies of $K_{n/2}$.
Let $F$ be the size of a minimum cut in $G$.
Then $G$ can be decomposed into $\lfloor \min \{D,F\} /2 \rfloor$ Hamilton cycles and $D - 2 \lfloor \min \{D,F\} /2 \rfloor$ perfect matchings.
\end{thm}
Note that Theorem~\ref{1factstrong} provides structural insight into the extremal graphs for Theorem~\ref{HCDthm}%
		\COMMENT{Previously have Theorem~\ref{HCDthm}, but I think it should be Theorem~\ref{1factthm}.
Deryk: we do mean Theorem~\ref{HCDthm}, the range is allowed to go below that for Theorem~\ref{HCDthm} but not below that of~\ref{1factthm}}
-- they are those with a cut of size less than $D$.

Throughout this paper, we will use the following fact.
\begin{equation} \label{minexact}
n - 2\lfloor n/4 \rfloor -1=
\begin{cases}
n/2-1 & \textrm{if $n = 0 \pmod 4$,}\\
(n-1)/2 & \textrm{if $n = 1 \pmod 4$,}\\
n/2 & \textrm{if $n = 2 \pmod 4$,}\\
(n+1)/2 & \textrm{if $n = 3 \pmod 4$.}
\end{cases}
\end{equation}

The next result from~\cite{KellyII} (derived from the main result of~\cite{Kelly})%
\COMMENT{osthus added bracket}
shows that every even-regular robust expander of linear degree has a Hamilton decomposition.
It will be used to prove Theorems~\ref{1factthm} and~\ref{HCDthm} in the case when our given graph $G$ is a robust expander.

\begin{thm}\label{undir_decomp}
For every $ \alpha >0$ there exists $\tau>0$ such that for every $\nu> 0$ there exists $n_0=n_0 (\alpha,\nu,\tau)$ for which the following holds. 
Suppose that
\begin{itemize}
\item[{\rm (i)}] $G$ is an $r$-regular graph on $n \ge n_0$ vertices, where $r\ge \alpha n$ is even;
\item[{\rm (ii)}] $G$ is a robust $(\nu,\tau)$-expander.
\end{itemize}
Then $G$ has a Hamilton decomposition.
\end{thm}

The following result implies Theorems~\ref{1factthm} and \ref{HCDthm} in the case when
our given graph is $\eps$-close to $K_{n/2,n/2}$.
Note that unlike the case when $G$ is $\eps$-close to the union of two disjoint copies of $K_{n/2}$, we have room to spare in the lower bound on~$D$.

\begin{thm}\label{1factbip}
There are $\eps_{\rm ex} >0$ and $n_0\in\mathbb{N}$ such that the following holds. Let $n\ge n_0$ and suppose that
$D \ge (1/2-\eps_{\rm ex})n$ is even. 
Suppose that $G$ is a $D$-regular graph on $n$ vertices which is $\eps_{\rm ex}$-bipartite.
Then $G$ has a Hamilton decomposition.
\end{thm}
Theorem~\ref{1factbip} is one of the two main results proven in Chapter~\ref{paper2}.
The following result is an easy consequence of Tutte's theorem
and gives the degree threshold for a single perfect matching in  a regular graph.
Note the condition on $D$ is the same as in Theorem~\ref{1factthm}.
\begin{prop}\label{singlematching}
Suppose that $D\ge 2\lceil n/4 \rceil -1$ and $n$ is even.
Then every $D$-regular graph $G$ on $n$ vertices has a perfect matching.
\end{prop}
\proof If $D\ge n/2$ then $G$ has a Hamilton cycle (and thus a perfect matching) by Dirac's theorem.
So we may assume that $D=n/2-1$ and so $n = 0 \pmod 4$. In this case, we will use Tutte's
theorem which states that a graph $G$ has a perfect matching if for every
set $S\subseteq V(G)$ the graph $G-S$ has at most $|S|$ odd components (i.e.~components on an odd number of vertices).
The latter condition  holds if $|S|\le 1$ and%
    \COMMENT{If $|S|=0$ then this follows since $n= 0\pmod 4$ (If $G$ is disconnected, it is the union of two cliques on $n/2$ vertices).
If $|S|=1$ then $G-S$ can have at most two components
(as $D\ge n/2-1$) and at most one of these can be odd (as $n$ is even).}
if $|S|\ge n/2$.

If $|S|=n/2-1$ and $G-S$ has more than $|S|$ odd components, then $G-S$ consists of isolated vertices.%
     \COMMENT{If one component has at least 3 vertices then we have at most $n/2+1-2=|S|$ components.
If $k\ge 1$ components are an edge, then have at most $n/2+1-k$ components and so at most $n/2+1-2k\le |S|$ odd components.}
But this implies that each vertex outside $S$ is joined to all vertices in $S$, contradicting the $(n/2-1)$-regularity of $G$.

If $2\le |S|\le n/2-2$, then every component of $G-S$ has at least $n/2-|S|$ vertices
and so $G-S$ has at most $\lfloor (n-|S|)/(n/2-|S|)\rfloor$ components. But $\lfloor (n-|S|)/(n/2-|S|)\rfloor\le |S|$
unless $n=8$ and $|S|=2$.
(Indeed, note that $(n-|S|)/(n/2-|S|)\le |S|$ if and only if $n+|S|^2-(n/2+1)|S|\le 0$.
The latter holds for $|S|=3$ and $|S|=n/2-2$, and so for all values in between.
The case $|S|=2$ can be checked separately.)%
   \COMMENT{Indeed, if $|S|=n/2-2$ then $(n-|S|)/(n/2-|S|)=(n/2+2)/2=n/4+1\le n/2-2=|S|$ if $n\ge 12$. (If $n<12$ then
we have $n=8$ (since $n = 0 \pmod 4$) and $|S|=n/2-2=2$.)
If $|S|=3$ (and so $n\ge 12$) then $(n-|S|)/(n/2-|S|)=(n-3)/(n/2-3)\le 3$.
In the remaining case $|S|=2$ we have
$\lfloor(n-|S|)/(n/2-|S|)\rfloor=\lfloor(n-2)/(n/2-2)\rfloor = 2+\lfloor 4/(n-4) \rfloor = 2 $ unless $n=8$.}
If $n=8$ and $|S|=2$, it is easy to see that $G-S$ has at most two odd components.
\endproof

\removelastskip\penalty55\medskip\noindent{\bf Proof of Theorem~\ref{1factthm}. }
Let $\tau=\tau(1/3)$ be the constant returned by Theorem~\ref{undir_decomp} for $\alpha:=1/3$.
Choose $n_0\in\mathbb{N}$ and constants $\nu,\eps_{\rm ex}$ such that $1/n_0\ll \nu\ll \tau, \eps_{\rm ex}$ and $\eps_{\rm ex}\ll 1$.
Let $n\ge n_0$ and let $G$ be a $D$-regular graph as in Theorem~\ref{1factthm}.
Lemma~\ref{lem:characteriseexpanders} implies that $G$ satisfies one of the following properties:
\begin{enumerate}
\item[{\rm (i)}] $G$ is $\eps_{\rm ex}$-close to $K_{n/2,n/2}$;
\item[{\rm (ii)}] $G$ is $\eps_{\rm ex}$-close to the union of two disjoint copies of $K_{n/2}$;
\item[{\rm (iii)}] $G$ is a robust $(\nu,\tau)$-expander.
\end{enumerate}
If (i) holds and $D$ is even, then as observed at the beginning of this subsection, this implies that
$G$ is $2\eps_{\rm ex}$-bipartite. So Theorem~\ref{1factbip} implies that $G$ has a Hamilton decomposition and thus also
a $1$-factorization (as $n$ is even and so every Hamilton cycle can be
decomposed into two perfect matchings). 
Suppose that  (i) holds and $D$ is odd. Then Proposition~\ref{singlematching}
implies that $G$ contains a perfect matching $M$. Now $G-M$ is still $\eps_{\rm ex}$-close to $K_{n/2,n/2}$
and so Theorem~\ref{1factbip} implies that $G-M$ has a Hamilton decomposition.
Thus $G$ has a $1$-factorization.
If (ii) holds, then Theorem~\ref{1factstrong} and~\eqref{minexact} imply that $G$ has a $1$-factorization.
If (iii) holds and $D$ is odd, we use Proposition~\ref{singlematching} to choose a perfect matching $M$ in $G$
and let $G':=G-M$.  
If $D$ is even, let $G':=G$. In both cases, $G'-M$ is still a robust $(\nu/2,\tau)$-expander.
So Theorem~\ref{undir_decomp} gives a Hamilton decomposition of $G'$. 
So  $G$ has a $1$-factorization.
\endproof

The proof of Theorem~\ref{HCDthm} is similar to that of Theorem~\ref{1factthm}.

\removelastskip\penalty55\medskip\noindent{\bf Proof of Theorem~\ref{HCDthm}. }
Choose $n_0\in\mathbb{N}$ and constants $\tau, \nu,\eps_{\rm ex}$ as in the proof of Theorem~\ref{1factthm}.
Let $n\ge n_0$ and let $G$ be a $D$-regular graph as in Theorem~\ref{HCDthm}.
As before, Lemma~\ref{lem:characteriseexpanders} implies that $G$ satisfies one of (i)--(iii).
Suppose first that (i) holds. 
If $D$ is odd, $n$ must be even and so $D \ge n/2$. Choose a perfect matching $M$ in $G$ (e.g.~by applying Dirac's theorem)
and let $G':=G-M$. If $D$ is even, let $G':=G$. Note that in both cases $G'$ is $\eps_{\rm ex}$-close to $K_{n/2,n/2}$
and so $2\eps_{\rm ex}$-bipartite. Thus Theorem~\ref{1factbip} implies that $G'$ has a Hamilton decomposition.

Suppose next that (ii) holds. 
Note that by~\eqref{minexact}, $D  \ge n-2 \lfloor n/4 \rfloor -1$
unless $n =3 \pmod 4$ and $D=\lfloor n/2 \rfloor$. But the latter would mean that both $n$ and $D$ are odd, which is impossible.
So the conditions of  Theorem~\ref{1factstrong} are satisfied.
Moreover, since  $D \ge \lfloor n/2 \rfloor$, Proposition~\ref{prp:e(A',B')}(ii) implies that the size of 
a minimum cut in $G$ is at least $D$. Thus
Theorem~\ref{1factstrong} implies that $G$ has a decomposition into Hamilton cycles and at most one perfect matching.

Finally, suppose that (iii) holds. 
If $D$ is odd (and thus $n$ is even), we can apply Proposition~\ref{singlematching} again to find a perfect matching $M$ in $G$ and let $G':=G-M$.
If $D$ is even, let $G':=G$. In both cases, $G'$ is still a robust $(\nu/2,\tau)$-expander.
So Theorem~\ref{undir_decomp} gives a Hamilton decomposition of $G'$.
\endproof

\subsection{Deriving Theorem~\ref{NWmindeg}}

The derivation of Theorem~\ref{NWmindeg} is similar to that of the previous two results.
We will replace the use of Lemma~\ref{lem:characteriseexpanders} and Theorem~\ref{undir_decomp}
with the following result, which is an immediate consequence of the two main results in~\cite{KLOmindeg}.

\begin{thm}\label{NWmindegrob}
For every $\eps_{\rm ex}>0$ there exists an $n_0 \in \mathbb N$ such that the following holds. Suppose that $G$ is a graph on $n\ge n_0$ vertices with
$\delta(G)\ge n/2$.
Then $G$ satisfies one of the following properties:
\begin{itemize}
\item[{\rm (i)}] $G$ is $\eps_{\rm ex}$-close to $K_{n/2,n/2}$;
\item[{\rm (ii)}] $G$ is $\eps_{\rm ex}$-close to the union of two disjoint copies of $K_{n/2}$;
\item[{\rm (iii)}] $G$ contains ${\rm reg}_{\rm even}(n,\delta)/2$ edge-disjoint Hamilton cycles.
\end{itemize}
\end{thm}
To deal with the near-bipartite case (i), we will apply the following result which we prove in  Chapter~\ref{paper2}.
\begin{thm}\label{NWmindegbip}
For each $\alpha>0$ there are $\eps_{\rm ex} >0$ and $n_0\in\mathbb{N}$ such that the following holds. Suppose that
$F$ is an $\eps_{\rm ex}$-bipartite graph on $n \ge n_0$ vertices with $\delta(F)\ge (1/2-\eps_{\rm ex})n$.
Suppose that $F$ has a $D$-regular spanning subgraph $G$%
\COMMENT{later on we say e.g. `let $G$ be as in Theorem~\ref{NWmindegbip} so this is necessary here}  
 such that $n/100\le D\le (1/2-\alpha)n$ and $D$ is even. Then $F$ contains $D/2$ edge-disjoint Hamilton cycles.
\end{thm}

The next result immediately implies Theorem~\ref{NWmindeg} in Case (ii) when $G$ is $\eps$-close
to the union of two disjoint copies of $K_{n/2}$.
We will prove it in Chapter~\ref{paper1} (Section~\ref{sec:NWminclique}).
Since~$G$ is far from extremal in this case, we obtain almost twice as many edge-disjoint Hamilton cycles as
needed for Theorem~\ref{NWmindeg}. 

\begin{thm}\label{thm:NWminclique}
For every $\eps >0$, there exist $\eps_{\rm ex} >0$ and $n_0 \in \mathbb N$ such that the following holds.
Suppose $n \geq n_0$ and $G$ is a graph on $n$ vertices such that $G$ is $\eps_{\rm ex}$-close to the union of two disjoint copies
of $K_{n/2}$ and such that $\delta(G) \ge n/2$. Then $G$ has at least $(1/4 - \eps ) n$ edge-disjoint Hamilton cycles.
\end{thm}

We will also use the following well-known result of Petersen.

\begin{thm}\label{petersen}
Every regular graph of positive even degree contains a $2$-factor.
\end{thm}

\removelastskip\penalty55\medskip\noindent{\bf Proof of Theorem~\ref{NWmindeg}. }
Choose $n_0\in\mathbb{N}$ and $\eps_{\rm ex}$ such that%
   \COMMENT{Daniela: replaced $n$ by $n_0$ in the hierarchy} 
$1/n_0\ll \eps_{\rm ex}\ll 1$.
In particular, we choose $\eps_{\rm ex}\le \eps^1_{\rm ex}(1/12)$, where $\eps^1_{\rm ex}(1/12)$ is the constant
returned by Theorem~\ref{thm:NWminclique} for $\eps:=1/12$, as well as
$\eps_{\rm ex}\le \eps^2_{\rm ex}(1/6)/2$, where $\eps^2_{\rm ex}(1/6)$ is the constant
returned by Theorem~\ref{NWmindegbip} for $\alpha:=1/6$.
Let $G$ be a graph on $n\ge n_0$ vertices with $\delta:=\delta(G)\ge n/2$. 
Theorem~\ref{NWmindegrob} implies that we may assume that $G$ satisfies either (i) or (ii).
Note that in both cases it follows that $\delta(G)\le (1/2+5\eps_{\rm ex})n$.%
    \COMMENT{Let $A$ be a set of size $\lfloor n/2 \rfloor$. If $\delta(G)\ge (1/2+5\eps_{\rm ex})n$
 then every vertex in $A$ must have at least $5\eps_{\rm ex}n$ neighbours outside $A$ and
so $e(A,V(G)\setminus A)\ge 2 \eps_{\rm ex}n^2$. Also every vertex in $A$ must have at least
$5\eps_{\rm ex}n-1$ neighbours inside $A$ and so $e(A)> \eps_{\rm ex}n^2$.}
So~\eqref{regbound} implies that $n/5\le {\rm reg}_{\rm even}(n,\delta)\le 3n/10$.

Suppose first that (i) holds. As mentioned above, this implies that $G$ is $2\eps_{\rm ex}$-bipartite.
Let $G'$ be a $D$-regular spanning subgraph of $G$ such that $D$ is even and
$D\ge {\rm reg}_{\rm even}(n,\delta)$. 
Petersen's theorem (Theorem~\ref{petersen}) implies that by successively deleting $2$-factors
of $G'$, if necessary, we may in addition assume that $D\le n/3$. Then Theorem~\ref{NWmindegbip} (applied with $\alpha:=1/6$)
implies that $G$ contains at least $D/2\ge {\rm reg}_{\rm even}(n,\delta)/2$ edge-disjoint Hamilton cycles.

Finally suppose that (ii) holds. Then Theorem~\ref{thm:NWminclique} (applied with $\eps:=1/12)$
implies that $G$ contains $n/6\ge {\rm reg}_{\rm even}(n,\delta)/2$ edge-disjoint Hamilton cycles.
\endproof


\section{Tools} \label{sec:tools}
\subsection{$\eps$-regularity} 
If $G=(A,B)$ is an undirected bipartite graph with vertex classes $A$ and $B$, then the
\emph{density} of $G$ is defined as
$$ d(A, B) := \frac{e_G(A,B)}{|A||B|}.$$
For any $\eps >0$, we say that $G$ is \emph{$\eps$-regular} if for any $A'\subseteq A$
and $B' \subseteq B$ with $|A'| \geq \eps |A|$ and $|B'| \geq \eps |B|$
we have $|d(A',B') - d(A, B)| < \eps$. We say that $G$ is \emph{$(\eps, \ge d)$-regular} if
it is $\eps$-regular and has density $d'$ for some $d' \ge d-\eps$.

We say that $G$ is \emph{$[\eps,d]$-superregular} if it is $\eps$-regular and $d_G(a)=(d\pm \eps)|B|$
for every $a\in A$ and $d_G(b)=(d\pm \eps)|A|$ for every $b\in B$. $G$ is \emph{$[\eps, \ge d]$-superregular} if
it is  \emph{$[\eps, d']$-superregular} for some $d' \ge d$.

Given disjoint vertex sets $X$ and $Y$ in a digraph $G$, recall that $G[X,Y]$ denotes
the bipartite subdigraph of $G$ whose vertex classes are $X$ and $Y$ and whose
edges are all the edges of $G$ directed from $X$ to $Y$. We often view $G[X,Y]$ as
an undirected bipartite graph. In particular, we say $G[X,Y]$ is \emph{$\eps$-regular},
\emph{$(\eps,\ge d)$-regular}, $[\eps,d]$-superregular or $[\eps, \ge d]$-superregular
if this holds when $G[X,Y]$ is viewed as an undirected graph.

The following proposition states that the graph obtained from a superregular pair
by removing a small number of edges at every vertex is still
superregular (with slightly worse parameters). We omit the proof which follows straightforwardly from the definition of superregularity.
A similar argument is for example included in~\cite{Kelly}.%
   \COMMENT{In the Kelly paper we have $0 <1/m \ll \eps \le d' \le d \ll 1$ and the proof doesn't work for
$0 <1/m \ll \eps \le d' \le d \le 1/2$ say (we want to keep $d' \le d$ there). So here is the proof for
the hierarchy in the prop. Let us first prove check that $G'$ is $2\sqrt{d'}$-regular. Let $d^*$ denote the density of $G$.
Suppose that $S\subseteq A$, $T\subseteq B$ are such that $|S|\ge \sqrt{d'}m$
and $|T|\ge \sqrt{d'}m$. 
Then $e_{G'}(S,T) \ge (d^*-\eps)|S||T| - |S|d'm   \ge (d^*-\eps)|S||T|-|S|d'\cdot |T|/\sqrt{d'}\ge (d^*-2\sqrt{d'})|S||T|$.
Since clearly $e_{G'}(S,T)\le e_{G}(S,T)\le (d^*+\eps)|S||T|$, $G$ is $2\sqrt{d'}$-regular.
Note that in $G'$ the degrees of the vertices in $A'$ are still at most
$(d+\eps)m\le (d+\eps)|B'|/(1-d')\le (d+2\sqrt{d'})|B'|$. Similarly, the degrees in $G'$ of the vertices in
$B'$ are still at most $(d+2\sqrt{d'})|A'|$.}

\begin{prop} \label{superslice}
Suppose that $0 <1/m \ll \eps \le d' \ll d \leq 1$.
Let $G$ be a bipartite graph with vertex classes $A$ and $B$ of size $m$.
Suppose that $G'$ is obtained from $G$ by removing at most $d'm$ vertices from each vertex class and at most $d'm$ edges incident to each vertex from $G$.
If $G$ is $[\eps,d]$-superregular then $G'$ is $[2\sqrt{d'},d]$-superregular.
\end{prop}

We will also use the following well-known observation, which 
easily follows from Hall's theorem and the definition of $[\eps, d]$-superregularity.

\begin{prop}\label{perfmatch}
Suppose that $0 <1/m \ll \eps \ll d \le 1$.
Suppose that $G$ is an $[\eps, d]$-superregular bipartite graph with vertex classes of size $m$.
Then $G$ contains a perfect matching.
\end{prop}

We will also apply the following simple fact.
\begin{fact}\label{simplefact} Let $\eps >0$.
Suppose that $G$ is a bipartite graph with vertex classes  of size $n$ such that
$\delta (G) \geq (1-\eps)n$. Then $G$ is $[\sqrt{\eps},1]$-superregular.
\end{fact}

\subsection{A Chernoff-Hoeffding Bound}
We will often use the following Cher\-noff-Hoeffding bound for binomial and hypergeometric
distributions (see e.g.~\cite[Corollary 2.3 and Theorem 2.10]{Janson&Luczak&Rucinski00}).
Recall that the binomial random variable with parameters $(n,p)$ is the sum
of $n$ independent Bernoulli variables, each taking value $1$ with probability $p$
or $0$ with probability $1-p$.
The hypergeometric random variable $X$ with parameters $(n,m,k)$ is
defined as follows. We let $N$ be a set of size $n$, fix $S \subseteq N$ of size
$|S|=m$, pick a uniformly random $T \subseteq N$ of size $|T|=k$,
then define $X:=|T \cap S|$. Note that $\ex X = km/n$.

\begin{prop}\label{chernoff}
Suppose $X$ has binomial or hypergeometric distribution and $0<a<3/2$. Then
$\mathbb{P}(|X - \mathbb{E}X| \ge a\mathbb{E}X) \le 2 e^{-a^2\mathbb{E}X /3}$.
\end{prop}

\subsection{Other Useful Results}
We will need the following fact, which is a simple consequence of 
Vizing's theorem and was first observed by McDiarmid and independently by de Werra (see e.g.~\cite{west}).%
\COMMENT{no need for original references as it's not that deep...
osthus deleted the reference to a proof in paperII , as it's been deleted there}

\begin{prop} \label{prop:matchingdecomposition}
Let $G$ be a graph with $\chi '(G) \le m $. Then $G$ has a decomposition into $m$ matchings $M_1, \dots ,M_m$
with $|e(M_i) - e(M_j)| \le 1$ for all $i, j  \le m$.
\end{prop}
It is also useful to state Proposition~\ref{prop:matchingdecomposition} in the following alternative form.
\begin{cor}\label{basic_matching_dec}\
 Let $H$ be a graph with maximum degree at most $\Delta.$
Then $E(H)$ can be decomposed into $\Delta+1$ edge-disjoint matchings $M_1, \ldots, M_{\Delta+1}$
such that $|e(M_i)-e(M_j)|\le 1$ for all $i, j\le \Delta+1$.
\end{cor}

The following partition result will also be useful.

\begin{lemma}\label{lma:partition2}
Suppose that $0 <  1/n \ll \eps , \eps_1 \ll \eps_2  \ll 1/K \ll 1$, that $r \le 2K$,
that $Km \ge n/4$ and that $r,K,n,m \in \mathbb N$.%
\COMMENT{Andy: added "and that $r,K,n,m \in \mathbb N$."}
Let $G$ and $F$ be graphs on $n$ vertices with $V(G)=V(F)$.
Suppose that there is a vertex partition of $V(G)$ into $U,R_1, \dots, R_r$ with the following properties:
\begin{itemize}
                \item $|U| = K m $.
                \item $\delta(G[U]) \ge \eps n$ or $\Delta(G[U]) \le  \eps n $.
                \item For each $j \le r$ we either have $d_G(u,R_j)\le \eps n$ for all $u\in U$ or $d_G(x,U)\ge \eps n$ for all $x\in R_j$.
\end{itemize}
Then there exists a partition of $U$ into $K$ parts $U_1, \dots, U_K$
satisfying the following properties:
\begin{itemize}
                \item[\rm (i)] $|U_{i}|  = m $ for all $i \le K$.
                \item[\rm (ii)] $d_G(v,U_i)  = (d_G(v,U) \pm \eps_1 n) /K $ for all $v \in V(G)$ and all $i\le K$.
                \item[\rm (iii)] $e_G( U_i, U_{i'} )  = 2 (e_G(U) \pm \eps_2 \max\{  n, e_G(U) \} ) / K^2$ for all $1\le i \neq i' \le K$.
                \item[\rm (iv)] $e_G( U_i)  = (e_G(U) \pm \eps_2 \max\{  n, e_G(U) \} ) / K^2$ for all $ i \le K$.
                \item[\rm (v)] $e_G( U_i, R_{j} ) = ( e_G(U,R_j) \pm   \eps_2 \max \{ n , e_G(U,R_j) \} ) /K  $ for all $i \le K$ and $j \le r$.
                \item[\rm (vi)] $d_F(v,U_i)  = (d_F(v,U) \pm \eps_1 n) /K $ for all $v \in V(F)$ and all $i\le K$.
\end{itemize}
\end{lemma}
 
\begin{proof}
Consider an equipartition $U_1, \dots , U_K$ of $U$ which is chosen uniformly at random.%
     \COMMENT{the previous binomial approach had the problem that at some stage the argument involved a sum over a random variable.
The present one also seems simpler.}
So (i) holds by definition.
Note that for a given vertex $v \in V(G)$, $d_G(v,U_i)$ has the hypergeometric distribution
with mean $d_G(v, U)/K$.
So if $d_G(v, U) \ge \eps_1 n/K$,  Proposition~\ref{chernoff} implies that
$$
                \pr \left( \left| d_G(v,U_{i}) -  \frac{d_G(v, U)}{K} \right|  \ge \frac{\eps_1 d_G(v,U) }{K}\right) 
  \le 2 \exp \left(- \frac{\eps_1^2 d_G(v,U)}{3K} \right)
\le \frac{1}{n^2}.
$$
Thus we deduce that for all $v \in V(G)$ and all $i \le K$,
\begin{align*} 
                \pr \left( \left| d_G(v,U_{i}) -  d_G(v, U)/K \right|  \ge \eps_1 n /K \right) 
\le 1/n^2.
\end{align*}
Similarly, 
\begin{equation*} 
\pr \left( \left| d_F(v,U_{i}) -  d_F(v, U)/K \right|  \ge \eps_1 n /K \right) \le 1/n^2.
\end{equation*}
So with probability at least 3/4, both (ii) and (vi) are satisfied.
 
We now consider (iii) and (iv).
Fix $i,i' \le K$. If $i \neq i'$, let $X:=e_G(U_{i}, U_{i'})$.
If $i = i'$, let $X:=2e_G(U_{i})$.
For an edge $f \in E(G[U])$, let $E_f$ denote the event that $f \in E(U_i, U_{i'})$.
So if $f=xy$ and $i \neq i'$, then%
    \COMMENT{to derive these and other expressions, it's best to view the sample space as the set of all
injective mappings from $U$ to $[|U|]$, where
we assign $x$ to $U_i$ if it is mapped to $[(i-1)m+1,im]$. Then each equipartition has the same number of mappings which produce this equipartition.
So as probability space one can consider the uniform distribution over the above mappings.}
\begin{equation} \label{prf}
\pr (E_f) =2 \pr (x \in U_i) \pr (y \in U_{i'} \mid x \in U_i) = 2 \frac{m}{|U|} \cdot  \frac{m}{|U|-1}.
\end{equation}
Similarly, if $f$ and $f'$ are disjoint (that is, $f$ and $f'$ have no common 
endpoint)\COMMENT{B.Cs: the explanation for disjointness is given here} and $i \neq i'$, then%
   \COMMENT{note $\frac{m-1}{m} \frac{|U|}{|U|-2} \le (1-1/m)(1+3/|U|) \le 1$ and similarly for the second fraction.}
\begin{equation} \label{ff'}
\pr ( E_{f'} \mid E_f )=2\frac{m-1}{|U|-2} \cdot \frac{m-1}{|U|-3} \le 
2\frac{m}{|U|} \cdot \frac{m}{|U|-1} =
\pr ( E_{f'}).
\end{equation}
By (\ref{prf}), if $i \not =i'$, we also have
\begin{equation}  \label{exX}
\ex ( X ) = 2 \frac{e_G(U)}{K^2} \cdot \frac{|U|}{|U|-1}= \left( 1 \pm \frac{2}{|U|} \right) \frac{2e_G(U)}{K^2}= \left( 1 \pm \eps_2/4 \right)\frac{2e_G(U)}{K^2}.
\end{equation}
If%
\COMMENT{Deryk divided by 4 in the final equality as we need this in eq. devex}
$f=xy$ and $i = i'$, then  
\begin{equation} \label{prf2}
\pr (E_f) = \pr (x \in U_i) \pr (y \in U_{i} \mid x \in U_i) =  \frac{m}{|U|} \cdot  \frac{m-1}{|U|-1}.
\end{equation}
So if $i = i'$, similarly to (\ref{ff'}) we also obtain $\pr ( E_{f'} \mid E_f ) \le \pr (E_f)$ for disjoint $f$ and $f'$%
\COMMENT{$\pr ( E_{f'} \mid E_f )=\frac{m-2}{|U|-2} \cdot \frac{m-3}{|U|-3} \le
\frac{m}{|U|} \frac{m-1}{|U|-1} =\pr ( E_{f'}).$}
and we obtain the same bound as in (\ref{exX}) on $\ex ( X )$ (recall that $X=2e_G(U_i)$ in this case).%
     \COMMENT{Daniela added brackets}
 
Note that if $i\neq i'$ then%
   \COMMENT{To see the first inequality note that by (\ref{ff'}) we have $\pr(E_{f'} \mid E_f ) -   \pr ( E_{f'} )\le 0$ if $f$ and
$f'$ are disjoint. But there are at most $2 \Delta (G[U])$ edges $f$ which share an endvertex with $f'$ (this includes the
case $f=f'$) and each of these edges $f$ contributes at most 1.}
\begin{eqnarray*}
\textrm{Var} ( X ) 
    & = & \sum_{f \in E(U)}  \sum_{f' \in E(U)}\left(\pr ( E_f \cap E_{f'} )  -   \pr ( E_f ) \pr ( E_{f'} )\right) \\
                & = & \sum_{f \in E(U)} \pr ( E_f ) \sum_{f' \in E(U)}  \left( \pr ( E_{f'} \mid E_f )  -    \pr ( E_{f'} ) \right) \\  
                & \stackrel{(\ref{ff'})}{\le} & \sum_{f \in E(U)} \pr ( E_f )\cdot 2 \Delta (G[U])  
 \stackrel{(\ref{exX})}{\le} \frac{3e_G(U)}{K^2}  \cdot 2\Delta (G[U]) \\
     & \le & e_G(U)  \Delta (G[U]).
\end{eqnarray*}
Similarly, if $i=i'$ then
\begin{eqnarray*}
\textrm{Var} ( X ) 
     =  4\sum_{f \in E(U)}  \sum_{f' \in E(U)}\left(\pr ( E_f \cap E_{f'} )  -   \pr ( E_f ) \pr ( E_{f'} ) \right)
                \le e_G(U)  \Delta (G[U]).
\end{eqnarray*}
Let $a:= e_G(U)  \Delta (G[U])$. In both cases, from Chebyshev's inequality,%
   \COMMENT{$\pr \left( | X- \ex(X) |  \ge b  \right) \le Var(X)/b^2$}
it follows that
\begin{align}
\pr \left( | X- \ex(X) |  \ge \sqrt{ a/\eps^{1/2} } \right) \le \eps^{1/2}.  \nonumber
\end{align}
Suppose that $\Delta (G[U]) \le \eps n$.
If we also have have $e_G(U) \le  n$, then 
$\sqrt{ a/\eps^{1/2} } \le   \eps^{1/4}  n \le \eps_2n/2K^2$.
If $e_G(U) \ge n$, then $\sqrt{ a/ \eps^{1/2} } \le \eps^{1/4 }e_G(U) \le \eps_2 e_G(U)/2K^2$.

If we do not have $\Delta (G[U]) \le \eps n$, then our assumptions imply that $\delta (G[U]) \ge \eps n$.
So $\Delta(G[U]) \le n \le \eps e_G(G[U])$ with room to spare.
This in turn means that $\sqrt{ a/ \eps^{1/2} } \le \eps^{1/4} e_G(U) \le \eps_2 e_G(U)/2K^2$.
So in all cases, we have 
\begin{align}
                \pr \left(  \left| X - \ex(X) \right| \ge  \frac{\eps  _2 \max \{ n, e_G(U)\} }{2K^2}\right) \le \eps^{1/2}. \label{eqn:preU}
\end{align}
Now note that by (\ref{exX}) we have
\begin{equation} \label{devex}
\left| \ex(X)- \frac{2e_G(U)}{K^2} \right| \le \frac{\eps  _2 e_G(U) }{2K^2}.
\end{equation}
So (\ref{eqn:preU}) and~(\ref{devex}) together imply that for fixed $i,i'$ the bound in (iii) fails with probability
at most $\eps^{1/2}$. The analogue holds for the bound in~(iv).
By summing over all possible values of $i,i' \le K$, we have that~(iii) and~(iv) hold with probability at least $3/4$.

A similar argument shows that for all $i \le K$ and $j \le r$, we have%
    \COMMENT{
Fix $i \le K$, $j \le r$ and let $X:= e_G(U_i, R_j)$.
For an edge $f$, let $E_f$ denote the event that $f \in E(U_i, R_j)$.
Note $\pr (E_f)=m/|U|=1/K$ and that if $f,f'$ have different endpoints in $U$, we have
$\pr ( E_{f'} \mid E_f )= \frac{m-1}{|U|-1} \le \frac{m}{|U|} =\pr(E_f).
$
Clearly, $\ex ( X ) = e_G(U, R_j) /K$.
Also let $\Delta'(U,R_j))=\max_{u \in U} d (u,R_j)$.
Note that
\begin{eqnarray*}
\textrm{Var} ( X ) 
                & = & \sum_{f \in E(U,R_j)} \pr ( E_f ) \sum_{f' \in E(U,R_j)}  \left( \pr ( E_{f'}) | E_f )  -    \pr ( E_{f'} ) \right) \\       
                & \le & \sum_{f \in E(U,R_j)} \pr ( E_f )\Delta' (U,R_j) \le  \frac{e_G(U,R_j)}{K}  \cdot \Delta' (U,R_j)
\le  e_G(U,R_j)  \Delta' (U,R_j).
\end{eqnarray*}
Let $a:=e_G(U,R_j)\Delta' (U,R_j)$. From Chebyshev's inequality, it follows that
\begin{align}
\pr \left( | X- \ex(X) |  \ge \sqrt{ a/\eps^{1/2} } \right) \le \eps^{1/2}.  \nonumber
\end{align}
Suppose that $\Delta' (U,R_j) \le \eps n$.
If we also have have $e_G(U,R_j) \le  n$, then 
$\sqrt{ a/\eps^{1/2} } \le  \eps^{1/4} n \le \eps_2n/K$.
If $e_G(U,R_j) \ge n$, then $\sqrt{ a/ \eps^{1/2} } \le  \eps^{1/4} e_G(U,R_j) \le \eps_2 e_G(U,R_j)/K$.
If we do not have $\Delta' (U,R_j) \le \eps n$, then our assumptions imply that $d_G(x,U) \ge \eps n$
for every $x \in R_j$.
So $\Delta'(U,R_j) \le |R_j| \le \eps e_G(U,R_j)$ with room to spare.
This in turn means that $\sqrt{ a/ \eps^{1/2} } \le \eps^{1/4} e_G(U,R_j) \le \eps_2 e_G(U,R_j)/K$.
}
\begin{equation} \label{eqn:preUR}
                \pr \left(  \left| e_G(U_i, R_j) - \frac{e_G(U,R_j)}{K} \right| \ge  \frac{\eps_2 \max \{ n, e_G(U,R_j)\} }{K} \right)
\le \eps^{1/2}.  
\end{equation}
Indeed, fix $i \le K$, $j \le r$ and let $X:= e_G(U_i, R_j)$.
For an edge $f \in G[U,R_j]$, let $E_f$ denote the event that $f \in E(U_i, R_j)$.
Then $\pr (E_f)=m/|U|=1/K$ and so $\ex ( X ) = e_G(U, R_j) /K$.
The remainder of the argument proceeds as in the previous case
(with slightly simpler calculations).
 
So (v) holds with probability at least 3/4, by summing over all possible values of $i \le K$ and $j \le r$ again.
So with positive probability, the partition satisfies all requirements.
\end{proof}

\chapter{The two cliques case}\label{paper1}
This chapter is concerned with proving Theorems~\ref{1factthm},~\ref{HCDthm} and ~\ref{NWmindeg} in the case when our graph is close  to the union of two disjoint copies of a clique $K_{n/2}$ (Case (ii)).
More precisely, we prove Theorem~\ref{thm:NWminclique} (i.e.~Case (ii) of Theorem~\ref{NWmindeg})
and  Theorem~\ref{1factstrong}, which is a common generalization of 
Case (ii) of Theorems~\ref{1factthm} and~\ref{HCDthm}.
In Section~\ref{overview}, we give a sketch of the arguments for the `two cliques' Case (ii)
(i.e.~the proofs of Theorems~\ref{1factstrong} and~\ref{thm:NWminclique}).
Sections~\ref{sec:new}--\ref{sec:schemes} (and part of  Section~\ref{sec:NWminclique})
are common to the proofs of both Theorems~\ref{1factstrong} and~\ref{thm:NWminclique}.
Theorem~\ref{thm:NWminclique} is proved in Section~\ref{sec:NWminclique}.
All the subsequent sections of this chapter are devoted to the proof of Theorem~\ref{1factstrong}.

In this chapter (and Chapter~\ref{paper4}) it is convenient to view matchings as graphs (in which every vertex has degree precisely one).

\section{Overview of the Proofs of Theorems~$\text{\ref{1factstrong}}$ and~$\text{\ref{thm:NWminclique}}$}\label{overview}

The proof of Theorem~\ref{thm:NWminclique} is much simpler than that of Theorems~\ref{1factstrong}
(mainly because its assertion leaves some leeway -- one could probably find a slightly larger set of edge-disjoint Hamilton cycles than guaranteed by Theorem~\ref{thm:NWminclique}).
Moreover, the ideas used in the former all appear in the proof of the latter too.

\subsection{Proof Overview for Theorem~\ref{thm:NWminclique}} \label{sec:overI}

Let $G$ be a graph on $n$ vertices with $\delta(G) \ge n/2$ which is close to being the union of two disjoint cliques.
So there is a vertex partition of $G$ into sets $A$ and $B$ of roughly equal size so that $G[A]$ and $G[B]$ are almost complete.
Our aim is to construct almost $n/4$ edge-disjoint Hamilton cycles.

Several techniques have recently been developed which yield
approximate decompositions of dense (almost) regular graphs, i.e.~a set of Hamilton cycles covering almost all the edges 
(see e.g.~\cite{CKO,FKS,fk,KOTkelly,OS}).
This leads to the following idea: replace $G[A]$ and $G[B]$ by multigraphs $G_A$ and $G_B$
so that any suitable pair of Hamilton cycles $C_A$ and $C_B$ of $G_A$ and $G_B$ respectively corresponds
to a single Hamilton cycle $C$ in the original graph $G$. 
We will construct $G_A$ and $G_B$ by deleting some edges of $G$ and introducing some `fictive edges'.
(The introduction of these fictive edges is the reason why $G_A$ and $G_B$ are multigraphs.)

We next explain the key concept of these `fictive edges'.
The following graph $G$ provides an instructive example: suppose that $n=0  \pmod 4$. 
Let $G$ be obtained from two disjoint cliques induced by sets $A$ and $B$ of size $n/2$
by adding a perfect matching $M$ between~$A$ and $B$. Note that $G$ is $n/2$-regular. 
Now pair up the edges of $M$ into $n/4$ pairs $(e_i,e_{i+1})$ for $i = 1, 3, \dots, n/2-1$.
Write $e_i=:x_iy_i$ with $x_i \in A$ and $y_i \in B$.
Next let $G_A$ be the multigraph obtained from $G[A]$ by adding all the edges $x_ix_{i+1}$, where $i$ is odd.
Similarly, let $G_B$ be obtained from $G[B]$ by adding all the edges $y_iy_{i+1}$, where $i$ is odd.
We call the edges $x_ix_{i+1}$ and $y_iy_{i+1}$ fictive edges.
Note that $G_A$ and $G_B$ are regular multigraphs.
Now pair off the  fictive edges in $G_A$ with those in $G_B$, i.e.~$x_ix_{i+1}$ is paired off with 
$y_iy_{i+1}$. Suppose that $C_A$ is a Hamilton cycle in $G_A$ which contains $x_ix_{i+1}$ (and no
other fictive edges)
and $C_B$ is a Hamilton cycle in $G_B$ which contains $y_iy_{i+1}$ (and no
other fictive edges).
Then together, $C_A$ and $C_B$ correspond to a Hamilton cycle $C$ in the original graph $G$
(where fictive edges are replaced by the corresponding matching edges in $M$ again).

So we have reduced the problem of finding many edge-disjoint Hamilton cycles in $G$ to that of finding
many edge-disjoint Hamilton cycles in the almost complete graph $G_A$ (and $G_B$), 
with the additional requirement that each such Hamilton cycle contains a unique fictive edge.
This can be achieved via the `approximate decomposition result'  (see Lemma~\ref{almostthm} which is proved in Chapter~\ref{paper3}).

Additional difficulties arise from `exceptional' vertices, namely those which have high degree into both $A$ and $B$.
(It is easy to see that there cannot be too many of these vertices.)
Fictive edges also provide a natural way of `eliminating' these exceptional vertices.
Suppose for example that $G'$ is obtained from the graph~$G$ above by adding a vertex
$a$  so that $a$ is adjacent to half of the vertices in $A$ and half of the vertices in $B$. 
(Note that $\delta(G')$ is a little smaller than $|G'|/2$, but $G'$ is%
   \COMMENT{Daniela: replaced it by $G'$} 
similar to graphs actually occurring in the proof.)
Then we can pair off the neighbours of $a$ into pairs  within $A$ and introduce a fictive edge $f_i$ between each pair of neighbours.
We also introduce fictive edges $f_i$ between pairs of neighbours of $a$ in $B$. 
Without loss of generality, we have fictive edges $f_1,f_3,\dots,f_{n/2-1}$ (and recall that $|G'| = n+1$).
So we have $V(G'_A)=A$ and $V(G'_B)=B$ again.
We then require each pair of Hamilton cycles $C_A$, $C_B$ of $G'_A$ and $G'_B$ to contain 
$x_ix_{i+1}$, $y_iy_{i+1}$ and a fictive edge $f_i$ (which may lie in $A$ or $B$) where $i$ is odd, see Figure~\ref{fig1}.
Then $C_A$ and $C_B$ together correspond to a Hamilton cycle $C$ in $G'$ again.
The subgraph~$J$ of $G'$ which corresponds to three such fictive edges $x_ix_{i+1}$, $y_iy_{i+1}$ and $f_i$
of $C$ is called a `Hamilton exceptional system'.
$J$ will always be a path system.
So in general, we will first find a sufficient number of edge-disjoint Hamilton exceptional systems~$J$.
Then we apply Lemma~\ref{almostthm}  to find edge-disjoint Hamilton cycles in $G'_A$ and $G'_B$,
where each pair of cycles contains a suitable set~$J^*$ of fictive edges (corresponding to some Hamilton exceptional system~$J$).

\begin{figure}[tbp]
\centering
\begin{tikzpicture}[scale=0.6]
			\draw  (-3,0) ellipse (2.3  and 3.8 );
			\draw (3,0) ellipse (2.3  and 3.8 );
			\node at (-3,-4.5) {$A$};
			\node at (3,-4.5) {$B$}; 	
			\fill (0,4) circle (4pt);
			\node at (0,4.5) {$a$}; 
			\node at (-1.5,1) {$x_i$};
			\node at (-1.6,-1.9) {$x_{i+1}$};
			\node at (1.5,1) {$y_i$};
			\node at (1.6,-1.9) {$y_{i+1}$};
			\node at (3,2.3) {$f_i$};
			\node at (-3,1) {$C_A$};
			\node at (3,1) {$C_B$};
			\begin{scope}[start chain]
			\foreach \i in {0.5,-1.5}
			\fill (-1.5,\i) circle (4pt);
			\end{scope}
			\begin{scope}[start chain]
			\foreach \i in {0.5,-1.5}
			\fill (1.5,\i) circle (4pt);
			\end{scope}
			\fill (2.2,2.7) circle (4pt);
			\fill (3.8,2.7) circle (4pt);
			
			\begin{scope}[dashed,line width=1.5pt]
			\draw (-1.5,0.5)--(1.5,0.5);
			\draw (-1.5,-1.5)--(1.5,-1.5);
			\draw (2.2,2.7)--(0,4);
			\draw (0,4)--(3.8,2.7);
			\end{scope}

			\begin{scope}[line width=1.5pt]
			\draw (-1.5,0.5)--(-1.5,-1.5);
			\draw (1.5,0.5)--(1.5,-1.5);
			\draw (2.2,2.7)--(3.8,2.7);
			\end{scope}
			\begin{scope}[line width=1pt]
			\draw[rounded corners]  (-1.5,0.5) to [out=170, in=-70] (-3,2) to [out=65, in=-150] (-1.5,2.2) to [out=120, in=-10] (-3,3.4) to [out=-150, in=110] (-3.5,1.5) to [out=-150, in=-20] (-4.5,2)  to [out=-135, in=105] (-3,-0.5) to [out=-110, in=30] (-4.5,0) to  [out=-60, in= 180] (-3,-3.5) to [out=30, in= -120] (-1.8,-2.5) to [out=120, in=-65] (-3,-1.5) to [out=65, in=-130] (-2.5,0.5) to [out=-30, in=160] (-1.5,-1.5);		
			\draw[rounded corners] (3.8,2.7)  to [out=-45, in=105] (5,1.2) to [out=-115, in=55] (4.5,0.5) to [out=160, in=-30] (2.5,1.75) to [out=-135, in=110] (3, 0)  to [out=-120, in=75] (2.5,-1) to [out=-20, in=155] (4.2,0) to [out=-30, in= 120] (5,-1) to [out=-110, in= 50] (3.5,-2) to [out=-45, in= 130] (4.5, -2.5)  to [out=-110, in= -70] (2,-2.5) to [out=75, in= -150] (3,-2) to [out=80, in=10](1.5,-1.5);
			\draw[rounded corners] (2.2,2.7) to [out=-165, in=20] (1.5,0.5);
			\end{scope}
\end{tikzpicture}
\caption{Transforming the problem of finding a Hamilton cycle on $V(G')$ into finding two Hamilton cycles $C_A$ and $C_B$ on $A$ and $B$ respectively.}
\label{fig1}
\end{figure}
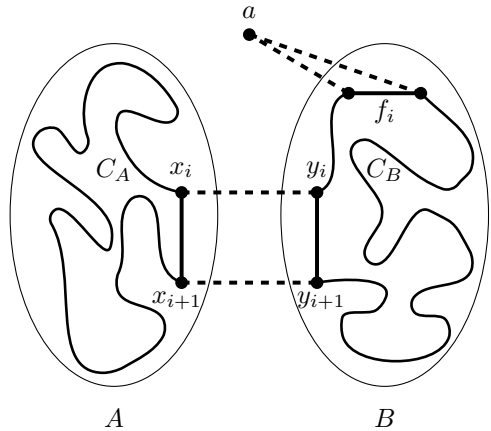

For Lemma~\ref{almostthm}, we need each of the Hamilton exceptional systems $J$ to be `localized':
given a partition of $A$ and $B$ into clusters, the endpoints of the corresponding set~$J^*$ of fictive edges need to be contained
in a single cluster of $A$ and of $B$. 
The fact that the Hamilton exceptional systems need to be localized is one reason for treating exceptional vertices
differently from the others by introducing fictive edges for them. 

\subsection{Proof Overview for Theorem~\ref{1factstrong}}

The main result of this chapter is Theorem~\ref{1factstrong}.
Suppose that $G$ is a $D$-regular graph satisfying the conditions of that theorem.

Using the approach of the previous subsection, one can obtain an approximate decomposition of $G$,
i.e.~a set of edge-disjoint Hamilton cycles covering almost all edges of~$G$.
However, one does not have any control over the `leftover' graph~$H$, which makes a complete decomposition seem infeasible.
This problem was overcome in~\cite{Kelly} by introducing the concept of a `robustly decomposable graph'~$G^{\rm rob}$.
Roughly speaking, this is a sparse regular graph with the following property:
given \emph{any} very sparse regular graph~$H$ with $V(H)=V(G^{\rm rob})$ which is edge-disjoint from $G^{\rm rob}$,
one can guarantee that $G^{\rm rob} \cup H$ has a Hamilton decomposition.
This leads to a natural (and very general) strategy to obtain a decomposition of $G$:
\begin{itemize}
\item[(1)] find a (sparse) robustly decomposable graph~$G^{\rm rob}$ in $G$ and let $G'$ denote the leftover;
\item[(2)] find an approximate Hamilton decomposition of $G'$ and let $H$ denote the (very sparse) leftover;
\item[(3)] find a Hamilton decomposition of~$G^{\rm rob} \cup H$.
\end{itemize}
It is of course far from clear that one can always find such a graph $G^{\rm rob}$.
The main `robust decomposition lemma' of~\cite{Kelly} guarantees such a graph $G^{\rm rob}$ in any regular robustly expanding graph of linear degree.
Since $G$ is close to the disjoint union of two cliques, we are of course not in this situation. 
However, a regular almost complete graph is certainly a robust expander,
i.e.~our assumptions imply that~$G$ is close to being the disjoint union of two regular robustly expanding graphs $G_A$ and $G_B$,
with vertex sets $A$ and $B$.

So very roughly, the strategy is to apply the robust decomposition lemma of~\cite{Kelly} to $G_A$ and $G_B$ separately, 
to obtain a Hamilton decomposition of both $G_A$ and $G_B$.
Now we pair up Hamilton cycles of $G_A$ and $G_B$ in this decomposition, so that each such pair corresponds to a single 
Hamilton cycle of $G$ and so that all edges of $G$ are covered.
It turns out that we can achieve this as in the proof of Theorem~\ref{thm:NWminclique}:
we replace all edges of $G$ between $A$ and $B$  by suitable `fictive edges' in $G_A$ and $G_B$.
We then need to ensure that each Hamilton cycle in $G_A$ and $G_B$ contains a suitable set of fictive edges
-- and the set-up of the robust decomposition lemma does allow for this.

One significant difficulty compared to the proof of Theorem~\ref{thm:NWminclique}
is that this time we need a \emph{decomposition} of all the `exceptional' edges
(i.e.~those between $A$ and $B$ and those incident to the exceptional vertices)
into Hamilton exceptional systems. 
The nature of the decomposition depends on the structure of the bipartite subgraph $G[A',B']$ of $G$,
where $A'$ is obtained from $A$ by including some subset $A_0$ of the exceptional vertices, and $B'$ is obtained
from $B$ by including the remaining set $B_0$ of exceptional vertices.
We say that $G$ is `critical' if many edges of $G[A',B']$ are incident 
to very few (exceptional) vertices. In our decomposition into Hamilton exceptional systems,
we will need to distinguish between the critical and non-critical case (when in addition $G[A',B']$ contains many edges)
and the case when $G[A',B']$ contains
only a few edges. The lemmas guaranteeing this decomposition are stated and discussed in Section~\ref{sec:locES},
but their proofs are deferred until Chapter~\ref{paper4}.

Finding these localized Hamilton exceptional systems becomes more feasible if we can assume that 
there are no edges with both endpoints in the exceptional set $A_0$ or both endpoints in~$B_0$.
So in Section~\ref{sec:V_0}, we find and remove a set of edge-disjoint Hamilton cycles covering all edges in $G[A_0]$ and $G[B_0]$.
We can then find the localized Hamilton exceptional systems in Section~\ref{sec:locES}. After this, we need to extend and combine them into 
certain path systems and factors in Section~\ref{sec:SFandEF}, before we can use them as an `input' for the robust decomposition lemma in~Section~\ref{sec:robdec}.
Finally, all these steps are combined in Section~\ref{sec:1factstrong} to prove Theorem~\ref{1factstrong}.

\section{Partitions and Frameworks}\label{sec:new}

\subsection{Edges between Partition Classes}

Let $A'$, $B'$ be a partition of the vertex set of a graph $G$. 
The aim of this subsection is to give some useful bounds on the number $e_G(A',B')$ of edges between $A'$ and $B'$ in~$G$.

\begin{prop} \label{prp:e(A',B')}
Let $G$ be a graph on $n$ vertices with $\delta(G) \ge D$.
Let $A',B'$ be a partition of $V(G)$.%
    \COMMENT{Daniela: deleted "If $|A'|\ge |B'|$" in (i) since it is not needed}
Then the following properties hold:
\begin{itemize}
\item[{\rm (i)}] $e_G(A', B')  \ge  (D - |B'|+1) |B'|. $ 
\item[{\rm (ii)}] If $D \geq n - 2\lfloor n/4 \rfloor -1$, then $e_G (A',B') \geq D$ unless
$n = 0 \pmod4$, $D = n/2 -1$ and $|A'| = |B'| = n/2$.
\end{itemize} 
\end{prop}

\begin{proof}
Since $\delta(G) \ge D$ we have $d(v,A') \ge  D - |B'|+1$ for all $v \in B'$ and so $e_G(A', B') \ge  (D - |B'|+1) |B'|$,
which implies (i). (ii) follows from~\eqref{minexact} and (i).%
\COMMENT{Without loss of generality, $|A'|\ge |B'|$. $(D - |B'|+1) |B'|$ is a quadratic function in $|B'|$ which is maximized if $|B'|=(D+1)/2$. So we have to check whether
$(D - |B'|+1) |B'|\ge D$ if $|B'|=1$ and if $|B'|=\lfloor n/2\rfloor$. The former is always the case. The latter holds unless
$n = 0 \pmod4$, $D = n/2 -1$ and $|A'| = |B'| = n/2$.}
\end{proof}

\begin{prop} \label{prp:e(A',B')parity}
Let $G$ be a $D$-regular graph on $n$ vertices together with a vertex partition $A',B'$.
Then 
\begin{itemize}
\item[{\rm (i)}] $e_G(A',B')$ is odd if and only if both $|A'|$ and $D$ are odd. 
\item[{\rm (ii)}]
$e_G(A', B')   = e_{\overline{G}}(A') +e_{\overline{G}}(B') + \frac{(2D+2-n)n}{4} - \frac{(|A'|-|B'|)^2}4.$
\end{itemize}
\end{prop}
\begin{proof}
Note that $e_G(A',B') = \sum_{v \in A'} d(v,B') = \sum_{v \in A'} (D - d(v,A')) = |A'| D - 2 e_G(A')$.
Hence (i) follows.

For (ii), note that
\begin{align*}
	e_{\overline{G}}(A') & =  \binom{|A'|}2 - e_G (A') = \binom{|A'|}2 - \frac12 \left( D |A'| - e_G(A',B') \right),
\end{align*}
and similarly $	e_{\overline{G}}(B')  = \binom{|B'|}2 - \left( D|B'|  - e_G(A',B')\right)/2$.
Since $|A'|+|B'|=n$ it follows that
\begin{align*}
	e_G(A',B') & = e_{\overline{G}}(A') + e_{\overline{G}}(B') - \frac12 \left( |A'|^2 +|B'|^2 -n \left( D + 1\right) \right)\\
	& = e_{\overline{G}}(A') + e_{\overline{G}}(B') + \frac{(2D +2 - n )n}{4} - \frac{(|A'|-|B'|)^2}{4},
\end{align*}
as required.%
    \COMMENT{To see the last equality we need to check that $\frac{|A'|^2+|B'|^2}{2}=\frac{n^2}{4}+\frac{(|A'|-|B'|)^2}{4}$
which holds since $\frac{n^2}{4}=\frac{(|A'|+|B'|)^2}{4}$.}
\end{proof}

\begin{prop}  \label{prp:e(A',B')2}
Let $G$ be a $D$-regular graph on $n$ vertices with $D \ge \lfloor n/2 \rfloor$.
Let $A',B'$ be a partition of $V(G)$ with $|A'| , |B'| \ge D/2$ and $\Delta(G[A',B']) \le D/2$.
Then 
\begin{align}
e_{G - U}(A',B') \ge 
\begin{cases}
D - 28 & \textrm{if $D \ge n/2$,}\\
D/2 - 28 & \textrm{if $D = (n-1)/2$}
\end{cases}
\nonumber
\end{align}
for every $U \subseteq V(G)$ with $|U| \le 3$.
\end{prop}
\begin{proof}
Without loss of generality, we may assume that $|A'|\ge |B'|$.
Set  $G' := G[A',B']$.
If $|B'| \le D-4$, then $e(G') \ge (D - |B'|+1) |B'| \ge 5D/2$ by Proposition~\ref{prp:e(A',B')}(i).
Since $\Delta(G') \le D/2$ we have $e(G'- U) \ge e(G')  - 3 D/2 \ge D$.
Thus we may assume that $|B'| \ge D-3$. For every $v \in B'$, we have
\begin{align*}
d_{G'}(v) = d_G(v,A')  = D - d_G(v,B')
	= D - (|B'| - d_{\overline{G}}(v,B') - 1)
	\le  d_{\overline{G}}(v,B') + 4,
\end{align*}
and similarly $d_{G'}(v) \le d_{\overline{G}}(v,A') + 4$ for all $v \in A'$.
Thus
\begin{align}\label{eq:sumdegU}
\sum_{u \in U} d_{G'}(u) & \le  12+\sum_{u \in U \cap A'} d_{\overline{G}}(u,A') + \sum_{u \in U \cap B'} d_{\overline{G}}(u,B') \nonumber\\
& \le 15 + e_{\overline{G}}(A') + e_{\overline{G}}(B').
\end{align}
Note that $|A'| - |B'| \le 7$ since $|A'|\geq |B'| \geq D-3\ge \lfloor n/2\rfloor -3$.
By Proposition~\ref{prp:e(A',B')parity}(ii), we have 
\begin{eqnarray*}
	e(G' - U) & \ge & e(G') - \sum_{u \in U} d_{G'}(u) \\
	& \ge & e_{\overline{G}}(A') +e_{\overline{G}}(B') + \frac{(2D+2-n)n}{4} - \frac{(|A'|-|B'|)^2}4 - \sum_{u \in U} d_{G'}(u)\\
	& \stackrel{\eqref{eq:sumdegU}}{\ge} & \frac{(2D+2-n)n}{4} - \frac{(|A'|-|B'|)^2}4 - 15
	\ge \frac{(2D+2-n)n}{4} - 28.
\end{eqnarray*}
Hence the proposition follows.%
   \COMMENT{$\frac{(2D+2-n)n}{4}\ge D$ if $D\ge n/2$. So assume that $D=(n-1)/2$. We need to check that $\frac{(2D+2-n)n}{4}\ge D/2$.
But this holds iff $(2D+2-n)n\ge 2D$ iff $(n-1)2D\ge n(n-2)$ iff $D\ge \frac{n(n-2)}{2(n-1)}=\frac{n}{2}-\frac{n}{2(n-1)}$. So it holds
if $D=(n-1)/2$ since $\frac{n}{2(n-1)}\ge 1/2$.}
\end{proof}

The following result is an analogue of Proposition~\ref{prp:e(A',B')2} for the case when $G$ is $(n/2-1)$-regular with $n = 0 \pmod4$ and $|A'| = n/2 = |B'|$.

\begin{prop}  \label{prp:e(A',B')3}
Let $G$ be an $(n/2-1)$-regular graph on $n$ vertices with $n = 0 \pmod{4}$.
Let $A'$, $ B'$ be a partition of $V(G)$ with $|A'| = n/2 = |B'|$.
Then $$e_G(A' \setminus X,B') \ge e_G( X, B' ) - |X|(|X|-1)$$
for every vertex set $X \subseteq A'$.
Moreover, $\Delta(G[A',B']) \le e_G(A',B')/2$.
\end{prop}
\begin{proof}
For every $v \in A'$, we have
\begin{align*}
  d_G(v,B')  & = n/2 - 1 - d_G(v,A') = |A'| -1  - d_G(v,A')	=  d_{\overline{G}}(v,A') .
\end{align*}
By summing over all $v \in A'$ we obtain
\begin{align*}
e_G(A',B') & = 2e_{\overline{G}}(A')  \ge  2 \left( \sum_{ x \in X} d_{\overline{G}}(x,A') - \binom{|X|}{2} \right) \\ &
 =  2 \sum_{ x \in X} d_{G}(x,B') - |X|(|X|-1)\\
& = 2 e_G(X, B') - |X| (|X|-1).
\end{align*}
Therefore,
\begin{align}
\nonumber	e_G( A' \setminus X, B') & = e_G(A',B') -e_G(X,B')
\nonumber \ge e_G(X,B') - |X| (|X|-1).
\end{align}
In particular, this implies that for each vertex $x \in A'$ we have $e_G( A' \setminus \{x\}, B') \ge e_G(\{x\},B') = d_G(x,B')$ and so $2 d_G(x,B') \le e_G( A', B')$.
By symmetry, for any $y \in B'$ we have $2 d(y,A') \le e_G( A', B')$.
Therefore, $\Delta( G[A',B']) \le e_G( A', B')/2$.
\end{proof}


\subsection{Frameworks}

Throughout this chapter, we will consider partitions into sets $A$ and $B$ of equal size (which induce `near-cliques')
as well as `exceptional sets' $A_0$ and $B_0$. The following definition formalizes this.
Given a graph $G$, we say that $(G,A,A_0,B,B_0)$ is an \emph{$(\epszero,K)$-framework} if the following holds,
where $A':=A_0\cup A$, $B':=B_0\cup B$ and $n:=|V(G)|$: 
\begin{itemize}
\item[{\rm (FR1)}] $A,A_0,B,B_0$ forms a partition of $V(G)$. 
\item[{\rm (FR2)}] $e(A',B')\le \epszero n^2$.
\item[{\rm (FR3)}] $|A|=|B|$ is divisible by $K$, $|A_0| \ge |B_0|$ and $|A_0| + |B_0| \le \epszero n$.
\item[{\rm (FR4)}] If $v\in A$ then $d(v, B')< \epszero n$ and if $v\in B$ then $d(v, A')< \epszero n$.\COMMENT{Deryk}
\end{itemize}
We often write $V_0$ for $A_0\cup B_0$ and think of the vertices in $V_0$ as `exceptional vertices'.
Also, whenever $(G,A,A_0,B,B_0)$ is an $(\epszero,K)$-framework, we will write $A':=A_0\cup A$, $B':=B_0\cup B$.

\begin{prop} \label{prop:framework}
Let $0 < 1/n \ll \eps_{\rm ex} ,1/K \ll 1$ and $\eps_{\rm ex} \ll \epszero \ll 1$.
Let $G$ be a graph on $n$ vertices with $\delta(G) = D \ge n - 2\lfloor n/4 \rfloor -1$ that is $\eps_{\rm ex} $-close to the union of two disjoint copies of $K_{n/2}$.
Then there is a partition $A,A_0,B,B_0$ of $V(G)$ such that $(G,A,A_0,B,B_0)$ is an $(\epszero, K)$-framework, $d(v, A') \ge d(v)/2$ for all $v\in A'$ and
$d(v, B') \ge d(v)/2$ for all $v\in B'$.
\end{prop}
\begin{proof}
Write $\eps: = \eps_{\rm ex}$.
Since $G$ is $\eps $-close to the union of two disjoint copies of $K_{n/2}$, there exists a partition $A'', B''$ of $V(G)$ such that
$|A''|=\lfloor n/2\rfloor$ and $e(A'',B'') \le \eps n^2$.
If there exists a vertex $v \in A''$ such that $d(v,A'') < d(v,B'')$, then we move $v$ to $B''$. We still denote the
vertex classes thus obtained by $A''$ and $B''$.
Similarly, if there exists a vertex $v \in B''$ such that $d(v,B'') < d(v,A'')$, then we move $v$ to $A''$.
We repeat this process until $d(v,A'') \ge d(v,B'')$ for all $v \in A''$ and $d(v,B'') \ge d(v,A'')$ for all $v \in B''$.
Note that this process must terminate since at each step the value of $e(A'',B'')$ decreases.
Let $A',B'$ denote the resulting partition. By relabeling the classes if necessary we may assume that $|A'| \ge |B'|$.
By construction, $e(A',B') \le e(A'',B'') \le \eps n^2$ and so (FR2) holds. 
Suppose that $|B'| < (1- 5\eps)n/2$. Then at some stage in the process we have that $|B''|=(1- 5 \eps)n/2$.
But then by Proposition~\ref{prp:e(A',B')}(i), 
\begin{align*}
	 e(A'',B'') \ge (D - |B''|+1 ) |B''| > \eps n^2,
\end{align*}
a contradiction to the definition of $\eps$-closeness (as the number of edges between the partition classes has 
not increased while moving the vertices). Hence, $|A'|\ge |B'| \ge (1- 5\eps)n/2$. Let $B'_0$ be the set of vertices $v$ in $B'$ such that
$d(v,A') \ge \sqrt{\eps} n$. Since $\sqrt{\eps} n |B'_0|  \le e(A',B') \le \eps n^2$ we have $|B'_0| \le \sqrt{\eps}n$.
Note that\COMMENT{Deryk} 
\begin{align}\label{eq:sizeB'minus}
|B'| - |B'_0| \ge ( 1- 5 \eps ) n /2 - \sqrt{\eps } n \ge (1- 3\sqrt{ \eps} )n/2.
\end{align}
Similarly, let $A'_0$ be the set of vertices $v$ in $A'$ such that $d(v,B') \ge \sqrt{\eps} n $.
Thus, $|A'_0|  \le \sqrt{\eps }n$ and $|A'| - |A'_0| \ge n/2-|A'_0| \ge (1-2\sqrt{ \eps} )n/2$.
Let $m$ be the largest integer such that $Km \le |A'| - |A'_0|, |B'| - |B'_0|$.
Let $A$ and $B$ be $Km$-subsets of $A' \setminus A'_0$ and $B' \setminus B'_0$ respectively.
Set $A_0 := A' \setminus A$ and $B_0 := B' \setminus B$.
Note that (\ref{eq:sizeB'minus}) and its analogue for~$A'$ together imply that%
   \COMMENT{Daniela: changed the explanation here}
$|A_0| + |B_0| \le 3 \sqrt{\eps}n +2K\le \eps_0 n$.
Therefore, $(G, A,A_0,B,B_0)$ is an $(\epszero,K)$-framework.
\end{proof}


\section{Exceptional Systems and $(K,m,\eps_0)$-Partitions}\label{sec:BES}

The definitions and observations in this section will enable us to `reduce' the problem of 
finding Hamilton cycles in $G$ to that of finding suitable pairs $C_A$, $C_B$ of cycles with $V(C_A)=A$ and $V(C_B)=B$.
In particular, they will enable us to `ignore' the exceptional set $V_0=A_0 \cup B_0$.
Roughly speaking, for each Hamilton cycle we seek, we find a certain path system $J$  covering $V_0$
(called an exceptional system). From this, we derive a set $J^*$ of edges whose endvertices lie in $A \cup B$
by replacing paths of $J$ with `fictive edges' in a suitable way. We can then work with $J^*$ instead of $J$
when constructing our Hamilton cycles (see Proposition~\ref{prop:ES} and the explanation preceding it).

Suppose that $A,A_0,B,B_0$ forms a partition of a vertex set $V$ of size $n$ such that $|A| = |B|$. Let $V_0:=A_0\cup B_0$.
An \emph{exceptional cover} $J$ is a graph which satisfies the following properties:
\begin{enumerate}[label={(EC{\arabic*})}]
\item $J$ is a path system with $V_0\subseteq V(J)\subseteq V$.%
   \COMMENT{One might set $V(J)=V$ instead. Or alternatively, $d_J(v)= 1$ for every $v \in V(J) \setminus V_0$.
However, this does not fit well with the bipartite case.}
\item $d_J(v) =2 $ for every $v \in V_0$ and $d_J(v) \le 1$ for every $v \in V(J) \setminus V_0$.
\item $e_J(A), e_J(B) = 0$.
\end{enumerate}
We say that $J$ is an \emph{exceptional system with parameter~$\eps_0$}, or an \emph{ES} for short, if $J$ satisfies the following properties:
\begin{enumerate}[label={(ES{\arabic*})}]
	\item $J$ is an exceptional cover.
	\item One of the following is satisfied:
	\begin{itemize}
	\item[(HES)] The number of $AB$-paths in $J$ is even and positive. In this case we say $J$ is a \emph{Hamilton exceptional system}, or \emph{HES} for short.
	\item[(MES)] $e_J(A',B')=0$. In this case we say $J$ is a \emph{matching exceptional system}, or \emph{MES} for short. 
\end{itemize}
	\item $J$ contains at most $\sqrt{\epszero} n $ $AB$-paths.
\end{enumerate}
Note that by definition, every $AB$-path in $J$ is maximal. 
So the number of $AB$-paths in $J$ is the number of genuine `connections' between $A$ and $B$
(and thus between $A'$ and $B'$).
If we want to extend $J$ into a Hamilton cycle using only edges induced by $A$ and edges induced by~$B$,
this number clearly has to be even and positive.
Hamilton exceptional systems will always be extended into Hamilton cycles and matching exceptional systems will
always be extended into two disjoint even cycles which together span all vertices (and thus consist of two edge-disjoint perfect matchings).

Since each maximal path in $J$ has endpoints in $A \cup B$ and internal vertices in $V_0$, an exceptional system $J$
naturally induces a matching $J^*_{AB}$ on $A \cup B$.
More precisely, if $P_1, \dots ,P_{\ell'}$ are the non-trivial paths%
   \COMMENT{Daniela: need non-trivial here (don't want to have loops for the isolated vertices in $J$). TO DO: have to change this
in paper~2 (I've already changed it in paper 3)}
in~$J$ and $x_i, y_i$ are the endpoints of $P_i$, then
we define $J^*_{AB} := \{x_iy_i : i  \le \ell'\}$. Thus $e_{J^*_{AB}}(A,B)$ is equal to the number of $AB$-paths in $J$.
In particular, if $J$ is a matching exceptional system,%
\COMMENT{The def of a MES had to change. The previous one was $e_{J^*_{AB}}(A,B)=0$.
But then Prop\ref{prop:ES}(ii) does not hold.}
then $e_{J^*_{AB}}(A,B)=0$.

Let $x_1y_1, \dots, x_{2\ell}y_{2\ell}$ be a fixed enumeration of the edges of $J^*_{AB}[A,B]$ with $x_i \in A$ and $y_i \in B$.
Define 
$$
J_A^* := J^*_{AB}[A] \cup \{x_{2i-1} x_{2i} :1 \le i \le \ell \}  \ \mbox{ and } \ J_B^* := J^*_{AB}[B] \cup \{y_{2i} y_{2i+1} :1 \le i \le \ell \}$$ 
(with indices considered modulo $2\ell$).
Let $J^* := J_A^* + J_B^*$, see Figure~\ref{fig2}. Note that $J^*$ is the union of one matching induced by $A$ and another on $B$,
and $e(J^*) = e(J^*_{AB})$.%
\COMMENT{Deryk: new first part of sentence, ref to figure and additional inequality for later reference.}
Moreover, by (EC2) we have
\begin{equation}
	e(J^*) = e(J^*_{AB}) \le |V_0| + e_J(A',B') \le 2 \sqrt{\eps_0} n. \label{ESeq}
\end{equation}
We will call the edges in $J^*$ \emph{fictive} edges. Note that if $J_1$ and $J_2$ are two edge-disjoint exceptional systems, then $J^*_1$
and $J^*_2$ may not be edge-disjoint. However, we will always view fictive edges as being distinct from each other and from the edges
in other graphs. So in particular, whenever $J_1$ and $J_2$ are two exceptional systems, we will view $J^*_1$
and $J^*_2$ as being edge-disjoint.

We say that a path $P$ is \emph{consistent with $J_A^*$} if $P$ contains $J_A^*$ and (there is an orientation of $P$ which)
visits the vertices $x_1, \dots ,x_{2\ell}$ in this order.%
    \COMMENT{We need to prescribe a vertex rather than an edge ordering as it is important in Proposition~\ref{prop:ES}.}
A path $P$ is \emph{consistent with $J_B^*$} if $P$ contains $J_B^*$ and visits the vertices $y_2, \dots ,y_{2\ell},y_1$ in this order.%
   \COMMENT{Note that the order is $y_2,y_3, \dots ,y_{2\ell}, y_1$ rather than $y_1, \dots ,y_{2\ell}$.}
In a similar way we define when a cycle is consistent with $J_A^*$ or $J_B^*$.

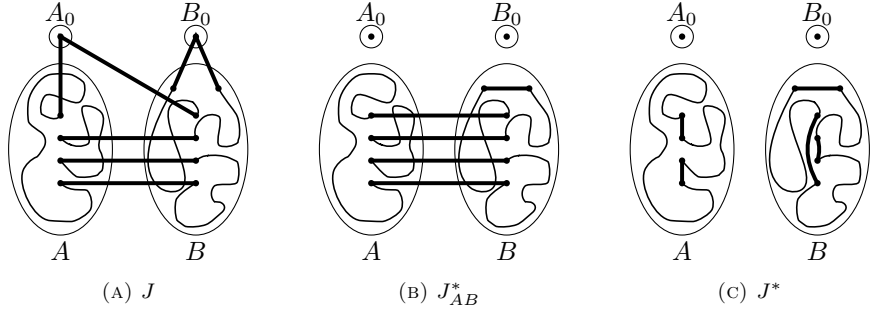
\begin{figure}[tbp]
\centering
\subfloat[$J$]{
\begin{tikzpicture}[scale=0.30]
			\draw  (-3,0) ellipse (2.3  and 3.8 );
			\draw (3,0) ellipse (2.3  and 3.8 );
			\node at (-3,-4.5) {$A$};
			\node at (3,-4.5) {$B$}; 			
			\draw (-3,5) circle(0.5);
			\draw (3,5) circle (0.5);
			\node at (-3,6) {$A_0$};
			\node at (3,6) {$B_0$};
			\fill (-3,5) circle (4pt);
			\fill (3,5) circle (4pt);
			\begin{scope}[start chain]
			\foreach \i in {1.5,0.5,...,-1.5}
			\fill (-3,\i) circle (4pt);
			\end{scope}
			\begin{scope}[start chain]
			\foreach \i in {1.5,0.5,...,-1.5}
			\fill (3,\i) circle (4pt);
			\end{scope}
			\fill (2,2.7) circle (4pt);
			\fill (4,2.7) circle (4pt);

			\begin{scope}[line width=1.5pt]
			\draw (-3,1.5)--(-3,5)--(3,1.5);
			\draw (2,2.7)--(3,5)--(4,2.7);
				\begin{scope}[start chain]
				\foreach \i in {0.5,-0.5,-1.5}
				\draw (-3,\i)--(3,\i);
				\end{scope}
			\end{scope}
			\begin{scope}[line width=0.6pt]
			\draw[rounded corners]  (-3,1.5) to [out=235, in=-70] (-4,2) to [out=50, in=-130] (-1.8,2.5) to [out=110, in=-10] (-3,3.4) to [out=-170, in=80] (-4.5,1.7) to [out=-80, in=145] (-3.5,0.5) to [out=-120, in=20] (-4.5,-0.5) 
 to [out=-110, in=175]  (-3,-3.2) to [out=0, in= -110] (-1.5,-2.2) to [out=130, in=-45] (-3,-1.5);		
 			\draw[rounded corners]  (-3,0.5) to [out=-45, in=150] (-2.2,0) to [out=30, in=-170] (-1.5,2) to [out=-50, in= 120] (-1,-0.5) to [out=-110, in=20] (-1.8,-1.3) to [out=170, in= -50] (-3,-0.5);
			 \draw[rounded corners]  (3,-0.5) to [out=40, in=-170] (4,-0.2) to [out=-20, in= 110] (5,-1) to [out=-110, in= 50] (3.5,-2) to [out=-45, in= 160] (4.5, -2.5) to [out=-140, in= 20] (3,-3.5) to [out=170, in= -70] (2,-2.7) to [out=75, in= -150] (3,-1.5);
			\draw[rounded corners] (2,2.7) to [out=-120, in= 140]  (1.75,-2) to [out=20, in= -155]  (2.5,2.3) to [out=-30, in= 30] (3,1.5);
			\draw[rounded corners] (4,2.7)  to [out=-60, in= 110]  (5,1)  to [out=-100, in= 40]  (4.5,0) to [out=160, in=-30] (3.6,1.7) to [out=-135, in=110]  (3,0.5);
			\end{scope}
\end{tikzpicture}
}
\qquad
\subfloat[$J_{AB}^*$]{
\begin{tikzpicture}[scale=0.30]
			\draw  (-3,0) ellipse (2.3  and 3.8 );
			\draw (3,0) ellipse (2.3  and 3.8 );
			\node at (-3,-4.5) {$A$};
			\node at (3,-4.5) {$B$}; 			
			\draw (-3,5) circle(0.5);
			\draw (3,5) circle (0.5);
			\node at (-3,6) {$A_0$};
			\node at (3,6) {$B_0$};
			\fill (-3,5) circle (4pt);
			\fill (3,5) circle (4pt);
			\begin{scope}[start chain]
			\foreach \i in {1.5,0.5,...,-1.5}
			\fill (-3,\i) circle (4pt);
			\end{scope}
			\begin{scope}[start chain]
			\foreach \i in {1.5,0.5,...,-1.5}
			\fill (3,\i) circle (4pt);
			\end{scope}
			\fill (2,2.7) circle (4pt);
			\fill (4,2.7) circle (4pt);

			\begin{scope}[line width=1.5pt]
			\draw (-3,1.5)--(3,1.5);
			\draw (2,2.7)--(4,2.7);
				\begin{scope}[start chain]
				\foreach \i in {0.5,-0.5,-1.5}
				\draw (-3,\i)--(3,\i);
				\end{scope}
			\end{scope}

			\begin{scope}[line width=0.6pt]
			\draw[rounded corners]  (-3,1.5) to [out=235, in=-70] (-4,2) to [out=50, in=-130] (-1.8,2.5) to [out=110, in=-10] (-3,3.4) to [out=-170, in=80] (-4.5,1.7) to [out=-80, in=145] (-3.5,0.5) to [out=-120, in=20] (-4.5,-0.5) 
 to [out=-110, in=175]  (-3,-3.2) to [out=0, in= -110] (-1.5,-2.2) to [out=130, in=-45] (-3,-1.5);		
 			\draw[rounded corners]  (-3,0.5) to [out=-45, in=150] (-2.2,0) to [out=30, in=-170] (-1.5,2) to [out=-50, in= 120] (-1,-0.5) to [out=-110, in=20] (-1.8,-1.3) to [out=170, in= -50] (-3,-0.5);
			 \draw[rounded corners]  (3,-0.5) to [out=40, in=-170] (4,-0.2) to [out=-20, in= 110] (5,-1) to [out=-110, in= 50] (3.5,-2) to [out=-45, in= 160] (4.5, -2.5) to [out=-140, in= 20] (3,-3.5) to [out=170, in= -70] (2,-2.7) to [out=75, in= -150] (3,-1.5);
			\draw[rounded corners] (2,2.7) to [out=-120, in= 140]  (1.75,-2) to [out=20, in= -155]  (2.5,2.3) to [out=-30, in= 30] (3,1.5);
			\draw[rounded corners] (4,2.7)  to [out=-60, in= 110]  (5,1)  to [out=-100, in= 40]  (4.5,0) to [out=160, in=-30] (3.6,1.7) to [out=-135, in=110]  (3,0.5);
				\end{scope}
\end{tikzpicture}
}
\qquad
\subfloat[$J^*$]{
\begin{tikzpicture}[scale=0.30]
			\draw  (-3,0) ellipse (2.3  and 3.8 );
			\draw (3,0) ellipse (2.3  and 3.8 );
			\node at (-3,-4.5) {$A$};
			\node at (3,-4.5) {$B$}; 			
			\draw (-3,5) circle(0.5);
			\draw (3,5) circle (0.5);
			\node at (-3,6) {$A_0$};
			\node at (3,6) {$B_0$};
			\fill (-3,5) circle (4pt);
			\fill (3,5) circle (4pt);
			\begin{scope}[start chain]
			\foreach \i in {1.5,0.5,...,-1.5}
			\fill (-3,\i) circle (4pt);
			\end{scope}
			\begin{scope}[start chain]
			\foreach \i in {1.5,0.5,...,-1.5}
			\fill (3,\i) circle (4pt);
			\end{scope}
			\fill (2,2.7) circle (4pt);
			\fill (4,2.7) circle (4pt);

			\begin{scope}[line width=1.5pt]
			\draw (-3,1.5)--(-3,0.5);
			\draw (-3,-1.5)--(-3,-0.5);
			\draw (2,2.7)--(4,2.7);
			\draw (3,0.5)to[out=-75, in=75](3,-0.5);
			\draw (3,1.5)to[out=-120, in=120](3,-1.5);
			\end{scope}

			\begin{scope}[line width=0.6pt]		
			\draw[rounded corners]  (-3,1.5) to [out=235, in=-70] (-4,2) to [out=50, in=-130] (-1.8,2.5) to [out=110, in=-10] (-3,3.4) to [out=-170, in=80] (-4.5,1.7) to [out=-80, in=145] (-3.5,0.5) to [out=-120, in=20] (-4.5,-0.5) 
 to [out=-110, in=175]  (-3,-3.2) to [out=0, in= -110] (-1.5,-2.2) to [out=130, in=-45] (-3,-1.5);		
 			\draw[rounded corners]  (-3,0.5) to [out=-45, in=150] (-2.2,0) to [out=30, in=-170] (-1.5,2) to [out=-50, in= 120] (-1,-0.5) to [out=-110, in=20] (-1.8,-1.3) to [out=170, in= -50] (-3,-0.5);
			 \draw[rounded corners]  (3,-0.5) to [out=40, in=-170] (4,-0.2) to [out=-20, in= 110] (5,-1) to [out=-110, in= 50] (3.5,-2) to [out=-45, in= 160] (4.5, -2.5) to [out=-140, in= 20] (3,-3.5) to [out=170, in= -70] (2,-2.7) to [out=75, in= -150] (3,-1.5);
			\draw[rounded corners] (2,2.7) to [out=-120, in= 140]  (1.75,-2) to [out=20, in= -155]  (2.5,2.3) to [out=-30, in= 30] (3,1.5);
			\draw[rounded corners] (4,2.7)  to [out=-60, in= 110]  (5,1)  to [out=-100, in= 40]  (4.5,0) to [out=160, in=-30] (3.6,1.7) to [out=-135, in=110]  (3,0.5);
			\end{scope}
\end{tikzpicture}
}
\caption{The thick lines illustrate the edges of $J$, $J_{AB}^*$ and $J^*$ respectively.}
\label{fig2}
\end{figure}

The next result shows that if $J$ is a Hamilton exceptional system and $C_A, C_B$ are two Hamilton cycles on
$A$ and $B$ respectively which are consistent with $J^*_A$ and $J^*_B$, then the graph obtained from $C_A+ C_B$ by replacing $J^*=J^*_A + J^*_B$
with $J$ is a Hamilton cycle on $V$ which contains~$J$, see Figure~\ref{fig2}.
When choosing our Hamilton cycles, this property will enable us ignore all the vertices in $V_0$ and to consider
the (almost complete) graphs induced by $A$ and by $B$ instead. 
Similarly, if $J$ is a matching exceptional system and both $|A'|$ and $|B'|$ are even,
then the graph obtained from $C_A+ C_B$ by replacing $J^*$ with $J$ is the edge-disjoint union of two perfect matchings on~$V$.

\begin{prop} \label{prop:ES}
Suppose that $A,A_0,B,B_0$ forms a partition of a vertex set $V$.
Let $J$ be an exceptional system. Let $C_A$ and $C_B$ be two cycles such that 
\begin{itemize}
	\item $C_A$ is a Hamilton cycle on $A$ that is consistent with $J_A^*$;
	\item $C_B$ is a Hamilton cycle on $B$ that is consistent with $J_B^*$.
\end{itemize}
Then the following assertions hold. 
\begin{itemize}
	\item[\rm (i)] If $J$ is a Hamilton exceptional system, then $C_A+C_B - J^* +J$ is a Hamilton cycle on $V$.
	\item[\rm (ii)] If $J$ is a  matching exceptional system, then $C_A+C_B - J^* +J$ is the union of a Hamilton cycle on $A'$ and a Hamilton cycle on $B'$.
In particular, if both $|A'|$ and $|B'|$ are even, then $C_A+C_B - J^* +J$ is the union of two edge-disjoint perfect matchings on $V$.
\end{itemize}
\end{prop}
\begin{proof}
Suppose that $J$ is a Hamilton exceptional system.
Let $x_1y_1, \dots, x_{2\ell}y_{2\ell}$ be an enumeration of the edges of $J^*_{AB}[A,B]$ with $x_i \in A$ and $y_i \in B$ and such that
$J_A^* = J^*_{AB}[A] \cup \{x_{2i-1} x_{2i} :1 \le i \le \ell \}$ and $J_B^* = J^*_{AB}[B] \cup \{y_{2i} y_{2i+1} :1 \le i \le \ell \}$.
Let $P_1^A, \dots ,P_\ell^A$ be the paths in $C_A - \{x_{2i-1} x_{2i} : 1 \le i \le \ell \}$.
Since $C_A$ is consistent with $J_A^*$, we may assume that $P_i^A$ is a path from $x_{2i-2}$ to $x_{2i-1}$ for all $i \le \ell$.
Similarly, let $P_1^B, \dots ,P_\ell^B$ be the paths in $C_B - \{y_{2i} y_{2i+1} : 1 \le i \le \ell \}$.
Again, we may assume that $P_i^B$ is a path from $y_{2i-1}$ to $y_{2i}$ for all $i \le \ell$.
Define $C^*$ to be the $2$-regular graph on $A \cup B$ obtained from concatenating $P^A_1,x_1y_1,P^B_1,y_2x_2,P^A_2,x_3y_3, \dots ,P^B_\ell$ and $y_{2\ell}x_{2\ell}$.
Together with (HES), the construction implies that $C^*$ is a Hamilton cycle on $A\cup B$ and $C^*=C_A+C_B - J^*+J^*_{AB}$.
Thus $C:=C^* - J^*_{AB}+J$ is a Hamilton cycle on~$V$. Since $C=C_A+C_B - J^* +J$, (i) holds.

The proof of (ii) is similar to that of (i). Indeed, the previous argument shows that $C^*$ is the union of a Hamilton cycle on $A$ and a Hamilton cycle on $B$. (MES) now implies that $C$ is the union of a Hamilton cycle on $A'$ and one on $B'$.
\end{proof}

In general, we construct an exceptional system by first choosing an exceptional system candidate (defined below) and then extending it to an
exceptional system. More precisely, suppose that $A,A_0,B,B_0$ forms a partition of a vertex set $V$.
Let $V_0:=A_0\cup B_0$. A graph $F$ is called an \emph{exceptional system candidate with parameter
$\eps_0$}, or an \emph{ESC} for short, if $F$ satisfies the following properties:
\begin{itemize}
	\item[(ESC1)] $F$ is a path system with $V_0\subseteq V(F)\subseteq V$ and such that $e_F(A), e_F(B) = 0$.
	\item[(ESC2)] $d_F(v) \le 2$ for all $v \in V_0$ and $d_F(v) = 1$ for all $v \in V(F) \setminus V_0$.
	\item[(ESC3)] $e_F(A',B') \le \sqrt{ \epszero}  n/2$. In particular, $|V(F) \cap A|, |V(F) \cap B| \le 2|V_0| + \sqrt{ \epszero } n/2$.
	\item[(ESC4)] One of the following holds:
\begin{itemize}
\item[(HESC)] Let $b(F)$ be the number of maximal paths in $F$ with one endpoint in $A'$ and the other in $B'$. Then $b(F)$ is even
and $b(F)>0$. In this case we say that $F$ is a \emph{Hamilton exceptional system candidate}, or \emph{HESC} for short.
\item[(MESC)] $e_F(A',B') =0$. In this case, $F$ is called a \emph{matching exceptional system candidate} or \emph{MESC} for short.
\end{itemize}
\end{itemize}
Note that if $d_F(v) = 2$ for all $v \in V_0$, then $F$ is an exceptional system. 
Also, if $F$ is a Hamilton exceptional system candidate with $e(F) =2$, then $F$ consists of two independent $A'B'$-edges.
Moreover, note that (EC2) allows an exceptional cover~$J$
(and so also an exceptional system~$J$) to contain vertices in $A\cup B$ which are isolated in~$J$. However, (ESC2) does not allow for this
in an exceptional system candidate~$F$.%
    \COMMENT{The latter is quite handy since it allows us to write
$|V(F) \cap A|$ for the number of vertices in $A$ which a incident to an edge of $F$. The former ensures that things are the same as in the
bipartite case for (balanced) exceptional systems and so in the approx paper we don't use two different versions...}

Similarly to condition (HES), in (HESC) the parameter $b(F)$ counts the number of `connections' between $A'$ and $B'$.
In order to extend a Hamilton exceptional system candidate into a Hamilton cycle without using any additional $A'B'$-edges, it is clearly necessary that $b(F)$
is positive and even.

The next result shows that we can extend an exceptional system candidate into an exceptional system by adding suitable $A_0A$- and $B_0B$-edges.
In the proof of Lemma~\ref{V_0BES} we will use that if $G$ is a $D$-regular graph with $D\ge n/100$ (say) and  $(G,A,A_0,B,B_0)$ is an $(\epszero,K)$-framework
with $\Delta(G[A',B'])\le D/2$, then conditions (i) and~(ii) below are satisfied.
\begin{lemma} \label{lma:ESextend}
Suppose that $ 0 < 1/n \ll \epszero \ll 1$ and that $n \in \mathbb{N}$.
Let $G$ be a graph on $n$ vertices so that%
\COMMENT{previously, we said that $G$ is a framework instead of the 1stcondition below. But then we have to define a framework in paper4, which
we don't want to do.}
\begin{itemize}
\item[{\rm (i)}] $A,A_0,B,B_0$ forms a partition of $V(G)$ with $|A_0 \cup B_0| \le \eps_0 n$. 
\item[{\rm (ii)}] $d(v,A) \ge \sqrt{\epszero} n$ for all $v \in A_0$ and $d(v,B) \ge \sqrt{\epszero} n$ for all $v \in B_0$.
\end{itemize}
Let $F$ be an exceptional system candidate with parameter $\eps_0$. 
Then there exists an exceptional system $J$ with parameter $\eps_0$ such that $F\subseteq J\subseteq G+ F$
and such that every edge of $J-F$ lies in $G[A_0,A]+G[B_0,B]$.
Moreover, if $F$ is a Hamilton exceptional system candidate, then $J$ is a Hamilton exceptional system.
Otherwise $J$ is a matching exceptional system.
\end{lemma}
\begin{proof}
For each vertex $v \in A_0$, we select $2-d_{F}(v)$ edges $uv$ in $G$ with $u \in A \setminus V(F)$.
Since $ d_G(v,A) \ge \sqrt{\epszero} n \ge |V(F) \cap A| + 2|V_0|$ by (ESC3), these edges can be chosen such that they have no common endpoint in $A$.
Similarly, for each vertex $v \in B_0$, we select $2-d_{F}(v)$ edges $uv$ in $G$ with $u \in B \setminus V(F)$.
Again, these edges are chosen such that they have no common endpoint in $B$. Let $J$ be the graph obtained from $F$ by adding all these edges. 
Note that $J$ is an exceptional cover such that every edge of $J-F$ lies in $G[A_0,A]+G[B_0,B]$.
Furthermore, the number of $AB$-paths in $J$ is at most $e_F(A',B') \le \sqrt{\epszero} n/2$.

Suppose $F$ is a Hamilton exceptional system candidate with parameter $\eps_0$.
Our construction of $J$ implies that the number of $AB$-paths in $J$  equals $b(F)$. So (HES) follows from (HESC).
Now suppose $F$ is a matching exceptional system candidate. Then (MES) is satisfied since $e_J(A',B')=e_F(A',B')=0$ by (MESC). This proves the lemma.
\end{proof}

Let $K,m\in\mathbb{N}$ and $\eps_0>0$.
A \emph{$(K,m,\eps_0)$-partition $\mathcal{P}$} of a set $V$ of vertices is a partition of $V$ into sets $A_0,A_1,\dots,A_K$
and $B_0,B_1,\dots,B_K$ such that $|A_i|=|B_i|=m$ for all $i\ge 1$ and $|A_0\cup B_0|\le \eps_0 |V|$.
The sets $A_1,\dots,A_K$ and $B_1,\dots,B_K$ are called \emph{clusters} of $\mathcal{P}$
and $A_0$, $B_0$ are called \emph{exceptional sets}. We often write $V_0$ for $A_0\cup B_0$ and think of the
vertices in $V_0$ as `exceptional vertices'. Unless stated otherwise, whenever $\mathcal{P}$ is a $(K,m,\eps_0)$-partition,
we will denote the clusters by $A_1,\dots,A_K$ and $B_1,\dots,B_K$ and the exceptional sets by $A_0$ and $B_0$.
We will also write $A:=A_1\cup\dots\cup A_K$, $B:=B_1\cup\dots\cup B_K$, $A':=A_0\cup A_1\cup\dots\cup A_K$
and $B':=B_0\cup B_1\cup\dots\cup B_K$.

Given a $(K,m, \epszero)$-partition $\mathcal{P}$ and $1\le i,i' \le K$, we say that $J$ is an \emph{$ (i,i')$-localized Hamilton exceptional system}
(abbreviated as \emph{$(i,i')$-HES}) if $J$ is a Hamilton exceptional system and $V(J)\subseteq V_0 \cup A_{i} \cup B_{i'}$.
In a similar way, we define
\begin{itemize}
\item \emph{$ (i,i')$-localized matching exceptional systems} (\emph{$(i,i')$-MES}),
\item \emph{$ (i,i')$-localized exceptional systems} (\emph{$(i,i')$-ES}),
\item \emph{$ (i,i')$-localized Hamilton exceptional system candidates} (\emph{$(i,i')$-HESC}),
\item \emph{$ (i,i')$-localized matching exceptional system candidates} (\emph{$(i,i')$-MESC}),
\item \emph{$ (i,i')$-localized exceptional system candidates} (\emph{$(i,i')$-ESC}). 
\end{itemize}
To make clear with which partition we are working, we sometimes also say that
$J$ is an $ (i,i')$-localized Hamilton exceptional system with respect to $\cP$ etc.


\section{Schemes and Exceptional Schemes} \label{sec:schemes}

It will often be convenient to consider the `exceptional' and `non-exceptional' part of a graph $G$ separately.
For this, we introduce a `scheme' (which corresponds to the non-exceptional part and also incorporates a refined partition of~$G$) 
and an `exceptional scheme' (which corresponds to the exceptional part and also incorporates a refined partition of~$G$).

Given a graph $G$ and a partition $\mathcal P$ of a vertex set $V$, we call $(G, \mathcal P)$ a
\emph{$(K,m,\eps _0, \break \eps )$-scheme} if the following properties hold:
\begin{enumerate}[label={(Sch{\arabic*})}]
	\item  $\mathcal P$ is a $(K,m,\eps _0)$-partition of $V$.
	\item $V(G)=A\cup B$ and $e_G(A,B)=0$.
	\item For all $1\le i \le K$ and  all $v \in A$ we have $d(v,A_i) \ge (1- \eps) m$. Similarly, for all $1\le i \le K$ and  all $v \in B$
we have $d(v,B_i) \ge (1- \eps) m$.%
	\COMMENT{Replaced superregularity with the degree condition. This implies that $G[A_i,A_{i'}]$ and $G[B_i,B_{i'}]$ are
$(\sqrt{\eps}, 1)$-superregular for all $i, i' \le K$ such that $i \ne i'$.}
\end{enumerate}

The next proposition shows that if $(G, \mathcal{P})$ is a scheme and $G'$ is obtained from $G$ by removing a small number of edges
at each vertex, then $(G', \mathcal P)$ is also a scheme with slightly worse parameters. Its proof is immediate from
the definition of a scheme.
 
\begin{prop} \label{deleteBS}
Suppose that $0 <1/m \ll\eps, \eps' \ll 1$ and that $K,m\in \mathbb{N}$.
Let $(G, \mathcal{P})$ be a $(K,  m, \epszero,\eps)$-scheme. Let $G'$ be a spanning subgraph of $G$ such that $\Delta(G- G') \le\eps' m$.
Then $(G', \mathcal P)$ is a $(K, m, \epszero, \eps + \eps')$-scheme.
\end{prop}

Given a graph $G$ on $n$ vertices and a partition $\mathcal P$ of $V(G)$ we call
$(G, \mathcal{P})$ a \emph{$(K,  m, \epszero,\eps)$-exceptional scheme} if the following properties are satisfied:
\begin{enumerate}[label={(ESch{\arabic*})}]
	\item $\mathcal P$ is a $(K, m, \epszero)$-partition of $V(G)$.
	\item $e(A),e(B) = 0$.
	\item If $v\in A$ then $d(v, B')< \epszero n$ and if $v\in B$ then $d(v, A')< \epszero n$.%
	\COMMENT{Previously had "If $d(v,A'), d(v,B') \ge \epszero n$, then $v \in V_0$". But this doesn't seem to be strong enough.
Also, note that there is no condition on $\Delta(G[A',B'])$ in the definition of exceptional scheme.}
	\item For all $v \in V(G)$ and all $1\le i \le K$ we have
    $d(v, A_i) = ( d(v,A) \pm\eps n ) / K$ and $d(v,B_i) = ( d(v,B) \pm\eps n ) / K$.
	\item For all $1\le i,i' \le K$ we have
	\begin{align*}
	&e(A_0,A_i)  = ( e(A_0,A) \pm \eps \max \{ e(A_0,A) , n \}     )/K ,\\
   &e(B_0,A_i)  = ( e(B_0,A) \pm \eps \max \{ e(B_0,A) , n \}     )/K ,\\
	&e(A_0,B_i)  = ( e(A_0,B) \pm \eps \max \{ e(A_0,B) , n \}     )/K ,\\
	&e(B_0,B_i)  = ( e(B_0,B) \pm \eps \max \{ e(B_0,B) , n \}     )/K ,\\
	&e(A_i,B_{i'})  = ( e(A,B) \pm \eps \max \{ e(A,B) , n \}     )/K^2 .
	\end{align*}
\end{enumerate}

Suppose that $(G,A,A_0,B,B_0)$ is an $(\epszero, K)$-framework.
The next lemma shows that there is a refinement of the vertex partition $A,A_0,B,B_0$ of $V(G)$ into a $(K, m, \epszero)$-partition $\mathcal{P}$
such that $(G[A] +G[B], \mathcal{P})$ is a scheme and $(G - G[A]  - G[B], \mathcal{P})$ is an exceptional scheme.

\begin{lemma} \label{lma:partition}
Suppose that  $0 <  1/n \ll \epszero \ll 1/K \ll 1$, that $ \eps_0 \ll \eps_1\le \eps_2\ll 1$, that $1/n \ll \mu\ll \eps_2$
and that $n,K,m \in \mathbb N$. Let $G$ be a graph on $n$ vertices such that $\delta(G) \ge (1- \mu) n/2$.
Let $(G,A,A_0,B,B_0)$ be an $(\epszero, K)$-framework with $|A| = |B| = Km$.
Then there are partitions $A_1,\dots,A_K$ of $A$ and $B_1,\dots,B_K$ of $B$ which satisfy the following properties:
\begin{itemize}
\item[{\rm (i)}] The partition $\mathcal{P}$ formed by $A_0$, $B_0$ and all these $2K$ clusters is a $(K, m, \epszero)$-partition of $V(G)$.
\item[{\rm (ii)}] $(G[A]+G[B], \mathcal{P})$ is a $(K, m, \epszero,\eps_2 )$-scheme.
\item[{\rm (iii)}] $(G - G[A] - G[B], \mathcal{P})$ is a $(K, m, \epszero,\eps_1)$-exceptional scheme.
\item[{\rm (iv)}] For all $v \in V(G)$ and all $1 \le i \le K$ we have $d_G(v,A_i) = (d_G(v,A) \pm \eps_0 n)/K $ and $d_G(v,B_i) = (d_G(v,B) \pm \eps_0 n)/K $.
\end{itemize}
\end{lemma}
\begin{proof}
Define a new constant $\eps'_1$  such that $\eps_0 \ll  \eps'_1 \ll \eps_1,1/K$.%
   \COMMENT{Before, we also had: "Let $A'_0$ be the set of  vertices $v\in A_0$ for which $d(v,B) \ge \epszero n$ and set $A''_0 := A_0 \setminus A'_0$.
Similarly let $B'_0$ be the set of  vertices $v\in B_0$ for which $d(v,A) \ge \epszero n$ and set $B''_0 := B_0 \setminus B'_0$."
But I don't think that this partition is necessary.
Moreover, before, we have $\eps'_1 = \eps_0$ and $\eps_0 \ll  \eps''_1 \ll \eps_1,1/K$. I have found the definition of $\eps''_1$ redundant.
So we just define $\eps_0 \ll  \eps'_1 \ll \eps_1,1/K$.
}
In order to find the required partitions $A_1,\dots,A_K$ of $A$ and $B_1,\dots,B_K$ of $B$
we will apply Lemma~\ref{lma:partition2} twice, as follows. 

In our first application of Lemma~\ref{lma:partition2} we let $F:=G$, $U:=A$ and let $A_0,B_0,B$ play the roles of $R_1, R_2,R_3$.
Note that $\delta(G[A]) \ge  \delta(G)-|A_0|-\eps_0 n\ge \eps_0 n$ (with room to spare) by (FR3), (FR4) and that
$d(a,R_j)\le |R_j|\le \eps_0 n$ for all $a\in A$ and $j=1,2$ by (FR3).
Moreover, (FR4) implies that $d(a,R_3)\le d(a,B') \le \eps_0 n$ for all $a\in A$.
Thus we can apply Lemma~\ref{lma:partition2} with $ \epszero , \eps_0$ and $ \eps_1'$ playing the roles of $\eps, \eps _1$ and $\eps _2$
to obtain a partition of $A$ into $K$ clusters $A_1,\dots,A_K$, each of size~$m$.
Then by Lemma~\ref{lma:partition2}(ii) for all $v \in V(G)$ and all $1 \le i \le K$ we have 
\begin{equation}\label{eq:dvAi}
d_G(v,A_i) = (d_G(v,A) \pm \eps_0 n)/K.
\end{equation}
Moreover, Lemma~\ref{lma:partition2}(v) implies that the first two equalities in~(ESch5) hold 
with respect to $\eps ' _1$ (for $G$ and thus also for $G - G[A] - G[B]$).
Furthermore,
\begin{equation}\label{eq:eAiB}
e_G(A_i, B)=(e_G(A, B) \pm \eps'_1 \max\{n, e_G(A, B)\})/K.
\end{equation}
For the second application of Lemma~\ref{lma:partition2} we let $F:=G$, $U:=B$ and let
$B_0,A_0,A_1,$ $\dots,A_K$  play the roles of $R_1, \dots ,R_{K+2}$.
As before, $\delta(G[B]) \ge \eps_0 n$ by (FR3), (FR4) and $d(b,R_j)\le |R_j|\le \eps_0 n$ for all $b\in B$ and $j=1,2$ by (FR3).
Moreover, (FR4) implies that $d(b,R_j)\le d(b,A')\le \eps_0 n$ for all $b\in B$ and all $j=3,\dots,K+2$.
Thus we can apply Lemma~\ref{lma:partition2} with $ \epszero , \eps_0$ and $ \eps_1'$ playing the roles of $\eps, \eps _1$ and $\eps _2$
to obtain a partition of $B$ into $K$ clusters $B_1,\dots,B_K$, each of size~$m$.
Similarly as before one can show that for all $v \in V(G)$ and all $1 \le i \le K$ we have
\begin{equation}\label{eq:dvBi}
d_G(v,B_i) = (d_G(v,B) \pm \eps_0 n)/K,
\end{equation}
and that the third and the fourth equalities in~(ESch5) hold with respect to $\eps ' _1$
(for $G$ and thus also for $G - G[A] - G[B]$).
Moreover, Lemma~\ref{lma:partition2}(v) implies that for all $1 \le i' \le K$ we have
\begin{eqnarray*}
e_G(A_i, B_{i'}) & = & (e_G(A_i, B) \pm \eps'_1 \max\{n, e_G(A_i, B)\})/K\\
& \stackrel{\eqref{eq:eAiB}}{=} & \frac{e_G(A, B) \pm \eps'_1 \max\{n, e_G(A, B)\} \pm K \eps'_1 \max\{n, e_G(A_i, B)\}}{K^2}\\
& = & (e_G(A, B) \pm \eps_1 \max\{n, e_G(A, B)\})/K^2,
\end{eqnarray*}
i.e.~the last equality in~(ESch5) holds too.%
   \COMMENT{Daniela: replaced $e_G(A, B)$ by $ \max\{n, e_G(A, B)\}$ in the last displayed expression}
Let $\mathcal{P}$ be the partition formed by $A_0,A_1,\dots,A_K$ and $B_0,B_1,\dots,B_K$.
Then (i) holds.

Let us now verify~(ii). Clearly $(G[A]+G[B], \mathcal{P})$ satisfies (Sch1) and~(Sch2). In order to check~(Sch3), let
$G_1:=G[A]+G[B]$ and note that for all $v \in A$ and all $1 \le i \le K$ we have 
\begin{align*}
d_{G_1}(v,A_i) & = d_{G}(v,A_i)\stackrel{\eqref{eq:dvAi}}{\ge} (d_G(v,A) -\eps_0 n)/K 
\stackrel{\rm (FR4)}{\ge} (\delta(G)- |A_0|-2\eps_0 n)/K\\
& \stackrel{\rm (FR3)}{\ge} ((1- \mu ) n/2-3\eps_0 n)/K\ge (1-\eps_2)m.
\end{align*}
Similarly one can use \eqref{eq:dvBi} to show that $d_{G_1}(v,B_i)\ge (1-\eps_2)m$ for all $v \in B$ and all $1 \le i \le K$.
This implies (Sch3) and thus~(ii).

Note that (iv) follows from \eqref{eq:dvAi} and \eqref{eq:dvBi}. Thus it remains to check~(iii).
Clearly $(G - G[A] - G[B], \mathcal{P})$ satisfies (ESch1), (ESch2) and we have already verified~(ESch5).
(ESch3) follows from~(FR4) and (ESch4) follows from \eqref{eq:dvAi} and \eqref{eq:dvBi}.
\end{proof}

\section{Proof of Theorem~$\text{\ref{thm:NWminclique}}$} \label{sec:NWminclique}

An important tool in the proof of Theorem~\ref{thm:NWminclique} is Lemma~\ref{almostthm},
which guarantees an `approximate' Hamilton decomposition of a graph $G$, provided that $G$ is close to the union of 
two disjoint copies of $K_{n/2}$. This yields the required number of Hamilton cycles for Theorem~\ref{thm:NWminclique}.
As an `input', Lemma~\ref{almostthm} requires an appropriate number of localized Hamilton exceptional systems.

To find these, we proceed as follows: the next lemma (Lemma~\ref{lma:BESnash}) guarantees many edge-disjoint
Hamilton exceptional systems in a given framework. 
We will apply it to `localized subgraphs' (obtained from Lemma~\ref{lma:randomslice}) of the original graph
to ensure that the exceptional systems guaranteed by Lemma~\ref{lma:BESnash} are also localized.
These can then be used as the required input for Lemma~\ref{almostthm}.

\begin{lemma}\label{lma:BESnash}
Suppose that $ 0 < 1/n \ll \epszero \ll \eps \ll  \alpha \ll 1$ and that $n, \alpha n \in \mathbb{N}$.
Let $G$ be a graph on $n$ vertices. Suppose that $(G,A,A_0,B,B_0)$ is an $(\epszero, K)$-framework
which satisfies the following conditions:
\begin{itemize}
\item[\rm (a)] $e_G(A',B') \ge 2 (\alpha + \eps) n $.
\item[\rm (b)] $e_{G- v}(A',B') \ge  \alpha n  $ for all $v \in A_0 \cup B_0$.
\item[\rm (c)] $d(v) \ge 2(\alpha + \eps) n $ for all $v \in A_0 \cup B_0$.
\item[\rm (d)] $d(v,A') \ge d(v,B') - \eps n$ for all $v \in A_0$ and $d(v,B') \ge d(v,A') - \eps n $ for all $v \in B_0$.
\end{itemize}
Then there exist $\alpha n $ edge-disjoint Hamilton exceptional systems with parameter $\eps_0$ in~$G$.
\end{lemma}
\begin{proof}
First we will find $\alpha n$ edge-disjoint matchings of size~2 in $G[A',B']$.
If $\Delta(G[A',B']) \le (\alpha+\eps/2) n$, then by (a) and Proposition~\ref{prop:matchingdecomposition} we can find such matchings.
So suppose that $\Delta(G[A',B']) \ge (\alpha+\eps/2) n$ and let $v$ be a vertex such that $d_{G[A',B']}(v) \ge (\alpha+\eps/2) n$. 
Thus $v\in A_0 \cup B_0$ by~(FR4).
By (b) there are $\alpha n$ edges $e_1,\dots,e_{\alpha n}$ in $G[A',B']-v$. Since $d_{G[A',B']}(v) \ge (\alpha+\eps/2) n$, for each
$e_s$ in turn we can find an edge $e'_s$ incident to $v$ in $G[A',B']$ such that $e'_s$ is vertex-disjoint from $e_s$ and such
that the $e'_s$ are distinct for different indices $s\le \alpha n$. Then the matchings consisting of $e_s$ and $e'_s$ are as required. 
Thus in both cases we can find edge-disjoint matchings $M_1, \dots ,M_{\alpha n }$ of size~2 in $G[A',B']$.

Our aim is to extend each $M_s$ into a Hamilton exceptional system $J_s$ such that all these $J_s$ are pairwise edge-disjoint. 
Initially, we set $F_s:=M_s$ for all $s \le \alpha n$. So each $F_s$ is a Hamilton exceptional system candidate.
For each $v \in V_0$ in turn, we are going to assign at most two edges joining $v$ to $A\cup B$ to each of $F_1, \dots ,F_{\alpha n }$ in such a way
that now each $F_s$ is a Hamilton exceptional system candidate with $d_{F_s}(v) = 2$. Thus after we have carried out these
assignments for all $v\in V_0$, every $F_s$ will be a Hamilton exceptional system with parameter~$\eps_0$.

So consider any $v \in V_0$. Without loss of generality we may assume that $v \in A_0$.
Moreover, by relabelling the $F_s$ if necessary, we may assume that there exists an integer $0\le r \le \alpha n $ such that
$d_{F_s}(v) = 1$ for all $s \le r$ and $d_{F_s}(v) = 0$ for $r < s \le \alpha n$.
For each $s\le r$ our aim is to assign some edge $vw_s$ between $v$ and $A$ to $F_s$ such that $w_s \notin V(F_s)$ and such
that the vertices $w_s$ are distinct for different $s\le r$.
To check that such an assignment of edges is possible, note that $|V(F_s) \cap A|, |V(F_s) \cap B| \le 2|V_0|+2\le 3\eps_0 n$.
Together with (c) and (d) this implies that
$$d(v,A) \ge d(v,A') - |A_0| \ge (\alpha + \eps/2 - \epszero) n > r + |V(F_s) \cap A|.$$
Thus for all $s\le r$ we can assign an edge $vw_s$ to $F_s$ as required. 

It remains to assign two edges at $v$ to each of $F_{r+1},\dots,F_{\alpha n}$. We will do this for each $s=r+1,\dots, \alpha n$
in turn and for each such $s$ we will either assign two edges between $v$ and $A$ to $F_s$ or two edges between $v$ and $B$.
(This will ensure that we still have $b(F_s)=2$, where $b(F_s)$ is the number of vertex-disjoint $A'B'$-paths in the path system~$F_s$.)
So suppose that for some $r<s\le \alpha n$ we have already assigned  two edges at $v$ to each of $F_{r+1}, \dots ,F_{s-1}$.
Set $G_s := G - \sum_{s'=1}^{\alpha n} F_{s'}$.%
    \COMMENT{We really want to sum over all $F_{s'}$ here.}
The fact that $v$ has degree at most two in each $F_{s'}$ and (c) together imply that
$d_{G_s}(v) \ge d_{G}(v) - 2 \alpha n \ge 10\epszero n$. So either $d_{G_s}(v,A')\ge 5\eps_0 n$ or $d_{G_s}(v,B')\ge 5\eps_0 n$.
If the former holds then%
\COMMENT{Deryk: had $A'$ instead of $A_0$ here}
$$d_{G_s}(v,A) \ge d_{G_s}(v,A')-|A_0|\ge 4\eps_0 n \ge |V(F_s) \cap A|+2$$
and so we can assign two edges $vw$ and $vw'$ of $G_s$ to $F_s$ such that $w,w' \in A \setminus V(F_s)$.
Similarly if $d_{G_s}(v,B')\ge 5\eps_0 n$ then we can assign two edges $vw$ and $vw'$ in $G_s$ to $F_s$ such that $w,w' \in  B \setminus V(F_s)$.
This shows that to each of $F_{r+1},\dots,F_{\alpha n}$ we can assign two suitable edges at~$v$. 

Let $J_1,\dots,J_{\alpha n}$ be the graphs obtained after carrying out these assignments for all $v\in V_0$.
Then the $J_s$ are pairwise edge-disjoint and it is easy to check that each $J_s$ is a Hamilton exceptional system with parameter~$\eps_0$.
(Note that (ES2) and (ES3) hold since $ b(J_s)=2$ and so the number of $AB$-paths is two.)%
\COMMENT{Deryk: previously had $e_{(J_s)^*_{AB}}(A,B)=b(J_s)=2$, but it seems nicer to check ES2 without the additional notation}
\end{proof}

The next lemma guarantees a decomposition of an exceptional scheme $(G, \mathcal{P})$ into suitable `localized slices' $G(i,i')$ whose edges are induced by
$A_0$, $B_0$ and two clusters of~$\mathcal{P}$.%
   \COMMENT{Daniela: replaced "of a given partition" by "of~$\mathcal{P}$}
We will use it again in~Chapter~\ref{paper4}.

\begin{lemma} \label{lma:randomslice}
Suppose that $0 <  1/n  \ll \epszero \ll  \eps  \ll 1/K \ll 1$ and that $n, K,m\in \mathbb{N}$.
Let $(G, \mathcal{P})$ be a $(K, m, \epszero, \eps )$-exceptional scheme with $|G|=n$ and $e_G(A_0),e_G(B_0)$  $=0$.
Then $G$ can be decomposed into edge-disjoint spanning subgraphs $H(i,i')$ and $H'(i,i')$ of $G$ (for all $i,i' \le K$)
such that the following properties hold, where $G(i,i'):=H(i,i')+H'(i,i')$:
\begin{itemize}
\item[$(a_1)$] Each $H(i,i')$ contains only $A_0A_i$-edges and $B_0B_{i'}$-edges.
\item[$(a_2)$] All edges of $H'(i,i')$ lie in $G[A_0 \cup A_i, B_0 \cup B_{i'}]$.
\item[$(a_3)$] $e ( H'(i,i') )   =  ( e_{G}(A',B')  \pm  4 \eps \max \{ n, e_{G}(A',B') \}) /K^2$.
\item[$(a_4)$] $d_{H'(i,i')}(v )  =  ( d_{G[A',B']}(v)  \pm 2\eps n)/K^2$ for all $v \in V_0$.
\item[$(a_5)$] $d_{G(i,i')}(v )  =  ( d_{G}(v)  \pm 4 \eps n)/K^2$ for all $v \in V_0$.
\end{itemize}
\end{lemma}

\begin{proof}
First we decompose $G$ into $K^2$ `random' edge-disjoint spanning subgraphs $G(i,i')$ (one for all $i,i' \le K$)
as follows:
\begin{itemize}
\item Initially set $V(G(i,i')):=V(G)$ and $E(G(i,i')):=\emptyset$ for all $i,i' \le K$.
\item Add all the $A_iB_{i'}$-edges of $G$ to $G(i,i')$.
\item Choose a partition of $E(A_0,B_0)$ into $K^2$ sets $U_{i,i'}$ (one for all $i,i'\le K$) whose sizes are as
equal as possible. Add the edges in $U_{i,i'}$ to $G(i,i')$.
\item For all $i\le K$, choose a random partition of $E(A_0,A_i)$ into $K$ sets $U'_{i'}$ of equal size (one for each $i'\le K$)
and add the edges in $U'_{i'}$ to $G(i,i')$. (If $e(A_0,A_i)$ is not divisible by $K$, first distribute up to $K-1$ edges
arbitrarily among the $U'_{i'}$ to achieve divisibility.) For all $i'\le K$ proceed similarly to distribute each edge in $E(B_0,B_{i'})$
to $G(i,i')$ for some $i\le K$.
\item For all $i'\le K$, choose a random partition of $E(A_0,B_{i'})$ into $K$ sets $U''_{i}$ of equal size (one for each $i\le K$)
and add the edges in $U''_{i}$ to $G(i,i')$. (If $e(A_0,B_{i'})$ is not divisible by $K$, first distribute up to $K-1$ edges
arbitrarily among the $U''_{i}$ to achieve divisibility.) For all $i\le K$ proceed similarly to distribute each edge in $E(B_0,A_{i})$
to $G(i,i')$ for some $i'\le K$.
\end{itemize}
Thus every edge of $G$ is added to precisely one of the subgraphs $G(i,i')$.
Set $H(i,i'): = G(i,i')[A'] + G(i,i')[B']$ and $H'(i,i') :=  G(i,i')[A',B']$. 
So conditions (a$_1$) and (a$_2$) hold. Fix any $i, i' \le K$ and set $H:= H(i,i')$ and $H':=H'(i,i')$.
To verify~(a$_3$), note that
\begin{align*}
&e(H')  =e_{H'}(A_i,B_{i'})+e_{H'}(A_0,B_0)+e_{H'}(A_0,B_{i'})+e_{H'}(B_0,A_i)\\
& = e_{G}(A_i,B_{i'})+e_{G}(A_0,B_0)/K^2+e_{G}(A_0,B_{i'})/K+e_{G}(B_0,A_i)/K\pm 3\\
& = \frac{e_{G}(A,B)+e_{G}(A_0,B_0)+e_{G}(A_0,B)+e_{G}(B_0,A)\pm 3\eps \max \{e_G(A',B'), n \}}{K^2}\pm 3\\
& = \frac{e_{G}(A',B')\pm 4\eps \max \{e_G(A',B'), n \}}{K^2}.
\end{align*}
Here the third equality follows from (ESch5).

To prove (a$_4$), suppose first that $v\in A_0$. If $d_G(v,B_{i'}) \leq  \eps n/K^2$ then
clearly $0 \leq d_{H'(i,i')} (v)  \leq  \eps n/K^2+ |V_0| \leq 2 \eps n/K^2$.
Further by (ESch4) we have
$d_G(v,B) \leq K d_G(v,B_{i'}) +\eps n =\eps n/K +\eps n.$
So $d_G(v,B') \leq 2 \eps n$.
Together this shows that (a$_4$) is satisfied.

So assume that $d_G(v,B_{i'}) \geq  \eps n/K^2$.
 Proposition~\ref{chernoff}  implies that with probability at least $1-{\rm e}^{-\sqrt{n}}$ (with room to spare) we have%
    \COMMENT{Daniela: added (ESch4) as stackrel}
\begin{equation}\label{eq:degBi'}
d_{G(i,i')}(v,B_{i'} )=(d_{G}(v,B_{i'})\pm \eps n/2K)/K \stackrel{{\rm (ESch4)}}{=}( d_G(v,B) \pm 3\eps n/2 ) / K^2.
\end{equation}
Since
\begin{eqnarray*}
d_{H'(i,i')}(v ) & = & d_{G(i,i')}(v,B_{i'} )+d_{G(i,i')}(v,B_0 )= d_{G(i,i')}(v,B_{i'} )\pm \eps_0 n\\
& \stackrel{\eqref{eq:degBi'}}{=} & (d_{G}(v,B' )\pm 2\eps n)/K^2,
\end{eqnarray*}
it follows that $v$ satisfies~(a$_4$). The argument for the case when $v\in B_0$ is similar.
Thus (a$_4$) holds with probability at least  $1-n {\rm e}^{-\sqrt{n}}$.

Similarly as \eqref{eq:degBi'} one%
    \COMMENT{Note that \eqref{eq:degBi'} trivially holds for vertices $v$ with $d_G(v,B_{i'})< 3\eps n/2K^2$.}
can show that with probability at least  $1-n {\rm e}^{-\sqrt{n}}$ we have
$d_{G(i,i')}(v,A_{i} )= ( d_G(v,A) \pm 3\eps n/2 ) / K^2$ for all $v\in A_0$ and
$d_{G(i,i')}(v,B_{i'} ) =( d_G(v,B) \pm 3\eps n/2) / K^2$ for all $v\in B_0$.
Together with the fact that%
   \COMMENT{Daniela: previously had $|V_0|\le \eps_0 n$ instead of $e_G(A_0), e_G(B_0)=0$ here. Both are correct, but perhaps the
first is a little misleading}
$e_G(A_0), e_G(B_0)=0$ and (a$_4$) this now implies~(a$_5$). 
\end{proof}

The next lemma first applies the previous one to construct localized subgraphs $G(i,i')$ and then applies 
Lemma~\ref{lma:BESnash} to find many Hamilton exceptional systems within each of the localized slices $G(i,i')$.
Altogether, this yields many localized Hamilton exceptional systems in~$G$.

\begin{lemma} \label{lma:BESnash2}
Suppose that $0 <  1/n  \ll \epszero \ll  \eps  \ll \phi,  1/K \ll 1$ and that $n, K, m, (1/4 - \phi)n/K^2 \in \mathbb{N}$.%
   \COMMENT{In addition, previously also had that $\phi n \in \mathbb{N}$. But I don't think this is needed and one cannot require
   that both $\phi n, (1/4 - \phi)n/K^2 \in \mathbb{N}$.}
Suppose that $(G,A,A_0,B,B_0)$ is an $(\eps_0,K)$-framework with%
   \COMMENT{Daniela: reworded}
$|G|=n$, $\delta(G) \ge n/2$ and such that 
$d_G(v,A') \ge d_G(v)/2$ for all $v \in A'$ and $d_G(v,B') \ge d_G(v)/2$ for all $v \in B'$.
Suppose that $\mathcal{P}=\{A_0,A_1,\dots,A_K,B_0,B_1,\dots,$ $B_K\}$ is a refinement of the partition $A,A_0,B,B_0$
such that $(G - G[A] - G[B],\mathcal{P})$ is a $(K,m, \epszero, \eps)$-exceptional scheme. 
Then there is a set $\mathcal{J}$ of $(1/4 - \phi) n$ edge-disjoint Hamilton exceptional systems with parameter $\eps_0$ in $G$
such that, for each $i, i' \le K$, $\mathcal{J}$ contains precisely $(1/4 - \phi) n/K^2$ $(i,i')$-HES.
\end{lemma}
\begin{proof}
Let $\alpha : = (1/4 - \phi)/K^2$ and choose a new constant $\eps'$ such that $\eps  \ll \eps'\ll \phi,  1/K$.
Note that (FR3) implies that $|A'| \ge |B'|$.
If $|B'| < n/2$, then Proposition~\ref{prp:e(A',B')}(i) implies that%
    \COMMENT{ This clearly holds if $|B'| \le n/2-1$. But if $n$ is odd and $|B'|=\lfloor n/2\rfloor$ then $\delta(G) \ge \lceil n/2\rceil$
so we have that $e_{G}(A',B') \ge 2|B'|$ in this case as well.}
$e_{G}(A',B') \ge 2|B'| \ge (1- \epszero )n \ge 3 K^2 \alpha n$ (where the second inequality follows from (FR3)
and there is room to spare in the final inequality).
Since $d_{G[A',B']}(v)\le n/2$ for every vertex $v \in V(G)$, it follows that $e_{G- v}(A',B') \ge (1/2- \epszero)n \ge 3K^2\alpha n/2$.
If $|B'| = n/2$, then $|A'|  = |B'|$ and Proposition~\ref{prp:e(A',B')}(i) implies that $e_{G}(A',B') \ge |B'| = n/2 \ge 2 K^2( \alpha + \eps')n$.
Moreover, $|A'|  = |B'|$ together with the fact that $\delta(G) \ge n/2 $ also implies that $d_{G[A',B']}(v) \ge 1$ for any vertex $v \in V(G)$. 
Hence $e_{G- v}(A',B') \ge n/2 - 1\ge 3K^2\alpha n/2$. Thus regardless of the size of $B'$, we always have
\begin{equation}\label{eqn:nash1}
	e_{G}(A',B') \ge  2 K^2( \alpha + \eps') n
\end{equation}
and
\begin{equation}\label{eqn:nash2}
	e_{G- v}(A',B') \ge 3K^2\alpha n/2 \ge K^2(\alpha + \eps') n \ \ \  \textrm{for any $v \in V(G)$.}
\end{equation}
Set $G^{\diamond} : = G - G[A] - G[B] - G[A_0] - G[B_0]$.
Note that each vertex $v \in V_0$ satisfies
\begin{equation}\label{eq:degGdiamond}
d_{G^{\diamond}}(v) \ge (1/2 - \epszero) n \ge 2K^2( \alpha + \eps') n.
\end{equation}
Moreover, both \eqref{eqn:nash1} and \eqref{eqn:nash2} also hold for $G^{\diamond}$, and since $(G - G[A] - G[B],\mathcal{P})$ is a $(K,m, \epszero, \eps)$-exceptional scheme,
$(G^{\diamond}, \mathcal{P})$ is also a $(K,m, \epszero, \eps)$-exceptional scheme. 
Thus we can apply Lemma~\ref{lma:randomslice} to $G^{\diamond}$ to obtain edge-disjoint spanning subgraphs $H(i,i')$, $H'(i,i')$ of
$G^{\diamond}$ (for all $i,i' \le K$) which satisfy (a$_1$)--(a$_5$) of Lemma~\ref{lma:randomslice}.
Set $G(i,i') : = H(i,i') + H'(i,i')$ for all $i,i' \le K$.
We claim that each $G(i,i')$ satisfies the following properties:
\begin{itemize}
\item[{\rm (i)}] All edges of $G(i,i')$ lie in $G^{\diamond}[A_0 \cup A_i \cup B_0 \cup B_{i'}]$.
\item[{\rm (ii)}] $e_{G(i,i')}(A',B')  \ge  2(\alpha + \sqrt{\eps} ) n $.
\item[{\rm (iii)}] $e_{G(i,i')- v }(A',B')  \ge  \alpha n $ for all $v \in V_0$.
\item[{\rm (iv)}] $d_{G(i,i')}(v )  \ge  2(\alpha + \sqrt{\eps}) n $ for all $v \in V_0$.
\item[{\rm (v)}] $d_{G(i,i')}(v,A' )  \ge  d_{G(i,i')}(v,B' ) - \sqrt{\eps} n$ for all $v \in A_0$ and%
   \COMMENT{We don't get $\eps$ here (but we could get $10\eps$ if we want).}
$d_{G(i,i')}(v,B' )  \ge  d_{G(i,i')}(v,A' ) - \sqrt{\eps} n$ for all $v \in B_0$.
\end{itemize}
Indeed, (i) follows from (a$_1$) and~(a$_2$).
To prove~(ii), note that $e_{G(i,i')}(A',B')=e(H'(i,i'))$.
Now apply (a$_3$) and~\eqref{eqn:nash1}.
For~(iii),%
	\COMMENT{Previsouly, we have `(iii) follows from~\eqref{eqn:nash2} and (a$_3$).', but I don't think that is not enough. This stuff is new.}
note that (a$_4$) and $\Delta(G[A',B']) \le n/2$ imply that for all $v \in V_0$, 
$$d_{G(i,i')[A',B']}(v) = d_{H'(i,i')} (v) \le (d_{G[A',B']}(v)+ 2 \eps n) /K^2 \le (1/2 +2\eps)n/K^2.$$ 
If $e_G(A',B') \ge n$, then (a$_3$) implies that $e_{G(i,i')}(A',B') \ge (1- 4 \eps) n/K^2 \ge \alpha n + d_{G(i,i')[A',B']}(v)$ and so (iii) follows.
If $e_G(A',B') < n$, then for all $v \in V_0$
\begin{align*}
e_{G(i,i')- v }(A',B') 
& = e ( H'(i,i') ) - d_{H'(i,i')}(v) \\ & \stackrel{({\rm a_3}),({\rm a_4})}{\ge} (  e_{G - v}(A',B')   - 6 \eps n  )/K^2 \stackrel{\eqref{eqn:nash2}}{\ge} \alpha n.
\end{align*}
So (iii) follows again.
(iv) follows from (a$_5$) and~\eqref{eq:degGdiamond}.
For~(v), note that (a$_1$) and (a$_2$) imply that 
for $v \in A_0$, 
\begin{align*}
d_{G(i,i')}(v,A' ) & =d_{G(i,i')}(v) -d_{H'(i,i')}(v) \stackrel{({\rm a_4}),({\rm a_5})}{\geq} ( d_{G}(v,A')  - 6\eps n)/K^2 \\
& \ge  (d_{G}(v,B')  - 6\eps n)/K^2 \stackrel{({\rm a_4})}{\geq} d_{H'(i,i')}(v)- 8\eps n= d_{G(i,i')}(v,B' )- 8\eps n.
\end{align*}
The second part of~(v) follows similarly.

Note that each $(G(i,i'),A,A_0,B,B_0)$ is an $(\epszero, K)$-framework since this holds for $(G,A,A_0,B,B_0)$.
Thus for all $i,i'\le K$ we can apply Lemma~\ref{lma:BESnash} (with $\sqrt{\eps}$ playing the role of $\eps$)
to the $(\epszero, K)$-framework $(G(i,i'),A,A_0,B,B_0)$
in order to obtain $\alpha n $ edge-disjoint Hamilton exceptional systems
with parameter $\eps_0$ in $G(i,i')$. 
By (i), we may delete any vertices outside $A_0 \cup A_i \cup B_0 \cup B_{i'}$ from these systems without affecting their edges.
So each of these Hamilton exceptional systems is in fact an $(i,i')$-HES.
The set $\mathcal{J}$ consisting of all these $K^2\alpha n$ Hamilton exceptional systems is as required in the lemma.
\end{proof}

Given the appropriate set $\cJ$ of localized Hamilton exceptional systems, the next lemma
guarantees a set of $|\cJ|$ edge-disjoint Hamilton cycles in a graph $G$ such that each of them contains
one exceptional system from $\cJ$, provided that $G$ is sufficiently close to the union of two disjoint copies of~$K_{n/2}$.
The lemma also allows $\cJ$ to contain matching exceptional systems (each of these will then be extended into a
perfect matching of $G$). Note that with a suitable $\cJ$ and an appropriate choice of parameters we can achieve that the `uncovered'
graph has density $2\rho \pm 2/K \ll 1$, i.e.~we do have an approximate decomposition.
We defer the proof of the lemma until Chapter~\ref{paper3}.

\begin{lemma}\label{almostthm}
Suppose that $0<1/n \ll \eps_0  \ll 1/K \ll \rho  \ll 1$ and $0 \le \mu \ll 1$,
where $n,K \in \mathbb N$ and $K$ is odd.%
	\COMMENT{previously had $1/K \leq \mu$, but this seems clearer. When reading, need to make sure all estimates involving $\mu$ are ok now}
Suppose that $G$ is a graph on $n$ vertices and $\mathcal{P}$ is a $(K, m, \eps _0)$-partition of $V(G)$.
Furthermore, suppose that the following conditions hold:%
	\COMMENT{Previously we had `$(G[A]+G[B],\mathcal{P})$ is a $(K, m, \eps _0, \eps )$-scheme', but this is not neeeded.}
\begin{itemize}
	\item[{\rm (a)}] $d(v,A_i) = (1 - 4 \mu \pm 4 /K) m $ and $d(w,B_i) = (1 - 4 \mu \pm 4 /K) m $ for all
	$v \in A$, $w \in B$ and $1\leq i \leq K$.
	\item[{\rm (b)}] There is a set $\mathcal J$ which consists of at most $(1/4-\mu - \rho)n$ edge-disjoint exceptional systems with parameter $\eps_0$ in~$G$.%
	\COMMENT{By the definition of exceptional system, $J \subseteq G - G[A] - G[B]$ for all $J\in \mathcal{J}$. So we can omit this
condition (which we had previously).}
	\item[{\rm (c)}] $\mathcal J$ has a partition into $K^2$ sets $\mathcal J_{i,i'}$ (one for all $1\le i,i'\le K$) such that each $\mathcal J_{i,i'}$ consists of precisely $|\mathcal J|/{K^2}$ $(i,i')$-ES with respect to~$\cP$.
   \item[{\rm (d)}] If $\mathcal{J}$ contains matching exceptional systems then $|A'|=|B'|$ is even.
\end{itemize}
Then $G$ contains $|\mathcal J|$ edge-disjoint spanning subgraphs $H_1,\dots,H_{|\mathcal J|}$ which satisfy the following
properties:
\begin{itemize}
\item For each $H_s$ there is some $J_s\in \mathcal{J}$ such that $J_s\subseteq H_s$.
\item If $J_s$ is a Hamilton exceptional system, then $H_s$ is a Hamilton cycle of $G$. If
$J_s$ is a matching exceptional system, then $H_s$ is the edge-disjoint union of two perfect matchings in $G$.
\end{itemize}
\end{lemma}
Matching exceptional systems do no play any role in the current application to prove
Theorem~\ref{thm:NWminclique}, but they will occur when we use Lemma~\ref{almostthm} again in the proof of Theorem~\ref{1factstrong}.

To prove Theorem~\ref{thm:NWminclique}, we first apply Lemma~\ref{lma:BESnash2} to find suitable localized 
Hamilton exceptional systems and then apply Lemma~\ref{almostthm} to transform these into Hamilton cycles.

\removelastskip\penalty55\medskip\noindent{\bf Proof of Theorem~\ref{thm:NWminclique}. }
Choose new constants $\eps_{\rm ex}$, $\eps_0$, $\eps_1$, $\eps_2$, $\phi$ and an odd number $K\in \mathbb{N}$ such that
$$1/n_0\ll \eps_{\rm ex}\ll \eps_0\ll \eps_1\ll \eps_2\ll 1/K\ll \phi\ll \eps.$$
Further, we may assume that $\eps \ll 1$.
Let $n\ge n_0$ and let $G$ be any graph on $n$ vertices such that $\delta(G)\ge n/2$ and such that
$G$ is $\eps_{\rm ex}$-close to two disjoint copies of $K_{n/2}$. By modifying $\phi$ slightly, we may assume that
$(1/4-\phi)n/K^2\in\mathbb{N}$.

Apply Proposition~\ref{prop:framework} to obtain a partition
$A,A_0,B,B_0$ of $V(G)$ such that such that $(G,A,A_0,B,B_0)$ is an $(\epszero, K)$-framework,
$d(v, A') \ge d(v)/2$ for all $v\in A'$ and
$d(v, B') \ge d(v)/2$ for all $v\in B'$. Let $m:=|A|/K=|B|/K$. Apply Lemma~\ref{lma:partition} with $\eps_0$ playing the
role of $\mu$ to obtain partitions $A_1,\dots,A_K$ of $A$ and $B_1,\dots,B_K$ of $B$ which satisfy the following properties,
where $\mathcal{P}=\{A_0,A_1,\dots,A_K,B_0,B_1,\dots,B_K\}$:
\begin{itemize}
\item $(G[A]+G[B], \mathcal{P})$ is a $(K, m, \epszero,\eps_2 )$-scheme.
\item $(G - G[A] - G[B], \mathcal{P})$ is a $(K, m, \epszero,\eps_1)$-exceptional scheme.
\end{itemize}
Apply Lemma~\ref{lma:BESnash2} to obtain a set $\mathcal{J}$ of $(1/4 - \phi) n$ edge-disjoint Hamilton exceptional
systems with parameter $\eps_0$ in $G$ such that, for each $i, i' \le K$, $\mathcal{J}$ contains precisely $(1/4 - \phi) n/K^2$ $(i,i')$-HES.
Finally, our aim is to apply Lemma~\ref{almostthm} with $\mu:=1/K$ and $\rho:=\phi-1/K$.
So let us check that conditions~(a)--(c) of Lemma~\ref{almostthm} hold (note that (d) is not relevant).%
\COMMENT{Deryk: added the bracket}
Clearly (b) and (c) hold.
 To verify (a) note that (Sch3) implies that for all $v\in A$ we have
$d(v,A_i) \ge (1-\eps_2)m\ge (1-1/K)m \ge (1 - 4 \mu - 4 /K) m$. Similarly, for all $w\in B$ we have
$d(w,B_i) \ge (1 - 4 \mu - 4 /K) m$.
So we can apply Lemma~\ref{almostthm} to obtain $|\mathcal J|\ge (1/4-\eps)n$ edge-disjoint Hamilton cycles.
\endproof

\section{Eliminating the Edges inside $A_0$ and $B_0$}\label{sec:V_0}

This and the remaining sections of the chapter are all devoted to the proof of Theorem~\ref{1factstrong}.
Suppose that $G$ is a $D$-regular graph and $(G,A,A_0,B,B_0)$ is an $(\epszero,K)$-framework with $\Delta (G[A',B']) \le D/2$.
The aim of this section is to construct a small number of Hamilton cycles 
(and perfect matchings if appropriate) which together cover all the edges of $G[A_0]$ and $G[B_0]$.
The first step is to construct a small number of exceptional systems containing all the edges of $G[A_0]$ and $G[B_0]$.

\begin{lemma} \label{V_0BES}
Suppose that $ 0 < 1/n \ll \epszero \le  \lambda \ll 1$ and that $n, \lambda n, D, K \in \mathbb{N}$.%
   \COMMENT{Here, $K$ is not important.}
Let $G$ be a $D$-regular graph on $n$ vertices with $D \ge n - 2\lfloor n/4 \rfloor -1$.
Suppose that $(G,A,A_0,B,B_0)$ is an $(\epszero,K)$-framework with $\Delta (G[A',B']) \le D/2$. Let 
\begin{align*}
\ell & := \left\lfloor \frac{ \max \{0 , D - e_G(A',B') \} }2 \right\rfloor \ \ \ \ \text{and} \ \ 
& \phi n & := \begin{cases} 
2 \lambda n +1 & \textrm{if $D$ is odd,}\\
2 \lambda n  & \textrm{if $D$ is even.}
\end{cases}
\end{align*}
Let $w_1$ and $w_2$ be vertices of $G$ such that $d_{G[A',B']}(w_1)\ge d_{G[A',B']}(w_2)\ge d_{G[A',B']}(v)$ for all $v\in V(G)\setminus \{w_1,w_2\}$.
Then there exist $\lambda n+1$ edge-disjoint subgraphs $J_0,J_1, \dots ,J_{\lambda n }$ of $G$ which cover all the edges in $G[A_0]+G[B_0]$
and satisfy the following properties: 
\begin{itemize}
	\item[{\rm (i)}] If $D$ is odd, then $J_0$ is a perfect matching in $G$ with $e_{J_0}(A',B') \le 1$.
	If $D$ is even, then $J_0$ is empty. 
	\item[{\rm (ii)}] $J_s$ is a matching exceptional system with parameter $\eps_0$ for all $1 \le  s \le \min \{ \ell, \lambda n \}$.
	\item[{\rm (iii)}] $J_s$ is a Hamilton exceptional system with parameter $\eps_0$ and such that $e_{J_s}(A',B') = 2$ for all $\ell < s \le \lambda n$. 
	\item[{\rm (iv)}] Let $\mathcal{J} $ be the union of all the $J_s$ and let $H^{\diamond} := G[A',B'] - \mathcal{J}$.
	Then $e_{\mathcal{J}}(A',B') \le \phi n $ and $d_{\mathcal{J}}(v) = \phi n $ for all $v \in V_0$.
	Moreover, $e(H^{\diamond})$ is even.
	\item[{\rm (v)}] $ d_{ H^{\diamond} }(w_1) \le (D-\phi n )/2$.
	Furthermore, if $D = n/2-1$%
   \COMMENT{Note that this means that $n = 0 \pmod4$. Previously this was part of the statement of (v).}
    then $ d_{H^{\diamond} }(w_2) \le (D-\phi n )/2$.
	\item[{\rm (vi)}] If $e_G(A',B')<D$, then $e(H^\diamond)\le D-\phi n$ and $ \Delta( H^{\diamond} ) \le e(H^{\diamond})/2$.%
  \COMMENT{Conditions (v) and (vi) are only needed when $D = (n-1)/2$ or $D = n/2 - 1$. Note that if $\ell>0$ then $e_G(A',B') < D$
and so Proposition~\ref{prp:e(A',B')}(ii) implies that $D= n/2-1$, $n = 0 \pmod4$ and $|A'| = |B'| = n/2$. So
conditions (ii) and (vi) are only relevant in this case.}
\end{itemize}
\end{lemma}
As indicated in Section~\ref{overview},%
	\COMMENT{Deryk: added pointer to Sec 11.}
 the main proof of Theorem~\ref{1factstrong} splits into three cases: (a) the non-critical case with $e_G(A',B') \ge D$,
(b) the critical case with $e_G(A',B') \ge D$ and (c) the case with $e_G(A',B') < D$.
The formal definition of `critical' and a more detailed discussion of the different cases is given in Section~\ref{sec:locES}.

The above lemma will be used in all three cases.
In these different cases, we will need that the Hamilton cycles or perfect matchings produced by the lemma use
appropriate edges between $A'$ and $B'$ (and thus the `leftover' $H^\diamond$ has suitable properties).
In particular, (v) will ensure that we can apply Lemma~\ref{lma:BESdecomcritical} in case (b).
Similarly, (vi) will ensure that we can apply Lemma~\ref{lma:PBESdecom} in case (c).
(ii) and (vi) will only be relevant in case (c).

\removelastskip\penalty55\medskip\noindent{\bf Proof of Lemma~\ref{V_0BES}. }
Set $H := G[A',B']$ and $W:= \{w_1,w_2\}$.
First, we construct $J_0$.
If $D$ is even, then (i) is trivial, so we may assume that $D$ is odd (and so $n$ is even). 
We will construct $J_0$ such that it satisfies (i) as well as the following additional property: 
\begin{itemize}
	\item[(i$'$)] If $ w_1w_2$ is an edge in $G[A']+G[B']$, then $w_1w_2$ lies in $J_0$. Moreover, $e_{J_0}(A',B')=1$ if $|A'|$ is odd
and $e_{J_0}(A',B')=0$ if $|A'|$ is even.
\end{itemize}
Suppose first that $|A'|$ is even (and so $|B'|$ is even as well).
Since our assumptions imply that $\delta(G[A']) \ge \lceil D/2 \rceil \ge 3 \epszero n $, there exists a matching $M'_A$ in $G[A']$ of size at
most $|A_0| + 2$ covering all the vertices of $A_0 \cup (A' \cap W)$.
Moreover, if $w_1w_2$ is an edge in $G[A']$, then we can ensure that $w_1w_2\in M'_A$.
Note that $A'' : = A' \setminus V(M'_A)$ is a subset of $A$ and $|A''|$ is even. 
(FR4) implies that $\delta(G[A'']) \ge D - \epszero n - 2(|A_0|+2) \ge |A''|/2$.
Therefore, there exists a perfect matching $M_A''$ in $G[A'']$ (e.g.~by Dirac's theorem).
Hence, $M_A:= M_A' + M_A''$ is a perfect matching in $G[A']$.
Similarly, there is a perfect matching $M_B$ in $G[B']$ such that if $w_1w_2$ is an edge in $G[B']$, then $w_1w_2$ is in $M_B$.
Set $J_0: = M_A + M_B$.

Next assume that $|A'|$ is odd.
If $D \ge \lfloor n/2\rfloor$, then Proposition~\ref{prp:e(A',B')2} implies that $e(H- W) >0$.%
    \COMMENT{We can apply Proposition~\ref{prp:e(A',B')2} since $D\le (1/2+\eps_0)n$ as
$(G,A,A_0,B,B_0)$ is an $(\epszero,K)$-framework and so $|A'|\ge |B'|\ge D/2$.} 
If $D= n/2-1$, then $n = 0 \pmod4$ and so $|B'|\le n/2-1$ since $|A'|$ is odd.
Together with Proposition~\ref{prp:e(A',B')}(ii) this implies that $e(H) \ge n/2 -1$.
Since in this case we also have that $\Delta(H) \le \lfloor D/2 \rfloor = n/4-1$, it follows that $e(H- W) \ge e(H) - 2 \Delta(H) >0$.
Thus in both cases there exists an edge $ab$ in $H-W$ with $a\in A'$ and $b\in B'$.
Note that both $|A' \setminus \{a\}|$ and $|B' \setminus \{b\}|$ are even.
Moreover, $\delta(G[A' \setminus \{a\}]) \ge \lceil D/2 \rceil -1 \ge 3 \epszero n $ and
$\delta(G[B' \setminus \{b\}]) \ge \lceil D/2 \rceil -1 \ge 3 \epszero n $.
Thus we can argue as in the case when $|A'|$ is even to find perfect matchings $M_A$ and $M_B$ in $G[A' \setminus \{a\}]$ and $G[B' \setminus \{b\}]$ respectively
such that if $ w_1w_2$ is an edge in $G[A']+G[B']$ then $w_1w_2 \in M_A+M_B$. Set $J_0: = M_A + M_B+ab$.

This completes the construction of $J_0$. (If $D$ is even we set $J_0:=\emptyset$.) So (i) and~(i$'$) hold.%
   \COMMENT{Daniela: added (i$'$)}
Let $G' := G - J_0$ and $H' := G'[A',B']$.
Since $|A_0|+|B_0| \le \epszero n \le \lambda n $, Vizing's theorem implies that
we can decompose $G'[A_0] + G'[B_0]$ into $\lambda n$ edge-disjoint (possibly empty) matchings $M_1, \dots ,M_{\lambda n }$.
By relabeling these matchings if necessary, we may assume that if
$w_1w_2\in E_{G'}(A_0)$ or $w_1w_2\in E_{G'}(B_0)$, then $w_1w_2 \in M_1$.

\medskip

\noindent \textbf{Case 1: $e(H)\ge D$.}

\smallskip

\noindent
Note that in this case $\ell =0$ and $e(H') \ge D-1$.
For each $s=1,\dots,\lambda n$ in turn we will extend $M_s$ into a Hamilton exceptional system $J_s$ with
$e_{J_s}(A',B') =2$ and such that $J_s$ and $J_{s'}$ are edge-disjoint for
all $0\le s' < s$. In order to do this, we will first extend $M_s$ into a Hamilton exceptional system candidate $F_s$
by adding two independent $A'B'$-edges $f_s$ and $f'_s$. We will then use Lemma~\ref{lma:ESextend}
to extend $F_s$ into a Hamilton exceptional system $J_s$. 
For all $s$ with $1\le s\le \lambda n$, we will choose these edges and sets to satisfy the following:
\begin{itemize}
\item[($\alpha_1$)] $J_s$ is a Hamilton exceptional system with parameter $\eps_0$ such that \break $e_{J_s}(A',B') = 2$. 
\item[($\alpha_2$)] Suppose that $d_{H}(w_1) \ge 2 \lambda n$. Then $w_1$ is an endpoint of $f_s$.  
\item[($\alpha_3$)] Suppose that $d_{H}(w_2) \ge 2 \lambda n$. Then $w_2$ is an endpoint of $f'_s$,
unless both $s=1$ and $w_1w_2\in M_1$.
\item[($\alpha_4$)] $J_s$ contains $M_s$ as well as the edges $f_s$ and $f'_{s}$.
$J_s -M_s-f_s-f'_s$ only contains $A_0A$-edges and $B_0B$-edges of~$G$.
$J_s$ is edge-disjoint from $J_0,\dots,J_{s-1}$.
\end{itemize}
First suppose that $w_1w_2 \in M_1$. We construct $J_1$ satisfying the above.
Our assumption means that $w_1w_2$ is an edge in $G[A'] +G[B']$, so $D$ is even (or else $w_1w_2 \in J_0$ by~(i$'$)).
Moreover, $H' = H$ and $D \ge \lfloor n/2 \rfloor$ by~\eqref{minexact} and the fact that $D$ is even.
Together with Proposition~\ref{prp:e(A',B')2} this implies that $e(H'-W) = e(H- W) >0$. 
Pick an $A'B'$-edge $f'_1$ in $H'-W$. Let $U_1$ be the connected component in $M_1+f'_1$ containing $f'_1$.
So $|U_1| \le 4$ and $w_1\notin U_1$.%
    \COMMENT{This holds since neither $w_1$ nor $w_2$ is an endpoint of $f'_1$ and $w_1w_2\in M_1$.}
If $d_{H}(w_1) \ge 2 \lambda n$, we can find an $A'B'$-edge $f_1$ such that $w_1$ is one endpoint of $f_1$
and the other endpoint of $f_1$ does not lie in $U_1$. If $d_{H}(w_1) < 2 \lambda n$, then the choice of $w_1$ implies that $\Delta(H) \le 2 \lambda n$.
So there exists an $A'B'$-edge $f_1$ in $H'-V(U_1) = H-V(U_1)$ since $e(H-V(U_1)) \ge e(H) - |U_1| \Delta(H) \ge e(H) - 8 \lambda n >0$. 
Set $F_1 := M_1 + f_1 +f'_1$. Note that $f_1$ satisfies ($\alpha_2$) and that $F_1$ is a Hamilton exceptional system candidate with $e_{F_1}(A',B') = 2$.
By Lemma~\ref{lma:ESextend}, we can extend $F_1$ into a Hamilton exceptional system $J_1$ with parameter $\eps_0$ in
$G$ such that $F_1 \subseteq J_1$ and such that $J_1 - F_1$ only contains $A_0A$-edges and $B_0B$-edges of $G$.

Next, suppose that for some $1\le s\le \lambda n$ we have already constructed $J_0, \dots , \break J_{s-1}$ satisfying ($\alpha_1$)--($\alpha_4$). 
So $s \ge 2$ if $w_1w_2 \in M_1$. Let $G_s := G - \sum_{j=s}^{\lambda n} M_j - \sum_{j=0}^{s-1} J_{j} $ and $H_s := G_s[A',B']$.
Note that
\begin{align} \label{eqn:V_0BES}
e(H_s) & \ge e(H) - 2(s-1) -1  \ge D - 2 \lambda n .
\end{align}
Moreover, note that $d_{G_s}(v,A) \ge d_{G}(v,A) - 2(s-1)-1 \ge \sqrt{ \epszero} n $ for all $v \in A_0$
and $d_{G_s}(v,B) \ge \sqrt{ \epszero} n $ for all $v \in B_0$.

We%
   \COMMENT{Daniela: new sentence}
first pick the edge $f'_s$ as follows.
If $d_{H}(w_2) \ge 2\lambda n$, then $d_{H_s}(w_2) \ge d_{H}(w_2) - s \ge \lambda n$.%
    \COMMENT{Since each of $J_1,\dots,J_{s-1}$ contains precisely two $A'B'$-edges, these 2 edges will be independent.
So at most one of them can be incident to $w_2$.}
So we can pick an $A'B'$-edge $f'_s$ of $H_s$ such that $w_2$ is an endpoint of $f'_s$ and the connected component $U_s$ of $M_s+f'_s$ containing $f'_s$
does not contain~$w_1$. If $d_{H}(w_2) < 2 \lambda n$, then pick an $A'B'$-edge $f'_s$ of $H_s$ such that the connected component
$U_s$ of $M_s+f'_s$ containing $f'_s$ does not contain $w_1$. To see that such an edge exists, note that in this case the neighbour $w'_1$ of
$w_1$ in $M_s$ satisfies $d_H(w'_1)\le d_H(w_2)< 2 \lambda n$ (if $w'_1$ exists) and that~\eqref{eqn:V_0BES}
implies that $e(H_s) \ge D- 2 \lambda n >D/2+2 \lambda n\ge d_H(w_1)+2 \lambda n$.
Observe that in both cases $|U_s| \le 4$.

We%
   \COMMENT{Daniela: new sentence and more detail in the next 2 sentences}
now pick the edge $f_s$ as follows. If $d_{H}(w_1) \ge 2 \lambda n$, then $d_{H_s}(w_1) \ge d_{H}(w_1) - s \ge \lambda n$.
So we can find an $A'B'$-edge $f_s$ of $H_s$
such that $w_1$ is one endpoint of $f_s$ and the other endpoint of $f_s$ does not lie in $U_s$.
If $d_{H}(w_1) < 2 \lambda n$, then $\Delta(H) \le 2 \lambda n$ and thus~\eqref{eqn:V_0BES} implies that
\begin{align} \nonumber
e(H_s - V(U_s)) & \ge  D- 2\lambda n - 2 \lambda n |U_s|  \ge 1 .
\end{align}
So there exists an $A'B'$-edge $f_s$ in $H_s-V(U_s)$.

In all cases the edges $f_s$ and $f'_s$ satisfy ($\alpha_2$) and ($\alpha_3$).
Set $F_s := M_s +f_s +f'_s$. Clearly, $F_s$ is a Hamilton exceptional system candidate with $e_{F_s}(A',B') = 2$.
Recall that $d_{G_s}(v,A) \ge \sqrt{ \epszero} n $ for all $v \in A_0$ and $d_{G_s}(v,B) \ge \sqrt{ \epszero} n $ for all $v \in B_0$.
Thus by Lemma~\ref{lma:ESextend}, we can extend $F_s$ into a Hamilton exceptional system $J_s$ with parameter $\eps_0$
such that $F_s \subseteq J_s\subseteq G_s+F_s$ and such that $J_s - F_s$ only contains $A_0A$-edges and $B_0B$-edges of $G_s$.
Hence we have constructed $J_1, \dots ,J_{\lambda n }$ satisfying ($\alpha_1$)--($\alpha_4$). 
So (iii) holds. Note (ii) and (vi) are vacuously true.

To verify (iv), recall that $\mathcal{J}:=J_0\cup\dots \cup J_{\lambda n}$ and $H^{\diamond}= G[A', B'] - \mathcal{J}$.
For all $1 \le s \le \lambda n$ we have $e_{J_s}(A',B') = 2$ by~(iii).
Moreover, (i) and (i$'$) together imply that $e_{J_0}(A',B') = 1$ if and only if both $|A'|$ and $D$ are odd.
Therefore, $e_{\mathcal{J}}(A',B') \le \phi n $. Moreover, since $e(H^{\diamond})=e(H)-2\lambda n-e_{J_0}(A',B')$,
Proposition~\ref{prp:e(A',B')parity}(i) implies that $e(H^{\diamond})$ is even. Thus (iv) holds.

To verify (v), note that if $d_H(w_1) \le  2 \lambda n $ then clearly $d_{H^{\diamond}}(w_1)\le 2 \lambda n\le (D- \phi n )/2$.
If $d_H(w_1) \ge  2 \lambda n $ then ($\alpha_2$) implies that $d_{J_s[A',B']}(w_1) = 1$ for all $1 \le s \le \lambda n$.
Hence $d_{H^{\diamond}}(w_1) \le \lfloor D/2 \rfloor - \lambda n = (D- \phi n )/2$.
Now suppose that $D = n/2 - 1$ and so $n = 0 \pmod4$ by~\eqref{minexact}. Thus $D$ is odd and so (i$'$) implies that
if $w_1w_2$ is an edge in $G[A']+G[B']$, then $w_1w_2 \in J_0$. In particular $w_1w_2\notin M_1$.
(Note that if $w_1w_2 \in G[A',B']$, then $w_1w_2$ is not contained in $M_1$ either since $M_1\subseteq G[A_0]+G[B_0]$.)%
    \COMMENT{Daniela: slightly changed the previous sentence}
Thus in the case when $d_H(w_2) \ge 2 \lambda n$, ($\alpha_3$) implies that $d_{J_s[A',B']}(w_2) = 1$ for all $1 \le s \le \lambda n$.
Hence $d_{H^{\diamond}}(w_2) \le \lfloor D/2 \rfloor - \lambda n = (D- \phi n )/2$.
If $d_H(w_2) \le  2 \lambda n $ then clearly $d_{H^{\diamond}}(w_2)\le 2 \lambda n\le (D- \phi n )/2$.
Therefore (v) holds. 

\medskip

\noindent \textbf{Case 2: $e(H)<D$}

\smallskip

\noindent Together with Proposition~\ref{prp:e(A',B')}(ii) this implies that $n = 0 \pmod4$, $D = n/2-1$ and $|A'| = n/2 = |B'|$.
So $D$ is odd and $|A'|$ is even.
In particular, by Proposition~\ref{prp:e(A',B')parity}(i) $e(H)$ is even and by~(i) and~(i$'$) $J_0$ is a perfect matching with $e_{J_0}(A',B') =0$.
Moreover, Proposition~\ref{prp:e(A',B')3} implies that $\Delta(H) \le e(H) /2$ in this case
(recall that $H:=G[A',B']$).%
\COMMENT{Deryk: this bracket and the next sentence are new.}

Note that each $M_s$ is a matching exceptional system candidate.
By Lemma~\ref{lma:ESextend}, for each $1 \le s \le \min \{\ell, \lambda n \}$ in turn, we can extend $M_s$ into a matching exceptional system $J_s$
with parameter $\eps_0$ in $G'=G-J_0$ such that $M_s \subseteq J_s$, 
and such that $J_s$ and $J_{s'}$ are edge-disjoint whenever $1\le s' < s\le \min \{\ell, \lambda n \}$. Thus~(ii) holds.

If $\ell \ge \lambda n $, then $e(H) \le D - 2\lambda n  = D - \phi n+1$. But since $e(H)$ is even and $D - \phi n+1$ is odd this
means that $e(H) \le D - \phi n$. Thus $\Delta(H) \le e(H) /2 \le (D- \phi n ) /2$.
Moreover, $d_{\mathcal J}(v)=2 \lambda n+ d_{J_0}(v)= \phi n$ for all $v \in V_0$.
Hence (iv)--(vi) hold since $H^{\diamond} = H$. ((iii) is vacuously true.)

Therefore, we may assume that $\ell < \lambda n$. 
Using a similar argument as in Case~1, for all $\ell < s \le \lambda n$ we can extend the matchings $M_s$
into edge-disjoint Hamilton exceptional systems $J_s$ satisfying ($\alpha_1$)--($\alpha_4$) and which are edge-disjoint from $J_0,\dots,J_\ell$. 
Indeed, suppose that for $\ell < s\le \lambda n$ we have already constructed $J_{\ell+1}, \dots ,J_{s-1}$ satisfying ($\alpha_1$)--($\alpha_4$). 
(Note that (i$'$) implies that the exception in~($\alpha_3$) is not relevant.) The fact that $D$ is odd and $e(H)$ is even implies that $\ell=(D-e(H)-1)/2$.
Then defining $H_s$ analogously to Case~1, we have%
   \COMMENT{Daniela: previously had $e(H_s)  \ge e(H) - 2(s-1-\ell)-1$ below, deleted the $-1$ and added some detail after the display}
$$ 
e(H_s)  \ge e(H) - 2(s-1-\ell) = D - 2s  \ge D-2 \lambda n,
$$
where in the first inequality we use that $e_{J_0}(A',B')=0$ by~(i$'$).
So the analogue of~\eqref{eqn:V_0BES} holds. 
Hence we can proceed exactly as in Case~1 to construct $J_s$ (the remaining calculations go through as before).%
\COMMENT{Deryk: previously had `are changed slightly but not significantly 
by the removal of the matching exceptional systems from $G$).' but I don't see the difference}
Thus~(iii) holds.

To verify~(iv), 
note that $e_{\mathcal{J}}(A',B') = 2(\lambda n - \ell )$. So
\begin{align}
e(H^{\diamond}) = e(H) - 2(\lambda n - \ell ) = e(H) - 2 \lambda n + (D- e(H)-1)  = D - \phi n \label{eqn:eHdia}.
\end{align}
In particular, $e(H^{\diamond})$ is even and $e_\mathcal{J}(A',B')=e(H)-e(H^{\diamond})<\phi n$. So (iv) holds.

In order to verify~(vi), recall that $\Delta(H) \le e(H)/2$. Moreover, note that ($\alpha_2$) implies that if $d_H (w_1) \ge 2 \lambda n$,
then $d_{J_s[A',B']}(w_1) = 1$ for all $\ell <  s \le \lambda n$.
Hence 
\begin{align}
	d_{H^{\diamond}}(w_1) & = d_{H}(w_1) - (\lambda n - \ell) = \Delta(H) -  \lambda n + \ell \nonumber \\
	& \le e(H)/2 - \lambda n + (D - e(H)-1)/2 = (D- \phi n )/ 2 \stackrel{\eqref{eqn:eHdia}}{=} e( H^{\diamond} )/2. \nonumber
\end{align}
Similarly if $d_H (w_2) \ge 2 \lambda n$, then $d_{H^{\diamond}}(w_2) \le e(H^{\diamond})/2$.
If $d_H (w_1) \le 2 \lambda n$, then $d_{H^{\diamond}}(w_1) \le 2 \lambda n\le e(H^{\diamond})/2$ by \eqref{eqn:eHdia} and
the analogue also holds for $w_2$. Thus in all cases $d_H (w_1), d_H (w_2) \le e(H^{\diamond})/2$.
Our choice of $w_1$ and $w_2$ implies that for all $v \in V(G) \setminus W$ we have
$$d_{H}(v) \le (e(H)+3) /3 \le (D+3)/3 \stackrel{\eqref{eqn:eHdia}}{<} e(H^{\diamond})/2.$$
Therefore, $\Delta(H^{\diamond}) \le e(H^{\diamond})/2$. Together with \eqref{eqn:eHdia} this implies~(vi) and thus~(v).
\endproof

The next lemma implies that each of the exceptional systems $J_s$ guaranteed by Lemma~\ref{V_0BES} can
be extended into a Hamilton cycle (if $J_s$ is a Hamilton exceptional system) or into two perfect matchings
(if $J_s$ is a matching exceptional system and both $|A'|$ and $|B'|$ are even).

\begin{lemma} \label{badBES}
Suppose that $ 0 < 1/n \ll \epszero \le  \lambda \ll 1$ and that $n, \lambda n,K \in \mathbb{N}$.
Suppose that $(G,A,A_0,B,B_0)$ is an $(\epszero,K)$-framework such that $\delta(G[A]) \ge 4|A|/5$ and $\delta(G[B]) \ge 4 |B| /5$.%
     \COMMENT{$K$ is not important here. Previously had instead "Let $V$ be a set of $n$ vertices and let $A,A_0,B,B_0$ be a partition of $V$.
Suppose that $G$ is a graph on $V$ such that $\delta(G[A]) \ge 4|A|/5$ and $\delta(G[B]) \ge 4 |B| /5$."
But if we state the lemma this way, we also need to add that $|B|$ is close to $n/2$. So I thought it's shorter to
talk about frameworks.}
Let $J_1, \dots ,J_{ \lambda n }$ be exceptional systems with parameter $\eps_0$.
Suppose that $G$ and $J_1, \dots ,J_{\lambda n}$ are pairwise edge-disjoint.
Then there are edge-disjoint subgraphs $H_1, \dots ,H_{\lambda n }$ in $G + \sum_{s=1}^{\lambda n} J_s$ which satisfy the following properties:%
   \COMMENT{Daniela: previously had $H_s - J_s\subseteq G$ instead of $E(H_s - J_s)\subseteq E(G[A]+G[B])$, but in the proof of
Lemma~\ref{V_0elimination} we need that we don't add any $A'B'$-edges}
\begin{itemize}
	\item[{\rm (i)}] $J_s\subseteq H_s$ and $E(H_s - J_s)\subseteq E(G[A]+G[B])$ for all $1\le s\le \lambda n$. 
	\item[{\rm (ii)}] If $J_s$ is a Hamilton exceptional system, then $H_s$ is a Hamilton cycle on $V(G)$. 
	\item[{\rm (iii)}] If $J_s$ is a matching exceptional system, then $H_s$ is an union of a Hamilton cycle on $A'=A\cup A_0$
and a Hamilton cycle on~$B'=B\cup B_0$.
\end{itemize}
\end{lemma}
\begin{proof}
Recall that, given an exceptional system $J$, we have defined matchings $J^*_A$, $J^*_B$ and $J^*=J^*_A+J^*_B$ in Section~\ref{sec:BES}.
We will write $J_{s,A}^*:=(J_s)_A^*$ and $J_{s,B}^*:=(J_s)_B^*$.
For each $s \le \lambda n$ in turn, we will find a subgraph $H_s^*$ of $G[A]+G[B] +J_s^*$ containing%
    \COMMENT{Daniela: replaced $G+J_s^*$ by $G[A]+G[B] +J_s^*$}
$J_s^*$ such that $H_s^*$ is edge-disjoint from $H_1^*, \dots ,H_{s-1}^*$. 
Moreover, $H_s^*$ will be the union of two cycles $C_A$ and $C_B$ such that $C_A$ is a Hamilton cycle on $A$ which is consistent with
$J_{s,A}^*$ and $C_B$  is a Hamilton cycle on $B$ which is consistent with $J_{s,B}^*$.
(Recall from Section~\ref{sec:BES} that we always view different
$J^*_i$ as being edge-disjoint from each other. So asking $H_s^*$ to be edge-disjoint from $H_1^*, \dots ,H_{s-1}^*$
is the same as asking $H_s^*-J^*_s$ to be edge-disjoint from $H_1^*-J^*_1, \dots ,H_{s-1}^*-J^*_{s-1}$.)

Suppose that for some $1\le s\le \lambda n$ we have already found $H_1^*, \dots ,H_{s-1}^*$. For all $i<s$, let $H_i:=H^*_i-J^*_i+J_i$.
Let $G_s:=G-(H_1\cup\dots \cup H_{s-1})$.
First we construct $C_A$ as follows. Recall from (\ref{ESeq}) that $J_{s,A}^*$ is a matching of size at most $2\sqrt{\epszero} n$.%
\COMMENT{Deryk: new extra factor of two here}
Note that $\delta(G_s[A]) \ge \delta(G[A]) - 2s \ge (4/5 - 5 \lambda n ) |A|$.
So we can greedily find a path $P_A$ of length at most $6 \sqrt{\epszero} n$ in $G_s[A] + J_{s,A}^*$ such that
$P_A$ is consistent with $J_{s,A}^*$. Let $u$ and $v$ denote the endpoints of $P_A$.
Let $G_s^A$ be the graph obtained from $G_s[A]-V(P_A)$ by adding a new vertex $w$ whose neighbourhood is precisely
$(N_{G_s}(u) \cap N_{G_s}(v)) \setminus V(P_A)$.
Note that $\delta(G_s^A) \ge |G_s^A|/2$ (with room to spare). Thus $G_s^A$ contains a Hamilton cycle $C'_A$ by Dirac's theorem.
But $C'_A$ corresponds to a Hamilton cycle $C_A$ of $G_s[A]+J_{s,A}^*$ that is consistent with $J_{s,A}^*$.
Similarly, we can find a Hamilton cycle $C_B$ of $G_s[B] +J_{s,B}^*$ that is consistent with $J_{s,B}^*$.
Let $H^*_s=C_A+C_B$. This completes the construction of $H_1^*, \dots ,H_{\lambda n}^*$.

For each $1\le s\le \lambda n$ we take $H_s : = H_s^* - J_s^* +J_s$. Then~(i) holds.
Proposition~\ref{prop:ES} implies (ii) and (iii).
\end{proof}

By combining Lemmas~\ref{V_0BES} and~\ref{badBES} we obtain the following result, which guarantees a set of edge-disjoint
Hamilton cycles covering all edges of $G[A_0]$ and $G[B_0]$.

\begin{lemma} \label{V_0elimination}
Suppose that $ 0 < 1/n \ll \epszero \ll \phi \ll 1$ and that $D, n, (D- \phi n )/2,K \in \mathbb{N}$.%
   \COMMENT{Again, $K$ is not important. Previously, we have $\phi n \in \mathbb{N}$, but this is implied by $D, (D- \phi n )/2 \in \mathbb{N}$.}
Let $G$ be a $D$-regular graph on $n$ vertices with $D \ge n - 2\lfloor n/4 \rfloor -1$.
Suppose that $(G,A,A_0,B,B_0)$ is an $(\epszero,K)$-framework with $\Delta (G[A',B']) \le D/2$.
Let $w_1$ and $w_2$ be (fixed) vertices of $G$ such that $d_{G[A',B']}(w_1)\ge d_{G[A',B']}(w_2)\ge d_{G[A',B']}(v)$ for all $v\in V(G)\setminus \{w_1,w_2\}$.
Then there exists a $\phi n $-regular spanning subgraph $G_0$ of $G$ which satisfies the following properties:
\begin{itemize}
	\item[{\rm (i)}] $G[A_0]+G[B_0] \subseteq G_0$. 
	\item[{\rm (ii)}] $e_{G_0}(A',B') \le \phi n$ and $e_{G - G_0}(A',B')$ is even.
	\item[{\rm (iii)}] $G_0$ can be decomposed into $\lfloor e_{G_0}(A',B')/2 \rfloor$ Hamilton cycles and
$\phi n - 2\lfloor e_{G_0}(A',B')/2 \rfloor$ perfect matchings. Moreover, if $e_G(A',B') \ge D$, then
this decomposition of $G_0$ uses $\lfloor \phi n/2 \rfloor$ Hamilton cycles and one perfect matching if $D$ is odd.
    \item[{\rm (iv)}] Let $H^{\diamond}:=G[A',B'] - G_0$. Then $ d_{H^{\diamond}}(w_1) \le (D-\phi n)/2$.
	Furthermore, if $D = n/2-1$ then $ d_{H^{\diamond} }(w_2) \le (D-\phi n)/2$.
	\item[{\rm (v)}] If $e_G(A',B') < D$, then $ \Delta( H^{\diamond} ) \le e(H^{\diamond})/2 \le (D-\phi n)/2$.
\end{itemize}
\end{lemma}
\begin{proof}
Let
\begin{align*}
\ell & := \left\lfloor \frac{ \max \{0 , D - e_G(A',B') \} }2 \right\rfloor \  \text{and}  
& \lambda n := \lfloor \phi n /2 \rfloor = & \begin{cases} 
(\phi n -1)/2 & \textrm{if $D$ is odd,}\\
\phi n/2  & \textrm{if $D$ is even.}
\end{cases}
\end{align*}
(The last equality holds since our assumption that $(D- \phi n )/2 \in \mathbb{N}$
implies that $D$ is odd if and only if $\phi n$ is odd.)
So $\ell$, $\phi $ and $\lambda$ are as in Lemma~\ref{V_0BES}.
Thus we can apply Lemma~\ref{V_0BES} to $G$ in order to obtain $\lambda n +1$ subgraphs $J_0, \dots ,J_{\lambda n}$ as described there.
Let $G'$ be the graph obtained from $G[A']+G[B']$ by removing all the edges in $J_0 \cup \dots \cup J_{\lambda n}$.
Recall that $J_0$ is either a perfect matching in $G$ or empty. 
Since each of $J_1, \dots ,J_{\lambda n}$ is an exceptional system and so by (EC3) we have $e_{J_s}(A)=0$ for all $1\le s \le \lambda n$,
it follows that $\delta(G'[A])\ge \delta(G[A])-1\ge 4|A|/5$, where%
   \COMMENT{Daniela: had $\delta(G'[A])\ge \delta(G[A])-1\ge \delta(G[A'])-|A_0|-1\ge 4|A|/5$, but then the last inequality is wrong}
the final inequality follows from (FR3) and (FR4).
Similarly $\delta(G'[B])\ge 4|B|/5$.
So we can apply Lemma~\ref{badBES} with $G'$ playing the role of $G$ in order to extend $J_1, \dots ,J_{\lambda n }$ into
edge-disjoint subgraphs $H_1, \dots ,H_{\lambda n }$ of $G'+ \sum_{s=1}^{\lambda n} J_s$
such that
\begin{itemize}
\item[(a)] $H_s$ is a Hamilton cycle on $V(G)$ which contains precisely two $A'B'$-edges for all $\ell<s\le \lambda n$;
\item[(b)] $H_s$ is the union of a Hamilton cycle on $A'$ and a Hamilton cycle on~$B'$ for all $1\le s\le \min \{ \ell,\lambda n \}$.
\end{itemize}
Indeed, the property $e_{H_s}(A',B')=2$ in (a) follows from Lemma~\ref{V_0BES}(iii) and \ref{badBES}(i).%
   \COMMENT{Daniela: new sentence}
Let $G_0: = J_0 + \sum_{s=1}^{\lambda n} H_s$.
Then~(i) holds since by Lemma~\ref{V_0BES} all the $J_0, \dots ,J_{\lambda n}$ together
cover all edges in $G[A_0]$ and $G[B_0]$.
Let $\mathcal{J}_{\rm HC}$ be the union of all $J_s$ with $\ell<s\le \lambda n$ and
let $\mathcal{J}$ be the union of all $J_s$ with $0\le s\le \lambda n$.
The definition of $G_0$,%
   \COMMENT{Daniela: had $G'$ instead of $G_0$}
Lemma~\ref{V_0BES}(ii),(iii) and Lemma~\ref{badBES}(i) together imply that
$G_0[A',B']=\mathcal{J}[A',B']=J_0[A',B']+\mathcal{J}_{\rm HC}[A',B']$ and so%
\COMMENT{Deryk: I split this display to make clearer what we use when}
\begin{align}
e_{G_0}(A',B') & =e_{\mathcal{J}}(A',B') \label{eq:edgesG0} \\
& =e_{J_0}(A',B')+2(\max\{0,\lambda n-\ell\}). \label{eq:edgesG02}
\end{align}
Together with Lemma~\ref{V_0BES}(iv), (\ref{eq:edgesG0}) implies~(ii).
Moreover, the graph $H^{\diamond}$ defined in~(iv) is the same as the graph $H^{\diamond}$ defined
in Lemma~\ref{V_0BES}(iv). Thus (iv) and (v) follow from Lemma~\ref{V_0BES}(v) and~(vi).

So it remains to verify~(iii).
Note that if $\ell>0$ then $e_G(A',B') < D$
and so $n = 0 \pmod4$, $D = n/2-1$ and $|A'| = n/2 = |B'|$ by Proposition~\ref{prp:e(A',B')}(ii). In particular,
both $A'$ and $B'$ are even and so for all $1\le s\le \ell$ the graph $H_s$ can be decomposed into two edge-disjoint
perfect matchings. Recall that by Lemma~\ref{V_0BES}(i) the graph $J_0$ is a perfect matching if $D$ is odd
and empty if $D$ is even. 
Thus, if $ \ell \le \lambda n$, then $G_0$ can be decomposed into $\lambda n-\ell$ edge-disjoint Hamilton cycles and $n_{\rm match}$ edge-disjoint
perfect matchings, where $n_{\rm match}=2\ell$ if $D$ is even and $n_{\rm match}=2\ell+1$ if $D$ is odd.
In particular, this implies the `moreover part' of (iii) (since $\ell=0$ if $e_G(A',B')\ge D$).
Also, \eqref{eq:edgesG02} together with the fact that $e_{J_0}(A',B')\le 1$ by Lemma~\ref{V_0BES}(i) implies that
$\lambda n-\ell = \lfloor e_{G_0}(A',B')/2 \rfloor$ and so
$\phi n - 2\lfloor e_{G_0}(A',B')/2 \rfloor=n_{\rm match}$. Thus~(iii) holds in this case.
If $ \ell > \lambda n$, then (a) implies that there are no Hamilton cycles at all in the decomposition.
Also~\eqref{eq:edgesG02} implies that $\lfloor e_{G_0}(A',B')/2 \rfloor =0$, as required in~(iii).
Similarly, (b) implies that $n_{\rm match}=2 \lambda n$ if $D$ is even and $n_{\rm match}=2 \lambda n+1$ if $D$ is odd, which also agrees with~(iii).
\end{proof}

\section{Constructing Localized Exceptional Systems}\label{sec:locES}

Suppose that $(G,A,A_0,B,B_0)$ is an $(\epszero,K)$-framework and that $G_0$ is the spanning subgraph of our given $D$-regular graph $G$ obtained by Lemma~\ref{V_0elimination}.
Set $G' := G - G_0$. (So $G'$ has no edges inside $A_0$ or $B_0$.) 
Roughly speaking, the aim of this section is to decompose $G'- G'[A] - G'[B]$ into edge-disjoint exceptional systems.
Each of these exceptional systems $J$ will then be extended into a Hamilton cycle (in the case when $J$ is a Hamilton exceptional system) or into
two perfect matchings (in the case when $J$ is a matching exceptional system). We will ensure that all but a small number of these
exceptional systems are localized (with respect to some $(K,m, \epszero)$-partition $\mathcal{P}$ of $V(G)$ refining
the partition $A,A_0,B,B_0$). Moreover, for all $1\le i, i' \le K$, the number of $ (i,i')$-localized exceptional systems in our decomposition
will be the same. (Recall that $ (i,i')$-localized exceptional systems were defined in Section~\ref{sec:BES}.) 

However, rather than decomposing the above `leftover' $G'- G'[A] - G'[B]$ in a single step, we actually need to proceed in two steps:
initially, we find a small number of exceptional systems $J$ which have some additional useful properties (e.g.~the number of $A'B'$-edges of $J$
is either zero or two). These exceptional systems will be used to construct the robustly decomposable graph $G^{\rm rob}$.
(Recall that the role of $G^{\rm rob}$ was discussed in Section~\ref{overview}.)%
    \COMMENT{Daniela: new sentence}
Let $G'':=G-G_0-G^{\rm rob}$. Some of the additional properties of the exceptional systems contained in $G^{\rm rob}$ then allow us to 
find the desired decomposition of  $G^{\diamond} := G''- G''[A] - G''[B]$.
(We need to proceed in two steps rather than one as we have little control over the structure of $G^{\rm rob}$.)

Recall that
in order to construct the required (localized) exceptional systems, we will distinguish three cases:
\begin{itemize}
\item[(a)] the case when $G$ is `non-critical' and contains at least $D$ $A'B'$-edges (see Lemma~\ref{lma:BESdecom}); 
\item[(b)] the case when $G$ is `critical' and contains
at least $D$ $A'B'$-edges (see Lemma \ref{lma:BESdecomcritical});
\item[(c)] the case when $G$ contains less than $D$ $A'B'$-edges (see Lemma~\ref{lma:PBESdecom}).  
\end{itemize}
Each of the three lemmas above is formulated in such a way that we can apply it twice:
firstly to obtain the small number of exceptional systems needed for the robustly decomposable graph $G^{\rm rob}$ and secondly for the 
decomposition of the graph $G^{\diamond}$ into exceptional systems. The proofs of all the results in this section are deferred until Chapter~\ref{paper4}.

\subsection{Critical Graphs}

Roughly speaking, $G$ is critical if most of its $A'B'$-edges are incident to only a few vertices. More precisely, given a partition $A',B'$ of $V(G)$ and $D \in \mathbb N$, we say that
$G$ is \emph{critical} (with respect to $A',B'$ and $D$) if both of the following hold:
\begin{itemize}
\item $\Delta(G[A',B']) \ge 11 D/40$;
\item  $e(H) \le 41 D/40$ for all subgraphs $H$ of $G[A',B']$ with $\Delta(H) \le 11 D/40$.%
     \COMMENT{Note that there is no assumption on $e_G(A',B')$.}
\end{itemize}
Note that the property of $G$ being critical depends only on $D$ and the partition $A'=A\cup A_0$ and $B'=B\cup B_0$ of $V(G)$, which is
fixed after we have applied Proposition~\ref{prop:framework} to obtain a framework $(G,A,A_0,B,B_0)$.
In particular, it does not depend on the choice of the $(K,m, \epszero)$-partition $\mathcal{P}$ of $V(G)$ refining $A,A_0,B,B_0$.
(In the proof of Theorem~\ref{1factstrong} we will fix a framework $(G,A,A_0,B,B_0)$, but will then choose two different
partitions refining $A,A_0,B,B_0$.)%
     \COMMENT{i.e. $G$ will not change from being critical to not.}

One example of a critical graph is the following:
$G_{\rm crit}$ consists of two disjoint cliques on $(n-1)/2$ vertices with vertex set $A$ and $B$ respectively, where $n = 1 \pmod{4}$.
In addition, there is a vertex $a$ which is adjacent to exactly half of the vertices in each of $A$ and $B$.
Also, add a perfect matching $M$ between those vertices of $A$ and those vertices in $B$ not adjacent to $a$.
Let  $A':=A \cup \{a\}$, $B':=B$ and $D:=(n-1)/2$.
Then $G_{\rm crit}$ is critical, and $D$-regular with $e(A',B')=D$. Note that $e(M)=D/2$.
To obtain a Hamilton decomposition of $G_{\rm crit}$, 
we will need to decompose $G_{\rm crit}[A',B']$ into $D/2$ Hamilton exceptional system candidates $J_s$
(which need to be matchings of size exactly two in this case).
In this example, this decomposition is essentially unique:%
   \COMMENT{Daniela: added essentially}
every $J_s$ has to consist of exactly one edge in $M$ and one edge incident to $a$.
Note that in this way, every edge between $a$ and $B$ yields a `connection'  (i.e.~a maximal path) between $A'$ and $B'$ required in (ESC4).

The following lemma (proved in Section~\ref{secnewzz}) collects some properties of critical graphs.
In particular, there is a set $W$ consisting of  between one and three vertices
with many neighbours in both $A$ and $B$. 
We will need to use $A'B'$-edges incident to one or two vertices of $W$
to provide `connections' between $A'$ and $B'$ when constructing 
the Hamilton exceptional system candidates in the critical case~(b).

\begin{lemma} \label{critical}
Suppose that $0< 1/n \ll 1$ and that $D, n \in \mathbb N$ with
$D\ge n - 2\lfloor n/4 \rfloor -1$.
Let $G$ be a $D$-regular graph on $n$ vertices and let $A',B'$ be a partition of $V(G)$ with $|A'| , |B'| \ge D/2$ and $\Delta(G[A',B']) \le D/2$.
Suppose that $G$ is critical.
Let $W$ be the set of vertices $w \in V(G)$ such that $d_{G[A',B']}(w) \ge 11D/40$.
Then the following properties are satisfied:
\begin{itemize}
	\item[$ \rm (i)$] $1 \le |W| \le 3$.
	\item[$ \rm (ii)$] Either $D = (n-1)/2$ and $n = 1 \pmod{4}$, or $D = n/2-1$ and $n = 0 \pmod{4}$.
	Furthermore, if $n =1 \pmod{4}$, then $|W|=1$.
	\item[$\rm (iii)$] $e_{G}(A',B') \le 17D/10+5 < n$.
\end{itemize}
\end{lemma}

Recall from Proposition~\ref{prp:e(A',B')}(ii) that we have $e_G(A',B') \ge D$ unless $D = n/2 - 1$, $n = 0 \pmod{4}$ and $|A| = |B| = n/2$.
Together with Lemma~\ref{critical}(ii) this shows that in order to find the decomposition into exceptional
systems, we can distinguish the following three cases.%
   \COMMENT{Daniela: previously had $|A'| \ge |B'| \ge D/2$ in the corollary}

\begin{cor}\label{cor:cases}
Suppose that $0< 1/n  \ll 1$ and that $D, n \in \mathbb N$ with $D \ge n - 2 \lfloor n/4 \rfloor -1$.
Let $G$ be a $D$-regular graph on $n$ vertices and let $A',B'$ be a partition of $V(G)$ with $|A'|,|B'| \ge D/2$ and $\Delta(G[A',B']) \le D/2$.
Then exactly one of the following holds:
\begin{itemize}
	\item[\rm (a)] $e_G(A',B') \ge D$ and $G$ is not critical.
	\item[\rm (b)] $e_G(A',B') \ge D$ and $G$ is critical. In particular, $e_G(A',B')< n$ and either $D = (n-1)/2$ and
$n = 1 \pmod{4}$, or $D = n/2-1$ and $n = 0 \pmod{4}$.
	\item[\rm (c)] $e_G(A',B') < D$. In particular, $D = n/2 - 1$, $n = 0 \pmod{4}$ and $|A| = |B| = n/2$.
\end{itemize}
\end{cor}

\subsection{Decomposition into Exceptional Systems}
Recall from the beginning of Section~\ref{sec:locES} that our aim is
to find a decomposition of $G - G_0 - G[A] -G[B]$ into suitable exceptional systems (in particular, most of
these exceptional systems have to be localized). The following lemma (proved in Section~\ref{noncritical}) states that this can be done if we are in
Case~(a) of Corollary~\ref{cor:cases}, i.e.~if $G$ is not critical and $e_G(A',B') \ge D$. 

\begin{lemma} \label{lma:BESdecom}
Suppose that $0 <  1/n  \ll \epszero \ll  \eps \ll   \lambda, 1/K \ll 1$, that $D\ge n/3$, that $0 \le \phi  \ll 1$
and that%
    \COMMENT{Previously also had that $\phi n \in \mathbb{N}$. But this follows since $D,K, (D - \phi n)/(2K^2) \in \mathbb{N}$.
Also, in this lemma we only need that $(D - \phi n)/(2K^2) \in \mathbb{N}$ instead of $(D - \phi n)/(400K^2) \in \mathbb{N}$.
But the latter is needed in Lemma~\ref{lma:BESdecomcritical}. So it's easier to require it here too.}
$D, n, K, m, \lambda n/K^2, (D - \phi n)/(2K^2) \in \mathbb{N}$. Suppose  
that the following conditions hold:
\begin{itemize}
	\item[{\rm(i)}] $G$ is a $D$-regular graph on $n$ vertices.%
\COMMENT{Note that $D$ satisfies $1/3 \le D/n < 1$ not $D \ge n - 2\lfloor n/4\rfloor -1$.}
	\item[{\rm(ii)}] $\mathcal{P}$ is a $(K, m, \epszero)$-partition of $V(G)$ such that $D \le e_G(A',B') \le \epszero n^2$ and $\Delta(G[A',B']) \le D/2$.
Furthermore, $G$ is not critical.
	\item[{\rm(iii)}] $G_0$ is a subgraph of $G$ such that $G[A_0]+G[B_0] \subseteq G_0$, $e_{G_0}(A',B') \leq \phi n$ and $d_{G_0}(v) = \phi n $ for all $v \in V_0$.
	\item[{\rm(iv)}] Let $G^{\diamond} := G - G[A] - G[B] -G_0$.  $e_{G^\diamond} (A',B')$ is even and $(G^{\diamond}, \mathcal{P})$
is a $(K, m, \epszero,\eps)$-exceptional scheme.
\end{itemize}
Then there exists a set $\mathcal{J}$ consisting of $(D-\phi n)/2$ edge-disjoint Hamilton exceptional systems with parameter $\eps_0$ in $G^{\diamond}$ which satisfies the
following properties: 
\begin{itemize}
    \item[\rm (a)] Together all the Hamilton exceptional systems in $\mathcal{J}$ cover all edges of $G^{\diamond}$.
	\item[\rm (b)] For all $1\le i,i' \le K$, the set $\mathcal{J}$ contains $(D - (\phi + 2 \lambda) n)/(2K^2)$ $(i,i')$-HES.
	Moreover, $\lambda n /K^2$ of these $(i,i')$-HES $J$ are such that $e_J(A',B') =2$.
\end{itemize}
\end{lemma}

Note that (b) implies that $\mathcal{J}$ contains $\lambda n$ Hamilton exceptional systems which might not be localized.
This will make them less useful for our purposes and we extend them into Hamilton cycles in a separate step.
On the other hand, the lemma is `robust' in the sense that we can remove a sparse subgraph $G_0$ before we find the decomposition $\mathcal{J}$
into Hamilton exceptional systems. In our first application of Lemma~\ref{lma:BESdecom} (i.e.~to construct the
exceptional systems for the robustly decomposable graph $G^{\rm rob}$),
we will let $G_0$ be the graph obtained from Lemma~\ref{V_0elimination}.
In the second application, $G_0$ also includes $G^{\rm rob}$.%
     \COMMENT{Also, note that $G_0$ is not assumed to be regular.} 
In our first application of Lemma~\ref{lma:BESdecom}, we will only use the
$(i,i')$-HES $J$ with $e_J(A',B') =2$.

The next lemma is an analogue of Lemma~\ref{lma:BESdecom} for the case when $G$ is critical and $e_G(A',B') \ge D$.
By Corollary~\ref{cor:cases}(b) we know that in this case $D = (n-1)/2$ or $D = n/2-1$.
(Again we defer the proof to Section~\ref{sec:critical}.)

\begin{lemma}\label{lma:BESdecomcritical}
Suppose that $0 <  1/n  \ll \epszero \ll  \eps \ll   \lambda, 1/K \ll 1$, that $D\ge n - 2\lfloor n/4 \rfloor -1$,%
    \COMMENT{Daniela: added the bound on $D$, which we need to apply Lemma~\ref{critical'}}
that $0\le \phi\ll 1$ and that
$n, K, m, \lambda n/K^2, (D - \phi n)/(400K^2) \in \mathbb{N}$.%
    \COMMENT{Previously also had that $\phi n \in \mathbb{N}$. But this follows since $D,K, (D - \phi n)/(400K^2) \in \mathbb{N}$.}
Suppose that the following conditions hold:
\begin{itemize}
	\item[{\rm(i)}] $G$ is a $D$-regular graph on $n$ vertices.
	\item[{\rm(ii)}] $\mathcal{P}$ is a $(K, m, \epszero)$-partition of $V(G)$ such that%
\COMMENT{Previously this also included
$d_{G[A',B']}(v)\le \eps_0 n \text{ for all } v\in A\cup B.$ However, we can replaced this by the bound on $d_{G^\diamond[A',B']}(v)$ implied by~(iv).} 
$e_G(A',B') \ge D$ and \\ $\Delta(G[A',B']) \le D/2$.
Furthermore, $G$ is critical. In particular, $e_G(A',B')$ $< n$ and $D = (n-1)/2$ or $D= n/2-1$ by Lemma~\ref{critical'}(ii) and~(iii).%
	\COMMENT{Also $n = 1 \pmod4$ or $n = 0 \pmod4$ respectively.}
	\item[{\rm(iii)}] $G_0$ is a subgraph of $G$ such that $G[A_0]+G[B_0] \subseteq G_0$, $e_{G_0}(A',B') \leq \phi n$ and $d_{G_0}(v) = \phi n $ for all $v \in V_0$.
	\item[{\rm(iv)}] Let $G^{\diamond} := G - G[A] - G[B] -G_0$. $e_{G^\diamond} (A',B')$ is even and
$(G^{\diamond}, \mathcal{P})$ is a $(K, m, \epszero,\eps)$-exceptional scheme.
	\item[{\rm(v)}] Let $w_1$ and $w_2$ be (fixed) vertices such that $d_{G[A',B']}(w_1) \ge d_{G[A',B']}(w_2) \ge d_{G[A',B']}(v)$
for all $v \in V(G) \setminus\{w_1,w_2\}$.
Suppose that
\begin{equation} \label{additional}
d_{G^{\diamond}[A',B']}(w_1), d_{G^{\diamond}[A',B']}(w_2) \le (D-\phi n)/2.
\end{equation}
\end{itemize}
Then there exists a set $\mathcal{J}$ consisting of $(D-\phi n)/2$ edge-disjoint Hamilton exceptional systems with parameter $\eps_0$ in $G^{\diamond}$
which satisfies the following properties: 
\begin{itemize}
    \item[\rm (a)] Together the Hamilton exceptional systems in $\mathcal{J}$ cover all edges of $G^\diamond$.
	\item[\rm (b)] For each $1\le i,i' \le K$, the set $\mathcal{J}$ contains $(D - (\phi + 2 \lambda) n)/(2K^2)$ $(i,i')$-HES.
	Moreover, $\lambda n /K^2$ of these $(i,i')$-HES are such that
\begin{itemize}
\item[\rm (b$_1$)] $e_J(A',B') =2$ and 
\item[\rm (b$_2$)] $d_{J[A',B']}(w) = 1$ for all $w \in \{w_1,w_2\}$ with $d_{G[A',B']}(w) \ge 11D/40$.%
    \COMMENT{i.e. $w \in W' \cap \{w_1,w_2\}$}
\end{itemize}
\end{itemize}
\end{lemma}

Similarly as for Lemma~\ref{lma:BESdecom}, (b) implies that $\mathcal{J}$ contains $\lambda n$ Hamilton exceptional systems
which might not be localized. Another similarity is that when constructing the robustly decomposable graph $G^{\rm rob}$, 
we only use those Hamilton exceptional systems $J$ which have some additional useful properties, namely~(b$_1$) and~(b$_2$) in this case. 
This guarantees that~\eqref{additional} will be satisfied in the second application of
Lemma~\ref{lma:BESdecomcritical} (i.e.~after the removal of $G^{\rm rob}$), by `tracking' the degrees of the high
degree vertices $w_1$ and $w_2$.
Indeed, if  $d_{G[A',B']}(w_2)\ge 11D/40$, then (b$_2$) will imply that $d_{G^{\rm rob}[A',B']}(w_i)$ is large for $i=1,2$.
This in turn means that after removing $G^{\rm rob}$, in the leftover graph $G^\diamond$, $d_{G^\diamond[A',B']}(w_i)$ is comparatively small, 
i.e.~condition~\eqref{additional} will hold in the  second application  of Lemma~\ref{lma:BESdecomcritical}.%
    \COMMENT{We construct $G^{\text rob}$ as follows.
We pick those $J_s$ satisfying (b$_1$) and (b$_2$) to construct $G^{\text rob}$.
Let $G'' : = G - G_0 - G^{\text rob}$, where $G_0$ is the graph defined by Lemma~\ref{V_0elimination}.
Next, we apply Lemma~\ref{lma:BESdecomcritical} to $G''$ taking $G_0 = G_0 +G^{rob}$ to decompose $(G'')^{\diamond} := G'' - G''[A] -G''[B]$ into exceptional systems for almost decomposition on a different partition $\mathcal{P}'$, which is found by applying Lemma~\ref{lma:partition2} to $G''$. 
Notice that $(G'')^{\diamond}$ is basically the union of $J_s$ that are not used to constructed $G^{\text rob}$.
Since $G^{\text rob}$ contains those $J_s$ satisfying (b$_1$) and (b$_2$), the degree of $w_1$ in $G''[A',B']$ does indeed decrease if it is large, and similarly for $w_2$. 
Hence, (v) is holds for $(G'')^{\diamond}$.
The reason of using a different $\mathcal{P}'$ rather than refinement of $\mathcal{P}$ is because we do not have any control on the structure of $G^{rob}$.
Hence, $( G''[A]+G''[B], \mathcal{P}) $ is $(K, m, \epszero, \eps ')$-scheme, where $\eps'$ is too large.
Moreover, taking refinement of $\mathcal{P}$ does not decrease $\eps'$.
}

Condition~\eqref{additional} itself is natural for the following reason:
suppose for example that it is violated  for $w_1$ and that $w_1 \in A_0$. 
Then for some Hamilton exceptional system $J$ returned by the lemma, both edges of $J$ incident to $w_1$ will
have their other endpoint in $B'$. So (the edges at) $w_1$ cannot be used as a `connection' between $A'$ and $B'$
in the Hamilton cycle which will extend $J$, and it may be impossible to find these connections elsewhere.

The next lemma is an analogue of Lemma~\ref{lma:BESdecom} for the case when $e_G(A',B')< D$. (Again we defer the proof to Section~\ref{1factsec}.)
Recall that Proposition~\ref{prp:e(A',B')}(ii) (or Corollary~\ref{cor:cases}) implies that in this case we have
$n = 0 \pmod4$, $D = n/2-1$ and $|A'| = |B'| = n/2$.
In particular, $|A'|$ and $|B'|$ are both even.
This agrees with the fact that the decomposition may also involve matching exceptional systems in the current case:
we will later extend each such system to a cycle spanning $A'$ and one spanning $B'$.
As $|A'|$ and $|B'|$ are both even, these cycles correspond to two edge-disjoint perfect matchings in~$G$.

\begin{lemma} \label{lma:PBESdecom}
Suppose that $0 <  1/n  \ll \epszero \ll  \eps \ll   \lambda, 1/K \ll 1$, that $0 \le \phi  \ll 1$ and that
$n/4, K, m, \lambda n/K^2, (n/2-1 - \phi n)/(2K^2) \in \mathbb{N}$.
Suppose that the following conditions hold:
\begin{itemize}
	\item[\rm (i)] $G$ is an $(n/2 -1)$-regular graph on $n$ vertices.
	\item[\rm (ii)] $\mathcal{P}$ is a $(K, m, \epszero)$-partition of $V(G)$ such that
$\Delta(G[A',B']) \le n/4$ and $|A'| = |B'| = n/2$.%
\COMMENT{Probably could actually assume that $\Delta(G[A',B']) \le n/4-1$. We also have $e_G(A',B') < D$ but this statement
is not needed explicitly. This statement is incorporated in (v)}
	\item[\rm (iii)] $G_0$ is a subgraph of $G$ such that $G[A_0]+G[B_0] \subseteq G_0$ and $d_{G_0}(v) = \phi n $ for all $v \in V_0$.
	\item[\rm (iv)] Let $G^{\diamond} := G - G[A]  - G[B] - G_0$. $e_{G^\diamond} (A',B')$ is even and
$(G^{\diamond}, \mathcal{P})$ is a $(K, m, \epszero,\eps)$-exceptional scheme.
	\item[\rm (v)] $\Delta(G^{\diamond}[A',B']) \le e_{G^{\diamond}}(A',B')/2  \le (n/2 - 1-  \phi n ) /2$.
\end{itemize}
Then there exists a set $\mathcal{J}$ consisting of $(n/2-1-\phi n)/2$ edge-disjoint exceptional systems in $G^{\diamond}$
which satisfies the following properties:
\begin{itemize}
    \item[\rm (a)] Together the exceptional systems in $\mathcal{J}$ cover all edges of $G^\diamond$.
Each $J$ in $\mathcal J$ is either a Hamilton exceptional system with $e_{J}(A',B') = 2$ or a matching exceptional system.
	\item[\rm (b)] For all $1\le i,i' \le K$, the set $\mathcal{J}$ contains $(n/2-1-(\phi n+2 \lambda))/(2K^2)$ $(i,i')$-ES.
\end{itemize}
\end{lemma}
As in the other two cases, we will use the exceptional systems in (b) to construct the robustly decomposable graph $G^{\rm rob}$.
Unlike the critical case with $e_G(A',B') \ge D$, there is no need to `track' the degrees of the vertices $w_i$ of high degree in $G[A',B']$
this time. Indeed, let $G'':=G - G_0 - G^{\rm rob}$, where $G_0$ is the graph defined by Lemma~\ref{V_0elimination}.
Then $G''[A',B']$ is the union of all those $J$ in $\mathcal J$ (from the first application of Lemma~\ref{lma:PBESdecom}) not used in the construction of $G^{\rm rob}$.
So (a) implies that $G''[A',B']$ is a union of matchings of size two. So (v) will be trivially satisfied when we apply 
Lemma~\ref{lma:PBESdecom} for the second time (i.e. with $G_0 +G^{\rm rob}$ playing the role of $G_0$).



\section{Special Factors and Exceptional Factors}\label{sec:SFandEF}

As discussed in the proof sketch, the main proof proceeds as follows.
First we remove a sparse `robustly decomposable' graph $G^{\rm rob}$ from the original graph $G$.
Then we find an approximate decomposition of $G-G^{\rm rob}$.
Finally we find a decomposition of $G^{\rm rob}+G'$, where $G'$ is  the (very sparse)
leftover from the approximate decomposition.

Both the approximate decomposition as well as the actual decomposition step assume that we work with a graph with two components,
one on $A$ and the other on $B$. So in both steps, we would need $A_0 \cup B_0$ to be empty, which we clearly cannot assume. 
We build on the ideas of Section~\ref{sec:BES} to deal with this problem.%
\COMMENT{Deryk modified a bit.}
In both steps, one can choose `exceptional path systems' in $G$ with the following crucial property:
one can replace each such exceptional path system $EPS$ with a path system $EPS^*$ so that
\begin{itemize}
\item[($\alpha_1$)] $EPS^*$ can be partitioned into $EPS^*_A$ and $EPS^*_B$ with the vertex sets of $EPS^*_A$ and $EPS^*_B$
being contained in $A$ and $B$ respectively;
\item[($\alpha_2$)] the union of any Hamilton cycle $C^*_A$ in $G^*_A:=G[A]-EPS+ EPS^*_A$ containing $EPS^*_A$ and
any Hamilton cycle $C^*_B$ in $G^*_B :=G[B]-EPS + EPS^*_B$ containing $EPS^*_B$ corresponds to either a Hamilton cycle of $G$ containing $EPS$
or to the union of two edge-disjoint perfect matchings in $G$ containing $EPS$.
\end{itemize}
Each exceptional path system $EPS$ will contain one of the exceptional systems $J$ constructed in Section~\ref{sec:locES}.
$EPS^*$ will then be obtained from $EPS$%
   \COMMENT{Daniela: added from $EPS$}
by replacing $J$ by $J^*$. (Recall that $J^*$ was defined in Section~\ref{sec:BES} and that
we view the edges of $J^*$ as `fictive edges' which are different from the edges of $G$.)
So $G^*_A$ is obtained from $G[A]$ by adding $J^*_A=J^*[A]$.
Furthermore, $J$ determines which of the cases in~($\alpha_2$) holds: If $J$ is a Hamilton exceptional system,
then ($\alpha_2$) will give a Hamilton cycle of $G$, while in the case when $J$ is a matching exceptional system,
($\alpha_2$) will give the union of two edge-disjoint perfect matchings in $G$.

So, roughly speaking, this allows us to work with $G^*_A$ and $G^*_B$ rather than $G$ in the two steps.
A convenient way of handling these exceptional path systems is to combine many of them into an `exceptional factor' $EF$
(see Section~\ref{sec:BEPS} for the definition).

One complication is that the `robust decomposition lemma' (Lemma~\ref{rdeclemma}) we use from~\cite{Kelly}
deals with digraphs rather than undirected graphs. 
So in order to be able to apply it, we need to suitably orient the edges of $G$ and so we will actually consider a
directed path system $EPS^*_{\rm dir}$ instead of the $EPS^*$ above (the exceptional path system $EPS$ itself will
still be undirected). Moreover,  we have to apply the robust
decomposition lemma twice, once to $G^*_A$ and once to $G^*_B$. 

The formulation of the robust decomposition lemma is quite general and rather than guaranteeing ($\alpha_2$) directly,
it assumes the existence of certain directed `special paths systems' $SPS$ which are combined into `special factors' $SF$.
These are introduced in Section~\ref{sec:SF}. Each of the Hamilton cycles produced by the lemma then contains
exactly one of these special path systems. So to apply the lemma, it suffices to check that each of our
exceptional path systems $EPS$ corresponds to two path systems $EPS^*_{A,\rm dir}$ and $EPS^*_{B,\rm dir}$
which both satisfy the conditions required of a special path system.


\subsection{Special Path Systems and Special Factors}\label{sec:SF}
As mentioned above, the robust decomposition lemma requires `special path systems' and `special factors' as an input when constructing 
the robustly decomposable graph.
These are defined in this subsection.

Let $K,m \in \mathbb{N}$.
A \emph{$(K,m)$-equipartition} $\mathcal{Q}$ of a set $V$ of vertices is a partition of $V$ into sets $V_1, \dots, V_K$ such that $|V_i| = m$ for all $i \le K$.%
         \COMMENT{Note that there is no exceptional vertex set. Observe that $\mathcal{Q}_A:= \{A_1, \dots, A_K\}$ is a $(K,m)$-partition of $A$.
I added `equi' to increase the differences between $(K,m)$-equipartition to $(K,m, \epszero)$-partition.}
The $V_i$ are called \emph{clusters} of~$\mathcal{Q}$.
Suppose that $\mathcal{Q}=\{V_1,\dots,V_K\}$ is a $(K,m)$-equipartition of~$V$ and $L,m/L\in\mathbb{N}$.
We say that $(\mathcal{Q},\mathcal{Q'})$ is a \emph{$(K, L, m)$-equipartition of~$V$} if $\mathcal{Q'}$ is 
obtained from $\mathcal{Q}$ by partitioning each cluster $V_i$ of $\mathcal{Q}$
into $L$ sets $V_{i,1},\dots,V_{i,L}$ of size~$m/L$.
So $\mathcal{Q}'$ consists of the $KL$ clusters $V_{i,j}$.

Let $(\mathcal{Q},\mathcal{Q'})$ be a $(K,L,m)$-equipartition of~$V$.
Consider a spanning cycle $C = V_1 \dots V_K$ on the clusters of $\mathcal{Q}$.
Given an integer $f$ dividing $K$, the \emph{canonical interval partition} $\mathcal{I}$ of $C$ into $f$ intervals
consists of the intervals $$V_{(i-1)K/f+1}  V_{(i-1)K/f+2} \dots V_{iK/f+1}$$ for all $i\le f$ (with
addition modulo $K$).

Suppose that $G$ is a digraph on~$V$%
	\COMMENT{We will take $G = G^{*}_A$ with $V = A $ and $G = G^{*}_B$ with $V =B$.}
and $h\le L$. Let $I=V_jV_{j+1}\dots V_{j'}$ be an interval in~$\mathcal{I}$.
A \emph{special path system $SPS$ of style~$h$ in~$G$ 
spanning the interval~$I$} consists of $m/L$
vertex-disjoint directed paths $P_1,\dots,P_{m/L}$ such that the following conditions hold:
\begin{itemize}
\item[(SPS1)] Every $P_s$ has its initial vertex in $V_{j,h}$ and its final vertex in $V_{j',h}$.
\item[(SPS2)] $SPS$ contains a matching ${\rm Fict}(SPS)$ such that all the edges in ${\rm Fict}(SPS)$ avoid the
endclusters $V_j$ and $V_{j'}$ of $I$ and such that $E(P_s)\setminus {\rm Fict}(SPS)\subseteq E(G)$.
\item[(SPS3)] The vertex set of $SPS$ is $V_{j,h}\cup V_{j+1,h}\cup \dots \cup  V_{j',h}$.
\end{itemize}
The edges in ${\rm Fict}(SPS)$ are called \emph{fictive edges of} $SPS$.

Let $\mathcal{I}=\{I_1,\dots,I_f\}$.
A \emph{special factor $SF$ with parameters $(L,f)$ in $G$ (with respect to $C$, $\mathcal{Q}'$)} is a $1$-regular digraph on $V$
which is the union of $Lf$ digraphs $SPS_{j,h}$ (one for all $j\le f$ and $h\le L$) such that each $SPS_{j,h}$ is
a special path system of style $h$ in $G$ which spans~$I_j$.
We write ${\rm Fict}(SF)$ for the union of the sets ${\rm Fict}(SPS_{j,h})$ over all $j\le f$ and $h\le L$
and call the edges in ${\rm Fict}(SF)$ \emph{fictive edges of $SF$}.

We will always view fictive edges as being distinct from each other and from the edges in other digraphs.
So if we say that special factors $SF_1,\dots,SF_r$ are pairwise edge-disjoint from each other and from some digraph $Q$ on $V$,
then%
   \COMMENT{It's not enough to have this for all $Q\subseteq G$ since in the robust decomposition lemma $H$ need not be a subgraph of $G$.}
this means that $Q$ and all the $SF_i- {\rm Fict}(SF_i)$
are pairwise edge-disjoint, but for example there could be an edge from $x$ to $y$ in $Q$ as well as in ${\rm Fict}(SF_i)$ 
for several indices $i\le r$.
But these are the only instances of multiedges that we allow, i.e.~if there is more than one edge from $x$ to $y$, then all but
at most one of these edges are fictive edges.

\subsection{Exceptional Path Systems and Exceptional Factors}\label{sec:BEPS}
We now introduce `exceptional path systems' which will be combined into `exceptional factors'.
These will satisfy the requirements of special path systems and special factors respectively.
So they can be used as an `input' for the robust decomposition lemma.
Moreover, they will satisfy the properties ($\alpha_1$) and ($\alpha_2$) described at the beginning of Section~\ref{sec:SFandEF}
(see Proposition~\ref{prop:EPS}).
More precisely, suppose that 
$$
\mathcal{P}=\{A_0,A_1,\dots,A_K,B_0,B_1,\dots,B_K\}
$$ is a $(K,m,\eps_0)$-partition of a vertex set~$V$ and $L,m/L\in\mathbb{N}$.
We say that $(\mathcal{P},\mathcal{P'})$ is a \emph{$(K, L, m, \epszero)$-partition of~$V$} if $\mathcal{P'}$ is 
obtained from $\mathcal{P}$ by partitioning each cluster $A_i$ of $\mathcal{P}$
into $L$ sets $A_{i,1},\dots,A_{i,L}$ of size~$m/L$ and partitioning each cluster $B_i$ of $\mathcal{P}$ into $L$ sets $B_{i,1},\dots,B_{i,L}$ of size~$m/L$.
(So $\cP'$ consists of the exceptional sets $A_0$, $B_0$, the $KL$ clusters $A_{i,j}$ and the $KL$ clusters $B_{i,j}$.)
Set 
\begin{align}
\mathcal{Q}_A & := \{A_1, \dots, A_K\}, & 
\mathcal{Q}'_A &:= \{A_{1,1}, \dots, A_{K,L}\}, \label{defparts}\\ 
\mathcal{Q}_B &:= \{B_1, \dots, B_K\}, &
\mathcal{Q}'_B & := \{B_{1,1}, \dots, B_{K,L}\}\nonumber. 
\end{align}
Note that $(\mathcal{Q}_A, \mathcal{Q}_A')$ and $(\mathcal{Q}_B, \mathcal{Q}_B')$ are $(K,L,m)$-equipartitions of $A$ and $B$ respectively
(where we recall that $A = \bigcup_{i=1}^K A_i$ and $B = \bigcup_{i=1}^K B_i$).

Suppose that $J$ is a Hamilton exceptional system (for the partition $A,A_0,B,B_0$) with $e_J(A',B')=2$.
Thus $J$ contains precisely%
    \COMMENT{Daniela: had "two paths having one endvertex in $A$ and the other endvertex in $B$" instead of "$AB$-paths"}
two $AB$-paths. Let $P_1=a_1\dots b_1$ and $P_2=a_2\dots b_2$ be these two paths, where $a_1,a_2\in A$ and $b_1,b_2\in B$.
Recall from Section~\ref{sec:BES} that $J^*_A$ is the matching consisting of the edge $a_1a_2$ and an edge
between any two vertices $a,a'\in A$ for which $J$ contains a path $P_{aa'}$ whose endvertices are $a$ and $a'$. 
We also defined a matching $J^*_B$ in a similar way and set $J^*:=J^*_A\cup J^*_B$.
We say that an \emph{orientation of $J$ is good} if every path in $J$ is oriented consistently and
one of the paths $P_1$, $P_2$ is oriented towards $B$ while the other is oriented towards $A$.
Given a good orientation $J_{\rm dir}$ of $J$, the \emph{orientation $J^*_{\rm dir}$ of $J^*$ induced by $J_{\rm dir}$} is defined as follows:
\begin{itemize}
\item For every path $P_{aa'}$ in $J$ whose endvertices $a,a'$ both belong to $A$, we orient the edge $aa'$ of $J^*$
towards its endpoint of the (oriented) path $P_{aa'}$ in $J_{\rm dir}$.
\item If in $J_{\rm dir}$ the path $P_1$ is oriented towards $b_1$ (and thus $P_2$ is oriented towards $a_2$),
 then we orient the edge $a_1a_2$ of $J^*$ towards $a_2$ and the edge $b_1b_2$ of $J^*$ towards $b_1$.
The analogue holds if $P_1$ is oriented towards $a_1$ (and thus $P_2$ is oriented towards $b_2$).
\end{itemize}
If $J$ is a matching exceptional system, we define good orientations of $J$ and the corresponding
induced orientations of $J^*$ in a similar way.

We now define exceptional path systems. As mentioned at the beginning of Section~\ref{sec:SFandEF}, each such exceptional
path system $EPS$ will correspond to two directed path systems $EPS^*_{A,{\rm dir}}$ and $EPS^*_{B, {\rm dir}}$
satisfying the conditions of a special path system (for $(\mathcal{Q}_A, \mathcal{Q}_A')$ and
$(\mathcal{Q}_B, \mathcal{Q}_B')$ respectively).

Let $(\mathcal{P},\mathcal{P'})$ be a $(K,L,m, \epszero)$-partition of a vertex set $V$.
Suppose that $K/f\in\mathbb{N}$. The \emph{canonical interval partition} $\mathcal{I}(f,K)$ of $[K]:=\{1,\dots,K\}$ into $f$ intervals
consists of the intervals $$\{(i-1)K/f+1, (i-1)K/f+2, \dots, iK/f+1\}$$ for all $i\le f$ (with addition modulo $K$).

Suppose that $G$ is an oriented graph on $A \cup B$ such that $G = G[A] +G[B]$.
Let $h\le L$ and suppose that $I \in \mathcal{I}(f,K)$ is an interval with $I=\{j,j+1, \dots, j'\}$.%
   \COMMENT{Daniela: added brackets}
An \emph{exceptional path system $EPS$ of style~$h$ for $G$ spanning $I$}
consists of $2 m/L$ vertex-disjoint undirected paths $P_0, P'_0 , P_1^A, \dots,P_{m/L-1}^A, P_1^B, \dots,P_{m/L-1}^B,$
such that the following conditions hold:
\begin{enumerate}
\item[(EPS1)] $V(P_s^A)\subseteq A$ and $P_s^A$ has one endvertex in $A_{j,h}$ and its other endvertex in $A_{j',h}$ (for all $1 \le s < m/L$).
The analogue holds for every $P_s^B$.%
\COMMENT{So no edges of the exceptional system $J$ are contained in $P_s^A$ and $P_s^B$ for $ 1 \le s < m/L$.}
\item[(EPS2)] Each of $P_0$ and $P_0'$ has one endvertex in $A_{j,h} \cup B_{j,h}$ and its other endvertex in $A_{j',h} \cup B_{j',h}$.%
	\COMMENT{We can fix $P_0$ to have one endvertex in $A_{j,h}$ (and $P'_0$ to have one endvertex in $B_{j,h}$). However, the other endvertex of $P_0$ depends on the exceptional system.
	If $J$ is a Hamilton exceptional system, then the other endvertex of $P_0$ is in $B_{j',h}$.
	Otherwise (i.e.~$J$ is a matching exceptional system) the other endvertex of $P_0$ is in $A_{j',h}$.}
\item[(EPS3)] $J:= EPS - EPS[A] - EPS[B]$ is either a Hamilton exceptional system with $e_J ( A',B')=2$
or a matching exceptional system (with respect to the partition~$A,A_0,B,B_0$).
Moreover $E(J)\subseteq E(P_0) \cup E(P_{0}')$ and no edge of $J$ has an endvertex in $A_{j,h}\cup A_{j',h}\cup B_{j,h}\cup B_{j',h}$.
\item[(EPS4)] Let $P_{0,{\rm dir}}$ and $P'_{0,{\rm dir}}$ be the paths obtained by orienting $P_0$ and $P'_0$ towards their
endvertices in $A_{j',h} \cup  B_{j',h}$. Then the orientation $J_{\rm dir}$ of $J$ obtained in this way is good. 
Let $J^*_{\rm dir}$ be the orientation of $J^*$ induced by $J_{\rm dir}$.
Then $(P_{0,{\rm dir}} + P'_{0,{\rm dir}})-J_{\rm dir}+J^*_{\rm dir}$ consists of two vertex-disjoint paths $P^A_{0,{\rm dir}}$
and $P^B_{0,{\rm dir}}$ such that $V(P_{0,{\rm dir}}^A)\subseteq A$, $P_{0,{\rm dir}}^A$ has one endvertex in $A_{j,h}$ and its other endvertex in $A_{j',h}$ and
such that the analogue holds for $P^B_{0, {\rm dir} }$.
\item[(EPS5)] The vertex set of $EPS$ is $V_0\cup A_{j,h} \cup A_{j+1,h} \dots \cup A_{j',h} \cup B_{j,h} \cup B_{j+1,h} \dots \cup B_{j',h}$.
\item[(EPS6)] For each $1 \le s < m/L$, let $P^A_{s,{\rm dir}}$ be the path obtained by orienting $P^A_s$ towards its endvertex in $A_{j',h}$.
Define $P^B_{s,{\rm dir}}$ in a similar way. Then
$E(P^A_{0,{\rm dir}})\setminus E(J_{\rm dir}), E(P^B_{0,{\rm dir}})\setminus E(J_{\rm dir}) \subseteq E(G)$
and $E(P^A_{s,{\rm dir}}), E(P^B_{s,{\rm dir}})\subseteq E(G)$ for every $1\le s < m/L$.  
\end{enumerate}

We call $EPS$ \emph{a Hamilton exceptional path system} if $J$ (as defined in (EPS3)) is a Hamilton exceptional system,
and a \emph{matching exceptional path system} otherwise.
Let $EPS^*_{A,{\rm dir}}$ be the (directed) path system consisting of $P^A_{0,{\rm dir}},P^A_{1,{\rm dir}}, \break \dots,P^A_{m/L-1,{\rm dir}}$.
Then $EPS^*_{A,{\rm dir}}$ is a special path system of style $h$ in $G[A]$ which spans the
interval~$A_{j} A_{j+1} \dots A_{j'}$%
\COMMENT{Deryk: had $A_{j,h} A_{j+1,h} \dots A_{j',h}$}
of the cycle $A_1\dots A_K$ and satisfies
${\rm Fict}(EPS^*_{A,{\rm dir}})  \break = J^*_{\rm dir}[A]$.
Define $EPS^*_{B,{\rm dir}}$ similarly and let $EPS^*_{{\rm dir}} : = EPS^*_{A,{\rm dir}} + EPS^*_{B,{\rm dir}} $
and ${\rm Fict}(EPS^*_{\rm dir}) : = {\rm Fict}(EPS^*_{A,{\rm dir}}) \cup {\rm Fict}(EPS^*_{B,{\rm dir}})$
(see Figure~\ref{fig:eps}).
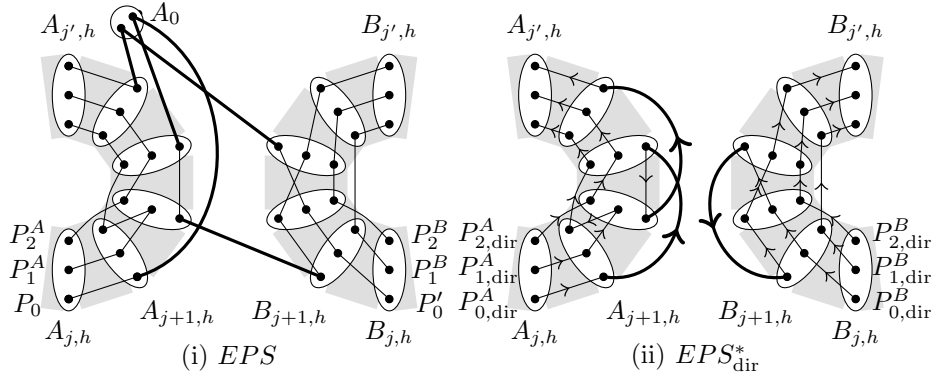
\begin{figure}
{
\begin{tikzpicture}[scale = 0.775]
\tikzstyle{every node}=[fill=black,draw,circle,minimum width=1pt,outer sep = 0, inner sep = 1]

\begin{scope}

\draw[line width=11mm,color=gray!25] (-105:1.5) -- (-90:1.5);
\draw[line width=11mm,color=gray!25] (105:1.5) -- (90:1.5);
\draw[line width=11mm,color=gray!25]
     \foreach \x in {-90,-54,-18,18,54}
    {
        (\x:1.5) -- (\x+36:1.5)
    };

\draw[fill=white,rotate=-90] (0:1.5) ellipse (20pt and 8pt);
\draw[fill=white,rotate=-54] (0:1.5) ellipse (20pt and 8pt);
\draw[fill=white,rotate=-18] (0:1.5) ellipse (20pt and 8pt);
\draw[fill=white,rotate=18] (0:1.5) ellipse (20pt and 8pt);
\draw[fill=white,rotate=54] (0:1.5) ellipse (20pt and 8pt);
\draw[fill=white,rotate=90] (0:1.5) ellipse (20pt and 8pt);

\draw[fill=black]
 \foreach \x in {-90,-54,-18,18,54,90}
    {
        (\x:1) node (\x l1) {}
        (\x:1.5) node (\x l2) {}
        (\x:2) node (\x l3) {}
    };

\draw (-90l3) -- (-54l3);
\draw (-90l2) -- (-54l2) -- (-18l2) -- (-54l1);
\draw (18l1) -- (54l1) -- (90l1);
\draw (-18l1) -- (18l2) -- (54l2) -- (90l2);
\draw (-18l3) -- (18l3);
\draw (54l3) -- (90l3);

\draw (-90l1) .. controls (-54:0.7) .. (-18l1);
\draw (-54l1) .. controls (-18:0.7) .. (18l1);


\node at (-0.7,-0.9) [fill=none, draw=none] {$P_2^A$};
\node at (-0.7,-1.5) [fill=none, draw=none] {$P_1^A$};
\node at (-0.7,-2.1) [fill=none, draw=none] {$P_0$};

\node at (90l3) [label=above:{{$A_{j',h}$}}] (){};
\node at (-90l3) [label=below:{{$A_{j,h}$}}] (){};
\node at (-54l3) [label=-54:{~~{$A_{j+1,h}$}}] (){};
\end{scope}

\begin{scope}[xshift=5.5cm,rotate=180]

\draw[line width=11mm,color=gray!25] (-105:1.5) -- (-90:1.5);
\draw[line width=11mm,color=gray!25] (105:1.5) -- (90:1.5);
\draw[line width=11mm,color=gray!25]
     \foreach \x in {-90,-54,-18,18,54}
    {
        (\x:1.5) -- (\x+36:1.5)
    };

\draw[fill=white,rotate=-90] (0:1.5) ellipse (20pt and 8pt);
\draw[fill=white,rotate=-54] (0:1.5) ellipse (20pt and 8pt);
\draw[fill=white,rotate=-18] (0:1.5) ellipse (20pt and 8pt);
\draw[fill=white,rotate=18] (0:1.5) ellipse (20pt and 8pt);
\draw[fill=white,rotate=54] (0:1.5) ellipse (20pt and 8pt);
\draw[fill=white,rotate=90] (0:1.5) ellipse (20pt and 8pt);

\draw[fill=black]
 \foreach \x in {-90,-54,-18,18,54,90}
    {
        (\x:1) node (\x r1) {}
        (\x:1.5) node (\x r2) {}
        (\x:2) node (\x r3) {}
    };

\draw (90r3) -- (54r2) -- (18r2) -- (-18r3);
\draw (90r2) -- (54r1) -- (-54r1) -- (-90r1);
\draw (18r1) -- (-18r1) -- (-54r2) -- (-90r2);
\draw (54r3) -- (18r3) -- (-18r2) -- (-54r3) -- (-90r3);

\draw (90r1) .. controls (54:0.7) .. (18r1);


\node at (-0.7,0.9) [fill=none, draw=none] {$P_2^B$};
\node at (-0.7,1.5) [fill=none, draw=none] {$P_1^B$};
\node at (-0.7,2.1) [fill=none, draw=none] {$P_0'$};

\node at (-90r3) [draw=none,
label=above:{{$B_{j',h}$}}] (){};
\node at (90r3) [draw=none,
label=below:{{$B_{j,h}$}}] (){};
\node at (54r3) [draw=none,
label=-120:{~~{$B_{j+1,h}$}}] (){};
\end{scope}

\node at (1.2,2.9) [draw=none,
label=right:{{$A_0$}}] (){};
\draw[fill=white] (1,2.75) circle (8pt);
\draw (0.9,2.65) node (a1) {};
\draw (1.1,2.85) node (a2) {};

\draw[very thick] (54l3) -- (a1) -- (-18r3);
\draw[very thick] (18l3) -- (a2);
\draw[very thick] (a2) .. controls (3,2) and (3,-1) .. (-54l3);
\draw[very thick] (-18l3) -- (54r3);

\tikzset{->-/.style={decoration={
  markings,
  mark=at position #1 with {\arrow[scale=1.5]{>}}},postaction={decorate}}}

\begin{scope}[xshift=8cm]

\draw[line width=11mm,color=gray!25] (-105:1.5) -- (-90:1.5);
\draw[line width=11mm,color=gray!25] (105:1.5) -- (90:1.5);
\draw[line width=11mm,color=gray!25]
     \foreach \x in {-90,-54,-18,18,54}
    {
        (\x:1.5) -- (\x+36:1.5)
    };

\draw[fill=white,rotate=-90] (0:1.5) ellipse (20pt and 8pt);
\draw[fill=white,rotate=-54] (0:1.5) ellipse (20pt and 8pt);
\draw[fill=white,rotate=-18] (0:1.5) ellipse (20pt and 8pt);
\draw[fill=white,rotate=18] (0:1.5) ellipse (20pt and 8pt);
\draw[fill=white,rotate=54] (0:1.5) ellipse (20pt and 8pt);
\draw[fill=white,rotate=90] (0:1.5) ellipse (20pt and 8pt);

\draw[fill=black]
 \foreach \x in {-90,-54,-18,18,54,90}
    {
        (\x:1) node (\x L1) {}
        (\x:1.5) node (\x L2) {}
        (\x:2) node (\x L3) {}
    };

\draw[->-=.5] (-90L3) -- (-54L3);
\draw[->-=.55] (-90L2) -- (-54L2);
\draw[->-=.55] (-54L2) -- (-18L2);
\draw[->-=.65] (-18L2) -- (-54L1);
\draw[->-=.55] (18L1) -- (54L1);
\draw[->-=.56] (54L1) -- (90L1);
\draw[->-=.45] (-18L1) -- (18L2);
\draw[->-=.55] (18L2) -- (54L2);
\draw[->-=.56] (54L2) -- (90L2);
\draw[->-=.5] (18L3) -- (-18L3);
\draw[->-=.5] (54L3) -- (90L3);

\draw[->-=.5] (-90L1) .. controls (-54:0.7) .. (-18L1);
\draw[->-=.5] (-54L1) .. controls (-18:0.7) .. (18L1);

\draw[very thick,->-=0.5] (-54L3) .. controls (-36:3) and (0:3) .. (18L3);
\draw[very thick,->-=0.5] (-18L3) .. controls (0:3) and (36:3) .. (54L3);


\node at (-0.8,-0.9) [fill=none, draw=none] {$P_{2,{\rm dir}}^A$};
\node at (-0.8,-1.5) [fill=none, draw=none] {$P_{1,{\rm dir}}^A$};
\node at (-0.8,-2.1) [fill=none, draw=none] {$P_{0,{\rm dir}}^A$};

\node at (90L3) [draw=none,
label=above:{{$A_{j',h}$}}] (){};
\node at (-90L3) [draw=none,
label=below:{{$A_{j,h}$}}] (){};
\node at (-54L3) [draw=none,
label=-54:{~~{$A_{j+1,h}$}}] (){};

\end{scope}

\begin{scope}[xshift=13.5cm,rotate=180]

\draw[line width=11mm,color=gray!25] (-105:1.5) -- (-90:1.5);
\draw[line width=11mm,color=gray!25] (105:1.5) -- (90:1.5);
\draw[line width=11mm,color=gray!25]
     \foreach \x in {-90,-54,-18,18,54}
    {
        (\x:1.5) -- (\x+36:1.5)
    };

\draw[fill=white,rotate=-90] (0:1.5) ellipse (20pt and 8pt);
\draw[fill=white,rotate=-54] (0:1.5) ellipse (20pt and 8pt);
\draw[fill=white,rotate=-18] (0:1.5) ellipse (20pt and 8pt);
\draw[fill=white,rotate=18] (0:1.5) ellipse (20pt and 8pt);
\draw[fill=white,rotate=54] (0:1.5) ellipse (20pt and 8pt);
\draw[fill=white,rotate=90] (0:1.5) ellipse (20pt and 8pt);

\draw[fill=black]
 \foreach \x in {-90,-54,-18,18,54,90}
    {
        (\x:1) node (\x R1) {}
        (\x:1.5) node (\x R2) {}
        (\x:2) node (\x R3) {}
    };

\draw[->-=.6] (90R3) -- (54R2);
\draw[->-=.56] (54R2) -- (18R2);
\draw[->-=.5] (18R2) -- (-18R3);
\draw[->-=.65] (90R2) -- (54R1);
\draw[->-=.5] (54R1) -- (-54R1);
\draw[->-=.6] (-54R1) -- (-90R1);
\draw[->-=.6] (18R1) -- (-18R1);
\draw[->-=.45] (-18R1) -- (-54R2);
\draw[->-=.6] (-54R2) -- (-90R2);
\draw[->-=.5] (54R3) -- (18R3);
\draw[->-=.5] (18R3) -- (-18R2);
\draw[->-=.5] (-18R2) -- (-54R3);
\draw[->-=.5] (-54R3) -- (-90R3);

\draw[->-=.5] (90R1) .. controls (54:0.7) .. (18R1);

\draw[very thick,->-=0.5] (-18R3) .. controls (0:3) and (36:3) .. (54R3);


\node at (-0.8,0.9) [fill=none, draw=none] {$P_{2,{\rm dir}}^B$};
\node at (-0.8,1.5) [fill=none, draw=none] {$P_{1,{\rm dir}}^B$};
\node at (-0.8,2.1) [fill=none, draw=none] {$P_{0,{\rm dir}}^B$};

\node at (-90R3) [draw=none,
label=above:{{$B_{j',h}$}}] (){};
\node at (90R3) [draw=none,
label=below:{{$B_{j,h}$}}] (){};
\node at (54R3) [draw=none,
label=-120:{~~{$B_{j+1,h}$}}] (){};
\end{scope}

\node at (2.75,-3) [draw=none, fill=none] {(i) $EPS$};
\node at (10.75,-3) [draw=none, fill=none] {(ii) $EPS^*_{\rm dir}$};
\end{tikzpicture}
}
\caption{An example of an exceptional path system $EPS$ and the corresponding directed version $EPS^*_{\rm dir}$ in the
case when $|A_0|=2$, $B_0=\emptyset$, $m/L=3$ and $|I|=6$. The thick edges indicate $J$ and $J^*_{\rm dir}$ respectively.}
\label{fig:eps}
\end{figure}

Let $\mathcal{I}(f,K)=\{I_1,\dots,I_f\}$.
An \emph{exceptional factor $EF$ with parameters $(L,f)$ for $G$ (with respect to $(\mathcal{P},\mathcal{P'})$)} is 
the union of $Lf$ edge-disjoint%
    \COMMENT{Daniela: added edge-disjoint, otherwise the $EPS_{j,h}$ could use the same edges inside $A_0\cup B_0$ and
then (\ref{EFdeg}) may not hold}
undirected graphs $EPS_{j,h}$ (one for all $j\le f$ and $h\le L$) such that each $EPS_{j,h}$ is
an exceptional path system of style $h$ for $G$ which spans~$I_j$.
We write $EF^*_{A,\rm dir}$ for the union of $EPS_{j,h,A,{\rm dir}}^*$ over all
$j\le f$ and $h\le L$.
Note that $EF^*_{A,{\rm dir}}$ is a special factor with parameters $(L,f)$ in $G[A]$ (with respect to $C=A_1 \dots A_K$, $\mathcal{Q}'_A$)
such that ${\rm Fict}(EF^*_{A,{\rm dir}})$ is the union of $J^*_{j,h,{\rm dir}}[A]$ over all $j\le f$ and $h\le L$, where $J_{j,h}$ is the exceptional
system contained in $EPS_{j,h}$ (see condition (EPS3)). 
Define $EF^*_{B,{\rm dir}}$ similarly and let $EF^*_{{\rm dir}} : = EF^*_{A,{\rm dir}} + EF^*_{B,{\rm dir}} $%
   \COMMENT{Daniela: $+$ instead of $\cup$ since we have $+$ for $EPS^*_{{\rm dir}}$}
and ${\rm Fict}(EF^*_{\rm dir}):={\rm Fict}(EF^*_{A,{\rm dir}}) \cup {\rm Fict}(EF^*_{B,{\rm dir}})$.
Note that $EF^*_{{\rm dir}}$ is a $1$-regular directed graph
    \COMMENT{Daniela: had multigraph before, but I don't see why we could get multiple edges}
on $A \cup B$ while in $EF$ is an undirected graph on $V$ with 
\begin{align}\label{EFdeg}
	d_{EF}(v) = 2 \ \ \text{for all } v \in V \setminus V_0 \ \ \ \text{ and } \ \ \ d_{EF}(v) = 2Lf \ \ \text{for all } v \in V_0.	
\end{align}

Given an exceptional path system $EPS$, let $J$ be as in (EPS3) and
let 
$$
EPS^*:=EPS-J+J^*,\ \ EPS^*_A:=EPS^*[A] \ \mbox{ and }  \ EPS^*_B:=EPS^*[B].
$$ 
(Hence $EPS^*$, $EPS^*_A$ and $EPS^*_B$ are the undirected graphs obtained
from $EPS^*_{\rm dir}$, $EPS^*_{A,\rm dir}$ and $EPS^*_{B,\rm dir}$ by ignoring the orientations of all edges.)
The following result is an immediate consequence of (EPS3), (EPS4) and Proposition~\ref{prop:ES}.
Roughly speaking, it implies that to find a Hamilton cycle in the `original' graph with vertex set $V$, 
it suffices to find a Hamilton cycle on $A$ and one on $B$, containing (the edges corresponding to) an exceptional path system.

\begin{prop} \label{prop:EPS}
Let $(\cP,\cP')$ be a $(K,L, m , \epszero)$-partition of a vertex set $V$.
Suppose that $G$ is a graph on~$V \setminus V_0$, that $G_{\rm dir}$ is an orientation of $G[A]+G[B]$ and that $EPS$ is an exceptional path system for~$G_{\rm dir}$.
Let $J$ be as in (EPS3) and $J_A^*$ as defined in Section~\ref{sec:BES}. Let $C_A$ and $C_B$ be two cycles such that%
   \COMMENT{Daniela: replaced Hamilton cycle of  $G[A] + J^*_A$ by  Hamilton cycle on $A$ as otherwise it is difficult to say which graph
plays the role of $G$ when we apply the prop in the proof of Corollary~\ref{rdeccor}.}
\begin{itemize}
	\item $C_A$ is a Hamilton cycle on $A$ which contains $EPS^*_A$;
	\item $C_B$ is a Hamilton cycle on $B$ which contains $EPS^*_B$.
\end{itemize}
Then the following assertions hold. 
\begin{itemize}
	\item[\rm (i)] If $EPS$ is a Hamilton exceptional path system, then $C_A+C_B - EPS^* +EPS$ is a Hamilton cycle on $V$.
	\item[\rm (ii)] If $EPS$ is a matching exceptional path system, then $C_A+C_B - EPS^* +EPS$ is the union of a Hamilton cycle on $A'$ and a Hamilton cycle on $B'$.
In particular, if both $|A'|$ and $|B'|$ are even, then $C_A+C_B - EPS^* +EPS$ is the union of two edge-disjoint perfect matchings on $V$.
\end{itemize}
\end{prop}
\proof
Note that $C_A+C_B - EPS^*+EPS = C_A+C_B-J^*+J $. Recall that $J^*_{AB}$ was defined in Section~\ref{sec:BES}.
(EPS3) implies that $|E(J^*_{A}) \setminus E(J^*_{AB})| \le 1$.
Recall from Section~\ref{sec:BES} that a path $P$ is said to consistent with $J_A^*$ if $P$
contains $J_A^*$ and (there is an orientation of $P$ which) visits the endvertices of the edges in $E(J^*_{A}) \setminus E(J^*_{AB})$ in a prescribed order.
Since $E(J^*_{A}) \setminus E(J^*_{AB})$ contains at most one edge, any path containing $J^*_A$ is also consistent with $J^*_A$.
Therefore, $C_A$ is consistent with $J^*_A$ and, by a similar argument, $C_B$ is consistent with $J^*_B$.
So the proposition follows immediately from Proposition~\ref{prop:ES}.
\endproof


\subsection{Finding Exceptional Factors in a Scheme}\label{sec:findBF}

The next lemma (Lem\-ma~\ref{lma:EPS}) will allow us to extend a suitable exceptional system  $J$ into an exceptional path system.
In particular, we assume that $J$ is `localized'. This allows us to choose the path system in such a way that it spans only a few clusters.
The structure within which we find the path system is called a `scheme'.
Roughly speaking, this is the structure we obtain from  $G[A]+G[B]$ (i.e.~the union of two almost complete graphs)
by considering a random equipartition of $A$ and $B$ and a random orientation of its edges.

We now define this `oriented' version of the (undirected) schemes which were introduced in Section~\ref{sec:schemes}.
Given an oriented graph $G$ and partitions $\mathcal P$ and $\cP'$ of a vertex set $V$, we call $(G, \mathcal{P},\mathcal{P}')$ a
\emph{$[K,L,m,\eps_0,\eps]$-scheme} if the following conditions hold:%
    \COMMENT{Daniela: had $G=G[A]+G[B]$ instead of $e_G(A,B)=0$ in (Sch$2'$). But using $e_G(A,B)=0$ makes it more similar to (Sch2).}
\begin{itemize}
\item[(Sch$1'$)] $(\mathcal{P},\mathcal{P}')$ is a $(K,L,m,\eps_0)$-partition of $V$.
\item[(Sch$2'$)] $V(G)=A\cup B$ and $e_G(A,B)=0$.
\item[(Sch$3'$)] $G[A_{i,j},A_{i',j'}]$ and $G[B_{i,j}, B_{i',j'}]$ are $[\eps, 1/2]$-superregular
for all $i, i'\le K$ and all $j, j' \leq L$ such that $(i,j)\not =(i',j')$. Moreover, $G[A_i,A_{i'}]$ and $G[B_i, B_{i'}]$ are $[\eps, 1/2]$-superregular
for all $i\not =i'\le K$.
\item[(Sch$4'$)] $|N_{G}^+(x)\cap N_{G}^-(y)\cap A_{i,j}|\ge (1/5 - \eps) m/L $ for all $x,y\in A$, all $i\le K$ and all $j\le L$.
Similarly, $|N_{G}^+(x)\cap N_{G}^-(y)\cap B_{i,j}|\ge (1/5 - \eps) m/L$ for all $x,y\in B$, all $i\le K$ and all $j\le L$.
\end{itemize}

Note that if $L=1$ (and so $\mathcal{P}=\mathcal{P}'$), then (Sch$1'$) just says that
$\mathcal{P}$ is a $(K,m,\eps_0)$-partition of $V$.%
\COMMENT{Deryk: eliminated the special notation for  $L=1$ as it occurred only once}

Suppose that $J$ is an $(i, i' )$-ES with respect to $\cP$.
Given $h\le L$, we say that $J$ has \emph{style $h$ (with respect to the $(K,L,m,\eps_0)$-partition
$(\mathcal{P},\mathcal{P}')$)} if all the edges of $J$
have their endvertices in $V_0\cup A_{i,h}\cup B_{i',h}$.

\begin{lemma} \label{lma:EPS}
Suppose that $K, L, n, m/L \in \mathbb{N}$, that $0 <1/n  \ll \eps, \eps_0\ll 1$ and $\eps_0\ll 1/K,1/L$.
Let $(G,\mathcal{P}, \mathcal{P}')$ be a $[K, L, m,\eps_0,\eps]$-scheme with $|V(G)\cup V_0|=n$.
Let $I  = \{j, j+1, \dots, j'\}\subseteq [K]$ be an integer interval with $|I| \ge 4$.
Let $J$ be either an $(i_1, i_2)$-HES of style $h\le L$ with $e_J(A',B')= 2$
or an $(i_1, i_2)$-MES of style $h\le L$ (with respect to $(\cP,\cP')$), for some $i_1, i_2 \in \{j+1, \dots, j'-1\}$.
Then there exists an exceptional path system of style $h$ for $G$
which spans the interval $I$ and contains all edges of~$J$.
\end{lemma}
\proof
Let $J_{\rm dir}$ be a good orientation of $J$ and let $J^*_{\rm dir}$ be the induced orientation of $J^*$.
Let $x_1x_2, \dots, x_{2s'-1} x_{2s'}$ be the edges of $J^*_{A, {\rm dir}}: = J^*_{\rm dir}[A]$.
Since $J$ is an $(i_1,i_2)$-ES of style~$h$ with $e_J(A',B')\le 2$ it follows that
$s' = e(J^*_A) \le |V_0|+1\le 2\epszero n$ and $x_i \in A_{i_1,h}$ for all $i \le 2s'$.
Since $|I| \ge 4$ we have $i_1+1 \in \{j+1, \dots, j'-1\}$ or $i_1-1 \in \{j+1, \dots, j'-1\}$.
We will only consider the case when $i_1+1 \in \{j+1, \dots, j'-1\}$. (The argument for the other case
is similar.) 

Our assumption that $\eps_0\ll 1/K,1/L$ implies that
$\eps_0 n\le m/100L$ (say). 
Together with (Sch$4'$) this ensures that for every $1\le r < s'$, we can pick a vertex $w_r \in A_{i_1+1, h}$
such that $x_{2r}w_r$ and $w_r x_{2r+1}$ are (directed) edges in $G$ and such that $w_1,\dots,w_{s'-1}$ are distinct from each other.
We also pick a vertex $w_{s'} \in A_{i_1+1, h} \setminus \{ w_1, \dots, w_{s'-1}\} $ such that $x_{2s'}w_{s'}$ is a (directed) edge in $G$.
Let $Q_0$ be the path $x_1 x_2 w_1 x_3 x_4 w_2 \dots x_{2s' - 1 } x_{2s'} w_{s'}$.
Thus $Q_0$ is a directed path from $A_{i_1,h}$ to $A_{i_1 +1 ,h}$ in $G + J^*_{\rm dir}$ which contains all edges of $J^*_{A , { \rm dir }} $. 
Note that $|V(Q_0) \cap A_{i_1,h}| = 2s'$ and $|V(Q_0) \cap A_{i_1+1,h}| = s'$.
Moreover, $V(Q_0) \cap A_i = \emptyset$ for all $i \notin \{i_1, i_1+1\}$ and $V(Q_0) \cap B = \emptyset$.

Pick a vertex $w_0 \in A_{j,h}$ so that $w_{0} x_1$ is an edge of $G$.
Find a path $Q'_0$ from $w_{s'}$ to $A_{j',h}$ in $G$ such that the vertex set of $Q'_0$ consists of $w_{s'}$ and precisely one vertex in
each $A_{i,h}$ for all $i \in \{j+1, \dots, j' \} \setminus \{i_1, i_1+1 \}$ and no other vertices.
(Sch$4'$) ensures that this can be done greedily.
Define $P^A_{0,{\rm dir}}$ to be the concatenation of $w_0 x_1$, $Q_0$ and $Q'_0$.
Note that $P^A_{0,{\rm dir}}$ is a directed path from $A_{j,h}$ to $A_{j',h}$ in $G + J^*_{\rm dir}$ which contains $J^{*}_{A, {\rm dir}}$.
Moreover, 
\begin{align*}
	 |V(P^A_{0,{\rm dir}}) \cap A_{i,h}| &= 
	 \begin{cases}
	1	& \text{for $i \in \{j, \dots, j'\} \setminus \{i_1 , i_1+1 \}$,}\\
	2s'	& \text{for $i = i_1$,}\\
	s'	& \text{for $i = i_1+1$,}\\
	0	& \text{otherwise,}
	 \end{cases}
\end{align*}
while $V(P^A_{0,{\rm dir}}) \cap B = \emptyset$ and $V(P^A_{0,{\rm dir}}) \cap A_{i,h'}=\emptyset$ for all $i \le K$ and all $h' \ne h$.
(Sch$4'$) ensures that we can also choose $2s'-1$ (directed) paths $P^A_{1, {\rm dir} }, \dots, P^A_{2s'-1, {\rm dir} }$ in $G$ such that
the following conditions hold: 
\begin{itemize}
	\item For all $1 \le r < 2s'$, $P^A_{r, {\rm dir} }$ is a path from $A_{j,h}$ to $A_{j',h}$.
	\item For all $1 \le r \le s'$, $P^A_{r, {\rm dir} }$ contains precisely one vertex in $A_{i,h}$ for each $i\in \{j, \dots, j'\} \setminus \{i_1 \}$ and no other vertices.
	\item For all $s' <  r < 2s'$, $P^A_{r, {\rm dir} }$ contains precisely one vertex in $A_{i,h}$ for each $i\in \{j, \dots, j'\} \setminus \{i_1,i_1+1\}$ and no other vertices.
	\item $P^A_{0, {\rm dir} }, \dots, P^A_{2s'-1, {\rm dir} }$ are pairwise vertex-disjoint.
\end{itemize}

Let $Q$ be the union of $P^A_{0, {\rm dir} }, \dots, P^A_{2s'-1, {\rm dir} }$.
Thus $Q$ is a path system consisting of $2s'$ vertex-disjoint directed paths from $A_{j,h}$ to $A_{j',h}$.
Moreover, $V(Q)$ consists of precisely $2s'$ vertices in $A_{i,h}$ for every $j\le i\le j'$ and no other vertices.
Set $A'_{i,h}: = A_{i,h} \setminus V(Q)$ for all $i \le K$.
Note that
\begin{equation}\label{eq:sizeAih}
|A'_{i,h}| = \frac{m}{L}-2s'
\ge \frac{m}{L}-4 {\eps_0} n
\ge \frac{m}{L}-10{\eps_0} mK
\ge (1-\sqrt{\eps_0})\frac{m}{L}
\end{equation}
since $\eps_0\ll 1/K,1/L$.
Pick a new constant $\eps'$ such that $\eps,\eps_0 \ll \eps'\ll 1$.
Then Proposition~\ref{superslice}, (Sch$3'$) and \eqref{eq:sizeAih} together imply that 
$G[A'_{i,h}, A'_{i+1,h}]$ is still $[\eps',1/2]$-superregular
and so by Proposition~\ref{perfmatch} we can find a perfect matching in $G[A_{i,h}',A_{i+1,h}']$
for all $j \le i < j'$.
The union $Q'$ of all these matchings forms $m/L-2s'$ vertex-disjoint directed paths $P^A_{2s', {\rm dir}}, \dots, P^A_{m/L-1, {\rm dir}}$.
Note that $P^A_{0, {\rm dir}}, P^A_{1, {\rm dir}}, \dots, P^A_{m/L-1, {\rm dir}}$ are pairwise
vertex-disjoint and 
together cover precisely the vertices in $\bigcup_{i = j}^{j'} A_{i,h}$.%
    \COMMENT{Daniela: $\bigcup$ instead of $\cup$}
Moreover, $P^A_{0, {\rm dir}}$ contains $J^{*}_{A, {\rm dir}}$.

Similarly, we find $m/L$ vertex-disjoint directed paths $P^B_{0, {\rm dir}},P^B_{1, {\rm dir}}, \dots, \break P^B_{m/L-1, {\rm dir}}$
from $B_{j,h}$ to $B_{j',h}$ such that $P^B_{0, {\rm dir}}$ contains $J^{*}_{B, {\rm dir}}$ and
together the paths cover precisely the vertices in 
$\bigcup_{i = j}^{j'} B_{i,h}$.%
    \COMMENT{Daniela: $\bigcup$ instead of $\cup$}
For each $1 \le r < m/L$, let $P^A_r$ and $P^B_r$ be the
undirected paths obtained from $P^A_{r,{\rm dir}}$ and $P^B_{r,{\rm dir}}$ by ignoring the directions of all the edges.

Since $J^*_{A, {\rm dir}}\subseteq P^A_{0,{\rm dir}}$ and $J^*_{B, {\rm dir}}\subseteq P^B_{0,{\rm dir}}$
and since $J^*_{\rm dir}$ is the orientation of $J^*$ induced by $J_{\rm dir}$, it follows that
$P^A_{0,{\rm dir}} + P^B_{0,{\rm dir}} - J^*_{\rm dir } + J_{\rm dir}$ consists of two vertex-disjoint paths
$P_{0,{\rm dir}}$ and $P_{0,{\rm dir}}'$ from $A_{j,h} \cup B_{j,h}$ to $A_{j',h} \cup B_{j',h}$ with
$V(P_{0,{\rm dir}}) \cup V(P'_{0,{\rm dir}}) = V_0\cup V(P^A_{0,{\rm dir}}) \cup V(P^B_{0,{\rm dir}})$.
Let $P_{0}$ and $P'_{0}$ be the undirected paths obtained from $P_{0,{\rm dir}}$ and $P'_{0,{\rm dir}}$ by ignoring the directions of all the edges.
Let $EPS$ be the union of $P_{0}, P'_{0}, P^A_1, \dots, P^A_{m/L-1}, P^B_1, \dots, P^B_{m/L-1}$.
Then $EPS$ is an exceptional path system for $G$, as required. To see this, note that $J=EPS-EPS[A]-EPS[B]$ since
$e_J(A),e_J(B)=0$ by the definition of an exceptional system (see (EC3) in Section~\ref{sec:BES}).
\endproof

The next lemma uses the previous one to show that we can obtain many edge-disjoint exceptional factors by extending exceptional systems with suitable properties.

\begin{lemma} \label{lma:EF}
Suppose that $L,f,q,n,m/L,K/f \in \mathbb{N}$, that $K/f\ge 3$, that $0 <1/n \ll \eps,\eps_0  \ll 1$,
that $\eps_0\ll 1/K,1/L$ and $Lq/m\ll 1$.
Let $(G,\mathcal{P}, \mathcal{P}')$ be a $[K, L, m,\eps_0,\eps]$-scheme with $|V(G)\cup V_0|=n$.
Suppose that there exists a set $\mathcal{J}$ of $Lf q $ edge-disjoint exceptional systems
satisfying the following conditions: 
\begin{itemize}
    \item[{\rm (i)}] Each $J\in \mathcal{J}$ is either a Hamilton exceptional system with $e_J(A',B')=2$ or
     a matching exceptional system. 
	\item[{\rm (ii)}] For all $i \le f$ and all $h \le L$, $\mathcal{J}$ contains precisely $q$ $(i_1,i_2)$-ES of style $h$
(with respect to $(\cP,\cP')$) for which $i_1,i_2 \in \{ (i-1)K/f+2, \dots, iK/f \}$.
\end{itemize}
Then there exist $q$ edge-disjoint exceptional factors with parameters~$(L,f)$ for $G$ (with respect to  $(\cP,\cP')$)
covering all edges in~$\bigcup\mathcal{J}$.
\end{lemma}

Recall that the canonical interval partition $\mathcal{I}(f,K)$ of $[K]$ into $f$ intervals consists of the intervals
$\{ (i-1)K/f+1, \dots, iK/f +1\}$ for all $i\le f$. So (ii) ensures
that for each interval $I\in \mathcal{I}(f,K)$ and each $h\le L$, the set $\mathcal{J}$ contains precisely $q$ exceptional systems of style $h$
whose edges are only incident to vertices in $V_0$ and vertices
belonging to clusters $A_{i_1}$ and $B_{i_2}$ for which both $i_1$ and $i_2$ lie in the interior of $I$.
We will use Lemma~\ref{lma:EPS} to extend each such exceptional system into an exceptional path system of style $h$ spanning~$I$.

\removelastskip\penalty55\medskip\noindent{\bf Proof of Lemma~\ref{lma:EF}. }
Choose a new constant $\eps'$ with $\eps,Lq/m\ll \eps'\ll 1$.
Let $\mathcal{J}_{1}, \dots, \mathcal{J}_{q }$ be a partition of $\mathcal{J}$ such that for all $j\le q$, $h \le L$ and $i \le f$, 
the set $\mathcal{J}_j$ contains precisely one $(i_1,i_2)$-ES of style $h$ with $i_1,i_2 \in \{ (i-1)K/f+2, \dots, iK/f \}$.
Thus each $\mathcal{J}_j$ consists of $Lf$ exceptional systems.
For each $j\le q$ in turn, we will choose an exceptional factor $EF_j$ with parameters $(L,f)$ for $G$ (with respect to $(\cP,\cP')$)
such that $EF_j$ and $EF_{j'}$ are edge-disjoint for all $j' < j$ and $EF_j$ contains all edges
of the exceptional systems in $\mathcal{J}_j$.
Assume that for some $1\le j \le q$ we have already constructed $EF_1, \dots, EF_{j-1}$. In order to construct $EF_j$, we will choose the $Lf$ exceptional path systems
forming $EF_j$ one by one, such that each of these exceptional path systems is edge-disjoint from $EF_1,\dots,EF_{j-1}$ and
contains precisely one of the exceptional systems in $\mathcal{J}_j$. Suppose that we have already chosen some of these exceptional
path systems and that next we wish to choose an exceptional path system of style $h$ which spans the interval $I$ of the canonical
interval partition $\mathcal{I}(f,K)$ and contains $J\in \mathcal{J}_j$.
Let $G'$ be the oriented graph obtained from $G$ by deleting all the edges in the path systems already chosen for $EF_j$ as well as deleting all the edges in $EF_1,\dots,EF_{j-1}$.
Recall that $V(G)=A\cup B$. Thus $\Delta(G - G') \le 2j < 3q $ by~\eqref{EFdeg}.
Together with Proposition~\ref{superslice} this implies that $(G',\mathcal{P},\mathcal{P}')$ is still a $[K, L, m, \epszero, \eps']$-scheme.
(Here we use that $\Delta(G - G') < 3 q=3Lq/m\cdot m/L$ and
$\eps,Lq/m\ll \eps'\ll 1$.) So we can apply Lemma~\ref{lma:EPS} with $\eps'$ playing the role of $\eps$ to obtain an exceptional path system of style $h$ for $G'$ (and thus for $G$) which spans $I$ and contains all edges of $J$.
This completes the proof of the lemma.
\endproof


\section{The Robust Decomposition Lemma} \label{sec:robdec}

The aim of this section is to state the robust decomposition lemma (Lem\-ma~\ref{rdeclemma}).
This is the key lemma proved in~\cite{Kelly} and guarantees
the existence of a `robustly decomposable' digraph $G^{\rm rob}_{\rm dir}$ within a `setup'.
For our purposes, we will then derive an undirected version in Corollary~\ref{rdeccor} to construct a robustly decomposable graph $G^{\rm rob}$.
Then $G^{\rm rob}+H$ will have a Hamilton decomposition for any sparse regular graph~$H$ which is edge-disjoint from $G^{\rm rob}$.
The crucial ingredient of a setup is a `universal walk', which we introduce in the next subsection.
The (proof of the) robust decomposition lemma then uses
edges guaranteed by this universal walk to `balance out' edges of the graph $H$ when constructing the Hamilton decomposition of
$G^{\rm rob}+H$.

\subsection{Chord Sequences and Universal Walks}\label{newlabel}
Let $R$ be a digraph whose vertices are $V_1,\dots,V_k$ and suppose that $C=V_1\dots V_k$ is a Hamilton cycle of $R$.
(Later on the vertices of $R$ will be clusters. So we denote them by capital letters.)

A \emph{chord sequence $CS(V_i,V_j)$} from $V_i$ to $V_j$ in $R$ is an ordered sequence of edges of the form
\[ CS(V_i,V_j) = (V_{i_1-1} V_{i_2}, V_{i_2-1} V_{i_3},\dots, V_{i_t-1} V_{i_{t+1}}),\]
where $V_{i_1}=V_i$, $V_{i_{t+1}} = V_j$ and the edge $V_{i_s-1} V_{i_{s+1}}$ belongs to $R$ for each $s\le t$.

If $i=j$ then we consider the empty set to be a chord sequence from $V_i$ to $V_j$. 
Without loss of generality, we may assume that $CS(V_i,V_j)$ does not contain any edges of $C$.
(Indeed, suppose that  $V_{i_s-1} V_{i_{s+1}}$ is an edge of $C$.
Then $i_s=i_{s+1}$ and so we can obtain a chord sequence from $V_i$ to $V_j$ with fewer edges.) 
For example, if $V_{i-1}V_{i+1}\in E(R)$, then the edge $V_{i-1}V_{i+1}$ is a chord sequence from $V_i$ to $V_{i+1}$.

The crucial property of chord sequences is that they satisfy a `local balance' condition. 
Suppose that $CS$ is obtained by concatenating several chord sequences 
$$CS(V_{i_1},V_{i_2}),CS(V_{i_2},V_{i_3}),\dots,CS(V_{i_{k-1}},V_{i_k})$$ so that $V_{i_1}=V_{i_k}$.
Then for every cluster $V_i$, the number of edges of $CS$ leaving $V_{i-1}$ equals the number of edges entering $V_i$.
We will not use this property explicitly, but it underlies the proof of the robust decomposition lemma (Lemma~\ref{rdeclemma}) that we apply
and appears implicitly e.g.~in~(U3).

A closed walk $U$ in $R$ is a \emph{universal walk for $C$
with parameter $\ell'$} if the following conditions hold:
\begin{itemize}
\item[(U1)] For every $i\le k$ there is a chord sequence $ECS (V_i,V_{i+1})$
from $V_i$ to $V_{i+1}$ such that $U$ contains all edges of all these chord sequences (counted with multiplicities) and all remaining edges of $U$ lie on $C$.
\item[(U2)] Each $ECS (V_i,V_{i+1})$ consists of at most $\sqrt{\ell'}/2$ edges.
\item[(U3)] $U$ enters each $V_i$ exactly $\ell'$ times and  leaves each $V_i$ exactly $\ell'$ times.
\end{itemize}
Note that condition~(U1) means that if an edge $V_iV_j\in E(R)\setminus E(C)$ occurs in total 5 times (say) in
$ECS(V_1,V_2),\dots,ECS(V_{k},V_1)$ then it occurs precisely 5 times in $U$. We will identify each occurrence of $V_iV_j$ in
$ECS(V_1,V_2),\dots,ECS(V_{k},V_1)$ with a (different) occurrence of $V_iV_j$ in $U$. 
Note that the edges of $ECS(V_i,V_{i+1})$ are allowed to appear in a different order within $ECS(V_i,V_{i+1})$ and within $U$.

\begin{lemma}\label{lem:univwalk}
Let $R$ be a digraph with vertices $V_1,\dots,V_k$. Suppose that $C=V_1\dots V_k$ is a Hamilton cycle of $R$
and that $V_{i}V_{i+2}\in E(R)$ for every $1\le i\le k$.
Let $\ell'\ge 4$ be an integer.%
\COMMENT{need 4 rather than 2 to ensure $1 \le \sqrt{\ell'}/2$.}
Let $U_{\ell'}$ the multiset obtained from $\ell'-1$ copies of $E(C)$ by adding
$V_{i}V_{i+2}\in E(R)$ for every $1\le i\le k$. Then the edges in $U_{\ell'}$ can be ordered so that the
resulting sequence forms a universal walk for $C$ with parameter~$\ell'$.
\end{lemma}
In the remainder of this section, we will also write $U_{\ell'}$ for the universal walk guaranteed by Lemma~\ref{lem:univwalk}. 

\proof
Let us first show that the edges in $U_{\ell'}$ can be ordered so that the resulting sequence forms a closed walk in~$R$.
To see this, consider the multidigraph $U$ obtained from $U_{\ell'}$ by deleting one copy of~$E(C)$.
Then $U$ is $(\ell'-1)$-regular and thus has a decomposition into 1-factors. We order the edges of $U_{\ell'}$ as follows: 
We first traverse all cycles of the 1-factor decomposition of $U$ which contain the cluster $V_1$.
Next, we traverse the edge $V_1V_2$ of $C$. Next we traverse all those cycles of the 1-factor decomposition which contain
$V_2$ and which have not been traversed so far. Next we traverse the edge $V_2V_3$ of $C$ and so on
until we reach $V_1$ again. 

Recall that, for each $1\le i\le k$, the edge $V_{i-1}V_{i+1}$ is a chord sequence from $V_i$ to $V_{i+1}$. Thus we can take
$ECS(V_i,V_{i+1}):=V_{i-1}V_{i+1}$. Then $U_{\ell'}$ satisfies (U1)--(U3).
\endproof

\subsection{Setups and the Robust Decomposition Lemma}\label{newlabel2}

The aim of this subsection is to state the robust decomposition lemma (Lemma~\ref{rdeclemma}, proved in~\cite{Kelly})
and derive Corollary~\ref{rdeccor}, which we shall use later on in order to prove Theorem~\ref{1factstrong}.
The robust decomposition lemma guarantees the existence of a `robustly decomposable' digraph $G^{\rm rob}_{\rm dir}$
within a `setup'. Roughly speaking, a setup is a digraph $G$ together with its `reduced digraph' $R$,
which contains a Hamilton cycle $C$ and a universal walk $U$.
In our application, we will have two setups: $G[A]$ and $G[B]$ will play the role of $G$, and $R$ will be the complete digraph
in both cases. To define a setup formally, we first need to define certain `refinements' of partitions.%
   \COMMENT{osthus changed graph to digraph (3 times)}

Given a digraph $G$ and a partition $\cP$ of $V(G)$ into $k$ clusters $V_1,\dots,V_k$ of equal size,
we say that a partition $\cP'$ of $V$ is an \emph{$\ell'$-refinement of $\cP$} if $\cP'$ is obtained by splitting each $V_i$
into $\ell'$ subclusters of equal size. (So $\cP'$ consists of $\ell'k$ clusters.)
$\cP'$ is an \emph{$\eps$-uniform $\ell$-refinement}
of $\cP$ if it is an $\ell$-refinement of $\cP$ which satisfies the following condition:
Whenever $x$ is a vertex of $G$, $V$ is a cluster in $\cP$ and $|N^+_G(x)\cap V|\ge \eps |V|$
then $|N^+_G(x)\cap V'|=(1\pm \eps)|N^+_G(x)\cap V|/\ell$ for each cluster $V'\in \cP'$ with $V'\subseteq V$.
The inneighbourhoods of the vertices of $G$ satisfy an analogous condition.
We need the following simple observation from~\cite{Kelly}. The proof proceeds by considering a random partition
to obtain a uniform refinement.
\begin{lemma} \label{randompartition}
Suppose that $0<1/m \ll 1/k,\eps \ll \eps', d,1/\ell \le 1$ and that%
     \COMMENT{Daniela: previously only had $m/\ell\in\mathbb{N}$}
$n,k,\ell,m/\ell\in\mathbb{N}$.
Suppose that $G$ is a digraph on $n=km$ vertices and that $\cP$ is a partition of $V(G)$ into
$k$ clusters of size $m$. Then there exists an $\eps$-uniform $\ell$-refinement of $\cP$. Moreover, any
$\eps$-uniform $\ell$-refinement $\cP'$ of $\cP$ automatically satisfies
the following condition:
\begin{itemize}
\item Suppose that $V$, $W$ are clusters in $\cP$ and $V',W'$ are clusters in $\cP'$ with $V'\subseteq V$ and
$W'\subseteq W$. If $G[V,W]$ is $[\eps,d']$-superregular for some $d'\ge d$ then $G[V',W']$ is $[\eps',d']$-superregular.
\end{itemize}
\end{lemma}

We will also need the following definition from~\cite{Kelly}.
$(G,\cP,\cP',R,C,U,U')$ is called an \emph{$(\ell',k,m,\eps,d)$-setup} if the following properties are satisfied:
\begin{itemize}
\item [(ST1)] $G$ and $R$ are digraphs. $\mathcal{P}$ is a partition of $V(G)$ into
$k$ clusters of size $m$. The vertex set of $R$ consists of these clusters.
\item[(ST2)] For every edge $VW$%
\COMMENT{NOTE: I have changed each $U$ to a $V$ in (ST2) and (ST3). Otherwise, confusing with
$U$ being a universal walk in (ST4).}
of $R$ the corresponding pair $G[V,W]$ is $(\eps,\ge d)$-regular.
\item[(ST3)] $C$ is a Hamilton cycle of $R$ and for every edge $VW$ of $C$ the corresponding pair $G[V,W]$ is $[\eps,\ge d]$-superregular.
\item[(ST4)] $U$ is a universal walk for $C$ with parameter~$\ell'$ and $\cP'$ is an $\eps$-uniform $\ell'$-refinement%
   \COMMENT{need $\eps$-uniformity here in order to be able to derive (ST6) from (Sch3') in the proof of Lemma~\ref{lem:setup}.}
of $\cP$.
\item[(ST5)] Suppose that $C=V_1\dots V_{k}$ and let $V_j^1,\dots,V_j^{\ell'}$ denote the clusters in $\cP'$ which are contained
in $V_j$ (for each $1\le j\le k$). Then $U'$ is a closed walk on the clusters in $\cP'$ which is obtained from $U$ as follows:
When $U$ visits $V_j$ for the $a$th time, we let $U'$ visit the subcluster $V_j^a$ (for all $1\le a\le \ell'$).
\item[(ST6)] Each edge of $U'$ corresponds to an $[\eps,\ge d]$-superregular pair in $G$.
\end{itemize}
In~\cite{Kelly}, in a setup, the digraph $G$ could also contain an exceptional set, but since we are only using the definition
in the case when there is no such exceptional set, we have only stated it in this special case.

Suppose that $(G,\mathcal{P},\mathcal{P}')$ is a $[K,L,m,\eps_0,\eps]$-scheme.
Recall that $A_1,\dots,A_K$ and $B_1,\dots,B_K$ denote the clusters of $\mathcal{P}$.
Let $\mathcal{Q}_{A}:=\{A_1,\dots,A_K\}$, $\mathcal{Q}_{B}:=\{B_1,\dots,B_K\}$ and let $C_A = A_1 \dots A_K$ and $C_B = B_1 \dots B_K$ be (directed)
cycles. Suppose that $\ell',m/\ell'\in\mathbb{N}$ with $\ell'\ge 4$.
Let $\mathcal{Q}'_{A}$ be an $\eps$-uniform $\ell'$-refinement of $\mathcal{Q}_A$.
Let $R_{A}$ be the complete digraph whose vertices are the clusters
in $\mathcal{Q}_A$. 
Let $U_{A,\ell'}$ be a universal walk for $C_A$ with parameter $\ell'$ as defined in Lemma~\ref{lem:univwalk}.
Let $U'_{A,\ell'}$ be the closed walk obtained from $U_{A,\ell'}$ as described in~(ST5).
 We will call
$$
(G[A], \mathcal{Q}_A , \mathcal{Q}'_A ,R_{A},C_{A},U_{A,\ell'},U'_{A,\ell'})$$ the \emph{$A$-setup
associated to~$(G, \mathcal{P},\mathcal{P}')$}. 
Define $\mathcal{Q}'_{B}$,  $R_B$, $U_{B,\ell'}$ and $U'_{B,\ell'}$ similarly. We will call
$$
(G[B], \mathcal{Q}_B , \mathcal{Q}'_B ,R_{B},C_{B},U_{B,\ell'},U'_{B,\ell'})$$ the \emph{$B$-setup
associated to~$(G, \mathcal{P},\mathcal{P}')$}. 
The following lemma shows that both the $A$-setup and the $B$-setup indeed satisfy all the conditions in the definition of a setup.%
    \COMMENT{Daniela: added $1/K$ into the hierarchy of Lemam~\ref{lem:setup} as we need it to apply Lemma~\ref{randompartition}}

\begin{lemma}\label{lem:setup}
Suppose that $1/m\ll 1/K,\eps_0,\eps\ll \eps',1/\ell'$ and $K,L,m/L, \ell',$ $m/\ell'\in\mathbb{N}$ with $\ell'\ge 4$.
Suppose that $(G,\mathcal{P},\mathcal{P}')$ is a $[K,L,m,\eps_0,\eps]$-scheme.
Then each of
$$(G[A], \mathcal{Q}_A , \mathcal{Q}'_A ,R_{A},C_{A},U_{A,\ell'},U'_{A,\ell'}) \ \ \ \text{and}\ \ \ (G[B], \mathcal{Q}_B , \mathcal{Q}'_B ,R_{B},C_{B},U_{B,\ell'},U'_{B,\ell'})$$
is an $(\ell',K,m,\eps',1/2)$-setup.
\end{lemma}
\proof
It suffices to show that $(G[A], \mathcal{Q}_A , \mathcal{Q}'_A ,R_{A},C_{A},U_{A,\ell'},U'_{A,\ell'})$ is an $(\ell',K,m,$ $\eps',1/2)$-setup.
Clearly, (ST1) holds.
(Sch$3'$) implies that (ST2) and~(ST3) hold.%
	\COMMENT{The only reason why (Sch3') also includes a superregularity condition for pairs of clusters in $\cP$
is enables us to get (ST2) and~(ST3) without thinking. Alternatively, we could have added the following prop.
PROPOSITION: Suppose that $L,m \in \mathbb{N}$ and $0 <\eps \ll d \le 1$. Let $G$ be a graph and suppose that
$A_j,B_{j'}\subseteq V(G)$ are sets of vertices with $|A_j|=|B_{j'}|=m$ for all $j,j' \le L$
such that all these $2L$ sets are pairwise disjoint. 
Suppose that $G[A_{j},B_{j'}]$  is $(\eps, d)$-superregular for all $j, j' \leq L$.
Let $A:= \bigcup_{j\le L} A_{j}$ and $B:= \bigcup_{j\le L} B_{j}$.
Then $G[A,B]$ is $(2\eps^{1/3}, d)$-superregular.
PROOF:
Let $n:=mL$.
Consider $X \subseteq A$ and $Y \subseteq B$ with $|X|,|Y| \ge \eps^{1/3} n$.
Let $X_j:=X \cap A_j$ and $Y_{j}:= Y \cap B_j$ for all $j \le m$.
Then%
    \COMMENT{Have $|X_j||Y_{j'}|-\eps m^2$ instead of just $|X_j||Y_{j'}|$ to account for
the possibility that at least one of $|X_j$, $Y_{j'}$ has size less than $\eps m$.}
\begin{align*}
e_G(X,Y) & \ge  \sum_{j,j'\le m}(d-\eps) (|X_j||Y_{j'}|-\eps m^2) =(d-\eps)(|X||Y|-\eps n^2) \\
& \ge (d-\eps)(|X||Y|-\eps^{1/3}|X||Y|) \ge (d-2\eps^{1/3})|X||Y|.
\end{align*}
An upper bound follows similarly. Moreover, the minimum degree requirement on $G[A,B]$ is trivially satisfied.
}
Lemma~\ref{lem:univwalk} implies~(ST4). (ST5) follows from the definition of $U'_{A,\ell'}$.
(ST6) follows from Lemma~\ref{randompartition} since $\mathcal{Q}'_{A}$ is an $\eps$-uniform $\ell'$-refinement of~$\mathcal{Q}_A$.
\endproof

We now state the robust decomposition lemma from~\cite{Kelly}.
Recall that this guarantees the existence of a `robustly decomposable' digraph $G^{\rm rob}_{\rm dir}$, whose crucial property
is that $H + G^{\rm rob}_{\rm dir}$ has a Hamilton decomposition for any sparse regular digraph~$H$ which is edge-disjoint
from $G^{\rm rob}_{\rm dir}$.

$G^{\rm rob}_{\rm dir}$ consists of digraphs $CA_{{\rm dir}}(r)$ (the `chord absorber') and $PCA_{{\rm dir}}(r)$ 
(the `parity extended cycle switcher')
together with some special factors. $G^{\rm rob}_{\rm dir}$ is constructed in two steps:
given a suitable set $\mathcal{SF}$ of special factors, the lemma first `constructs' $CA_{{\rm dir}}(r)$ and then,
given another suitable set $\mathcal{SF}'$ of special factors, the lemma `constructs' $PCA_{{\rm dir}}(r)$.
The reason for having two separate steps is that in~\cite{Kelly}, it is not clear how to construct $CA_{{\rm dir}}(r)$
after constructing $\mathcal{SF}'$ (rather than before), as the removal of $\mathcal{SF}'$ from the digraph under consideration 
affects its properties considerably.

\begin{lemma} \label{rdeclemma}
Suppose that $0<1/m\ll 1/k\ll \eps \ll 1/q \ll 1/f \ll r_1/m\ll d\ll 1/\ell',1/g\ll 1$ and%
   \COMMENT{In the Kelly paper have $1/n$ instead of $1/m$ in the hierarchy. But since $V_0=\emptyset$ in our setting,
it doesn't make that much sense to introduce $n$ here. So I replaced $1/n$ by $1/m$.}
that $rk^2\le m$. Let
$$r_2:=96\ell'g^2kr, \ \ \ r_3:=rfk/q, \ \ \ r^\diamond:=r_1+r_2+r-(q-1)r_3, \ \ \ s':=rfk+7r^\diamond
$$
and suppose that $k/14, k/f, k/g, q/f, m/4\ell', fm/q, 2fk/3g(g-1) \in \mathbb{N}$.
Suppose that $(G,\cP,\cP',R,C,U,U')$ is an $(\ell',k,m,\eps,d)$-setup and $C=V_1\dots V_k$.
Suppose that $\cP^*$ is a $(q/f)$-refinement
of $\cP$ and that $SF_1,\dots, SF_{r_3}$ are edge-disjoint special factors with parameters $(q/f,f)$ 
with respect to $C$, $\cP^*$ in $G$. Let $\mathcal{SF}:=SF_1+\dots +SF_{r_3}$.
Then there exists a digraph $CA_{{\rm dir}}(r)$ for which the following holds:
\begin{itemize}
\item[\rm (i)] $CA_{{\rm dir}}(r)$ is an $(r_1+r_2)$-regular spanning subdigraph of $G$ which is edge-disjoint from $\mathcal{SF}$.
\item[\rm (ii)] Suppose that $SF'_1,\dots, SF'_{r^\diamond}$ are special factors with parameters $(1,7)$
with respect to $C$, $\cP$ in $G$ which are edge-disjoint from each other and from $CA_{{\rm dir}}(r)+ \mathcal{SF}$.%
   \COMMENT{In the Kelly paper we write $CA_{{\rm dir}}(r)\cup \mathcal{SF}$ instead of $CA_{{\rm dir}}(r)+ \mathcal{SF}$
(and similarly below). But with our def of $+$ and $\cup $ in this paper, we have to use $+$ since we allow for a fictive edge
$xy$ in $\mathcal{SF}$ to also occur in $CA_{{\rm dir}}(r)$, and in this case $CA_{{\rm dir}}(r)+ \mathcal{SF}$
willl contain two copies of that edge.} 
Let $\mathcal{SF}':=SF'_1+\dots +SF'_{r^\diamond}$.
Then there exists a digraph $PCA_{{\rm dir}}(r)$ for which the following holds:
\begin{itemize}
\item[\rm (a)] $PCA_{{\rm dir}}(r)$ is a $5r^\diamond$-regular spanning subdigraph of $G$ which
is edge-disjoint from $CA_{{\rm dir}}(r)+ \mathcal{SF}+ \mathcal{SF}'$.
\item[\rm (b)] Let $\mathcal{SPS}$ be the set consisting of all the $s'$ special path systems
contained in $\mathcal{SF}+ \mathcal{SF}'$.
Suppose that $H$ is an $r$-regular digraph on $V(G)$ which is edge-disjoint from $G^{\rm rob}_{\rm dir}:=CA_{{\rm dir}}(r)+ PCA_{{\rm dir}}(r)+ \mathcal{SF}+ \mathcal{SF}'$.
Then $H+G^{\rm rob}_{\rm dir}$ has a decomposition into $s'$
edge-disjoint Hamilton cycles $C_1,\dots,C_{s'}$.
Moreover, $C_i$ contains one of the special path systems from $\mathcal{SPS}$, for each $i\le s'$.
\end{itemize}
\end{itemize}
\end{lemma}

Recall from Section~\ref{sec:SF} that we always view fictive edges in special factors as being distinct from each other and
from the edges in other graphs. So for example, saying that $CA_{{\rm dir}}(r)$ and $\mathcal{SF}$ are edge-disjoint in
Lemma~\ref{rdeclemma} still allows for a fictive edge $xy$ in $\mathcal{SF}$ to occur in $CA_{{\rm dir}}(r)$ as well
(but $CA_{{\rm dir}}(r)$ will avoid all non-fictive edges in $\mathcal{SF}$).

In the proof of Theorem~\ref{1factstrong} we will use the following `undirected' consequence of Lemma~\ref{rdeclemma}.

\begin{cor} \label{rdeccor}
Suppose that $0<1/m\ll \eps_0,1/K\ll \eps \ll 1/L \ll 1/f \ll r_1/m\ll 1/\ell',1/g\ll 1$ and
that $rK^2\le m$. Let
$$r_2:=96\ell'g^2Kr, \ \ \ r_3:=rK/L, \ \ \ r^\diamond:=r_1+r_2+r-(Lf-1)r_3, \ \ \ s':=rfK+7r^\diamond
$$
and suppose that $K/14, K/f, K/g, m/4\ell', m/L, 2fK/3g(g-1) \in \mathbb{N}$.
Suppose that $(G_{\rm dir},\cP,\cP')$ is a $[K,L,m,\eps_0,\eps]$-scheme and let $G'$ denote the underlying undirected graph
of $G_{\rm dir}$.
Suppose that $EF_1,\dots, EF_{r_3}$ are edge-disjoint exceptional factors with parameters $(L,f)$ for $G_{\rm dir}$ (with respect to $(\cP,\cP')$).
Let $\mathcal{EF}:=EF_1+\dots + EF_{r_3}$.
Then there exists a graph $CA(r)$ for which the following holds:
\begin{itemize}
\item[\rm (i)] $CA(r)$ is a $2(r_1+r_2)$-regular spanning subgraph of $G'$ which is edge-disjoint from $\mathcal{EF}$.
\item[\rm (ii)] Suppose that $EF'_1,\dots, EF'_{r^\diamond}$ are exceptional factors with parameters $(1,7)$ for $G_{\rm dir}$
(with respect to $(\cP,\cP)$) which are edge-disjoint from each other and from $CA(r)+ \mathcal{EF}$.
Let $\mathcal{EF}':=EF'_1+\dots + EF'_{r^\diamond}$.
Then there exists a graph $PCA(r)$ for which the following holds:
\begin{itemize}
\item[\rm (a)] $PCA(r)$ is a $10r^\diamond$-regular spanning subgraph of $G'$ which
is edge-disjoint from $CA(r)+ \mathcal{EF}+ \mathcal{EF}'$.
\item[\rm (b)] Let $\mathcal{EPS}$ be the set consisting of all the $s'$ exceptional path systems
contained in $\mathcal{EF}+ \mathcal{EF}'$.
Suppose that $H_A$ is a $2r$-regular graph on $A=\bigcup_{i=1}^K A_i$ and $H_B$ is a $2r$-regular graph on $B=\bigcup_{i=1}^K B_i$.
Suppose that $H := H_A + H_B$ is edge-disjoint from $G^{\rm rob}:=CA(r)+ PCA(r)+ \mathcal{EF}+ \mathcal{EF}'$.
Then $H+ G^{\rm rob}$ has a decomposition into $s'$
edge-disjoint $2$-factors $H_1, \dots, H_{s'}$ such that each $H_i$ contains one of the exceptional path systems from $\mathcal{EPS}$.
Moreover, for each $i \le s'$, the following assertions hold:%
   \COMMENT{Daniela: replaced $V(G_{\rm dir})$ by $V(G_{\rm dir})\cup V_0$ in (b$_1$) and (b$_2$). Also added new sentence after the statement of the cor. TO DO: probably have
to change this in the bip paper as well!!} 
\begin{itemize}
	\item[\rm (b$_1$)] If the exceptional path system contained in $H_i$ is a Hamilton exceptional path system, then $H_i$ is a Hamilton cycle on $V(G_{\rm dir})\cup V_0$.
	\item[\rm (b$_2$)] If the exceptional path system contained in $H_i$ is a matching exceptional path system, then $H_i$ is the union of a Hamilton cycle
on $A'=A\cup A_0$ and a Hamilton cycle on $B'=B\cup B_0$.
In particular, if both $|A'|$ and $|B'|$ are even, then $H_i$ is the union of two edge-disjoint perfect matchings on $V(G_{\rm dir})\cup V_0$.
\end{itemize}
\end{itemize}
\end{itemize}
\end{cor}
We remark that, as usual, in Corollary~\ref{rdeccor} we write $A_0$ and $B_0$ for the exceptional sets of $\cP$, $V_0$ for $A_0\cup B_0$, and
$A_1,\dots,A_K,B_1,\dots,B_K$ for the clusters in $\cP$. Note that the vertex set of each of $\mathcal{EF}$, $\mathcal{EF}'$, $G^{\rm rob}$
includes $V_0$ while that of $G_{\rm dir}$, $CA(r)$, $PCA(r)$, $H$ does not.

Moreover, note that matching exceptional systems are only constructed if both $|A'|$ and $|B'|$ are even.
Indeed, we only construct matching exceptional systems in the case when $e_G(A',B')  < D$.
But by Proposition~\ref{prp:e(A',B')}(ii), in this case we have that $n = 0 \pmod4$ and $|A'| = |B'| = n/2$.
Therefore, Corollary~\ref{rdeccor}(ii)(b) implies that $H+ G^{\rm rob}$ has a decomposition into Hamilton cycles and perfect matchings.
The proportion of Hamilton cycles (and perfect matchings) in this decomposition is determined by $\mathcal{EF}+ \mathcal{EF}'$, 
and does not depend on~$H$.

{\removelastskip\penalty55\medskip\noindent{\bf Proof of Corollary~\ref{rdeccor}. }
Choose  new constants $\eps', d$ such that $\eps\ll \eps'\ll 1/L$ and $r_1/m\ll d\ll 1/\ell',1/g$.
Consider the $A$-setup $(G_{\rm dir}[A], \mathcal{Q}_A , \mathcal{Q}'_A ,R_{A},C_{A},U_{A,\ell'},$ $U'_{A,\ell'})$ associated to $(G_{\rm dir},\cP,\cP')$.
By Lemma~\ref{lem:setup}, this
is an $(\ell',K,m,\eps',1/2)$-setup and thus also an $(\ell',K,m,\eps',d)$-setup.

Recall%
    \COMMENT{Daniela: new stuff}
that $\mathcal{P}'$ is obtained from $\mathcal{P}$ by partitioning each cluster $A_i$ of $\mathcal{P}$ into $L$ sets $A_{i,1},\dots,A_{i,L}$
of equal size and partitioning each cluster $B_i$ of $\mathcal{P}$ into $L$ sets $B_{i,1},\dots,B_{i,L}$ of equal size.
Let $\mathcal{Q}''_A:=\{A_{1,1},\dots,A_{K,L}\}$. (So $\mathcal{Q}''_A$ plays the role of $\mathcal{Q}'_A$ in~(\ref{defparts}).)
Let $EF^*_{i,A,{\rm dir}}$ be as defined in Section~\ref{sec:BEPS}.
Recall from there that, for each $i\le r_3$, $EF^*_{i,A,{\rm dir}}$ is a special factor with parameters $(L,f)$ with respect to
$C_A =A_1 \dots A_K$, $\mathcal{Q}''_A$ in%
    \COMMENT{Daniela: replaced $\mathcal{Q}'_A$ by $\mathcal{Q}''_A$}
$G_{\rm dir}[A]$ such that
${\rm Fict}(EF^*_{i,A,{\rm dir}})$ is the union of $J^*[A]$ over all the $Lf$ exceptional systems $J$ contained in $EF_{i}$.
Thus we can apply Lemma~\ref{rdeclemma} to $(G_{\rm dir}[A], \mathcal{Q}_A , \mathcal{Q}'_A ,R_{A},C_{A},U_{A,\ell'},U'_{A,\ell'})$
with $K$, $Lf$, $\eps'$ playing the roles of $k$, $q$, $\eps$ in order to obtain a spanning subdigraph $CA_{A,{\rm dir}}(r)$
of $G_{\rm dir}[A]$ which satisfies Lemma~\ref{rdeclemma}(i). 
Similarly, we obtain a spanning subdigraph $CA_{B,{\rm dir}}(r)$ of $G_{\rm dir}[B]$ which satisfies Lemma~\ref{rdeclemma}(i)
(with $G_{{\rm dir}}[B]$ playing the role of $G$). 
Thus the underlying undirected graph $CA(r)$ of $CA_{A,{\rm dir}}(r) + CA_{B,{\rm dir}}(r)$ satisfies Corollary~\ref{rdeccor}(i). 

Now let $EF'_1,\dots, EF'_{r^\diamond}$ be exceptional factors as described in Corollary~\ref{rdeccor}(ii).
Similarly as before, for each $i\le r^\diamond$, $(EF'_i)^*_{A,{\rm dir}}$ is a special factor
with parameters $(1,7)$ with respect to $C_A$, $\mathcal{Q}_A$ in $G_{\rm dir}[A]$ such that
${\rm Fict}((EF'_i)^*_{A,{\rm dir}})$ is the union of $J^*[A]$ over all the $7$ exceptional systems $J$ contained in $EF'_i$.
Thus we can apply Lemma~\ref{rdeclemma} (with $G_{{\rm dir}}[A]$ playing the role of $G$) to obtain a spanning subdigraph $PCA_{A,{\rm dir}}(r)$
of $G_{\rm dir}[A]$ which satisfies Lemma~\ref{rdeclemma}(ii)(a) and~(ii)(b).
Similarly, we obtain a spanning subdigraph $PCA_{B,{\rm dir}}(r)$ of $G_{\rm dir}[B]$ which satisfies Lemma~\ref{rdeclemma}(ii)(a) and~(ii)(b)
(with $G_{{\rm dir}}[B]$ playing the role of $G$).
Thus the underlying undirected graph $PCA(r)$ of $PCA_{A,{\rm dir}}(r) + PCA_{B,{\rm dir}}(r)$ satisfies Corollary~\ref{rdeccor}(ii)(a).

It remains to check that Corollary~\ref{rdeccor}(ii)(b) holds too. Thus let $H=H_A+H_B$ be as described in Corollary~\ref{rdeccor}(ii)(b).
Let $H_{A,{\rm dir}}$ be an $r$-regular orientation of $H_A$. (To see that such an orientation exists, apply Petersen's theorem, i.e.~Theorem~\ref{petersen}, to obtain
a decomposition of $H_A$ into $2$-factors and then orient each $2$-factor to obtain a (directed) $1$-factor.)%
     \COMMENT{However, even if $H\subseteq G'$, this orientation might not agree with $G_{\rm dir}$, i.e.~we might not have that
$H\subseteq G_{\rm dir}$. This is the reason why we cannot assume that $H\subseteq G$ in Lemma~\ref{rdeclemma}(ii)(b).
($H\subseteq G$ would fit better to the remainder of the Kelly paper.)}
Let $\mathcal{EF}^*_{A,{\rm dir}}  : =  EF^*_{1,A,{\rm dir}} + \dots +  EF^*_{r_3,A,{\rm dir}}$ and let $(\mathcal{EF}')^*_{A,{\rm dir}} : =  (EF'_1)^*_{A,{\rm dir}} + \dots +  (EF'_{r^{\diamond}})^*_{A,{\rm dir}}$.%
	\COMMENT{Previsouly, we have `Let $\mathcal{EF}^*_{A,{\rm dir}}$ be the union of the $EF^*_{i,A,{\rm dir}}$ over all $i\le r_3$
and let $(\mathcal{EF}')^*_{A,{\rm dir}}$ be the union of the $(EF'_i)^*_{A,{\rm dir}}$ over all $i\le r^\diamond$.'
But $\mathcal{EF}^*_{A,{\rm dir}}$ and $(\mathcal{EF}')^*_{A,{\rm dir}}$ are multigraphs, so we should use `sum' instead of `union'.
}
Then Lemma~\ref{rdeclemma}(ii)(b) implies that
$H_{A,{\rm dir}}+ CA_{A,{\rm dir}}(r)+ PCA_{A,{\rm dir}}(r)+ \mathcal{EF}^*_{A,{\rm dir}}+ (\mathcal{EF}')^*_{A,{\rm dir}}$
has a decomposition into $s'$ edge-disjoint (directed) Hamilton cycles $C'_{1,A},\dots,C'_{s',A}$ such that each $C'_{i,A} $ contains
$EPS^*_{i',A,{\rm dir}}$ for some exceptional path system $EPS_{i'}\in \mathcal{EPS}$.
Similarly, let $H_{B,{\rm dir}}$ be an $r$-regular orientation of $H_B$. Then
$H_{B,{\rm dir}}+ CA_{B,{\rm dir}}(r)+ PCA_{B,{\rm dir}}(r)+ \mathcal{EF}^*_{B,{\rm dir}}+ (\mathcal{EF}')^*_{B,{\rm dir}}$
has a decomposition into $s'$ edge-disjoint (directed) Hamilton cycles $C'_{1,B},\dots,C'_{s',B}$ such that each $C'_{i,B} $ contains $EPS^*_{i'',B,{\rm dir}}$
for%
   \COMMENT{Daniela: replaced $EPS^*_{i',B,{\rm dir}}$ by $EPS^*_{i'',B,{\rm dir}}$ to make it different from the statement for $A$}
some exceptional path system $EPS_{i''}\in \mathcal{EPS}$.
By relabeling the $C'_{i,A}$ and $C'_{i,B}$ if necessary, we may assume that $C'_{i,A} $ contains $EPS^*_{i,A,{\rm dir}}$ and
$C'_{i,B} $ contains $EPS^*_{i,B,{\rm dir}}$.
Let $C_{i,A}$ and $C_{i,B}$ be the undirected cycles obtained from $C'_{i,A}$ and $C'_{i,B}$ by ignoring the directions of all
the edges. So $C_{i,A}$ contains $EPS^*_{i,A}$ and $C_{i,B}$ contains $EPS^*_{i,B}$.
Let $H_i:=C_{i,A} + C_{i,B} - EPS^*_{i} + EPS_i$. Then Proposition~\ref{prop:EPS} (applied with $G'$ playing the role of $G$)
implies that $H_1,\dots,H_{s'}$ is a decomposition of $H+ G^{\rm rob}=H+ CA(r)+ PCA(r)+ \mathcal{EF}+ \mathcal{EF}'$
into edge-disjoint 2-factors satisfying Corollary~\ref{rdeccor}(ii)(b${_1}$) and~(b${_2}$).
\endproof 


\section{Proof of Theorem~$\text{\ref{1factstrong}}$} \label{sec:1factstrong}

Before we can prove Theorem~\ref{1factstrong}, we need the following two observations.
Recall that a $(K, m, \epszero, \eps)$-scheme was defined in Section~\ref{sec:schemes} and that a 
$[K,L,m,\eps_0,\eps']$-scheme was defined in Section~\ref{sec:findBF}.
\begin{prop}\label{lem:dirscheme}
Suppose that $0 <1/m  \ll \eps, \eps_0\ll \eps'\ll 1/K,1/L\ll 1$ and that $K, L, m/L \in \mathbb{N}$.
Suppose that $(G,\mathcal{P}')$ is a $(KL, m/L, \epszero, \eps)$-scheme. Suppose that $\cP$ is a $(K,m,\eps_0)$-partition
such that $\mathcal{P}'$ is an $L$-refinement of $\cP$.
Then there exists an orientation $G_{\rm dir}$ of $G$ such that $(G_{\rm dir}, \mathcal{P},\mathcal{P}')$ is a
$[K,L,m,\eps_0,\eps']$-scheme.
\end{prop}
\proof
Randomly orient every edge in $G$ to obtain an oriented graph $G_{\rm dir}$. (So given any edge $xy$ in $G$
with probability $1/2$, $xy \in E(G_{\rm dir})$ and with probability $1/2$, $yx \in E(G_{\rm dir})$.)
(Sch$1'$) and (Sch$2'$) follow immediately from (Sch$1$) and (Sch$2$).

Note that 
(Sch$3$) imply that $G[A_{i,j},B_{i',j'}]$ is $[1, \sqrt{\eps}]$-superregular
with density at least $1-\eps$,
for all $i,i'\leq K$ and $j,j' \leq L$.
Using this, (Sch$3'$) follows easily from the large deviation bound in Proposition~\ref{chernoff}.
(Sch$4'$) follows from Proposition~\ref{chernoff} in a similar way.
\endproof

\begin{prop} \label{Dupper}
Suppose that $G$ is a $D$-regular graph on $n$ vertices which is $\eps$-close to the union of two disjoint copies of $K_{n/2}$.
Then $D \le (1/2 + 4 \eps ) n$.
\end{prop}
\proof 
Let $B \subseteq V(G)$ with $|B| = \lfloor n/2 \rfloor$ be such that $e(B, V(G) \setminus B) \le \eps n^2$.
Note that $B$ exists since $G$ is $\eps$-close to the union of two disjoint copies of $K_{n/2}$.
Let $A = V(G) \setminus B$.
If $D >  (1/2 + 4 \eps ) n$, then Proposition~\ref{prp:e(A',B')}(i) implies that $e(A,B) > \eps n^2$, a contradiction.
\endproof

We can now put everything together and prove Theorem~\ref{1factstrong} in the following steps. We choose the (localized) exceptional systems needed as an `input'
for Corollary~\ref{rdeccor} to construct the robustly decomposable graph~$G^{\rm rob}$ in Step 3.
For this, we first  choose appropriate constants and a suitable vertex partition in Steps~1 and~2 respectively
(in Step~1, we also find some Hamilton cycles covering `bad' edges).
In Step~4, we then apply Corollary~\ref{rdeccor} to find~$G^{\rm rob}$.
Similarly, we then choose the (localized) exceptional systems needed as an `input'
for the `approximate decomposition lemma' (Lemma~\ref{almostthm}) in Step~6
(in this step, we also find some Hamilton cycles which extend those exceptional systems which are not localized).
For Step~6, we first choose a suitable vertex partition in Step~5.
In Step~7, we find an approximate decomposition using Lemma~\ref{almostthm}
and in Step~8, we decompose the union of  the `leftover' and $G^{\rm rob}$ via Corollary~\ref{rdeccor}.

\removelastskip\penalty55\medskip\noindent{\bf Proof of Theorem~\ref{1factstrong}. }

\noindent\textbf{Step 1: Choosing the constants and a framework.}
Choose $n_0\in\mathbb{N}$ to be sufficiently large compared to $1/\eps_{\rm ex}$.%
    \COMMENT{Daniela: new sentence}
Let $G$ and $D$ be as in Theorem~\ref{1factstrong}. 
By Proposition~\ref{Dupper}
\begin{align}
	n/2-1 \le D \le (1/2 + 4 \eps_{\rm ex} )n. \label{eq:Dupper}
\end{align}
Define new constants such that 
\begin{align*}
0 & < 1/n_0 \ll \eps_{\rm ex} \ll  \epszero \ll \phi_0  \ll \eps_* \ll \eps'_*   \ll  \eps_1' \ll \lambda_{K_2} \ll 1/K_2  \ll \gamma \ll  1/K_1  \\
 & \ll  \eps''_* \ll 1/L \ll 1/f \ll \gamma_1  \ll 1/g \ll \eps_2', \lambda_{K_1L} \ll \eps \ll 1,
\end{align*}
where $K_1, K_2, L, f, g \in  \mathbb{N}$ and $K_2$ is odd. 
Note that we can choose the constants such that
\begin{equation}\label{eq:div}
 \frac{D- \phi_0 n}{400 (K_1 L K_2)^2}, \phi_0 n  ,\frac{\lambda_{K_1L}n}{(K_1L)^2}, \frac{\lambda_{K_2}n}{K_2^2},
\frac{K_1}{14fg},  \frac{2fK_1}{3g(g-1)}  \in \mathbb{N}.
\end{equation}

Apply Proposition~\ref{prop:framework} to obtain a partition $A,A_0,B,B_0$ of $V(G)$ such that $(G,A,A_0,B,B_0)$ is an
$(\eps_0, 4g K_1 L K_2)$-framework with $ \Delta( G [A' , B' ]) \le D/2$ (where $A':=A\cup A_0$ and $B':=B\cup B_0$).
Let $w_1$ and $w_2$ be two vertices of $G$ such that $d_{G[A',B']}(w_1) \ge d_{G[A',B']}(w_2) \ge d_{G[A',B']}(v)$ for all $v \in V(G) \setminus \{w_1,w_2\}$.
Note that the partition $A,A_0,B,B_0$ of $V(G)$ and the two vertices $w_1$ and $w_2$ are fixed throughout the proof. 
Moreover, in the remainder of the proof, given a graph $H$ on $V(G)$, we will always write $H^{\diamond}$ for $H - H[A] - H[B]$.
  
Next we apply Lemma~\ref{V_0elimination} with $\phi_0$ and $4g K_1LK_2$ playing the roles of $\phi$ and $K$ to find a
spanning subgraph $\mathcal{H}'_1$ of $G$.
Let $G_1:=G-\mathcal{H}'_1$. Thus the following properties are satisfied:
\begin{itemize}
	\item[{\rm ($\alpha_1$)}] $G[A_0]+G[B_0] \subseteq \mathcal{H}'_1$ and $\mathcal{H}'_1$ is a $\phi _0 n $-regular spanning graph of $G$.
	\item[{\rm ($\alpha_2$)}] $e_{\mathcal{H}'_1}(A',B') \le \phi_0 n$ and $e_{G_1}(A',B')$ is even.
	\item[{\rm ($\alpha_3$)}] The edges of $\mathcal{H}'_1$ can be decomposed into $\lfloor e_{\mathcal{H}'_1}(A',B')/2 \rfloor$ Hamilton cycles and
$\phi_0 n - 2\lfloor e_{\mathcal{H}'_1}(A',B')/2 \rfloor$ perfect matchings. 
Moreover, if $e_G(A',B') \ge D$, then this decomposition consists of $\lfloor \phi_0 n/2 \rfloor$
Hamilton cycles and one perfect matching if $D$ is odd.%
    \COMMENT{Daniela: reworded}
	\item[{\rm ($\alpha_4$)}] $ d_{G_1[A',B']}(w_1) \le (D-\phi_0 n)/2$.
	Furthermore, if $D = n/2-1$ then \\ $ d_{ G_1[A',B'] }(w_2) \le (D-\phi_0 n)/2$.
	\item[{\rm ($\alpha_5$)}] If $e_G(A',B') < D$, then $ \Delta( G_1[A',B'] ) \le e(G_1[A',B'])/2 \le (D-\phi_0 n)/2$.
\end{itemize}
Let $\mathcal{H}_1$ be the collection of Hamilton cycles and perfect matchings guaranteed by ($\alpha_3$). (So $\mathcal{H}'_1=\bigcup \mathcal{H}_1$.)
Note that 
\begin{equation} \label{eqD1}
D_1 : = D - \phi_0 n
\end{equation}
is even (since (\ref{eq:div}) implies that $D$ and $\phi_0 n$ have the same parity)%
    \COMMENT{Daniela: added brackets}
and that $G_1$ is $D_1$-regular. 
Moreover, $(G_1,A,A_0,B,B_0)$ is an $(\eps_0, 4g K_1 L K_2)$-framework with $ \Delta( G_1 [A' , B' ]) \le D/2$.
Let     
\begin{align}
m_1 & :=\frac{|A|}{K_1}=\frac{|B|}{K_1}, &
r  & :=  \gamma m_1, \qquad
r_1 := \gamma_1 m_1, \qquad
r_2:= 96g^3K_1r, \nonumber\\
r_3 & := \frac{rK_1}{L}, &
r^\diamond & := r_1 +r_2 +r -(Lf -1)r_3, \nonumber\\
m_2 & := \frac{|A|}{K_2}=\frac{|B|}{K_2}, &
D_4 & :=D_1-2(Lfr_3+7r^\diamond). \label{D4eq}
\end{align}
Note that (FR3) implies $m_1/L\in \mathbb{N}$. Moreover, 
\begin{equation}\label{eq:rs}
r_2,r_3\le \gamma^{1/2}m_1\le \gamma^{1/3} r_1,  \qquad  r_1/2\le r^\diamond\le 2r_1.
\end{equation}
Furthermore, by changing $\gamma,\gamma_1$ slightly, we may assume that $r/400LK_2^2, r_1/400K^2_2 \break \in\mathbb{N}$.
This implies that $r_2/400K_2^2, r_3/400K_2^2, r^{\diamond}/400K_2^2 \in \mathbb{N}$.
Together with the fact that $D_1/400K^2_2 = (D - \phi_0 n ) /400K_2^2 \in\mathbb{N}$ by (\ref{eq:div}),%
   \COMMENT{Daniela: added back reference}
this in turn implies that
\begin{equation}\label{divD4}
D_4/400K^2_2\in\mathbb{N}.
\end{equation}

\smallskip

\noindent\textbf{Step 2: Choosing a $(K_1, L, m_1, \epszero)$-partition $(\mathcal{P}_1,\mathcal{P}'_1)$.}
We now prepare the ground for the construction of the robustly decomposable graph $G^{\rm rob}$, 
which we will obtain via the robust decomposition lemma (Corollary~\ref{rdeccor}) in Step~4.

Since $(G_1,A,A_0,B,B_0)$ is an $(\eps_0,  4g K_1 L K_2)$-framework, it is also an $(\eps_0, K_1 L)$-framework.
Recall that $G_1$ is $D_1$-regular and $D_1 = D - \phi_0 n  \ge (1- 3 \phi_0 )n/2$ (as $D \ge n/2 - 1$).
Apply Lemma~\ref{lma:partition} with $G_1$, $m_1/L$, $ 3\phi_0 $, $K_1L$, $\eps_*$, $\eps_*$ playing the roles of
$G$, $m$, $\mu$, $K$, $\eps_1$, $\eps_2$ to obtain partitions $A'_1,\dots,A'_{K_1L}$ of $A$ and $B'_1,\dots,B'_{K_1L}$ of $B$
into sets of size $m_1/L$ such that the following properties are satisfied:
\begin{itemize}
\item[(S$_1$a)] Together with $A_0$ and $B_0$ all
these sets $A'_i$ and $B'_i$ form a $(K_1L,m_1/L,\eps_0)$-partition $\cP'_1$ of $V(G_1)$.
\item[(S$_1$b)] $(G_1[A] +G_1[B], \cP'_1 )$ is a $(K_1L,m_1/L,\eps_0, \eps_*)$-scheme.
\item[(S$_1$c)] $(G^{\diamond}_1, \cP'_1)$ is a $(K_1L,m_1/L,\eps_0, \eps_*)$-exceptional scheme (where $G^{\diamond}_1 : = G_1 - G_1[A] - G_2[B]$).
\end{itemize}
Note that $(1-\eps_0)n\le n-|A_0\cup B_0|=2K_1m_1\le n$ by (FR3).
For all $i\le K_1$ and all $h\le L$, let $A_{i,h}:=A'_{(i-1)L+h}$. (So this is just a relabeling of the sets $A'_i$.)
Define $B_{i,h}$ similarly and let $A_i:= \bigcup_{h\le L} A_{i,h}$ and $B_i:= \bigcup_{h\le L} B_{i,h}$.
Let $\cP_1:=\{A_0,B_0,A_1,\dots,A_{K_1},B_1,\dots,B_{K_1}\}$ denote the
corresponding $(K_1,m_1,\eps_0)$-partition of $V(G)$. Thus $(\cP_1,\cP'_1)$ is a $(K_1,L,m_1,\eps_0)$-partition of $V(G)$,
as defined in Section~\ref{sec:BEPS}.

\smallskip

\noindent\textbf{Step 3: Exceptional systems for the robustly decomposable graph.}
In order to be able to apply Corollary~\ref{rdeccor} to obtain the robustly decomposable graph $G^{\rm rob}$, we first need to construct
suitable exceptional systems with parameter $\eps_0$. The construction of these exceptional systems depends on whether $G$ is critical and whether $e_G(A',B') \ge D$.
First we show that in each case, for all $1 \le i'_1, i'_2  \le K_1L$, we can always find sets $\cJ_{i'_1,i'_2}$ of $\lambda_{K_1L} n/(K_1 L)^2$ $(i'_1,i'_2)$-ES
with respect to $\mathcal{P}_1'$.

\noindent\textbf{Case 1: $e_G(A',B') \ge D$ and $G$ is not critical.}
Our aim is to apply Lemma~\ref{lma:BESdecom} to $G$ with $\mathcal{H}'_1$, $m_1/L$, $K_1L$, $\cP'_1$, $\eps_*$, $\phi_0$, $\lambda_{K_1L}$
playing the roles of $G_0$, $m$, $K$, $\cP$, $\eps$, $\phi$, $\lambda$.
First we verify that Lemma~\ref{lma:BESdecom}(i)--(iv) are satisfied. 
Lemma~\ref{lma:BESdecom}(i) holds trivially.
(FR2) implies that $e_G(A',B') \le \eps_0 n^2$.
Moreover, recall from (S$_1$a) that $\cP'_1$ is a $(K_1L,m_1/L,\eps_0)$-partition of $V(G)$ and that $A'$ and $B'$ were chosen
(by Proposition~\ref{prop:framework}) such that $\Delta(G[A',B']) \le D/2$. Altogether this shows that Lemma~\ref{lma:BESdecom}(ii) holds.
Lemma~\ref{lma:BESdecom}(iii) follows from ($\alpha_1$) and ($\alpha_2$). To verify Lemma~\ref{lma:BESdecom}(iv),
note that $G^{\diamond}_1$ plays the role of $G^{\diamond}$ in Lemma~\ref{lma:BESdecom} and $G^{\diamond}_1[A',B'] = G_1[A',B']$.
So $e_{G^{\diamond}_1}(A',B')$ is even by ($\alpha_2$).
Together with the fact that $(G^{\diamond}_1, \cP'_1)$ is a $(K_1L, m_1/L, \eps_0, \eps_*)$-exceptional scheme by (S$_1$c),%
    \COMMENT{Daniela: added ref to (S$_1$c)}
this implies Lemma~\ref{lma:BESdecom}(iv).

By Lemma~\ref{lma:BESdecom}, we obtain a set $\cJ$ of $\lambda_{K_1L} n$ edge-disjoint Hamilton exceptional systems $J$ in
$G_1^\diamond$ such that $e_J(A',B') =2 $ for each $J \in \cJ$ and such that
for all $1 \le i'_1, i'_2  \le K_1L$ the set $\cJ$ contains precisely $\lambda_{K_1L} n/(K_1 L)^2$ $(i'_1,i'_2)$-HES with respect to the partition $\cP'_1$.
For all $1 \le i'_1, i'_2  \le K_1L$, let $\cJ_{i'_1,i'_2}$ be the set of these $\lambda_{K_1L} n/(K_1 L)^2$ $(i'_1,i'_2)$-HES in $\cJ$.
So $\cJ$ is the union of all the sets $\cJ_{i'_1,i'_2}$.
(Note that the set $\cJ$ here is a subset of the set $\cJ$ in Lemma~\ref{lma:BESdecom}, i.e.~we do not use all the Hamilton
exceptional systems constructed by Lemma~\ref{lma:BESdecom}. So we do not need the full strength of Lemma~\ref{lma:BESdecom} at this point.)

\noindent\textbf{Case 2: $e_G(A',B') \ge D$ and $G$ is critical.}
Recall from Lemma~\ref{critical}(ii) that in this case we have $D = (n-1)/2$ or $D = n/2 -1$.
Our aim is to apply Lemma~\ref{lma:BESdecomcritical} to $G$ with $\mathcal{H}'_1$, $m_1/L$, $K_1L$, $\cP'_1$, $\eps_*$, $\phi_0$, $ \lambda_{K_1L}$
playing the roles of $G_0$, $m$, $K$, $\cP$, $\eps$, $\phi$, $\lambda$. Similar arguments as in Case~1 show that
Lemma~\ref{lma:BESdecomcritical}(i)--(iv) hold.
Recall that $w_1$ and $w_2$ are (fixed) vertices in $V(G)$ such that $d_{G[A',B']}(w_1) \ge d_{G[A',B']}(w_2) \ge d_{G[A',B']}(v)$
for all $v \in V(G) \setminus \{w_1,w_2\}$.
Since $G^{\diamond}_1[A',B'] = G_1[A',B']$, ($\alpha_4$) implies that $d_{G^{\diamond}_1[A',B']}(w_1) \le (D- \phi_0 n ) /2$.
Moreover, if $D = n/2 -1$, then $d_{G^{\diamond}_1[A',B']}(w_2) \le (D- \phi_0 n ) /2$.
Let $W$ be the set of vertices $w \in V(G)$ such that $d_{G[A',B']}(w) \ge 11D / 40$, as defined in Lemma~\ref{critical}.
If $D = (n-1)/2$, then $|W| = 1$ by Lemma~\ref{critical}(ii).
This means that $w_2 \notin W$ and so $d_{G^{\diamond}_1[A',B']}(w_2) \le d_{G[A',B']}(w_2) \le 11 D /40$.
Thus in both cases we have that
\begin{equation}\label{eq:degw1w2}
d_{G^{\diamond}_1[A',B']}(w_1), d_{G^{\diamond}_1[A',B']}(w_2) \le (D- \phi_0 n ) /2.
\end{equation}
Therefore, Lemma~\ref{lma:BESdecomcritical}(v) holds.

By Lemma~\ref{lma:BESdecomcritical}, we obtain a set $\cJ$ of $ \lambda_{K_1L} n$ edge-disjoint Hamilton exceptional systems
$J$ in $G_1^\diamond$ such that, for all $1 \le i'_1, i'_2  \le K_1L$, the set $\cJ$ contains precisely $\lambda_{K_1L} n/(K_1 L)^2$ $(i'_1,i'_2)$-HES
with respect to the partition $\cP'_1$. Moreover, each $J \in \cJ$ satisfies $e_J(A',B') =2 $ and $d_{J[A',B']}(w) =1 $ for all $w \in \{w_1,w_2\}$
with $d_{G[A',B']}(w) \ge 11D/40$. For all $1 \le i'_1, i'_2  \le K_1L$, let $\cJ_{i'_1 , i'_2}$ be the set of these
$\lambda_{K_1L} n/(K_1 L)^2$ $( i'_1, i'_2)$-HES. So $\cJ$ is the union of all the sets $\cJ_{ i'_1 , i'_2}$.
(So similarly as in Case~1, we do not use all the Hamilton exceptional systems constructed by Lemma~\ref{lma:BESdecomcritical} at this point.)

\noindent\textbf{Case 3: $e_G(A',B') < D$.}
Recall from Proposition~\ref{prp:e(A',B')}(ii) that in this case we have $D = n/2 - 1$, $n = 0 \pmod 4$ and $|A'| = |B'| = n/2$.
Our aim is to apply Lemma~\ref{lma:PBESdecom} to $G$ with $\mathcal{H}'_1$, $m_1/L$, $K_1L$, $\cP'_1$, $\eps_*$, $\phi_0$, $\lambda_{K_1L}$
playing the roles of $G_0$, $m$, $K$, $\cP$, $\eps$, $\phi$, $\lambda$. Similar arguments as in Case~1 show that
Lemma~\ref{lma:PBESdecom}(i)--(iv) hold.
Since $G^{\diamond}_1[A',B'] = G_1[A',B']$ and $D = n/2 -1$, Lemma~\ref{lma:PBESdecom}(v) follows from ($\alpha_5$).

By Lemma~\ref{lma:PBESdecom}, $G^{\diamond}_1$ can be decomposed into a set $\cJ'$ of $D_1/2$ edge-disjoint exceptional systems
such that each of these exceptional systems $J$ is either a Hamilton exceptional system with $e_{J}(A',B') = 2$ or a matching exceptional system.
(So $\cJ'$ plays the role of the set $\cJ$ in Lemma~\ref{lma:PBESdecom}.)
Lemma~\ref{lma:PBESdecom}(b) guarantees that we can choose a subset $\cJ$ of $\cJ'$ such that $\cJ$ consists of $\lambda_{K_1L} n$
edge-disjoint exceptional systems $J$ in $G^\diamond_1$ such that for all $1 \le i'_1, i'_2  \le K_1L$ the set $\cJ$ contains precisely
$\lambda_{K_1L} n/(K_1 L)^2$ $(i'_1,i'_2)$-ES with respect to the partition $\cP'_1$.
For all $1 \le i'_1, i'_2  \le K_1L$, let $\cJ_{i'_1,i'_2}$ be the set of these $\lambda_{K_1L} n/(K_1 L)^2$ $(i'_1,i'_2)$-ES.%
     \COMMENT{Note that when $e_G(A',B')< D$, we do not require $\cJ_{i'_1,i'_2}$ contains as many HES (or MES) as possible.}
So $\cJ$ is the union of all the sets $\cJ_{i'_1,i'_2}$. (Note that to construct the robustly decomposable graph we will only use the exceptional systems in $\cJ$.
However, in order to prove condition ($\beta_5$) below, we will also use the fact that $G^{\diamond}_1$ has a decomposition into edge-disjoint exceptional systems.)

\smallskip

Thus in each of the three cases, $\cJ$ is the union of all the sets $\cJ_{i'_1,i'_2}$, where for all $1 \le i'_1, i'_2  \le K_1L$,
the set $\cJ$ consists of precisely $\lambda_{K_1L} n/(K_1 L)^2$ $(i'_1,i'_2)$-ES with respect to the partition $\cP'_1$.
Moreover, all the $\lambda_{K_1L} n$ exceptional systems in $\cJ$ are edge-disjoint.

Our next aim is to choose two disjoint subsets $\mathcal{J}_{\rm CA}$ and $\mathcal{J}_{\rm PCA}$ of $\cJ$ with the
following properties:
\begin{itemize}	
\item[(a)] In total $\mathcal{J}_{\rm CA}$ contains $L f r_3$ exceptional systems. For each $i\le f$ and each $h \le L$, $\mathcal{J}_{\rm CA}$ contains precisely $r_3$ $(i_1,i_2)$-ES of style $h$ (with respect to the $(K_1,L,m_1,\eps_0)$-partition $(\cP_1,\cP'_1)$) such that $i_1,i_2 \in \{(i-1)K_1/f+2,\dots,iK_1/f\}$.
\item[(b)] In total $\mathcal{J}_{\rm PCA}$ contains $7r^\diamond$ exceptional systems. For each $i\le 7$,
$\mathcal{J}_{\rm PCA}$ contains precisely $r^\diamond$ $(i_1,i_2)$-ES (with respect to
the partition $\cP_1$) with $i_1,i_2 \in \{(i-1)K_1/7+2,\dots,iK_1/7\}$.
\item[(c)] Each exceptional system $J\in \mathcal{J}_{\rm CA}\cup \mathcal{J}_{\rm PCA}$ is either a Hamilton exceptional system with
$e_J(A',B')=2$ or a matching exceptional system.
\end{itemize}
(Recall that we defined in Section~\ref{sec:findBF} when an $(i_1,i_2)$-ES has style $h$ with respect to a $(K_1,L,m_1,\eps_0)$-partition $(\cP_1,\cP'_1)$.)
To see that it is possible to choose $\mathcal{J}_{\rm CA}$ and $\mathcal{J}_{\rm PCA}$, split $\cJ$ into two sets $\cJ_1$ and $\cJ_2$ such that both $\cJ_1$ and $\cJ_2$
contain at least $\lambda_{K_1L} n /3(K_1L)^2$ $(i'_1,i'_2)$-ES with respect to~$\cP'_1$, for all $1 \le i'_1 , i'_2   \le K_1L$.
Note that, for each $i \le f$, there are $(K_1/f-1)^2$ choices of pairs $(i_1,i_2)$ with $i_1,i_2 \in \{(i-1)K_1/f+2,\dots,iK_1/f\}$.
Moreover, for each such pair $(i_1,i_2)$ and each $h\le L$ there is precisely one pair $(i'_1,i'_2)$ with $1 \le i'_1,i'_2  \le K_1L$
and such that any $(i'_1,i'_2)$-ES with respect to~$\cP'_1$ is an $(i_1,i_2)$-ES of style $h$ with respect to $(\cP_1,\cP'_1)$.
Together with the fact that $\gamma\ll \lambda_{K_1L}, 1/L, 1/f$ and
\begin{align*}
\frac{(K_1/f-1)^2 \lambda_{K_1L} n}{3 (K_1 L)^2} \ge
 \frac{\gamma n}{L} \ge \frac{\gamma K_1 m_1}{L} = \frac{r K_1}{L } = r_3,
\end{align*}
this implies that we can choose a set $\mathcal{J}_{\rm CA}\subseteq \cJ_1$ satisfying~(a).

Similarly,  for each $i \le 7$, there are $(K_1/7-1)^2$ choices of pairs $(i_1,i_2)$ with $i_1,i_2\in \{(i-1)K_1/7+2,\dots,iK_1/7\}$.
Moreover, for each such pair $(i_1,i_2)$ there are $L^2$ distinct pairs $(i'_1,i'_2)$ with $1 \le i'_1,i'_2  \le K_1L$ and
such that any $(i'_1,i'_2)$-ES with respect to~$\cP'_1$ is an $(i_1,i_2)$-ES with respect to $\cP_1$.
Together with the fact that $\gamma_1 \ll \lambda_{K_1L}$ and
$$ \frac{(K_1/7-1)^2 L^2 \lambda_{K_1L} n }{3 (K_1 L)^2 }
\ge \gamma_1 n 
\ge 2 \gamma_1 m_1 = 2 r_1
\stackrel{\eqref{eq:rs}}{\ge} r^\diamond,$$
this implies that we can choose a set $\mathcal{J}_{\rm PCA}\subseteq \cJ_2$ satisfying~(b).
Our choice of $\cJ\supseteq \mathcal{J}_{\rm CA}\cup \mathcal{J}_{\rm PCA}$ guarantees that~(c) holds too.
Let 
\begin{equation} \label{robphieq}
\mathcal{J}^{\rm rob} := \mathcal{J}_{\rm CA} \cup \mathcal{J}_{\rm PCA}, \ \ 
\phi^{\rm rob}_0: = (Lfr_3 + 7 r^{\diamond})/n \ \ \mbox{and} \ \  G^{\diamond}_4 : = G_1^{\diamond} - \bigcup \mathcal{J}^{\rm rob}.
\end{equation}
(In Step~5 below we will define a graph $G_4$ which will satisfy%
    \COMMENT{Daniela: replaced $(G_4)^{\diamond}$ by $G_4^\diamond$}
$G_4^{\diamond} = G_4 - G_4[A] - G_4[B]$. So this
will fit with our definition of the operator $^\diamond$.)
Note that
\begin{equation}\label{eq:phi0rob}
\phi^{\rm rob}_0\ge \frac{7r^\diamond}{n}\stackrel{\eqref{eq:rs}}{\ge} \frac{3r_1}{n}=\frac{3\gamma_1m_1}{n}\ge \frac{\gamma_1}{K_1}\ge 2\phi_0
\ \ \ \text{and} \ \ \ 2 \phi^{\rm rob}_0 n \stackrel{\eqref{D4eq}}{=} D_1 - D_4.
\end{equation}
Moreover, we claim that $\bigcup \mathcal{J}^{\rm rob}$ is a subgraph of $G^\diamond_1\subseteq G$
satisfying the following properties:%
    \COMMENT{Daniela: replaced $D - (\phi_0  + 2 \phi^{\rm rob}_0n)/2$ by $(D - (\phi_0  + 2 \phi^{\rm rob}_0)n)/2$ in ($\beta_5$)}
\begin{itemize}
	\item[($\beta_1$)] $d_{\bigcup\mathcal{J}^{\rm rob}} (v) = 2(Lfr_3 + 7 r^{\diamond}) = 2 \phi^{\rm rob}_0 n $ for each $v \in V_0$. 
	\item[($\beta_2$)] $e_{\bigcup\mathcal{J}^{\rm rob}}(A',B')\le 2 \phi^{\rm rob}_0 n$ is even.
	\item[($\beta_3$)] $\mathcal{J}^{\rm rob}$ contains exactly $\phi^{\rm rob}_0 n$ exceptional systems, of which
precisely \\ $e_{\bigcup \mathcal{J}^{\rm rob}}(A',B')/2$ are Hamilton exceptional systems.
If $e_G(A',B') \ge D$, then $\mathcal{J}^{\rm rob}$ consists entirely of Hamilton exceptional systems.
 If $\mathcal{J}^{\rm rob}$ contains a matching exceptional system, then $|A'|=|B'|=n/2$ is even.
	\item[($\beta_4$)] If $e_G(A',B') \ge D$ and $G$ is critical, then $d_{\bigcup\mathcal{J}^{\rm rob}[A',B'] }(w) =\phi^{\rm rob}_0 n $ for
all $w \in \{w_1,w_2\}$ with $d_{G[A',B']}(w) \ge 11D/40$.
	Moreover, $d_{G^{\diamond}_4[A',B']}(w_1),$ $ d_{G^{\diamond}_4[A',B']}(w_2) \le  (D - (\phi_0  + 2 \phi^{\rm rob}_0)n)/2$.	
	\item[($\beta_5$)] If $e_G(A',B') < D$, then $\Delta( G^{\diamond}_4 [A',B']) \le e(G^{\diamond}_4  [A',B'])/2  \le  D_4/2 = (D - (\phi_0  + 2 \phi^{\rm rob}_0)n)/2$.
\end{itemize}
To verify the above, note that $ \mathcal{J}^{\rm rob}$ consists of precisely $\phi^{\rm rob}_0 n$ exceptional systems $J$ (each of which is an exceptional cover).
So ($\beta_1$) follows from~(EC2). Moreover, each such $J$ is either a Hamilton exceptional system with $e_{J}(A',B') = 2$ or a matching exceptional system
(with $e_{J}(A',B') = 0$ by (MES)), which implies ($\beta_2$) and the first part of ($\beta_3$).
If $e_G(A',B') \ge D$, then we are in Case~$1$ or~$2$ and so the second part of ($\beta_3$) follows from our construction of $\cJ\supseteq \mathcal{J}^{\rm rob}$.
The first part of ($\beta_4$) follows from our construction
of $\cJ\supseteq \mathcal{J}^{\rm rob}$ in Case~2. Since $11D/40 < (D - (\phi_0  + 2 \phi^{\rm rob}_0)n)/2$,%
\COMMENT{Daniela: reworded, previously we said that we need $11D/40 < (D - (\phi_0  + 2 \phi^{\rm rob}_0)n)/2$ to prove
the first part of ($\beta_4$)}
we can combine the first part of ($\beta_4$) with \eqref{eq:degw1w2} to obtain the `moreover part' of ($\beta_4$).
Thus it remains to verify ($\beta_5$). So suppose that $e_G(A',B') < D$.
Recall from Case~3 that $G^{\diamond}_1$ has a decomposition into a set $\cJ'$ of $D_1/2$ edge-disjoint exceptional systems $J$,
each of which is either a Hamilton exceptional system with $e_{J}(A',B') = 2$ or a matching exceptional system.
This means that $J[A',B']$ is either empty or a matching of size~$2$.
Note that $G^{\diamond}_4 [A',B']$ is precisely the union of $J[A',B']$ over all those $D_1/2 - \phi^{\rm rob}_0 n=D_4/2$%
   \COMMENT{Daniela: added $=D_4/2$}
exceptional systems $J\in \cJ'\setminus \mathcal{J}^{\rm rob}$. So ($\beta_5$) holds.

\smallskip

\noindent\textbf{Step 4: Finding the robustly decomposable graph.}
Let $G_2 : = G_1[A] + G_1[B]$.
Recall from (S$_1$b) that $(G_2,\cP'_1)$ is a $(K_1L, m_1/L, \epszero,\eps_*)$-scheme.
Apply Proposition~\ref{lem:dirscheme} with $G_2$, $\cP_1$, $\cP'_1$, $K_1$, $m_1$, $\eps_*$, $\eps'_*$ playing the roles of
$G$, $\cP$, $\cP'$, $K$, $m$, $\eps$, $\eps'$ to obtain an orientation $G_{2,{\rm dir}}$ of $G_2$ such that
$(G_{2,{\rm dir}}, \mathcal{P}_1,\cP'_1)$ is a $[K_1, L, m_1, \epszero,\eps'_*]$-scheme. 

Our next aim is to use Lemma~\ref{lma:EF} in order to extend the exceptional systems in $\mathcal{J}_{\rm CA}$ into $r_3$
edge-disjoint exceptional factors with parameters $(L,f)$ for $G_{2,{\rm dir}}$ (with respect to $(\cP_1,\cP'_1)$).
For this, note that (a) and (c) guarantee that $\cJ_{CA}$ satisfies Lemma~\ref{lma:EF}(i),(ii) with $r_3$ playing the role of $q$. 
Moreover, $Lr_3/m_1= rK_1/m_1=\gamma K_1\ll 1$. 
Thus we can indeed apply Lemma~\ref{lma:EF} to $(G_{2,{\rm dir}}, \mathcal{P}_1,\cP'_1)$
with $\cJ_{CA}$, $m_1$, $\eps'_*$, $K_1$, $r_3$ playing the roles of $\cJ$, $m$, $\eps$, $K$, $q$ in order to obtain $r_3$
edge-disjoint exceptional factors $EF_1,\dots,EF_{r_3}$ with parameters $(L,f)$
for $G_{2,{\rm dir}}$ (with respect to $(\cP_1,\cP'_1)$) such that together these exceptional factors
cover all edges in $\bigcup \mathcal{J}_{\rm CA}$. Let $\mathcal{EF}_{\rm CA}:=EF_1 + \dots + EF_{r_3}$.
Since $G_2 = G_1[A] + G_1[B]$, we have $(\mathcal{EF}_{\rm CA})^{\diamond} = \mathcal{J}_{\rm CA}$.
Moreover, each exceptional path system in $\mathcal{EF}_{\rm CA}$ contains a unique exceptional system in $\mathcal{J}_{\rm CA}$ (in particular, their numbers are equal).

Note that $m_1/4g,m_1/L\in\mathbb{N}$ since $m_1=|A|/K_1$ and $|A|$ is divisible by $4gK_1L$ as $(G,A,A_0,B,B_0)$
is an $(\eps_0, 4gK_1LK_2)$-framework. Furthermore, $rK_1^2=\gamma m_1K_1^2\le \gamma^{1/2}m_1\le m_1$.
Thus we can apply Corollary~\ref{rdeccor} to the $[K_1, L, m_1, \epszero, \eps_*'']$-scheme
$(G_{2,{\rm dir}}, \mathcal{P}_1,\cP'_1)$ with $K_1$, $m_1$, $\eps_*''$, $g$ playing the roles of $K$, $m$, $\eps$, $\ell'$ to obtain a spanning
subgraph $CA(r)$ of $G_2$ as described there. (Note that $G_2$ equals the graph $G'$ defined in Corollary~\ref{rdeccor}.)
In particular, $CA(r)$ is $2(r_1+r_2)$-regular and edge-disjoint from $\mathcal{EF}_{\rm CA}$.

Let $G_3$ be the graph obtained from $G_2$ by deleting all the edges of $CA(r)+ \mathcal{EF}_{\rm CA}$.
Thus $G_3$ is obtained from $G_2$ by deleting at most $2(r_1+r_2+r_3)\le 6r_1=6\gamma_1m_1$ edges at every vertex in $A\cup B$.
Let $G_{3,{\rm dir}}$ be the orientation of $G_3$ in which every edge is oriented in the same way as in $G_{2,{\rm dir}}$.
Since $(G_{2,{\rm dir}}, \mathcal{P}_1,\cP'_1)$ is a $[K_1, L, m_1, \epszero,\eps'_*]$-scheme, 
Proposition~\ref{superslice} and the fact that $\eps_*'',\gamma_1 \ll \eps$ imply that 
$(G_{3,{\rm dir}},\cP_1, \cP_1)$ is a $[K_1, 1, m_1, \epszero,\eps]$-scheme.%
    \COMMENT{Deryk inserted `1' here and replaced $(G_{3,{\rm dir}},\cP_1)$ by $(G_{3,{\rm dir}},\cP_1, \cP_1)$}
Moreover,
$$\frac{r^\diamond}{m_1}\stackrel{\eqref{eq:rs}}{\le} \frac{2r_1}{m_1}=2\gamma_1\ll 1.$$
Together with~(b) and~(c) this ensures that we can apply
Lemma~\ref{lma:EF} to $(G_{3,{\rm dir}},\cP_1,$ $ \cP_1)$ with $\cJ_{PCA}$, $m_1$, $K_1$, $1$, $7$, $r^\diamond$ playing the roles of
$\cJ$, $m$, $K$, $L$, $f$, $q$ in order to obtain $r^\diamond$ edge-disjoint exceptional factors $EF'_1,\dots,EF'_{r^\diamond}$ with parameters $(1,7)$
for $G_{3,{\rm dir}}$ (with respect to $(\cP_1,\cP_1)$)
such that together these exceptional factors
cover all edges in $\bigcup \mathcal{J}_{\rm PCA}$. Let $\mathcal{EF}_{\rm PCA}:=EF'_1+\dots + EF'_{r^\diamond}$.
Since $G_3 \subseteq G_1[A] + G_1[B]$ we have $(\mathcal{EF}_{\rm PCA})^{\diamond} = \bigcup \mathcal{J}_{\rm PCA}$.
Moreover, each exceptional path system in $\mathcal{EF}_{\rm PCA}$ 
contains a unique exceptional system in $\mathcal{J}_{\rm PCA}$.

Apply Corollary~\ref{rdeccor} to obtain a spanning
subgraph $PCA(r)$ of $G_2$ as described there. In particular, $PCA(r)$ is $10r^\diamond$-regular
and edge-disjoint from $CA(r)+ \mathcal{EF}_{\rm CA}+ \mathcal{EF}_{\rm PCA}$.

Let $G^{\rm rob}:=CA(r)+ PCA(r)+ \mathcal{EF}_{\rm CA}+ \mathcal{EF}_{\rm PCA}$.
Note that by~\eqref{EFdeg} all the vertices in $V_0:=A_0\cup B_0$ have the same degree  
$r_0^{\rm rob}:=2(Lfr_3+7r^\diamond) = 2 \phi^{\rm rob}_0 n $ in $G^{\rm rob}$.
So 
\begin{equation}\label{eq:rrob}
7r_1\stackrel{\eqref{eq:rs}}{\le} r_0^{\rm rob} \stackrel{\eqref{eq:rs}}{\le} 30r_1.
\end{equation}
Moreover,~\eqref{EFdeg} also implies that all the vertices in $A\cup B$ have the same degree $r^{\rm rob}$ in $G^{\rm rob}$,
where $r^{\rm rob} := 2(r_1+r_2)+10r^\diamond+2r_3+2r^\diamond=2(r_1+r_2+r_3+6r^\diamond)$. So 
$$
r_0^{\rm rob}-r^{\rm rob}=2 \left(Lfr_3+ r^\diamond- (r_1+r_2+r_3)\right)=2(Lfr_3+r-(Lf-1)r_3-r_3 )=2r.
$$
Note that $(G^{\rm rob})^{\diamond}=\bigcup(\mathcal{J}_{\rm CA} \cup \mathcal{J}_{\rm PCA})=\bigcup \cJ^{\rm rob}$.
Recall that the number of Hamilton exceptional path systems in $\mathcal{EF}_{\rm CA}$ equals the number of Hamilton exceptional
systems in $\mathcal{J}_{\rm CA}$, and that the analogue holds for $\mathcal{EF}_{\rm PCA}$.
Hence, ($\beta_1$), ($\beta_2$) and ($\beta_3$) imply the follow statements:%
   \COMMENT{Daniela: replaced $e_{\mathcal{J}^{\rm rob}}(A',B')$ by $e_{\bigcup\mathcal{J}^{\rm rob}}(A',B')$ in ($\beta_2'$)}
\begin{itemize}
    \item[($\beta_1'$)] $d_{G^{\rm rob}}(v) =r^{\rm rob}_0=2\phi_0^{\rm rob}n$ for all $v\in V_0$.
	\item[($\beta_2'$)] $e_{G^{\rm rob}}(A',B')=e_{\bigcup\mathcal{J}^{\rm rob}}(A',B')\le r^{\rm rob}_0 = 2\phi_0^{\rm rob}n$ is even.
	\item[($\beta_3'$)] $\mathcal{EF}_{\rm CA}+ \mathcal{EF}_{\rm PCA}$ contains exactly $\phi_0^{\rm rob}n$ 
exceptional path systems (and each such path system contains a unique exceptional system in 
$\mathcal{J}^{\rm rob}$, where $|\mathcal{J}^{\rm rob}|=\phi_0^{\rm rob}n$). Precisely $e_{\bigcup \mathcal{J}^{\rm rob}}(A',B')/2$ of these are
    Hamilton exceptional path systems. If $e_G(A',B') \ge D$, then every exceptional path system in $\mathcal{EF}_{\rm CA}+ \mathcal{EF}_{\rm PCA}$
    is a Hamilton exceptional path system.
    If $\mathcal{EF}_{\rm CA}+ \mathcal{EF}_{\rm PCA}$ contains a matching exceptional path system, then $|A'|=|B'|=n/2$ is even.
\end{itemize}

\smallskip

\noindent\textbf{Step 5: Choosing a $(K_2,m_2,\eps_0)$-partition $\cP_2$.}
We now prepare the ground for the approximate decomposition step (i.e.~to apply Lemma~\ref{almostthm}). 
For this, we need to work with a finer partition of $A \cup B$ than the previous one
(this will ensure that the leftover from the approximate decomposition step is sufficiently sparse compared to $G^{\rm rob}$).

So let $G_4:=G_1-G^{\rm rob}$ (where $G_1$ was defined in Step~1) and note that 
\begin{equation} \label{D4D1}
D_4 \stackrel{\eqref{D4eq}}{=} D_1-r_0^{\rm rob}=D_1-r^{\rm rob}-2r.
\end{equation}
So 
\begin{equation} \label{degrees4}
d_{G_4}(v)=D_4+2r \mbox{ for all } v \in A \cup B \qquad \mbox{and} \qquad d_{G_4}(v)=D_4 \mbox{ for all } v \in V_0.
\end{equation}
Hence
\begin{align*}
\delta(G_4) \ge D_4 \stackrel{(\ref{eq:phi0rob})}{=} D_1 - 2 \phi_0^{\rm rob}n \stackrel{(\ref{eqD1})}{=} D - (\phi_0 + 2 \phi^{\rm rob}_0 )n \ge (1 - 6 \phi^{\rm rob}_0) n/2
\end{align*}
as $\phi^{\rm rob}_0 \ge  2\phi_0$ by~\eqref{eq:phi0rob}.
Moreover, note that
$$2 \phi_0^{\rm rob} n = r^{\rm rob}_0 \stackrel{\eqref{eq:rrob}}{\le} 30 r_1 = 30 \gamma_1 m_1 \le 30 \gamma_1 n /K_1,$$ 
so $\phi_0^{\rm rob} \ll \eps'_2$.
Since $(G,A,A_0,B,B_0)$ is an $(\eps_0, 4g K_1 L K_2)$-framework, $(G_4,A,A_0, \break B,B_0)$ is an $(\eps_0, K_2)$-framework.
Now apply Lemma~\ref{lma:partition} to $(G_4,A,A_0,B,B_0)$ with%
    \COMMENT{Daniela: replaced $3\phi^{\rm rob}_0$ by $6 \phi^{\rm rob}_0$ and changed the RHS in the display after (\ref{degrees4}) to make
it clearer where the $6 \phi^{\rm rob}_0$ comes from}
$K_2$, $m_2$, $\eps'_1$, $\eps'_2$, $6 \phi^{\rm rob}_0$
playing the roles of $K$, $m$, $\eps_1$, $\eps_2$, $\mu$ in order to obtain partitions $A_1,\dots,A_{K_2}$ and
$B_1,\dots,B_{K_2}$ of $A$ and $B$ satisfying the following conditions:
\begin{itemize}
\item[(S$_2$a)] The vertex partition $\cP_2 : = \{A_0, B_0, A_1, \dots A_{K_2}, B_1, \dots, B_{K_2} \}$ is a $(K_2, \break m_2, \eps_0)$-partition of $V(G)$.
\item[(S$_2$b)] $(G_4[A] + G_4[B], \cP_2)$ is a $(K_2, m_2, \eps_0, \eps'_2)$-scheme.
\item[(S$_2$c)] $(G_4^{\diamond}, \cP_2)$ is a $(K_2, m_2, \eps_0, \eps'_1)$-exceptional scheme. 
\end{itemize}
(Recall that $G_4^{\diamond} : = G_1^{\diamond} - \bigcup\mathcal{J}^{\rm rob}$
was defined towards the end of  Step~3. Since $G_4 = G_1 - G^{\rm rob}$, we have $(G_4)^{\diamond} = G^{\diamond}_1 - (G^{\rm rob})^{\diamond} = G_1^{\diamond} - \bigcup\mathcal{J}^{\rm rob}$,
so $(G_4)^{\diamond}$ is indeed the same as $G_4^{\diamond}$.)
Moreover, by Lemma~\ref{lma:partition}(iv) we have
\begin{align}
	d_{G_4}(v,A_i)  = (d_{G_4}(v,A) \pm \eps_0 n)/ K_2 \qquad \text{and} \qquad
	d_{G_4}(v,B_i)  = (d_{G_4}(v,B) \pm \eps_0 n)/ K_2 \label{eq:D4}
\end{align}
for all $v \in V(G)$ and $1 \le i \le K_2$.
(Note that the previous partition of $A$ and $B$ plays no role in the subsequent argument, so
denoting the clusters in~$\cP_2$ by $A_i$ and $B_i$ again will cause no notational conflicts.)

Since $(G_4, A,A_0,B,B_0 )$ is an $(\eps_0, K_2)$-framework, (FR3) and (FR4) together imply that
each $v \in A$ satisfies $d_{G_4}(v, A_0) \le |V_0| \le \eps _ 0 n $ and $d_{G_4}(v, B') \le \eps _ 0 n $. So $d_{G_4}(v,A) = d_{G_4}(v) \pm 2 \eps_0 n  $.
Therefore, for all $v \in A$ and all $ 1 \le i \le K_2$ we have
\begin{align}
	d_{G_4}(v,A_i) &  \stackrel{\eqref{eq:D4}}{=} \frac{d_{G_4}(v,A) \pm \eps_0 n}{ K_2 }
	 = \frac{d_{G_4}(v) \pm 3 \eps_0 n}{ K_2 }
    = \frac{d_{G_4}(v) \pm 7\eps_0 K_2 m_2}{K_2}. 
 \label{eq:D2}
\end{align}
The analogue holds for $d_{G_4}(v,B_i)$ (where $v \in B$ and $ 1 \le i \le K_2$).

\smallskip

\noindent\textbf{Step 6: Exceptional systems for the approximate decomposition.}
In order to apply Lemma~\ref{almostthm}, we first need to construct suitable exceptional systems.
We will show that $G_4^{\diamond}$ can be decomposed completely into $D_4/2$ exceptional systems with parameter $\eps_0$.
Moreover, these exceptional systems can be partitioned into sets $\mathcal{J}'_0$ and $\mathcal{J}'_{i_1,i_2}$ (one set for each pair $1 \le i_1, i_2  \le K_2$)
such that the following conditions hold, where $\cJ''$ denotes the union of $\mathcal{J}'_{i_1, i_2}$ over all $1 \le i_1, i_2  \le K_2$: 
\begin{itemize}
	\item[($\gamma_1$)] Each $\mathcal{J}'_{i_1, i_2}$ consists of precisely $(D_4- 2 \lambda_{K_2} n)/2K^2_2$ $(i_1,i_2)$-ES with parameter $\eps _0$ with respect to the partition $\cP_2$.
	\item[($\gamma_2$)] $\mathcal{J}'_0$ contains precisely $\lambda_{K_2} n$ exceptional systems
	 with parameter $\eps _0$.
	\item[($\gamma_3$)] If $e_G(A',B') \ge D$, then all exceptional systems in $\mathcal{J}'_0 \cup \cJ''$ are Hamilton exceptional systems.
	\item[($\gamma_4$)] If $e_G(A',B') < D$, then each exceptional system $J \in \mathcal{J}'_0 \cup \cJ''$ is a Hamilton exceptional system
    with $e_J(A',B') = 2$ or a matching exceptional system.
	In particular, $\mathcal{J}'_0$  contains precisely
$e_{ \bigcup \mathcal{J}'_0 }(A',B')/2$ Hamilton exceptional systems and $\mathcal{J}''$  contains precisely
$e_{ \bigcup \mathcal{J}'' }(A',B')/2$ Hamilton exceptional systems.
\end{itemize}
As in Step~3, the construction of $\mathcal{J}'_0$ and the $\mathcal{J}'_{i_1,i_2}$ will depend on whether $G$ is critical and whether $e_G(A',B') \ge D$.
Recall that $G_4=G_1-G^{\rm rob}$ and note that
\begin{equation} \label{divD4b}
\frac{D- \phi_0n - 2 \phi^{\rm rob}_0 n}{400K^2_2}=\frac{D_4}{400K^2_2}\in \mathbb{N}
\end{equation}
by~\eqref{divD4}.

\noindent\textbf{Case 1: $e_G(A',B') \ge D$ and $G$ is not critical.}
Our aim is to apply Lemma~\ref{lma:BESdecom} to $G$ with $G - G_4$, $m_2$, $K_2$, $\cP_2$, $\eps'_1$, $\phi_0+2 \phi^{\rm rob}_0$, $ \lambda_{K_2}$
playing the roles of $G_0$, $m$, $K$, $\cP$, $\eps$, $\phi$, $\lambda$.
(So $G^{\diamond}_4$ will play the role of $G^{\diamond}$.) First we verify that the conditions in Lemma~\ref{lma:BESdecom}(i)--(iv) are satisfied. 
Clearly, Lemma~\ref{lma:BESdecom}(i) and ~(ii) hold.
Note that $G - G_4 = \mathcal{H}'_1 + G^{\rm rob}$, so ($\alpha_1$), ($\alpha_2$), ($\beta'_1$) and ($\beta'_2$) imply Lemma~\ref{lma:BESdecom}(iii).
By ($\alpha_2$) and ($\beta'_2$), $e_{G_4^{\diamond}}(A',B')$ is even. Together with the fact (S$_2$r) that
$(G^{\diamond}_4 , \cP_2)$ is a $(K_2, m_2, \eps_0, \eps'_1)$-exceptional scheme, this shows that Lemma~\ref{lma:BESdecom}(iv) holds.
Together with \eqref{divD4b} this ensures that we can indeed apply Lemma~\ref{lma:BESdecom} to obtain a set of
$(D- (\phi_0 + 2 \phi^{\rm rob}_0) n)/2 = D_4/2$ edge-disjoint Hamilton exceptional systems with parameter $\eps_0$ in $G_4$.
Moreover, these Hamilton exceptional systems can be partitioned into sets $\mathcal{J}'_0$ and
$\mathcal{J}'_{i_1,i_2}$ (for all $1 \le i_1, i_2  \le K_2$) such that ($\gamma_1$)--($\gamma_3$) hold.

\noindent\textbf{Case 2: $e_G(A',B') \ge D$ and $G$ is critical.}
Our aim is to apply Lemma~\ref{lma:BESdecomcritical} to $G$ with $G - G_4$, $m_2$, $K_2$, $\cP_2$, $\eps'_1$, $\phi_0+2 \phi^{\rm rob}_0$, $ \lambda_{K_2}$
playing the roles of $G_0$, $m$, $K$, $\cP$, $\eps$, $\phi$, $\lambda$.
(So as before, $G^{\diamond}_4$ will play the role of $G^{\diamond}$.) Similar arguments as in Case~1 show that
Lemma~\ref{lma:BESdecomcritical}(i)--(iv) hold. ($ \beta_4$) implies Lemma~\ref{lma:BESdecomcritical}(v).
Together with \eqref{divD4b} this ensures that we can indeed apply
Lemma~\ref{lma:BESdecomcritical} to obtain a set of $D_4/2$ edge-disjoint Hamilton exceptional
systems with parameter $\eps_0$ in $G_4$. Moreover, these Hamilton exceptional systems can be partitioned into sets $\mathcal{J}'_0$ and
$\mathcal{J}'_{i_1,i_2}$ (for $1 \le i_1, i_2  \le K_2$) such that ($\gamma_1$)--($\gamma_3$) hold.

\noindent\textbf{Case 3: $e_G(A',B') < D$.}
Recall from Proposition~\ref{prp:e(A',B')}(ii) that in this case we have $D = n/2 - 1$, $n = 0 \pmod 4$ and $|A'| = |B'| = n/2$.
Our aim is to apply Lemma~\ref{lma:PBESdecom} to $G$ with $G - G_4$, $m_2$, $K_2$, $\cP_2$, $\eps'_1$, $\phi_0+2 \phi^{\rm rob}$, $ \lambda_{K_2}$
playing the roles of $G_0$, $m$, $K$, $\cP$, $\eps$, $\phi$, $\lambda$.
(So as before, $G^{\diamond}_4$ will play the role of $G^{\diamond}$.) Similar arguments as in Case~1 show that
Lemma~\ref{lma:PBESdecom}(i)--(iv) hold. ($ \beta_5$) implies Lemma~\ref{lma:BESdecomcritical}(v).
Together with \eqref{divD4b} this ensures that we can indeed apply 
Lemma~\ref{lma:PBESdecom} to obtain a set of $D_4 /2 $ edge-disjoint exceptional systems in $G_4$.
Moreover, these exceptional systems can be partitioned into sets $\mathcal{J}'_0$ and $\mathcal{J}'_{i_1,i_2}$ (for all $1 \le i_1, i_2  \le K_2$)
such that ($\gamma_1$), ($\gamma_2$) and ($\gamma_4$) hold. (In particular, ($\gamma_4$) implies that each exceptional system
in these sets has parameter $\eps_0$.)

\smallskip

Therefore, in each of the three cases we have constructed sets $\mathcal{J}'_0$ and $\mathcal{J}'_{i_1,i_2}$ (for all $1 \le i_1, i_2  \le K_2$)
satisfying ($\gamma_1$)--($\gamma_4$).

We now find Hamilton cycles and perfect matchings covering the `non-localized' exceptional systems (i.e.~the ones in $\mathcal{J}'_0$).
Let $G'_4 = G_4 - G_4^{\diamond}$. So $G'_4$ is obtained from $G_4$ by keeping all edges inside $A$ as well as all edges inside $B$,
and deleting all other edges.%
   \COMMENT{Daniela: previously had "Let $G'_4:= G_4[A] + G_4[B]$. Note that $G'_4 = G_4 - G_4^{\diamond}$." But then $G'_4$ doesn't contain $V_0$
and so $(G'_4,A,A_0,B,B_0)$ wouldn't be an $(\eps_0,K_2)$-framework}
Note that $(G'_4,A,A_0,B,B_0)$ is an $(\eps_0,K_2)$-framework since $(G_4,A,A_0,B,B_0)$ is an $(\eps_0,K_2)$-framework. 
Apply Lemma~\ref{badBES} to $(G'_4,A,A_0,B,$ $B_0)$ with $K_2$, $ \lambda_{K_2}$, $\cJ'_0$ playing the roles of $K$, $\lambda$, $\{ J_1, \dots, J_{\lambda n} \}$.
(Recall from (S$_2$b) that $(G_4[A] +G_4[B], \cP_2)$ is a $(K_2, m_2, \eps_0, \eps'_2)$-scheme, so $\delta(G'_4[A])=\delta(G_4[A]) \ge 4|A|/5$
and $\delta(G'_4[B])=\delta(G_4[B]) \ge 4|B|/5$ by (Sch3).)%
   \COMMENT{Daniela: added $\delta(G'_4[A])=$ and similarly for $B$} 
We obtain edge-disjoint subgraphs $H_1, \dots, H_{|\mathcal{J}_0'|}$ of $G'_4 + \bigcup \mathcal{J}'_0$ such that,
writing $\mathcal{H}_2:=\{H_1, \dots, H_{|\mathcal{J}'_0|}\}$, the following conditions hold:
\begin{itemize}
\item[($\delta_1$)] For each $H_s\in \mathcal{H}_2$ there is some $J_s\in \mathcal{J}'_0$ such that $J_s\subseteq H_s$.
\item[($\delta_2$)] If $J_s$ is a Hamilton exceptional system, then $H_s$ is a Hamilton cycle on $V(G)$. If
$J_s$ is a matching exceptional system, then $H_s$ is the edge-disjoint union of two perfect matchings on $V(G)$.
\item[($\delta_3$)] Let $\mathcal{H}'_2:=H_1+ \dots +H_{|\mathcal{J}'_0|}$. If $e_G(A',B') < D$, then $\mathcal{H}_2$ contains precisely $e_{\mathcal{H}'_2}(A',B')/2$ Hamilton cycles on $V(G)$.
\end{itemize}
Indeed, ($\delta_1$) follows from Lemma~\ref{badBES}(i). 
($\delta_2$) follows from Lemma~\ref{badBES}(ii),(iii).
(For the second part, note that ($\gamma_3$) and ($\gamma_4$) imply that $\cJ'_0$ contains matching exceptional systems only in the case when
$e_G(A',B') < D$. But in this case, Proposition~\ref{prp:e(A',B')}(ii) implies that $n = 0 \pmod 4$ and $|A'| = |B'| = n/2$, i.e.~$|A'|$ and $|B'|$ are even.)
For ($\delta_3$), note that $G'_4$ has no $A'B'$-edges
and so $e_{ \bigcup\mathcal{J}'_0}(A',B')=e_{\mathcal{H}'_2}(A',B')$.
Together with ($\delta_2$) and ($\gamma_4$), this now implies ($\delta_3$).%
\COMMENT{Deryk rephrased this paragraph}

Recall that $\cJ''$ is the union of $\mathcal{J}'_{i_1, i_2}$ over all $1 \le i_1, i_2  \le K_2$.
Let $G_5:=G_4-\mathcal{H}'_2$ and $D_5:=D_4-2|\mathcal{H}_2| = D_4 - 2 \lambda_{K_2} n $.
So \eqref{degrees4} implies that
\begin{equation} \label{degrees5}
d_{G_5}(v)=D_5+2r \mbox{ for all } v \in A \cup B \qquad \mbox{and} \qquad d_{G_5}(v)=D_5 \mbox{ for all } v \in V_0.
\end{equation}
Note that 
\begin{equation} \label{G5diam}
G^{\diamond}_5 := G_5 - G_5[A] - G_5[B] = G_4^{\diamond} - \mathcal{H}'_2 = G_4^{\diamond} - \bigcup \mathcal{J}'_0=\bigcup \cJ''.
\end{equation}
Since $d_J(v) =2$ for all $v \in V_0$ and all $J \in \mathcal{J}''$, it follows that
\begin{align}
D_5 = 2 |\cJ''|. \label{eq:D5}
\end{align}
Moreover, since  $(G_4[A] +G_4[B], \mathcal{P}_2)$ is a $(K_2, m _2, \eps_0, \eps'_2)$-scheme and $\eps'_2 + 2\lambda_{K_2} \le \eps$,
Proposition~\ref{deleteBS} implies that $(G_5[A] +G_5[B], \mathcal{P}_2)$ is a $(K_2, m _2, \eps_0, \eps)$-scheme.

\smallskip

\noindent\textbf{Step 7: Approximate Hamilton cycle decomposition.}
Our next aim is to apply Lemma~\ref{almostthm} to obtain an approximate decomposition of $G_5$. 
Let 
\begin{align*}
\mu:=(r^{\rm rob}_0-2r)/(4K_2m_2) \qquad \text{and} \qquad \rho := \gamma/(4 K_1).
\end{align*}
We will apply the lemma with $G_5$, $\cP_2$, $K_2$, $m_2$, $\cJ''$, $\eps$ playing the roles of $G$, $\cP$, $K$, $m$, $\cJ$, $\eps$.
Clearly, conditions~(c) and~(d) of Lemma~\ref{almostthm} hold.

In order to see that condition (a) is satisfied, recall that
$m_1K_1=|A|=m_2K_2$. So%
\COMMENT{Deryk replaced lower bound on $\mu$ with 0 and replaced upper bound with $1$ rather than $\eps$ (as statement of
approx lemma changed slightly} 
$$
0\le \frac{7r_1-2r}{4K_2m_2}\stackrel{\eqref{eq:rrob}}{\le} \mu 
\stackrel{\eqref{eq:rrob}}{\le} \frac{30r_1}{4K_2m_2}
= \frac{30\gamma_1 }{4K_1}\ll 1.
$$
Therefore, every vertex $v\in A \cup B$ satisfies
\begin{eqnarray}
d_{G_4}(v) & \stackrel{\eqref{degrees4}}{=} & D_4+2r 
\stackrel{\eqref{D4D1}}{=} D_1-r^{\rm rob}_0+2r
 \stackrel{\eqref{eqD1}}{=}  D - \phi_0 n - 4 K_2 m _2 \mu \nonumber \\
 & \stackrel{\eqref{eq:Dupper}}{ = } & (1/2  \pm 4 \eps_{\rm ex} ) n - \phi_0 n - 4 K_2 m _2 \mu \nonumber \\
& = & \left( 1-4\mu \pm 3 \phi_0 \right)   K_2 m_2, 
 \label{eq:D6}
\end{eqnarray}
where in the last equality we recall that $(1-\eps_0) n /2 \le |A| = K_2 m_2 \le n/2$ and $\eps_0 ,  \eps_{\rm ex} \ll \phi_0$.
Recall that $G_5=G_4-\mathcal{H}'_2$ and note%
\COMMENT{Deryk replaced first inequality with $=$}
that
$$
\Delta(\mathcal{H}'_2) = 2 | \mathcal{H}_2 | =2 \lambda_{K_2} n \le 5 \lambda_{K_2}  K_2m_2.
$$
Altogether this implies that for each $v \in A$ and for all $1 \le i \le K_2$ we have
\begin{eqnarray*}
	d_{G_5} (v, A_i) & = & d_{G_4} (v,A_i) - d_{\mathcal{H}'_2} (v,A_i) 
= d_{G_4} (v,A_i) \pm 5 \lambda_{K_2}  K_2m_2 \\
	& \stackrel{ \eqref{eq:D2}}{=} &	( d_{G_4} (v) \pm 7 \eps_0 K_2 m_2  )/K_2 \pm 5 \lambda_{K_2}  K_2m_2 \\
	& \stackrel{\eqref{eq:D6}}{=} &	\left( 1-4\mu \pm (3 \phi_0 + 7 \eps_0 + 5 \lambda_{K_2}  K_2) \right)  m_2.
\end{eqnarray*}
Since $\phi_0, \eps_0, \lambda_{K_2}  \ll 1/K_2$, it follows that $d_{G_5}(v,A_i)=(1-4\mu\pm 4/K_2)m_2$. Similarly
one can show that $d_{G_5}(w,B_j)=(1-4\mu\pm 4/K_2)m_2$ for all $w\in B$. So Lemma~\ref{almostthm}(a) holds.

To check condition~(b), note that $r= \gamma |A|/K_1 \ge \gamma n/3K_1$. So 
\begin{eqnarray*}
|\cJ''| & \stackrel{\eqref{eq:D5}}{=} & \frac{D_5}{2} \le \frac{D_4}{2}
\stackrel{\eqref{D4D1}}{=}  \frac{D-r^{\rm rob}_0}{2} 
\stackrel{\eqref{eq:Dupper}}{\le} \frac{n}4 + 2 \eps_{\rm ex} n - \frac{r^{\rm rob}_0}{2} \\
& =  &\frac{n}4 + 2 \eps_{\rm ex} n - 2 K_2 m_2  \mu - r 
 \le  \left( \frac{1}{4} + 2 \eps_{\rm ex} -  (1- \eps_0)\mu - \frac{\gamma}{3 K_1}\right)n \\
& \le  & \left( \frac{1}{4}- \mu - \frac{\gamma}{4 K_1}  \right) n = \left( \frac{1}{4}- \mu - \rho  \right) n.
\end{eqnarray*}
Thus Lemma~\ref{almostthm}(b) holds.

So we can indeed apply Lemma~\ref{almostthm} to obtain a collection $\mathcal{H}_3$ of $|\cJ''|$ edge-disjoint
spanning subgraphs $H'_1,\dots,H'_{ |\cJ''|}$ of $G_5$ which satisfy the following properties:
\begin{itemize}
\item[($\eps_1$)] For each $H'_s\in \mathcal{H}_3$ there is some $J'_s\in \mathcal{J}''$ such that $J'_s\subseteq H'_s$.
\item[($\eps_2$)] If $J'_s$ is a Hamilton exceptional system then $H'_s$ is a Hamilton cycle on $V(G)$. If
$J'_s$ is a matching exceptional system then $H'_s$ is the edge-disjoint union of two perfect matchings on $V(G)$.
\item[($\eps_3$)] Let $\mathcal{H}'_3:=H'_1+\dots+H'_{|\cJ''|}$. If $e_G(A',B') < D$, then $\mathcal{H}_3$ contains precisely
$e_{\mathcal{H}'_3}(A',B')/2$ Hamilton cycles on $V(G)$.
\end{itemize}
For ($\eps_3$), note that~(\ref{G5diam}) implies $G^{\diamond}_5 =\bigcup \cJ''$ and thus we have
$e_{ \bigcup\mathcal{J}''}(A',B')=e_{\mathcal{H}'_3}(A',B')$.
Together with ($\eps_2$) and ($\gamma_4$), this now implies~($\eps_3$).%
\COMMENT{Deryk rephrased this}
\smallskip

\noindent\textbf{Step 8: Decomposing the leftover and the robustly decomposable graph.}
Finally, we can apply the `robust decomposition property' of $G^{\rm rob}$ guaranteed by Corollary~\ref{rdeccor}
to obtain a decomposition of the leftover from the previous step together with $G^{\rm rob}$ into Hamilton cycles 
(and perfect matchings if applicable).
 
To achieve this, let $H':=G_5-\mathcal{H}'_3$. Thus~\eqref{degrees5} and ~\eqref{eq:D5} imply that
every vertex in $V_0$ is isolated in $H'$ while every vertex $v\in A\cup B$ has degree $d_{G_5}(v)-2|\cJ''|=D_5+2r-2|\cJ''|=2r$ in~$H'$
(the last equality follows from~\eqref{eq:D5}). Moreover, $(H')^{\diamond}$ contains no edges. (This holds
since $\bigcup \cJ''\subseteq \mathcal{H}'_3$ and so $H'\subseteq G_5-\bigcup \cJ''= G_5- G_5^{\diamond}$ by~(\ref{G5diam}).)%
    \COMMENT{Daniela: added back reference}
Now let $H_A:=H'[A]$, $H_B:=H'[B]$, $H:=H_A+H_B$. 
Note that $H$ is the $2r$-regular subgraph of $H'$ obtained by removing all the vertices in $V_0$.
Let%
\COMMENT{Daniela: made extra back reference + added one more step}
$$
s':=rfK_1+7r^\diamond \stackrel{(\ref{D4eq})}{=}Lfr_3+7r^\diamond \stackrel{(\ref{robphieq})}{=}\phi_0^{\rm rob} n.
$$
Recall from~($\beta_3'$) that each of the $s'$ exceptional path systems in $\mathcal{EF}_{\rm CA}+\mathcal{EF}_{\rm PCA}$ contains
a unique exceptional system%
\COMMENT{Deryk strengthened ($\beta_3'$) to make reference to ($\beta_3$) superfluous}
and  $\cJ^{\rm rob}$ is the set of all these $s'$ exceptional systems.
Thus Corollary~\ref{rdeccor}(ii)(b) implies that $H + G^{\rm rob}$ has a decomposition into
edge-disjoint spanning subgraphs
$H''_1,\dots,H''_{s'}$ such that, writing $\mathcal{H}_4:=\{H''_1,\dots,H''_{s'}\}$,
we have:%
   \COMMENT{Daniela: deleted "If $e_G(A',B') < D$" in ($\zeta_3$) since it holds in general}
\begin{itemize}
\item[($\zeta_1$)] For each $H''_s\in \mathcal{H}_4$ there is some exceptional system $J''_s\in \cJ^{\rm rob}$ such that $J''_s\subseteq H''_s$.
\item[($\zeta_2$)] If $J''_s$ is a Hamilton exceptional system then $H''_s$ is a Hamilton cycle on $V(G)$. If
$J''_s$ is a matching exceptional system then $H''_s$ is the edge-disjoint union of two perfect matchings on $V(G)$.
\item[($\zeta_3$)] Let $\mathcal{H}'_4:=H''_1+\dots+H''_{s'}$. Then $\mathcal{H}_4$ contains precisely
$e_{\mathcal{H}'_4}(A',B')/2$ Hamilton cycles on $V(G)$.
\end{itemize}
Indeed, ($\zeta_1$) and ($\zeta_2$) follow from Corollary~\ref{rdeccor}(ii)(b)
(recall that if $\cJ^{\rm rob}$ contains a matching exceptional system, then $|A'|=|B'|=n/2$ is even by ($\beta_3'$)).
For ($\zeta_3$), note that $e_{\mathcal{H}'_4}(A',B')=  e_{G^{\rm rob}}(A',B')= e_{\bigcup\cJ^{\rm rob}}(A',B')$ by ($\beta_2'$).%
   \COMMENT{Daniela: replaced $e_{\cJ^{\rm rob}}(A',B')$ by $e_{\bigcup\cJ^{\rm rob}}(A',B')$}
Now ($\zeta_3$) follows from ($\beta'_3$) and ($\zeta_2$).

Note that $\mathcal{H}_1\cup \mathcal{H}_2\cup \mathcal{H}_3\cup \mathcal{H}_4$ corresponds to a decomposition of $G$ into Hamilton cycles
and perfect matchings. It remains to show that the proportion of Hamilton cycles in this decomposition is as desired.

First suppose that $e_G(A',B') \ge D$.
By ($\alpha_3$), $\mathcal{H}_1$ consists of Hamilton cycles and one perfect matching if $D$ is odd.
By ($\gamma_3$), ($\delta_2$) and ($\eps_2$), both $\mathcal{H}_2$ and $\mathcal{H}_3$ consist of Hamilton cycles.
By ($\beta'_3$) and ($\zeta_2$) this also holds for $\mathcal{H}_4$.
So $\mathcal{H}_1\cup \mathcal{H}_2\cup \mathcal{H}_3\cup \mathcal{H}_4$ consists of Hamilton cycles and one perfect matching if $D$ is odd. 

Next suppose that $e_G(A',B') < D$. Then by ($\alpha_3$), ($\delta_3$), ($\eps_3$) and ($\zeta_3$)
the numbers of Hamilton cycles in $\mathcal{H}_1$, $\mathcal{H}_2$, $\mathcal{H}_3$ and $\mathcal{H}_4$ are precisely
$\lfloor e_{\mathcal{H}'_1}(A',B')/2 \rfloor $, $ e_{\mathcal{H}'_2}(A',B')/2 $, $ e_{\mathcal{H}'_3}(A',B')/2 $ and $ e_{\mathcal{H}'_4}(A',B')/2 $.
Hence, $\mathcal{H}_1\cup \mathcal{H}_2\cup \mathcal{H}_3\cup \mathcal{H}_4$ contains precisely
\begin{align*}
\left\lfloor \frac{e_{\mathcal{H}'_1  \cup \mathcal{H}'_2  \cup \mathcal{H}'_3 \cup \mathcal{H}'_4}(A',B')}2 \right\rfloor 
= \left\lfloor \frac{e_{G}(A',B')}2 \right\rfloor\ge  \left\lfloor \frac{F}2 \right\rfloor
\end{align*}
edge-disjoint Hamilton cycles, where $F$ is the size of the minimum cut in $G$. Since clearly $G$ cannot have more than $\lfloor F/2 \rfloor$
edge-disjoint Hamilton cycles, it follows that we have equality in the final step, as required.%
\COMMENT{Deryk: we need not have $F=e_{G}(A',B')$, as previously claimed}
\endproof

\chapter{Exceptional systems for the two cliques case}\label{paper4}

In this chapter we prove all the results that were stated in Section~\ref{sec:locES}.
Recall that the exceptional edges are all those edges incident to $A_0$ and $B_0$ as well as all those edges joining $A'$ to $B'$.
The results stated in Section~\ref{sec:locES} generated a decomposition of 
these exceptional edges  into exceptional  systems: Each such exceptional system
was then  extended into a Hamilton cycle. (Recall that actually, the exceptional systems  may contain some non-exceptional edges as well.)
This is the most difficult part of the construction of the Hamilton cycle decomposition and so forms the heart of the argument for the two clique case.

Let $G$ be a $D$-regular graph and let $A',B'$ be a partition of $V(G)$. 
Recall that we say that
$G$ is \emph{critical} (with respect to $A',B'$ and $D$) if both of the following hold:
\begin{itemize}
\item $\Delta(G[A',B']) \ge 11 D/40$;
\item  $e(H) \le 41 D/40$ for all subgraphs $H$ of $G[A',B']$ with $\Delta(H) \le 11 D/40$.%
     \COMMENT{Note that there is no assumption on $e_G(A',B')$.}
\end{itemize}

Recall that Lemmas~\ref{lma:BESdecom}--\ref{lma:PBESdecom} guarantee our desired decomposition of the exceptional edges into exceptional systems.
Lemma~\ref{lma:BESdecom} covers the non-critical case when $G[A',B']$  contains many edges, Lemma~\ref{lma:BESdecomcritical} covers the critical case when $G[A',B']$  contains many edges
and Lemma~\ref{lma:PBESdecom} tackles the case when $G[A',B']$ contains
only a few edges.

\section{Proof of Lemma~\ref{critical}}\label{secnewzz}

The following lemma (which collects some basic properties of critical graphs) immediately implies Lemma~\ref{critical}.

\begin{lemma} \label{critical'}
Suppose that $0< 1/n \ll 1$ and that $D, n \in \mathbb N$ are such that
\begin{equation} \label{minexact'}
D\ge n - 2\lfloor n/4 \rfloor -1=
\begin{cases}
n/2-1 & \textrm{if $n = 0 \pmod 4$,}\\
(n-1)/2 & \textrm{if $n = 1 \pmod 4$,}\\
n/2 & \textrm{if $n = 2 \pmod 4$,}\\
(n+1)/2 & \textrm{if $n = 3 \pmod 4$.}
\end{cases}
\end{equation}
Let $G$ be a $D$-regular graph on $n$ vertices and let $A',B'$ be a partition of $V(G)$ with $|A'|,|B'| \ge D/2$ and $\Delta(G[A',B']) \le D/2$.
Suppose that $G$ is critical.
Let $W$ be the set of vertices $w \in V(G)$ such that $d_{G[A',B']}(w) \ge 11D/40$.
Then the following properties are satisfied:%
\begin{itemize}
	\item[$ \rm (i)$] $1 \le |W| \le 3$.
	\item[$ \rm (ii)$] Either $D = (n-1)/2$ and $n = 1 \pmod{4}$, or $D = n/2-1$ and $n = 0 \pmod{4}$.
	Furthermore, if $n =1 \pmod{4}$, then $|W|=1$.
	\item[$\rm (iii)$] $e_{G}(A',B') \le 17D/10+5 < n$.
	\item[$\rm (iv)$] \begin{align*}
e_{G- W}(A',B') \le 
\begin{cases}
3D/4+5	& \textrm{if $|W|=1$,} \\
19D/{40}+5	& \textrm{if $|W|=2$,} \\
D/5+5	& \textrm{if $|W|=3$.}
\end{cases}
\end{align*}
	\item[$ \rm (v)$] There exists a set $W'$ of vertices such that $W \subseteq W'$, $|W'| \le 3$ and for all $w' \in W'$ and $v \in V(G) \setminus W'$
	we have 
\begin{align*}
d_{G[A',B']}(w') & \ge \frac{21D}{80}, \, 
d_{G[A',B']}(v) & \le \frac{11D}{40} \
{\rm and} & \
d_{G[A',B']}(w') - d_{G[A',B']}(v) \ge \frac{D}{240}.
\end{align*}
	\end{itemize}
\end{lemma}
\begin{proof}
Let $w_1, \dots, w_4$ be vertices of $G$ such that 
\begin{align*}
d_{G[A',B']}(w_1) \ge \dots \ge d_{G[A',B']}(w_4) \ge d_{G[A',B']}(v)
\end{align*}
for all $v \in V(G) \setminus \{w_1, \dots, w_4\}$. Let $W_4:=\{w_1, \dots, w_4\}$. Suppose that $d_{G[A',B']}(w_4)$ $\ge 21D/80$.
Let $H$ be a spanning subgraph of $G[A',B']$ such that
$d_H(w_i)= \lceil 21D/80\rceil$ for all $i\le 4$ and such that every vertex $v\in V(G)\setminus W_4$ satisfies
$N_H(v)\subseteq W_4$.%
    \COMMENT{Such a $H$ can be obtained by starting with $G[A'\cap W_4,B'\cap W_4]$ and adding suitable further edges.}
Thus $\Delta(H)= \lceil 21D/80\rceil$ and so $e(H) \le 41D/40$ since $G$ is critical. On the other hand, $e(H)\ge 4\cdot \lceil 21D/80\rceil -4$,
a contradiction. (Here we subtract four to account for the edges of $H'$ between vertices in $W$.)
Hence, $d_{G[A',B']}(w_4) < 21 D /80$ and so $|W|\le 3$. But $|W|\ge 1$ since $G$ is critical. So (i) holds.

Let $j$ be minimal such that
$d_{G[A',B']}(w_j)\le 21D/80$. So $1<j\le 4$. Choose an index $i$ with $1\le i<j$ such that
$W\subseteq \{w_1,\dots,w_i\}$ and $d_{G[A',B']}(w_i) - d_{G[A',B']}(w_{i+1}) \ge D/240$.
Then the set $W':=\{w_1,\dots,w_i\}$ satisfies~(v).

Let $H'$ be a spanning subgraph of $G[A',B']$ such that $G[A'\setminus W,B'\setminus W]\subseteq H'$ and
$d_{H'}(w)= \lfloor 11D/40 \rfloor$ for all $w\in W$.%
    \COMMENT{Such a $H'$ can be obtained by starting with $G[A'\cap W,B'\cap W]+ G[A'\setminus W,B'\setminus W]$ and adding
suitable further edges between $W$ and $V(G)\setminus W$.}
Similarly as before, $e(H') \le 41D/40$ since $G$ is critical. Thus 
\begin{align*}
	41D/40 & \ge e(H') \ge e ( H'- W ) + \lfloor 11D / 40 \rfloor |W| - 2\\
	& = e_{G- W}(A',B') + \lfloor 11D / 40 \rfloor |W| - 2.
\end{align*}
This in turn implies that
\begin{align}
e_{G- W}(A',B') & \le (41 - 11 |W|)D / 40 + 5 \label{eq:e(HnotW)}.
\end{align}
Together with~(i) this implies~(iv).
If $D \ge n/2$, then by Proposition~\ref{prp:e(A',B')2} we have $e_{G- W}(A',B')  \ge D - 28$.
This contradicts~(iv). Thus~(\ref{minexact'}) implies that $D = (n-1)/2$ and $n = 1 \pmod{4}$, or $D = n/2-1$ and $n = 0 \pmod{4}$.
If $n = 1 \pmod 4$ and $D=(n-1)/2$, then Proposition~\ref{prp:e(A',B')2} implies that $e_{G-W}(A',B')\ge D/2-28$.
Hence, by (iv) we deduce that $|W| =1$ and so (ii) holds.
Since $|W| \le 3$ and $\Delta(G[A',B']) \le D/2$, we have
\begin{align*}
	e_G (A',B') & \le e_{G- W}(A',B') + \frac{|W| D}2
		 \stackrel{(\ref{eq:e(HnotW)})}{\le} \frac{(41+9|W|)D}{40}+5 \le \frac{17D}{10}+5 <n.
\end{align*}
(The last inequality follows from~(ii).) This implies~(iii).
\end{proof}

\section{Non-critical Case with $e(A',B') \ge D$.} \label{noncritical}
In this section we prove Lemma~\ref{lma:BESdecom}. Recall that Lemma~\ref{lma:BESdecom} gives a decomposition of the exceptional edges into exceptional systems in the non-critical case when $e(A',B') \ge D$.
The proof splits into the following four steps:
\begin{itemize}
	\item[\bf Step 1] We first decompose $G^{\diamond}$ into edge-disjoint `localized' subgraphs $H(i,i')$ and $H'(i,i')$ (where $1\le i,i' \le K$).
	More precisely, each $H(i,i')$ only contains $A_0A_i$-edges and $B_0B_{i'}$-edges of $G^\diamond$ while all edges of $H'(i,i')$
lie in $G^{\diamond}[A_0 \cup A_i,  B_0 \cup B_{i'}]$, and all the edges of $G^\diamond$ are distributed evenly amongst the $H(i,i')$ and $H'(i,i')$
(see Lemma~\ref{lma:randomslice}). 
We will then move a small number of $A'B'$-edges between the $H'(i,i')$ in order to obtain graphs $H''(i,i')$ such that $e(H''(i,i'))$ is even (see Lemma~\ref{lma:move}).
	\item[\bf Step 2] We decompose each $H''(i,i')$ into $(D - \phi n )/(2K^2)$ Hamilton exceptional system candidates (see Lemma~\ref{lma:BESdecomprelim}).
	\item[\bf Step 3] Most of the Hamilton exceptional system candidates constructed in Step~2 will be extended into an $(i,i')$-HES (see Lemma~\ref{lma:BESextend2}).
	\item[\bf Step 4] The remaining Hamilton exceptional system candidates will be extended into Hamilton exceptional systems, which need not be localized
	(see Lem\-ma~\ref{globalBES}).
	(Altogether, these will be the $\lambda n$ Hamilton exceptional systems in $\mathcal{J}$ which are not mentioned in Lemma~\ref{lma:BESdecom}(b).)
\end{itemize}

\subsection{Step $1$: Constructing the Graphs $H''(i,i')$}
Let $H(i,i')$ and $H'(i,i')$ be the graphs obtained by applying Lemma~\ref{lma:randomslice} to $G^{\diamond}$.
We would like to decompose each $H'(i,i')$ into Hamilton exceptional system candidates.
In order to do this, $e(H'(i,i'))$ must be even. The next lemma shows that we can ensure this property
without destroying the other properties of the $H'(i,i')$ too much by moving a small number of edges 
between the $H'(i,i')$.

\begin{lemma} \label{lma:move}
Suppose that $0 <  1/n  \ll \epszero \ll  \eps \ll   \eps' \ll \lambda, 1/K \ll 1$, that $D\ge n/3$, that
$0 \le \phi  \ll 1$ and that $D, n, K, m,  (D - \phi n)/(2K^2) \in \mathbb{N}$.
Define $\alpha$ by 
\begin{align} \label{alpha}
2 \alpha n := \frac{ D - \phi n }{K^2}  \ \  \ \ \ \ \text{and let} \ \ \ \  \  \ 
\gamma  : = \alpha -  \frac{2\lambda}{K^2}.
\end{align}
Suppose that the following conditions hold:
\begin{itemize}
	\item[$ \rm (i)$] $G$ is a $D$-regular graph on $n$ vertices.
	\item[$ \rm (ii)$] $\mathcal{P}$ is a $(K, m, \epszero)$-partition of $V(G)$ such that $D \le e_G(A',B') \le \epszero n^2$ and $\Delta(G[A',B']) \le D/2$.
Furthermore, $G$ is not critical.
	\item[$ \rm (iii)$] $G_0$ is a subgraph of $G$ such that $G[A_0]+G[B_0] \subseteq G_0$, $e_{G_0}(A',B') \leq \phi n$ and $d_{G_0}(v) = \phi n $ for all $v \in V_0$.
	\item[$ \rm (iv)$] Let $G^{\diamond} := G - G[A] - G[B] -G_0$.  $e_{G^\diamond} (A',B')$ is even and $(G^{\diamond}, \mathcal{P})$ is a $(K, m, \epszero,\eps)$-exceptional scheme.
\end{itemize}
Then $G^{\diamond}$ can be decomposed into edge-disjoint spanning subgraphs $H(i,i')$ and $H''(i,i')$ of $G^{\diamond}$ (for all $1 \le i,i' \le K$)%
	\COMMENT{AL: changed $i,i' \le K$ to  $1 \le i,i' \le K$.}
such that the following properties hold, where $G'(i,i'):=H(i,i')+H''(i,i')$:
\begin{itemize}
\item[\rm (b$_1$)] Each $H(i,i')$ contains only $A_0A_i$-edges and $B_0B_{i'}$-edges.
\item[\rm (b$_2$)] $H''(i,i')\subseteq G^{\diamond}[A',B']$. Moreover,
all but at most $\eps' n$ edges of $H''(i,i')$ lie in $G^{\diamond}[A_0 \cup A_i, B_0 \cup B_{i'}]$.
\item[\rm (b$_3$)] $e(H''(i,i'))$ is even and $ 2 \alpha n \le e(H''(i,i')) \le 11\epszero n^2/(10K^2)$.
\item[\rm (b$_4$)] $\Delta(H''(i,i')) \le 31 \alpha n/30 $.
\item[\rm (b$_5$)] $d_{G'(i,i')}(v )  =  \left( 2 \alpha \pm   \eps' \right) n $ for all $v \in V_0$.
\item[\rm (b$_6$)] Let $\widetilde{H}$ be any spanning subgraph of $H''(i,i')$ which maximises $e(\widetilde{H})$
under the constraints that $\Delta(\widetilde{H}) \le 3\gamma n /5$, $H''(i,i')[A_0,B_0] \subseteq \widetilde{H}$ and $e(\widetilde{H})$ is even.
Then $e(\widetilde{H}) \ge 2 \alpha n $.
\end{itemize}
\end{lemma}
\begin{proof}
Since $\phi  \ll 1/3\le D/n$, we deduce that%
    \COMMENT{$\gamma=\alpha -2\lambda/K^2\ge \alpha -2\lambda \cdot 7\alpha=(1-14\lambda)\alpha$}
\begin{align} \label{alphahier}
	\alpha \ge 1/(7K^2), \ \ \ \ \ (1-14\lambda)\alpha\le \gamma<\alpha \ \ \ \ \text{and} \ \ \ \ \eps \ll \eps'  \ll \lambda, 1/K,\alpha, \gamma \ll 1.
\end{align}
Note that (ii) and (iii) together imply that
\begin{align}	\label{alpha1}
e_{G^{\diamond}}(A',B') \ge D - \phi n \stackrel{(\ref{alpha})}{=}  2 K^2 \alpha n \stackrel{(\ref{alphahier})}{\ge } n/4.
\end{align}
By (i) and (iii), each $v \in V_0$ satisfies
\begin{equation}\label{eq:degGdiam}
d_{G^{\diamond}}(v) = D - \phi n \stackrel{(\ref{alpha})}{=} 2 K^2\alpha n.
\end{equation}
Apply Lemma~\ref{lma:randomslice} to decompose $G^{\diamond}$ into subgraphs $H(i,i')$, $H'(i,i')$ (for all $1\le i,i' \le K$)
satisfying the following properties, where $G(i,i'):=H(i,i')+H'(i,i')$:
\begin{itemize}
\item[(a$_1'$)] Each $H(i,i')$ contains only $A_0A_i$-edges and $B_0B_{i'}$-edges.
\item[(a$_2'$)] All edges of $H'(i,i')$ lie in $G^{\diamond}[A_0 \cup A_i, B_0 \cup B_{i'}]$.
\item[(a$_3'$)] $e ( H'(i,i') )   =   (1 \pm 16\eps)  e_{G^\diamond}(A',B') /K^2$.
	In particular, 
\begin{align*}
	 2 (1- 16\eps) \alpha n \le e(H'(i,i')) \le (1+16 \eps) \epszero n^2/K^2.
\end{align*}
\item[(a$_4'$)] $d_{H'(i,i')}(v )  =  ( d_{G^\diamond[A',B']}(v)  \pm 2 \eps n)/K^2$ for all $v \in V_0$.
\item[(a$_5'$)] $d_{G(i,i')}(v )  =  ( 2 \alpha  \pm 4 \eps/K^2)n$ for all $v \in V_0$.
\end{itemize}
Indeed, (a$'_3$) follows from (\ref{alpha1}), Lemma~\ref{lma:randomslice}(a$_3$) and~(ii),
while (a$'_5$) follows from (\ref{eq:degGdiam}) and Lemma~\ref{lma:randomslice}(a$_5$).
We now move some $A'B'$-edges of $G^\diamond$ between the $H'(i,i')$ such that the graphs $H''(i,i')$
obtained in this way satisfy the following conditions:
\begin{itemize}
\item Each $H''(i,i')$ is obtained from $H'(i,i')$ by adding or removing at most $32K^2 \eps \alpha n\le \sqrt{\eps}n$ edges.
\item $e(H''(i,i'))\ge 2 \alpha n$ and $e(H''(i,i'))$ is even.
\end{itemize}
Note that this is possible by~(a$_3'$) and since $\alpha n\in\mathbb{N}$ and $e_{G^\diamond} (A',B') \geq 2K^2\alpha n$ is even by~(iv).

We will show that the graphs $H(i,i')$ and $H''(i,i')$ satisfy conditions (b$_1$)--(b$_6$).
Clearly both (b$_1$) and (b$_2$) hold.
(a$_3'$) implies that
\begin{equation} \label{eH'}
e(H''(i,i')) = (1\pm 16 \eps) e_{G^{\diamond}}(A',B')/K^2\pm \sqrt{\eps}n \overset{(\ref{alphahier}),(\ref{alpha1}) }{=}
(1\pm \eps') e_{G^{\diamond}}(A',B')/K^2.
\end{equation}
Together with (ii) and our choice of the $H''(i,i')$ this implies (b$_3$). (b$_5$) follows from (a$'_5$)
and the fact that $d_{G'(i,i')}(v )  =d_{G(i,i')}(v ) \pm \sqrt{\eps}n$.
Similarly, (a$_4'$) implies that for all $v \in V_0$ we have
\begin{equation}\label{dH'}
d_{H''(i,i')}(v )  =  ( d_{G^{\diamond} [A',B']}(v)  \pm \eps' n)/{K^2}.
\end{equation}
Recall that $\Delta(G[A',B']) \le D/2$ by~(ii). Thus
$$
	\Delta(H''(i,i'))  \overset{\eqref{dH'}}{\le}   
\frac{D/2 + \eps' n   }{K^2}   \overset{\eqref{alpha}}{=}  \left( \alpha +  \frac{\phi + 2\eps' }{2K^2} \right) n 
 \overset{\eqref{alphahier}}{\le}  \frac{31 \alpha n }{30},
$$
so (b$_4$) holds.

So it remains to verify (b$_6$). To do this, fix $1 \le i, i' \le K$%
\COMMENT{AL: changed $i,i' \le K$ to  $1 \le i,i' \le K$.}
 and set $H'' :=H''(i,i')$.
Let $\widetilde{H}$ be a subgraph of $H''$ as defined in (b$_6$). We need to show that $e(\widetilde{H}) \ge 2 \alpha n$.
Suppose the contrary that $e(\widetilde{H}) < 2 \alpha n$.
We will show that this contradicts the assumption that $G$ is not critical.
Roughly speaking, the argument will be that if $\widetilde{H}$ is sparse, then so is~$H''$.
This in turn implies that $G^\diamond$ is also sparse, and thus any subgraph of $G[A',B']$ of comparatively small maximum degree is also sparse, 
which leads to a contradiction.

Let $X$ be the set of all those vertices $x$ for which $d_{\widetilde{H}}(x) \geq  3 \gamma n/ 5 -2$. So $X\subseteq V_0$ by~(iv) and (ESch3).
Note that if $X = \emptyset$, then  $\widetilde{H} = H''$ and%
   \COMMENT{If $\widetilde{H} \neq H''$ then $e(H'')-e(\widetilde{H}) \ge 2$ (as it is even). But the maximality of $\widetilde{H}$ implies that
 adding any two edges in $H''-\widetilde{H}$ to $\widetilde{H}$ would create a vertex
of degree $>3 \gamma n/ 5$.}
so $e(\widetilde{H}) \ge  2 \alpha n$ by (b$_3$). If $|X| \ge 4$, then $e(\widetilde{H}) \ge 4 (  3 \gamma n/ 5  -2)  -4\ge 2\alpha n $
by~\eqref{alphahier}. Hence $1\le |X| \le 3$. Note that $\widetilde{H}-X$ contains all but at most one edge from $H''-X$.%
   \COMMENT{Might have deleted one edge in $H''-X$ in order to ensure that $e(\widetilde{H})$ is even.}
Together with the fact that $\widetilde{H}[X]$ contains at most two edges (since $|X|\le 3$ and $\widetilde{H}$ is bipartite) this implies
that
\begin{align}\label{eq:tildeH}
2\alpha n  > e(\widetilde{H}) & \ge e(\widetilde{H}-X)+ \left( \sum_{x \in X} d_{\widetilde{H}}(x) \right) -2 \nonumber \\
& \ge e(H''-X)-1 +|X|(3 \gamma n/ 5  -2)-2\nonumber \\
& \ge e(H'')-\sum_{x \in X} d_{H''}(x)+|X|(3 \gamma n/ 5  -2)-3\nonumber \\
& =e(H'')-\sum_{x \in X} (d_{H''}(x)-3 \gamma n/ 5  +2)-3
\end{align}
and so
\begin{align}\label{eq:eH''}
e(H'')& \overset{(\ref{dH'})}{< } 2 \alpha n + \sum_{x \in X} \left( \frac{d_{G^{\diamond}[A',B']}(x)+\eps' n}{K^2} - 3 \gamma n /5 +2 \right) +3.
\end{align}
Note that (b$_4$) and (\ref{eq:tildeH}) together imply that if $e(H'')\ge 4\alpha n$ then
$e(\widetilde{H})\ge e(H'')-|X|(31\alpha n/30- 3 \gamma n/ 5  +2)-3\ge 2\alpha n$. Thus $e(H'')< 4\alpha n$ and by (\ref{eH'})
we have $e_{G^{\diamond}}(A',B')\le 4K^2\alpha n/(1-\eps')\le 5K^2\alpha n\le 3n$. Hence
\begin{align}\label{alpha3}
	e_{G^{\diamond}}(A',B') & \stackrel{(\ref{eH'})}{\le}  K^2 e(H'')+\eps' e_{G^{\diamond}}(A',B')\le K^2 e(H'')+3\eps'n\nonumber \\
& \stackrel{(\ref{eq:eH''})}{\le}
  D- \phi n + 7 \eps' n + \sum_{x \in X } \left( d_{G^{\diamond}[A',B']}(x) -  K^2(3\gamma n/5) \right).
\end{align}
Let $G'$ be any subgraph of $G^{\diamond}[A',B']$ which maximises $e(G')$ under the constraint that $\Delta(G') \le K^2(3 \gamma / 5 +2\eps' ) n$.
Note that if $d_{G^{\diamond}[A',B']}(v) \ge  K^2(3 \gamma / 5 +2 \eps' ) n $, then $v\in V_0$ (by~(iv) and (ESch3)) and so
$d_{H''}(v) > 3 \gamma n/ 5$ by~\eqref{dH'}. This in turn implies that $v \in X$.
Hence%
   \COMMENT{The +2 in the next inequality accounts for the edges in $G^{\diamond}[A'\cap X,B'\cap X]$.}
\begin{eqnarray}
	e(G') & \le & e_{G^{\diamond}}(A',B') -  \sum_{x \in X } \left( d_{G^{\diamond}[A',B']}(x) -  K^2( 3\gamma /5+ 2\eps' )n  \right)+2 \nonumber\\
	& \overset{\eqref{alpha3}}{\le} & D- \phi n + 7K^2\eps' n.\label{eq:eG'}
\end{eqnarray}
Note that (\ref{dH'}) together with the fact that $X\neq \emptyset$ implies that
$$\Delta(G[A',B'])\ge \Delta (G^\diamond[A',B'])\ge K^2(3 \gamma n/ 5 -2) -\eps' n\stackrel{(\ref{alpha}), (\ref{alphahier})}{\ge} 11D/40.$$
Since $G$ is not critical this means that
there exists a subgraph $G''$ of $G[A',B']$ such that $\Delta(G'') \le 11 D /40 \le   K^2( 3\gamma /5+ 2 \eps' )n  $ and $e(G'') \ge 41 D/40$.
Thus
\begin{align*}
	D - \phi n + 7K^2 \eps' n \stackrel{(\ref{eq:eG'})}{\ge} e(G') \ge e(G'') - e_{G_0}(A',B') \ge 41 D/40 - \phi n,
\end{align*}
which is a contradiction.
Therefore, we must have $e(\widetilde{H}) \ge 2 \alpha n$.
Hence (b$_6$) is satisfied.
\end{proof}


\subsection{Step $2$: Decomposing $H''(i,i')$ into Hamilton Exceptional System Candidates}

Our next aim is to decompose each $H''(i,i')$ into $\alpha n$%
	\COMMENT{$\alpha n = \frac{D- \phi n }{2K^2}$ as defined in Lemma~\ref{lma:move}.}
 Hamilton exceptional system candidates (this will follow from Lemma~\ref{lma:BESdecomprelim}).
Before we can do this, we need the following result on decompositions of bipartite graphs into `even matchings'.
We say that a matching is \emph{even} if it contains an even number of edges, otherwise it is \emph{odd}.

\begin{prop}\label{prop:evenmatching}
Suppose that $0<1/n \ll \gamma \le 1$ and that $n, \gamma n \in \mathbb N$.
Let $H$ be a bipartite graph on $n$ vertices with $\Delta(H) \le 2\gamma n/3$ and where $ e(H) \geq 2 \gamma n$ is even.
Then $H$ can be decomposed into $\gamma n$ edge-disjoint non-empty even matchings, each of size at most $3e(H)/(\gamma n)$.
\end{prop}

\begin{proof}
First note that since $e(H) \ge 2\gamma n $, it suffices to show that $H$ can be decomposed into at most $\gamma n $ edge-disjoint non-empty even matchings,
each of size at most $3e(H)/(\gamma n)$. Indeed, by splitting these matchings further if necessary, one can obtain precisely $\gamma n $
non-empty even matchings.

Set $n' := \lfloor 2\gamma n/3 \rfloor$.
K\"onig's theorem implies that $\chi '(H) \le n'$. So Proposition~\ref{prop:matchingdecomposition} implies that there
is a decomposition of $H$ into $n'$ edge-disjoint
matchings $M_1, \dots ,M_{n'}$ such that $|e(M_s) - e(M_{s'})| \le 1$ for all $s,s' \le n'$.
Hence we have 
\begin{align*}
2 \le \frac{e(H)}{n'} -1 \le e(M_s) \le \frac{e(H)}{n'}+1 \le \frac{3e(H)}{\gamma n}
\end{align*}
for all $s \le n'$.%
   \COMMENT{The final inequality holds since $ \frac{e(H)}{n'}+1\le  \frac{e(H)}{\gamma n/2}+1\le \frac{3e(H)}{\gamma n}$ since $\frac{e(H)}{\gamma n}\ge 2$.}
Since $e(H)$ is even, there are an even number of odd matchings.
Let $M_s$ and $M_{s'}$ be two odd matchings.
So $e(M_s),e(M_{s'}) \ge 3$ and thus there exist two disjoint edges $e \in M_s$ and $e' \in M_{s'}$.
Hence, $M_s - e$, $M_{s'} - e'$ and $\{e, e'\}$ are three even matchings.
Thus, by pairing off the odd matchings and repeating this process,  the proposition follows.
\end{proof}

\begin{lemma} \label{lma:BESdecomprelim}
Suppose that $0 < 1/n \ll \eps_0 \ll \gamma  < 1$, that $ \gamma + \gamma' < 1$ and that $n, \gamma n, \gamma' n \in \mathbb N$.
Let $H$ be a bipartite graph on $n$ vertices with vertex classes $A \dot\cup A_0$ and $ B  \dot\cup B_0$, where $|A_0|+|B_0| \le \eps_0 n$.
Suppose that
\begin{itemize}
	\item[{\rm (i)}] $e(H)$ is even, $\Delta(H) \le 16 \gamma n /15$ and $\Delta(H[A,B]) < (3 \gamma /5 - \epszero) n$. 
\end{itemize}
Let $H'$ be a spanning subgraph of $H$ which maximises $e(H')$ under the constraints  that $\Delta(H') \le 3 \gamma n /5$, $H[A_0,B_0] \subseteq H'$
and $e(H')$ is even. Suppose that
\begin{itemize}
\item[{\rm (ii)}] 
$2( \gamma +  \gamma' )n \le e(H') \le 10 \epszero  \gamma n^2$.	
\end{itemize}
Then there exists a decomposition of $H$ into edge-disjoint Hamilton exceptional system candidates $F_1, \dots ,F_{\gamma n},F'_1, \dots ,F'_{\gamma' n}$
with parameter $\eps_0$ such that $e(F'_s) = 2$ for all $s \le \gamma' n$.
\end{lemma}

Since we are in the non-critical case with many edges between $A'$ and $B'$, we will be able to assume that the subgraph $H'$
satisfies (ii).

Roughly speaking, the idea of the proof of Lemma~\ref{lma:BESdecomprelim} is to apply the previous proposition to decompose $H'$ into a suitable number of even matchings 
$M_i$ (using the fact that it has small maximum degree).%
\COMMENT{AL: rephrased the following sentences}
We then extend these matchings into Hamilton exceptional system candidates to cover all edges of $H$.
The additional edges added to each $M_i$ will be vertex-disjoint from $M_i$ and form vertex-disjoint 2-paths $uvw$ with $v \in V_0$.
So the number of connections from $A'$ to $B'$ remains the same (as $H$ is bipartite).
Each matching $M_i$ will already be a Hamilton exceptional system candidate, which means that $M_i$ and its extension will have the correct number of connections from $A'$ to $B'$ 
(which makes this part of the argument simpler than in the critical case).

\removelastskip\penalty55\medskip\noindent{\bf Proof of Lemma~\ref{lma:BESdecomprelim}. }
Set $A': = A_0 \cup A$ and $B':= B_0 \cup B$. We first construct the $F'_s$.
If $\gamma '=0$, there is nothing to do. So suppose that $\gamma ' >0$.
Note that each $F'_s$ has to be a matching of size~2 (this follows from the definition of a Hamilton exceptional system candidate and the
fact that $e(F'_s) = 2$).
Since $H'$ is bipartite and so
$$\frac{e(H')}{\chi '(H')}=\frac{e(H')}{\Delta(H')}\ge \frac{2( \gamma +  \gamma' )n}{3 \gamma n /5} > \frac{10}{3},$$
we can find a 2-matching $F'_1$ in $H'$. Delete the edges in $F'_1$ from $H'$ and choose another 2-matching $F'_2$.
We repeat this process until we have chosen $\gamma'n$ edge-disjoint $2$-matchings $F'_1, \dots ,F'_{\gamma'n}$.

We now construct  $F_1, \dots ,F_{\gamma n}$ in two steps: first we construct matchings $M_1,\break \dots,M_{\gamma n}$ in $H'$
and then extend each $M_i$ into the desired $F_i$.
Let $H_1$ and $H'_1$ be obtained from $H$ and $H'$ by removing all the edges in $F'_1, \dots ,F'_{\gamma'n}$.
So now $2\gamma n \le e(H'_1) \le 10 \epszero \gamma n^2 $ and both $e(H_1)$ and $e(H'_1)$ are even. Thus Proposition~\ref{prop:evenmatching}
implies that there is a decomposition of $H'_1$ into edge-disjoint non-empty even matchings $M_1, \dots ,M_{\gamma n }$,
each of size at most $30 \epszero  n$. 

Note that each $M_i$ is a Hamilton exceptional system candidate with parameter~$\eps_0$.
So if $H'_1 = H_1$, then we are done by setting $F_s: = M_s$ for each $s \le \gamma n$.
Hence, we may assume that $H'' := H_1 - H'_1=H-H'$ contains edges.
Let $X$ be the set of all those vertices $x \in A_0 \cup B_0$ for which $d_{H''}(x)> 0$.
Note that each $x\in X$ satisfies $N_{H''}(x)\subseteq A\cup B$ (since $H[A_0,B_0] \subseteq H'$).
This implies that
each $x \in X$ satisfies
$d_{H'}(x) \ge \lfloor  3\gamma n /5 \rfloor -1$ or $d_{H''}(x)=1$.
(Indeed, suppose that $d_{H'}(x) \le \lfloor  3\gamma n /5 \rfloor -2$ and $d_{H''}(x)\ge 2$.
Then we can move two edges incident to $x$ from $H''$ to $H'$.
The final assumption in (i) and the assumption on $d_{H'}(x)$ together imply that we would still have $\Delta(H') \le 3 \gamma n/5$, a contradiction.)
Since  $\Delta(H) \le 16 \gamma n /15$ by~(i) this in turn implies that
$d_{H''}(x) \le 7 \gamma n /15 +2$ for all $x \in X$.

Let $\mathcal{M}$ be a random subset of $\{M_1, \dots, M_{\gamma n}\}$ where each $M_i$ is chosen independently with probability $2/3$.
By Proposition~\ref{chernoff}, with high probability, the following assertions hold:
\begin{align}
	r : = |\mathcal{M}| & = (2/3 \pm \epszero){\gamma n} \nonumber \\
	|\{M_s \in \mathcal{M} : d_{M_s}(v) = 1 \}| & = 2 d_{H'_1}(v)/3  \pm \epszero{\gamma n}	& \text{for all } v \in V(H). \label{eqn:Mi}
\end{align}
By relabeling if necessary, we may assume that $\mathcal{M} = \{ M_1, M_2, \dots, M_r\}$.
For each $s \le r$, we will now extend $M_s$ to a Hamilton exceptional system candidate $F_s$ with parameter $\eps_0$ by adding edges from~$H''$.
Suppose that for some $1\le s\le r$ we have already constructed $F_1, \dots ,F_{s-1}$. Set $H''_s := H'' - \sum_{j < s} F_j$.
Let $W_s$ be the set of all those vertices $w \in X$ for which $d_{M_s}(w)=0$ and $d_{H''_s}(w) \ge 32\epszero n  \ge 2|A_0\cup B_0|+ e(M_s)$.
Recall that $X\subseteq A_0\cup B_0$ and $N_{H''_s}(w) \subseteq N_{H''}(w) \subseteq A \cup B$ for each $w\in X$ and thus also for
each $w \in W_s$. Thus there are $|W_s|$ vertex-disjoint 2-paths $uwu'$ with $w \in W_s$ and $u,u' \in N_{H''_s}(w) \setminus V(M_s)$.
Assign these 2-paths to $M_s$ and call the resulting graph $F_s$.
Observe that $F_s$ is a Hamilton exceptional system candidate with parameter $\eps_0$.
Therefore, we have constructed $F_1, \dots ,F_r$ by extending $M_1, \dots ,M_r$.

We now construct $F_{r+1}, \dots ,F_{\gamma n}$. For this, we first prove that the above construction implies that the current `leftover' $H_{r+1}''$
has small maximum degree. Indeed, 
note that if $w \in W_s$, then $d_{H''_{s+1}}(w) = d_{H''_{s}}(w) - 2$.
By~\eqref{eqn:Mi}, for each $ x \in X$, the number of $M_s \in \mathcal{M}$ with $d_{M_s}(x) = 0 $ is 
\begin{align}
r - |\{M_s \in \mathcal{M} : d_{M_s}(x) = 1 \}| 
& \ge 	(2/3-\epszero)\gamma n - (2d_{H'_1}(x)/3 + \epszero \gamma n)  \nonumber \\
& \ge 2\gamma n/3- 2d_{H'}(x)/3  -2 \epszero \gamma n  \nonumber \\
& \ge 	2\gamma n/3- 2/3\cdot \lfloor  3\gamma n /5 \rfloor   -2 \epszero \gamma n  \nonumber \\
& \ge  (4/15 - 2 \epszero)\gamma n > d_{H''}(x) /2 . \nonumber
\end{align}
Hence, we have $d_{H''_{r+1}}(x) <  32 \epszero n$ for all  $ x \in X$
(as we remove $2$ edges at $x$ each time we have $d_{M_s}(x)=0$ and $d_{H''_s}(x) \ge 32\epszero n$).
Note that by definition of $H'$, all but at most one edge in $H''$ must have an endpoint in $X$. 
So for $x \notin X$,  $d_{H''}(x) \le |X|+1 \leq |A_0\cup B_0|+1\le \epszero n+1$.
Therefore, $\Delta(H''_{r+1}) <  32 \epszero n$.

Let $H''':=H_1-(F_1+\dots +F_r) $.
So $H'''$ is the union of $H''_{r+1}$ and all the $M_s$ with $r < s \le \gamma n$.
Since each of $H_1$ and $F_1,\dots,F_r $ contains an even number of edges, $e(H''')$ is even.
In addition, $M_s \subseteq H'''$ for each $r< s \le \gamma n$, so $e(H''') \ge 2(\gamma n-r)$.
By~\eqref{eqn:Mi}, since $\Delta(H''_{r+1}) \le 32 \epszero n$, we deduce that for every vertex $v \in V(H''')$,
we have
\begin{align*}
	d_{H'''} (v) & \le \left( \frac{d_{H'_1}(v)}3+ \epszero\gamma n \right)  + \Delta(H''_{r+1})
\le \frac{3\gamma n/5}{3}+\eps_0 \gamma n + 32\epszero n 
\le \frac{2(\gamma n-r)}3
\end{align*}
In the second inequality, we used that $d_{H'_1}(v)\le d_{H'}(v)$.
Moreover, we have
$$e(H''')=e(H''_{r+1})+e(M_{r+1} + \dots +  M_{\gamma n})\le 32\epszero n^2+ 30\eps_0n (\gamma n-r)\le 62\epszero n^2.
$$
Thus, by Proposition~\ref{prop:evenmatching} applied with $H'''$ and $\gamma-r/n$ playing the roles of $H$ and~$\gamma$,
there exists a decomposition of $H'''$ into $\gamma n-r$ edge-disjoint non-empty even matchings $F_{r+1}$, $\dots$, $F_{\gamma n}$,
each of size at most $3e(H''')/(\gamma n-r)\le \sqrt{\eps_0}n/2$. Thus each such $F_s$ is 
a Hamilton exceptional system candidate with parameter $\eps_0$. This completes the proof.
\endproof

\subsection{Step $3$: Constructing the Localized Exceptional Systems}

The next lemma will be used to extend most of the exceptional system candidates guaranteed by Lemma~\ref{lma:BESdecomprelim}
into localized exceptional systems. These extensions are required to be `faithful' in the following sense.
Suppose that $F$ is an exceptional system candidate.%
	\COMMENT{Previsouly, $F$ is an $(i,i')$-ESC, but don't think we need it in the definition.}
Then $J$ is a \emph{faithful extension of} $F$ if the following holds:
\begin{itemize}
\item $J$ contains $F$ and $F[A',B']=J[A',B']$.
\item If $F$ is a Hamilton exceptional system candidate, then $J$ is a
Hamilton exceptional system and the analogue holds if $F$ is a matching exceptional system candidate.
\end{itemize}

\begin{lemma} \label{lma:BESextend2}
Suppose that $0 <  1/n  \ll \epszero \ll   1$, that $0 \le \gamma \le 1$ and that $n, K, m, \gamma n \in \mathbb{N}$.
Let $\mathcal{P}$ be a $(K, m , \epszero)$-partition of a set $V$ of $n$ vertices.
Let $1\le i,i'\le K$.
Suppose that $H$ and $F_1, \dots ,F_{\gamma n}$ are pairwise edge-disjoint graphs which satisfy the following conditions:
\begin{itemize}
\item[\rm (i)] $V(H)=V$ and $H$ contains only $A_0A_i$-edges and $B_0B_{i'}$-edges.
\item[\rm (ii)] Each $F_s$ is an $(i,i')$-ESC with parameter~$\eps_0$.
\item[\rm (iii)] Each $v \in V_0$ satisfies $d_{H+\sum F_s}(v )  \ge (2 \gamma  +\sqrt{\epszero} )n$.
\end{itemize}
Then there exist edge-disjoint $(i,i')$-ES $J_1, \dots ,J_{\gamma n }$ with parameter $\eps_0$ in $H +\sum F_s$ 
such that $J_s$ is a faithful extension of $F_s$
for all $s \le \gamma n$. 
\end{lemma}
\begin{proof}
For each $s \le \gamma n$ in turn, we extend $F_s$ into an $(i,i')$-ES $J_s$ with parameter $\eps_0$ in $H+\sum F_s$ such that
$J_s$ and $J_{s'}$ are edge-disjoint for all $s' < s$.
Since $H$ does not contain any $A'B'$-edges, the $J_s$ will automatically satisfy $J_s[A',B'] = F_s[A',B']$.
Suppose that for some $1\le s\le \gamma n$ we have already constructed $J_1, \dots ,J_{s-1}$.
Set $H_s := H - \sum_{s' < s} J_{s'}$. 
Consider any $v \in V_0$.
Since $v$ has degree at most 2 in an exceptional system and in an exceptional system candidate, (iii) implies that
\begin{align*}
	d_{H_s}(v) & \ge d_{H+ \sum F_s}(v) - 2 \gamma n \ge \sqrt{\epszero} n.
\end{align*}
Together with~(i) this shows that condition (ii) in Lemma~\ref{lma:ESextend} holds (with $H_s$ playing the role of $G$).
Since $\cP$ is a $(K,m,\eps_0)$-partition of $V$, Lemma~\ref{lma:ESextend}(i) holds too.
Hence we can apply Lemma~\ref{lma:ESextend} to obtain an exceptional system $J_s$ with parameter $\eps_0$ in $H_s + F_s$
such that $J_s$ is a faithful extension of $F_s$. (i) and (ii) ensure that $J_s$ is an $ (i,i')$-ES, as required.
\end{proof}

\subsection{Step $4$: Constructing the Remaining Exceptional Systems.}
Due to condition~(iii), Lemma~\ref{lma:BESextend2} cannot be used to extend \emph{all} the exceptional system candidates
returned by Lemma~\ref{lma:BESdecomprelim} into localized exceptional systems.
The next lemma will be used to deal with the remaining exceptional system candidates
(the resulting exceptional systems will not be localized).

\begin{lemma} \label{globalBES}
Suppose that $0 <  1/n  \ll \epszero \ll \eps' \ll \lambda \ll  1$ and that $n, \lambda n \in \mathbb{N}$.
Let $A,A_0,B,B_0$ be a partition of a set $V$ of $n$ vertices such that $|A_0| +|B_0| \le \epszero n$ and $|A| = |B|$.
Suppose that $H,F_1, \dots ,F_{\lambda n}$ are pairwise edge-disjoint graphs  which satisfy the following conditions:
\begin{itemize}
\item[{\rm (i)}] $V(H)=V$ and $H$ contains only $A_0A$-edges and $B_0B$-edges.
\item[{\rm (ii)}] Each $F_s$ is an exceptional system candidate with parameter $\eps_0$.
\item[{\rm (iii)}] For all but at most $\eps' n$ indices $s\le \lambda n$ the graph $F_s$ is either
a matching exceptional system candidate with $e(F_s) = 0$ or a Hamilton exceptional system candidate with $e(F_s) = 2$.
In particular, all but at most $\eps' n$ of the $F_s$ satisfy $d_{F_s}(v) \le 1$ for all $v \in V_0$.
\item[{\rm (iv)}] All $v \in V_0$ satisfy $d_{H+\sum F_s}(v )  = 2 \lambda n$.
\item[{\rm (v)}] All $v \in A \cup B $ satisfy $d_{H+\sum F_s}(v) \le 2 \epszero n$.
\end{itemize}
Then there exists a decomposition of $H+\sum F_s$ into edge-disjoint exceptional 
systems $J_1, \dots ,J_{\lambda n }$ with parameter $\eps_0$
such that $J_s$ is a faithful extension of $F_s$ for all $s \le \lambda n$. 
\end{lemma}

\begin{proof}
Let $V_0:=A_0 \cup B_0$ and let $v_1, \dots ,v_{|V_0|}$ denote the vertices of $V_0$. 
We will decompose $H$ into graphs $J'_s$ in such a way that
the graphs $J_s:=J'_s + F_s$ satisfy $d_{J_s}(v_i) = 2$ for all $i\le |V_0|$ and $d_{J_s}(v) \le 1$ for all $v \in A \cup B$.
Hence each $J_s$ will be an exceptional system with parameter~$\eps_0$.%
   \COMMENT{Each $J_s$ must be a path system since a cycle would have to lie in $J_s[V_0]$ and so by (i) already in $F_s[V_0]$.}
Condition (i) guarantees that $J_s$ will be a faithful extension of~$F_s$.
Moreover, the $J_s$ will form a decomposition of $H+\sum F_s$.
We construct the decomposition of $H$ by considering each vertex $v_i$ of $A_0 \cup B_0$ in turn.

Initially, we set $V(J'_s)=E(J'_s) = \emptyset$ for all $s \le \lambda n$.
Suppose that for some $1\le i\le |V_0|$ we have already assigned (and added) all the edges of $H$
incident with each of $v_1, \dots, v_{i-1}$ to the $J'_s$.
Consider $v_i$. Without loss of generality assume that $v_i \in A_0$. Note that $N_H(v_i)\subseteq A$ by~(i).
Define an auxiliary bipartite graph $Q_i$ with vertex classes $V_1$ and $V_2$ as follows:
$V_1:=N_{H} (v_i)$ and $V_2$ consists of $2-d_{F_s}(v_i)$ copies of $F_s$ for each $ s\le \lambda n $. 
Moreover, $Q_i$ contains an edge between $v \in V_1$ and $F_s \in V_2$ if and only if $v\notin V(F_s + J'_s)$.%
    \COMMENT{We can write this in this way since by def an exceptional system candidate does not contain
isolated vertices in $A\cup B$. If we change the def, then we here have to write that $v$ is not incident to an
edge of $F_s$.}

We now show that $Q_i$ contains a perfect matching. For this, 
note that $|V_1|=2\lambda n - d_{\sum F_s} (v_i)= |V_2|$ by~(iv). 
(v) implies that for each $v \in V_1 \subseteq A$ we have  $d_{\sum (F_s+J'_s)}(v) \le d_{H+\sum F_s}(v) \le 2 \epszero n$.
So $v$ lies in at most $2 \epszero n$ of the graphs $F_s+J'_s$.
Therefore, $d_{Q_i} (v) \geq |V_2| -4 \eps _0n \geq |V_2|/2$ for all $v \in V_1$. (The final
inequality follows since (iii) and (iv) together imply that
$d_H(v_i)=2\lambda n-d_{\sum F_s} (v_i)\ge 2\lambda n -(\lambda n-\eps' n)-2\eps' n\ge \lambda n/2$%
    \COMMENT{The $2\eps' n$ comes from the fact that $d_{F_s}(v_i)\le 2$ for every exceptional system candidate $F_s$.} 
and so $|V_2|=|V_1|\geq \lambda n/2$.)
On the other hand, since each $F_s+J'_s$ is an exceptional system candidate with parameter $\eps_0$, (ESC3) implies that
$|V(F_s+J'_s)\cap A| \le (\sqrt{\epszero}/2  +2 \epszero) n \le \sqrt{\epszero} n$ for each $F_s \in V_2$.
Therefore $d_{Q_i} (F_s) \geq |V_1| -|V(F_s+J'_s)\cap A| \geq |V_1|/2$ for each $F_s\in V_2$.
Thus we can apply Hall's theorem to find a perfect matching $M$ in $Q_i$.
Whenever $M$ contains an edge between $v$ and $F_s$, we add the edge $v_iv$ to $J'_s$.
This completes the desired assignment of the edges of $H$ at $v_i$ to the $J'_s$.
\end{proof}

\subsection{Proof of Lemma~\ref{lma:BESdecom}}

In our proof of Lemma~\ref{lma:BESdecom} we will use the following result, which is a consequence of
Lemmas~\ref{lma:BESextend2} and~\ref{globalBES}. 
Given a suitable set of exceptional system candidates in an exceptional scheme, the lemma extends these into exceptional systems which form a
decomposition of the exceptional scheme.
We prove the lemma in a slightly more general form than needed for the current case, 
as we will also use it 
in the other two cases.

\begin{lemma} \label{BEScons}
Suppose that $0 <  1/n  \ll \epszero \ll \eps\ll \eps' \ll \lambda, 1/K \ll  1$, that $1/(7K^2)\le \alpha < 1/K^2$
and that $n, K,m,\alpha n,\lambda n/K^2  \in \mathbb{N}$.
Let $$\gamma:=\alpha -\frac{\lambda}{K^2}\ \ \ \ \ \ \ \ \text{and} \ \ \ \ \ \ \ \ \gamma':=\frac{\lambda}{K^2}.$$
Suppose that the following conditions hold:
\begin{itemize}
\item[{\rm (i)}] $(G^*,\cP)$ is a $(K,m,\eps_0,\eps)$-exceptional scheme with $|G^*|=n$.
\item[{\rm (ii)}] $G^*$ is the edge-disjoint union of $H(i,i')$, $F_1(i,i'),\dots,F_{\gamma n}(i,i')$ and\break $F'_1(i,i'),\dots,F'_{\gamma' n}(i,i')$
over all $1\le i,i'\le K$. 
\item[{\rm (iii)}] Each $H(i,i')$ contains only $A_0A_i$-edges and $B_0B_{i'}$-edges.
\item[{\rm (iv)}] Each $F_s(i,i')$ is an $(i,i')$-ESC with parameter $\eps_0$.
\item[{\rm (v)}] Each $F'_s(i,i')$ is an exceptional system candidate with parameter $\eps_0$.
Moreover, for all but at most $\eps' n$ indices $s\le \gamma' n$ the graph $F'_s(i,i')$ is either a
matching exceptional system candidate with $e(F'_s(i,i'))=0$ or a Hamilton exceptional system candidate with $e(F'_s(i,i'))=2$.
\item[{\rm (vi)}] $d_{G^*}(v)=2K^2\alpha n$ for all $v\in V_0$.
\item[{\rm (vii)}] For all $1\le i,i'\le K$ let $G^*(i,i'):=H(i,i')+\sum_{s\le \gamma n} F_s(i,i')+$ \\$\sum_{s\le \gamma' n} F'_s(i,i')$.
Then $d_{G^*(i,i')}(v)=(2\alpha \pm \eps')n$ for all $v\in V_0$.
\end{itemize}
Then $G^*$ has a decomposition into $K^2\alpha n$ edge-disjoint exceptional systems
$$J_1(i,i'),\dots,J_{\gamma n}(i,i') \ \ \ \ \ \ \text{and} \ \ \ \ \ \ J'_1(i,i'),\dots,J'_{\gamma' n}(i,i')$$
with parameter $\eps_0$, where $1\le i,i'\le K$, such that $J_s(i,i')$ is an $(i,i')$-ES which is a
faithful extension of $F_s(i,i')$ for all $s \le \gamma n$
and $J'_s(i,i')$ is a faithful extension of $F'_s(i,i')$ for all $s \le \gamma' n$.%
\COMMENT{Added $J'_s(i,i')$ is a faithful extension of $F'_s(i,i')$ instead of $F'_s(i,i') \subseteq J'_s(i,i')$.}
\end{lemma}
\begin{proof}
Fix any $i,i'\le K$ and set $H:=H(i,i')$ and $F_s:=F_s(i,i')$ for all $s\le \gamma n$.
Our first aim is to apply Lemma~\ref{lma:BESextend2} in order to extend each of $F_1, \dots , F_{\gamma n}$
into a $(i,i')$-HES. (iii) and (iv) ensure that conditions~(i) and~(ii) of Lemma~\ref{lma:BESextend2} hold. To verify Lemma~\ref{lma:BESextend2}(iii),
note that by (v) and (vii) each $v\in V_0$ satisfies
\begin{align*}
d_{H + \sum F_s}(v) & = d_{G^*(i,i')}(v)-d_{\sum_s F'_s(i,i')}(v)\ge (2\alpha-\eps')n-(\gamma'-\eps')n-2\eps' n\\
& =(2\alpha-\gamma'-2\eps') n\ge (2\gamma+\sqrt{\eps_0})n.
\end{align*}
(Here the first inequality follows since (v) implies that $d_{F'_s(i,i')}(v)\le 1$ for all but at most $\eps' n$ indices $s\le \gamma' n$.)
Thus we can indeed apply Lemma~\ref{lma:BESextend2} to find edge-disjoint $(i,i')$-ES $J_1(i,i'), \dots ,J_{\gamma n}(i,i')$ with parameter $\eps_0$
in $H + \sum F_s$ such that $J_s(i,i')$ is a faithful extension of $F_s$ for all $s \le \gamma n$.
We repeat this procedure for all $1\le i,i' \le K$ to obtain $K^2\gamma n$
edge-disjoint (localized) exceptional systems. 

Our next aim is to apply Lemma~\ref{globalBES} in order to construct the $J'_s(i,i')$.
Let $H_0$ be the union of $H(i,i')-(J_1(i,i')+\dots+J_{\gamma n}(i,i'))$ over all $i,i'\le K$.
Relabel the $F'_s(i,i')$ (for all $s\le \gamma' n$ and all $i,i'\le K$) to obtain exceptional system candidates $F'_1,\dots,F'_{\lambda n}$.
Note that by~(vi) each $v\in V_0$ satisfies
\begin{equation}\label{eq:degH0}
d_{H_0+\sum F'_s}(v) =d_{G^*}(v)-2K^2 \gamma n= 2K^2\alpha n-2K^2 \gamma n=2\lambda n.
\end{equation}
Thus condition~(iv) of Lemma~\ref{globalBES} holds with $H_0, F'_s$ playing the roles of $H,F_s$.
(iii) and (v) imply that conditions (i)--(iii) of Lemma~\ref{globalBES} hold with $K^2 \eps'$ playing the role of $\eps'$.
To verify Lemma~\ref{globalBES}(v), note that
each $v\in A$ satisfies $d_{H_0+\sum F'_s}(v)\le d_{G^*}(v,A_0)+d_{G^*}(v,B')\le 2\eps_0 n$ by~(iii), (i) and (ESch3).
Similarly each $v\in B$ satisfies $d_{H_0+\sum F'_s}(v)\le 2\eps_0 n$. Thus we can apply
Lemma~\ref{globalBES} with $H_0, F'_s, K^2\eps'$ playing the roles of $H, F_s, \eps'$ to obtain
a decomposition of $H_0+\sum_s F'_s$ into
$\lambda n$ edge-disjoint exceptional systems $J'_1,\dots,J'_{\lambda n}$ with parameter $\eps_0$ such that
$J'_s$ is a faithful extension of $F'_s$ for all $s\le \lambda n$.
Recall that each $F'_s$ is a $F'_{s'}(i,i')$ for some $i,i'\le K$ and some $s'\le \gamma' n$.
Let $J'_{s'}(i,i'):=J'_s$. Then all the $J_s(i,i')$ and all the $J'_s(i,i')$ are as required in the lemma. 
\end{proof}

We now combine Lemmas~\ref{lma:move}, \ref{lma:BESdecomprelim} and~\ref{BEScons} in order
to prove Lemma~\ref{lma:BESdecom}.

\removelastskip\penalty55\medskip\noindent{\bf Proof of Lemma~\ref{lma:BESdecom}. }
Let $G^{\diamond}$ be as defined in Lemma~\ref{lma:BESdecom}(iv).
Choose a new constant $\eps'$ such that $\eps \ll \eps' \ll \lambda, 1/K$.
Set 
\begin{align} \label{alpha'}
2 \alpha n & := \frac{ D - \phi n }{K^2}, 
& \gamma_1 & : = \alpha -  \frac{2\lambda}{K^2}
& \text{and} &
& \gamma'_1 & : = \frac{2\lambda }{K^2}.
\end{align}
Similarly as in the proof of Lemma~\ref{lma:move}, since $\phi  \ll 1/3\le D/n$, we have
\begin{align} \label{alphahier2}
	\alpha \ge 1/(7K^2), \ \ \ \ \ (1-14\lambda)\alpha\le \gamma_1 <\alpha \ \ \ \ \  \text{and} \ \ \ \ \ \eps \ll \eps'  \ll \lambda, 1/K,\alpha, \gamma_1 \ll 1.
\end{align}
Apply Lemma~\ref{lma:move} with $\gamma_1$ playing the role of $\gamma$ in order to obtain a decomposition
of $G^{\diamond}$ into edge-disjoint spanning subgraphs $H(i,i')$ and $H''(i,i')$ (for all $1\le i,i' \le K$) which satisfy the following
properties, where $G'(i,i'):=H(i,i')+H''(i,i')$:
\begin{itemize}
\item[\rm (b$_1$)] Each $H(i,i')$ contains only $A_0A_i$-edges and $B_0B_{i'}$-edges.
\item[\rm (b$_2$)] $H''(i,i')\subseteq G^{\diamond}[A',B']$. Moreover,
all but at most $\eps' n$ edges of $H''(i,i')$ lie in $G^{\diamond}[A_0 \cup A_i, B_0 \cup B_{i'}]$.
\item[\rm (b$_3$)] $e(H''(i,i'))$ is even and $ 2 \alpha n \le e(H''(i,i')) \le 11\epszero n^2/(10K^2)$.
\item[\rm (b$_4$)] $\Delta(H''(i,i')) \le 31 \alpha n/30 $.
\item[\rm (b$_5$)] $d_{G'(i,i')}(v )  =  \left( 2 \alpha \pm   \eps' \right) n $ for all $v \in V_0$.
\item[\rm (b$_6$)] Let $\widetilde{H}$ any spanning subgraph of $H''(i,i')$ which maximises $e(\widetilde{H})$
under the constraints that $\Delta(\widetilde{H}) \le 3\gamma_1 n /5$, $H''(i,i')[A_0,B_0] \subseteq \widetilde{H}$ and $e(\widetilde{H})$ is even.
Then $e(\widetilde{H}) \ge 2 \alpha n $.
\end{itemize}
Fix any $1 \le i,i' \le K$.%
\COMMENT{AL: added $1 \le$}
 Set $H := H(i,i')$ and $H'' := H''(i,i')$. Our next aim is to 
decompose $H''$ into suitable `localized' Hamilton exceptional system candidates. For this, we will
apply 
Lemma~\ref{lma:BESdecomprelim} with $H'', \gamma_1,\gamma'_1$ playing the roles of $H, \gamma,\gamma'$.
Note that $\Delta(H'') \le 31 \alpha n/30 \le 16\gamma_1 n/15$ by (b$_4$) and (\ref{alphahier2}).
Moreover, $\Delta(H''[A,B])\le \Delta(G^\diamond[A,B])\le \eps_0 n$ by~(iv) and (ESch3). Since $e(H'')$ is even by (b$_3$), it follows that
condition~(i) of Lemma~\ref{lma:BESdecomprelim} holds. Condition~(ii) of Lemma~\ref{lma:BESdecomprelim}
follows from (b$_6$) and the fact that any $\widetilde{H}$ as in (b$_6$) satisfies
$e(\widetilde{H})\le e(H'') \le 11\eps_0 n^2/(10K^2)\le 10\eps_0\gamma_1 n^2$
(the last inequality follows from (\ref{alphahier2})). Thus we can indeed apply Lemma~\ref{lma:BESdecomprelim}
in order to decompose $H''$ into $ \alpha n $ edge-disjoint Hamilton exceptional system candidates $F_1, \dots , F_{\gamma_1 n}, F'_1, \dots , F'_{\gamma_1 ' n}$
with parameter $\eps_0$ such that $e(F'_s) =2$ for all $s \le \gamma_1' n$. 
Next we set
\begin{align*}
\gamma_2 & := \alpha - \frac{\lambda}{K^2} & 
& \text{and} &
\gamma'_2 & : = \frac{\lambda}{K^2}.
\end{align*}
Condition (b$_2$) ensures that by relabeling the $F_s$'s and $F'_s$'s we obtain $ \alpha n $ edge-disjoint Hamilton exceptional system
candidates $F_1(i,i'), \dots , F_{\gamma_2 n}(i,i'), F'_1(i,i'),$ $ \dots , F'_{\gamma_2 ' n}(i,i')$ with parameter $\eps_0$%
     \COMMENT{Note here that $\gamma_1$ and $\gamma_1'$ are replaced with $\gamma_2$ and $\gamma_2'$ respectively.}
such that properties~(a$'$) and (b$'$) hold:
\begin{itemize}
	\item[(a$'$)] $F_s(i,i')$ is an $(i,i')$-HESC for every $s \le \gamma_2 n$. Moreover, at least $\gamma_2 ' n$ of the $F_s(i,i')$ satisfy $e(F_s(i,i')) = 2$.
	\item[(b$'$)] $e(F'_s(i,i')) = 2$ for all but at most $ \eps' n$ of the $F'_s(i,i')$.
\end{itemize}
Indeed, we can achieve this by relabeling each $F_s$ which is a subgraph of $G^{\diamond}[A_0 \cup A_i, B_0 \cup B_{i'}]$
as one of the $F_{s'}(i,i')$ and each $F_s$ for which is not the case as one of the $F'_{s'}(i,i')$.

Our next aim is to apply Lemma~\ref{BEScons} with $G^\diamond,\gamma_2,\gamma_2'$ playing the roles of $G^*,\gamma,\gamma'$.
Clearly conditions (i) and (ii) of Lemma~\ref{BEScons} hold. (iii) follows from (b$_1$). (iv) and (v) follow from (a$'$) and (b$'$).
(vi) follows from Lemma~\ref{lma:BESdecom}(i),(iii). Finally, (vii) follows from (b$_5$) since $G'(i,i')$ plays the role of $G^*(i,i')$.
Thus we can indeed apply Lemma~\ref{BEScons} to obtain a decomposition of $G^\diamond$ into $K^2\alpha n$
edge-disjoint Hamilton exceptional systems $J_1(i,i'),\dots,J_{\gamma_2 n}(i,i')$ and $J'_1(i,i'),\dots,J'_{\gamma'_2 n}(i,i')$
with parameter $\eps_0$, where $1\le i,i'\le K$, such that $J_s(i,i')$ is an $(i,i')$-HES which is a faithful extension of  $F_s(i,i')$ for all $s \le \gamma_2 n$
and $J'_s(i,i')$ is a faithful extension of $F'_s(i,i')$ for all $s \le \gamma'_2 n$.
Then the set $\mathcal{J}$ of all these Hamilton exceptional systems is as required in Lemma~\ref{lma:BESdecom}.
\endproof


\section{Critical Case with $e(A',B') \ge D$} \label{sec:critical}

The aim of this section is to prove Lemma~\ref{lma:BESdecomcritical}. Recall that Lemma~\ref{lma:BESdecomcritical} gives a decomposition of the exceptional edges into exceptional systems in the critical case when $e(A',B') \ge D$.
The overall strategy for the proof  is similar to that of Lemma~\ref{lma:BESdecom}.
As before, it consists of four steps. 
In Step~1, we use Lemma~\ref{lma:movecritical} instead of Lemma~\ref{lma:move}. In Step~2, we use Lemma~\ref{lma:BESdecomprelim2} instead of
Lemma~\ref{lma:BESdecomprelim}. We still use Lemma~\ref{BEScons} which combines Steps~3 and~4.


\subsection{Step $1$: Constructing the Graphs $H''(i,i')$}
The next lemma is an analogue of Lemma~\ref{lma:move}. We will apply it with the graph $G^\diamond$ from Lemma~\ref{lma:BESdecomcritical}(iv)
playing the role of $G$. Note that instead of assuming that our graph $G$ given in Lemma~\ref{lma:BESdecomcritical} is critical, the
lemma assumes that $e_{G^{\diamond}}(A',B')\le 2n$. This is a weaker assumption, since if $G$ is critical, then
$e_{G^\diamond}(A',B')\le e_G(A',B') < n$ by Lemma~\ref{critical'}(iii). 
Using only this weaker assumption has the advantage that we can also apply the lemma in the proof of Lemma~\ref{lma:PBESdecom}, i.e.~the
case when $e_G(A',B') < D$. (b$_7$) is only used in the latter application.

\begin{lemma} \label{lma:movecritical}
Suppose that $0 <  1/n  \ll \epszero \ll  \eps \ll 1/K \ll 1$ and that $n, K,m\in \mathbb{N}$.
Let $(G, \mathcal{P})$ be a $(K, m, \epszero, \eps )$-exceptional scheme with $|G| = n$ and $e_G(A_0), e_G(B_0) $ $=0$.
Let $W_0$ be a subset of $V_0$ of size at most~$2$ such that for each $w \in W_0$, we have%
    \COMMENT{In general, $W_0 = \{ w_1, w_2 \}$, where $w_1$ and $w_2$ are the two vertices such that $d_{G[A',B']}(w_1) \ge  d_{G[A',B']}(w_2) \ge d_{G[A',B']}(v)$ for $v \in V(G)$.
By Lemma~\ref{lma:BESdecomcritical}(v), $d_{G^{\diamond}[A',B']}(w_1), d_{G^{\diamond}[A',B']}(w_2) \le (D- \phi n)/2$.
Recall that $e_G(A',B') \ge D$ and $e_{G_0}(A',B') \le \phi n$, so we have $d_{G^{\diamond}[A',B']}(w) \le e_{G^{\diamond}}(A',B')/2$.}
\begin{equation}\label{eq:degw}
K^2 \le d_{G[A',B']}(w) \le e_G(A',B')/2.
\end{equation}
Suppose that $e_G(A',B')\le 2n$ is even.
Then $G$ can be decomposed into edge-disjoint spanning subgraphs $H(i,i')$ and $H''(i,i')$ of $G$ (for all $1 \le i,i' \le K$)%
\COMMENT{AL: added $1 \le$}
such that the following properties hold, where $G'(i,i'):=H(i,i')+H''(i,i')$:
\begin{itemize}
\item[\rm (b$_1$)] Each $H(i,i')$ contains only $A_0A_i$-edges and $B_0B_{i'}$-edges.
\item[\rm (b$_2$)] $H''(i,i')\subseteq G[A',B']$. Moreover, all but at most  $20 \eps n/K^2$ edges of $H''(i,i')$ lie in $G[A_0 \cup A_i, B_0 \cup B_{i'}]$.
\item[\rm (b$_3$)] $e(H''(i,i')) =  2 \left\lceil e_G(A',B') / (2 K^2) \right\rceil$ or $e(H''(i,i')) =  2 \left\lfloor e_G(A',B') / (2 K^2) \right\rfloor$.
\item[\rm (b$_4$)] $d_{H''(i,i')}(v )  =  ( d_{G[A',B']}(v)  \pm 25 \eps n)/K^2$ for all $v \in V_0$.
\item[\rm (b$_5$)] $d_{G'(i,i')}(v )  =  \left( d_{G}(v) \pm   25 \eps n \right)  / K^2 $ for all $v \in V_0$.
\item[\rm (b$_6$)] Each $w \in W_0$ satisfies $d_{H''(i,i')}(w) = \lceil d_{G[A',B']}(w) / K^2 \rceil $ or 
$d_{H''(i,i')}(w) = \lfloor d_{G[A',B']}(w) / K^2 \rfloor $.
\item[\rm (b$_7$)] Each $w \in W_0$ satisfies $2 d_{H''(i,i')}(w) \le  e(H''(i,i'))$.%
    \COMMENT{This statement is used when $e_G(A',B') < D$.}
\end{itemize}
\end{lemma}
\begin{proof}
Since $e_G(A',B')$ is even, there exist unique non-negative integers $b$ and $q$ such that $e_G(A',B') = 2K^2 b+ 2q$ and $q < K^2$.
Hence, for all $1 \le i,i' \le K$,%
\COMMENT{AL: changed to  $1 \le i,i' \le K$}
 there are integers $b_{i,i'}\in \{2b, 2b+2\}$ such that $\sum_{i,i' \le K} b_{i,i'} = e_G(A',B')$.
In particular, the number of pairs $i,i'$ for which $b_{i,i'}=b+2$ is precisely~$q$.
We will choose the graphs $H''(i,i')$ such that $e(H''(i,i'))=b_{i,i'}$. (In particular, this will ensure that (b$_3$) holds.)
The following claim will help to ensure~(b$_6$) and~(b$_7$). 

\medskip

\noindent
\textbf{Claim.} \emph{For each $w \in W_0$ and all $i,i' \le K$ there is an integer  $a_{i,i'}=a_{i,i'}(w)$ which satisfies the following properties:
\begin{itemize}
	\item $a_{i,i'} = \lceil d_{G[A',B']}(w) /K^2 \rceil $ or $a_{i,i'} = \lfloor d_{G[A',B']}(w) /K^2 \rfloor$.
	\item $2a_{i,i'} \le b_{i,i'}$.
	\item $\sum_{i,i' \le K} a_{i,i'} = d_{G[A',B']} (w)$.
	\end{itemize}}
	
\smallskip

\noindent To prove the claim,
note that there are unique non-negative integers $a$ and $p$ such that $d_{G[A',B']} (w) = K^2 a + p $ and $p<K^2$.
Note that $a\ge 1$ by (\ref{eq:degw}). Moreover,
\begin{align}\label{eq:apbq}
 2(K^2 a +p) =   2 d_{G[A',B']} (w) \stackrel{(\ref{eq:degw})}{\le} e_G(A',B') = 2 K^2 b+ 2q.
\end{align}
This implies that $a\le b$. Recall that $b_{i,i'} \in \{2b,2b+2\}$. So if $b > a$, then the claim holds by
choosing any $a_{i,i'}\in \{a,a+1\}$ such that $\sum_{i,i' \le K} a_{i,i'} = d_{G[A',B']} (w)$.
Hence we may assume that $a = b$. Then~(\ref{eq:apbq}) implies that $p \le q$.
Therefore, the claim holds by setting $a_{i,i'}:= a+1$ for exactly $p$ pairs $i,i'$ for which $b_{i,i'} = 2b+2$ and
setting $a_{i,i'}:= a$ otherwise. This completes the proof of the claim.

\medskip

\noindent 
Apply Lemma~\ref{lma:randomslice} to decompose $G$ into subgraphs $H(i,i')$, $H'(i,i')$ (for all $i,i' \le K$) satisfying the following properties,
where $G(i,i') = H(i,i') + H'(i,i')$:
\begin{itemize}
\item[\rm (a$'_1$)] Each $H(i,i')$ contains only $A_0A_i$-edges and $B_0B_{i'}$-edges.
\item[\rm (a$'_2$)] All edges of $H'(i,i')$ lie in $G[A_0 \cup A_i, B_0 \cup B_{i'}]$.
\item[\rm (a$'_3$)] $e ( H'(i,i') )   =  ( e_{G}(A',B')  \pm  8 \eps n ) /K^2$.
\item[\rm (a$'_4$)] $d_{H'(i,i')}(v )  =  ( d_{G[A',B']}(v)  \pm 2 \eps n)/K^2$ for all $v \in V_0$.
\item[\rm (a$'_5$)] $d_{G(i,i')}(v )  =  ( d_{G}(v)  \pm 4 \eps n)/K^2$ for all $v \in V_0$.
\end{itemize}
Indeed, (a$'_3$) follows from Lemma~\ref{lma:randomslice}(a$_3$) and our assumption that $e_G(A',B') \le 2n$.

Clearly, (a$'_1$) implies that the graphs $H(i,i')$ satisfy (b$_1$). We will now move some $A'B'$-edges of $G$ between
the $H'(i,i')$ such that the graphs $H''(i,i')$ obtained in this way satisfy the following conditions:
\begin{itemize}
\item Each $H''(i,i')$ is obtained from $H'(i,i')$ by adding or removing at most $20\eps n/K^2$ edges of $G$.
\item $e(H''(i,i'))=b_{i,i'}$.
\item $d_{H''(i,i')}(w)=a_{i,i'}(w)$ for each $w\in W_0$, where $a_{i,i'}(w)$ are integers satisfying the claim.
\end{itemize}
Write $W_0=:\{w_1\}$ if $|W_0|=1$ and $W_0=:\{w_1,w_2\}$ if $|W_0|=2$.
If $W_0 \ne \emptyset$, then (a$'_4$) implies that $d_{H'(i,i')}(w_1 )  =   a_{i,i'}(w_1)  \pm (2 \eps n/K^2+1)$.
For each $i,i'\le K$, we add or remove at most $2 \eps n /K^2+1$ edges incident to $w_1$
such that the graphs $H''(i,i')$ obtained in this way satisfy $d_{H''(i,i')}(w_1) =  a_{i,i'}(w_1)$.
Note that since $a_{i,i'}(w_1) \ge \lfloor d_{G[A',B']}(w_1) /K^2 \rfloor \ge 1$ by~(\ref{eq:degw}), we can do this in such a way that
we do not move the edge $w_1w_2$ (if it exists).%
\COMMENT{This is not really needed.}
Similarly, if $|W_0|=2$, then for each $i,i'\le K$ we add or remove at most $2 \eps n /K^2+1$ edges incident to $w_2$
such that the graphs $H''(i,i')$ obtained in this way satisfy $d_{H''(i,i')}(w_2) =  a_{i,i'}(w_2)$.
As before, we do this in such a way that we do not move the edge $w_1w_2$ (if it exists).

Thus $d_{H''(i,i')}(w_1)  =  a_{i,i'}(w_1)$ and $d_{H''(i,i')}(w_2) =  a_{i,i'}(w_2)$
for all $1 \le i,i' \le K$%
\COMMENT{AL: added $1 \le$}
 (if $w_1,w_2$ exist).
In particular, together with the claim, this implies that $d_{H''(i,i')}(w_1), d_{H''(i,i')}(w_2) \le b_{i,i'}/2$. Thus the number of edges
of $H''(i,i')$ incident to $W_0$ is at most
\begin{align}
	\sum_{w \in W_0}d_{H''(i,i')}(w) \le b_{i,i'}.  \label{H''1}
\end{align}
(This holds regardless of the size of $W_0$.) On the other hand, (a$'_3$) implies that for all $i, i' \le K$ we have
\begin{align}
e ( H''(i,i') )   =  ( e_{G}(A',B')  \pm  8 \eps n ) /K^2 \pm 2(2 \eps n /K^2 + 1) = b_{i,i'} \pm 13 \eps n/K^2. \nonumber
\end{align}
Together with~\eqref{H''1} this ensures that we can add or delete at most $13 \eps n/K^2$ edges which do not intersect $W_0$
to or from each $H''(i,i')$ in order to ensure that $e(H''(i,i')) = b_{i,i'}$ for all $i,i' \le K$.
Hence, (b$_3$), (b$_6$) and (b$_7$) hold. Moreover,
\begin{equation}\label{eq:edgediff}
e(H''(i,i')-H'(i,i'))\le |W_0| (2\eps n/K^2 +1) + 13 \eps n/K^2 \le 20 \eps n/K^2.
\end{equation}
So (b$_2$) follows from (a$'_2$). Finally, (b$_4$) and (b$_5$) follow from (\ref{eq:edgediff}), (a$'_4$) and (a$'_5$).
\end{proof}


\subsection{Step $2$: Decomposing $H''(i,i')$ into Hamilton Exceptional System Candidates}

Before we can prove an analogue of Lemma~\ref{lma:BESdecomprelim}, we need the following result.
It will allow us to distribute the edges incident to the (up to three) vertices $w_i$ of high degree in $G[A',B']$ in a suitable way
among the localized Hamilton exceptional system candidates $F_j$.
The degrees of these high degree vertices $w_i$ will play the role of the $a_i$.
The $c_j$ will account for edges (not incident to $w_i$) which have already been assigned to the $F_j$.
(b) and (c) will be used to ensure  (ESC4), i.e.~that the total number of `connections' between $A'$ and $B'$ is even and positive.

\begin{lemma} \label{matrix}
Let $1 \le q \le 3$ and $0 \le \eta < 1$ and $r, \eta r \in \mathbb{N}$.
Suppose that $a_1, \dots, a_q \in \mathbb{N}$ and $c_1, \dots, c_r \in \{0,1,2\}$ satisfy the following conditions:
\begin{itemize}
	\item[\rm (i)] $c_1 \ge \dots \ge c_r \ge c_1-1 $.
	\item[\rm (ii)] $\sum_{i \le q} a_i + \sum_{j \le r} c_j = 2(1+\eta)r$.
	\item[\rm (iii)] $31 r /60  \le a_1, a_2 \le r$ and $31 r /60  \le a_3 \le 31r/30$.
\end{itemize}
Then for all $i \le q$ and all $j \le r$ there are $a_{i,j} \in \{0,1,2\}$ such that the following properties hold:
\begin{itemize}
	\item[\rm (a)] $\sum_{j \le r} a_{i,j} = a_i$ for all $i \le q$.
	\item[\rm (b)] $c_j + \sum_{i \le q} a_{i,j} = 4$ for all $j \le \eta r$ and $c_j + \sum_{i \le q} a_{i,j} = 2$ for all $\eta r < j \le r$.
	\item[\rm (c)] For all $j \le r$ there are at least $2- c_j$ indices $i \le q$ with $a_{i,j}=1$.
\end{itemize}
\end{lemma}
\begin{proof}
We will choose $a_{i,1}, \dots, a_{i,r}$ for each $i\le q$ in turn such that the following properties
($\alpha_i$)--($\rho_i$) hold, where we write $c_j^{(i)} : = c_j + \sum_{i' \le i} a_{i',j}$ for each $ 0 \le  i \le q$ (so $c_j^{(0)}=c_j$):
\begin{itemize}
\item[\rm ($\alpha_i$)] If $i \ge 1$ then $\sum_{j \le r} a_{i,j} = a_i$.
\item[\rm ($\beta_i$)] $4\ge c^{(i)}_1 \ge \dots \ge c^{(i)}_r$.
\item[\rm ($\gamma_i$)] If $\sum_{j\le r} c^{(i)}_j < 2r$, then $| c^{(i)}_j - c^{(i)}_{j'} | \le 1$ for all $j,j' \le r$.
\item[\rm ($\delta_i$)] If $\sum_{j\le r} c^{(i)}_j \ge 2r$, then $c^{(i)}_j \ge 2$ for all $j \le \eta r$ and $c^{(i)}_j = 2$ for all $\eta r< j \le r$.
\item[\rm ($\rho_i$)] If $1\le i\le q$ and $c^{(i-1)}_j < 2$ for some $j\le r$, then $a_{i,j}  \in \{ 0 , 1 \}$.
\end{itemize}
We will then show that the $a_{i,j}$ defined in this way are as required in the lemma.

Note that (i) and the fact that $c_1, \dots, c_r \in \{0,1,2\}$ together imply ($\beta_0$)--($\delta_0$).
Moreover, ($\alpha_0$) and ($\rho_0$) are vacuously true.
Suppose that for some $1\le i\le q$ we have already defined $a_{i',j}$ for all $i'<i$ and all $j\le r$ such that ($\alpha_{i'}$)--($\rho_{i'}$) hold.
In order to define $a_{i,j}$ for all $j\le r$, we distinguish the following cases.

\medskip 

\noindent\textbf{Case 1: $\sum_{j \le r} c^{(i-1)}_j \ge 2r$.}

\smallskip

\noindent Recall that in this case $c^{(i-1)}_j \ge 2$ for all $j \le r$ by ($\delta_{i-1}$).
For each $j\le r$ in turn we choose $a_{i,j}\in \{0,1,2\}$ as large as possible subject to the constraints
that
\begin{itemize}
\item $a_{i,j}+ c^{(i-1)}_j\le 4$ and
\item $\sum_{j'\le j} a_{i,j'}\le a_i$.
\end{itemize}
Since $c^{(i)}_j=a_{i,j}+ c^{(i-1)}_j$, ($\beta_i$) follows from ($\beta_{i-1}$) and our choice of the $a_{i,j}$.
($\gamma_i$) is vacuously true. To verify ($\delta_i$), note that $c_j^{(i)} \ge c_j^{(i-1)}\ge 2$ by ($\delta_{i-1}$).
Suppose that the second part of ($\delta_{i}$) does not hold, i.e.~that $c_{\eta n+1}^{(i)}>2$.
This means that $a_{i,\eta n+1}>0$.
Together with our choice of the $a_{i,j}$ this implies that $c_j^{(i)}=4$ for all $j\le \eta n$.
Thus%
   \COMMENT{We can't replace $\le$ by $=$ since at the moment we only know that $\sum_{j\le r} a_{i,j}\le a_i$.}
$$
2(1+\eta) r  =4\eta r+2(r-\eta r)< \sum_{j\le r} c^{(i)}_j=\sum_{j\le r} a_{i,j}+\sum_{i'<i} a_{i'}+ \sum_{j\le r} c_j
\le \sum_{i'\le i} a_{i'}+ \sum_{j\le r} c_j
$$
contradicting~(ii). Thus the second part of ($\delta_{i}$) holds too. Moreover, 
$c_{\eta n+1}^{(i)}=c_{\eta n+1}^{(i-1)}=2$ also means that $a_{i,\eta n+1}=0$.
So $\sum_{j'\le \eta n} a_{i,j'}= a_i$, i.e.~ ($\alpha_i$) holds. ($\rho_i$) is vacuously true since $c^{(i-1)}_j \ge 2$ by ($\delta_{i-1}$).

\medskip

\noindent\textbf{Case 2: $2r - a_i  \le \sum_{j \le r} c^{(i-1)}_j < 2r$.}

\smallskip

\noindent
If $i\in \{1,2\}$ then together with (iii) this implies that
\begin{equation}\label{eq:sum1st}
\sum_{j \le r} c^{(i-1)}_j\ge r\ge a_i.
\end{equation}
If $i=3$ then
\begin{equation}\label{eq:sum2nd}
\sum_{j \le r} c^{(i-1)}_j\ge \sum_{j\le r} \sum_{i' \le 2} a_{i',j}=a_1+a_2\ge \frac{31r}{30}\ge a_3
\end{equation}
 by~(iii). In particular, in both cases we have $\sum_{j \le r} c^{(i-1)}_j\ge r$.
Together with ($\gamma_{i-1}$) this implies that $c^{(i-1)}_j \in \{ 1 , 2\}$ for all $j \le r$.
Let $0\le r'\le r$ be the largest integer such that $c^{(i-1)}_{r'} = 2$. 
So $r' < r$ and $\sum_{j \le r} c^{(i-1)}_j  = r+r' $. Together with (\ref{eq:sum1st}) and (\ref{eq:sum2nd})
this in turn implies that $a_i \le r+r'$ (regardless of the value of~$i$). 

Set $a_{i,j}:=1$ for all $r'<j\le r$. Note that
$$\sum_{r'<j\le r} a_{i,j}=r-r'=2r-\sum_{j \le r} c^{(i-1)}_j \le a_i,$$
where the final inequality comes from the assumption of Case~2. Take $a_{i,1},\dots,a_{i,r'}$ to be a sequence of the form $2,\dots,2,0,\dots,0$
(in the case when $a_i-\sum_{r'<j\le r} a_{i,j}$ is even)
or $2,\dots,2,1,0,\dots,0$ (in the case when $a_i-\sum_{r'<j\le r} a_{i,j}$ is odd) which is chosen in such a way
that $\sum_{j\le r'} a_{i,j}=a_i-\sum_{r'<j\le r} a_{i,j} = a_i-r+r'$. This can be done since $a_i \le r+r'$ implies that the right hand side is at most $2r'$.

Clearly, ($\alpha_i$), ($\beta_i$) and ($\rho_i$) hold.
Since $\sum_{j\le r} c^{(i)}_j=a_i+\sum_{j\le r} c^{(i-1)}_j \ge 2r$ as we are in Case~2,
($\gamma_i$) is vacuously true. Clearly, our choice of the $a_{i,j}$ guarantees that $c^{(i)}_j\ge 2$ for all $j\le r$.
As in Case~1 one can show that%
    \COMMENT{If $c^{(i)}_{\eta r+1}> 2$ then we must have that $r'>\eta r$ and $c^{(i)}_j= 4$ for all $j\le \eta r$.
Now the same argument as in Case~1 gives a contradiction.}
$c^{(i)}_j= 2$ for all $\eta r<j\le r$. Thus ($\delta_i$) holds.

\medskip

\noindent\textbf{Case 3: $\sum_{j \le r} c^{(i-1)}_j < 2r - a_i$.}

\smallskip

\noindent
Note that in this case
$$
2r>\sum_{j \le r} c^{(i-1)}_j +a_i = \sum_{i' \le i} a_{i'} + \sum_{j \le r }c_j ,$$
and so $i < q$ by (ii). Together with (iii) this implies that $a_i \le r$. Thus for all $j \le r$
we can choose $a_{i,j} \in \{0,1\}$ such that ($\alpha_i$)--($\gamma_i$) and ($\rho_i$) are satisfied.
($\delta_i$) is vacuously true.

\medskip

This completes the proof of the existence of numbers $a_{i,j}$ (for all $i\le q$ and all $j\le r$)
satisfying ($\alpha_i$)--($\rho_i$). It remains to show that these $a_{i,j}$ are as required in the lemma.
Clearly, ($\alpha_1$)--($\alpha_q$) imply that (a) holds. Since
$c_j^{(q)}  = c_j + \sum_{i \le q} a_{i,j}$ the second part of (b) follows from ($\delta_q$).
Since $c_j^{(q)}\le 4$ for each $j\le \eta r$ by~($\beta_q$), together with (ii) this in turn implies that the
first part of (b) must hold too. If $c_j < 2$, then ($\rho_1$)--($\rho_q$) and (b) together imply that for at least
$2-c_j$ indices $i$ we have $a_{i,j} = 1$. Therefore, (c) holds.
\end{proof}


We can now use the previous lemma to decompose the bipartite graph induced by $A'$ and $B'$ into 
Hamilton exceptional system candidates. 

\begin{lemma} \label{lma:BESdecomprelim2}
Suppose that $0< 1/n \ll \eps_0 \ll \alpha < 1$, that $0\le  \eta < 199/200$ and
that%
   \COMMENT{Previously had $\eta \alpha n/200\in\mathbb{N}$. But in the proof of Lemma~\ref{lma:BESdecomcritical}
we cannot ensure that this holds, which is fine since we do not need it.}
$n, \alpha n/200, \eta \alpha n  \in \mathbb{N}$.
Let $H$ be a bipartite graph on $n$ vertices with vertex classes $A \dot\cup A_0$ and $ B  \dot\cup B_0$ where $|A_0|+|B_0| \le \eps_0 n$.
Furthermore, suppose that the following conditions hold:
\begin{itemize}
	\item[\rm (c$_1$)] $e(H) = 2(1+ \eta)\alpha n$.
	\item[\rm (c$_2$)] There is a set $W'\subseteq V(H)$ with $1 \le |W'| \le 3$%
\COMMENT{$W'$ will be the vertex set as defined in Lemma~\ref{critical'}.} 
	and such that 
\begin{align*}
\textrm{$e(H- W' ) \le 199\alpha n/100$ and $d_H(w) \ge 13 \alpha n / 25 $ for all $w \in W'$.}
\end{align*}
	\item[\rm (c$_3$)] There exists a set $W_0 \subseteq W'$ with $|W_0| = \min \{2, |W'| \}$ and such that $d_H(w)  \le \alpha n$ for all $w \in W_0$
and $d_H(w') \le 41 \alpha n / 40$ for all $w' \in W' \setminus W_0$.%
	\COMMENT{Later on, we will take $W_0 = \{w_1, w_2\}$, where $w_1$ and $w_2$ are the two vertices that we can bound the degree in $G^{\diamond}[A',B']$.}
	\item[\rm (c$_4$)] For all $w \in W'$ and all $v \in V(H) \setminus W'$ we have $d_H(w) - d_H(v) \ge \alpha n/150$.
	\item[\rm (c$_5$)] For all $v \in A \cup B$ we have $d_H(v) \le \epszero n$. 
\end{itemize}
Then there exists a decomposition of $H$ into edge-disjoint Hamilton exceptional system candidates $F_1, \dots, F_{\alpha n }$
such that $e(F_s) = 4$ for all $s \le \eta \alpha n$ and $e(F_s) = 2$ for all $\eta \alpha n < s \le \alpha n $.
Furthermore, at least $\alpha n/200$ of the $F_s$ satisfy the following two properties:
\begin{itemize}
\item $d_{F_s}(w) =1 $ for all $w \in W_0$,
\item $e(F_s) = 2$.
\end{itemize}
\end{lemma}

Roughly speaking, the idea of the proof is first to find the $F_s$ which satisfy the final two properties.
Let $H_1$ be the graph obtained from $H$ by removing the edges in all these $F_s$. We will
decompose $H_1-W'$ into matchings $M_j$ of size at most two.
Next, we extend these matchings into Hamilton exceptional system candidates $F_j$ using Lemma~\ref{matrix}.
In particular, if $e(M_j)<2$, then we will use one or more edges incident to $W'$ to ensure that
the number of $A'B'$-connections is positive and even, as required by (ESC4).
(Note that it does not suffice to ensure that the number of $A'B'$-edges is positive and even for this.)

\begin{proof}
Set $H':=H - W'$, $W_0 =: \{w_1, w_{|W_0|} \}$ and $W' =: \{ w_1, \dots, w_{|W'|} \}$.
Hence, if $|W'| = 3$, then $W' \setminus W_0 = \{ w_3 \}$. Otherwise $W'=W_0$.

We will first construct $e_{H}(W')$ Hamilton exceptional system candidates $F_s$, such that each of them is a matching
of size two and together they cover all edges in $H[W']$. So suppose that $e_{H}(W')>0$. Thus $|W'| =2$ or $|W'| = 3$. 
If $|W'| = 2$, let $f$ denote the unique edge in $H[W']$. Note that
$$e(H')  \ge e(H) - (d_H(w_1) + d_H(w_2) -1) \ge 2(1+\eta)\alpha n-(2\alpha n-1)\ge 1$$
by (c$_1$) and~(c$_3$). So there exists an edge $f'$ in $H'$. Therefore, $M'_1 : = \{f,f'\}$ is a matching.
If $|W'| = 3$, then $e_{H}(W') \le 2$ as $H$ is bipartite. Since by (c$_2$) each $w \in W'$ satisfies $d_H(w) \ge 13\alpha n / 25$, it
is easy to construct $e_H(W')$ $2$-matchings $M'_1,M'_{e_{H}(W')}$ such that $d_{M_s'}(w) = 1$ for all $w \in W'$ and all $s \le e_H(W')$
and such that $H[W'] \subseteq M'_1 \cup M'_{e_{H}(W')}$. Set $F_{ \alpha  n - s+1} := M'_s$ for all $s \le e_H(W')$ (regardless of the size of $W'$).

We now greedily choose $\alpha n/200-e_{H}(W')$ additional $2$-matchings $F_{199\alpha n/200+1},$ $\dots,F_{ \alpha n - e_{H}(W')}$ in $H$ which are
edge-disjoint from each other and from $F_{ \alpha n },$\\ $F_{ \alpha n - e_{H}(W') +1}$ and such that $d_{F_s}(w)=1$
for all $w\in W_0$ and all $199\alpha n/200<s\le  \alpha n - e_{H}(W')$. To see that this can be done, recall that by (c$_2$) we have
$d_H(w)\ge 13 \alpha n / 25$ for all $w\in W'$ (and thus for all $w\in W_0$) and that (c$_1$) and (c$_3$) together imply that
$e(H-W_0)\ge 2(1+ \eta)\alpha n-\alpha n>\alpha n$ if $|W_0|=1$.

Thus $F_{199\alpha n/200+1},\dots,F_{\alpha n}$ are Hamilton exceptional system candidates
satisfying the two properties in the `furthermore part' of the lemma.
Let $H_1$ and $H'_1$ be the graphs obtained from $H$ and $H'$ by deleting all the $\alpha n/100$ edges in these Hamilton exceptional system candidates.
Set
\begin{align}\label{eq:eta'}
r : = 199\alpha n/200 \ \ \ \ \ \  \text{and} \ \ \ \ \ \ \eta' := \eta \alpha n/r=200\eta/199.
\end{align}
Thus $0\le \eta'<1$ and we now have%
   \COMMENT{Note that $e(H)-\alpha n/100=2 ( 1 + \eta )\alpha n -\alpha n/100=2(199/200+\eta )\alpha n=2 ( 1 + \eta' )199\alpha n/200$.}
\begin{align}
	H_1[W'] & = \emptyset, 
	& e(H_1) & = e(H)-\alpha n/100=2 ( 1 + \eta' )r 
	& \text{and} &
	& e(H'_1) & \le 2 r.
	\label{eqn:H}
\end{align}
(To verify the last inequality note that $e(H'_1) \le e(H-W')\le 2 r$ by~(c$_2$).)
Also, (c$_2$) and (c$_4$) together imply that for all $w \in W'$ and all $v \in V(H) \setminus W'$ we have
\begin{align}
d_{H_1}(w) & \ge \alpha n/2\ge 4 \epszero n  & \text{and} & &d_{H_1}(w) - d_{H_1}(v) &\ge 2 \epszero n . \label{eqn:H2}
\end{align}
Moreover, by (c$_2$) and (c$_3$), each $w \in W_0$ satisfies
\begin{align}\label{eq:W'H1}
	31r/60 & \le 13\alpha n/25- \alpha n/200\le d_{H}(w)-d_{H-H_1}(w)=  d_{H_1}(w)\nonumber \\
 & \le \alpha n-\alpha n/200= r .
\end{align}
Similarly, if $|W'|=3$ and so $w_3$ exists, then
\begin{align}\label{eq:W'H1w3}
	31r/60 & \le 13\alpha n/25- \alpha n/200 \le d_{H}(w_3)-d_{H-H_1}(w_3)=d_{H_1}(w_3)\nonumber \\
& \le 41\alpha n/40\le 31r/30.
\end{align}
(\ref{eqn:H2}) and~(\ref{eq:W'H1}) together imply that $d_{H'_1}(v) \le d_{H_1}(v)< d_{H_1}(w_1)\le r $ for all $v\in V(H)\setminus W'$.
Thus $\chi'(H'_1)\le \Delta(H'_1)\le r$.
Together with Proposition~\ref{prop:matchingdecomposition} this implies that $H'_1$ can be decomposed into $r$ edge-disjoint matchings 
$M_1, \dots ,M_{r}$
such that $| m_j- m_{j'}| \le 1$ for all $1 \le j , j' \le r$, where we set $m_j := e(M_j)$.

Our next aim is to apply Lemma~\ref{matrix} with $|W'|$, $d_{H_1}(w_i)$, $m_j$, $\eta'$ playing the roles of $q$, $a_i$, $c_j$, $\eta$
(for all $i\le |W'|$ and all $j \le r$). Since $\sum_{j\le r} m_j=e(H'_1)\le 2r$ by (\ref{eqn:H}) and since
$| m_j- m_{j'}| \le 1$, it follows that $m_j\in \{0,1,2\}$ for all $j\le r$.
Moreover, by relabeling the matchings $M_j$ if necessary, we may assume that $m_1 \ge m_2 \ge \dots \ge m_r$. Thus
condition~(i) of Lemma~\ref{matrix} holds. (ii) holds too since $\sum_{i\le |W'|} d_{H_1}(w_i)+\sum_{j\le r} m_j=e(H_1)=2(1+\eta')r$
by~(\ref{eqn:H}). Finally, (iii) follows from (\ref{eq:W'H1}) and~(\ref{eq:W'H1w3}).
Thus we can indeed apply Lemma~\ref{matrix} in order to obtain numbers $a_{i,j}  \in \{0,1,2\}$ (for all $i \le |W'|$ and $ j \le r $)
which satisfy the following properties:
\begin{itemize}
\item[(a$'$)]	$\sum_{j \le r } a_{i,j} = d_{H_1}(w_i)$ for all $i\le |W'|$.
	\item[(b$'$)]	$ m_j + \sum_{i \le |W'|} a_{i,j} = 4 $ for all $j \le \eta' r$ and
	         $m_j + \sum_{i \le |W'|} a_{i,j} = 2$ for all $\eta' r < j \le  r$.
\item[(c$'$)] 	If $m_j < 2$ then there exist at least $2 - m_j$ indices $i$ such that $a_{i,j} = 1$.
\end{itemize}
For all $j\le r$, our Hamilton exceptional system candidate $F_j$ will consist of the edges in $M_j$
as well as of $a_{i,j}$ edges of $H_1$ incident to $w_i$ (for each $i\le |W'|$).   
So let $F_j^0: = M_j$ for all $j \le r$. For each $i =1,\dots, |W'|$ in turn, we will now assign the edges of $H_1$ incident with $w_i$
to $F_1^{i-1},\dots,F_r^{i-1}$ such that the resulting graphs $F_1^{i},\dots,F_r^{i}$ satisfy the following properties: 
\begin{itemize}
	\item[($\alpha_i$)] If $i \ge 1$, then $e(F_j^i) - e(F_j^{i-1}) = a_{i,j}$.
	\item[($\beta_i$)] $F^i_j$ is a path system. Every vertex $v \in A \cup B$ is incident to at most one edge of $F_j^i$.
For every $v \in V_0 \setminus W'$ we have $d_{F_j^i}(v) \le 2$. If $e(F_j^i) \le 2$, we even have $d_{F_j^i}(v) \le 1$.
    \item[($\gamma_i$)] Let $b_{j}^i$ be the number of vertex-disjoint maximal paths in $F_j^i$ with one endpoint in $A'$ and the other in $B'$.
If $a_{i,j}=1$ and $i \ge 1$, then $b_j^i = b_j^{i-1}+1$. Otherwise $b_j^i = b_j^{i-1}$. 
\end{itemize}
We assign the edges of $H_1$ incident with $w_i$ to $F_1^{i-1},\dots,F_r^{i-1}$ in two steps.
In the first step, for each index $j\le r$ with $a_{i,j}=2$ in turn, we assign an edge of $H_1$ between $w_i$ and $V_0$ to
$F_j^{i-1}$ whenever there is such an edge left. More formally, to do this, we set $N_0 := N_{H_1}(w_i)$.
For each $j \le r$ in turn, if $a_{i,j} = 2$ and $N_{j-1} \cap V_0 \ne \emptyset$, then we choose a vertex
$v \in N_{j-1} \cap V_0$ and set $F_j':= F_j^{i-1} + w_i v$, $N_j: = N_{j-1} \setminus \{v\}$ and $a'_{i,j} := 1$.
Otherwise, we set $F_j':= F_j^{i-1}$, $N_j:= N_{j-1}$ and $a'_{i,j} := a_{i,j}$.

Therefore, after having dealt with all indices $j\le r$ in this way, we have that
\begin{align}
\textrm{either $a'_{i,j} \le 1$ for all $j \le r$ or $ N_r \cap V_0 = \emptyset$ (or both).} \label{V0cond}
\end{align}
Note that by (b$'$) we have $e(F_j') \le m_j + \sum_{i'\le i} a_{i',j} \le 4$ for all $j\le r$.
Moreover, (a$'$) implies that $|N_r| = \sum_{j \le r}  a'_{i,j}$.
Also, $N_r  \setminus V_0=N_{H_1}(w_i) \setminus V_0$, and so $N_{H_1}(w_i) \setminus N_r  \subseteq V_0$.
Hence
\begin{align}
|N_r| & = 
| N_{H_1}(w_i) | - | N_{H_1}(w_i) \setminus N_r | 
\ge d_{H_1}(w_i) - |V_0| \ge d_{H_1}(w_i) - \epszero n. 
\label{v'1}
\end{align}

In the second step, we assign the remaining edges of $H_1$ incident with $w_i$ to $F'_1,\dots,F'_r$.
We achieve this by finding a perfect matching $M$ in a suitable auxiliary graph.

\medskip

\noindent
{\bf Claim.} \emph{Define a graph $Q$ with vertex classes $N_r$ and $V'$ as follows:
$V'$ consists of $a'_{i,j}$ copies of $F_j'$ for each $ j \le r$.
$Q$ contains an edge between $v \in N_r$ and $F'_j \in V'$ if and only $v$ is not an endpoint of an edge in $F'_j$.
Then $Q$ has a perfect matching~$M$.}

\smallskip

\noindent
To prove the claim, note that 
\begin{align}\label{eqn:V1=V2}
	|V'| = \sum_{j \le r} a'_{i,j} = |N_r|  \overset{\eqref{v'1}}{\ge} d_{H_1}(w_i) - \epszero n.  
\end{align}
Moreover, since $F'_j\subseteq H$ is bipartite and so every edge of $F'_j$ has at most one endpoint in $N_r$,
it follows that
\begin{align}
d_{Q} (F'_j) \geq |N_r| - e(F_j')  \ge |N_r| -4 \label{eqn:dQF_j}
\end{align}
for each $F_j' \in V'$.
Consider any $v \in N_r$. Clearly, there are at most $d_{H_1}(v)$ indices $j\le r$ such that $v$ is an endpoint of an
edge of $F_j'$. If $v \in N_r \setminus V_0 \subseteq A \cup B$, then by (c$_5$), $v$ lies in at most
$2 d_{H_1}(v) \le 2 d_{H}(v)\le 2 \epszero n$ elements of $V'$. (The factor~2 accounts for the fact that each $F'_j$ occurs in
$V'$ precisely $a'_{i,j}\le 2$ times.) So
$$
d_{Q} (v) \geq |V'| - 2\epszero n  \overset{\eqref{eqn:V1=V2}}{\geq} d_{H_1}(w_i)- 3\epszero  n  \overset{\eqref{eqn:H2}}{\geq} \epszero n.
$$
If $v \in N_r \cap V_0$, then \eqref{V0cond} implies that $a'_{i,j} \le  1$ for all $j\le r$.
Thus
\begin{align*}
d_{Q} (v) \geq |V'| - d_{H_1}(v) \overset{\eqref{eqn:V1=V2}}{\geq} ( d_{H_1}(w_i) - d_{H_1}(v) )- \epszero n
\overset{\eqref{eqn:H2}}{\geq} 2 \epszero n  - \epszero n = \epszero n .
\end{align*}
To summarize, for all $v \in N_r$ we have $d_{Q} (v) \geq \epszero n $.
Together with \eqref{eqn:dQF_j} and the fact that $|N_r| = |V'|$ by~\eqref{eqn:V1=V2}
this implies that $Q$ contains a perfect matching $M$ by Hall's theorem.
This proves the claim.

\medskip

For each $j\le r$, let $F^i_j$ be the graph obtained from $F'_j$ by adding the edge $w_iv$ whenever the perfect matching~$M$ 
(as guaranteed by the claim) contains
an edge between $v$ and $F'_j$.

Let us now verify ($\alpha_i$)--($\gamma_i$) for all $i \le |W'|$. Clearly, ($\alpha_0$)--($\gamma_0$) hold and $b_j^0=m_j$.
Now suppose that $i \ge 1$ and that ($\alpha_{i-1}$)--($\gamma_{i-1}$) hold.
Clearly, ($\alpha_i$) holds by our construction of $F^i_1,\dots,F^i_r$. 
Now consider any $j\le r$. If $a_{i,j} = 0$, then ($\beta_i$) and ($\gamma_i$) follow from ($\beta_{i-1}$) and ($\gamma_{i-1}$).
If $a_{i,j} = 1$, then the unique edge in $F^i_j-F_j^{i-1}$ is vertex-disjoint from any edge of $F_j^{i-1}$ (by the definition of $Q$)
and so ($\beta_i$) holds.
Moreover, $b_j^i = b_j^{i-1}+1$ and so ($\gamma_i$) holds. 
So suppose that $a_{i,j} = 2$. Then the unique two edges in $F_j^{i} - F_j^{i-1}$ form a path $P=v'w_iv''$ of length two with internal vertex~$w_i$.
Moreover, at least one of the edges of $P$, $w_iv''$ say, was added to $F_j^{i-1}$ in the second step of our construction of $F^i_j$.
Thus $d_{F^i_j}(v'')=1$. The other edge $w_iv'$ of $P$ was either added in the first or in the second step.
If $w_iv'$ was added in the second step, then $d_{F^i_j}(v')=1$. Altogether this shows that in this case
($\gamma_i$) holds and ($\beta_i$) follows from ($\beta_{i-1}$).
So suppose that $w_iv'$ was added to $F_j^{i-1}$ in the first step of our construction of $F^i_j$.
Thus $v'\in V_0\setminus W'$. But since $a_{i,j} = 2$, (b$'$) implies that $e(F_{j}^{i-1}) = m_j + \sum_{i' < i } a_{i',j} \le 2$.
Together with ($\beta_{i-1}$) this shows that $d_{F_{j}^{i-1}}(v) \le 1$ for all $v \in V_0 \setminus W'$.
Hence $d_{F_{j}^{i-1}}(v') \le 1$ and so $d_{F_{j}^{i}}(v') \le 2$. Together with ($\beta_{i-1}$) this implies ($\beta_i$).
(Note that if $e(F_j^{i-1})=0$, then the above argument actually shows that $d_{F_j^i(v')} \le 1$, as required.)
Moreover, the above observations also guarantee that ($\gamma_i$) holds. Thus $F^i_1,\dots,F^i_r$ satisfy ($\alpha_i$)--($\gamma_i$).

After having assigned the edges of $H_1$ incident with $w_i$ for all $i\le |W'|$, we have obtained graphs
$F^{|W'|}_1,\dots,F^{|W'|}_r$. Let $F_j:=F^{|W'|}_j$ for all $j\le r$. Note that by ($\gamma_{|W'|}$) for all $j\le r$
the number of vertex-disjoint maximal $A'B'$-paths in $F_j$ is precisely $b^{|W'|}_j$.

We now claim that $b_j^{|W'|}$ is positive and even.
To verify this,  recall that $b_j^0=m_j$. Let ${\rm odd}_j$ be the number of $a_{i,j}$ with $a_{i,j}=1$ and $i \le |W'|$.
So $b_j^{|W'|}=m_j+{\rm odd}_j$.
Together with (c$'$) this immediately implies that $b_j^{|W'|} \ge 2$.
Moreover,  since $a_{i,j} \in \{0,1,2 \}$ we have
$$
b_j^{|W'|}=m_j+{\rm odd}_j=m_j+\sum_{i \le |W'|,\ a_{i,j} {\rm \ is\ odd}} a_{i,j}.
$$
Together with (b$'$) this now implies that $b_j^{|W'|}$ is even. This proves the claim.

Together with (a$'$), (b$'$) and ($\alpha_i$), ($\beta_i$)
for all $i\le |W'|$ this in turn shows that $F_1,\dots,F_r$ form a decomposition of $H_1$
into edge-disjoint Hamilton exceptional system candidates with
$e(F_j)=4$ for all $j\le \eta' r$ and $e(F_j)=2$ for all $\eta' r<j\le r$. Recall that $\eta' r=\eta \alpha n$
by~(\ref{eq:eta'}) and that we have already constructed Hamilton exceptional system candidates $F_{199\alpha n/200+1},\dots,F_{ \alpha n}$
which satisfy the `furthermore statement' of the lemma, and thus in particular consist of precisely two edges.
This completes the proof of the lemma.
\end{proof}

\subsection{Proof of Lemma~\ref{lma:BESdecomcritical}}

We will now combine Lemmas~\ref{lma:movecritical},~\ref{lma:BESdecomprelim2} and~\ref{BEScons} in order
to prove Lemma~\ref{lma:BESdecomcritical}.
This will complete the construction of the required exceptional sequences in the case when $G$ is both critical and $e(G[A',B'])\ge D$.

\removelastskip\penalty55\medskip\noindent{\bf Proof of Lemma~\ref{lma:BESdecomcritical}. }
Let $G^{\diamond}$ be as defined in Lemma~\ref{lma:BESdecomcritical}(iv).
Our first aim is to decompose $G^{\diamond}$ into suitable `localized' subgraphs via Lemma~\ref{lma:movecritical}.
Choose a new constant $\eps'$ such that $\eps \ll \eps' \ll \lambda, 1/K$ and
define $\alpha$ by
\begin{equation} \label{alphaeqD}
2 \alpha n  :=\frac{ D - \phi n}{K^2}.
\end{equation}
Recall from Lemma~\ref{lma:BESdecomcritical}(ii) that $D = (n-1)/2$ or $D = n/2-1$.
Together with our assumption that $\phi \ll 1$ this implies that%
   \COMMENT{$\alpha= \frac{ D/n - \phi}{2K^2}\ge \frac{ 1/2-1/n - \phi}{2K^2}=\frac{ 1-2/n - 2\phi}{4K^2}$.}
\begin{equation}\label{alphahier3}
\frac{ 1-2/n - 2\phi}{4K^2}\le \alpha \le  \frac{ 1 - 2\phi}{4K^2} \ \ \ \  \ \ \ \ \text{and} \ \ \ \  \ \ \ \ \eps \ll \eps'  \ll \lambda, 1/K, \alpha \ll 1.
\end{equation}
Note that by Lemma~\ref{lma:BESdecomcritical}(ii) and (iii) we have $ e_{G^{\diamond}}(A',B') \ge D - \phi n =  2 K^2 \alpha n$.
Together with Lemma~\ref{critical'}(iii) this implies that
\begin{align}
	2 K^2 \alpha n \le  e_{G^{\diamond}}(A',B') \le e_G(A',B') \le 17D/10+5
\stackrel{(\ref{alphaeqD})}{\le} {18 K^2 \alpha n}/{5} \stackrel{(\ref{alphahier3})}{<} n . \label{eW'}
\end{align}
Moreover, recall that by Lemma~\ref{lma:BESdecomcritical}(i) and (iii) we have
\begin{equation}\label{eq:degvV0}
d_{G^{\diamond}}(v) = 2K^2 \alpha n \ \ \ \  \ \ \text{for all } v \in V_0.
\end{equation}
Let $W$ be the set of all those vertices $w\in V(G)$ with $d_{G[A',B']}(w)\ge 11D/40$. So $W$ is as defined in Lemma~\ref{critical'}
and $1\le |W|\le 3$ by Lemma~\ref{critical'}(i). Let $W'\subseteq V(G)$ be as guaranteed by Lemma~\ref{critical'}(v). Thus $W\subseteq W'$, $|W'|\le 3$,
\begin{align}\label{eq:degrees}
d_{G[A',B']}(w')  \ge \frac{21D}{80}, \ 
d_{G[A',B']}(v) \le \frac{11D}{40} \ \
{\rm and} \ \
d_{G[A',B']}(w') - d_{G[A',B']}(v) \ge \frac{D}{240}.
\end{align}
for all $w'\in W'$ and all $v\in V(G)\setminus W'$. In particular, $W'\subseteq V_0$. (This follows since
Lemma~\ref{lma:BESdecomcritical}(iii),(iv) and (ESch3) together imply that
$d_{G[A',B']}(v) = d_{G^\diamond [A',B']}(v)+d_{G_0[A',B']}(v)\le \eps_0 n+e_{G_0}(A',B')\le \eps_0 n+\phi n$ for all $v\in A\cup B$.)
Let $w_1, w_2, w_3$ be vertices of $G$ such that
$$d_{G[A',B']}(w_1) \ge d_{G[A',B']}(w_{2}) \ge d_{G[A',B']}(w_{3}) \ge d_{G[A',B']}(v)$$ for
all $v \in V(G) \setminus \{w_1,w_2,w_3\}$, where $w_1$ and $w_2$ are as in Lemma~\ref{lma:BESdecomcritical}(v).
Hence $W$ consists of $w_1, \dots, w_{|W|}$ and $W'$ consists of $w_1, \dots, w_{|W'|}$.
Set $W_0 : = \{w_1,w_2\} \cap W'$. Since $d_{G_0}(v)=\phi n$ for each $v\in V_0$ (and thus for each $v\in W_0$), each $w \in W_0$ satisfies%
\COMMENT{Previously, we had $15 K^2 \alpha n /38 \stackrel{(\ref{alphaeqD})}{\le} 21D/80-\phi n $, but the first inequality can be replaced by $K^2$,
as this is what we actually use in Lemma~\ref{lma:movecritical}}
\begin{align}
\label{degW0} 
K^2 {\le} 21D/80-\phi n \stackrel{(\ref{eq:degrees})}{\le}d_{G^{\diamond}[A',B']}(w) \le K^2 \alpha n \stackrel{(\ref{eW'})}{\le} e_{G^{\diamond}}(A',B')/2.
\end{align}
(Here the third inequality follows from Lemma~\ref{lma:BESdecomcritical}(v).)
Apply Lemma~\ref{lma:movecritical} to $G^{\diamond}$ in order to obtain a decomposition
of $G^{\diamond}$ into edge-disjoint spanning subgraphs $H(i,i')$ and $H''(i,i')$ (for all $1\le i,i' \le K$) which satisfy the following
properties, where $G'(i,i'):=H(i,i')+H''(i,i')$:
\begin{itemize}
\item[\rm (b$'_1$)] Each $H(i,i')$ contains only $A_0A_i$-edges and $B_0B_{i'}$-edges.
\item[\rm (b$'_2$)] $H''(i,i')\subseteq G^{\diamond}[A',B']$.
Moreover, all but at most  $20 \eps n/K^2$ edges of $H''(i,i')$ lie in $G^{\diamond}[A_0 \cup A_i, B_0 \cup B_{i'}]$.
\item[\rm (b$'_3$)] $e(H''(i,i')) =  2 \left\lceil e_{G^{\diamond}}(A',B') / (2 K^2) \right\rceil$ or $e(H''(i,i')) =  2 \lfloor e_{G^{\diamond}}(A',B') / \break (2 K^2) \rfloor$.
In particular, $2\alpha n \le e(H''(i,i')) \le 19 \alpha n /5$ by~\eqref{eW'}.
\item[\rm (b$'_4$)] $d_{H''(i,i')}(v )  =  ( d_{G^{\diamond}[A',B']}(v)  \pm 25 \eps n)/K^2$ for all $v \in V_0$.
\item[\rm (b$'_5$)] $d_{G'(i,i')}(v )  =  ( d_{G^\diamond}(v)  \pm 25 \eps n)/K^2=\left( 2 \alpha  \pm   25 \eps/K^2 \right) n$ for all $v \in V_0$ by~\eqref{eq:degvV0}.
\item[\rm (b$'_6$)] Each $w \in W_0$ satisfies $d_{H''(i,i')}(w) \le \lceil d_{G{^{\diamond}}[A',B']}(w) / K^2 \rceil \le \alpha n $ by~\eqref{degW0}.
\end{itemize}
Our next aim is to apply Lemma~\ref{lma:BESdecomprelim2} to each $H''(i,i')$ to obtain suitable Hamilton exceptional system candidates
(in particular almost all of them will be `localized').
So consider any $ 1 \le i,i' \le K$%
\COMMENT{AL: added $1 \le$}
 and let $H'':=H''(i,i')$. We claim that there exists $0 \le \eta \le  9/10$
such that $H''$ satisfies the following conditions (which in turn imply conditions (c$_1$)--(c$_5$) of Lemma~\ref{lma:BESdecomprelim2}): 
\begin{itemize}
	\item[\rm (c$_1'$)] $e(H'') = 2(1+ \eta)\alpha n$ and $\eta \alpha n\in \mathbb{N}$.
	\item[\rm (c$_2'$)] $e(H''- W' ) \le 199 \alpha n/100$ and $d_{H''}(w) \ge 13 \alpha n / 25 $ for all $w \in W'$.
	\item[\rm (c$_3'$)] $d_{H''}(w) \le \alpha n$ for all $w \in W_0$ and $d_{H''}(w') \le 41 \alpha n /40$ for all $w'\in W'\setminus W_0$.%
	\COMMENT{Recall that if $W_0\neq W'$ then $W_0=\{w_,w_2\}$ and $W'=\{w_,w_2,w_3\}$.}
	\item[\rm (c$_4'$)] For all $w \in W'$ and all $v \in V(G) \setminus W'$ we have $d_{H''}(w) - d_{H''}(v) \ge \alpha n/150$.
	\item[\rm (c$_5'$)] For all $v \in A \cup B$ we have $d_{H''}(v) \le \epszero n$.
\end{itemize}
Clearly, (b$_3'$) implies the first part of (c$_1'$). Since $e(H'')$ is even by~(b$_3'$) and $\alpha n\in \mathbb{N}$,
it follows that $\eta \alpha n\in \mathbb{N}$.
To verify the first part of (c$_2'$), note that (b$'_3$) and (b$'_4$) together imply that
\begin{align*}
e(H''- W' ) & = e(H'')-\sum_{w\in W'} d_{H''}(w)+e(H''[W']) \\ & \le 
2 \left\lceil e_{G^{\diamond}}(A',B') / (2 K^2) \right\rceil- \sum_{w\in W'} (d_{G^\diamond[A',B']}(w) - 25 \eps n)/K^2 +3\\
& \le (e_{G^\diamond-W'}(A',B') +80\eps n)/K^2.
\end{align*}
Together with Lemma~\ref{critical'}(iv) this implies that
$$e(H''- W' )\le (e_{G-W'}(A',B') +80\eps n)/K^2\le ((3D/4+5)+80\eps n)/K^2\le 199 \alpha n/100.$$
To verify the second part of (c$_2'$), note that by (\ref{eq:degrees}) and Lemma~\ref{lma:BESdecomcritical}(iii) each $w\in W'$
satisfies $d_{G^\diamond [A',B']}(w)\ge d_{G[A',B']}(w)-\phi n \ge 21D/80-\phi n$. Together with (b$_4'$)
this implies $d_{H''}(w) \ge 26\alpha n/50$. Thus (c$_2'$) holds. By (b$_6'$) we have $d_{H''}(w) \le \alpha n $ for all $w \in W_0$.
If $w' \in W' \setminus W_0$, then Lemma~\ref{lma:BESdecomcritical}(ii) implies $d_{G[A',B']}(w') \le D/2 \le 51 K^2 \alpha n /50$.
Thus, $d_{H''}(w') \le 41 \alpha n /40$ by~(b$_4'$). Altogether this shows that~(c$_3'$) holds.
(c$_4'$) follows from (\ref{eq:degrees}), (b$'_4$) and the fact that $d_{G^\diamond [A',B']}(v)\ge d_{G[A',B']}(v)-\phi n$ for all
$v\in V(G)$ by Lemma~\ref{lma:BESdecomcritical}(iii).
(c$_5'$) holds since $d_{H''}(v) \le d_{G^\diamond[A',B']}(v) \le \eps_0 n$ for all $v\in A\cup B$ by (ESch3).

Now we apply Lemma~\ref{lma:BESdecomprelim2} in order to decompose $H''$ into $\alpha n$ edge-disjoint Hamilton exceptional system candidates $F_1,\dots,F_{\alpha n}$ such that $e(F_s) \in \{2,4\}$
for all $s\le \alpha n$ and such that at least $\alpha n /200$ of $F_s$ satisfy $e(F_s)=2$ and $d_{F_s}(w)=1$ for all $w \in W_0$. 
Let
\begin{align*}
\gamma & := \alpha - \frac{\lambda}{K^2} & 
& \text{and} &
\gamma' & : = \frac{\lambda}{K^2}.
\end{align*}   
Recall that by (b$_2'$) all but at most $20 \eps n /K^2 \le \eps' n $ edges of $H''$ lie in $G^{\diamond}[A_0 \cup A_i, B_0 \cup B_{i'}]$.
Together with (\ref{alphahier3}) this ensures that we can relabel the $F_s$ if necessary to obtain $\alpha n $ edge-disjoint Hamilton exceptional system
candidates
$F_1(i,i'), \dots , F_{\gamma n}(i,i')$ and $F'_1(i,i'), \dots , F'_{\gamma ' n}(i,i')$ such that the following properties hold:
\begin{itemize}
	\item[(a$'$)] $F_s(i,i')$ is an $(i,i')$-HESC for every $s \le \gamma n$.
	Moreover, $\gamma ' n$ of the $F_s(i,i')$ satisfy $e(F_s(i,i')) = 2$ and $d_{F_s(i,i')}(w) = 1$ for all $w \in W_0$.%
	\item[(b$'$)] $e(F'_s(i,i')) = 2$ for all but at most $\eps' n$ of the $F'_s(i,i')$.
	\item[(c$'$)] $e(F_s(i,i')), e(F'_s(i,i'))\in \{2,4\}$.
\end{itemize}
For (b$'$) and the `moreover' part of (a$'$), we use  that $ \alpha n /200 - \eps' n \ge 2\lambda n/K^2 = 2\gamma'n$.
Our next aim is to apply Lemma~\ref{BEScons} with $G^\diamond$ playing the role of $G^*$ to extend the above exceptional system candidates into 
exceptional systems.
Clearly conditions (i) and (ii) of Lemma~\ref{BEScons} hold. (iii) follows from (b$'_1$). (iv) and (v) follow from (a$'$)--(c$'$).
(vi) follows from Lemma~\ref{lma:BESdecomcritical}(i),(iii). Finally, (vii) follows from (b$'_5$) since $G'(i,i')$ plays the role of $G^*(i,i')$.
Thus we can indeed apply Lemma~\ref{BEScons} to obtain a decomposition of $G^\diamond$ into $K^2\alpha n$
edge-disjoint Hamilton exceptional systems $J_1(i,i'),\dots,J_{\gamma n}(i,i')$ and $J'_1(i,i'),\dots,J'_{\gamma' n}(i,i')$
with parameter $\eps_0$, where $1\le i,i'\le K$, such that $J_s(i,i')$ is an $(i,i')$-HES which is a faithful extension of  $F_s(i,i')$ for all $s \le \gamma n$
and $J'_s(i,i')$ is a faithful extension of $F'_s(i,i')$ for all $s \le \gamma' n$.
Then the set $\mathcal{J}$ of all these exceptional systems is as required in Lemma~\ref{lma:BESdecomcritical}.
(Since $W_0$ contains $\{w_1,w_2\} \cap W$, the `moreover part' of (a$'$) implies the `moreover part' of Lemma~\ref{lma:BESdecomcritical}(b).)
\endproof


\section{The Case when $e(A',B') < D$}\label{1factsec}

The aim of this section is to prove Lemma~\ref{lma:PBESdecom}. This lemma provides a decomposition of the exceptional edges into exceptional systems in the  case when $e(A',B') < D$.
In this case, we do not need to prove any auxiliary lemmas first, as we can apply those proved in the other two cases
(Lemmas~\ref{BEScons} and~\ref{lma:movecritical}).

\removelastskip\penalty55\medskip\noindent{\bf Proof of Lemma~\ref{lma:PBESdecom}.}
Let $\eps'$ be a new constant such that $\eps \ll \eps' \ll \lambda, 1/K$ and
set 
\begin{equation} \label{alphaD}
2\alpha n:= \frac{n/2 - 1 - \phi n }{K^2}.
\end{equation}
Similarly as in the proof of Lemma~\ref{lma:BESdecomcritical} we have
\begin{equation}\label{alphahier4}
\eps \ll \eps'  \ll \lambda, 1/K, \alpha \ll 1.
\end{equation}
We claim that $G^{\diamond}$ can be decomposed into edge-disjoint spanning subgraphs $H(i,i')$ and $H''(i,i')$ (for all $1\le i,i' \le K$)
which satisfy the following properties, where $G'(i,i'):=H(i,i')+H''(i,i')$:
\begin{itemize}
\item[\rm (b$_1'$)] Each $H(i,i')$ contains only $A_0A_i$-edges and $B_0B_{i'}$-edges.
 \item[\rm (b$_2'$)] $H''(i,i')\subseteq G^{\diamond}[A',B']$.
Moreover, all but at most  $\eps' n$ edges of $H''(i,i')$ lie in $G^{\diamond}[A_0 \cup A_i, B_0 \cup B_{i'}]$.
\item[\rm (b$_3'$)] $e ( H''(i,i') )$ is even and $e(H''(i,i')) \le 2 \alpha n $.
\item[\rm (b$_4'$)] $\Delta(H''(i,i')) \le e ( H''(i,i') )/2$.
\item[\rm  (b$_5'$)] $d_{G'(i,i')}(v )  =  (2\alpha   \pm  \eps' )n$ for all $v\in V_0$.
\end{itemize}
To see this, let us first consider the case when $ e_{G^{\diamond}}(A',B') \le 300 \eps n $.
Apply Lemma~\ref{lma:randomslice} to $G^{\diamond}$ in order to obtain a decomposition
of $G^{\diamond}$ into edge-disjoint spanning subgraphs $H(i,i')$ and $H'(i,i')$ (for all $1\le i,i' \le K$) which satisfy
Lemma~\ref{lma:randomslice}(a$_1$)--(a$_5$). Set $H''(1,1) := \bigcup_{i,i' \le K} H'(i,i') = G^{\diamond}[A',B']$ and
$H''(i,i') $ $:= \emptyset$ for all other pairs $1 \le i,i' \le K$.%
\COMMENT{AL: added $1 \le$}
 Then (b$_1'$) follows from (a$_1$).
(b$_2'$) follows from our definition of the $H''(i,i')$ and our assumption that $ e_{G^{\diamond}}(A',B') \le 300 \eps n< \eps' n <\alpha n$.
Together with Lemma~\ref{lma:PBESdecom}(iv) this also implies (b$_3'$). (b$_4'$) follows from Lemma~\ref{lma:PBESdecom}(v).
Note that by Lemma~\ref{lma:PBESdecom}(i) and~(iii),
every $v\in V_0$ satisfies $d_{G^\diamond}(v)=n/2-1-\phi n=2K^2\alpha n$.
So, writing $G(i,i'):=H(i,i')+H'(i,i')$, (a$_5$) implies that
$$
d_{G'(i,i')}(v )=d_{G(i,i')}(v)\pm 300 \eps n =(2\alpha \pm 4\eps/K^2)n\pm 300 \eps n=(2\alpha   \pm  \eps' )n.
$$
Thus (b$_5'$) holds too.

So let us next consider the case when $ e_{G^{\diamond}}(A',B') > 300 \eps n $.
Let $W_0$ be the set of all those vertices $v \in V(G)$ for which $d_{G^{\diamond}[A',B']}(v) \ge 3e_{G^{\diamond}}(A',B')/8$.
Then clearly $|W_0| \le 2$. Moreover, each $v \in V(G) \setminus W_0$ satisfies
\begin{align}\label{eq:degnonW0}
d_{G^{\diamond}[A',B']}(v) + 26 \eps n < 3e_{G^{\diamond}}(A',B')/8 + e_{G^{\diamond}}(A',B')/8 = e_{G^{\diamond}}(A',B')/2.
\end{align}
Recall from Lemma~\ref{lma:PBESdecom}(v) that $d_{G^{\diamond}[A',B']}(w) \le e_{G^{\diamond}}(A',B')/2$ for each $w\in W_0$. So we can
apply Lemma~\ref{lma:movecritical} to $G^{\diamond}$ in order to obtain a decomposition
of $G^{\diamond}$ into edge-disjoint spanning subgraphs $H(i,i')$ and $H''(i,i')$ (for all $1\le i,i' \le K$) which satisfy
Lemma~\ref{lma:movecritical}(b$_1$)--(b$_7$). Then (b$_1$) and (b$_2$) imply (b$'_1$) and (b$'_2$).
(b$_3'$) follows from (b$_3$),~(\ref{alphaD}) and Lemma~\ref{lma:PBESdecom}(v). Note that (b$_3$), (b$_4$) and~(\ref{eq:degnonW0})
together imply that
\begin{equation}\label{eq:degH''v}
d_{H''(i,i')}(v) \le \frac{e_{G^{\diamond}}(A',B')/2-\eps n}{K^2}\le \frac{e ( H''(i,i') )}{2}
\end{equation}
for all $v \in V_0 \setminus W_0$. Note that each $v\in A\cup B$ satisfies $d_{H''(i,i')}(v) \le d_{G^\diamond[A',B']}(v)\le \eps_0 n$ by
Lemma~\ref{lma:PBESdecom}(iv) and (ESch3). Together with the fact that $e ( H''(i,i'))\ge 2\lfloor 300\eps n/(2K^2)\rfloor \ge 2\eps_0 n$
by~(b$_3$), this implies that~(\ref{eq:degH''v}) also holds for all $v\in A\cup B$. 
 Together with (b$_7$) this implies (b$'_4$).
(b$'_5$) follows from (b$_5$) and the fact that by Lemma~\ref{lma:PBESdecom}(i) and~(iii)
every $v\in V_0$ satisfies $d_{G^\diamond}(v)=n/2-1-\phi n=2K^2\alpha n$.
So (b$'_1$)--(b$'_5$) hold in all cases.

We now decompose the localized subgraphs $H''(i,i')$ into exceptional system candidates. For this, 
fix $i,i' \le K$ and write $H''$ for $H''(i,i')$. By (b$_4'$) we have $\Delta(H'') \le e ( H'')/2$ and so $\chi'(H'') \le e ( H'')/2$.
Apply Proposition~\ref{prop:matchingdecomposition} with $e ( H'')/2$ playing the role of $m$ to decompose $H''$ into $e ( H'')/2$ edge-disjoint
matchings, each of size~$2$. Note that $\alpha n-e ( H'')/2\ge 0$ by (b$_3'$). So we can add some empty matchings to obtain a
decomposition of $H''$ into $\alpha n$ edge-disjoint $M_1, \dots ,M_{\alpha n}$ such that each $M_s$ is either empty or has size~2.
Let
\begin{align*}
\gamma & := \alpha - \frac{\lambda}{K^2} & 
& \text{and} &
\gamma' & : = \frac{\lambda}{K^2}.
\end{align*} 
Recall from (b$_2'$) that all but at most $\eps' n \le \gamma' n$ edges of $H''$ lie in $G^{\diamond}[A_0 \cup A_i, B_0 \cup B_{i'}]$.
Hence by relabeling if necessary, we may assume that  $M_s \subseteq G^{\diamond}[A_0 \cup A_i ,  B_0 \cup B_{i'}]$ for every $s \le \gamma n$.
So by setting $F_s(i,i') := M_s$ for all $s \le \gamma n$ and $F_s'(i,i') := M_{ \gamma n + s }$ for all $s \le \gamma' n$
we obtain a decomposition of $H''$ into edge-disjoint exceptional system candidates $F_1(i,i'), \dots , F_{\gamma n}(i,i')$
and $ F'_1(i,i'), \dots , F'_{\gamma ' n}(i,i')$ such that
the following properties hold:
\begin{itemize}
	\item[(a$'$)] $F_s(i,i')$ is an $(i,i')$-ESC for every $s \le \gamma n$.
	\item[(b$'$)] Each $F_s(i,i')$ is either a Hamilton exceptional system candidate with $e(F_s(i,i')) = 2$or a matching exceptional system candidate with $e(F_s(i,i'))$  $= 0$. The analogue holds for each $F'_{s'}(i,i')$.
\end{itemize}
Our next aim is to apply Lemma~\ref{BEScons} with $G^\diamond$ playing the role of $G^*$, to extend the above 
exceptional system candidates into exceptional systems.
Clearly conditions (i) and (ii) of Lemma~\ref{BEScons} hold. (iii) follows from (b$'_1$). (iv) and (v) follow from (a$'$) and (b$'$).
(vi) follows from Lemma~\ref{lma:PBESdecom}(i),(iii). Finally, (vii) follows from (b$'_5$) since $G'(i,i')$ plays the role of $G^*(i,i')$ in Lemma~\ref{BEScons}.
Thus we can indeed apply Lemma~\ref{BEScons} to obtain a decomposition of $G^\diamond$ into $K^2\alpha n$
edge-disjoint exceptional systems $J_1(i,i'),\dots,J_{\gamma n}(i,i')$ and $J'_1(i,i'),\dots,J'_{\gamma' n}(i,i')$,
where $1\le i,i'\le K$, such that $J_s(i,i')$ is an $(i,i')$-ES which is a faithful extension of  $F_s(i,i')$ for all $s \le \gamma n$
and $J'_s(i,i')$ is a faithful extension of $F'_s(i,i')$ for all $s \le \gamma' n$.
Then the set $\mathcal{J}$ of all these exceptional systems is as required in Lemma~\ref{lma:PBESdecom}.
\endproof

\chapter{The bipartite case}\label{paper2}
The aim of this chapter is to prove Theorems~\ref{1factbip} and~\ref{NWmindegbip}. Recall that Theorem~\ref{NWmindegbip} guarantees many edge-disjoint Hamilton cycles in a graph $G$ when $G$ has large minimum degree and is
close to bipartite, whilst Theorem~\ref{1factbip} guarantees a Hamilton decomposition of $G$ when $G$ has sufficiently large minimum degree, is regular and is
close to bipartite.
In Section~\ref{sec:sketch} we give an outline of the proofs.
The results from Sections~\ref{eliminating} and~\ref{findBES} are used in both the proofs of Theorems~\ref{1factbip} and~\ref{NWmindegbip}.
In Sections~\ref{sec:spec} and~\ref{sec:robust} we build up machinery for the proof of
Theorem~\ref{1factbip}. We then prove Theorem~\ref{NWmindegbip} in Section~\ref{sec:proof1} and Theorem~\ref{1factbip} in Section~\ref{sec:proof2}.

Unlike in the previous chapters,
in this chapter we view a matching $M$ as a set of edges. (So $|M|$ for example, denotes the number of edges in $M$.)

\section{Overview of the Proofs of Theorems~$\text{\ref{1factbip}}$ and~$\text{\ref{NWmindegbip}}$}\label{sec:sketch}
Note%
    \COMMENT{Deryk changed this section - so read again}
that, unlike in Theorem~\ref{1factbip}, in Theorem~\ref{NWmindegbip} we do not require a complete
decomposition of our graph $F$ into edge-disjoint Hamilton cycles. Therefore, the proof of Theorem~\ref{1factbip} is considerably more involved than the proof of Theorem~\ref{NWmindegbip}.
Moreover, the ideas in the proof of Theorem~\ref{NWmindegbip} are all used in the proof of Theorem~\ref{1factbip} too.

\subsection{Proof Overview for Theorem~\ref{NWmindegbip}} \label{sec:sketch1}
Let $F$ be a graph on $n$ vertices with $\delta (F) \geq (1/2-o(1))n$ which is close to the balanced
bipartite graph $K_{n/2, n/2}$. Further, suppose that $G$ is a $D$-regular spanning subgraph of $F$
as in Theorem~\ref{NWmindegbip}.
Then there is a partition $A$, $B$ of $V(F)$ such that $A$
and $B$ are of roughly equal size and most edges in $F$ go between $A$ and $B$.
 Our ultimate aim is to construct $D/2$ edge-disjoint Hamilton
cycles in $F$.

Suppose first that, in the graph $F$, both $A$ and $B$ are independent sets of equal size. So
$F$ is an almost complete balanced bipartite graph.
In this case, the densest spanning even-regular subgraph $G$ of $F$ is also almost complete bipartite.
This means that one can extend existing techniques (developed e.g. in \cite{CKO, FKS,fk, HartkeHCs,OS})
to find an approximate Hamilton decomposition.
(In Chapter~\ref{paper3}, using such techniques, we prove an approximate decomposition result (Lemma~\ref{almostthmbip}) which is suitable for our purposes. In particular, Lemma~\ref{almostthmbip} is sufficient to prove Theorem~\ref{NWmindegbip} in this special case.)
The real difficulties arise when
\begin{itemize}
	\item[{\rm (i)}] $F$ is unbalanced (i.e. $|A|\not = |B|$);
	\item[{\rm (ii)}] $F$ has vertices having high degree in both $A$ and $B$ (these are called exceptional vertices).
\end{itemize}

To illustrate (i), recall the following example due to Babai%
\COMMENT{osthus added Babai} 
(which is the extremal construction for Corollary~\ref{NWmindegcor}).
Consider the graph $F$ on $n = 8k+2$ vertices consisting of one vertex class $A$ of size $4k+2$ containing
a perfect matching and no other edges, one empty
vertex class $B$ of size $4k$,  and all possible edges between $A$ and $B$.
Thus the minimum degree of $F$ is $4k+1=n/2$. 
Then one can use Tutte's factor theorem to show that the largest even-regular spanning subgraph $G$ of $F$ has degree $D = 2k=(n-2)/4$.
Note that to prove Theorem~\ref{NWmindegbip} in this case, each of the $D/2=k$ Hamilton cycles we find must contain exactly two of the $2k+1$ edges in $A$.
In this way, we can `balance out' the difference in the vertex class sizes.

More generally we will construct our Hamilton cycles in two steps.
In the first step, we find a path system $J$ which balances out the vertex class sizes (so in the above example, $J$ would contain two edges in $A$).
Then we extend $J$ into a Hamilton cycle using only $AB$-edges in $F$.
It turns out that the first step is the difficult one.
It is easy to see that a path system $J$%
\COMMENT{Andy: added $J$}
will balance out the sizes of $A$ and $B$ (in the sense that the number of uncovered vertices in $A$ and $B$ is the same) if 
and only if
\begin{align}
e_J(A)- e_J(B) = |A| - |B|. \label{eJA}
\end{align}
Note that any Hamilton cycle also satisfies this identity.
So we need to find a set of $D/2$ path systems $J$ satisfying \eqref{eJA} (where $D$ is the degree of $G$).
This is achieved (amongst other things) in Sections~\ref{sec:slice} and~\ref{sec:global}.

As indicated above, our aim is to use Lemma~\ref{almostthmbip} (our approximate decomposition result for the bipartite case) in order to extend each such $J$ into a Hamilton cycle.
To apply Lemma~\ref{almostthmbip} we also need to extend the balancing path systems $J$ into `balanced exceptional (path)
systems' which contain all the exceptional vertices from (ii).
This is achieved in Section~\ref{besconstruct}.
Lemma~\ref{almostthmbip} also assumes that the path systems are `localized' with respect to a given subpartition of $A,B$ (i.e.~they are induced by a small number of partition classes).
Section~\ref{partition} prepares the ground for this. The balanced exceptional systems are the analogues of the exceptional systems which we use in the two cliques case (i.e. in Chapter~\ref{paper1}).

Finding the balanced exceptional systems is extremely difficult if $G$ contains edges between the set $A_0$ of exceptional vertices in $A$ and the
set $B_0$ of exceptional vertices in $B$.
So in a preliminary step, we find and remove a small number of edge-disjoint Hamilton cycles covering all $A_0B_0$-edges in Section~\ref{eliminating}.
We put all these steps together in Section~\ref{sec:proof1}.
(Sections~\ref{sec:spec}, \ref{sec:robust} and~\ref{sec:proof2}%
\COMMENT{osthus added sec 9} 
are only relevant for the proof of Theorem~\ref{1factbip}.)

\subsection{Proof Overview for Theorem~\ref{1factbip}}
The main result of this chapter is Theorem~\ref{1factbip}.
Suppose that $G$ is a $D$-regular graph satisfying the conditions of that theorem.
Using the approach of the previous subsection, one can obtain an approximate decomposition of $G$,
i.e.~a set of edge-disjoint Hamilton cycles covering almost all edges of~$G$.
However, one does not have any control over the `leftover' graph~$H$, which makes a complete decomposition seem infeasible.
As in the proof of Theorem~\ref{1factstrong}, we use the following strategy to overcome this issue and
 obtain a decomposition of $G$:
\begin{itemize}
\item[(1)] find a (sparse) robustly decomposable graph~$G^{\rm rob}$ in $G$ and let $G'$ denote the leftover;
\item[(2)] find an approximate Hamilton decomposition of $G'$ and let $H$ denote the (very sparse) leftover;
\item[(3)] find a Hamilton decomposition of~$G^{\rm rob} \cup H$.
\end{itemize}
As before, it is of course far from obvious that such a graph $G^{\rm rob}$ exists.
By assumption our graph $G$ can be partitioned into two classes $A$ and $B$ of almost equal size such that almost all the edges in $G$ go between $A$ and $B$.
If both $A$ and $B$ are independent sets of equal size then the `bipartite' version of the robust decomposition lemma of~\cite{Kelly} guarantees our desired subgraph $G^{\rm rob}$ of $G$.
Of course, in general our graph $G$  will contain edges in $A$ and $B$. 
Our aim is therefore to replace such edges with `fictive  edges' between $A$ and $B$, so that we can apply this version of the
robust decomposition lemma (Lemma~\ref{rdeclemma'}). (We note here that  Lemma~\ref{rdeclemma'} is designed to deal with bipartite graphs. So its statement is slightly different to the  robust decomposition lemma (Lemma~\ref{rdeclemma}) that was applied in the proof of  Theorem~\ref{1factstrong}.)

More precisely, similarly as in the proof of Theorem~\ref{NWmindegbip}, we construct a collection of localized balanced exceptional systems.
Together these path systems contain all the edges in $G[A]$ and $G[B]$. 
Again,%
\COMMENT{osthus added `again'} 
each balanced exceptional system balances out the sizes of $A$ and $B$ and covers the exceptional
vertices in $G$ (i.e.~those vertices having high degree into both $A$ and $B$).

Similarly as in the two cliques case, we now introduce fictive edges. This time,
by replacing  edges of the balanced exceptional systems with fictive edges,
we obtain from $G$ an auxiliary (multi)graph $G^*$ which only contains edges between $A$ and $B$ and 
which does not contain the exceptional vertices of $G$. 
This will allow us to apply
 the robust decomposition lemma. In particular this ensures that each Hamilton cycle obtained in $G^*$ contains
a collection of fictive edges corresponding to a single balanced exceptional system (as before the set-up of
the robust decomposition lemma does allow for this). Each such Hamilton cycle in $G^*$ then corresponds
to a Hamilton cycle in $G$.

We now give an example of how we introduce fictive edges.
Let $m$ be an integer so that $(m-1)/2$ is even. Set $m':= (m-1)/2$ and $m'':= (m+1)/2$. 
Define the graph $G$ as follows:
Let $A$ and $B$ be disjoint vertex sets of size $m$. Let $A_1, A_2$ be a partition of $A$ and
$B_1,B_2$ be a partition of $B$ such that $|A_1|=|B_1|=m''$. Add all edges between $A$ and $B$. Add
a matching $M_1=\{e_1, \dots, e_{m'/2}\}$ covering precisely the
vertices of $A_2$ and add a matching $M_2=\{e'_1, \dots, e'_{m'/2}\}$ covering precisely the vertices of $B_2$. Finally add a vertex
$v$ which sends an edge to every vertex in $A_1 \cup B_1$. So $G$ is $(m+1)$-regular (and $v$ would be regarded as an exceptional vertex).

Now pair up each edge $e_i$ with the edge $e'_i$. Write $e_i=x_{2i-1} x_{2i}$ and
$e'_i=y_{2i-1} y_{2i}$ for each $1\leq i \leq m'/2$.
Let $A_1 =\{a_1, \dots , a_{m''}\}$ and $B_1=\{b_1, \dots , b_{m''}\}$ and write
$f_i: =a_ib_i$ for all $1 \leq i \leq m''$. Obtain $G^*$ from $G$ by deleting  $v$ together with the
edges in $M_1 \cup M_2$ and by adding the following fictive edges: 
add $f_i$ for each $1\leq i \leq m''$ and add $x_jy_j$ for each $1\leq j \leq m'$.
Then $G^*$ is a balanced bipartite $(m+1)$-regular multigraph containing only edges between $A$ and $B$.

First, note that any Hamilton cycle $C^*$ in $G^*$ that contains precisely one fictive edge $f_i$ for some
$1\le i \leq m''$ corresponds to a Hamilton cycle $C$ in $G$, where we replace the fictive edge $f_i$ with
$a_iv$ and $b_iv$.
Next, consider any Hamilton cycle $C^*$ in $G^*$ that contains precisely three fictive edges; $f_i$ for some
$1\le i \leq m''$ together with $x_{2j-1}y_{2j-1}$ and $x_{2j}y_{2j}$ for some $1\leq j \leq m'/2$.
Further suppose $C^*$ traverses the vertices $a_i,b_i,x_{2j-1},y_{2j-1}, x_{2j},y_{2j}$ in this order. Then $C^*$ corresponds
to a Hamilton cycle $C$ in $G$, where we replace the fictive edges with $a_iv, b_i v, e_j$ and $e'_j$
(see Figure~\ref{fig:bridges}). Here the path system $J$ formed by the edges $a_iv, b_iv, e_j$ and $e'_j$ is an example
of a balanced exceptional system.
The above ideas are formalized in Section~\ref{sec:spec}.

\begin{figure}[htb!]
\begin{center}
\includegraphics[width=0.34\columnwidth]{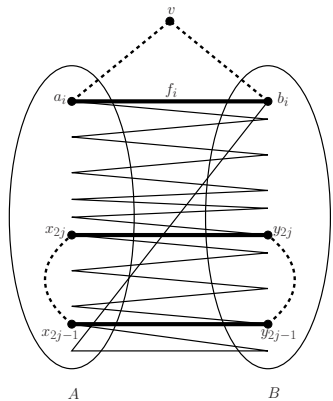}  
\caption{Transforming the problem of finding a Hamilton cycle in $G$ into finding a Hamilton cycle in the balanced bipartite
graph $G^*$}
\label{fig:bridges}
\end{center}
\end{figure}

We can now summarize the steps leading
to proof of Theorem~\ref{1factbip}.
In Section~\ref{eliminating}, we find and remove a set of edge-disjoint Hamilton cycles covering all edges in $G[A_0,B_0]$.
We can then find the localized balanced exceptional systems in Section~\ref{findBES}. After this, we need to extend and combine them into 
certain path systems and factors (which contain fictive edges) in Section~\ref{sec:spec}, before we can use them as an `input' for the robust decomposition lemma in~Section~\ref{sec:robust}.
Finally, all these steps are combined in Section~\ref{sec:proof2} to prove Theorem~\ref{1factbip}.

\section{Eliminating Edges between the Exceptional Sets}\label{eliminating}

Suppose that $G$ is a $D$-regular graph as in Theorem~\ref{1factbip}. 
The purpose of this section is to prove Corollary~\ref{coverA0B02c}.
Roughly speaking, given $K\in \mathbb{N}$, this corollary states that one can delete a small number of edge-disjoint Hamilton cycles from $G$
to obtain a spanning subgraph $G'$ of $G$ and a partition $A,A_0,B,B_0$ of $V(G)$ such that (amongst others) the following properties hold:
\begin{itemize}
\item almost all edges of $G'$ join $A\cup A_0$ to $B\cup B_0$;
\item $|A|=|B|$ is divisible by $K$;
\item every vertex in $A$ has almost all its neighbours in $B\cup B_0$ and every vertex in $B$ has almost all its neighbours in $A\cup A_0$;
\item $A_0\cup B_0$ is small and there are no edges between $A_0$ and
$B_0$ in $G'$.
\end{itemize}
We will call $(G',A,A_0,B,B_0)$ a bi-framework.
(The formal definition of a bi-framework is stated before Lemma~\ref{coverA0B02}.)%
\COMMENT{AL: added this sentense}
Both $A$ and $B$ will then be split into $K$ clusters of equal size.
Our assumption that $G$ is $\eps_{\rm ex}$-bipartite easily implies that there is such a partition $A,A_0,B,B_0$
which satisfies all these properties apart from the property that there are no edges
between $A_0$ and $B_0$.
So the main part of this section shows that we can cover the collection of all  edges 
between $A_0$ and $B_0$ by a small number of edge-disjoint
Hamilton cycles.

Since Corollary~\ref{coverA0B02c} will also be used in the proof of Theorem~\ref{NWmindegbip}, instead of working with regular
graphs we need to consider so-called balanced graphs. We also need to find the above Hamilton cycles in the graph $F\supseteq G$ rather than
in $G$ itself (in the proof of Theorem~\ref{1factbip} we will take $F$ to be equal to~$G$).
 
More precisely, suppose that $G$ is a graph and that $A'$, $B'$ is a partition of $V(G)$, 
where $A'=A_0\cup A$, $B'=B_0\cup B$ and $A,A_0,B,B_0$ are disjoint.
Then we say that $G$ is \emph{$D$-balanced (with respect to $(A,A_0,B,B_0)$)} if 
\begin{itemize}
\item[(B1)] $e_G(A')-e_G(B')=(|A'| -|B'|)D/2$;
\item[(B2)] all vertices in $A_0 \cup B_0$ have degree exactly $D$.%
    \COMMENT{Previously (B2) also included that all vertices in $A \cup B$ have degree at least $D$. But then
$A_0B_0$-path systems are not $2$-balanced since some vertices in $A\cup B$ have degree zero. Check whether we made the necessary changes
and worked with the new def throughout the paper...}
\end{itemize}
Proposition~\ref{edge_number} below implies that whenever $A,A_0,B,B_0$ is a partition of the vertex set of a $D$-regular graph $H$,
then $H$ is $D$-balanced with respect to $(A,A_0,B,B_0)$.
Moreover, note that if $G$ is $D_G$-balanced with respect to $(A,A_0,B,B_0)$ and
$H$ is a spanning subgraph of $G$ which is $D_H$-balanced with respect to $(A,A_0,B,B_0)$, then $G-H$ is
$(D_G-D_H)$-balanced with respect to $(A,A_0,B,B_0)$. Furthermore, a graph $G$ is $D$-balanced with respect to
$(A,A_0,B,B_0)$ if and only if $G$ is $D$-balanced with respect to $(B,B_0,A,A_0)$.
\begin{prop}\label{edge_number}
Let $H$ be a graph and let $A'$, $B'$ be a partition of $V(H)$. Suppose that $A_0$, $A$ is a partition of $A'$
and that $B_0$, $B$ is a partition of $B'$ such that $|A|=|B|$. Suppose that $d_H(v)=D$ for every $v\in A_0\cup B_0$
and $d_H(v)=D'$ for every $v\in A\cup B$. Then $e_H(A')-e_H(B')=(|A'| -|B'|)D/2.$ 
\end{prop}
\proof
Note that $$\sum_{x\in A'}d_H(x, B')=e_H(A',B')=\sum_{y\in B'} d_H(y, A').$$ Moreover,
\begin{align*}
2e_H(A')& =\sum_{x\in A_0}(D-d_H(x, B'))+\sum_{x\in A}(D'-d_H(x, B'))\\
&=D|A_0|+D'|A|-\sum_{x\in A'}d_H(x, B')
\end{align*}
and
\begin{align*}
2e_H(B')& =\sum_{y\in B_0}(D-d_H(y, A'))+\sum_{y\in B}(D'-d_H(y, A'))\\
&=D|B_0|+D'|B|-\sum_{y\in B'}d_H(y, A').
\end{align*}
Therefore $$2e_H(A')-2e_H(B')=D(|A_0|-|B_0|)+D'(|A|-|B|)=D(|A_0|-|B_0|)=D(|A'|-|B'|),$$
as desired.
\endproof

The following observation states that balancedness is preserved under
suitable modifications of the partition.
\begin{prop} \label{keepbalance}
Let $H$ be $D$-balanced with respect to $(A,A_0,B,B_0)$. 
 Suppose that $A'_0,B'_0$ is a partition of $A_0 \cup B_0$.
Then $H$ is $D$-balanced with respect to $(A,A'_0,B,B'_0)$.
\end{prop}
\proof
Observe that the general result follows if we can show that $H$ is $D$-balanced with respect to $(A,A'_0,B,B'_0)$, where
$A'_0=A_0 \cup \{ v\}$, $B'_0=B_0 \setminus \{v\}$ and $v \in B_0$.
(B2) is trivially satisfied in this case, so we only need to check (B1) for the new partition.
For this, let $A':=A_0 \cup A$ and $B':=B_0 \cup B$. Now note that (B1) for the original partition implies that
\begin{align*}
e_H(A'_0 \cup A)-e_H(B'_0 \cup B) & = e_H(A') +d_H(v,A') -(e_H(B')-d_H(v,B'))  \\
& = (|A'|-|B'|)D/2 +D =(|A'_0 \cup A|-|B'_0 \cup B|)D/2.
\end{align*}
Thus (B1) holds for the new partition.
\endproof

Suppose that $G$ is a graph and  $A',B'$ is a partition of $V(G)$.
For every vertex $v\in A'$ we call $d_G(v,A')$ the \emph{internal degree
of~$v$} in~$G$. Similarly, for every vertex $v\in B'$ we call $d_G(v,B')$ the \emph{internal degree
of~$v$} in~$G$.

Given a graph $F$ and a spanning subgraph $G$ of $F$ , 
we say that $(F,G,A,A_0,B,$ $B_0)$ is an \emph{$(\eps,\eps',K,D)$-weak framework} if the following holds,
where $A':=A_0\cup A$, $B':=B_0\cup B$ and $n:=|G|=|F|$: 
\begin{itemize}
\item[{\rm (WF1)}] $A,A_0,B,B_0$ forms a partition of $V(G)=V(F)$;
\item[{\rm (WF2)}] $G$ is $D$-balanced with respect to $(A,A_0,B,B_0)$;
\item[{\rm (WF3)}] $e_G(A'), e_G(B')\le \eps n^2$;
\item[{\rm (WF4)}] $|A|=|B|$ is divisible by $K$. Moreover,
 $a+b\le \eps n$, where $a:=|A_0|$ and $b:=|B_0|$;
\item[{\rm (WF5)}] all vertices in $A \cup B$  have internal degree at most $\eps' n$ in $F$;
\item[{\rm (WF6)}] any vertex $v$ has internal degree at most $d_G(v)/2$ in $G$.%
\COMMENT{NEW: Got rid of error term here.}
\end{itemize}
 Throughout the chapter, when referring to internal degrees without mentioning the partition, we always mean with
respect to the partition $A'$, $B'$, where $A'=A_0\cup A$ and $B'=B_0\cup B$. Moreover, $a$ and $b$ will always denote $|A_0|$ and $|B_0|$.

We say that $(F,G,A,A_0,B,B_0)$ is an \emph{$(\eps,\eps',K,D)$-pre-framework} if it satisfies (WF1)--(WF5).
The following observation states that pre-frameworks are preserved if we remove suitable balanced subgraphs.

\begin{prop} \label{WFpreserve}
Let $\eps, \eps ' >0$ and $K,D_G,D_{H} \in \mathbb N$.%
	\COMMENT{AL:changed $D_{G'}$ to $D_{H}$}
Let $(F,G,A,A_0,B,B_0)$ be an $(\eps,\eps',K,D_G)$-pre framework.
Suppose that $H$ is a $D_H$-regular spanning subgraph of $F$ such that
 $G\cap H$ is $D_H$-balanced with respect to $(A,A_0,B,B_0)$. 
Let $F':=F-H$ and $G':=G-H$. Then $(F',G',A,A_0,B,B_0)$ is an $(\eps,\eps',K,D_G-D_H)$-pre framework.
\end{prop}
\proof
Note that all required properties except possibly (WF2) are not affected by removing edges.
But $G'$ satisfies (WF2) since $G\cap H$ is $D_H$-balanced with respect to $(A,A_0,B,B_0)$.
\endproof

\begin{lemma}\label{bip_decomp}
Let $0<1/n\ll \varepsilon \ll  \varepsilon', 1/K\ll 1$ and let $D\ge n/200$. 
Suppose that $F$ is a  graph on $n$ vertices which is $\eps$-bipartite and that $G$ is a $D$-regular spanning subgraph of $F$.
Then there is a partition $A,A_0,B,B_0$ of $V(G)=V(F)$ so that 
$(F,G,A,A_0,B,B_0)$ is an $(\eps^{1/3},\eps',K,D)$-weak framework.
\end{lemma}
\proof
Let $S_1,S_2$ be a partition of $V(F)$ which is guaranteed by the assumption that $F$ is $\eps$-bipartite.
Let $S$ be the set of all those vertices $x \in S_1$ with $d_F(x,S_1) \ge \sqrt{\eps} n$
together with all those vertices $x \in S_2$ with $d_F(x,S_2) \ge \sqrt{\eps} n$.
Since $F$ is $\eps$-bipartite, it follows that $|S| \le 4\sqrt{\eps} n$.%
    \COMMENT{Have $4\sqrt{\eps} n$ instead of $2\sqrt{\eps} n$ since we get $|S\cap S_i|\sqrt{\eps} n\le 2e(S_i)\le 2\eps n^2$.}

Given a partition $X,Y$ of $V(F)$, we say that $v \in X$ is \emph{bad for $X,Y$} if $d_G(v,X)>d_G(v,Y)$
and similarly that  $v \in Y$ is \emph{bad for $X,Y$} if $d_G(v,Y)>d_G(v,X)$.
Suppose that there is a vertex $v \in S$ which is bad for $S_1$, $S_2$. 
Then we move $v$ into the class which does not currently contain $v$ to obtain a new partition $S'_1$, $S'_2$.
We do not change the set $S$.
If there is a vertex $v' \in S$ which is bad for $S'_1$, $S'_2$, then again we move it into the other class.

We repeat this process. After each step, the number of edges in $G$ between the two classes increases, so this process has to terminate with some partition $A'$, $B'$
such that $A' \bigtriangleup S_1 \subseteq  S$ and $B' \bigtriangleup S_2 \subseteq  S$.
Clearly, no vertex in $S$ is now bad for $A'$, $B'$. Also, for any $v \in A' \setminus S$ we have 
\begin{align}\label{eq:dvA'}
d_G(v,A') & \le d_F (v,A') \le d_F(v,S_1)+|S|\le \sqrt{\eps} n + 4\sqrt{\eps} n< \eps'n \\
& < D/2= d_G(v)/2. \nonumber
\end{align}
 Similarly, $d_G(v,B') < \eps' n< d_G(v)/2$ for all $v \in B' \setminus S$.
Altogether this implies that no vertex is bad for $A'$, $B'$ and thus (WF6) holds.
 Also note that
$e_G(A',B')\ge e_G(S_1,S_2)\ge e(G)-2\eps n^2$.%
	\COMMENT{AL: replaced $2\eps$ with $4\eps$ same with equation below, DO changed these back again...}
So 
\begin{equation}\label{A'B'edge}
e_G(A'),e_G(B') \le 2\eps n^2.
\end{equation}
This implies~(WF3).

Without loss of generality we may assume that $|A'|\ge |B'|$. 
Let $A'_0$ denote the set of all those vertices $v\in A'$ for which $d_F(v, A')\ge \varepsilon' n$.
Define $B'_0\subseteq B'$ similarly.
We will choose sets $A \subseteq A' \setminus A'_0$ and $A_0\supseteq A'_0$
and sets $B\subseteq B'\setminus B'_0$ and $B_0\supseteq B'_0$ such that $|A|=|B|$ is 
divisible by $K$ and so that $A,A_0$ and $B,B_0$ are partitions of $A'$ and $B'$ respectively.
We obtain such sets by moving at most 
$\left| |A'\setminus A'_0|-|B'\setminus B'_0|\right| +K$ vertices from 
$A'\setminus A'_0$ to $A'_0$ and at most $\left| |A'\setminus A'_0|-|B'\setminus B'_0|\right| +K$ vertices from $B'\setminus B'_0$ to $B'_0$.
The choice of $A,A_0,B,B_0$ is such that (WF1) and  (WF5) hold.  Further, since $|A|=|B|$, Proposition~\ref{edge_number} implies~(WF2).

In order to verify (WF4), it remains to show that $a+b=|A_0\cup B_0|\le  \eps^{1/3} n$.
But (\ref{eq:dvA'}) together with its analogue for the vertices in $B'\setminus S$ implies that $A'_0\cup B'_0\subseteq S$.
Thus $|A'_0|+ |B'_0|\le |S|\le 4\sqrt{\eps} n$. 
Moreover, (WF2), (\ref{A'B'edge}) and our assumption that $D\ge n/200$ together imply that
$$|A'|-|B'| =(e_G(A')-e_G(B'))/(D/2)\le 2\eps n^2/(D/2) \le 800 \eps n.$$
So altogether, we have 
\begin{align*}
a+b & \le |A'_0\cup B'_0|+2\left| |A'\setminus A'_0|-|B'\setminus B'_0|\right| +2K \\ &
\le
4\sqrt{\eps} n+2\left| |A'|-|B'|-(|A'_0|-|B'_0|)\right|+2K \\ &
 \le 4\sqrt{\eps} n+1600 \eps n+ 8\sqrt{\eps} n+2K
\le \eps^{1/3}n.
\end{align*}
Thus (WF4) holds.
\endproof

Our next goal is to cover the edges of $G[A_0 , B_0]$ by edge-disjoint Hamilton cycles.
To do this, we will first decompose $G[A_0 , B_0]$ into a collection of matchings. We will then extend each such matching into a system of
vertex-disjoint paths such that altogether these paths cover every vertex in $G[A_0 , B_0]$, each path has its endvertices in $A\cup B$ and
the path system is $2$-balanced. Since our path system will only contain a
small number of nontrivial paths, we can then extend the path system into a Hamilton cycle (see Lemma~\ref{extendpaths}).  

We will call the path systems we are working with $A_0B_0$-path systems. More precisely, an \emph{$A_0B_0$-path system (with respect to
$(A,A_0,B,B_0)$)} is a path system $Q$ satisfying the following properties:
\begin{itemize}
\item Every vertex in $A_0\cup B_0$ is an internal vertex of a path in $Q$.
\item $A\cup B$ contains the endpoints of each path in $Q$ but no internal vertex of a path in $Q$.
\end{itemize}
The following observation (which motivates the use of the word `balanced') will often be helpful.%
\COMMENT{Deryk added bracket}

\begin{prop} \label{balpathcheck}
Let $A_0,A,B_0,B$ be a partition of a vertex set $V$.
Then an $A_0B_0$-path system $Q$ with $V(Q)\subseteq V$ is $2$-balanced with respect to $(A,A_0,B,B_0)$ if and only if
the number of vertices in $A$ which are endpoints of nontrivial paths
in $Q$ equals the number of vertices in $B$ which are endpoints of nontrivial paths in~$Q$.
\end{prop}
\proof
Note that by definition any $A_0B_0$-path system satisfies (B2), so we only need to consider (B1).
Let $n_A$ be the number of  vertices in $A$ which are endpoints of nontrivial paths
in $Q$ and define $n_B$ similarly.
Let $a:=|A_0|$, $b:=|B_0|$, $A':=A\cup A_0$ and $B':=B\cup B_0$.
Since $d_Q(v) =2 $ for all $v \in A_0$ and since every vertex in $A$ is either an endpoint of a nontrivial path
in $Q$ or has degree zero in~$Q$, we have
\begin{align*}
	2e_Q(A') + e_Q(A',B') = \sum_{v \in A'} d_Q (v) = 2a +n_A.
\end{align*}
So $n_A = 2(e_Q(A')-a) + e_Q(A',B')$, and similarly $n_B = 2(e_Q(B')-b) + e_Q(A',B') $.
Therefore, $n_A = n_B$ if and only if $2( e_Q(A') - e_Q(B')-a +b) = 0 $ if and only if $Q$ satisfies~(B1), as desired.
\endproof

The next observation shows that if we have a suitable path system satisfying (B1), we can extend it into a path system which 
also satisfies (B2).%
\COMMENT{Deryk added new sentence}

\begin{lemma}\label{balpathextend}
Let $0<1/n\ll \alpha \ll 1$.
Let $G$ be a graph on $n$ vertices such that there is a partition $A',B'$ of $V(G)$ which satisfies the following properties:
\begin{itemize}
\item[{\rm (i)}] $A'=A_0\cup A$, $B'=B_0\cup B$ and $A_0,A,B_0,B$ are disjoint;
\item[{\rm (ii)}] $|A|=|B|$ and $a+b \le \alpha n$, where $a:=|A_0|$ and $b:=|B_0|$;
\item[{\rm (iii)}] if $v \in A_0$ then $d_G(v,B) \ge 4 \alpha n$ and if $v \in B_0$ then $d_G(v,A) \ge 4 \alpha  n$.
\end{itemize}
Let $Q'\subseteq G$ be a path system consisting of at most $\alpha n$ nontrivial
paths such that $A \cup B$ contains no internal vertex of a path
in $Q'$ and $e_{Q'}(A')   - e_{Q'}(B')= a -b $.
Then $G$ contains a $2$-balanced $A_0B_0$-path system $Q$ (with respect to $(A,A_0,B,B_0)$) which extends~$Q'$ and
consists of at most $2 \alpha n$ nontrivial paths.
Furthermore, $E(Q)\setminus E(Q') $ consists of $A_0B$- and $AB_0$-edges only.
\end{lemma}
\proof
Since $A \cup B$ contains no internal vertex of a path in $Q'$ and since $Q'$ contains at most $\alpha n$ nontrivial
paths, it follows that at most $2 \alpha n$ vertices in $A \cup B$ lie on nontrivial paths in $Q'$.
We will now extend $Q'$ into an $A_0B_0$-path system $Q$ consisting of at most $a+b +  \alpha n \le 2 \alpha n$ nontrivial paths as follows:
\begin{itemize}
\item for every vertex $v \in A_0$, we join $v$ to $2-d_{Q'}(v)$ vertices in $B$;
\item for every vertex $v \in B_0$, we join $v$ to $2-d_{Q'}(v)$ vertices in $A$.
\end{itemize}
Condition~(iii) and the fact that at most $2\alpha n$ vertices in $A\cup B$ lie on nontrivial paths in $Q'$ together ensure that we can extend
$Q'$ in such a way that the endvertices in $A\cup B$ are distinct for different paths in $Q$.
Note that $e_Q(A')   - e_Q(B') = e_{Q'}(A')   - e_{Q'}(B')= a -b $.
Therefore, $Q$ is $2$-balanced with respect to $(A,A_0,B,B_0)$.
\endproof

The next lemma constructs a small number of $2$-balanced $A_0B_0$-path systems covering the edges of $G[A_0, B_0]$.
Each of these path systems will later be extended%
\COMMENT{osthus corrected new sentence}
into a Hamilton cycle.%
\COMMENT{In Lemma~\ref{coverA0B01}: We really do use tightness of (WF6) but crucially when applying lemma 
in Chapter 8 $G$ is not regular so can't use this property in this lemma.}

\begin{lemma}\label{coverA0B01}
Let $0<1/n\ll \varepsilon \ll  \varepsilon',1/K\ll \alpha\ll 1$.
Let $F$ be a graph on $n$ vertices and let $G$ be a spanning subgraph of $F$.%
	\COMMENT{AL:add condition for $F$ and $G$}
Suppose that $(F,G,A,A_0,B,B_0)$
is an $(\eps,\eps',K,D)$-weak framework with $\delta(F)\ge (1/4+\alpha )n$
and  $D\geq n/200$.
Then for some $r^* \le \eps n$ the graph $G$ contains $r^*$ edge-disjoint
$2$-balanced $A_0B_0$-path systems $Q_1, \dots , Q_{r^*}$ which satisfy the following
properties:
\begin{itemize}
\item[{\rm (i)}] Together $Q_1,\dots,Q_{r^*}$  cover all edges in $G[A_0, B_0]$;
\item[{\rm (ii)}] For each $i \leq r^*$, $Q_i$ contains at most $2 \eps n$ nontrivial paths;
\item[{\rm (iii)}] For each $i \leq r^*$, $Q_i$ does not contain any edge from $G[A,B]$.
\end{itemize}
\end{lemma}
\proof
(WF4) implies that $|A_0|+|B_0|\leq \eps n$. Thus, by Corollary~\ref{basic_matching_dec}, there exists a collection $M'_1,
\dots , M'_{r^*}$ of $r^*$ edge-disjoint matchings in $G[A_0,B_0]$ that together cover all the
edges in $G[A_0,B_0]$, where $r^* \leq \eps n$. 

We may assume that $a \geq b$ (the case when $b>a$ follows analogously).
We will use  edges in $G[A']$ to extend each $M'_i$ into a $2$-balanced $A_0B_0$-path system. 
(WF2) implies that $e_G (A') \geq (a-b)D/2$. Since $d_G (v)=D$ for all $v \in A_0 \cup B_0$ by (WF2),
(WF5) and (WF6) imply that $\Delta (G[A']) \leq D/2$.
Thus Corollary~\ref{basic_matching_dec} implies that $E(G[A'])$ can be decomposed
into $\lfloor D/2 \rfloor+1$ edge-disjoint matchings $M_{A,1}, \dots , M_{A,\lfloor D/2 \rfloor+1}$ such that $||M_{A,i}|-|M_{A,j}|| \leq 1$ for all $i,j \leq \lfloor D/2 \rfloor+1$.

Notice that at least $\eps n$ of the matchings $M_{A,i}$ are such that $|M_{A,i}|\geq a-b$.
Indeed, otherwise we have that 
\begin{align*}
(a-b)D/2 \leq e_G (A') & \leq \eps n (a-b) +(a-b-1)(D/2+1-\eps n) 
\\ & = (a-b)D/2 + a-b-D/2-1+\eps n \\ &< (a-b)D/2 +2\eps n -D/2 < (a-b)D/2 ,
\end{align*}
a contradiction. (The last inequality follows since $D \geq n/200$.)
In particular, this implies that $G[A']$ contains $r^*$ edge-disjoint matchings $M''_1, \dots ,M''_{r^*}$ 
that each consist of precisely $a-b$ edges. 

For each $i \leq r^*$, set $M_i:=M'_i \cup M'' _i$. So for each $ i \leq r^*$, $M_i$ is a path
system consisting of at most $b+(a-b)=a \leq \eps n$ nontrivial paths such that $A \cup B$ contains
no internal vertex of a path in $M_i$ and $e_{M_i} (A')-e_{M_i} (B') =e_{M''_i} (A') =a-b$.

Suppose for some $0 \leq r < r^*$ we have already found a collection $Q_1, \dots , Q_r$ of $r$
edge-disjoint $2$-balanced $A_0B_0$-path systems which satisfy the following properties for
each $i \leq r$:
\begin{itemize}
\item[($\alpha$)$_i$] $Q_i$ contains at most $2\eps n$ nontrivial paths;
\item[($\beta$)$_i$]  $M_i \subseteq Q_i$;
\item[($\gamma$)$_i$] $Q_i$ and $M_j$ are edge-disjoint for each $j\leq r^*$ such that $i \not =j$;
\item[($\delta$)$_i$] $Q_i$ contains no edge from $G[A,B]$.
\end{itemize}
(Note that ($\alpha$)$_0$--($\delta$)$_0$ are vacuously true.)
Let $G'$ denote the spanning subgraph of $G$ obtained from $G$ by  deleting  the edges
lying in $Q_1 \cup \dots \cup Q_r$. (WF2), (WF4) and (WF6) imply that, if $v \in A_0$, $d_{G'} (v,B) \geq D/2-\eps n-2r \geq
4 \eps n$ and if $v \in B_0$ then $d_{G'} (v,A) \geq 4 \eps n$.
Thus Lemma~\ref{balpathextend} implies that $G'$ contains a $2$-balanced $A_0B_0$-path system 
$Q_{r+1}$ that satisfies ($\alpha$)$_{r+1}$--($\delta$)$_{r+1}$.

So we can proceed in this way in order to obtain edge-disjoint $2$-balanced $A_0B_0$-path systems
$Q_1, \dots , Q_{r^*}$ in $G$ such that ($\alpha$)$_{i}$--($\delta$)$_{i}$ hold for each $i \leq r^*$.
Note that (i)--(iii) follow immediately from these conditions, as desired.
\endproof


The next lemma (Corollary~5.4 in~\cite{3con}) allows us to extend a $2$-balanced path system into a Hamilton cycle.
Corollary~5.4 concerns so-called `$(A,B)$-balanced'-path systems rather than $2$-balanced $A_0B_0$-path systems.
But the latter satisfies the requirements of the former by Proposition~\ref{balpathcheck}.%
\COMMENT{Deryk changed sentence and deleted `balanced' in from of $H$ in the lemma below, Daniela deleted $n_0$ in the lemma below}
\begin{lemma}\label{3conlem}
Let $0<1/n\ll   \varepsilon' \ll \alpha \ll 1.$ 
Let $F$ be a graph and suppose that $A_0,A,B_0,B$ is a partition of $V(F)$ such
that $|A|=|B|=n$. Let $H$ be a  bipartite subgraph of $F$
with vertex classes $A$ and $B$ such that $\delta (H) \geq (1/2+\alpha)n$.
Suppose that $Q$ is a $2$-balanced $A_0B_0$-path system with respect to $(A,A_0,B,B_0)$ in $F$ which consists of at most $\eps' n$
nontrivial paths. Then $F$ contains a Hamilton cycle $C$ which satisfies the following properties:
\begin{itemize}
\item $Q \subseteq C$;
\item $E(C)\setminus E(Q) $ consists of edges from $H$.
\end{itemize}

\end{lemma}

Now we can apply Lemma~\ref{3conlem} to  extend a $2$-balanced $A_0B_0$-path system in
a pre-framework into a Hamilton cycle.%
  \COMMENT{AT: I have dropped $\delta (G)\geq D$ condition in this lemma.}

\begin{lemma}\label{extendpaths}
Let $0<1/n\ll \varepsilon \ll  \varepsilon',1/K \ll \alpha \ll 1$.
Let $F$ be a graph on $n$ vertices and let $G$ be a spanning subgraph of $F$.%
	\COMMENT{AL:add condition for $F$ and $G$}
Suppose that $(F,G,A,A_0,B,B_0)$ is an  $(\eps,\eps',K,D)$-pre-framework, i.e.~it
satisfies (WF1)--(WF5).%
   \COMMENT{this exception is mainly needed for applications in later sections}
Suppose also that  $\delta(F) \ge (1/4+\alpha )n$. 
Let $Q$ be a $2$-balanced $A_0B_0$-path system with respect to $(A,A_0,B,B_0)$ in $G$ which consists of at most $\eps' n$
nontrivial paths. Then $F$ contains a Hamilton cycle $C$ which satisfies the following properties:
\begin{itemize}
\item[{\rm (i)}] $Q\subseteq C$;
\item[{\rm (ii)}] $E(C)\setminus E(Q)$ consists of $AB$-edges;
\item[{\rm (iii)}] $C\cap G$ is $2$-balanced with respect to $(A,A_0,B,B_0)$.
\end{itemize}
\end{lemma}
\proof
Note that (WF4), (WF5) and our assumption that $\delta (F) \ge (1/4+\alpha )n$ together imply that
every vertex $x \in A$ satisfies
$$d_{F} (x,B) \geq d_F(x,B')-|B_0|\geq d_F (x) -\eps ' n -|B_0| \geq (1/4+\alpha /2)n
\geq (1/2+ \alpha /2)|B|.$$
Similarly, $d_F (x,A)\geq (1/2+\alpha /2)|A|$ for all $x \in B$. Thus,
$\delta (F[A,B]) \geq (1/2+\alpha /2)|A|$. Applying Lemma~\ref{3conlem} with
$F[A,B]$ playing the role of $H$, we obtain a Hamilton cycle $C$ in $F$ that
satisfies (i) and (ii). To verify (iii), note that (ii) and the 2-balancedness of $Q$ together imply that
$$e_{C \cap G}(A')-e_{C \cap G}(B')=e_Q(A')-e_Q(B')=a-b.$$
Since every vertex $v\in A_0\cup B_0$ satisfies $d_{C\cap G}(v)=d_Q(v)=2$, (iii) holds.
\endproof

We now combine Lemmas~\ref{coverA0B01} and~\ref{extendpaths} to find a collection of edge-disjoint Hamilton cycles covering all the edges
in $G[A_0, B_0]$.%
\COMMENT{AT: Again I have dropped that $G$ is $D$-regular. Need this relaxation for Chapter 8} 
\begin{lemma}\label{coverA0B0}
Let $0<1/n\ll \varepsilon \ll  \varepsilon',1/K\ll \alpha\ll 1$ and let $D\geq n/100$.
Let $F$ be a graph on $n$ vertices and let $G$ be a spanning subgraph of $F$.%
	\COMMENT{AL:add condition for $F$ and $G$}
Suppose that $(F,G,A,A_0,B,B_0)$ is an $(\eps,\eps',K,D)$-weak framework with $\delta(F)\ge (1/4+\alpha )n$. 
Then for some $r^* \le \eps n$ the graph $F$ contains edge-disjoint Hamilton cycles $C_1,\dots,C_{r^*}$ which
satisfy the following properties:
\begin{itemize}
\item[{\rm (i)}] Together $C_1,\dots,C_{r^*}$  cover all edges in $G[A_0, B_0]$;
\item[{\rm (ii)}] $(C_1\cup \dots \cup C_{r^*})\cap G$ is $2r^*$-balanced with respect to $(A,A_0,B,B_0)$.
\end{itemize}
\end{lemma}
\proof
Apply Lemma~\ref{coverA0B01} to obtain a collection of $r^* \leq \eps n$ edge-disjoint $2$-balanced
$A_0B_0$-path systems $Q_1, \dots , Q_{r^*}$ in $G$ which satisfy Lemma~\ref{coverA0B01}(i)--(iii).
We will extend each $Q_i$ to a Hamilton cycle $C_i$.

Suppose that for some $0 \le r < r^*$ we have found a collection $C_1,\dots,C_r$ of $r$ edge-disjoint Hamilton
cycles in $F$ such that the following holds for each $0 \leq i \leq r$:
\begin{itemize}
\item[($\alpha$)$_i$] $Q_i \subseteq C_i$;
\item[($\beta$)$_i$] $E(C_i) \setminus E(Q_i)$ consists of $AB$-edges;
\item[($\gamma$)$_i$] $G \cap C_i $ is $2$-balanced with respect to $(A,A_0,B,B_0)$.
\end{itemize}
(Note that ($\alpha$)$_0$--($\gamma$)$_0$ are vacuously true.)
Let $H_r:=C_1\cup \dots \cup C_r$ (where $H_0:=(V(G),\emptyset)$). So $H_r$ is $2r$-regular. 
Further, since  $G \cap C_i$ is $2$-balanced for each $i \leq r$, $G \cap H_r$ is $2r$-balanced. Let
$G_r:=G-H_r$ and $F_r:=F-H_r$. Since $(F,G,A,A_0,B,B_0)$ is an $(\eps,\eps',K,D)$-pre-framework,
Proposition~\ref{WFpreserve} implies that $(F_r,G_r,A,A_0,B,B_0)$ is an $(\eps,\eps',K,D-2r)$-pre-framework. 
Moreover, $\delta (F_r) \geq \delta (F)-2r \geq (1/4+\alpha/2)n$.
Lemma~\ref{coverA0B01}(iii) and ($\beta$)$_1$--($\beta$)$_r$ together%
    \COMMENT{Daniela added more detail}
imply that $Q_{r+1}$ lies in $G_r$.
Therefore, Lemma~\ref{extendpaths} implies that $F_r$ contains a Hamilton cycle
$C_{r+1}$ which satisfies ($\alpha$)$_{r+1}$--($\gamma$)$_{r+1}$.

So we can proceed in this way in order to obtain $r^*$ edge-disjoint Hamilton cycles 
$C_1, \dots , C_{r^*}$ in $F$ such that for each $i\leq r^*$, ($\alpha$)$_{i}$--($\gamma$)$_{i}$
hold. Note that this implies that (ii) is satisfied. Further, the choice of $Q_1, \dots, Q_{r^*}$
ensures that (i) holds.
\endproof

Given a graph $G$, we say that $(G,A,A_0,B,B_0)$ is an \emph{$(\eps,\eps',K,D)$-bi-framework} if the following holds,
where $A':=A_0\cup A$, $B':=B_0\cup B$ and $n:=|G|$:
\begin{itemize}
\item[{\rm (BFR1)}]  $A,A_0,B,B_0$ forms a partition of $V(G)$;
\item[(BFR2)] $G$ is $D$-balanced with respect to $(A,A_0,B,B_0)$;
\item[{\rm (BFR3)}] $e_G(A'), e_G(B')\le \varepsilon n^2$;
\item[{\rm (BFR4)}] $|A|=|B|$ is divisible by $K$. Moreover, $b\le a$ and $a+b\le \eps n$, where $a:=|A_0|$ and $b:=|B_0|$;
\item[{\rm (BFR5)}] all vertices in $A \cup B$ have internal degree at most $\varepsilon' n$ in $G$;
\item[{\rm (BFR6)}] $e(G[A_0, B_0])=0$;
\item[{\rm (BFR7)}] all vertices $v\in V(G)$ have internal degree at most $d_G(v)/2+\eps n$ in $G$.
\end{itemize}
Note that the main differences to a weak framework are (BFR6) and the fact that a weak framework involves an additional graph $F$.
In particular (BFR1)--(BFR4) imply (WF1)--(WF4).%
\COMMENT{Deryk: additional 2 sentences}
Suppose that $\eps_1\ge \eps$, $\eps'_1\ge \eps'$ and that $K_1$ divides $K$.
Then note that every $(\eps,\eps',K,D)$-bi-framework is also an $(\eps_1,\eps'_1,K_1,D)$-bi-framework.%
\COMMENT{AL:removed `$(F,G,A,A_0,B,B_0)$ is an $(\eps,\eps',K,D)$-weak framework, thus $F$ satisfies (WF5)
with respect to the partition $A,A_0,B,B_0$.' in the next lemma}

\begin{lemma}\label{coverA0B02}
Let $0<1/n\ll \varepsilon \ll \varepsilon',1/K\ll \alpha \ll 1$ and let
$ D \ge n/100$.
Let $F$ be a graph on $n$ vertices and let $G$ be a spanning subgraph of $F$.%
	\COMMENT{AL:add condition for $F$ and $G$}
Suppose that $(F,G,A,A_0,B,B_0)$ is an $(\eps,\eps',K,D)$-weak framework. Suppose also that $\delta(F) \ge (1/4+\alpha)n$ and $|A_0|\geq |B_0|$.
 Then the following properties hold:
\begin{itemize}
\item[{\rm (i)}] there is an $(\eps,\eps',K,D_{G'})$-bi-framework
$(G',A,A_0,B,B_0)$ such that $G'$ is a spanning subgraph of $G$ with
$D_{G'} \ge D -2\eps n$;%
\COMMENT{AL: removed. `and such that $F$ satisfies (WF5)
(with respect to the partition $A,A_0,B,B_0$)'.}
\item[{\rm (ii)}] there is a set of $(D-D_{G'})/2 \leq  \eps  n$
edge-disjoint Hamilton cycles in $F-G'$ containing all edges of $G-G'$.
In particular, if $D$ is even then $D_{G'}$ is even.%
\COMMENT{we do need the explicit statement of the number of ham cycles here.}
\end{itemize}
\end{lemma}
\proof
Lemma~\ref{coverA0B0} implies that there exists some $r^* \leq \eps  n$ such that
$F$ contains a spanning subgraph $H$ satisfying the following properties:
\begin{itemize}
\item[(a)] $H$ is $2r^*$-regular;
\item[(b)] $H$ contains all the edges in $G[A_0,B_0]$;
\item[(c)] $G \cap H$ is $2r^*$-balanced with respect to $( A,A_0,B,B_0)$;
\item[(d)] $H$ has a decomposition into $r^*$ edge-disjoint Hamilton cycles.
\end{itemize}

Set $G':=G-H$. Then $(G',A,A_0,B,B_0)$ is an $(\eps,\eps',K,D_{G'})$-bi-framework
where $D_{G'} :=D-2r^* \geq D- 2\eps n$. Indeed, since $(F,G,A,A_0,B,B_0)$ is an $(\eps ,\eps',K,$ $D)$-weak framework,
(BFR1) and (BFR3)--(BFR5) follow from (WF1) and (WF3)--(WF5). Further, (BFR2) follows
from (WF2) and (c) while (BFR6) follows from (b). (WF6) implies that all
vertices $v \in V(G)$ have internal degree at most $d_G (v)/2$ in $G$.
Thus  all vertices $v \in V(G')$ have internal degree at most $d_G (v)/2\leq (d_{G'} (v)+2r^*)/2
\leq d_{G'} (v)/2 +\eps n$ in $G'$. So (BFR7) is satisfied. Hence, (i) is satisfied.

Note that by definition of $G'$, $H$ contains all edges of $G-G'$. So since 
$r^* =(D-D_{G'})/2 \leq \eps n$, (d) implies (ii).
\endproof

The following result follows immediately from Lemmas~\ref{bip_decomp} and~\ref{coverA0B02}.%
\COMMENT{AT: It is convenient to state both Lemma~\ref{coverA0B02} and Corollary~\ref{coverA0B02c}.
We apply Lemma~\ref{coverA0B02} in a situation where $G$ is not regular. But elsewhere it is handy
to apply Corollary~\ref{coverA0B02c} instead of first applying Lemma~\ref{bip_decomp} and then
Lemma~\ref{coverA0B02}.}
\begin{cor}\label{coverA0B02c}
Let $0<1/n\ll \varepsilon \ll  \eps^*\ll \varepsilon',1/K\ll \alpha \ll 1$ and let
$ D \ge n/100$. Suppose that $F$ is an $\eps$-bipartite graph on $n$ vertices with $\delta(F) \ge (1/4+\alpha)n$.%
	\COMMENT{AL: added $n$ vertices.}
Suppose that $G$ is a $D$-regular spanning subgraph of $F$. Then the following properties hold:
\begin{itemize}
\item[{\rm (i)}] there is an $(\eps^*,\eps',K,D_{G'})$-bi-framework%
\COMMENT{Deryk: it seems odd to have $\eps^*$ here and $\eps ^{1/3}$ in the other two bounds? Daniela: yes, but in this way
we don't have to say that a $(\eps^{1/3},\eps',K,D_{G'})$-bi-framework is a $(\eps^*,\eps',K,D_{G'})$-bi-framework later one.}
$(G',A,A_0,B,B_0)$ such that $G'$ is a spanning subgraph of $G$,
$D_{G'} \ge D -2\eps^{1/3}n$ and such that $F$ satisfies (WF5)
(with respect to the partition $A,A_0,B,B_0$);
\item[{\rm (ii)}] there is a set of $(D-D_{G'})/2 \leq  \eps ^{1/3} n$
edge-disjoint Hamilton cycles in $F-G'$ containing all edges of $G-G'$.
In particular, if $D$ is even then $D_{G'}$ is even.%
\COMMENT{we do need
the explicit statement of the number of ham cycles here. Further, note that this implies that $\delta (G') \geq D_{G'}$}
\end{itemize}
\end{cor}

\section[Finding Path Systems which Cover All the Edges within Classes]{Finding Path Systems which Cover All the Edges within the Classes}\label{findBES}

The purpose of this section is to prove Corollary~\ref{BEScor} which, given a bi-framework $(G,A,A_0,B,B_0)$,
guarantees a set $\cC$ of edge-disjoint Hamilton cycles and a set $\cJ$ of suitable edge-disjoint
$2$-balanced $A_0B_0$-path systems such that the graph $G^*$ obtained from $G$ by deleting the edges in
all these Hamilton cycles and path systems is bipartite with vertex classes $A'$ and $B'$ and
$A_0\cup B_0$ is isolated in $G^*$. Each of the path systems in $\cJ$ will later be extended into
a Hamilton cycle by adding suitable edges between $A$ and~$B$.
The path systems in $\cJ$ will need to be `localized' with respect to a given partition.
We prepare the ground for this in the next subsection.

We will call the path systems in $\cJ$ balanced exceptional systems (see Section~\ref{besconstruct} for the definition). These will play a similar role as the
exceptional systems in the two cliques case (i.e. in Chapter~\ref{paper1}).

Throughout this section, given sets $S,S'\subseteq V(G)$ we often write $E(S)$, $E(S,S')$, $e(S)$ and $e(S,S')$ for
$E_G(S)$, $E_G(S,S')$, $e_G(S)$ and $e_G(S,S')$ respectively.

\subsection{Choosing the Partition and the Localized Slices}\label{partition}  

Let $K,m\in\mathbb{N}$ and $\eps>0$.
Recall that a \emph{$(K,m,\eps)$-partition} of a set $V$ of vertices is a partition of $V$ into sets $A_0,A_1,\dots,A_K$
and $B_0,B_1,\dots,B_K$ such that $|A_i|=|B_i|=m$ for all $1 \le i \le K $%
\COMMENT{AL: replaced $i \ge 1$ with $1 \le i \le K $}
 and $|A_0\cup B_0|\le \eps |V|$.
We often write $V_0$ for $A_0\cup B_0$ and think of the vertices in $V_0$ as `exceptional vertices'.
The sets $A_1,\dots,A_K$ and $B_1,\dots,B_K$ are called \emph{clusters} of the $(K,m,\eps_0)$-partition
and $A_0$, $B_0$ are called \emph{exceptional sets}. Unless stated otherwise,
when considering a $(K,m,\eps)$-partition $\mathcal P$ we denote the elements
of $\mathcal P$ by $A_0,A_1,\dots,A_K$ and $B_0,B_1,\dots,B_K$ as above.
Further, we will often write $A$ for $A_1 \cup \dots \cup A_K$ and 
$B$ for $B_1 \cup \dots \cup B_K$.

Suppose that $(G,A,A_0,B,B_0)$ is an $(\eps,\eps',K,D)$-bi-framework with $|G|=n$ and that
$\eps _1, \eps _2 >0$.
We say that $\cP$ is a \emph{$(K,m,\eps,\eps_1,\eps_2)$-partition for $G$} if 
$\cP$ satisfies the following properties:
\begin{itemize}
\item[(P1)] $\cP$ is a $(K,m,\eps)$-partition of $V(G)$
such that the exceptional sets $A_0$ and $B_0$ in the partition $\cP$ are the same as the
sets $A_0$, $B_0$ which are part of the bi-framework $(G,A,A_0,B,B_0)$. In particular, $m=|A|/K=|B|/K$;
\item[(P2)] $d(v, A_i)=(d(v, A)\pm \eps_1 n)/K$  for all $1 \le i\le K$ and $v\in V(G)$;
\item[(P3)] $e(A_i, A_j)=2(e(A) \pm  \eps_2 \max\{n, e(A)\} )/K^2$ for all $1\le i < j\le K$;
\item[(P4)] $e(A_i)=(e(A) \pm \eps_2\max\{n, e(A)\})/K^2$ for all $1 \le i\le K$;
\item[(P5)] $e(A_0, A_i)=(e(A_0, A)\pm \eps_2\max\{n, e(A_0, A)\})/K$ for all $1 \le i\le K$;
\item[(P6)] $e(A_i, B_j)=(e(A, B) \pm 3\eps_2 e(A, B))/K^2$ for all $1 \le i, j \le K$;%
   \COMMENT{Previously had $\max\{n, e(A, B)\}$ instead of just $e(A, B)$. But since we are always assuming that $D$ is
linear, we don't need the max here.}
\end{itemize}
and the analogous assertions hold if we replace $A$ by $B$ (as well as $A_i$ by $B_i$ etc.) in (P2)--(P5).\COMMENT{Check we never need e.g. a mixed condition
like $e(A_0^s, B_i)=...$ in (P5)}

Our first aim is to show that for every bi-framework we can find such a partition with suitable parameters.

\begin{lemma}\label{part} Let
$0<1/n\ll \eps\ll \eps'\ll \eps_1\ll \eps_2\ll 1/K\ll 1$.
Suppose that $(G,A,A_0,B,B_0)$ is an $(\eps,\eps',K,D)$-bi-framework with $|G|=n$ and 
$\delta (G) \geq D\ge n/200$. 
Suppose that $F$ is a graph with $V(F)=V(G)$. Then there exists a partition 
$\cP=\{A_0,A_1, \dots , A_K, B_0, B_1, \dots , B_K \}$ of $V(G)$ so that 
\begin{itemize}
\item[{\rm (i)}] $\cP$ is a $(K,m,\eps,\eps_1,\eps_2)$-partition for $G$.
\item[{\rm (ii)}] $d_F(v, A_i)=(d_F(v, A)\pm \eps_1 n)/K$ and $d_F(v, B_i)=(d_F(v, B)\pm \eps_1 n)/K$ for all $1 \le i\le K$ and $v\in V(G)$.%
\COMMENT{Deryk rephrased this}
\end{itemize}
\end{lemma}
\proof In order to find the required partitions $A_1,\dots,A_K$ of $A$ and $B_1,\dots,B_K$ of $B$
we will apply Lemma~\ref{lma:partition2} twice, as follows. 
In the first application we let $U:=A,$ $R_1:=A_0$,   $R_2:=B_0$ and
$R_3:= B$. 
 Note that
$\Delta(G[U]) \le  \eps' n $ by (BFR5) and $d_G(u,R_j)\le |R_j|\le \eps n\le \eps' n$ for all $u\in U$ and $j=1,2$ by (BFR4).
Moreover, (BFR4) and (BFR7) together imply that $d_G(x,U)\ge D/3\ge \eps' n$ for each $x\in R_3=B$.
Thus we can apply Lemma~\ref{lma:partition2} with $\eps'$ playing the role of $\eps$ to obtain a partition 
$U_1, \dots, U_K$ of $U$. We let $A_i:=U_i$  for all $i\le K.$
Then the $A_i$ satisfy (P2)--(P5) and
\begin{equation}\label{eq:eAiBz}
e_G(A_i, B)=(e_G(A, B) \pm \eps_2 \max\{n, e_G(A, B)\})/K=(1\pm \eps_2)e_G(A,B)/K.
\end{equation}
Further, Lemma~\ref{lma:partition2}(vi) implies that
$$d_F(v, A_i)=(d_F(v, A)\pm \eps_1 n)/K$$  for all $1 \le i\le K$ and $v\in V(G)$.

For the second application of Lemma~\ref{lma:partition2} we let $U:=B,$ $R_1:=B_0$, $R_2:=A_0$ 
and $R_j:=A_{j-2}$ for all $3 \le j \le K+2$. As before, $\Delta(G[U]) \le  \eps' n $ by (BFR5)
and $d_G(u,R_j)\le \eps n\le \eps' n$ for all $u\in U$ and  $j=1,2$ by (BFR4).
Moreover, (BFR4) and (BFR7) together imply that $d_G(x,U)\ge D/3\ge \eps' n$ for all $3 \le j \le K+2$ and each $x\in R_j=A_{j-2}$.
Thus we can apply Lemma~\ref{lma:partition2} with $\eps'$ playing the role of $\eps$ to obtain a partition 
$U_1, \dots, U_K$ of $U$. Let $B_i:=U_i$  for all $i \le K.$ Then the $B_i$ satisfy (P2)--(P5)
with $A$ replaced by~$B$, $A_i$ replaced by $B_i$, and so on. Moreover, for all $1 \leq i,j \leq K$,
\begin{eqnarray*}
e_G(A_i, B_j) & = & (e_G(A_i, B) \pm \eps_2 \max\{n, e_G(A_i, B)\})/K\\
& \stackrel{(\ref{eq:eAiBz})}{=} & ((1\pm \eps_2)e_G(A,B) \pm \eps_2 (1+ \eps_2)e_G(A,B))/K^2\\
& = & (e_G(A, B) \pm 3\eps_2 e_G(A, B))/K^2,
\end{eqnarray*}
i.e.~(P6) holds. Since clearly (P1) holds as well, $A_0,A_1,\dots,A_K$ and $B_0,B_1,\dots,B_K$
together form a $(K,m,\eps,\eps_1,\eps_2)$-partition for $G$.
Further, Lemma~\ref{lma:partition2}(vi) implies that
$$d_F(v, B_i)=(d_F(v, B)\pm \eps_1 n)/K$$  for all $1 \le i\le K$ and $v\in V(G)$.
\endproof

The next lemma gives a decomposition of $G[A']$ and $G[B']$ into suitable smaller edge-disjoint subgraphs $H_{ij}^A$ and $H_{ij}^B$.
We say that the graphs $H_{ij}^A$ and $H_{ij}^B$ guaranteed by Lemma~\ref{rnd_slice} are \emph{localized slices} of~$G$.
Note that the order of the indices $i$ and $j$ matters here, i.e.~$H^A_{ij}\neq H^A_{ji}$.
Also, we allow $i=j$. %
\COMMENT{Deryk added this and osthus added lower bound $1 \le i,j\le$}

\begin{lemma}\label{rnd_slice}
Let $0<1/n\ll \eps\ll \eps'\ll \eps_1\ll \eps_2\ll 1/K\ll 1$.
Suppose that $(G,A,A_0,B,B_0)$ is an $(\eps,\eps',K,D)$-bi-framework with $|G|=n$ and $D\ge n/200$.
Let $A_0,A_1,\dots,A_K$ and $B_0,B_1,\dots,B_K$ be a $(K,m,\eps,\eps_1,\eps_2)$-partition for $G$. 
Then for all $1 \le i,j\le K$ there are graphs $H_{ij}^A$ and $H_{ij}^B$ with the following properties:
\begin{enumerate}
\item[{\rm (i)}] $H_{ij}^A$ is a spanning subgraph of  $G[A_0, A_i\cup A_j]\cup G[A_i,A_j]\cup G[A_0]$; 
\item[{\rm (ii)}] The sets $E(H_{ij}^A)$ over all $1 \le i,j\le K$ form a partition of the edges of $G[A']$;
\item[{\rm (iii)}] $e(H_{ij}^A)=(e(A') \pm 9\eps_2\max\{n, e(A')\})/K^2$ for all $1 \le i, j \le K$;
\item[{\rm (iv)}] $e_{H_{ij}^A}(A_0, A_i\cup A_j)=(e(A_0, A)\pm 2\eps_2\max\{n, e(A_0, A)\})/K^2$ for all $1 \le i, j \le K$;
\item[{\rm (v)}] $e_{H_{ij}^A}(A_i, A_j)=(e(A)\pm  2\eps_2 \max\{n,e(A)\})/K^2$ for all $1 \le i, j \le K$;
\item[{\rm (vi)}] 
For all $1 \le i, j \le K$ and all $v\in A_0$ we have $d_{H_{ij}^A}(v)=d_{H_{ij}^A}(v, A_i\cup A_j)+d_{H_{ij}^A}(v, A_0) =(d(v, A)\pm 4\eps_1 n)/K^2$.
\end{enumerate} 
The analogous assertions hold if we replace $A$ by $B$, $A_i$ by $B_i$, and so on.
\end{lemma}
\proof In order to construct the graphs $H^A_{ij}$
we perform the following procedure:
\begin{itemize}
\item Initially each $H_{ij}^A$ is an empty graph with vertex set $A_0\cup A_i\cup A_j$.
\item For all $1 \le i\le K$ choose a random partition $E(A_0, A_i)$ into $K$ sets $U_j$ of equal size
and let $E(H_{ij}^A):=U_j$. 
(If $E(A_0, A_i)$ is not divisible by $K$, first distribute up to $K-1$ edges arbitrarily among the $U_j$ to achieve divisibility.)%
\COMMENT{We need to mention divisibility as the number of edges may be small.}
\item For all $i\le K$, we add all the edges in $E(A_i)$ to $H_{ii}^A$. 
\item For all $i, j\le K$ with $i\neq j$, half of the edges in $E(A_i, A_j)$ are added to $H_{ij}^A$ and
the other half is added to  $H_{ji}^A$ (the choice of the edges is arbitrary). 
\item The edges in $G[A_0]$ are distributed equally amongst the $H^A _{ij}$.
(So $e_{H^A _{ij}} (A_0)$ $= e(A_0)/K^2 \pm 1$.)
\end{itemize}
Clearly, the above procedure ensures that properties~(i) and (ii) hold.
(P5) implies (iv) and (P3) and (P4) imply (v).%
\COMMENT{The `2' could actually be omitted as the partition is now exact (if divisibility holds) all the numbers are large, so divisibility doesn't affect the numbers.
But it's probably easier to see as it is.}

Consider any $v \in A_0$.
To prove (vi), note that we may assume that $d(v,A) \ge  \eps_1 n/K^2$.
Let $X:=d_{H_{ij}^A}(v, A_i \cup A_j)$. Note that (P2) implies that%
    \COMMENT{Daniela replaced lots of $\ex [X]$ by $\ex (X)$}
$\ex (X)=(d(v, A)\pm 2\eps_1 n)/K^2$ and note that $\ex(X) \le n$. 
So the Chernoff-Hoeffding bound for the hypergeometric distribution in Proposition~\ref{chernoff} implies that
$$
\prob (|X-\ex (X)| > \eps_1 n/K^2) \le \prob (|X- \ex (X)| > \eps_1 \ex (X)/K^2 ) \le 2 e^{-\eps_1^2 \ex (X)/3K^4} \le 1/n^2. 
$$
Since $d_{H_{ij}^A}(v, A_0)\leq |A_0|\leq \eps _1 n/K^2$,
 a union bound implies the desired result.
Finally, observe that for any $a,b_1,\dots,b_4>0$, we have
$$
\sum_{i=1}^4 \max\{a,b_i\}\le 4\max\{a,b_1,\dots,b_4\} \le 4\max\{a,b_1+\dots +b_4\}.
$$
So (iii) follows from~(iv),~(v) and the fact that $e_{H^A _{ij}} (A_0) = e(A_0)/K^2 \pm 1$.
\endproof
Note that the construction implies that if $i\neq j$, then $H^A_{ij}$ will  contain edges between $A_0$ and $A_i$
but not between $A_0$ and $A_j$. However, this additional information is not needed in the subsequent argument.

\subsection{Decomposing the Localized Slices}\label{sec:slice} 
Suppose that $(G,A,A_0,B,B_0)$ is an $(\eps,\eps',K,D)$-bi-framework. 
Recall that $a=|A_0|$, $b=|B_0|$ and $a\ge b$. Since $G$ is $D$-balanced by~(BFR2),
we have $e(A')-e(B')=(a-b)D/2.$ 
So there are an integer $q \ge -b$ and a constant $0\le c<1$ such that
\begin{equation} \label{aqc}
e(A')=(a+q+c)D/2 \ \ \mbox{ and } \ \ e(B')=(b+q+c)D/2.
\end{equation} 
The aim of this subsection is to prove Lemma~\ref{matchingdec}, 
which guarantees a decomposition of each localized slice $H_{ij}^A$ into path systems
(which will be extended into $A_0B_0$-path systems in Section~\ref{besconstruct})%
\COMMENT{osthus modified this. Andy: deleted full stop}
and a sparse (but not too sparse) leftover graph $G_{ij}^A$.%
\COMMENT{AL: added `localized $A_0B_0$-'}

The following two results will be used in the proof of Lemma~\ref{matchingdec}.
\begin{lemma}\label{simple}
Let $0<1/n \ll \alpha, \beta , \gamma $ so that $\gamma<1/2$.  Suppose that $G$ is a graph on $n$ vertices such that
$\Delta (G) \leq \alpha n$ and $e(G) \geq \beta n$. Then $G$ contains a spanning subgraph
$H$ such that $e(H)= \lceil (1- \gamma) e(G)\rceil$ and $\Delta (G-H) \leq 6 \gamma \alpha n/5$.
\end{lemma}
\proof
Let $H'$ be a spanning subgraph of $G$ such that
\begin{itemize}
\item $\Delta (H') \leq 6\gamma \alpha n/5$;
\item $e(H') \geq \gamma e(G)$.
\end{itemize}
To see that such a graph $H'$ exists, consider a random subgraph of $G$ obtained by including each edge of
$G$ with probability $11 \gamma /10$.
Then $\mathbb E (\Delta (H')) \leq 11 \gamma \alpha n/10$ and  
$\mathbb E (e (H')) = 11 \gamma e(G)/10$. Thus applying Proposition~\ref{chernoff} we have that,
with high probability, $H'$ is as desired.%
   \COMMENT{We apply Chernoff in the usual way here.
For example, consider any $v \in V(G)$. If $d_G (v) \leq 6\gamma \alpha n/5$ then certainly $d_{H'} (v) \leq 6 \gamma \alpha n/5$.
Otherwise,
\begin{align*}
\mathbb P (d_{H'} (v) \geq 6\gamma \alpha n/5) & \leq \mathbb P \left(|d_{H'} (v) -\frac{11\gamma}{10} d_G (v)|\geq \frac{1}{11} 
\left ( \frac{11\gamma}{10} d_G (v) \right ) \right ) \\ &
\leq 2 \exp{\{-\frac{1}{363} \times \frac{11\gamma}{10} d_G (v)\} } \leq 2 
\exp{\{\frac{-11\gamma}{3630} \times 6\gamma \alpha n/5\} } \ll 1/n.
\end{align*}}

Define $H$ to be a spanning subgraph of $G$ such that $H\supseteq G-H'$%
	\COMMENT{AL: replaced $\subseteq$ with $\supseteq$.}
 and $e(H)= \lceil (1- \gamma) e(G)\rceil$.
Then $\Delta (G-H)\leq \Delta (H') \leq 6 \gamma \alpha n/5$, as required.
\endproof

\begin{lemma}\label{splittrick}
Suppose that $G$ is a graph such that $\Delta (G) \leq D-2$ where $D \in \mathbb N$ is even.
Suppose $A_0,A$ is a partition of $V(G)$ such that $d_G(x) \leq D/2-1$ for all $x \in A$ and
$\Delta (G[A_0]) \leq D/2-1$. 
Then $G$ has a decomposition into $D/2$ edge-disjoint path
systems $P_1, \dots , P_{D/2}$ such that the following conditions hold:
\begin{itemize}
\item[{\rm (i)}] For each $i \leq D/2$, any internal vertex on a path in $P_i$ lies in $A_0$;
\item[{\rm (ii)}] $|e(P_i)-e(P_j)|\leq 1 $ for all $i,j \leq D/2$.
\end{itemize}
\end{lemma}
\proof
Let $G_1$ be a maximal spanning subgraph of $G$ under the constraints that $G[A_0] \subseteq G_1$
and $\Delta (G_1) \leq D/2-1$.
Note that $G[A_0]\cup G[A] \subseteq G_1$. Set $G_2 := G-G_1$. So $G_2$ only contains $A_0A$-edges.
Further, since $\Delta (G) \leq D-2$, the maximality of $G_1$ implies that $\Delta (G_2) \leq
D/2-1$.

Define an auxiliary graph $G'$, obtained from $G_1$ as follows:
let $A_0=\{a_1, \dots , $ $a_m\}$. Add a new vertex set $A'_0=\{ a'_1, \dots , a'_m\}$ to $G_1$.
For each $i \leq m$ and $x \in A$, we add an edge between $a'_i$ and $x$ if and only if $a_i x$
is an edge in $G_2$.

Thus $G'[A_0 \cup A]$ is isomorphic to $G_1$ and $G'[A'_0 , A]$ is isomorphic to $G_2$.
By construction and since $d_G (x) \leq D/2-1$ for all $x \in A$,  we have that $\Delta (G') \leq D/2-1$. Hence, Corollary~\ref{basic_matching_dec} implies
that $E(G')$ can be decomposed into $D/2$ edge-disjoint matchings $M_1, \dots , M_{D/2}$ such that
$||M_i|-|M_j||\leq 1$ for all $i,j \leq D/2$.

By identifying each vertex $a'_i \in A'_0$ with the corresponding vertex $a_i \in A_0$,
$M_1, \dots , M_{D/2}$ correspond to edge-disjoint subgraphs $P_1, \dots, P_{D/2}$ of $G$ such that
\begin{itemize}
\item $P_1, \dots, P_{D/2}$ together cover all the edges in $G$;
\item  $|e(P_i)-e(P_j)|\leq 1 $ for all $i,j \leq D/2$.
\end{itemize}
Note that $d_{M_i} (x) \leq 1$ for each $x \in V(G')$. Thus $d_{P_i} (x) \leq 1$ for each $x \in A$ and $d_{P_i} (x) \leq 2$ for each $x \in A_0$. This implies that any cycle in $P_i$ must
lie in $G[A_0]$. However, $M_i$ is a matching and $G'[A'_0]\cup G'[A_0,A'_0]$ contains no edges.%
   \COMMENT{Daniela replaced $G'[A'_0]$ by $G'[A'_0]\cup G'[A_0,A'_0]$}
Therefore, $P_i$ contains no cycle,
and so $P_i$ is a path system such that any internal vertex on a path in $P_i$ lies in $A_0$.
Hence $P_1, \dots , P_{D/2}$ satisfy (i) and (ii).
\endproof

\begin{lemma} \label{matchingdec}
Let $0<1/n\ll \eps\ll \eps'\ll \eps_1\ll \eps_2\ll \eps_3\ll \eps_4\ll 1/K\ll 1$.
Suppose that $(G,A,A_0,B,B_0)$ is an $(\eps,\eps',K,D)$-bi-framework with $|G|=n$ and $D\ge n/200$.
Let $A_0,A_1,\dots,A_K$ and $B_0,B_1,\dots,B_K$ be a $(K,m,\eps,\eps_1,\eps_2)$-partition for $G$. 
Let $H_{ij}^A$ be a localized slice of $G$ as guaranteed by Lemma~\ref{rnd_slice}.
Define $c$ and $q$ as in~(\ref{aqc}).
Suppose that $t:=(1-20\eps_4)D/2K^2\in \mathbb{N}$.%
   \COMMENT{Later one we will also require that $t$ is divisible by $K^2$.}
If $e(B')\geq \eps _3 n$, set $t^*$ to be the largest integer which is at most $ct$ and is divisible by $K^2$. 
Otherwise, set $t^*:=0$.
Define
$$
\ell _a :=
\left\{
	\begin{array}{ll}
		0 & \mbox{if } e(A') < \eps _3 n ;\\
        a-b & \mbox{if } e(A') \geq \eps _3 n \mbox{ but } e(B') < \eps_ 3 n;\\
		a+q+c & \mbox{otherwise}  
	\end{array}
\right.
$$
and 
$$
\ell _b :=
\left\{
	\begin{array}{ll}
		0 & \mbox{if } e(B') < \eps _3 n ;\\
		b+q+c & \mbox{otherwise.}  
	\end{array}
\right.
$$
Then 
$H_{ij}^A$ has a decomposition into $t$ edge-disjoint path systems $P_1,\dots,P_{t}$ and a spanning subgraph $G_{ij}^A$
with the following properties:
\begin{itemize}
\item[{\rm (i)}] For each $s \leq t$, any internal vertex on a path in $P_s$ lies in $A_0$;
\item[{\rm (ii)}] $e(P_1)=\dots = e(P_{t^*}) =\lceil \ell_a \rceil$ and $e(P_{t^*+1}) = \dots =e(P_{t})= \lfloor \ell_a \rfloor$; 
\item[{\rm (iii)}] $e(P_s)\le \sqrt{\eps} n$ for every $s\le t$;
\item[{\rm (iv)}] $\Delta(G_{ij}^A)\le 13\eps_4 D/K^2$. 
\end{itemize}
The analogous assertion (with $\ell_a$ replaced by $\ell_b$ and $A_0$ replaced by $B_0$) holds for each localized slice~$H_{ij}^B$ of~$G$.
Furthermore, $\lceil \ell _a \rceil -\lceil \ell _b \rceil= \lfloor \ell _a \rfloor -\lfloor \ell _b \rfloor =a-b$.
\end{lemma}
\proof Note that~(\ref{aqc}) and  (BFR3) together imply that $\ell_a D/2\leq (a+q+c)D/2= e(A') \le \eps n^2$ and so
$\lceil \ell_a\rceil \le \sqrt{\eps} n$. Thus~(iii) will follow from~(ii). So it remains to prove~(i), (ii) and~(iv). We split the proof into three cases.

\medskip

\noindent
{\bf Case 1. $e(A') < \eps _3 n$}

(BFR2) and (BFR4) imply that $e(A')-e(B') =(a-b)D/2 \geq 0$. So $e(B') \leq e(A') < \eps _3 n$.
Thus $\ell _a =\ell _b =0$. Set $G^A _{ij}:=H^A _{ij}$ and $G^B _{ij}:=H^B _{ij}$.
Therefore, (iv) is satisfied as $\Delta(H_{ij}^A) \leq e(A') < \eps _3 n \leq 13\eps_4 D/K^2$.
Further, (i) and (ii) are vacuous (i.e. we set each $P_s$ to be the empty graph on $V(G)$). 

Note that $a=b$ since otherwise $a>b$ and therefore (BFR2) implies that $e(A') \geq (a-b)D/2 \geq D/2 > \eps_3 n$, a contradiction.
Hence, $\lceil \ell _a \rceil -\lceil \ell _b \rceil= \lfloor \ell _a \rfloor -\lfloor \ell _b \rfloor =0=a-b$.

\medskip

\noindent
{\bf Case 2.} $e(A') \geq \eps _3 n$ and $e(B') < \eps _3 n$

Since $\ell _b=0$ in this case, we set $G^B _{ij}:=H^B _{ij}$ and each $P_s$ to be the empty graph on $V(G)$. Then as in Case~1, 
(i), (ii) and (iv) are satisfied with respect to $H^B _{ij}$. Further, clearly 
$\lceil \ell _a \rceil -\lceil \ell _b \rceil= \lfloor \ell _a \rfloor -\lfloor \ell _b \rfloor =a-b$.

Note that $a>b$ since otherwise $a=b$ and thus $e(A')=e(B')$ by (BFR2), a contradiction
to the case assumptions. Since $e(A') -e(B') =(a-b)D/2$ by (BFR2), Lemma~\ref{rnd_slice}(iii)
implies that 
\begin{align}
\nonumber
e (H^A _{ij})& \geq (1-9 \eps _2) {e(A')}/{K^2}-9 \eps _2 n /K^2  \geq 
(1-9 \eps _2) (a-b)D/(2K^2) - 9 \eps _2 n /K^2 \\ 
& \geq (1- \eps _3)(a-b) D/(2K^2) > (a-b)t. \label{gambound}
\end{align}
Similarly, Lemma~\ref{rnd_slice}(iii) implies that 
\begin{align}\label{gamup}
e (H^A _{ij}) \leq (1+\eps _4)(a-b)D/(2K^2). 
\end{align}
Therefore, (\ref{gambound}) implies that there exists a constant $\gamma >0$ such that
$$(1-\gamma)e (H^A _{ij}) =(a-b)t.$$
Since $(1-19\eps _4 )(1-\eps _3) >(1-20\eps _4)$, (\ref{gambound}) implies that
$\gamma > 19 \eps _4 \gg 1/n$. Further, since $(1+\eps _4)(1-21 \eps _4) < (1-20 \eps _4)$,
(\ref{gamup}) implies that $\gamma < 21 \eps _4$.

Note that (BFR5), (BFR7) and Lemma~\ref{rnd_slice}(vi) imply that 
\begin{align}\label{boundup}
\Delta (H^A _{ij}) \leq (D/2+5 \eps _1 n)/K^2.
\end{align}
Thus Lemma~\ref{simple} implies that $H^A _{ij}$ contains a spanning subgraph $H$ such that
$e(H)=(1- \gamma ) e(H^A _{ij})=(a-b)t$ and 
$$\Delta (H^A _{ij} -H) \leq 6 \gamma (D/2+5 \eps _1 n)/(5K^2)
\leq 13 \eps _4 D/K^2,$$
where the last inequality follows since $ \gamma <21 \eps _4$ and $\eps _1 \ll 1$.
Setting $G^A _{ij}:=H^A _{ij} -H$ implies that (iv) is satisfied.

Our next task is to decompose $H$ into $t$ edge-disjoint path systems so that (i) and (ii) 
are satisfied. Note that (\ref{boundup}) implies that
$$\Delta (H) \leq \Delta (H^A _{ij}) \leq (D/2+5 \eps _1 n)/K^2 <2t-2.$$
Further, (BFR4) implies that $\Delta (H[A_0]) \leq |A_0| \leq \eps n < t-1$ and 
(BFR5) implies that $d_H (x ) \leq \eps ' n <t-1$ for all $x \in A$. 
Since $e(H)=(a-b)t$,  Lemma~\ref{splittrick} implies that $H$ has a decomposition into $t$ edge-disjoint path systems $P_1,\dots,P_{t}$ satisfying (i) and so that $e(P_s)=a-b=\ell _a$ for all
$ s\leq t$. In particular, (ii) is satisfied.

\medskip

\noindent
{\bf Case 3.} $e(A'), e(B') \geq \eps _3 n$ 

By definition of $\ell _a $ and $\ell _b$, we have that
$\lceil \ell _a \rceil -\lceil \ell _b \rceil= \lfloor \ell _a \rfloor -\lfloor \ell _b \rfloor =a-b$.
Notice that since $e(A') \geq \eps _3 n$ and $\eps _2 \ll \eps _3$, certainly $\eps _3 e(A')/(2K^2) > 9 \eps _2 n/K^2$.
Therefore, Lemma~\ref{rnd_slice}(iii) implies that 
\begin{align}
\nonumber
e (H^A _{ij})& \geq (1-9 \eps _2) {e(A')}/{K^2}-9 \eps _2 n /K^2 \\ & \geq 
(1- \eps _3) e(A')/K^2 \label{gambound2} \\ 
& \geq \eps _3 n/(2K^2). \nonumber
\end{align}
Note that $1/n \ll \eps _3/(2K^2)$.\COMMENT{split calculation up into 2 parts here since it is important to show that
there are many edges in $H^A_{ij}$ so that we can apply Lemma~\ref{simple}.}
Further, (\ref{aqc}) and (\ref{gambound2}) imply that
\begin{align}
\nonumber
e (H^A _{ij})& \geq 
(1- \eps _3) e(A')/K^2  \\ 
& = (1- \eps _3) (a+q+c)D/(2K^2) > (a+q)t+t^* \label{gambound3} . 
\end{align}
Similarly, Lemma~\ref{rnd_slice}(iii) implies that 
\begin{align}\label{gambound4}
e (H^A _{ij}) \leq (1+\eps _3) (a+q+c) D/(2K^2).
\end{align}
By (\ref{gambound3}) there exists a constant $\gamma >0$ such that
$$(1-\gamma)e (H^A _{ij}) =(a+q)t+t^*.$$
Note that  (\ref{gambound3}) implies that $1/n \ll 19 \eps _4 < \gamma$ and (\ref{gambound4}) implies that $\gamma < 21 \eps _4$.%
\COMMENT{the upper bound is not completely obvious this time around since $t^*$ is a `floor': Indeed, since $e(A') \geq \eps _3n$,
(\ref{gambound4}) implies that $(a+q+c)D/2K^2 \geq \eps _3 n/K^2$. Thus,
$$(1-21 \eps _4)(1+\eps _3) (a+q+c) D/(2K^2) < (1-20 \eps _4) (a+q+c) D/(2K^2)-K^2 \leq (a+q)t+t^*.$$ }
Moreover, as in Case~2, (BFR5), (BFR7) and Lemma~\ref{rnd_slice}(vi) together show that 
\begin{align}\label{boundup1}
\Delta (H^A _{ij}) \leq (D/2+5 \eps _1 n)/K^2.
\end{align}
Thus (as in Case~2 again), Lemma~\ref{simple} implies that $H^A _{ij}$ contains a spanning subgraph $H$ such that
$e(H)=(1- \gamma ) e(H^A _{ij})=(a+q)t+t^* $ and%
\COMMENT{Deryk shortened calculation and referred to case 2:
where the last inequality follows since $ \gamma <21 \eps _4$ and $\eps _1 \ll 1$.}
$$\Delta (H^A _{ij} -H) \leq 6 \gamma (D/2+5 \eps _1 n)/(5K^2)
\leq 13 \eps _4 D/K^2.$$
Setting $G^A _{ij}:=H^A _{ij} -H$ implies that (iv) is satisfied.
Next we decompose $H$ into $t$ edge-disjoint path systems so that (i) and (ii) 
are satisfied. Note that (\ref{boundup1}) implies that
$$\Delta (H) \leq \Delta (H^A _{ij}) \leq (D/2+5 \eps _1 n)/K^2 <2t-2.$$
Further, (BFR4) implies that $\Delta (H[A_0]) \leq |A_0| \leq \eps n < t-1$ and 
(BFR5) implies that $d_H (x ) \leq \eps ' n <t-1$ for all $x \in A$. 
Since $e(H)=(a+q)t+t^*$,  Lemma~\ref{splittrick} implies that $H$ has a decomposition into $t$ edge-disjoint path systems $P_1,\dots,P_{t}$ satisfying (i) and (ii). 
An identical argument implies that (i), (ii) and (iv) are satisfied with respect to $H^B _{ij}$ also. 
\endproof


\subsection{Decomposing the Global Graph}\label{sec:global} 
Let $G_{glob}^A$  be the union of the graphs $G_{ij}^A$ guaranteed by Lemma~\ref{matchingdec} over all $1\le i,j\le K$.  
Define $G_{glob}^B$ similarly.
The next lemma gives a decomposition of both $G_{glob}^A$ and $G_{glob}^B$ into suitable path systems.
Properties~(iii) and~(iv) of the lemma guarantee that one can pair up each such path system $Q_A\subseteq G_{glob}^A$ with
a different path system $Q_B\subseteq G_{glob}^B$ such that $Q_A \cup Q_B$ is $2$-balanced 
(in particular $e(Q_A)-e(Q_B)=a-b$). This
property will then enable us to apply Lemma~\ref{extendpaths} to 
extend $Q_A\cup Q_B$ into a Hamilton cycle using only edges between $A'$ and $B'$.

\begin{lemma}\label{G-glob} 
Let $0<1/n\ll \eps\ll \eps'\ll \eps_1\ll \eps_2\ll \eps_3\ll \eps_4\ll 1/K\ll 1$.
Suppose that $(G,A,A_0,B,B_0)$ is an $(\eps,\eps',K,D)$-bi-framework with $|G|=n$ and such that $D\ge n/200$
and $D$ is even.
Let $A_0,A_1,\dots,A_K$ and $B_0,B_1,\dots,B_K$ be a $(K,m,\eps,\eps_1,\eps_2)$-partition for $G$. 
Let $G_{glob}^A$  be the union of the graphs $G_{ij}^A$ guaranteed by Lemma~\ref{matchingdec} over all $1\le i,j\le K$.  
Define $G_{glob}^B$ similarly. Suppose that $k := 10 \eps_4 D\in\mathbb{N}$. Then the following properties hold:
\begin{itemize} 
\item[{\rm (i)}] There is an integer $q'$ and a real number  $0\le c'<1$ so that
$e(G_{glob}^A)=(a+q'+c')k$ and  $e(G_{glob}^B)=(b+q'+c')k$.
\item[{\rm (ii)}] $\Delta(G_{glob}^A), \Delta(G_{glob}^B)< 3k/2$.
\item[\rm{(iii)}] Let $k^*:=c'k$. Then $G_{glob}^A$ has a decomposition into $k^*$ path systems, each containing $a+q'+1$ edges, and $k-k^*$ path systems,
each containing $a+q'$ edges. Moreover, each of these $k$ path systems~$Q$ satisfies $d_Q(x)\le 1$ for all $x\in A$.
\item[\rm{(iv)}] $G_{glob}^B$ has a decomposition into $k^*$ path systems, each containing $b+q'+1$ edges, and $k-k^*$ path systems,
each containing $b+q'$ edges. Moreover, each of these $k$ path systems~$Q$ satisfies $d_Q(x)\le 1$ for all $x\in B$.
\item[\rm{(v)}] Each of the path systems guaranteed in~(iii) and~(iv) contains at most $\sqrt{\eps}n$ edges.
\end{itemize}
\end{lemma}
Note%
\COMMENT{osthus added sentence. Andy: changed A to B}
that in Lemma~\ref{G-glob} and several later statements the parameter $\eps_3$ is implicitly defined by the application
of Lemma~\ref{matchingdec} which constructs the graphs $G_{glob}^A$ and $G_{glob}^B$.
\proof
Let $t^*$ and $t$ be as defined in Lemma~\ref{matchingdec}. Our first task is to show that (i) is satisfied.
If $e(A'),e(B') < \eps _3 n$ then $G^A _{glob}=G[A']$ and $G^B _{glob} =G[B']$. Further, $a=b$ in this case since 
otherwise (BFR4) implies that $a>b$ and so (BFR2) yields that $e(A') \geq (a-b)D/2\geq D/2 > \eps _3 n$, a contradiction.
Therefore, (BFR2) implies that
\begin{align*}
e(G_{glob}^A)  -e(G_{glob}^B) & =e(A')-e(B') {=} (a-b)D/2=0=(a-b)k.
\end{align*}

If $e(A')\ge \eps _3 n$ and $e(B') < \eps _3 n$ then $G^B _{glob} =G[B']$. Further, $G^A _{glob}$ is obtained from
$G[A']$ by removing $tK^2$ edge-disjoint path systems, each of which contains precisely $a-b$ edges.
Thus (BFR2) implies that 
$$e(G_{glob}^A)  -e(G_{glob}^B)  =e(A')-e(B')-tK^2(a-b) =(a-b)(D/2-tK^2)=(a-b)k.$$

Finally, consider the case when $e(A'), e(B') > \eps _3 n$.
Then $G^A _{glob}$ is obtained from $G[A']$ by removing $t^*K^2$ edge-disjoint path systems, each of which contain exactly $a+q+1$ edges,
and by removing $(t-t^*)K^2$ edge-disjoint  path systems, each of which contain exactly $a+q$ edges.
Similarly, $G^B _{glob}$ is obtained from $G[B']$ by removing $t^*K^2$ edge-disjoint path systems, each of which contain exactly $b+q+1$ edges,
and by removing $(t-t^*)K^2$ edge-disjoint path systems, each of which contain exactly $b+q$ edges. So (BFR2) implies that
\begin{align*}
e(G_{glob}^A)  -e(G_{glob}^B)  =e(A')-e(B') -(a-b)tK^2= (a-b)k.
\end{align*}
Therefore, in every case,
\begin{align}\label{abk}
e(G_{glob}^A)  -e(G_{glob}^B)  =(a-b)k.
\end{align}
Define the integer $q'$ and $0 \le c' <1$ by $e(G_{glob}^A)=(a+q'+c')k$. 
Then~(\ref{abk}) implies that $e(G_{glob}^B)=(b+q'+c')k$. This proves (i).
To prove (ii), note that Lemma~\ref{matchingdec}(iv) implies that
$\Delta(G_{glob}^A)\le 13\eps_4 D<3k/2$ and similarly $\Delta(G_{glob}^B)<3k/2$. 

Note that (BFR5) implies that $d_{G_{glob}^A} (x) \leq \eps 'n < k-1$ for all $x \in A$ 
and $\Delta (G_{glob}^A[A_0]) \leq |A_0| \leq \eps n < k-1$.
Thus Lemma~\ref{splittrick} together with (i) implies that (iii) is satisfied.
(iv) follows from Lemma~\ref{splittrick} analogously.

(BFR3) implies that $e(G^A_{glob}) \leq e_G (A') \leq \eps n^2$ and $e(G^B_{glob}) \leq e_G (B') \leq \eps n^2$.
Therefore, each path system from (iii) and (iv) contains at most $\lceil \eps n^2 /k \rceil \leq \sqrt{\eps} n$ edges.%
   \COMMENT{Daniela replaced "has size" by "contains ... edges"}
So (v) is satisfied.
\endproof

We say that a path system $P\subseteq G[A']$ is \emph{$(i,j,A)$-localized} if%
\COMMENT{ANDREW changed the definition here. This is needed because, in proof of
Lemma~\ref{balmatchextend} we need the fact that edges e.g. don't lie in $G[A_i]$ so that
condition (BES2) is satisfied.}
\begin{itemize}
\item[(i)] $E(P) \subseteq E(G[A_0, A_i \cup A_j]) \cup E(G[A_i,A_j])\cup E(G[A_0])$;
\item[(ii)] Any internal vertex on a path in $P$ lies in $A_0$.
\end{itemize}
We%
 \COMMENT{Previously had the stronger condition that $M\subseteq H^A_{ij}$ instead.}
introduce an analogous notion of \emph{$(i,j,B)$-localized} for path systems $P\subseteq G[B']$.

The following result is a straightforward consequence of Lemmas~\ref{rnd_slice}, \ref{matchingdec} and \ref{G-glob}.
It gives a decomposition of $G[A'] \cup G[B']$ into pairs of paths systems so that most of these are localized and so that each pair
can be extended into a Hamilton cycle by adding $A'B'$-edges. %
\COMMENT{Deryk added sentnce}
\begin{cor} \label{finalcor}
Let $0<1/n\ll \eps\ll \eps'\ll \eps_1\ll \eps_2\ll \eps_3\ll \eps_4\ll 1/K\ll 1$.
Suppose that $(G,A,A_0,B,B_0)$ is an $(\eps,\eps',K,D)$-bi-framework with $|G|=n$ and such that $D\ge n/200$
and $D$ is even.
Let $A_0,A_1,\dots,A_K$ and $B_0,B_1,\dots,B_K$ be a $(K,m,\eps,\eps_1,\eps_2)$-partition for $G$. 
Let $t_K:=(1-20\eps_4)D/2K^4$ and $k:=10\eps_4 D$. Suppose that $t_K\in\mathbb{N}$.
Then there are $K^4$ sets $\cM_{i_1i_2i_3i_4}$, one for each $1 \le i_1,i_2,i_3,i_4 \le K$, such that each $\cM_{i_1i_2i_3i_4}$
consists of $t_K$ pairs of path systems and satisfies the following properties:
\begin{itemize} 
\item[{\rm (a)}]  Let $(P,P')$ be a pair of path systems which forms an element of $\cM_{i_1i_2i_3i_4}$. Then
\begin{itemize}
\item[{\rm (i)}] $P$ is an $(i_1,i_2,A)$-localized path system and $P'$ is an $(i_3$,$i_4$,$B)$-localized path system;
\item[{\rm (ii)}] $e(P)-e(P')=a-b$;
\item[{\rm (iii)}]  $e(P),e(P') \le \sqrt{\eps}n$.
\end{itemize}
\item[{\rm (b)}] The $2t_K$ path systems in the pairs belonging to $\cM_{i_1i_2i_3i_4}$ are all pairwise edge-disjoint.
\item[{\rm (c)}] Let $G(\cM_{i_1i_2i_3i_4})$ denote the spanning subgraph of $G$ whose edge set is the union of all the
path systems in the pairs belonging to $\cM_{i_1i_2i_3i_4}$. 
Then the $K^4$ graphs $G(\cM_{i_1i_2i_3i_4})$ are edge-disjoint.
Further, each $x \in A_0$ satisfies
$d_{G(\cM_{i_1i_2i_3i_4})}(x)\ge (d_G(x,A)-15\eps_4 D)/K^4$ while each $y \in B_0$ satisfies
$d_{G(\cM_{i_1i_2i_3i_4})}(y)\ge (d_G(y,B)-15\eps_4 D)/K^4$.%
    \COMMENT{Daniela: previously had "The analogous assertion holds for each $x \in B_0$ with $A$ replaced by $B$."}
\item[{\rm (d)}] Let $G_{glob}$ be the subgraph of $G[A'] \cup G[B']$ obtained by removing all edges contained in $G(\cM_{i_1i_2i_3i_4})$ for all 
$1\le i_1,i_2,i_3,i_4 \le K$. Then $\Delta(G_{glob})\le 3k/2$. Moreover, $G_{glob}$ has a decomposition into
$k$ pairs of path systems $(Q_{1,A},Q_{1,B}),\dots,(Q_{k,A},Q_{k,B})$ so that
\begin{itemize}
\item[{\rm (i$'$)}] $Q_{i,A}\subseteq G_{glob}[A']$ and $Q_{i,B}\subseteq G_{glob}[B']$ for all $i \le k$;
\item[\rm{(ii$'$)}] $d_{Q_{i,A}}(x)\le 1$ for all $x\in A$ and $d_{Q_{i,B}}(x)\le 1$ for all $x\in B$;
\item[\rm{(iii$'$)}] $e(Q_{i,A})-e(Q_{i,B})=a-b$ for all $i \le k$;
\item[\rm{(iv$'$)}] $e(Q_{i,A}),e(Q_{i,B}) \le \sqrt{\eps}n$ for all $i \le k$.
\end{itemize}
\end{itemize}
\end{cor}
\proof
Apply Lemma~\ref{rnd_slice} to obtain localized slices $H^A_{ij}$ and $H^B_{ij}$ (for all $i,j\le K$).
Let $t:=K^2 t_K$ and let $t^*$ be as defined in Lemma~\ref{matchingdec}. Since $t/K^2, t^*/K^2\in\mathbb{N}$ we have $(t-t^*)/K^2\in\mathbb{N}$.
For all $i_1,i_2\le K$, let $\cM_{i_1i_2}^A$ be the set of $t$ path systems  in $H_{i_1i_2}^A$ guaranteed by Lemma~\ref{matchingdec}.
We call the $t^*$ path systems in $\cM_{i_1i_2}^A$ of size $\lceil \ell_a \rceil$ \emph{large} and the others \emph{small}.
We define $\cM_{i_3i_4}^B$ as well as large and small path systems in $\cM_{i_3i_4}^B$ analogously (for all $i_3,i_4 \leq K$).

We now construct the sets $\cM_{i_1i_2i_3i_4}$ as follows:
For all $i_1,i_2 \le K$, 
consider a random partition of the set of all large path systems in $\cM_{i_1i_2}^A$ into $K^2$ sets of equal size $t^*/K^2$ and assign
(all the path systems in) each of these sets to 
one of the $\cM_{i_1i_2i_3i_4}$ with $i_3,i_4 \le K$.
Similarly, randomly partition the set of small path systems in $\cM_{i_1i_2}^A$ into $K^2$ sets, each containing $(t-t^*)/K^2$ path systems.
Assign each of these $K^2$ sets to one of the $\cM_{i_1i_2i_3i_4}$ with $i_3,i_4 \le K$.
Proceed similarly for each $\cM_{i_3i_4}^B$ in order to assign each of its path systems randomly to some $\cM_{i_1i_2i_3i_4}$.
Then to each $\cM_{i_1i_2i_3i_4}$ we have assigned exactly $t^*/K^2$ large path systems from both $\cM_{i_1i_2}^A$ and $\cM_{i_3i_4}^B$.
Pair these off arbitrarily. Similarly, pair off the small path systems assigned to $\cM_{i_1i_2i_3i_4}$ arbitrarily.
Clearly, the sets $\cM_{i_1i_2i_3i_4}$ obtained in this way satisfy (a) and (b).

We now verify (c). By construction, the $K^4$ graphs $G(\cM_{i_1i_2i_3i_4})$ are edge-disjoint.%
   \COMMENT{Daniela: new sentence}
So consider any vertex $x\in A_0$ and
write $d:=d_G(x,A)$. Note that $d_{H_{i_1i_2}^A}(x)\ge (d-4 \eps_1 n)/K^2$
by Lemma~\ref{rnd_slice}(vi). Let $G(\cM_{i_1i_2}^A)$ be the spanning subgraph of $G$ whose edge set is the union of all the path systems
in $\cM_{i_1i_2}^A$. Then Lemma~\ref{matchingdec}(iv) implies that
$$
d_{G(\cM_{i_1i_2}^A)}(x)\ge d_{H_{i_1 i_2}^A}(x)-  \Delta (G^A _{i_1 i_2} )  \ge \frac{d-4\eps_1 n}{K^2} -\frac{13\eps_4D}{K^2} 
\ge \frac{d-14\eps_4 D}{K^2}.
$$
So a Chernoff-Hoeffding estimate for the hypergeometric distribution (Proposition~\ref{chernoff}) implies that%
$$
d_{G(\cM_{i_1i_2i_3i_4})}(x)\ge \frac{1}{K^2} \left( \frac{d-14\eps_4 D}{K^2} \right)  -\eps n
\ge \frac{d-15\eps_4 D}{K^4}.
$$
(Note that we only need to apply the Chernoff-Hoeffding bound if $d \ge \eps n$ say, as (c) is vacuous otherwise.)

It remains to check condition~(d). First note that $k\in\mathbb{N}$ since $t_K, D/2\in\mathbb{N}$.
Thus we can apply Lemma~\ref{G-glob} to obtain a decomposition of both $G^A_{glob}$ and $G^B_{glob}$
into path systems.
Since $G_{glob}=G_{glob}^A\cup G_{glob}^B$, (d) is an immediate consequence of Lemma~\ref{G-glob}(ii)--(v).
\endproof

\subsection{Constructing Localized Balanced Exceptional Systems} \label{besconstruct}

The localized path systems obtained from Corollary~\ref{finalcor} do not yet cover all of the exceptional vertices.
This is achieved via the following lemma: we extend the path systems to achieve this additional property, while maintaining the property of being balanced.%
   \COMMENT{Deryk added 2 sentences}
More precisely, let
$$\cP:=\{A_0,A_1,\dots,A_K,B_0,B_1,\dots,B_K\}$$
be a $(K,m,\eps)$-partition%
    \COMMENT{Daniela replaced $(K,m,\eps_0)$-partition by $(K,m,\eps)$-partition to make clear that the $\eps$ need
not be the same as the parameter $\eps_0$ of the BES}
of a set $V$ of $n$ vertices.
Given $1\leq \I \leq K$ and $\eps_0>0$, an \emph{$\i$-balanced exceptional system with respect to $\cP$ and parameter~$\eps_0$}
is a path system $J$ with $V(J)\subseteq A_0\cup B_0\cup A_{i_1} \cup A_{i_2} \cup B_{i_3} \cup B_{i_4}$
such that the following conditions hold:
\begin{itemize}
\item[(BES$1$)] Every vertex in $A_0\cup B_0$ is an internal vertex of a path in~$J$.
Every vertex $v\in A_{i_1} \cup A_{i_2} \cup B_{i_3} \cup B_{i_4}$ satisfies $d_J(v)\le 1$.%
    \COMMENT{Can write that $v$ is an endpoint of a (possible trivial) path in~$J$ if we allow
$V(J)\subseteq A_0\cup B_0\cup A_{i_1} \cup A_{i_2} \cup B_{i_3} \cup B_{i_4}$ (instead of
$V(J)=A_0\cup B_0\cup A_{i_1} \cup A_{i_2} \cup B_{i_3} \cup B_{i_4}$). Also,
the BES we construct in Lemma~\ref{balmatchextend} have the property that none of these paths has both endpoints
in $A_{i_1} \cup A_{i_2}$ and its midoint in $A_0$ (and the analogue holds with for $B$).}   
\item[(BES$2$)]
Every edge of $J[A \cup B]$ is either an $A_{i_1}A_{i_2}$-edge or a $B_{i_3}B_{i_4}$-edge.%
	\COMMENT{AL: had `Every (maximal) path of length one in~$J$ has either one endpoint in $A_{i_1}$ and the other endpoint in $A_{i_2}$
or one endpoint in $B_{i_3}$ and the other endpoint in $B_{i_4}$.'
}
\item[(BES$3$)] The edges in $J$ cover precisely the same number of vertices in $A$ as in $B$.%
   \COMMENT{Cannot write $J$ instead of "edges in $J$" here since we don't want to count the vertices of degree 0 which lie in $V(J)$.}
\item[(BES$4$)] $e(J)\le \eps_0 n$.
\end{itemize}
To shorten the notation, we will often refer to $J$ as an \emph{$(i_1,i_2, i_3,i_4)$-BES}.
If $V$ is the vertex set of a graph $G$ and $J\subseteq G$, we also say that $J$ is an \emph{$(i_1,i_2, i_3,i_4)$-BES in~$G$}. 
Note that (BES2) implies that an $(i_1,i_2, i_3,i_4)$-BES does not contain edges between $A$ and $B$.
Furthermore, an $\i$-BES is also, for example, an $(i_2,i_1, i_4,i_3)$-BES.
We will sometimes omit the indices $i_1,i_2, i_3,i_4$ and just refer to a balanced exceptional system (or a BES for short).%
   \COMMENT{Daniela deleted: Throughout the paper, usually our balanced exceptional systems will have parameter $\eps_0$,
so in this case we will often omit mentioning the parameter.}
We will sometimes also omit the partition $\cP$, if it is clear from the context. As mentioned before, balanced exceptional systems will
play a similar role as the exceptional systems that we used in the two cliques case (i.e. in Chapter~\ref{paper1}).

(BES1) implies that each balanced exceptional system is an $A_0B_0$-path system as defined before Proposition~\ref{balpathcheck}.
(However, the converse is not true since, for example, a 2-balanced $A_0B_0$-path system need not satisfy~(BES4).)
So (BES3) and Proposition~\ref{balpathcheck} imply that each balanced exceptional system is also $2$-balanced.

We now extend each set $\cM_{i_1i_2i_3i_4}$ obtained from Corollary~\ref{finalcor} into a set
$\mathcal{J}_{i_1i_2i_3i_4}$ of $(i_1,i_2,i_3,i_4)$-BES.

\begin{lemma}\label{balmatchextend}
Let $0<1/n\ll \eps\ll \eps_0\ll \eps'\ll \eps_1\ll \eps_2\ll \eps_3\ll \eps_4\ll 1/K\ll 1$.
Suppose that $(G,A,A_0,B,B_0)$ is an $(\eps,\eps',K,D)$-bi-framework with $|G|=n$ and such that $D\ge n/200$
and $D$ is even. Let $\cP:=\{A_0,A_1,\dots,A_K,B_0,B_1,\dots,B_K\}$ be a $(K,m,\eps,\eps_1,\eps_2)$-partition for $G$. 
Suppose that $t_K:=(1-20\eps_4)D/2K^4\in \mathbb{N}$. Let
$\cM_{i_1i_2i_3i_4}$ be the sets returned by Corollary~\ref{finalcor}.
Then for all $1\leq \I \leq K$ there is a set $\mathcal{J}_{i_1i_2i_3i_4}$ which satisfies the
following properties:
\begin{itemize}
\item[{\rm (i)}] $\mathcal{J}_{i_1i_2i_3i_4}$ consists of $t_K$ edge-disjoint $(i_1,i_2,i_3,i_4)$-BES in $G$ with respect to $\cP$
and with parameter $\eps_0$.
\item[{\rm (ii)}] For each of the $t_K$ pairs of path systems $(P,P')\in \cM_{i_1i_2i_3i_4}$, 
there is a unique $J\in \mathcal{J}_{i_1i_2i_3i_4}$ which contains all the edges in $P\cup P'$. Moreover, all edges in
$E(J)\setminus E(P\cup P')$ lie in $G[A_0,B_{i_3}]\cup G[B_0,A_{i_1}]$.
\item[{\rm (iii)}] Whenever $(i_1,i_2,i_3,i_4)\neq (i'_1,i'_2,i'_3,i'_4)$, $J \in \mathcal{J}_{i_1i_2i_3i_4}$ and $J' \in \mathcal{J}_{i'_1i'_2i'_3i'_4}$,
then $J$ and $J'$ are edge-disjoint.
\end{itemize}
\end{lemma}
We let $\mathfrak{J}$ denote the union of the sets $\mathcal{J}_{i_1i_2i_3i_4}$ over all $1\leq \I \leq K$.

\proof
We will construct the sets $\mathcal{J}_{i_1i_2i_3i_4}$ greedily by extending each pair of path systems $(P,P')\in \cM_{i_1i_2i_3i_4}$
in turn into an $(i_1,i_2,i_3,i_4)$-BES containing $P\cup P'$. For this, 
consider some arbitrary ordering of the $K^4$ $4$-tuples $(i_1,i_2,i_3,i_4)$.  
Suppose that we have already constructed the sets ${\mathcal J}_{i'_1i'_2i'_3i'_4}$ for all 
$(i'_1,i'_2,i'_3,i'_4)$ preceding $(i_1,i_2,i_3,i_4)$ so that (i)--(iii) are satisfied. So our aim now is to construct ${\mathcal J}_{i_1i_2i_3i_4}$.
Consider an enumeration $(P_1,P'_1),\dots,(P_{t_K},P'_{t_K})$ of the pairs of path systems%
    \COMMENT{Daniela replaced matchings by paths systems}
in $\cM_{i_1i_2i_3i_4}$.
Suppose that for some $i\le t_K$ we%
   \COMMENT{Daniela added for some $i\le t_K$}
have already constructed edge-disjoint $(i_1,i_2,i_3,i_4)$-BES $J_1,\dots,J_{i-1}$, so that for each $i'<i$
the following conditions hold:
\begin{itemize}
\item $J_{i'}$ contains the edges in $P_{i'}\cup P'_{i'}$;
\item all edges in $E(J_{i'})\setminus E(P_{i'}\cup P'_{i'})$ lie in $G[A_0,B_{i_3}]\cup G[B_0,A_{i_1}]$;
\item $J_{i'}$ is edge-disjoint from all the balanced exceptional systems in \\ $\bigcup_{(i'_1,i'_2,i'_3,i'_4)} {\mathcal J}_{i'_1i'_2i'_3i'_4}$,
where the union is over all $(i'_1,i'_2,i'_3,i'_4)$ preceding $(i_1,i_2,i_3,i_4)$.
\end{itemize}
We will now construct $J:=J_i$. For this, we need to add suitable edges to $P_i\cup P'_i$ to ensure that all vertices  of $A_0 \cup B_0$ 
have degree two. We start with $A_0$. Recall that $a=|A_0|$ and write  $A_0=\{x_1,\dots,x_a\}$. 
Let $G'$ denote the subgraph of $G[A',B']$ obtained by removing all the edges lying in $J_1,\dots,J_{i-1}$ as well as all
those edges lying in the balanced exceptional systems belonging to $\bigcup_{(i'_1,i'_2,i'_3,i'_4)} {\mathcal J}_{i'_1i'_2i'_3i'_4}$ (where as before the union
is over all $(i'_1,i'_2,i'_3,i'_4)$ preceding $(i_1,i_2,i_3,i_4)$). 
We will choose the new edges incident to $A_0$ in $J$ inside $G'[A_0,B_{i_3}]$.

Suppose we have already found suitable edges for $x_1,\dots,x_{j-1}$ and let $J(j)$ be the set of all these edges.
We will first show that the degree of $x_j$ inside $G'[A_0,B_{i_3}]$ is still large.
Let $d_j:=d_G(x_j,A')$. Consider any $(i'_1,i'_2,i'_3,i'_4)$ preceding $(i_1,i_2,i_3,i_4)$.
Let $G({\mathcal J}_{i'_1i'_2i'_3i'_4})$ denote the union of the $t_K$ balanced exceptional systems belonging to ${\mathcal J}_{i'_1i'_2i'_3i'_4}$.
Thus $d_{G({\mathcal J}_{i'_1i'_2i'_3i'_4})}(x_j)=2t_K$. However, Corollary~\ref{finalcor}(c) implies that
$d_{G({\mathcal M}_{i'_1i'_2i'_3i'_4})}(x_j)\ge (d_j-15\eps_4 D)/K^4$. So altogether, when constructing (the balanced exceptional systems in)
${\mathcal J}_{i'_1i'_2i'_3i'_4}$, we have added at most $2t_K-(d_j-15\eps_4 D)/K^4$ new edges
at $x_j$, and all these edges join $x_j$ to vertices in $B_{i'_3}$. Similarly, when constructing
$J_1,\dots,J_{i-1}$, we have added at most $2t_K-(d_j-15\eps_4 D)/K^4$ new edges at $x_j$.
Since the number of $4$-tuples $(i'_1,i'_2,i'_3,i'_4)$ with $i'_3=i_3$ is $K^3$, it follows that%
   \COMMENT{Daniela: previously had $\le $ instead of the first $=$}
\begin{align*}
d_G(x_j,B_{i_3})-d_{G'}(x_j,B_{i_3})  
& \le K^3 \left(2t_K - \frac{d_j-15\eps_4 D}{K^4}\right) \\
& = \frac{1}{K} \left( (1-20\eps_4)D -d_j+15\eps_4D \right) \\
& = \frac{1}{K} \left( D-d_j -5 \eps_4D \right). 
\end{align*}
Also, (P2) with $A$ replaced by~$B$ implies that%
  \COMMENT{AL:changed $d_G(x_j,B_{i_3})=\frac{d_G(x_j,B)\pm \eps_1 n}{K}$ to $d_G(x_j,B_{i_3})\ge\frac{d_G(x_j,B)- \eps_1 n}{K}$}
$$d_G(x_j,B_{i_3}) \ge \frac{d_G(x_j,B) -  \eps_1 n}{K}\ge \frac{d_G(x_j)-d_G(x_j,A')-\eps_1 n}{K} = \frac{D-d_j -\eps_1 n}{K},
$$
where here we use (BFR2) and (BFR6).
So altogether, we have%
\COMMENT{AL: changed $\eps_4 n/25K$ with $\eps_4 n/50K$}
$$
d_{G'}(x_j,B_{i_3})  \ge (5\eps_4 D- \eps_1n)/K\ge \eps_4 n/50K.
$$
Let $B'_{i_3}$ be the set of vertices in $B_{i_3}$ not covered by the edges of $J(j)\cup P'_i$.%
\COMMENT{AL: added edges of}
Note that $|B'_{i_3}| \ge |B_{i_3}|- 2|A_0|-2e(P'_i)\ge |B_{i_3}|-3\sqrt{\eps}n$ since 
$a=|A_0|\le \eps n$ by (BFR4)
and $e(P'_i)\le \sqrt{\eps}n$ by Corollary~\ref{finalcor}(a)(iii).
So $d_{G'}(x_j,B'_{i_3}) \ge \eps_4 n/51K$. We can add up to two of these edges to~$J$ in order to ensure that $x_j$ has degree two in~$J$.
This completes the construction of the edges of $J$ incident to $A_0$. The edges incident to $B_0$ are found similarly.

Let $J$ be the graph on $A_0\cup B_0 \cup A_{i_1} \cup A_{i_2} \cup B_{i_3} \cup B_{i_4}$ whose edge set is constructed in this way.
By construction, $J$ satisfies (BES1) and (BES2) since $P_j$ and $P'_j$ are
$(i_1,i_2,A)$-localized and $(i_3,i_4,B)$-localized respectively.%
	\COMMENT{AL: deleted follows}
We now verify (BES3). As mentioned before the statement of the lemma, 
(BES1) implies that $J$ is an $A_0B_0$-path system (as defined before Proposition~\ref{balpathcheck}).
Moreover, Corollary~\ref{finalcor}(a)(ii) implies that $P_i \cup P'_i$ is a path system which satisfies (B1)
in the definition of $2$-balanced.
Since $J$ was obtained by adding only $A'B'$-edges, (B1) is preserved in $J$.
Since by construction $J$ satisfies (B2), it follows that $J$ is $2$-balanced.
So Proposition~\ref{balpathcheck} implies (BES3). 

 Finally, we verify (BES4). For this, note that Corollary~\ref{finalcor}(a)(iii) implies
that $e(P_i),e(P'_i)\le \sqrt{\eps}n$. Moreover, the number of edges added to $P_i\cup P'_i$ when constructing $J$
is at most $2(|A_0|+|B_0|)$, which is at most $2\eps n$ by (BFR4).
Thus $e(J)\le 2\sqrt{\eps}n+2\eps n\le \eps_0 n$.
\endproof

\subsection{Covering $G_{glob}$ by Edge-disjoint Hamilton Cycles}

We now find a set of edge-disjoint Hamilton cycles covering the edges of the `leftover' graph obtained from $G-G[A,B]$ by deleting all
those edges lying in balanced exceptional systems belonging to $\mathfrak{J}$.

\begin{lemma}\label{lem:HCglob}
Let $0<1/n\ll \eps\ll \eps_0\ll \eps'\ll \eps_1\ll \eps_2\ll \eps_3\ll \eps_4\ll 1/K\ll 1$.
Suppose that $(G,A,A_0,B,B_0)$ is an $(\eps,\eps',K,D)$-bi-framework with $|G|=n$ and such that $D\ge n/200$ and%
   \COMMENT{Daniela: had $\delta (G) \geq D\ge n/200$}
$D$ is even. Let $\cP:=\{A_0,A_1,\dots,A_K,B_0,B_1,\dots,B_K\}$ be a $(K,m,\eps,\eps_1,\eps_2)$-partition for $G$. 
Suppose that $t_K:=(1-20\eps_4)D/2K^4\in \mathbb{N}$. Let $\mathfrak{J}$ be as defined after Lemma~\ref{balmatchextend}
and let $G(\mathfrak{J})\subseteq G$ be the union of all the balanced exceptional systems lying in $\mathfrak{J}$.
Let $G^*:=G-G(\mathfrak{J})$, let $k:=10\eps_4 D$ and let $(Q_{1,A},Q_{1,B}),\dots,(Q_{k,A},Q_{k,B})$ be as in Corollary~\ref{finalcor}(d).
\begin{itemize}
\item[{\rm (a)}] The graph $G^*-G^*[A,B]$ can be decomposed into $k$ $A_0B_0$-path systems $Q_1,\dots,Q_k$ which are $2$-balanced and satisfy the
following properties:
\begin{itemize}
\item[{\rm (i)}] $Q_i$ contains all edges of $Q_{i,A} \cup Q_{i,B}$; 
\item[{\rm (ii)}]  $Q_1,\dots,Q_k$ are pairwise edge-disjoint;
\item[\rm{(iii)}] $e(Q_i) \le 3\sqrt{\eps}n$.
\end{itemize}
\item[{\rm (b)}] Let $Q_1,\dots,Q_k$ be as in~(a). Suppose that $F$ is a graph on $V(G)$ such that $G\subseteq F$, $\delta(F)\ge 2n/5$ and such that $F$ satisfies
(WF5) with respect to $\eps'$. Then there are edge-disjoint Hamilton cycles $C_1,\dots,C_k$ in $F-G(\mathfrak{J})$
such that $Q_i \subseteq C_i$ and  $C_i \cap G$ is $2$-balanced for each $i\le k$. 
\end{itemize}
\end{lemma}
\proof
We first prove (a). The argument is similar to that of Lemma~\ref{G-glob}. Roughly speaking, we will extend
each $Q_{i,A}$ into a path system $Q_{i,A}'$ by adding suitable $A_0B$-edges which ensure that 
every vertex in $A_0$ has degree exactly two in $Q_{i,A}'$. 
Similarly, we will extend each $Q_{i,B}$ into $Q'_{i,B}$ by adding suitable $AB_0$-edges. 
We will ensure that no vertex is an endvertex of both an edge in $Q'_{i,A}$ and an edge in $Q'_{i,B}$ and take $Q_i$ to
be the union of these two path systems. We first construct all the $Q_{i,A}'$.

\medskip

\noindent
{\bf Claim~1.} \emph{$G^*[A']\cup G^*[A_0,B]$ has a decomposition into edge-disjoint path systems $Q'_{1,A},\dots,Q'_{k,A}$ such that
\begin{itemize}
\item $Q_{i,A}\subseteq Q'_{i,A}$ and $E(Q'_{i,A})\setminus E(Q_{i,A})$ consists of $A_0B$-edges in $G^*$ (for each $i\le k$);
\item $d_{Q'_{i,A}}(x)=2$ for every $x\in A_0$ and $d_{Q'_{i,A}}(x)\le 1$ for every $x\notin A_0$;
\item no vertex is an endvertex of both an edge in $Q'_{i,A}$ and an edge in $Q_{i,B}$ (for each $i\le k$).%
    \COMMENT{Cannot simply say that $Q'_{i,A}$ and $Q_{i,B}$ are vertex-disjoint since we allow trivial paths in a path system
and so we might e.g. have that $V(Q_{i,B})=B'$.}
\end{itemize}
}

\medskip

\noindent
To prove Claim~1, let $G_{glob}$ be as defined in Corollary~\ref{finalcor}(d). Thus $G_{glob}[A']=Q_{1,A}\cup \dots \cup Q_{k,A}$.
On the other hand, Lemma~\ref{balmatchextend}(ii) implies that $G^*[A']=G_{glob}[A']$.
Hence,
\begin{equation}\label{eq:G*A'}
G^*[A']=G_{glob}[A']=Q_{1,A}\cup \dots \cup Q_{k,A}.
\end{equation}
Similarly, $G^*[B']=G_{glob}[B']=Q_{1,B}\cup \dots \cup Q_{k,B}$. Moreover, $G_{glob}=G^*[A']\cup G^*[B']$.
Consider any vertex $x \in A_0$. Let $d_{glob}(x)$ denote the degree of $x$ in $Q_{1,A}\cup \dots \cup Q_{k,A}$.
So $d_{glob}(x)=d_{G^*}(x,A')$ by~(\ref{eq:G*A'}). Let
\begin{align}
d_{loc}(x) & :=d_G(x,A')-d_{glob}(x)\label{eq:dloc1}\\
& =d_G(x,A')-d_{G^*}(x,A')=d_{G(\mathfrak{J})}(x,A')\label{eq:dloc2}.
\end{align}
Then
\begin{equation}\label{eq:dsum}
d_{loc}(x)+d_G(x,B')+d_{glob}(x)\stackrel{(\ref{eq:dloc1})}{=}d_G(x)=D,
\end{equation}
where the final equality follows from (BFR2).
Recall that $\mathfrak{J}$ consists of $K^4t_K$ edge-disjoint balanced exceptional systems.
Since $x$ has two neighbours in each of these balanced exceptional systems, 
the degree of $x$ in $G(\mathfrak{J})$ is $2K^4 t_K=D-2k$.
Altogether this implies that
\begin{eqnarray} \label{dB'x}
d_{G^*}(x,B') & = & d_{G}(x,B')-d_{G(\mathfrak{J})}(x,B')\nonumber \\
& = & d_{G}(x,B')-(d_{G(\mathfrak{J})}(x)-d_{G(\mathfrak{J})}(x,A'))\nonumber\\
& \stackrel{(\ref{eq:dloc2})}{=} & d_{G}(x,B')-(D-2k-d_{loc}(x))\stackrel{(\ref{eq:dsum})}{=} 2k-d_{glob}(x).
\end{eqnarray}
Note that this is precisely the total number of edges at $x$ which we need to add to $Q_{1,A},\dots,Q_{k,A}$
in order to obtain $Q'_{1,A},\dots,Q'_{k,A}$ as in Claim~1.

We can now construct the path systems $Q'_{i,A}$.  For each $x \in A_0$, let $n_i(x)=2-d_{Q_{i,A}}(x)$.
So $0\le n_i(x)\le 2$ for all $i\le k$.
Recall that $a:=|A_0|$ and consider an ordering $x_1,\dots,x_{a}$ of the vertices in $A_0$.
Let $G^*_j:=G^*[\{x_1,\dots,x_j\},B]$. Assume that for some 
$0 \le j < a$, we have already found a decomposition of $G^*_j$ into edge-disjoint path systems $Q_{1,j},\dots,Q_{k,j}$
satisfying the following properties (for all $i\le k$): 
\begin{itemize}
\item[(i$'$)] no vertex is an endvertex of both an edge in $Q_{i,j}$ and an edge in $Q_{i,B}$;
\item[(ii$'$)] $x_{j'}$ has degree $n_{i}(x_{j'})$ in $Q_{i,j}$ for all $j' \le j$
and all other vertices have degree at most one in $Q_{i,j}$.
\end{itemize}
We call this assertion ${\mathcal A}_j$. We will show that ${\mathcal A}_{j+1}$ holds 
(i.e.~the above assertion also holds with $j$ replaced by $j+1$).
This in turn implies Claim~1 if we let $Q'_{i,A}:=Q_{i,a}\cup Q_{i,A}$ for all $i\le k$.

To prove ${\mathcal A}_{j+1}$, consider the following bipartite auxiliary graph $H_{j+1}$.
The vertex classes of $H_{j+1}$ are $N_{j+1}:=N_{G^*}(x_{j+1})\cap B$ and $Z_{j+1}$, where
$Z_{j+1}$ is a multiset whose elements are chosen from $Q_{1,B},\dots,Q_{k,B}$.%
\COMMENT{AL: joined the two sentences together.}
Each $Q_{i,B}$ is included exactly $n_i(x_{j+1})$ times in $Z_{j+1}$. Note that
$N_{j+1}=N_{G^*}(x_{j+1})\cap B'$ since $e(G[A_0,B_0])=0$ by (BFR6).
Altogether this implies that
\begin{eqnarray}
\ \ \ \ \ \ |Z_{j+1}| & = & \sum_{i=1}^k n_i(x_{j+1})=2k-\sum_{i=1}^k d_{Q_{i,A}}(x_{j+1})=2k-d_{glob}(x_{j+1}) \label{eq:Zjplus1}\\
&  \stackrel{(\ref{dB'x})}{=} & d_{G^*}(x_{j+1},B') =|N_{j+1}| \ge k/2\nonumber.
\end{eqnarray}
The final inequality follows from (\ref{dB'x}) since
$$d_{glob}(x_{j+1})\stackrel{(\ref{eq:G*A'})}{\le} \Delta(G_{glob}[A']) \le 3k/2$$
by Corollary~\ref{finalcor}(d).
We include an edge in $H_{j+1}$ between $v \in N_{j+1}$ and $Q_{i,B} \in Z_{j+1}$ if $v$ is not an endvertex of an edge in 
$Q_{i,B}\cup Q_{i,j}$.

\medskip

\noindent
{\bf Claim~2.} \emph{$H_{j+1}$ has a perfect matching $M'_{j+1}$.}

\medskip

\noindent
Given the perfect matching guaranteed by the claim, we construct $Q_{i,j+1}$ 
from $Q_{i,j}$ as follows:
the edges of $Q_{i,j+1}$ incident to $x_{j+1}$ are precisely the edges $x_{j+1}v$ where $vQ_{i,B}$ is an edge of $M'_{j+1}$
(note that there are up to two of these). Thus Claim~2 implies that ${\mathcal A}_{j+1}$ holds. 
(Indeed, (i$'$)--(ii$'$) are immediate from the definition of~$H_{j+1}$.)
 
\medskip

\noindent
To prove Claim~2, consider any vertex $v \in N_{j+1}$. Since $v \in B$, the number of path systems $Q_{i,B}$ containing an edge at $v$ is 
at most $d_{G}(v,B')$. The number of indices $i$ for which $Q_{i,j}$ contains an edge at $v$ is at most $d_{G}(v,A_0) \le |A_0|$. 
Since each path system $Q_{i,B}$ occurs at most twice
in the multiset $Z_{j+1}$, it follows that the degree of $v$ in $H_{j+1}$ is at least
$|Z_{j+1}|-2d_{G}(v,B')-2|A_0|$. Moreover, $d_G(v,B') \le \eps' n \le k/16$ (say) by (BFR5).
Also, $|A_0| \le \eps n\le k/16$ by (BFR4).
So $v$ has degree at least $|Z_{j+1}|-k/4\ge |Z_{j+1}|/2$ in $H_{j+1}$.

Now consider any path system $Q_{i,B} \in Z_{j+1}$. Recall that $e(Q_{i,B})\le \sqrt{\eps}n\le k/16$ (say),
where the first inequality follows from Corollary~\ref{finalcor}(d)(iv$'$).
Moreover, $e(Q_{i,j}) \le 2|A_0|\le 2 \eps n \le k/8$,
where the second inequality follows from (BFR4). 
Thus the degree of $Q_{i,B}$ in $H_{j+1}$ is at least
$$
|N_{j+1}|-2e(Q_{i,B})-e(Q_{i,j})\ge |N_{j+1}|-k/4\ge |N_{j+1}|/2.
$$ 
Altogether this implies that $H_{j+1}$ has a perfect matching $M'_{j+1}$, as required.

\medskip
This completes the construction of $Q'_{1,A},\dots,Q'_{k,A}$. Next we construct $Q'_{1,B},\dots,$ $Q'_{k,B}$ using the same approach.%
\COMMENT{Deryk+AL: added sentences}

\medskip 

\noindent
{\bf Claim~3.} \emph{$G^*[B']\cup G^*[B_0,A]$ has a decomposition into edge-disjoint path systems $Q'_{1,B},\dots,Q'_{k,B}$ such that
\begin{itemize}
\item $Q_{i,B}\subseteq Q'_{i,B}$ and $E(Q'_{i,B})\setminus E(Q_{i,B})$ consists of $B_0A$-edges in $G^*$ (for each $i\le k$);
\item $d_{Q'_{i,B}}(x)=2$ for every $x\in B_0$ and $d_{Q'_{i,B}}(x)\le 1$ for every $x\notin B_0$;
\item no vertex is an endvertex of both an edge in $Q'_{i,A}$ and an edge in $Q'_{i,B}$ (for each $i\le k$).
\end{itemize}
}
\medskip

\noindent
The proof of Claim~3 is similar to that of Claim~1.%
\COMMENT{osthus changed 2 to 1}
The only difference is that when constructing $Q'_{i,B}$, we need to avoid the endvertices of all the edges in $Q'_{i,A}$
(not just the edges in $Q_{i,A}$). However, $e(Q'_{i,A}-Q_{i,A})\le 2|A_0|$, so this does not affect the calculations
significantly. 

\medskip

\noindent
We now take $Q_i:=Q'_{i,A}\cup Q'_{i,B}$ for all $i\le k$. Then the $Q_i$ are pairwise edge-disjoint
and $$e(Q_i)\le e(Q_{i,A})+e(Q_{i,B})+2|A_0\cup B_0|\le 2\sqrt{\eps}n+2\eps n\le 3\sqrt{\eps}n$$
by Corollary~\ref{finalcor}(d)(iv$'$) and (BFR4).
Moreover, Corollary~\ref{finalcor}(d)(iii$'$) implies that
\begin{equation*}
e_{Q_i}(A')-e_{Q_i}(B')=e(Q_{i,A})-e(Q_{i,B})=a-b.
\end{equation*}
Thus each $Q_i$ is a 2-balanced $A_0B_0$-path system. Further, $Q_1,\dots,Q_k$ form a decomposition of
$$G^*[A']\cup G^*[A_0,B]\cup G^*[B']\cup G^*[B_0,A]=G^*-G^*[A,B].$$
(The last equality follows since $e(G[A_0,B_0])=0$ by (BFR6).)
This completes the proof of~(a).

To prove (b), note that $(F,G,A,A_0,B,B_0)$ is an $(\eps,\eps',D)$-pre-framework, i.e.~it satisfies (WF1)--(WF5).
Indeed, recall that (BFR1)--(BFR4) imply (WF1)--(WF4) and that (WF5) holds by assumption.
So we can apply Lemma~\ref{extendpaths} (with $Q_1$ playing the role of $Q$) 
to extend $Q_1$ into a Hamilton cycle $C_1$.
Moreover, Lemma~\ref{extendpaths}(iii) implies that
$C_1 \cap G$ is $2$-balanced, as required. (Lemma~\ref{extendpaths}(ii) guarantees%
   \COMMENT{Daniela added ref to Lemma~\ref{extendpaths}(ii)}
that $C_1$ is edge-disjoint from $Q_2, \dots , Q_k$ and $G(\mathfrak{J})$.)

Let $G_1:=G-C_1$ and $F_1:=F-C_1$.%
\COMMENT{Daniela: previously had "Let $G_1$ be obtained from $G$ by removing the edges of $C_1$ and $F_1$ be obtained by the removing the edges of $C_1$."}
Proposition~\ref{WFpreserve} (with $C_1$ playing the role of $H$)
implies that $(F_1,G_1,A,A_0,B,B_0)$ is an $(\eps,\eps',D-2)$-pre-framework.
So we can now apply Lemma~\ref{extendpaths}%
    \COMMENT{Daniela replaced Lemma~\ref{coverA0B0} by Lemma~\ref{extendpaths}}
to $(F_1,G_1,A,A_0,B,B_0)$ to extend $Q_2$ into a Hamilton cycle $C_2$,
where $C_2 \cap G$ is also $2$-balanced.

We can continue this way to find $C_3,\dots, C_k$.
Indeed, suppose that we have found $C_1,\dots,C_i$ for $i<k$.
Then we can still apply Lemma~\ref{extendpaths}
since $\delta(F)-2i \ge \delta(F)-2k\ge n/3$.%
   \COMMENT{Daniela deleted " and since $D-2k\ge n/300$" - don't see why we need this} 
Moreover, $C_j \cap G$ is $2$-balanced for all $j \le i$, so $(C_1 \cup \dots \cup C_i) \cap G$ is $2i$-balanced.
This in turn means that 
Proposition~\ref{WFpreserve} (applied with $C_1 \cup \dots \cup C_i$ playing the role of $H$)
implies that after removing $C_1,\dots,C_i$, we still have an $(\eps,\eps',D-2i)$-pre-framework and can find $C_{i+1}$.
\endproof

 We can now put everything together to find a set of localized balanced exceptional systems
 and a set of Hamilton cycles which altogether cover all edges of $G$ outside $G[A,B]$.
 The localized balanced exceptional systems will be extended to Hamilton cycles later on.%
 \COMMENT{Deryk added sentences}
\begin{cor}\label{BEScor}
Let $0<1/n\ll \eps\ll \eps_0\ll \eps'\ll \eps_1\ll \eps_2\ll \eps_3\ll \eps_4 \ll 1/K\ll 1$.
Suppose that $(G,A,A_0,B,B_0)$ is an $(\eps,\eps',K,D)$-bi-framework with $|G|=n$ and so that $D\ge n/200$ and%
   \COMMENT{Daniela: had $\delta (G) \geq D\ge n/200$}
$D$ is even. Let $\cP:=\{A_0,A_1,\dots,A_K,B_0,B_1,\dots,$ $B_K\}$ be a $(K,m,\eps,\eps_1,\eps_2)$-partition for $G$. 
Suppose that $t_K:=(1-20\eps_4)D/2K^4\in \mathbb{N}$ and let $k:=10\eps_4 D$.
Suppose that $F$ is a graph on $V(G)$ such that $G\subseteq F$, $\delta(F)\ge 2n/5$ and such that $F$ satisfies
(WF5) with respect to $\eps '$. Then there are $k$ edge-disjoint Hamilton cycles $C_1,\dots,C_k$ in $F$
and for all $1\leq \I \leq K$ there is a set $\mathcal{J}_{i_1i_2i_3i_4}$ such that the
following properties are satisfied:
\begin{itemize}
\item[{\rm (i)}] $\mathcal{J}_{i_1i_2i_3i_4}$ consists of $t_K$ $(i_1,i_2,i_3,i_4)$-BES in $G$ with respect to $\cP$ and
with parameter $\eps_0$ which are edge-disjoint from each other
and from $C_1\cup\dots\cup C_k$.
\item[{\rm (ii)}] Whenever $(i_1,i_2,i_3,i_4)\neq (i'_1,i'_2,i'_3,i'_4)$, $J \in \mathcal{J}_{i_1i_2i_3i_4}$ and $J' \in \mathcal{J}_{i'_1i'_2i'_3i'_4}$,
then $J$ and $J'$ are edge-disjoint.
\item[{\rm (iii)}] Given any $i \leq k$ and $v \in A_0 \cup B_0$, the two edges incident 
to $v$ in $C_i$ lie in $G$.
\item[{\rm (iv)}] Let $G^\diamond$ be the subgraph of $G$ obtained by deleting the edges of all the $C_i$ and all
the balanced exceptional systems in $\mathcal{J}_{i_1i_2i_3i_4}$ (for all $1\leq \I \leq K$).
Then $G^\diamond$ is bipartite with vertex classes $A'$, $B'$ and $V_0=A_0 \cup B_0$%
\COMMENT{osthus added $A_0 cup B_0$} 
is an isolated set in $G^\diamond$.%
\COMMENT{Originally it was written that:
Moreover, $G^\diamond$ is $(D-2k-2K^4t_K)$-balanced with respect to $(A,A_0,B,B_0)$.
But note this is vacuous (i.e. $0$-balanced).}
\end{itemize}
\end{cor}
\proof
This follows immediately from Lemmas~\ref{balmatchextend} and~\ref{lem:HCglob}(b). Indeed, clearly (i)--(iii) are satisfied.
To check~(iv), note that $G^\diamond$ is obtained from the graph $G^*$ defined in Lemma~\ref{lem:HCglob}
by deleting all the edges of the Hamilton cycles $C_i$. But Lemma~\ref{lem:HCglob} implies that the $C_i$ together cover all the edges in
$G^*-G^*[A,B]$. Thus this implies that $G^\diamond$ is bipartite with vertex classes $A'$, $B'$
and $V_0$ is an isolated set in $G^\diamond$. 
\endproof


\section{Special Factors and Balanced Exceptional Factors}\label{sec:spec}

As discussed in the proof sketch, the proof of Theorem~\ref{1factbip} proceeds as follows.
First we find an approximate decomposition of the given graph $G$ and finally we find a decomposition of the (sparse) leftover from the approximate decomposition
(with the aid of a `robustly decomposable' graph we removed earlier).
Both the approximate decomposition as well as the actual decomposition steps assume that we work with a bipartite graph
on $A \cup B$ (with $|A|=|B|$). So in both steps,
we would need $A_0 \cup B_0$ to be empty, which we clearly cannot assume. On the other hand, in both steps, one can specify 
`balanced exceptional path systems' (BEPS) in $G$ with the following crucial property:
one can replace each BEPS with a path system BEPS$^*$ so that
\begin{itemize}
\item[($\alpha_1$)] BEPS$^*$ is bipartite with vertex classes $A$ and $B$;
\item[($\alpha_2$)] a Hamilton cycle $C^*$ in $G^*:=G[A,B] + {\rm BEPS}^*$ which%
    \COMMENT{Daniela replaced $G^*:=G[A,B] \cup {\rm BEPS}^*$ by $G^*:=G[A,B] + {\rm BEPS}^*$}
contains BEPS$^*$ corresponds to a Hamilton cycle $C$ in $G$
which contains BEPS (see Section~\ref{BESstar}).
\end{itemize}
Each BEPS will contain one of the balanced exceptional sequences BES constructed in Section~\ref{findBES}.
BEPS$^*$ will then be obtained by replacing the edges in BES by suitable `fictive' edges (i.e.~which are not necessarily contained in $G$). 

So, roughly speaking, this allows us to work with 
$G^*$ rather than $G$ in the two steps. Similarly as in the two clique case,
a convenient way of specifying and handling these balanced exceptional path systems is to combine them into `balanced exceptional factors' BF
(see Section~\ref{sec:BEPSq} for the definition). (The balanced exceptional path systems and balanced exceptional factors are analogues of 
the exceptional path systems and exceptional factors considered in Chapter~\ref{paper1}.)

As before, one complication is that the `robust decomposition lemma' (Lem\-ma~\ref{rdeclemma'}) we use from~\cite{Kelly}
deals with digraphs rather than undirected graphs. 
So to be able to apply it, we again need a suitable orientation of  the edges of $G$ and so we will actually consider directed path systems
BEPS$^*_{\rm dir}$ instead of BEPS$^*$ above (whereas the path systems BEPS are undirected).

Rather than guaranteeing ($\alpha_2$) directly, the (bipartite) robust decomposition lemma
 assumes the existence of certain directed `special paths systems' SPS which are combined into `special factors' SF.
(Recall that these notions were used in the proof of Theorem~\ref{1factstrong}; see Section~\ref{sec:SF}. In this chapter, we use slight variants of these definitions which
are introduced in Section~\ref{sec:SFq}.)
Each of the Hamilton cycles produced by the lemma then contains exactly one of these special path systems.
So to apply the lemma, it suffices to check separately that each BEPS$^*_{\rm dir}$ satisfies the conditions required of a special path system 
and that it also satisfies ($\alpha_2$).%
\COMMENT{Andy: changed title of section 6.1}

\subsection{Constructing the Graphs $J^*$ from the Balanced Exceptional Systems~$J$} \label{BESstar}

Suppose that $J$ is a balanced exceptional system in a graph~$G$ with respect to a $(K,m,\eps_0)$-partition
$\cP=\{A_0,A_1,\dots,A_K,B_0,B_1,\dots,B_K\}$ of $V(G)$. We will now use $J$ to define an auxiliary matching $J^*$.
Every edge of $J^*$ will have one endvertex in $A$ and its other endvertex in $B$.
We will regard $J^*$ as being edge-disjoint from the original graph~$G$. So 
even if both $J^*$ and $G$ have an edge between the same pair of endvertices, we will regard these as different edges.
The edges of such a $J^*$ will be called \emph{fictive edges}. Proposition~\ref{CES-H}(ii) below shows that a Hamilton cycle in $G[A\cup B]+ J^*$ containing
all edges of $J^*$ in a suitable order will correspond to a Hamilton cycle in~$G$ which contains~$J$.
So when finding our Hamilton cycles, this property will enable us to ignore all the vertices in $V_0=A_0\cup B_0$ and to consider
a bipartite (multi-)graph between $A$ and $B$ instead.

We construct $J^{*}$ in two steps.
First we will construct a matching $J^*_{AB}$ on $A \cup B$
and then $J^{*}.$%
	   \COMMENT{AL: I have rephrased the definition/construction of $J^*$. $J^*$ is still the same as before. The construction is similar to the two cliques case.  This is from paper 4.}
Since each maximal path in $J$ has endpoints in $A \cup B$ and internal vertices in $V_0$ by (BES1), a balanced exceptional system $J$ naturally induces a matching $J^*_{AB}$ on $A \cup B$.
More precisely, if $P_1, \dots ,P_{\ell'}$ are the non-trivial paths in~$J$ and $x_i, y_i$ are the endpoints of $P_i$, then we define $J^*_{AB} := \{x_iy_i : i  \le \ell'\}$. 
Thus $J^*_{AB}$ is a matching by~(BES1) and $e(J^*_{AB}) \le e(J)$.%
    \COMMENT{Daniela deleted " by~(BES4)"}
Moreover, $J^*_{AB}$ and $E(J)$ cover exactly the same vertices in $A$. 
Similarly, they cover exactly the same vertices in $B$. 
So (BES3) implies that $e(J^*_{AB}[A])=e(J^*_{AB}[B])$.
We can write $E(J^*_{AB}[A])=\{x_1x_2, \dots, x_{2s-1}x_{2s}\}$,
$E(J^*_{AB}[B])=\{y_1y_2, \dots, y_{2s-1}y_{2s}\}$ and $E(J^*_{AB}[A,B])=\{x_{2s+1}y_{2s+1}, \dots, x_{s'}y_{s'}\}$, where $x_i \in A$ and $y_i \in B$.
Define $J^*:= \{ x_iy_i : 1 \le i \le s' \}$.
Note that $	e(J^*) =  e(J^*_{AB}) \le e(J)$.
All edges of $J^*$ are called \emph{fictive edges}.

As mentioned before, we regard $J^*$ as being edge-disjoint from the original graph~$G$. 
Suppose that $P$ is an orientation of a subpath of (the multigraph) $G[A\cup B]+J^*$. 
We say that $P$ is \emph{consistent with $J^{*}$} if $P$ contains all the edges of $J^{*}$
and $P$ traverses the vertices $x_1,y_1,x_2,\dots,y_{s'-1},x_{s'},y_{s'}$ in this order.%
     \COMMENT{we need to prescribe a vertex rather than an edge ordering as it is important in the next lemma}
(This ordering will be crucial for the vertices $x_1,y_1,\dots,x_{2s},y_{2s}$,%
   \COMMENT{Daniela replaced "added in step (BES$^*$1)" by "$x_1y_1,\dots,x_{2s}y_{2s}$". Also defined when a cycle is consistent}
but it is also convenient to have an ordering
involving all vertices of~$J^*$.) 
Similarly, we say that a cycle $D$ in $G[A\cup B]+J^*$ is \emph{consistent with $J^{*}$} if $D$ contains all the edges of $J^{*}$
and there exists some orientation of $D$ which
traverses the vertices $x_1,y_1,x_2,\dots,y_{s'-1},x_{s'},y_{s'}$ in this order.

The next result shows that if $J$ is a balanced exceptional system and $C$ is a Hamilton cycle on
$A \cup B$ which is consistent with $J^*$, then the graph obtained from $C$ by replacing $J^*$
with $J$ is a Hamilton cycle on $V(G)$ which contains~$J$, see Figure~\ref{fig3}.
When choosing our Hamilton cycles, this property will enable us ignore all the vertices in $V_0$ and edges in $A$ and $B$ and to consider
the (almost complete) bipartite graph with vertex classes $A$ and $B$ instead.%
   \COMMENT{Daniela rephrased (ii) and also adjusted the proof to bring it in line with Allan's changed def of BES}

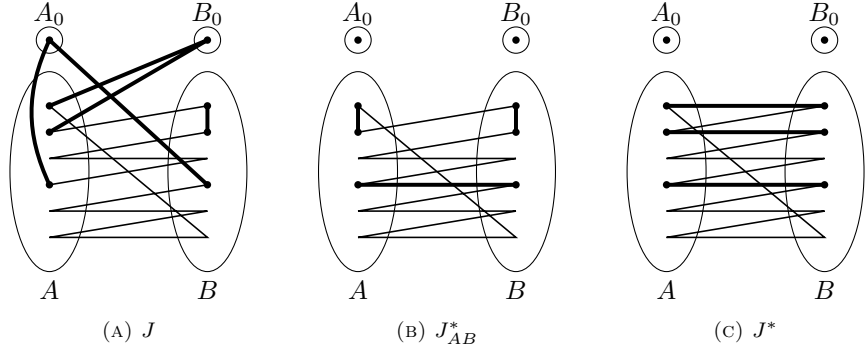
\begin{figure}[tbp]
\centering
\subfloat[$J$]{
\begin{tikzpicture}[scale=0.35]
			\draw  (-3,0) ellipse (1.5  and 3.8 );
			\draw (3,0) ellipse (1.5  and 3.8 );
			\node at (-3,-4.5) {$A$};
			\node at (3,-4.5) {$B$}; 			
			\draw (-3,5) circle(0.5);
			\draw (3,5) circle (0.5);
			\node at (-3,6) {$A_0$};
			\node at (3,6) {$B_0$};
			\fill (-3,5) circle (4pt);
			\fill (3,5) circle (4pt);
			\begin{scope}[start chain]
			\foreach \i in {2.5,1.5,-0.5}
			\fill (-3,\i) circle (4pt);
			\end{scope}
			\begin{scope}[start chain]
			\foreach \i in {2.5,1.5,-0.5}
			\fill (3,\i) circle (4pt);
			\end{scope}

			\begin{scope}[line width=1.5pt]
			\draw (-3,-0.5) to [out=115, in=-115] (-3,5)--(3,-0.5);
			\draw (-3,1.5)--(3,5)--(-3,2.5);
			\draw (3,2.5)--(3,1.5);
			\end{scope}
			
			\begin{scope}[line width=0.6pt]
			\draw (3,2.5)-- (-3,1.5);
			\draw (3,1.5)-- (-3,0.5)--(3,0.5)  -- (-3,-0.5);
			\draw (3,-0.5)--(-3,-1.5) -- (3,-1.5) --(-3,-2.5) -- (3,-2.5) -- (-3,2.5);
			\end{scope}
\end{tikzpicture}
}
\qquad
\subfloat[$J_{AB}^*$]{
\begin{tikzpicture}[scale=0.35]
			\draw  (-3,0) ellipse (1.5  and 3.8 );
			\draw (3,0) ellipse (1.5  and 3.8 );
			\node at (-3,-4.5) {$A$};
			\node at (3,-4.5) {$B$}; 			
			\draw (-3,5) circle(0.5);
			\draw (3,5) circle (0.5);
			\node at (-3,6) {$A_0$};
			\node at (3,6) {$B_0$};
			\fill (-3,5) circle (4pt);
			\fill (3,5) circle (4pt);
			\begin{scope}[start chain]
			\foreach \i in {2.5,1.5,-0.5}
			\fill (-3,\i) circle (4pt);
			\end{scope}
			\begin{scope}[start chain]
			\foreach \i in {2.5,1.5,-0.5}
			\fill (3,\i) circle (4pt);
			\end{scope}

			\begin{scope}[line width=1.5pt]
			\draw (-3,-0.5) --(3,-0.5);
			\draw (-3,1.5)--(-3,2.5);
			\draw (3,2.5)--(3,1.5);
			\end{scope}
			
			\begin{scope}[line width=0.6pt]
			\draw (3,2.5)-- (-3,1.5);
			\draw (3,1.5)-- (-3,0.5)--(3,0.5)  -- (-3,-0.5);
			\draw (3,-0.5)--(-3,-1.5) -- (3,-1.5) --(-3,-2.5) -- (3,-2.5) -- (-3,2.5);
			\end{scope}
\end{tikzpicture}
}
\qquad
\subfloat[$J^*$]{
\begin{tikzpicture}[scale=0.35]
			\draw  (-3,0) ellipse (1.5  and 3.8 );
			\draw (3,0) ellipse (1.5  and 3.8 );
			\node at (-3,-4.5) {$A$};
			\node at (3,-4.5) {$B$}; 			
			\draw (-3,5) circle(0.5);
			\draw (3,5) circle (0.5);
			\node at (-3,6) {$A_0$};
			\node at (3,6) {$B_0$};
			\fill (-3,5) circle (4pt);
			\fill (3,5) circle (4pt);
			\begin{scope}[start chain]
			\foreach \i in {2.5,1.5,-0.5}
			\fill (-3,\i) circle (4pt);
			\end{scope}
			\begin{scope}[start chain]
			\foreach \i in {2.5,1.5,-0.5}
			\fill (3,\i) circle (4pt);
			\end{scope}

			\begin{scope}[line width=1.5pt]
			\draw (-3,-0.5)--(3,-0.5);
			\draw (-3,1.5)--(3,1.5);
			\draw (3,2.5)--(-3,2.5);
			\end{scope}

			\begin{scope}[line width=0.6pt]
			\draw (3,2.5)-- (-3,1.5);
			\draw (3,1.5)-- (-3,0.5)--(3,0.5)  -- (-3,-0.5);
			\draw (3,-0.5)--(-3,-1.5) -- (3,-1.5) --(-3,-2.5) -- (3,-2.5) -- (-3,2.5);
			\end{scope}
\end{tikzpicture}
}
\caption{The thick lines illustrate the edges of $J$, $J_{AB}^*$ and $J^*$ respectively.}
\label{fig3}
\end{figure}

\begin{prop}\label{CES-H}
Let $\cP=\{A_0,A_1,\dots,A_K,B_0,B_1,\dots,B_K\}$ be a $(K,m,\eps)$-partition of a vertex set $V$.
Let $G$ be a graph on $V$ and let $J$ be a balanced exceptional system with respect to~$\cP$.
\begin{itemize}
\item[{\rm (i)}] Assume that $P$ is an orientation of a subpath of $G[A\cup B]+ J^*$ such that $P$ is consistent with~$J^{*}.$ 
Then the graph obtained from $P-J^{*}+J$ by ignoring the orientations of the edges is a path on
$V(P) \cup V_0$ whose endvertices are the same as those of~$P$.
\item[{\rm (ii)}] If $J\subseteq G$ and
$D$ is a Hamilton cycle of $G[A\cup B]+ J^*$ which is consistent with~$J^{*},$ 
then $D-J^*+J$ is a Hamilton cycle of $G$.
\end{itemize}
\end{prop}
\proof
We first prove~(i). 
Let $s:=e(J^*_{AB}[A]) = e(J^*_{AB}[B])$ and $J^\diamond:=\{x_1y_1,\dots,$ $x_{2s}y_{2s}\}$ (where the $x_i$ and $y_i$ are
as in the definition of $J^*$). So
$J^*:=J^\diamond\cup \{x_{2s+1}y_{2s+1},\dots,x_{s'}y_{s'}\}$, where $s':=e(J^*)$.
Let $P^c$ denote the path obtained from $P=z_1\dots z_2$ by reversing its direction. (So $P^c=z_2\dots z_1$ traverses
the vertices $y_{s'}, x_{s'}, y_{2s'-1}, \dots, x_2, y_1,x_1$ in this order.) 
First note
$$
P':=z_1Px_1x_2P^cy_1y_2Px_3x_4P^cy_3y_4\dots x_{2s-1}x_{2s}P^cy_{2s-1}y_{2s}Pz_2
$$
is a path on $V(P)$.
Moreover, the underlying undirected graph of $P'$ is%
\COMMENT{osthus added $=P-J^*+J^*_{AB}.$} 
precisely 
$$P-J^\diamond+(J^*_{AB}[A]\cup J^*_{AB}[B])=P-J^*+J^*_{AB}.
$$
In particular, $P'$ contains $J^*_{AB}$.
Now recall that if $w_1w_2$ is an edge in $J^*_{AB}$, then the vertices $w_1$ and $w_2$ are the endpoints of some path $P^*$ in $J$
(where the internal vertices on $P^*$ lie in $V_0$).
Clearly, $P'-w_1w_2+P^*$ is also a path. Repeating this step for every edge $w_1w_2$ of $J^*_{AB}$
gives a path $P''$ on $V(P)\cup V_0$. Moreover, $P''=P-J^{*}+J$. This completes the proof of~(i).

 (ii) now follows immediately from~(i).
\endproof


\subsection{Special Path Systems and Special Factors}\label{sec:SFq}
As mentioned earlier, in order to apply Lemma~\ref{rdeclemma'}, we first need  to prove the existence of certain `special path systems'.
These are defined below. Note that the definitions given in this section are slight variants of the corresponding definitions used in Chapter~\ref{paper1}. 

Suppose that 
$$
\mathcal{P}=\{A_0,A_1,\dots,A_K,B_0,B_1,\dots,B_K\}
$$ is a $(K,m,\eps_0)$-partition of a vertex set~$V$ and $L,m/L\in\mathbb{N}$.
Recall that we say that $(\mathcal{P},\mathcal{P'})$%
\COMMENT{the roles of $\mathcal{P},\mathcal{P'}$ were interchanged in a recent version. There may be places where this was overlooked.}
 is a \emph{$(K, L, m, \epszero)$-partition of~$V$} if $\mathcal{P'}$ is 
obtained from $\mathcal{P}$ by partitioning  $A_i$ 
into $L$ sets $A_{i,1},\dots,A_{i,L}$ of size~$m/L$ for all $1 \leq i \leq K$ and partitioning  $B_i$  into $L$ sets $B_{i,1},\dots,B_{i,L}$ of size~$m/L$ for all $1 \leq i \leq K$.
(So $\cP'$ consists of the exceptional sets $A_0$, $B_0$, the $KL$ clusters $A_{i,j}$ and the $KL$ clusters $B_{i,j}$.)
Unless stated otherwise, whenever considering a $(K, L, m, \epszero)$-partition $(\mathcal{P},\mathcal{P'})$ of a vertex set $V$
we use the above notation to denote the elements of $\mathcal P$ and $\mathcal P'$.

Let $(\mathcal{P},\mathcal{P'})$ be a $(K,L,m, \epszero)$-partition of~$V$.
Consider a spanning cycle $C = A_1 B_1 \dots A_K B_K$ on the clusters of $\mathcal{P}$.
Given an integer $f$ dividing $K$, the \emph{canonical interval partition} $\mathcal{I}$ of $C$ into $f$ intervals
consists of the intervals $$A_{(i-1)K/f+1} B_{(i-1)K/f+1} A_{(i-1)K/f+2} \dots B_{iK/f} A_{iK/f+1}$$ for all $i\le f$.
(Here $A_{K+1}:=A_1$.)

Suppose that $G$ is a digraph on~$V\setminus V_0$ and%
    \COMMENT{Have $V\setminus V_0$ instead of $V$ here since this is what we use later.}
$h\le L$. Let $I=A_jB_jA_{j+1}\dots A_{j'}$ be an interval in~$\mathcal{I}$.
A \emph{special path system $SPS$ of style~$h$ in~$G$ 
spanning the interval $I$} consists of precisely $m/L$ (non-trivial)
vertex-disjoint directed paths $P_1,\dots,P_{m/L}$ such that the following conditions hold:
\begin{itemize}
\item[(SPS1)] Every $P_s$ has its initial vertex in $A_{j,h}$ and its final vertex in $A_{j',h}$.
\item[(SPS2)] $SPS$ contains a matching ${\rm Fict}(SPS)$ such that all the edges in ${\rm Fict}(SPS)$ avoid the
endclusters $A_j$ and $A_{j'}$ of $I$ and such that $E(P_s)\setminus {\rm Fict}(SPS)\subseteq E(G)$.
\item[(SPS3)] The vertex set of $SPS$ is $A_{j,h}\cup B_{j,h}\cup A_{j+1,h}\cup \dots \cup B_{j'-1,h}\cup A_{j',h}$.
\end{itemize}
The edges in ${\rm Fict}(SPS)$ are called \emph{fictive edges of} $SPS$.

Let $\mathcal{I}=\{I_1,\dots,I_f\}$ be the canonical interval partition of $C$ into $f$ intervals.%
	\COMMENT{AL: added `the canonical interval partition of $C$ into $f$ intervals'}
A \emph{special factor $SF$ with parameters $(L,f)$ in $G$ (with respect to $C$, $\mathcal P'$)}
is a $1$-regular digraph on $V\setminus V_0$
which is the union of $Lf$ digraphs $SPS_{j,h}$ (one for all $j\le f$ and $h\le L$) such that each $SPS_{j,h}$ is
a special path system of style $h$ in $G$ which spans~$I_j$.
We write ${\rm Fict}(SF)$ for the union of the sets ${\rm Fict}(SPS_{j,h})$ over all $j\le f$ and $h\le L$
and call the edges in ${\rm Fict}(SF)$ \emph{fictive edges of $SF$}.

We will always view fictive edges as being distinct from each other and from the edges in other digraphs.
So if we say that special factors $SF_1,\dots,SF_r$ are pairwise edge-disjoint from each other and from some digraph $Q$ on $V\setminus V_0$,
then%
   \COMMENT{It's not enough to have this for all $Q\subseteq G$ since in the robust decomposition lemma $H$ need not be a subgraph of $G$.}
this means that $Q$ and all the $SF_i- {\rm Fict}(SF_i)$
are pairwise edge-disjoint, but for example there could be an edge from $x$ to $y$ in $Q$ as well as in ${\rm Fict}(SF_i)$ 
for several indices $i\le r$.
But these are the only instances of multiedges that we allow, i.e.~if there is more than one edge from $x$ to $y$, then all but
at most one of these edges are fictive edges.

\subsection{Balanced Exceptional Path Systems and Balanced Exceptional Factors}\label{sec:BEPSq}
We now define balanced exceptional path systems BEPS. 
It will turn out that they (or rather their bipartite directed versions BEPS$^*_{\rm dir}$ involving fictive edges)
will satisfy the conditions of the special path systems defined above. 
Moreover,  Hamilton cycles that respect the partition $A,B$ and which contain BEPS$^*_{\rm dir}$ correspond to Hamilton cycles in the `original' graph $G$
(see Proposition~\ref{prop:CEPSbiparite}).

Let $(\mathcal{P},\mathcal{P'})$ be a $(K,L,m, \epszero)$-partition of a vertex set $V$.
Suppose that $K/f\in\mathbb{N}$ and $h\le L$.
Consider a spanning cycle $C = A_1 B_1 \dots A_K B_K$ on the clusters of $\mathcal{P}$.
Let $\mathcal{I}$ be the canonical interval partition of $C$ into~$f$ intervals of equal size.
Suppose that $G$ is an oriented bipartite graph with vertex classes $A$ and $B$.
Suppose that $I=A_jB_j\dots A_{j'}$ is an interval in~$\mathcal{I}$.%
	\COMMENT{AL: changed line, Daniela added in~$\mathcal{I}$}
A \emph{balanced exceptional path system $BEPS$ of style~$h$ for $G$ spanning  $I$}
consists of precisely $m/L$ (non-trivial) vertex-disjoint undirected paths $P_1,\dots,P_{m/L}$ such that the following conditions hold:
\begin{enumerate}
\item[(BEPS1)] Every $P_s$ has one endvertex in $A_{j,h}$ and its other endvertex in $A_{j',h}$.
\item[(BEPS2)] $J := BEPS - BEPS[A,B]$ is a balanced exceptional system with respect to~$\cP$%
\COMMENT{AL: we cannot write 'There is a balanced exceptional system $J$\dots'. Because we want ALL balanced exceptional systems in BEPS to avoid $A_{j,h}$ and $A_{j',h}$.}
such that $P_1$ contains all edges of~$J$ and
so that the edge set of $J$ is disjoint from $A_{j,h}$ and $A_{j',h}$.%
    \COMMENT{Cannot just write $J\subseteq P_1$ since $J$ might contain vertices of degree 0 which don't belong to $P_1$.}
Let $P_{1,{\rm dir}}$ be the path obtained by orienting $P_1$ towards its endvertex in $A_{j',h}$
and let $J_{\rm dir}$ be the orientation of $J$ obtained in this way. Moreover, let $J^*_{\rm dir}$ be obtained from $J^*$ by orienting every edge in $J^*$
towards its endvertex in $B$. Then $P^*_{1,{\rm dir}}:=P_{1,{\rm dir}}-J_{\rm dir}+J^*_{\rm dir}$ is a directed path from $A_{j,h}$ to $A_{j',h}$ which is consistent with $J^*$.
\item[(BEPS3)] The vertex set of $BEPS$ is $V_0\cup A_{j,h}\cup B_{j,h}\cup A_{j+1,h}\cup \dots \cup B_{j'-1,h}\cup A_{j',h}$.
\item[(BEPS4)] For each $2\le s\le m/L$, define $P_{s,{\rm dir}}$ similarly as $P_{1,{\rm dir}}$.
Then $E(P_{s,{\rm dir}})\setminus E(J_{\rm dir}) \subseteq E(G)$ for every $1\le s\le m/L$.  
\end{enumerate}

Let%
    \COMMENT{Daniela deleted: Recall (BES2) that a balanced exceptional system does not contain $AB$-edges.
Since $G$ is bipartite, (BEPS2) and (BEPS4) implies that $J := BEPS - BEPS[A,B]$.} 
$BEPS^*_{\rm dir}$ be the path system consisting of $P^*_{1,{\rm dir}},P_{2,{\rm dir}},\dots,P_{m/L,{\rm dir}}$.
Then $BEPS^*_{\rm dir}$ is a special path system of style $h$ in $G$ which spans the interval~$I$ and such that ${\rm Fict}(BEPS^*_{\rm dir})=J^*_{\rm dir}$.

Let $\mathcal{I}=\{I_1,\dots,I_f\}$ be the canonical interval partition of $C$ into $f$ intervals.%
	\COMMENT{AL: added `the canonical interval partition of $C$ into $f$ intervals'}
A \emph{balanced exceptional factor $BF$ with parameters $(L,f)$ for $G$ (with respect to $C$, $\mathcal P'$)} is 
the union of $Lf$ undirected graphs $BEPS_{j,h}$ (one for all $j\le f$ and $h\le L$) such that each $BEPS_{j,h}$ is
a balanced exceptional path system of style $h$ for $G$ which spans~$I_j$. We write $BF^*_{\rm dir}$ for the union of $BEPS_{j,h,{\rm dir}}^*$ over all
$j\le f$ and $h\le L$. Note that $BF^*_{\rm dir}$ is a special factor with parameters $(L,f)$ in $G$ (with respect to $C$, $\mathcal P'$)
such that ${\rm Fict}(BF^*_{\rm dir})$ is the union of $J^*_{j,h,{\rm dir}}$ over all $j\le f$ and $h\le L$, where
$J_{j,h}=BEPS_{j,h}-BEPS_{j,h}[A,B]$ is%
    \COMMENT{Daniela added $=BEPS_{j,h}-BEPS_{j,h}[A,B]$}
the balanced exceptional system contained in $BEPS_{j,h}$ (see condition (BEPS2)). In particular, $BF^*_{\rm dir}$ is a $1$-regular digraph on $V\setminus V_0$ 
while  $BF$ is an undirected graph on $V$ with 
\begin{align}\label{EFdegq}
	d_{BF}(v) = 2 \ \ \text{for all } v \in V \setminus V_0 \ \ \ \textrm{ and } \ \ \ d_{BF}(v) = 2Lf \ \ \text{for all } v \in V_0.	
\end{align}

Given a balanced exceptional path system $BEPS$, let $J$ be as in (BEPS2) and
let $BEPS^*:=BEPS-J+J^*$. So $BEPS^*$ consists of $P^*_1:=P_1-J+J^*$ as well as $P_2,\dots,P_{m/L}$.
The following is an immediate consequence of (BEPS2) and Proposition~\ref{CES-H}.

\begin{prop} \label{prop:CEPSbiparite}
Let $(\cP,\cP')$ be a $(K,L, m , \epszero)$-partition of a vertex set $V$.
Suppose that $G$ is a graph on~$V\setminus V_0$, that $G_{\rm dir}$ is an orientation of $G[A,B]$ and that $BEPS$ is a balanced exceptional
path system for~$G_{\rm dir}$. Let $J$ be as in (BEPS2). Let $C$ be a Hamilton cycle of $G+J^*$ which
contains $BEPS^*$. Then $C - BEPS^*+BEPS$ is a Hamilton cycle of $G\cup J$.
\end{prop}
\proof
Note that $C - BEPS^*+BEPS=C-J^*+J$. Moreover, (BEPS2) implies that $C$ contains all edges of $J^*$ and is
consistent with $J^*$.%
   \COMMENT{Daniela: had "traverses the
vertices $x_1,y_1,x_2,\dots,y_{s'-1},x_{s'},y_{s'}$ in this order" instead of "is consistent with $J^*$"}
So the proposition follows from Proposition~\ref{CES-H}(ii)
applied with $G\cup J$ playing the role of~$G$.
\endproof


\subsection{Finding Balanced Exceptional Factors in a Bi-scheme}\label{sec:findBFq}
The following definition of a `bi-scheme' captures the `non-exceptional' part of the graphs we are working with.%
\COMMENT{osthus added para}
For example, this will be the structure within which we find the edges needed to extend a balanced exceptional system into a balanced exceptional path system.

Given an oriented graph $G$ and partitions $\mathcal P$ and $\cP'$ of a vertex set $V$, we call $(G, \mathcal{P},\mathcal{P}')$ a
\emph{$[K,L,m,\eps_0,\eps]$-bi-scheme} if the following properties hold:
\begin{itemize}
\item[(BSch$1'$)] $(\mathcal{P},\mathcal{P}')$ is a $(K,L,m,\eps_0)$-partition of $V$. Moreover, $V(G)=A\cup B$.
\item[(BSch$2'$)] Every edge of $G$ has one endvertex in $A$ and its other endvertex in~$B$.
\item[(BSch$3'$)] $G[A_{i,j},B_{i',j'}]$ and $G[B_{i',j'}, A_{i,j}]$ are $[\eps, 1/2]$-superregular
for all $i, i'\le K$ and all $j, j' \leq L$. Further, $G[A_i,B_j]$ and $G[B_j,A_i]$ are $[\eps, 1/2]$-superregular
for all $i,j \leq K$.
\item[(BSch$4'$)] $|N_{G}^+(x)\cap N_{G}^-(y)\cap B_{i,j}|\ge (1-\eps) m/5L$ for all distinct $x,y\in A$, all $i\le K$ and all $j\le L$.
Similarly, $|N_{G}^+(x)\cap N_{G}^-(y)\cap A_{i,j}|\ge(1-\eps)  m/5L$ for all distinct $x,y\in B$, all $i\le K$ and all $j\le L$.\COMMENT{NOTE! added error term here!}
\end{itemize}
If $L=1$ (and so $\cP=\cP'$), then (BSch$1'$) just says that
$\mathcal{P}$ is a $(K,m,\eps_0)$-partition of $V(G)$.%
    \COMMENT{AL:removed notation for $[K,m, \eps_0, \eps]$}

The next lemma allows us to extend a suitable balanced exceptional system into a balanced exceptional path system.%
   \COMMENT{It does not seem to be possible to unify this in a simple way with the balancing lemma of the approximate cover chapter}
Given $h\le L$, we say that an $(i_1, i_2 , i_3 , i_4)$-BES $J$ has \emph{style $h$ (with respect to the $(K,L,m,\eps_0)$-partition
$(\mathcal{P},\mathcal{P}')$)} if all the edges of $J$
have their endvertices in $V_0\cup A_{i_1,h}\cup A_{i_2,h}\cup B_{i_3,h}\cup B_{i_4,h}$.

\begin{lemma} \label{lma:bipartite:CEPS}
Suppose that $K, L, n, m/L \in \mathbb{N}$, that $0 <1/n  \ll \eps, \eps_0\ll 1$ and $\eps_0\ll 1/K,1/L$.
Let $(G,\mathcal{P}, \mathcal{P}')$ be a $[K, L, m,\eps_0,\eps]$-bi-scheme with $|V(G)\cup V_0|=n$.
Consider a spanning cycle $C = A_1 B_1 \dots A_K B_K$ on the clusters of $\mathcal{P}$ and let 
$I  = A_j B_j A_{j+1} \dots  A_{j'}$ be an interval on~$C$ of length at least~$10$.
Let $J$ be an $(i_1, i_2 , i_3 , i_4)$-BES of style $h\le L$ with parameter $\eps_0$%
    \COMMENT{We use $\eps_0$ for both the $(K,m, \epszero)$-partition $\cP'$ and the bound on $e(J)$. But this seems
to be ok for our applications. Also, can probably replace the 10 for the interval length by something smaller, but 10 seems safe.}
(with respect to $(\cP,\cP')$), for some $i_1, i_2, i_3 , i_4 \in \{j+1, \dots, j'-1\}$.
Then there exists a balanced exceptional path system of style $h$ for $G$
which spans the interval $I$ and contains all edges in~$J$.
\end{lemma}
\begin{proof}
For each $k\le 4$, let $m_k$ denote the number of vertices in $A_{i_k,h} \cup B_{i_k,h}$ which are incident to edges of~$J$.%
     \COMMENT{Cannot just write $m_k := |V(J) \cap (A_{i_k,h} \cup B_{i_k,h})|$ since $V(J)$ contains vertices of degree zero.}
We only consider the case when $i_1$, $i_2$, $i_3$ and $i_4$ are distinct and $m_k>0$ for each $k \le 4$, as the other cases can be
proved by similar arguments.%
\COMMENT{The argument is essentially identical in the other cases. Only now, we don't consider 
$m_1,\dots, m_4$ but a subset of them. For example, consider the case when
$i_1,\dots, i_4$ are all distinct except $i_3=i_4$ and each $m_i >0$. Then note that
$m_3=m_4$. So we now follow the argument as before but only use $m_1,m_2,m_3$.
So for example, we now have
\begin{align*}
	 |V(P^*_{1,{\rm dir}}) \cap A_{i,h}| &= 
	 \begin{cases}
	1	& \textrm{for $i \in \{j, \dots, j'\} \setminus \{i_1, i_2, i_3 \}$,}\\
	m_k	& \textrm{for $i = i_k$ and $k \le 3$,}\\
	0	& \textrm{otherwise.}
	 \end{cases}
\end{align*}
Further, now for each $k \le 3$, we choose $m_k-1$ paths $P_1^{k}, \dots, P_{m_{k}-1}^{k}$ in $G$. So now $Q$ is a path system consisting of $m_1+m_2+m_3 - 2 $ vertex-disjoint directed paths from $A_{j,h}$ to $A_{j',h}$.
\\
In cases where one of the $m_i=0$, we similarly `ignore' $m_i$ and follow the corresponding
calculations. } 
Clearly $m_1+\dots+m_4 \le  2\epszero n$ by (BES4).
For every vertex $x \in A$, we define $B(x)$ to be the cluster $B_{i,h}\in\mathcal{P}'$ such that $A_i$ contains $x$.
Similarly, for every $y \in B$, we define $A(y)$ to be the cluster $A_{i,h}\in\mathcal{P}'$ such that $B_i$ contains~$y$.%
     \COMMENT{Daniela replaced $\mathcal{P}$ by $\mathcal{P}'$ (twice)}

Let $x_1y_1, \dots, x_{s'} y_{s'}$ be the edges of $J^*$, with $x_i \in A$ and $y_i \in B$ for all $i \le s'$.
(Recall that the ordering of these edges is fixed in the definition of $J^*$.)
Thus $s' = (m_1+\dots+m_4)/2 \le  \epszero n $. Moreover, our assumption that $\eps_0\ll 1/K,1/L$ implies that
$\eps_0 n\le m/100L$ (say). Together with (BSch$4'$) this in turn ensures that 
for every $r \le s'$, we can pick vertices $w_r \in B (x_r) $ and $z_r \in A(y_r)$ such that $w_rx_r$, $y_rz_r$ and $z_r w_{r+1}$ are
(directed) edges in $G$
and such that all the $4s'$ vertices $x_r,y_r,w_r,z_r$ (for $r\le s'$) are distinct from each other.
Let $P_1'$ be the path $w_1 x_1 y_1 z_1 w_2 x_2 y_2 z_2 w_3 \dots y_{s'} z_{s'}$.
Thus $P_1'$ is a directed path from $B$ to $A$ in $G + J^*_{\rm dir}$ which is consistent with~$J^*$. 
(Here $J^*_{\rm dir}$ is obtained from $J^*$ by orienting every edge towards~$B$.%
\COMMENT{Deryk})
Note that $|V(P'_1) \cap A_{i_k,h}| = m_k = |V(P'_1) \cap B_{i_k,h}|$ for all $k \le 4$.
(This follows from our assumption that $i_1$, $i_2$, $i_3$ and $i_4$ are distinct.) Moreover,
$V(P'_1) \cap ( A_i \cup B_i ) = \emptyset$ for all $i \notin \{i_1, i_2, i_3, i_4 \}$.

Pick a vertex $z'$ in $A_{j,h}$ so that $z' w_1 $ is an edge of $G$.
Find a path $P''_1$ from $z_{s'}$ to $A_{j',h}$ in $G$ such that the vertex set of $P''_1$ consists of $z_{s'}$ and precisely one vertex in
each $A_{i,h}$ for all $i \in \{j+1, \dots, j' \} \setminus \{i_1, i_2, i_3, i_4 \}$ 
and one vertex in each $B_{i,h}$ for all $i \in \{j, \dots, j'-1 \} \setminus \{i_1, i_2, i_3, i_4 \}$ and no other vertices.
(BSch$4'$) ensures that this can be done greedily.%
    \COMMENT{Daniela replaced (BSch$3'$) by (BSch$4'$)}
Define $P^*_{1,{\rm dir}}$ to be the concatenation of $z'w_1$, $P'_1$ and $P''_1$.
Note that $P^*_{1,{\rm dir}}$ is a directed path from $A_{j,h}$ to $A_{j',h}$ in $G + J^*_{\rm dir}$ which is consistent with $J^*$.
Moreover, $V(P^*_{1,{\rm dir}}) \subseteq \bigcup_{i \le K} A_{i,h} \cup B_{i,h}$,%
\COMMENT{AL: added $V(P^*_{1,{\rm dir}}) \subseteq \bigcup_{i \le K} A_{i,h} \cup B_{i,h}$}
\begin{align*}
	 |V(P^*_{1,{\rm dir}}) \cap A_{i,h}| &= 
	 \begin{cases}
	1	& \textrm{for $i \in \{j, \dots, j'\} \setminus \{i_1, i_2, i_3, i_4 \}$,}\\
	m_k	& \textrm{for $i = i_k$ and $k \le 4$,}\\
	0	& \textrm{otherwise,}
	 \end{cases}
\end{align*}
while	 
\begin{align*}
	  	 |V(P^*_{1,{\rm dir}}) \cap B_{i,h}| & = 
	 \begin{cases}
	1	& \textrm{for $i \in \{j, \dots, j'-1\} \setminus \{i_1, i_2, i_3, i_4 \}$},\\
	m_k	& \textrm{for $i = i_k$ and $k \le 4$},\\
	0	& \textrm{otherwise.}
	 \end{cases}
\end{align*}
(BSch$4'$) ensures that for each $k \le 4$, there exist $m_k-1$ (directed) paths $P_1^{k}, \dots,$ $ P_{m_{k}-1}^{k}$ in $G$ such that 
\begin{itemize}
	\item $P_r^{k}$ is a path from $A_{j,h}$ to $A_{j',h}$ for each $r\le m_k-1$ and $k \le 4$;
	\item each $P_r^{k}$ contains precisely one vertex in $A_{i,h}$ for each $i\in \{j, \dots, j'\} \setminus \{i_k\}$,
one vertex in $B_{i,h}$ for each $i\in \{j, \dots, j'-1\} \setminus \{i_k\}$ and no other vertices;
	\item $P^*_{1,{\rm dir}},P_1^{1},\dots, P_{ m_{1}-1}^{1},P_1^2, \dots, P_{m_4-1}^4$  are vertex-disjoint.
\end{itemize}
Let $Q$ be the union of $P^*_{1,{\rm dir}}$ and all the $P_r^{k}$ over all $k \le 4$ and $r \le m_k-1$.
Thus $Q$ is a path system consisting of $m_1+\dots+m_4 - 3 $ vertex-disjoint directed paths from $A_{j,h}$ to $A_{j',h}$.
Moreover, $V(Q)$ consists of precisely $m_1+\dots+m_4 - 3 \le  2\epszero n$ vertices in $A_{i,h}$
for every $j\le i\le j'$ and precisely $m_1+\dots+m_4-3$ vertices in $B_{i,h}$ for every $j\le i<j'$. 
Set $A'_{i,h}: = A_{i,h} \setminus V(Q)$ and $B'_{i,h}:= B_{i,h} \setminus V(Q)$ for all $i \le K$.
Note that, for all $j \leq i \leq j'$,
\begin{equation}\label{eq:sizeAihq}
|  A'_{i,h}|= \frac{m}{L}-(m_1+\dots+m_4-3)\ge \frac{m}{L}-2\eps_0 n\ge \frac{m}{L}-5\eps_0 mK\ge (1-\sqrt{\eps_0})\frac{m}{L}
\end{equation}
since $\eps_0\ll 1/K,1/L$. Similarly, $|  B'_{i,h}| \geq (1-\sqrt{\eps_0}){m}/{L}$ for all
 $j \leq i < j'$.
Pick a new constant $\eps'$ such that $\eps,\eps_0 \ll \eps'\ll 1$.
Then (BSch$3'$) and (\ref{eq:sizeAihq}) together with Proposition~\ref{superslice} imply that 
$G[A'_{i,h}, B'_{i,h}]$ is still $[\eps',1/2]$-superregular 
and so we can find a perfect matching in $G[A_{i,h}',B_{i,h}']$
for all $j \le i < j'$. Similarly, we can find a perfect matching in $G[B_{i,h}',A_{i+1,h}']$ for all $j \le i < j'$.
The union $Q'$ of all these matchings forms $m/L-(m_1+\dots+m_4) +3$ vertex-disjoint directed paths.

Let $P_1$ be the undirected graph obtained from $P^*_{1,{\rm dir}}-J^*_{\rm dir}+J$ by ignoring the
directions of all the edges. Proposition~\ref{CES-H}(i) implies that $P_1$
is a path on $V(P^*_{1,{\rm dir}})\cup V_0$ with the same endvertices as $P^*_{1,{\rm dir}}$.
Consider the path system obtained from $(Q\cup Q') \setminus \{P^*_{1,{\rm dir}}\}$ by ignoring the directions of the edges on all
the paths. Let $BEPS$ be the union of this path system and $P_1$. Then $BEPS$ is a balanced exceptional path system for $G$, as required.
\end{proof}

The next lemma shows that we can obtain many edge-disjoint balanced exceptional factors by extending balanced exceptional systems with suitable properties.

\begin{lemma} \label{lma:EF-bipartite}
Suppose that $L,f,q,n,m/L,K/f \in \mathbb{N}$, that $K/f\ge 10$, that $0 <1/n \ll \eps,\eps_0  \ll 1$,
that $\eps_0\ll 1/K,1/L$ and $Lq/m\ll 1$.
Let $(G,\mathcal{P}, \mathcal{P}')$ be a $[K, L, m,\eps_0,\eps]$-bi-scheme with $|V(G)\cup V_0|=n$.
Consider a spanning cycle $C = A_1 B_1 \dots A_K B_K$ on the clusters of $\mathcal{P}$.
Suppose that there exists a set $\mathcal{J}$ of $Lf q $ edge-disjoint balanced exceptional systems
with parameter $\eps_0$ such that 
\begin{itemize}
	\item for all $i \le f$ and all $h \le L$, $\mathcal{J}$ contains precisely $q$ $(i_1,i_2,i_3,i_4)$-BES of style $h$
(with respect to $(\cP,\cP')$) for which $i_1,i_2,i_3,i_4 \in \{ (i-1)K/f+2, \dots, iK/f \}$.
\end{itemize}
Then there exist $q$ edge-disjoint balanced exceptional factors with parameters~$(L,f)$ for $G$ (with respect to $C$, $\mathcal P'$) covering all
edges in~$\bigcup\mathcal{J}$.
\end{lemma}

Recall that the canonical interval partition $\mathcal{I}$ of $C$ into $f$ intervals consists of the intervals
$$A_{(i-1)K/f+1} B_{(i-1)K/f+1} A_{(i-1)K/f+2} \dots A_{iK/f+1}$$ for all $i\le f$. So the condition on~$\mathcal{J}$ ensures
that for each interval $I\in \mathcal{I}$ and each $h\le L$, the set $\mathcal{J}$ contains precisely $q$
balanced exceptional systems of style $h$ whose edges are only incident to vertices in $V_0$ and vertices
belonging to clusters in the interior of $I$. We will use Lemma~\ref{lma:bipartite:CEPS} to extend each such balanced exceptional system into
a balanced exceptional path system of style $h$ spanning~$I$.

\removelastskip\penalty55\medskip\noindent{\bf Proof of Lemma~\ref{lma:EF-bipartite}. }
Choose a new constant $\eps'$ with $\eps,Lq/m\ll \eps'\ll 1$.
Let $\mathcal{J}_{1}, \dots, \mathcal{J}_{q }$ be a partition of $\mathcal{J}$ such that for all $j\le q$, $h \le L$ and $i \le f$, 
the set $\mathcal{J}_j$ contains precisely one $(i_1,i_2,i_3,i_4)$-BES of style $h$ with $i_1,i_2,i_3,i_4 \in \{ (i-1)K/f+2, \dots, iK/f \}$.
Thus each $\mathcal{J}_j$ consists of $Lf$ balanced exceptional systems.
For each $j\le q$ in turn, we will choose a balanced exceptional factor $EF_j$ with parameters $(L,f)$ for $G$
such that $BF_j$ and $BF_{j'}$ are edge-disjoint for all $j' < j$ and $BF_j$ contains all edges
of the balanced exceptional systems in $\mathcal{J}_j$.
Assume that we have already constructed $BF_1, \dots, BF_{j-1}$. In order to construct $BF_j$, we will choose the $Lf$ balanced exceptional path systems
forming $BF_j$ one by one, such that each of these balanced exceptional path systems is edge-disjoint from $BF_1,\dots,BF_{j-1}$ and
contains precisely one of the balanced exceptional systems in $\mathcal{J}_j$. Suppose that we have already chosen some of these
balanced exceptional path systems and that next we wish to choose a balanced exceptional path system of style $h$ which spans the interval $I\in \mathcal{I}$
of $C$ and contains $J\in \mathcal{J}_j$.
Let $G'$ be the oriented graph obtained from $G$ by deleting all the edges in the balanced path systems already chosen for $BF_j$
as well as deleting all the edges in $BF_1,\dots,BF_{j-1}$.
Recall from (BSch1$'$) that $V(G)=A\cup B$. Thus $\Delta(G - G') \le 2j < 3q $ by~\eqref{EFdegq}.%
\COMMENT{Deryk added ref to BSch1'}
Together with Proposition~\ref{superslice} this implies that $(G',\mathcal{P},\mathcal{P}')$ is still a $[K, L, m, \epszero, \eps']$-bi-scheme.
(Here we use that $\Delta(G - G') < 3 q=3Lq/m\cdot m/L$ and
$\eps,Lq/m\ll \eps'\ll 1$.) So we can apply Lemma~\ref{lma:bipartite:CEPS} with $\eps'$ playing the role of $\eps$ to obtain a
balanced exceptional path system of style $h$ for $G'$ (and thus for $G$) which spans $I$ and contains all edges of $J$.
This completes the proof of the lemma.
\endproof


\section{The Robust Decomposition Lemma}\label{sec:robust}
The purpose of this section is to derive the version of the robust decomposition lemma (Corollary~\ref{rdeccorz}) that we will use
in this chapter to prove Theorem~\ref{1factbip}. (Recall from Section~\ref{sec:sketch} that we will not use it in the proof of Theorem~\ref{NWmindegbip}.)
Similarly as in the two cliques case, 
Corollary~\ref{rdeccorz} allows us to transform an approximate Hamilton decomposition into an exact one.
In the next subsection, we introduce the necessary concepts. In particular, Corollary~\ref{rdeccorz} relies on the existence of a
so-called bi-universal walk (which is a `bipartite version' of the universal walk introduced in Section~\ref{newlabel}).
The (proof of the) robust decomposition lemma then uses
edges guaranteed by this bi-universal walk to `balance out' edges of the graph $H$ when constructing the Hamilton decomposition of
$G^{\rm rob}+H$.
\subsection{Chord Sequences and Bi-universal Walks}
Let $R$ be a digraph whose vertices are $V_1,\dots,V_k$ and suppose that $C=V_1\dots V_k$ is a Hamilton cycle of $R$.
(Later on the vertices of $R$ will be clusters. So we denote them by capital letters.)

Recall from Section~\ref{newlabel} that a  \emph{chord sequence $CS(V_i,V_j)$} from $V_i$ to $V_j$ in $R$ is an ordered sequence of edges of the form
\[ CS(V_i,V_j) = (V_{i_1-1} V_{i_2}, V_{i_2-1} V_{i_3},\dots, V_{i_t-1} V_{i_{t+1}}),\]
where $V_{i_1}=V_i$, $V_{i_{t+1}} = V_j$ and the edge $V_{i_s-1} V_{i_{s+1}}$ belongs to $R$ for each $s\le t$.

As before, if $i=j$ then we consider the empty set to be a chord sequence from $V_i$ to $V_j$
and  we may assume that $CS(V_i,V_j)$ does not contain any edges of $C$.


A closed walk $U$ in $R$ is a \emph{bi-universal walk for $C$
with parameter $\ell'$} if the following conditions hold:
\begin{itemize}
\item[(BU1)] The edge set of $U$ has a partition into $U_{\rm odd}$ and $U_{\rm even}$.
For every $1\le i\le k$ there is a chord sequence $ECS^{\rm bi}(V_i,V_{i+2})$
from $V_i$ to $V_{i+2}$ such that $U_{\rm even}$ contains all edges of all these chord sequences for even $i$ (counted with multiplicities)
and $U_{\rm odd}$ contains all edges of these chord sequences for odd $i$.
All remaining edges of $U$ lie on $C$.
\item[(BU2)] Each $ECS^{\rm bi}(V_i,V_{i+2})$ consists of at most $\sqrt{\ell'}/2$ edges.
\item[(BU3)] $U_{\rm even}$ enters every cluster
$V_i$ exactly $\ell'/2$ times and it leaves every cluster $V_i$ exactly $\ell'/2$ times.
The same assertion holds for $U_{\rm odd}$.%
\COMMENT{So need to make sure $\ell'$ is even when we apply this.}
\end{itemize} 
Note that condition~(BU1) means that if an edge $V_iV_j\in E(R)\setminus E(C)$ occurs in total 5 times (say) in
$ECS^{{\rm bi}}(V_1,V_3),\dots,ECS^{{\rm bi}}(V_{k},V_2)$ then it occurs precisely 5 times in $U$. We will identify each occurrence of $V_iV_j$ in
$ECS^{{\rm bi}}(V_1,V_3),\dots,ECS^{{\rm bi}}(V_{k},V_2)$ with a (different) occurrence of $V_iV_j$ in $U$. 
Note that the edges of $ECS^{{\rm bi}}(V_i,V_{i+2})$ are allowed to appear in a different order 
within $U$.

\begin{lemma}\label{lem:univwalkzz}
Let $R$ be a digraph with vertices $V_1,\dots,V_k$ where $k \geq 4$ is even.
 Suppose that $C=V_1\dots V_k$ is a Hamilton cycle of $R$
and that $V_{i-1}V_{i+2}\in E(R)$ for every $1\le i\le k$. Let $\ell'\ge 4$ be an even integer.%
\COMMENT{need 4 rather than 2 to ensure $1 \le \sqrt{\ell'}/2$.}
 Let $U_{{\rm bi},\ell'}$ denote the multiset obtained from
$\ell'-1$ copies of $E(C)$ by adding
$V_{i-1}V_{i+2}\in E(R)$ for every $1\le i\le k$. Then the edges in $U_{{\rm bi},\ell'}$ can be ordered so that the
resulting sequence forms a bi-universal walk for $C$ with parameter~$\ell'$.
\end{lemma}
In the remainder of the chapter, we will also write $U_{{\rm bi},\ell'}$ for the bi-universal walk guaranteed by Lemma~\ref{lem:univwalkzz}. 

\proof
Let us first show that the edges in $U_{{\rm bi},\ell'}$ can be ordered so that the resulting sequence forms a closed walk in~$R$.
To see this, consider the multidigraph $U$ obtained from $U_{{\rm bi},\ell'}$ by deleting one copy of~$E(C)$.
Then $U$ is $(\ell'-1)$-regular and thus has a decomposition into 1-factors. We order the edges of $U_{{\rm bi},\ell'}$ as follows: 
We first traverse all cycles of the 1-factor decomposition of $U$ which contain the cluster $V_1$.
Next, we traverse the edge $V_1V_2$ of $C$. Next we traverse all those cycles of the 1-factor decomposition which contain
$V_2$ and which have not been traversed so far. Next we traverse the edge $V_2V_3$ of $C$ and so on
until we reach $V_1$ again. 

Recall that, for each $1\le i\le k$, the edge $V_{i-1}V_{i+2}$ is a chord sequence from $V_i$ to $V_{i+2}$. Thus we can take
$ECS^{{\rm bi}}(V_i,V_{i+2}):=V_{i-1}V_{i+2}$. Then $U_{{\rm bi},\ell'}$ satisfies (BU1)--(BU3).
Indeed, (BU2) is clearly satisfied. Partition one of the copies of $E(C)$ in $U_{{\rm bi},\ell'}$
into $E_{\rm even}$ and $E_{\rm odd}$ where $E_{\rm even}=\{V_iV_{i+1} | \ i \text{ even}\}$
and $E_{\rm odd}=\{V_iV_{i+1} | \ i \text{ odd}\}$. Note that the union of $E_{\rm even}$ together
with all $ECS^{{\rm bi}}(V_i,V_{i+2})$ for even $i$ is a $1$-factor in $R$. Add $\ell'/2-1$ of the
remaining copies of $E(C)$ to this $1$-factor to obtain $U_{\rm even}$. Define $U_{\rm odd}$
to be $E(U_{{\rm bi},\ell'})\setminus U_{\rm even}$. By construction of $U_{\rm even}$ and $U_{\rm odd}$,
(BU1) and (BU3) are satisfied.
\endproof

\subsection{Bi-setups and the Robust Decomposition Lemma} 
The%
\COMMENT{osthus adapted para from 2clique paper}
aim of this subsection is to state the `bipartite version' of the robust decomposition lemma (Lemma~\ref{rdeclemma'}, proved in~\cite{Kelly})
and derive Corollary~\ref{rdeccorz}, which we shall use later on in our proof of Theorem~\ref{1factbip}.
Lemma~\ref{rdeclemma'} guarantees the existence of a `robustly decomposable' digraph $G^{\rm rob}_{\rm dir}$
within a `bi-setup'. Roughly speaking, a bi-setup is a digraph $G$ together with its `reduced digraph' $R$,
which contains a Hamilton cycle $C$ and a bi-universal walk $U$. (So a bi-setup is a `bipartite analogue' of a setup that was introduced in
Section~\ref{newlabel2}.)
In our application, $G[A,B]$ will play the role of $G$ and $R$ will be the complete bipartite digraph.

To define a bi-setup formally, we first need to recall the following definitions.
Given a digraph $G$ and a partition $\cP$ of $V(G)$ into $k$ clusters $V_1,\dots,V_k$ of equal size,
recall that a partition $\cP'$ of $V(G)$%
   \COMMENT{Daniela: had $V$ instead of $V(G)$}
is an \emph{$\ell'$-refinement of $\cP$} if $\cP'$ is obtained by splitting each $V_i$
into $\ell'$ subclusters of equal size. (So $\cP'$ consists of $\ell'k$ clusters.) Recall also that
$\cP'$ is an \emph{$\eps$-uniform $\ell'$-refinement}
of $\cP$ if it is an $\ell'$-refinement of $\cP$ which satisfies the following condition:
Whenever $x$ is a vertex of $G$, $V$ is a cluster in $\cP$ and $|N^+_G(x)\cap V|\ge \eps |V|$
then $|N^+_G(x)\cap V'|=(1\pm \eps)|N^+_G(x)\cap V|/\ell'$%
\COMMENT{AL: added prime, Daniela added two more primes}
 for each cluster $V'\in \cP'$ with $V'\subseteq V$.
The inneighbourhoods of the vertices of $G$ satisfy an analogous condition.

We will need the following definition from~\cite{Kelly}, which describes the structure within which the robust decomposition lemma
finds the robustly decomposable graph.%
\COMMENT{Deryk added sentence}
$(G,\cP,\cP',R,C,U,U')$ is called an \emph{$(\ell',k,m,\eps,d)$-bi-setup} if the following properties are satisfied:
\begin{itemize}
\item [(BST1)] $G$ and $R$ are digraphs. $\mathcal{P}$ is a partition of $V(G)$ into
$k$ clusters of size $m$ where $k$ is even. The vertex set of $R$ consists of these clusters.
\item[(BST2)] For every edge $VW$ of $R$, the corresponding pair $G[V,W]$ is $(\eps,\ge d)$-regular.
\item[(BST3)] $C=V_1\dots V_{k}$ is a Hamilton cycle of $R$ and for every edge $V_iV_{i+1}$ of $C$ the corresponding pair $G[V_i,V_{i+1}]$ is $[\eps,\ge d]$-superregular.
\item[(BST4)] $U$ is a bi-universal walk for $C$ in $R$ with parameter~$\ell'$ and $\cP'$ is an $\eps$-uniform $\ell'$-refinement%
   \COMMENT{It was decided by DK+DO we need this notion of refinement here}
of $\cP$.
\item[(BST5)] Let $V_j^1,\dots,V_j^{\ell'}$ denote the clusters in $\cP'$ which are contained
in $V_j$ (for each $1\le j\le k$). Then $U'$ is a closed walk on the clusters in $\cP'$ which is obtained from $U$ as follows:
When $U$ visits $V_j$ for the $a$th time, we let $U'$ visit the subcluster $V_j^a$ (for all $1\le a\le \ell'$).
\item[(BST6)] For every edge $V_{i}^jV_{i'}^{j'}$ of $U'$ the corresponding pair $G[V_{i}^j,V_{i'}^{j'}]$ is $[\eps,\ge d]$-superregular.%
	\COMMENT{AL: added more details}
\end{itemize}
In~\cite{Kelly}, in a bi-setup, the digraph $G$ could also contain an exceptional set, but since we are only using
the definition in the case when there is no such exceptional set, we have only stated it in this special case.

Suppose that $(G,\mathcal{P},\mathcal{P}')$ is a $[K,L,m,\eps_0,\eps]$-bi-scheme and that $C=A_1B_1\dots$ $ A_KB_K$ is a spanning cycle
on the clusters of $\mathcal{P}$. Let $\mathcal{P}_{{\rm bi}}:=\{A_1,\dots,A_K,B_1,\dots,$ $B_K\}$. Suppose that $\ell',m/\ell'\in\mathbb{N}$ with
$\ell'\ge 4$. Let $\mathcal{P}''_{{\rm bi}}$ be an $\eps$-uniform $\ell'$-refinement of $\cP_{{\rm bi}}$ (which exists by Lemma~\ref{randompartition}).%
    \COMMENT{Daniela added brackets}
Let $C_{{\rm bi}}$ be the directed cycle obtained from $C$
in which the edge $A_1B_1$ is oriented towards $B_1$ and so on. Let $R_{{\rm bi}}$ be the complete bipartite digraph
whose vertex classes are $\{A_1,\dots,A_K\}$ and $\{B_1,\dots,B_K\}$.%
   \COMMENT{Daniela: had "whose vertices are the clusters in $\cP$"}
Let $U_{{\rm bi},\ell'}$ be a bi-universal walk for $C$
with parameter $\ell'$ as defined in Lemma~\ref{lem:univwalkzz}. Let $U'_{{\rm bi},\ell'}$ be the closed walk
obtained from $U_{{\rm bi},\ell'}$ as described in~(BST5). We will call
$$
(G,\cP_{{\rm bi}},\cP''_{{\rm bi}},R_{{\rm bi}},C_{{\rm bi}},U_{{\rm bi},\ell'},U'_{{\rm bi},\ell'})$$ the \emph{bi-setup
associated to~$(G,\mathcal{P},\mathcal{P}')$}. The following lemma shows that it is indeed a bi-setup.%
   \COMMENT{Daniela changed lemma below to bring it in line with Lemma~\ref{randompartition}}

\begin{lemma}\label{lem:bisetup}
Suppose that $K,L,m/L, \ell', m/\ell'\in\mathbb{N}$ with $\ell'\ge 4$, $K \geq 2$ and $0<1/m \ll 1/K,\eps \ll \eps',1/\ell'$.
Suppose that $(G,\mathcal{P},\mathcal{P}')$ is a $[K,L,m,\eps_0,\eps]$-bi-scheme and that $C=A_1B_1\dots A_KB_K$ is a spanning cycle
on the clusters of $\mathcal{P}$. Then $$(G,\cP_{{\rm bi}},\cP''_{{\rm bi}},R_{{\rm bi}},C_{{\rm bi}},U_{{\rm bi},\ell'},U'_{{\rm bi},\ell'})$$
is an $(\ell',2K,m,\eps',1/2)$-bi-setup.
\end{lemma}
\proof
Clearly, $(G,\cP_{{\rm bi}},\cP''_{{\rm bi}},R_{{\rm bi}},C_{{\rm bi}},U_{{\rm bi},\ell'},U'_{{\rm bi},\ell'})$ satisfies~(BST1).
(BSch$3'$) implies that (BST2) and~(BST3) hold.%
    \COMMENT{This is now immediate because of change of def of (BSch3$'$)}
Lemma~\ref{lem:univwalkzz} implies~(BST4). (BST5) follows from the definition of $U'_{{\rm bi},\ell'}$. Finally, (BST6) follows from (BSch3$'$) and
Lemma~\ref{randompartition} since $\cP''_{{\rm bi}}$ is an $\eps$-uniform $\ell'$-refinement of~$\cP_{{\rm bi}}$.
\endproof

We now state the `bipartite version' of the robust decomposition lemma which was proved in~\cite{Kelly}. It is an analogue of the robust decomposition 
lemma (Lem\-ma~\ref{rdeclemma}) used in Chapter~\ref{paper1} and works for bi-setups rather than setups. As before, the lemma
 guarantees the existence of a `robustly decomposable' digraph $G^{\rm rob}_{\rm dir}$, whose crucial property
is that $H + G^{\rm rob}_{\rm dir}$ has a Hamilton decomposition for any sparse bipartite regular digraph~$H$
which is edge-disjoint from $G^{\rm rob}_{\rm dir}$.%
    \COMMENT{Daniela added "which is..."}

Again, $G^{\rm rob}_{\rm dir}$ consists of digraphs $CA_{{\rm dir}}(r)$ (the `chord absorber') and $PCA_{{\rm dir}}(r)$ 
(the `parity extended cycle switcher')
together with some special factors. $G^{\rm rob}_{\rm dir}$ is constructed in two steps:
given a suitable set $\mathcal{SF}$ of special factors, the lemma first `constructs' $CA_{{\rm dir}}(r)$ and then,
given another suitable set $\mathcal{SF}'$ of special factors, the lemma `constructs' $PCA_{{\rm dir}}(r)$.
\begin{lemma} \label{rdeclemma'}
Suppose that $0<1/m\ll 1/k\ll \eps \ll 1/q \ll 1/f \ll r_1/m\ll d\ll 1/\ell',1/g\ll 1$ 
where $\ell '$ is even and%
   \COMMENT{In the Kelly paper have $1/n$ instead of $1/m$ in the hierarchy. But since $V_0=\emptyset$ in our setting,
it doesn't make that much sense to introduce $n$ here. So I replaced $1/n$ by $1/m$.}
that $rk^2\le m$. Let
$$r_2:=96\ell'g^2kr, \ \ \ r_3:=rfk/q, \ \ \ r^\diamond:=r_1+r_2+r-(q-1)r_3, \ \ \ s':=rfk+7r^\diamond
$$
and suppose that $k/14, k/f, k/g, q/f, m/4\ell', fm/q, 2fk/3g(g-1) \in \mathbb{N}$.
Suppose that $(G,\cP,\cP',R,C,U,U')$ is an $(\ell',k,m,\eps,d)$-bi-setup and $C=V_1\dots V_k$.
Suppose that $\cP^*$ is a $(q/f)$-refinement
of $\cP$ and that $SF_1,\dots, SF_{r_3}$ are edge-disjoint special factors with parameters $(q/f,f)$ 
with respect to $C$, $\cP^*$ in $G$. Let $\mathcal{SF}:=SF_1+\dots +SF_{r_3}$.
Then there exists a digraph $CA_{{\rm dir}}(r)$ for which the following holds:
\begin{itemize}
\item[\rm (i)] $CA_{{\rm dir}}(r)$ is an $(r_1+r_2)$-regular spanning subdigraph of $G$ which is edge-disjoint from $\mathcal{SF}$.
\item[\rm (ii)] Suppose that $SF'_1,\dots, SF'_{r^\diamond}$ are special factors with parameters $(1,7)$
with respect to $C$, $\cP$ in $G$ which are edge-disjoint from each other and from $CA_{{\rm dir}}(r)+ \mathcal{SF}$.%
   \COMMENT{In the Kelly paper we write $CA_{{\rm dir}}(r)\cup \mathcal{SF}$ instead of $CA_{{\rm dir}}(r)+ \mathcal{SF}$
(and similarly below). But with our def of $+$ and $\cup $ in this paper, we have to use $+$ since we allow for a fictive edge
$xy$ in $\mathcal{SF}$ to also occur in $CA_{{\rm dir}}(r)$, and in this case $CA_{{\rm dir}}(r)+ \mathcal{SF}$
willl contain two copies of that edge.} 
Let $\mathcal{SF}':=SF'_1+\dots +SF'_{r^\diamond}$.
Then there exists a digraph $PCA_{{\rm dir}}(r)$ for which the following holds:
\begin{itemize}
\item[\rm (a)] $PCA_{{\rm dir}}(r)$ is a $5r^\diamond$-regular spanning subdigraph of $G$ which
is edge-disjoint from $CA_{{\rm dir}}(r)+ \mathcal{SF}+ \mathcal{SF}'$.
\item[\rm (b)] Let $\mathcal{SPS}$ be the set consisting of all the $s'$ special path systems
contained in $\mathcal{SF}+ \mathcal{SF}'$. Let $V_{\rm even}$ denote the union of all $V_i$ over all even $1\le i\le k$
and define $V_{\rm odd}$ similarly.
Suppose that $H$ is an $r$-regular bipartite digraph on $V(G)$ with vertex classes $V_{\rm even}$ and $V_{\rm odd}$
which is edge-disjoint from $G^{\rm rob}_{\rm dir}:=CA_{{\rm dir}}(r)+ PCA_{{\rm dir}}(r)+ \mathcal{SF}+ \mathcal{SF}'$.
Then $H+G^{\rm rob}_{\rm dir}$ has a decomposition into $s'$
edge-disjoint Hamilton cycles $C_1,\dots,C_{s'}$.
Moreover, $C_i$ contains one of the special path systems from $\mathcal{SPS}$, for each $i\le s'$.
\end{itemize}
\end{itemize}
\end{lemma}

Recall from Section~\ref{sec:SFq} that we always view fictive edges in special factors as being distinct from each other and
from the edges in other graphs. So for example, saying that $CA_{{\rm dir}}(r)$ and $\mathcal{SF}$ are edge-disjoint in
Lemma~\ref{rdeclemma'} still allows for a fictive edge $xy$ in $\mathcal{SF}$ to occur in $CA_{{\rm dir}}(r)$ as well
(but $CA_{{\rm dir}}(r)$ will avoid all non-fictive edges in $\mathcal{SF}$).
 
We will use the following `undirected' consequence of Lemma~\ref{rdeclemma'}.

\begin{cor} \label{rdeccorz}
Suppose that $0<1/m\ll \eps_0,1/K\ll \eps \ll 1/L \ll 1/f \ll r_1/m\ll 1/\ell',1/g\ll 1$ 
where $\ell '$ is even and
that $4rK^2\le m$.%
\COMMENT{osthus: changed from $4rK^2$} Let
$$r_2:=192\ell'g^2Kr, \ \ \ r_3:=2rK/L, \ \ \ r^\diamond:=r_1+r_2+r-(Lf-1)r_3, \ \ \ s':=2rfK+7r^\diamond
$$
and suppose that $L, K/7, K/f, K/g, m/4\ell', m/L, 4fK/3g(g-1) \in \mathbb{N}$.%
   \COMMENT{Daniela added $L\in \mathbb{N}$}
Suppose that $(G_{\rm dir},\cP,\cP')$ is a $[K,L,m,\eps_0,\eps]$-bi-scheme and let $G'$ denote the underlying undirected graph
of $G_{\rm dir}$. Let $C=A_1B_1\dots A_KB_K$ be a spanning cycle on the clusters in $\cP$. 
Suppose that $BF_1,\dots, BF_{r_3}$ are edge-disjoint balanced exceptional factors with parameters $(L,f)$ for $G_{\rm dir}$ 
(with respect to $C$, $\mathcal P'$). Let $\mathcal{BF}:=BF_1+\dots + BF_{r_3}$.
Then there exists a graph $CA(r)$ for which the following holds:
\begin{itemize}
\item[{\rm (i)}] $CA(r)$ is a $2(r_1+r_2)$-regular spanning subgraph of $G'$ which is edge-disjoint from $\mathcal{BF}$.
\item[{\rm (ii)}] Suppose that $BF'_1,\dots, BF'_{r^\diamond}$ are balanced exceptional factors with parameters $(1,7)$ for $G_{\rm dir}$
(with respect to $C$, $\mathcal P$) which are edge-disjoint from each other and from $CA(r)+ \mathcal{BF}$.
Let $\mathcal{BF}':=BF'_1+\dots + BF'_{r^\diamond}$.
Then there exists a graph $PCA(r)$ for which the following holds:
\begin{itemize}
\item[{\rm (a)}] $PCA(r)$ is a $10r^\diamond$-regular spanning subgraph of $G'$ which
is edge-disjoint from $CA(r)+ \mathcal{BF}+ \mathcal{BF}'$.
\item[{\rm (b)}] Let $\mathcal{BEPS}$ be the set consisting of all the $s'$ balanced exceptional path systems
contained in $\mathcal{BF}+ \mathcal{BF}'$.
Suppose that $H$ is a $2r$-regular bipartite graph on $V(G_{\rm dir})$ with vertex classes $\bigcup_{i=1}^K A_i$ and $\bigcup_{i=1}^K B_i$
which is edge-disjoint from $G^{\rm rob}:=CA(r)+ PCA(r)+ \mathcal{BF}+ \mathcal{BF}'$.
Then $H+ G^{\rm rob}$ has a decomposition into $s'$
edge-disjoint Hamilton cycles $C_1,\dots,C_{s'}$.
Moreover, $C_i$ contains one of the balanced exceptional path systems from $\mathcal{BEPS}$, for each $i\le s'$.
\end{itemize}
\end{itemize}
\end{cor}
We remark that we write%
\COMMENT{osthus adapted sentence from paper4} 
$A_1,\dots,A_K,B_1,\dots,B_K$ for the clusters in $\cP$. Note that the vertex set of each of $\mathcal{EF}$, $\mathcal{EF}'$, $G^{\rm rob}$
includes $V_0$ while that of $G_{\rm dir}$, $CA(r)$, $PCA(r)$, $H$ does not.
Here $V_0=A_0\cup B_0$, where $A_0$ and $B_0$ are the exceptional sets of $\cP$.

\proof
Choose new constants $\eps'$ and $d$ such that $\eps\ll \eps'\ll 1/L$ and $r_1/m $ $\ll d\ll 1/\ell',1/g$.%
   \COMMENT{Daniela: introduced $\eps'$ which is needed since Lemma~\ref{lem:bisetup} changed. Before we had that
$(G_{\rm dir},\cP_{{\rm bi}},\cP''_{{\rm bi}},R_{{\rm bi}},C_{{\rm bi}},U_{{\rm bi},\ell'},U'_{{\rm bi},\ell'})$
is an $(\ell',2K,m,\eps^{1/2},1/2)$-bi-setup instead of just an $(\ell',2K,m,\eps',1/2)$-bi-setup. So check whether I made all the necessary changes below.}
Consider the bi-setup $(G_{\rm dir},\cP_{\rm bi},\cP''_{{\rm bi}},R_{{\rm bi}},C_{{\rm bi}},U_{{\rm bi},\ell'},U'_{{\rm bi},\ell'})$
associated to $(G_{\rm dir},\cP,\cP')$. By Lemma~\ref{lem:bisetup}, 
$(G_{\rm dir},\cP_{{\rm bi}},\cP''_{{\rm bi}},R_{{\rm bi}},C_{{\rm bi}},U_{{\rm bi},\ell'},U'_{{\rm bi},\ell'})$
is an $(\ell',2K,m,\eps',1/2)$-bi-setup and thus also an $(\ell',2K,m,\eps',d)$-bi-setup.
Let $BF^*_{i,{\rm dir}}$ be as defined in Section~\ref{sec:BEPSq}.
Recall from there that, for each $i\le r_3$, $BF^*_{i,{\rm dir}}$ is a special factor
with parameters $(L,f)$ with respect to $C$, $\mathcal P'$ in $G_{\rm dir}$ such that
${\rm Fict}(BF^*_{i,{\rm dir}})$ consists of all the edges in the $J^*$ for all the $Lf$ balanced exceptional systems $J$ contained in $BF_i$.
Thus we can apply Lemma~\ref{rdeclemma'} to $(G_{\rm dir},\cP_{{\rm bi}},\cP''_{{\rm bi}},R_{{\rm bi}},C_{{\rm bi}},U_{{\rm bi},\ell'},U'_{{\rm bi},\ell'})$
with $2K$, $Lf$, $\eps'$ playing the roles of $k$, $q$, $\eps$
in order to obtain a spanning subdigraph $CA_{{\rm dir}}(r)$
of $G_{\rm dir}$ which satisfies Lemma~\ref{rdeclemma'}(i). Hence the underlying undirected graph $CA(r)$ of
$CA_{{\rm dir}}(r)$ satisfies Corollary~\ref{rdeccorz}(i). Indeed, to check that $CA(r)$ and $\mathcal{BF}$ are edge-disjoint, by Lemma~\ref{rdeclemma'}(i)%
\COMMENT{osthus added by Lemma~\ref{rdeclemma'}(i)} 
it suffices to
check that $CA(r)$ avoids all edges in all the balanced exceptional systems $J$ contained in $BF_i$ (for all $i\le r_3$). But this
follows since $E(G_{\rm dir})\supseteq E(CA(r))$ consists only of $AB$-edges by (BSch2$'$) and since no balanced exceptional system
contains an $AB$-edge by~(BES2).%
   \COMMENT{Daniela added the last 2 sentences}

Now let $BF'_1,\dots, BF'_{r^\diamond}$ be balanced exceptional factors as described in Corollary~\ref{rdeccorz}(ii).
Similarly as before, for each $i\le r^\diamond$, $(BF'_i)^*_{\rm dir}$ is a special factor
with parameters $(1,7)$ with respect to $C$, $\mathcal P$ in $G_{\rm dir}$ such that
${\rm Fict}((BF'_i)^*_{\rm dir})$ consists of all the edges in the $J^*$ over all the $7$ balanced exceptional systems $J$ contained in $BF'_i$.
Thus we can apply Lemma~\ref{rdeclemma'} to obtain a spanning subdigraph $PCA_{{\rm dir}}(r)$
of $G_{\rm dir}$ which satisfies Lemma~\ref{rdeclemma'}(ii)(a) and~(ii)(b). Hence the underlying undirected graph
$PCA(r)$ of $PCA_{{\rm dir}}(r)$ satisfies Corollary~\ref{rdeccorz}(ii)(a).

It remains to check that Corollary~\ref{rdeccorz}(ii)(b) holds too. Thus let $H$ be as described in Corollary~\ref{rdeccorz}(ii)(b).
Let $H_{{\rm dir}}$ be an $r$-regular orientation of $H$. (To see that such an orientation exists, apply Petersen's theorem to obtain
a decomposition of $H$ into $2$-factors and then orient each $2$-factor to obtain a (directed) $1$-factor.)%
     \COMMENT{However, even if $H\subseteq G'$, this orientation might not agree with $G_{\rm dir}$, i.e.~we might not have that
$H\subseteq G_{\rm dir}$. This is the reason why we cannot assume that $H\subseteq G$ in Lemma~\ref{rdeclemma'}(ii)(b).
($H\subseteq G$ would fit better to the remainder of the Kelly paper.)}
Let $\mathcal{BF}^*_{\rm dir}$ be the union of the $BF^*_{i,{\rm dir}}$ over all $i\le r_3$
and let $(\mathcal{BF}')^*_{\rm dir}$ be the union of the $(BF'_i)^*_{\rm dir}$ over all $i\le r^\diamond$.
Then Lemma~\ref{rdeclemma'}(ii)(b) implies that
$H_{{\rm dir}}+ CA_{{\rm dir}}(r)+ PCA_{{\rm dir}}(r)+ \mathcal{BF}^*_{\rm dir}+ (\mathcal{BF}')^*_{\rm dir}$
has a decomposition into $s'$ edge-disjoint (directed) Hamilton cycles $C'_1,\dots,C'_{s'}$ such that each $C'_i$ contains
$BEPS^*_{i,{\rm dir}}$ for some balanced exceptional path system $BEPS_i$ from $\mathcal{BEPS}$.
Let $C_i$ be the undirected graph obtained from $C'_i-BEPS^*_{i,{\rm dir}}+BEPS_i$ by ignoring the directions of all
the edges. Then Proposition~\ref{prop:CEPSbiparite} (applied with $G'$ playing the role of $G$) implies that $C_1,\dots,C_{s'}$ is
a decomposition of $H+ G^{\rm rob}=H+ CA(r)+ PCA(r)+ \mathcal{BF}+ \mathcal{BF}'$ into edge-disjoint Hamilton cycles.
\endproof

\section{Proof of Theorem~$\text{\ref{NWmindegbip}}$}\label{sec:proof1}

The proof of Theorem~\ref{NWmindegbip} is similar to that of Theorem~\ref{1factbip} except that we  do not need to apply
the robust decomposition lemma in the proof of Theorem~\ref{NWmindegbip}. For both results, we will need an approximate decomposition result
(Lemma~\ref{almostthmbip}), which is stated below and proved in Chapter~\ref{paper3}.%
    \COMMENT{osthus changed the last sentence to refer to 3.2}
Lemma~\ref{almostthmbip} is a bipartite analogue of Lemma~\ref{almostthm}. It extends a suitable set of balanced exceptional systems into a set of edge-disjoint Hamilton cycles covering most edges of an almost complete 
and almost balanced bipartite graph.%
\COMMENT{osthus added extra sentence}
\begin{lemma}\label{almostthmbip}
Suppose that $0<1/n \ll \eps_0  \ll 1/K \ll \rho  \ll 1$ and $0 \leq \mu \ll 1$,
where $n,K \in \mathbb N$ and $K$ is even.
Suppose that $G$ is a graph on $n$ vertices and $\mathcal{P}$ is a $(K, m, \eps _0)$-partition of $V(G).$
Furthermore, suppose that the following conditions hold:%
	\COMMENT{Previously we had `$(G[A,B],\mathcal{P})$ is a $(K, m, \eps _0, \eps )$-bi-scheme'. now removed.}
\begin{itemize}
	\item[{\rm (a)}] $d(w,B_i) = (1 - 4 \mu \pm 4 /K) m $ and $d(v,A_i) = (1 - 4 \mu \pm 4 /K) m $ for all
	$w \in A$, $v \in B$ and $1\leq i \leq K$.%
	\COMMENT{Daniela swapped $v$ and $w$, now it is the same as in the proofs}
	\item[{\rm (b)}] There is a set $\mathcal J$ which consists of at most $(1/4-\mu - \rho)n$ edge-disjoint balanced exceptional systems with parameter $\eps_0$ in~$G$.
	\item[{\rm (c)}] $\mathcal J$ has a partition into $K^4$ sets $\mathcal J_{i_1,i_2,i_3,i_4}$ (one for all $1\le \I \le K$) such that each $\mathcal J_{\I}$ consists of precisely $|\mathcal J|/{K^4}$ $\i$-BES with respect to~$\cP$.
   \item[{\rm (d)}] Each $v \in A \cup B$ is incident with an edge in $J$ for at most $2 \eps_0 n $ $J \in \mathcal{J}$.%
   \COMMENT{This is a new property. This condition is implied by the fact that $(G[A,B],\mathcal{P})$ is a $(K, m, \eps _0, \eps )$-bi-scheme.($|A_0 \cup B_0| \le \eps_0 n$
and $\Delta( G[A]),\Delta( G[B]) \le \eps_0 n$.)}
\end{itemize}
Then $G$ contains $|\mathcal J|$ edge-disjoint Hamilton cycles such that each of these Hamilton cycles contains some $J\in\mathcal{J}$.
\end{lemma}

To%
\COMMENT{osthus added para} 
prove Theorem~\ref{NWmindegbip}, we find a bi-framework via Corollary~\ref{coverA0B02c}.
Then we choose suitable balanced exceptional systems using Corollary~\ref{BEScor}.
Finally, we extend these into Hamilton cycles using Lemma~\ref{almostthmbip}.

\removelastskip\penalty55\medskip\noindent{\bf Proof of Theorem~\ref{NWmindegbip}. }
\noindent\textbf{Step 1: Choosing the constants and a bi-framework.}
By%
   \COMMENT{Daniela changed this proof substantially, eg bi-schemes are not used anymore. So read all of it again}
making $\alpha$ smaller if necessary, we may assume that $\alpha\ll 1$.
Define new constants such that 
\begin{align}
0 & < 1/n_0 \ll \eps_{\rm ex} \ll \epszero \ll \eps'_0\ll \eps' \ll \eps_1 \ll \eps_2 \ll \eps_3 \ll
\eps_4 \ll 1/K\ll \alpha \ll \eps\ll 1, \nonumber
\end{align}
where $K\in  \mathbb{N}$ and $K$ is even.%
\COMMENT{$\eps_3$ and $\eps_4$ are implicitly used when applying Corollary~\ref{BEScor}. }

Let $G$, $F$ and $D$ be as in Theorem~\ref{NWmindegbip}. 
Apply Corollary~\ref{coverA0B02c} with $\eps_{\rm ex}$, $\eps_0$ playing the role of $\eps$, $\eps^*$
to find a set $\cC_1$ of at most $\eps _{\rm ex} ^{1/3} n$ edge-disjoint Hamilton cycles in $F$ so that the graph $G_1$ obtained from $G$
by deleting all the edges in these Hamilton cycles forms part of an
$(\eps_0,\eps',K,D_1)$-bi-framework $(G_1,A,A_0,B,B_0)$ with $D_1 \ge D-2\eps_{\rm ex}^{1/3}n$.
Moreover, $F$ satisfies~(WF5) with respect to $\eps '$ and
\begin{equation}\label{eq:sizeC1}
|\cC_1|=(D-D_1)/2.
\end{equation}
In particular, this implies that $\delta (G_1) \geq D_1$ and that $D_1$ is even (since $D$ is even). 
Let $F_1$ be the graph obtained from $F$ by deleting all those edges lying on Hamilton cycles
in $\cC_1$. Then 
\begin{equation}\label{eq:degF1}
\delta(F_1) \ge \delta(F)-2|\cC_1|\ge (1/2-3\eps_{\rm ex}^{1/3})n.
\end{equation}
 Let
\begin{align*}
m :=\frac{|A|}{K}=\frac{|B|}{K} \qquad \text{and} \qquad t_{K}:=\frac{(1-20\eps_4)D_1}{2K^4}.
\end{align*}
By changing $\eps_4$ slightly, we may assume that $t_K\in\mathbb{N}$.

\smallskip

\noindent\textbf{Step 2: Choosing a $(K,m,\eps_0)$-partition $\cP$.}
Apply Lemma~\ref{part} to the bi-framework $(G_1,A,A_0,B,B_0)$ with $F_1$, $\eps_0$
playing the roles of $F$, $\eps$ in order to obtain partitions $A_1,\dots,A_{K}$ and
$B_1,\dots,B_{K}$ of $A$ and $B$ into sets of size $m$ such that together with $A_0$ and $B_0$ the sets
$A_i$ and $B_i$ form a $(K,m,\eps_0,\eps_1,\eps_2)$-partition $\cP$ for $G_1$.%

Note that by Lemma~\ref{part}(ii) and since $F$ satisfies (WF5),
for all $x \in A$ and $1 \leq j \leq K$, we have
\begin{eqnarray}\label{eq:degF'1}
d_{F_1}(x,B_j)& \ge & \frac{d_{F_1}(x,B) - \eps_1n}{K} \stackrel{{\rm (WF5)}}{\ge} \frac{d_{F_1}(x)-\eps' n -|B_0|- \eps_1 n}{K}\nonumber\\
& \stackrel{(\ref{eq:degF1})}{\ge} & \frac{(1/2-3\eps_{\rm ex}^{1/3})n- 2\eps_1 n}{K}
\ge (1-5\eps_1)m_.
\end{eqnarray}
Similarly, $d_{F_1}(y,A_i)\ge (1-5\eps_1)m$%
\COMMENT{Andy: deleted a prime}
for all $y\in B$ and $1 \leq i \leq K$.

\smallskip

\noindent\textbf{Step 3: Choosing balanced exceptional systems for the almost decomposition.}
Apply Corollary~\ref{BEScor} to the $(\eps_0,\eps',K,D_1)$-bi-framework $(G_1$,$A$,$A_0$,$B$,$B_0)$ with
$F_1$, $G_1$, $\eps_0$, $\eps'_0$, $D_1$ playing the roles
of $F$, $G$, $\eps$, $\eps_0$, $D$. Let $\cJ'$ be the union of the sets $\cJ_{i_1i_2i_3i_4}$ guaranteed by Corollary~\ref{BEScor}.
So $\cJ'$ consists of $K^4t_{K}$ edge-disjoint balanced exceptional systems with parameter $\eps'_0$ in $G_1$ (with respect to~$\mathcal{P}$).
Let $\cC_2$ denote the set of $10 \eps _4 D_1$ Hamilton cycles guaranteed by Corollary~\ref{BEScor}.
Let $F_2$ be the subgraph obtained from $F_1$ by deleting all the Hamilton cycles in~$\cC_2$.  
Note that 
\begin{equation}\label{eq:D2zz}
D_2:=D_1 -2|\mathcal C_2| =(1-20 \eps _4)D_1=2K^4 t_K =2|\cJ'|.
\end{equation}
\noindent\textbf{Step 4: Finding the remaining Hamilton cycles.}
Our next aim is to apply Lemma~\ref{almostthmbip} with $F_2$, $\cJ'$,  $\eps'$ playing the
roles of $G$, $\cJ$,  $\eps_0$.

Clearly, condition~(c) of Lemma~\ref{almostthmbip} is satisfied.
In order to see that condition (a) is satisfied, let $\mu:=1/K$ and note that for all $w\in A$ we have 
$$
d_{F_2}(w,B_i)\ge d_{F_1}(w,B_i)-2|\cC_2|\stackrel{(\ref{eq:degF'1})}{\ge} (1-5\eps_1)m-20\eps_4 D_1\ge (1-1/K)m.
$$
Similarly $d_{F_2}(v,A_i)\ge (1-1/K)m$ for all $v\in B$.

To check condition~(b), note that
$$|\cJ'|\stackrel{(\ref{eq:D2zz})}{=} \frac{D_2}{2}\le \frac{D}{2}\le (1/2-\alpha)\frac{n}{2}\le (1/4-\mu-\alpha/3)n .$$
Thus condition~(b) of Lemma~\ref{almostthmbip} holds with $\alpha/3$ playing the role of $\rho$.
Since the edges in $\cJ'$ lie in $G_1$ and  $(G_1,A,A_0,B,B_0)$ is an $(\eps_0,\eps',K,D_1)$-bi-framework,
(BFR5) implies that each $v \in A \cup B$ is incident with an edge in $J$ for at most $\eps ' n+|V_0|\leq 2\eps 'n$
$J \in \cJ '$. 
(Recall that in a balanced exceptional system there are no edges between $A$ and $B$.)
So condition~(d) of Lemma~\ref{almostthmbip} holds with $\eps '$ playing the role of $\eps _0$.

So we can indeed apply Lemma~\ref{almostthmbip} to obtain a collection $\cC_3$ of $|\cJ'|$ edge-disjoint Hamilton cycles in $F_2$
which cover all edges of $\bigcup \cJ'$. Then $\cC_1\cup \cC_2\cup \cC_3$ is a set of edge-disjoint Hamilton cycles in~$F$
of size
$$|\cC_1|+|\cC_2|+|\cC_3|\stackrel{(\ref{eq:sizeC1}),(\ref{eq:D2zz})}{=} \frac{D-D_1}{2}+ \frac{D_1-D_2}{2}+ \frac{D_2}{2}=\frac{D}{2},$$
as required.
\endproof


\section{Proof of Theorem~$\text{\ref{1factbip}}$}\label{sec:proof2}

As mentioned earlier, the proof of Theorem~\ref{1factbip} is similar to that of Theorem~\ref{NWmindegbip} except that we will also need to apply 
the robust decomposition lemma (Corollary~\ref{rdeccorz}). This means Steps 2--4 and Step 8 in the proof of Theorem~\ref{1factbip} did not appear
in the proof of Theorem~\ref{NWmindegbip}.
Steps 2--4 prepare the ground for the application of the robust decomposition lemma and in Step~8 we apply it to cover the leftover from the approximate decomposition step
with Hamilton cycles. Steps 5--7 contain the approximate decomposition step, using Lemma~\ref{almostthmbip}.
\COMMENT{osthus added extra sentence}

In our proof of Theorem~\ref{1factbip} it will be convenient to work with an undirected version of the bi-schemes introduced in Section~\ref{sec:findBFq}.
Given a graph $G$ and partitions $\mathcal P$ and $\cP'$ of a vertex set $V$, we call $(G, \mathcal{P},\mathcal{P}')$ a
\emph{$(K,L,m,\eps_0, \eps)$-bi-scheme} if the following properties hold:%
   \COMMENT{Daniela changed def to make $G$ into a bipartite graph with vertex classes $A$ and $B$ (ie replacing the old $G$ by $G-V_0$).
Also deleted the def of a $(K,m,\eps_0, \eps)$-bi-scheme} 
\begin{itemize}
\item[(BSch$1$)] $(\mathcal{P},\mathcal{P}')$ is a $(K,L,m,\eps_0)$-partition of $V$. Moreover, $V(G)=A\cup B$.
\item[(BSch$2$)] Every edge of $G$ joins some vertex in $A$ to some vertex in~$B$.
\item[(BSch$3$)] $d_G(v,A_{i,j}) \geq (1 - \eps) m/L $ and $d_G(w,B_{i,j}) \geq (1 - \eps) m/L $ for all
    $v \in B$, $w \in A$, $ i \leq K$ and $j\leq L$.
\end{itemize}
We will also use the following proposition.

\begin{prop}\label{lem:dirschemezz}
Suppose that $K, L, n, m/L \in \mathbb{N}$ and $0 <1/n  \ll \eps, \eps_0\ll 1$.
Let $(G,\mathcal{P}, \mathcal{P}')$ be a $(K, L, m, \epszero, \eps)$-bi-scheme with $|G| = n$.
Then there exists an orientation $G_{\rm dir}$ of $G$%
   \COMMENT{Daniela replaced $G-V_0$ by $G$ (also twice in the first para of the proof)}
such that $(G_{\rm dir}, \mathcal{P},\mathcal{P}')$ is a
$[K,L,m,\eps_0,2\sqrt{\eps}]$-bi-scheme.
\end{prop}
\proof 
Randomly orient every edge in $G$ to obtain an oriented graph $G_{\rm dir}$. (So given any edge $xy$ in $G$
with probability $1/2$, $xy \in E(G_{\rm dir})$ and with probability $1/2$, $yx \in E(G_{\rm dir})$.)
(BSch$1'$) and (BSch$2'$) follow immediately from (BSch$1$) and (BSch$2$).

Note that Fact~\ref{simplefact} and (BSch$3$) imply that $G[A_{i,j},B_{i',j'}]$ is $[1, \sqrt{\eps}]$-superregu\-lar
with density at least $1-\eps$,
for all $i,i'\leq K$ and $j,j' \leq L$.
Using this, (BSch$3'$) follows easily from the large deviation bound in Proposition~\ref{chernoff}.
(BSch$4'$) follows from Proposition~\ref{chernoff} in a similar way.%
\endproof

\removelastskip\penalty55\medskip\noindent{\bf Proof of Theorem~\ref{1factbip}. }

\noindent\textbf{Step 1: Choosing the constants and a bi-framework.}
Define new constants such that%
    \COMMENT{$\eps_3$ and $\eps_4$ are implicitly used when applying Corollary~\ref{BEScor}. } 
\begin{align}
0 & < 1/n_0 \ll \eps_{\rm ex} \ll \eps_* \ll \epszero \ll \eps'_0\ll \eps' \ll \eps_1 \ll \eps_2  \ll \eps_3 \ll
\eps_4\ll 1/K_2  \\
 &  \ll \gamma \ll 1/K_1 \ll  \eps''\ll 1/L \ll 1/f \ll \gamma_1  \ll 1/g \ll \eps\ll 1, \nonumber
\end{align}
where $K_1, K_2, L, f, g \in  \mathbb{N}$ and both $K_2$, $g$ are even.%
   \COMMENT{Daniela added $g$ even, we need this when applying Cor~\ref{rdeccorz} with $\ell':=g$} 
Note that we can choose the constants such that
$$\frac{K_1}{28fgL}, \frac{K_2}{4gLK_1}, \frac{4fK_1}{3g(g-1)} \in \mathbb{N}.$$

Let $G$ and $D$ be as in Theorem~\ref{1factbip}. By applying Dirac's theorem to remove a suitable number of edge-disjoint Hamilton cycles if necessary,
we may assume that $D\le n/2$. 
Apply Corollary~\ref{coverA0B02c} with $G$, $\eps_{\rm ex}$, $\eps_*$, $\eps _0$, $K_2$ playing the roles of $F$, $\eps$, $\eps ^*$, $\eps'$, $K$
to find a set $\cC_1$ of 
at most $\eps _{\rm ex} ^{1/3} n$ 
edge-disjoint Hamilton cycles in $G$ so that the graph $G_1$ obtained from $G$
by deleting all the edges in these Hamilton cycles forms part of an
$(\eps_*,\eps_0,K_2,D_1)$-bi-framework $(G_1,A,A_0,B,B_0)$, where%
   \COMMENT{Daniela added $=D-2|\cC_1|$ below, osthus added $|A|+\eps_0 n \ge n/2$ } 
\begin{equation} \label{D1eq}
|A|+\eps_0 n \ge n/2 \ge D_1 =D-2|\cC_1|\ge D-2\eps_{\rm ex}^{1/3}n\ge D-\eps_0 n \ge n/2-2\eps_0 n \ge |A|-2\eps_0 n.
\end{equation}
Note that  $G_1$ is $D_1$-regular and that $D_1$ is even since $D$ was even.
Moreover, since $K_2/LK_1\in\mathbb{N}$, $(G_1,A,A_0,B,B_0)$ is also an $(\eps_*,\eps_0,K_1L,D_1)$-bi-framework
and thus an $(\eps_*,\eps',K_1L,D_1)$-bi-framework.

Let%
     \COMMENT{$r^\diamond$ is $t$ in the $2$-clique case}
\begin{align*}
m_1 & :=\frac{|A|}{K_1}=\frac{|B|}{K_1}, \qquad r:=  \gamma m_1, \qquad
r_1 := \gamma_1 m_1, \qquad
r_2:= 192g^3K_1r, \\
r_3 & := \frac{2rK_1}{L}, \qquad r^\diamond := r_1 +r_2 +r -(Lf -1)r_3, \\
D_4 & :=D_1-2(Lfr_3+7r^\diamond), \qquad t_{K_1L} :=\frac{(1-20\eps_4)D_1}{2(K_1L)^4} .
\end{align*}
Note that (BFR4) implies $m_1/L\in \mathbb{N}$. Moreover, 
\begin{equation}\label{eq:rszzz}
r_2,r_3\le \gamma^{1/2}m_1\le \gamma^{1/3} r_1,  \qquad  r_1/2\le r^\diamond\le 2r_1.
\end{equation}
Further, by changing $\gamma,\gamma_1,\eps_4$ slightly, we may assume that $r/K^2_2,r_1,t_{K_1L}\in\mathbb{N}$.
Since $K_1/L\in\mathbb{N}$ this implies that $r_3\in\mathbb{N}$.
Finally, note that 
\begin{equation} \label{D4bound}
(1+3\eps_*)|A| \ge D \ge D_4  \stackrel{(\ref{eq:rszzz})}{\ge} D_1-\gamma_1 n \stackrel{(\ref{D1eq})}{\ge} |A|-2\gamma_1 n \ge (1-5\gamma_1)|A|.
\end{equation}

\noindent\textbf{Step 2: Choosing a $(K_1, L, m_1, \epszero)$-partition $(\mathcal{P}_1,\mathcal{P}'_1)$.}
We now prepare the ground for the construction of the robustly decomposable graph $G^{\rm rob}$, 
which we will obtain via the robust decomposition lemma (Corollary~\ref{rdeccorz}) in Step~4.

Recall that $(G_1,A,A_0,B,B_0)$ is an $(\eps_*,\eps',K_1L,D_1)$-bi-framework.
Apply Lem\-ma~\ref{part} with $G_1$, $D_1$, $K_1L$, $\eps_*$ playing the roles of $G$, $D$, $K$, $\eps$ to obtain partitions
$A'_1,\dots,A'_{K_1L}$ of $A$ and $B'_1,\dots,B'_{K_1L}$ of $B$
into sets of size $m_1/L$ such that together with $A_0$ and $B_0$ all
these sets $A'_i$ and $B'_i$ form a $(K_1L,m_1/L,\eps_*,\eps_1,\eps_2)$-partition
$\cP'_1$ for $G_1$. 
Note that $(1-\eps_0)n\le n-|A_0\cup B_0|=2K_1m_1\le n$ by (BFR4).
For all $i\le K_1$ and all $h\le L$, let $A_{i,h}:=A'_{(i-1)L+h}$. (So this is just a relabeling of the sets $A'_i$.)
Define $B_{i,h}$ similarly and let
$A_i:= \bigcup_{h\le L} A_{i,h}$ and $B_i:= \bigcup_{h\le L} B_{i,h}$.
Let $\cP_1:=\{A_0,B_0,A_1,\dots,A_{K_1},B_1,\dots,B_{K_1}\}$ denote the
corresponding $(K_1,m_1,\eps_0)$-partition of $V(G)$. Thus $(\cP_1,\cP'_1)$ is a $(K,L,m_1,\eps_0)$-partition of $V(G)$,
as defined in Section~\ref{sec:SFq}.
 
Let $G_2:=G_1[A,B]$.%
   \COMMENT{Daniela: changed def because of change in def of (..)-system, before had "Let $G_2$ be the spanning subgraph of $G_1$ obtained by
keeping all $AB$-edges and deleting all other edges". So $G_2$ included $V_0$.}
We claim that $(G_2, \mathcal{P}_1,\mathcal{P'}_1)$ is a $(K_1, L, m_1, \epszero,\eps')$-bi-scheme. Indeed, clearly~(BSch1) and~(BSch2) hold.
To verify (BSch3), recall that that $(G_1,A,A_0,B,B_0)$ is an $(\eps_*,\eps_0,K_1L,D_1)$-bi-framework and so%
   \COMMENT{Daniela added more detail}
by (BFR5) for all $x \in B$ we have 
$$
d_{G_2}(x,A)\ge d_{G_1}(x)-d_{G_1}(x,B')-|A_0| \ge D_1 - \eps_0n -|A_0| \stackrel{(\ref{D1eq})}{\ge} |A| -4 \eps_0 n
$$ and
similarly $d_{G_2}(y,B)\ge |B| -4 \eps_0 n$ for all $y\in A$. Since $\eps_0\ll \eps'/K_1L$, this implies~(BSch3).\COMMENT{Here we are just using that $\eps_0 n \ll \eps' |A_{i,j}|$, not e.g. (P2)}

\smallskip

\noindent\textbf{Step 3: Balanced exceptional systems for the robustly decomposable graph.}
In order to apply Corollary~\ref{rdeccorz}, we first need to construct suitable balanced exceptional systems.
Apply Corollary~\ref{BEScor} to the $(\eps_*,\eps',K_1L,D_1)$-bi-framework $(G_1,A,A_0,B,B_0)$
with $G_1$, $K_1L$, $\cP'_1$, $\eps_*$ playing the roles of $F$, $K$, $\cP$, $\eps$
in order to obtain a set $\cJ$ of $(K_1L)^4t_{K_1L}$ edge-disjoint balanced exceptional systems in $G_1$ with parameter $\eps_0$
such that for all $1 \le i'_1,i'_2,i'_3,i'_4  \le K_1L$ the set $\cJ$ contains precisely $t_{K_1L}$
$(i'_1,i'_2,i'_3,i'_4)$-BES with respect to the partition $\cP'_1$.
(Note that $F$ in Corollary~\ref{BEScor} satisfies (WF5) since $G_1$ satisfies (BFR5).)%
\COMMENT{osthus added bracket}
So $\cJ$ is the union of all the sets $\cJ_{i'_1i'_2i'_3i'_4}$ returned by
Corollary~\ref{BEScor}. (Note that we will not use all the balanced exceptional systems in~$\cJ$ and
we do not need to consider the Hamilton cycles guaranteed by this result. 
So we do not need the full strength of Corollary~\ref{BEScor} at this point.)

Our next aim is to choose two disjoint subsets $\mathcal{J}_{\rm CA}$ and $\mathcal{J}_{\rm PCA}$ of $\cJ$ with the
following properties:
\begin{itemize}	
\item[(a)] In total $\mathcal{J}_{\rm CA}$ contains $L f r_3$ balanced exceptional systems. For each $i\le f$ and each $h \le L$,
$\mathcal{J}_{\rm CA}$ contains precisely $r_3$ $(i_1,i_2,i_3,i_4)$-BES of style $h$ (with respect to the
$(K,L,m_1,\eps_0)$-partition $(\cP_1,\cP'_1)$) such that $i_1,i_2,i_3,$ $i_4\in \{(i-1)K_1/f+2,\dots,iK_1/f\}$.
\item[(b)] In total $\mathcal{J}_{\rm PCA}$ contains $7r^\diamond$ balanced exceptional systems. For each $i\le 7$,
$\mathcal{J}_{\rm PCA}$ contains precisely $r^\diamond$ $(i_1,i_2,i_3,i_4)$-BES (with respect to
the partition $\cP_1$) with
$i_1,i_2,i_3,i_4\in \{(i-1)K_1/7+2,\dots,iK_1/7\}$.
\end{itemize}
(Recall that we defined in Section~\ref{sec:findBFq} when an $(i_1,i_2,i_3,i_4)$-BES has style $h$ with respect to a
$(K,L,m_1,\eps_0)$-partition $(\cP_1,\cP'_1)$.)
To see that it is possible to choose $\mathcal{J}_{\rm CA}$ and $\mathcal{J}_{\rm PCA}$, split $\cJ$ into two sets $\cJ_1$ and $\cJ_2$ such that both $\cJ_1$ and $\cJ_2$
contain at least $t_{K_1L}/3$ $(i'_1,i'_2,i'_3,i'_4)$-BES with respect to~$\cP'_1$, for all $1 \le i'_1,i'_2,i'_3,i'_4  \le K_1L$.
Note that there are $(K_1/f-1)^4$ choices of 4-tuples $(i_1,i_2,i_3,i_4)$
with $i_1,i_2,i_3,i_4\in \{(i-1)K_1/f+2,\dots,iK_1/f\}$. Moreover, for each such 4-tuple $(i_1,i_2,i_3,i_4)$
and each $h\le L$ there is one 4-tuple $(i'_1,i'_2,i'_3,i'_4)$ with $1 \le i'_1,i'_2,i'_3,i'_4  \le K_1L$
and such that any $(i'_1,i'_2,i'_3,i'_4)$-BES with respect to~$\cP'_1$ is an $(i_1,i_2,i_3,i_4)$-BES of style $h$ with respect to
$(\cP_1,\cP'_1)$. Together with the fact that
$$ 
\frac{(K_1/f-1)^4t_{K_1L}}{3}\ge \frac{D_1}{7(Lf)^4}\ge \gamma^{1/2}n\stackrel{(\ref{eq:rszzz})}{\ge} r_3,
$$
this implies that we can choose a set $\mathcal{J}_{\rm CA}\subseteq \cJ_1$ satisfying~(a).

Similarly, there are $(K_1/7-1)^4$ choices of 4-tuples $(i_1,i_2,i_3,i_4)$
with $i_1$,$i_2$,$i_3$,$i_4\in \{(i-1)K_1/7+2,\dots,iK_1/7\}$. Moreover, for each such 4-tuple $(i_1,i_2,i_3,i_4)$
there are $L^4$ distinct 4-tuples $(i'_1,i'_2,i'_3,i'_4)$ with $1 \le i'_1,i'_2,i'_3,i'_4  \le K_1L$
and such that any $(i'_1,i'_2,i'_3,i'_4)$-BES with respect to~$\cP'_1$ is an $(i_1,i_2,i_3,i_4)$-BES with respect to
$\cP_1$. Together with the fact that
$$ \frac{(K_1/7-1)^4L^4t_{K_1L}}{3}\ge \frac{D_1}{7^5}\ge \frac{n}{3\cdot 7^5}\stackrel{(\ref{eq:rszzz})}{\ge} r^\diamond,$$
this implies that we can choose a set $\mathcal{J}_{\rm PCA}\subseteq \cJ_2$ satisfying~(b).

\smallskip

\noindent\textbf{Step 4: Finding the robustly decomposable graph.}
Recall that $(G_2, \mathcal{P}_1,\cP'_1)$ is a $(K_1, L, m_1, \epszero,\eps')$-bi-scheme.
Apply Proposition~\ref{lem:dirschemezz} with $G_2$, $\cP_1$, $\cP'_1$, $K_1$, $m_1$, $\eps'$ playing the roles of
$G$, $\cP$, $\cP'$, $K$, $m$, $\eps$ to obtain an orientation $G_{2,{\rm dir}}$ of $G_2$%
  \COMMENT{Daniela: had $G_2-V_0$, but now $G_2$ doesn't include $V_0$ anymore}
such that
$(G_{2,{\rm dir}}, \mathcal{P}_1,\cP'_1)$ is a $[K_1, L, m_1, \epszero,2\sqrt{\eps'}]$-bi-scheme. 
Let $C=A_1B_1A_2\dots $ $A_{K_1}B_{K_1}$ be a spanning cycle on the clusters in $\cP_1$.

Our next aim is to use Lemma~\ref{lma:EF-bipartite} in order to extend the balanced exceptional systems in $\mathcal{J}_{\rm CA}$ into $r_3$ edge-disjoint
balanced exceptional factors with parameters $(L,f)$ for $G_{2,{\rm dir}}$ (with respect to $C$, $\mathcal P'_1$).
For this, note that the condition on $\cJ_{CA}$ in Lemma~\ref{lma:EF-bipartite} with $r_3$ playing the role of $q$ is satisfied by~(a). 
Moreover, $Lr_3/m_1=2rK_1/m_1=2\gamma K_1\ll 1$. 
Thus we can indeed apply Lemma~\ref{lma:EF-bipartite} to $(G_{2,{\rm dir}}, \mathcal{P}_1,\cP'_1)$
with $\cJ_{CA}$, $2\sqrt{\eps'}$, $K_1$, $r_3$ playing the roles of
$\cJ$, $\eps$, $K$, $q$ in order to obtain $r_3$ edge-disjoint balanced exceptional factors $BF_1,\dots,BF_{r_3}$ with parameters $(L,f)$
for $G_{2,{\rm dir}}$ (with respect to $C$, $\mathcal P'_1$) such that together these balanced exceptional factors
cover all edges in $\bigcup \mathcal{J}_{\rm CA}$. Let $\mathcal{BF}_{\rm CA}:=BF_1+\dots+ BF_{r_3}$.

Note that $m_1/4g,m_1/L\in\mathbb{N}$ since $m_1=|A|/K_1$ and $|A|$ is divisible by $K_2$ and thus $m_1$ is divisible
by $4gL$ (since $K_2/4gLK_1\in\mathbb{N}$ by our assumption).
Furthermore, $4rK_1^2=4\gamma m_1K_1^2\le \gamma^{1/2}m_1\le m_1$.%
\COMMENT{osthus adapted calculation from error in statement of rob dec corollary}
Thus we can apply Corollary~\ref{rdeccorz} to the $[K_1, L, m_1, \epszero,\eps'']$-bi-scheme
$(G_{2,{\rm dir}}, \mathcal{P}_1,\cP'_1)$ with $K_1$, $\eps''$, $g$ playing the roles of $K$, $\eps$, $\ell'$ to obtain a spanning
subgraph $CA(r)$ of $G_2$ as described there. (Note that $G_2$ equals the graph $G'$ defined in Corollary~\ref{rdeccorz}.)%
    \COMMENT{Daniela: had $G_2-V_0$ instead of $G_2$ (twice)}
In particular, $CA(r)$ is $2(r_1+r_2)$-regular and edge-disjoint from $\mathcal{BF}_{\rm CA}$.

Let $G_3$ be the graph obtained from $G_2$ by deleting all the edges of $CA(r)+ \mathcal{BF}_{\rm CA}$.
Thus $G_3$ is obtained from $G_2$ by deleting at most $2(r_1+r_2+r_3)\le 6r_1=6\gamma_1m_1$ edges at every vertex in $A\cup B=V(G_3)$.
Let $G_{3,{\rm dir}}$ be the orientation of $G_3$%
    \COMMENT{Daniela: had $G_3-V_0$ instead of $G_3$}
in which every edge is oriented in the same way as in $G_{2,{\rm dir}}$.
Then Proposition~\ref{superslice} implies that%
   \COMMENT{Daniela added ref to Proposition~\ref{superslice}}
$(G_{3,{\rm dir}},\cP_1,\cP_1)$ is still a $[K_1, 1,m_1, \epszero,\eps]$-bi-scheme.%
\COMMENT{AL: changed to as $[K_1, m_1, \epszero,\eps]$-bi-scheme is no longer defined.}
Moreover,
$$\frac{r^\diamond}{m_1}\stackrel{(\ref{eq:rszzz})}{\le} \frac{2r_1}{m_1}=2\gamma_1\ll 1.$$
Together with~(b) this ensures that we can apply
Lemma~\ref{lma:EF-bipartite} to $(G_{3,{\rm dir}},\cP_1)$ with $\cP_1$, $\cJ_{PCA}$, $K_1$, $1$, $7$, $r^\diamond$ playing the roles of
$\cP$, $\cJ$, $K$, $L$, $f$, $q$ in order to obtain $r^\diamond$ edge-disjoint balanced exceptional factors $BF'_1,\dots,BF'_{r^\diamond}$ with parameters $(1,7)$
for $G_{3,{\rm dir}}$ (with respect to $C$, $\mathcal P_1$) such that together these balanced exceptional factors
cover all edges in $\bigcup \mathcal{J}_{\rm PCA}$. Let $\mathcal{BF}_{\rm PCA}:=BF'_1+\dots + BF'_{r^\diamond}$.

Apply Corollary~\ref{rdeccorz} to obtain a spanning
subgraph $PCA(r)$ of $G_2$ as described there.%
   \COMMENT{Daniela: had $G_2-V_0$ instead of $G_2$}
In particular, $PCA(r)$ is $10r^\diamond$-regular
and edge-disjoint from $CA(r)+ \mathcal{BF}_{\rm CA}+ \mathcal{BF}_{\rm PCA}$.

Let $G^{\rm rob}:=CA(r)+ PCA(r)+ \mathcal{BF}_{\rm CA}+ \mathcal{BF}_{\rm PCA}$.
Note that by~(\ref{EFdegq}) all the vertices in $V_0:=A_0\cup B_0$ have the same degree  
$r_0^{\rm rob}:=2(Lfr_3+7r^\diamond)$ in $G^{\rm rob}$.
So 
\begin{equation}\label{eq:rrobzz}
7r_1\stackrel{(\ref{eq:rszzz})}{\le} r_0^{\rm rob} \stackrel{(\ref{eq:rszzz})}{\le} 30r_1.
\end{equation}
Moreover,~(\ref{EFdegq}) also implies that all the vertices in $A\cup B$ have the same degree $r^{\rm rob}$ in $G^{\rm rob}$,
where $r^{\rm rob}=2(r_1+r_2+r_3+6r^\diamond)$. So 
$$
r_0^{\rm rob}-r^{\rm rob}=2 \left(Lfr_3+ r^\diamond- (r_1+r_2+r_3)\right)=2(Lfr_3+r-(Lf-1)r_3-r_3 )=2r.
$$

\noindent\textbf{Step 5: Choosing a $(K_2,m_2,\eps_0)$-partition $\cP_2$.}
We now prepare the ground for the approximate decomposition step (i.e.~to apply Lemma~\ref{almostthmbip}). 
For this, we need to work with a finer partition of $A \cup B$ than the previous one
(this will ensure that the leftover from the approximate decomposition step is sufficiently sparse compared to $G^{\rm rob}$).

Let $G_4:=G_1-G^{\rm rob}$ (where $G_1$ was defined in Step~1) and note that 
\begin{equation} \label{D4D1zz}
D_4= D_1-r_0^{\rm rob}=D_1-r^{\rm rob}-2r.
\end{equation}
So 
\begin{equation} \label{degrees4zz}
d_{G_4}(x)=D_4+2r \mbox{ for all } x \in A \cup B \mbox{  and  } d_{G_4}(x)=D_4 \mbox{ for all } x \in V_0.
\end{equation}
(Note%
   \COMMENT{Daniela deleted "Thus every vertex in $V_0$
has degree $D_4$ in $G_4$ while every vertex in $A\cup B$ has degree $D_4+2r$." which we had immediately after (\ref{degrees4zz}), since
this is what (\ref{degrees4zz}) says}
that $D_4$ is even since $D_1$ and $r_0 ^{\rm rob}$ are even.)
So $G_4$ is $D_4$-balanced with respect to $(A,A_0,B,B_0)$
by Proposition~\ref{edge_number}.
Together with the fact that $(G_1$,$A$,$A_0$, $B$,$B_0)$ is an $(\eps_*,\eps_0,K_2,D_1)$-bi-framework, this implies that
$(G_4$,$G_4$,$A$,$A_0$,$B$,$B_0)$ satisfies conditions (WF1)--(WF5) in the definition of an $(\eps_*,\eps_0,K_2,D_4)$-weak framework.
However, some vertices in $A_0\cup B_0$ might violate condition~(WF6). 
(But every vertex in $A\cup B$ will still satisfy~(WF6) with room to spare.) 
So we need to modify the partition of $V_0=A_0\cup B_0$ to obtain a new weak framework.

Consider a partition $A^*_0,B^*_0$ of $A_0\cup B_0$ which maximizes the number of edges in $G_4$ between
$A^*_0\cup A$ and $B^*_0\cup B$. 
Then $d_{G_4}(v, A^*_0\cup A) \le d_{G_4}(v)/2$ for all $v\in A^*_0$ since otherwise $A^*_0\setminus \{v\},B^*_0\cup \{v\}$ would
be a better partition of $A_0\cup B_0$. Similarly $d_{G_4}(v, B^*_0\cup B) \le d_{G_4}(v)/2$ for all $v\in B^*_0$.
 Thus (WF6) holds in $G_4$ (with respect to the partition
$A\cup A^*_0$ and $B\cup B^*_0$).
Moreover, Proposition~\ref{keepbalance} implies that $G_4$ is still $D_4$-balanced with respect to $(A,A^*_0,B,B^*_0)$.
Furthermore, with (BFR3) and (BFR4) applied to $G_1$, we obtain
$e_{G_4}(A\cup A^*_0)\le e_{G_1}(A\cup A_0)+|A^*_0||A\cup A^* _0|\le \eps_0 n^2$ and similarly $e_{G_4}(B\cup B^*_0)\le \eps_0 n^2$.
Finally, every vertex in $A\cup B$ has internal degree at most $\eps_0 n+|A_0 \cup B_0|\le 2\eps_0 n$ in $G_4$ (with respect to the partition
$A\cup A^*_0$ and $B\cup B^*_0$).%
\COMMENT{osthus replaced V0 with A0 B0} 
Altogether this implies that
$(G_4,G_4,A,A_0^*,B,B_0^*)$ is an $(\eps_0,2\eps_0,K_2,D_4)$-weak framework and thus also an $(\eps_0,\eps',K_2,D_4)$-weak framework.     

Without loss of generality we may assume that $|A^*_0|\geq |B^*_0|$.
Apply Lem\-ma~\ref{coverA0B02} to the $(\eps_0,\eps',K_2,D_4)$-weak framework $(G_4,G_4,A,A_0^*,B,B_0^*)$ to find a set $\mathcal C_2$ of $|\mathcal C_2| \leq \eps _0 n$ edge-disjoint Hamilton cycles in $G_4$ so that the graph $G_5$
obtained from $G_4$ by deleting all the edges of these Hamilton cycles forms part of an
$(\eps_0,\eps',K_2,D_5)$-bi-framework $(G_5,A,A_0^*,B,B_0^*)$, where
\begin{align}\label{eqD5}
D_5=D_4-2|\mathcal C_2| \geq D_4 -2\eps _0 n.
\end{align}
Since $D_4$ is even, $D_5$ is even. Further, 
\begin{equation} \label{degrees5zz}
d_{G_5}(x)\stackrel{(\ref{degrees4zz})}{=} D_5+2r \mbox{ for all } x \in A \cup B \ \ \mbox{and} \ \ d_{G_5}(x)\stackrel{(\ref{degrees4zz})}{=}D_5 \mbox{ for all } x \in A^*_0\cup B^*_0.
\end{equation}

Choose an additional constant $\eps '_4$ such that $\eps _3 \ll \eps '_4 \ll 1/K_2$ and so
that%
\COMMENT{Vital: we have had to wait until now to define $\eps '_4$, since it depends on $D_5$ and
$D_5$ only just defined. In particular, $D_5$ depends on $|\mathcal C_2|$.}
$$t_{K_2}:=\frac{(1-20\eps'_4)D_5}{2K_2^4} \in \mathbb N.$$

Now apply Lemma~\ref{part} to $(G_5,A,A_0^*,B,B_0^*)$ with $D_5$, $K_2$, $\eps_0$
playing the roles of $D$, $K$, $\eps$ in order to obtain partitions $A_1,\dots,A_{K_2}$ and
$B_1,\dots,B_{K_2}$ of $A$ and $B$ into sets of size 
\begin{equation} \label{m2def}
m_2:=|A|/K_2
\end{equation} such that together with $A^*_0$ and $B^*_0$ the sets
$A_i$ and $B_i$ form a $(K_2,m_2,\eps_0,\eps_1,\eps_2)$-partition $\cP_2$ for $G_5$.
(Note that the previous partition of $A$ and $B$  plays no role in the subsequent argument, so
denoting the clusters in~$\cP_2$ by $A_i$ and $B_i$ again will cause no notational conflicts.)%
   \COMMENT{Daniela deleted an entire paragraph verifying that $G'_5:=G_5[A,B]$ is part of a bi-scheme since this is not used anymore}

\smallskip

\noindent\textbf{Step 6: Balanced exceptional systems for the approximate decomposition.}
In order to apply Lemma~\ref{almostthmbip}, we first need to construct suitable balanced exceptional systems.
Apply Corollary~\ref{BEScor} to the $(\eps_0,\eps',K_2,D_5)$-bi-framework $(G_5,A,A_0^*,B,B_0^*)$ with
 $G_5$, $\eps_0$, $\eps'_0$, $\eps'_4$, $K_2$, $D_5$, $\cP_2$ playing the roles
of $F$, $\eps$, $\eps_0$, $\eps_4$, $K$, $D$, $\cP$. 
(Note that since we are letting $G_5$ play the role of $F$, condition (WF5) in the corollary immediately follows from (BFR5).)
Let $\cJ'$ be the union of the sets $\cJ_{i_1i_2i_3i_4}$ guaranteed by Corollary~\ref{BEScor}.
So $\cJ'$ consists of $K_2^4t_{K_2}$ edge-disjoint balanced exceptional systems with parameter $\eps'_0$ in $G_5$
(with respect to $\cP_2$).%
   \COMMENT{Daniela added brackets}
Let $\cC_3$ denote the set of Hamilton cycles guaranteed by Corollary~\ref{BEScor}. So $|\cC_3|=10\eps'_4D_5$.%
   \COMMENT{Daniela added new sentence}

Let $G_6$ be the subgraph obtained from $G_5$ by deleting all those edges lying in the Hamilton cycles from~$\cC_3$.
Set $D_6:=D_5-2|\cC_3|$.
So
\begin{equation} \label{degrees6}
d_{G_6}(x)\stackrel{(\ref{degrees5zz})}{=}D_6+2r \mbox{ for all } x \in A \cup B \qquad \mbox{and} \qquad d_{G_6}(x)\stackrel{(\ref{degrees5zz})}{=}D_6 \mbox{ for all } x \in V_0.
\end{equation}
(Note that $V_0=A_0 \cup B_0=A_0^* \cup B_0^*.$)%
\COMMENT{osthus added bracket}
Let $G'_6$ denote the subgraph of $G_6$ obtained by deleting all those
edges lying in the balanced exceptional systems from $\cJ'$. Thus $G'_6=G^\diamond$, where $G^\diamond$ is as defined in Corollary~\ref{BEScor}(iv).
In particular, $V_0$ is an isolated set in $G'_6$ and $G'_6$ is bipartite with vertex classes $A\cup A^*_0$ and $B\cup B^*_0$
(and thus also bipartite with vertex classes $A'=A\cup A_0$ and $B'=B\cup B_0$).%
    \COMMENT{Daniela deleted the following paragraph (as well as the next one saying that $G'_6[A,B]$ is still part of a bi-scheme): 
"Moreover, recall that (as remarked after the corresponding definition in Section~\ref{besconstruct}), a balanced exceptional
system does not contain any $AB$-edges and thus the same holds for $\bigcup \cJ'$ too.
Thus $G'_6$ can be obtained from $G_6$ by keeping all $AB$-edges and deleting all other edges."}


Consider any vertex $v\in V_0$. Then $v$ has degree $D_5$ in $G_5$, degree two in each Hamilton cycle from $\cC_3$,
degree two in each balanced exceptional system from $\cJ'$ and degree zero in~$G'_6$.
Thus
$$D_6+2|\cC_3|=D_5\stackrel{(\ref{degrees5zz})}{=}d_{G_5}(v)=2|\cC_3|+2|\cJ'|+d_{G'_6}(v)=2|\cC_3|+2|\cJ'|$$
and so
\begin{equation}\label{eq:D5zz}
D_6=2|\cJ'|.
\end{equation}

\noindent\textbf{Step 7: Approximate Hamilton cycle decomposition.}
Our next aim is to apply Lemma~\ref{almostthmbip} with $G_6$, $\cP_2$, $K_2$, $m_2$, $\cJ'$, $\eps'$ playing the
roles of $G$, $\cP$, $K$, $m$, $\cJ$, $\eps_0$.
 Clearly, condition~(c) of Lemma~\ref{almostthmbip} is satisfied.
In order to see that condition (a) is satisfied, let $\mu:=(r^{\rm rob}_0-2r)/4K_2m_2$ and note that%
   \COMMENT{Daniela changed inequality below, before had $1/K_2\ll \mu\ll \eps$, but we don't need this for Lemma~\ref{almostthmbip}}
$$
0\le \frac{\gamma_1 m_1}{4K_2m_2}\leq
\frac{7r_1-2r}{4K_2m_2}\stackrel{(\ref{eq:rrobzz})}{\le} \mu \stackrel{(\ref{eq:rrobzz})}{\le} \frac{30r_1}{4K_2m_2}\le
\frac{30\gamma_1 }{K_1}\ll 1.
$$
Recall that every vertex $v\in B$ satisfies%
\COMMENT{osthus adapted calculation to be an `equality' and not just a lower bound}
$$d_{G_5}(v) \stackrel{(\ref{degrees5zz})}{=} D_5+2r \stackrel{(\ref{D4D1zz}),(\ref{eqD5})}{=} D_1-r^{\rm rob}_0+2r \pm 2 \eps_0 n
\stackrel{(\ref{D1eq})}{=}  |A| -r^{\rm rob}_0+2r \pm 4 \eps_0 n.
$$
Moreover,
$$d_{G_5}(v,A)= d_{G_5}(v)-d_{G_5}(v,B\cup B^*_0)-|A^*_0|\ge d_{G_5}(v)-2\eps ' n,$$
where the last inequality holds since $(G_5,A,A^*_0,B,B^*_0)$ is an $(\eps_0,\eps ',K_2,D_5)$-bi-framework (c.f.~conditions (BFR4) and (BFR5)).
Together with the fact that $\cP_2$ is a $(K_2,m_2,\eps_0,\eps_1,\eps_2)$-partition for $G_5$ (c.f. condition (P2)),
this implies that
\begin{align*}
d_{G_5}(v,A_i) & =\frac{d_{G_5}(v,A)\pm \eps_1 n}{K_2}
=\frac{|A|-r^{\rm rob}_0+2r\pm 2\eps_1 n}{K_2} \\ &=
\left(1-\frac{r^{\rm rob}_0-2r}{K_2m_2}\pm 5\eps_1\right)m_2
 =(1-4\mu\pm 5\eps_1)m_2 \\ & =(1-4\mu\pm 1/K_2)m_2.
\end{align*}
Recall that $G_6$ is obtained from $G_5$ by deleting all those edges lying in the Hamilton cycles in $\cC_3$ and that
$$
|\cC_3|= 10\eps'_4 D_5 \leq 10\eps'_4 D_4
 \stackrel{(\ref{D4bound})}{\le} 11 \eps'_4 |A| \stackrel{(\ref{m2def})}{\leq} m_2/K_2.
$$
Altogether this implies that
$d_{G_6}(v,A_i)=(1-4\mu\pm 4/K_2)m_2$. Similarly one can show that $d_{G_6}(w,B_j)=(1-4\mu\pm 4/K_2)m_2$ for all $w\in A$.
So condition~(a) of Lemma~\ref{almostthmbip} holds.

To check condition~(b), note that%
   \COMMENT{Daniela added $\frac{D_4}{2}$ to the display below}
$$|\cJ'|\stackrel{(\ref{eq:D5zz})}{=}  \frac{D_6}{2}\le  \frac{D_4}{2}
\stackrel{(\ref{D4D1zz})}{\le}  \frac{D_1-r^{\rm rob}_0}{2}\le \frac{n}{4}-\mu \cdot 2K_2m_2-r\le
\left(\frac{1}{4}-\mu-\frac{\gamma}{3 K_1}\right)n.$$
Thus condition~(b) of Lemma~\ref{almostthmbip} holds with $\gamma/3 K_1$ playing the role of $\rho$.

Since the edges in $\cJ'$ lie in $G_5$ and  $(G_5,A,A^*_0,B,B^*_0)$ is an $(\eps_0,\eps',K_2,D_5)$-bi-framework,
(BFR5) implies that each $v \in A \cup B$ is incident with an edge in $J$ for at most $\eps ' n+|V_0|\leq 2\eps 'n$ of the%
\COMMENT{osthus added `of the} 
$J \in \cJ '$. 
(Recall that in a balanced exceptional system there are no edges between $A$ and $B$.)
So condition~(d) of Lemma~\ref{almostthmbip} holds with $\eps '$ playing the role of $\eps _0$.

So we can indeed apply Lemma~\ref{almostthmbip} to obtain a collection $\cC_4$ of $|\cJ'|$ edge-disjoint Hamilton cycles in $G_6$
which cover all edges of $\bigcup \cJ'$.

\smallskip

\noindent\textbf{Step 8: Decomposing the leftover and the robustly decomposable graph.}
Finally, we can apply the `robust decomposition property' of $G^{\rm rob}$ guaranteed by Corollary~\ref{rdeccorz}
to obtain a Hamilton decomposition of the leftover from the previous step together with $G^{\rm rob}$.
 
To achieve this, let $H'$ denote the subgraph of $G_6$ obtained by deleting all those edges lying in the Hamilton cycles from~$\mathcal{C}_4$.
Thus~(\ref{degrees6}) and~(\ref{eq:D5zz}) imply that
every vertex in $V_0$ is isolated in $H'$ while every vertex $v\in A\cup B$ has degree $d_{G_6}(v)-2|\cJ'|=D_6+2r-2|\cJ'|=2r$ in~$H'$
(the last equality follows from~(\ref{eq:D5zz})).
Moreover, $H'[A]$ and $H'[B]$ contain no edges. (This holds
since $H'$ is a spanning subgraph of $G_6-\bigcup \cJ'=G'_6$ and since we have
already seen that $G'_6$ is bipartite with vertex classes $A'$ and $B'$.) Now let $H:=H'[A ,B]$. Then 
Corollary~\ref{rdeccorz}(ii)(b) implies that $H + G^{\rm rob}$
has a Hamilton decomposition. Let $\cC_5$ denote the set of Hamilton cycles thus obtained.
Note that $H+G^{\rm rob}$ is a spanning subgraph of $G$ which contains all edges of $G$ which were not covered by $\cC_1\cup \cC_2\cup \cC_3\cup \cC_4$.
So $\cC_1\cup \cC_2\cup \cC_3\cup \cC_4 \cup \cC_5$ is a Hamilton decomposition of $G$.
\endproof

\chapter{Approximate decompositions}\label{paper3}

In this chapter we prove the approximate decomposition results, Lemmas~\ref{almostthm} and~\ref{almostthmbip}. Recall that
Lemma~\ref{almostthm} gives an approximate Hamilton decomposition of our graph (with some additional properties) in the two cliques case whilst Lemma~\ref{almostthmbip} gives an approximate Hamilton decomposition of our graph (with some additional properties) in the bipartite case.
 After introducing some tools in Section~\ref{tools},
we prove Lemma~\ref{almostthm} in Sections~\ref{systembalanced}--\ref{sec:extendmerge}.
We then prove Lemma~\ref{almostthmbip} using a similar approach in the final section.
We remind the reader that many of the relevant definitions for Lemmas~\ref{almostthm} and~\ref{almostthmbip} are stated in Sections~\ref{sec:BES} and~\ref{findBES} respectively.

In this chapter it is convenient to view matchings as graphs (in which every vertex has degree precisely one).

\section{Useful Results} \label{tools}

\subsection{Regular Spanning Subgraphs}
The following lemma implies that any almost complete balanced bipartite graph has an approximate decomposition into perfect matchings.
The proof is a straightforward application of the MaxFlowMinCut theorem.

\begin{lemma}\label{regularsub}
Suppose that
$0<1/m \ll \eps \ll \rho \ll 1$, that $0 \le \mu \le 1/4$ and that $m, \mu m, \rho m \in \mathbb{N}$.
Suppose that $\Gamma$ is a bipartite graph with vertex classes $U$ and $V$ of size $m$ and with
$(1-\mu-\eps)m \leq \delta (\Gamma) \leq \Delta (\Gamma) \leq (1-\mu+\eps)m$.  Then $\Gamma$ contains a spanning $(1-\mu-\rho)m$-regular subgraph
$\Gamma '$.
In particular, $\Gamma$ contains at least $(1-\mu-\rho)m$ edge-disjoint perfect matchings.
\end{lemma}
\proof 
We first obtain a directed network $N$ from $\Gamma$ by adding a source $s$ and a sink $t$.
We add a directed edge $su$ of capacity $(1-\mu-\rho)m$ for each $u \in U$ and a directed edge $vt$ of capacity 
$(1-\mu-\rho)m$ for each $v \in V$. We give all the edges in $\Gamma$ capacity $1$ and direct them from $U$ to $V$.

Our aim is to show that the capacity of any $(s,t)$-cut%
	\COMMENT{Allan: replace cut with $(s,t)$-cut.}
 is at least $(1-\mu-\rho)m^2$. By the MaxFlowMinCut theorem this would imply
that $N$ admits an integer-valued flow of value $(1-\mu-\rho)m^2$ which by construction of $N$ implies the existence of our desired subgraph $\Gamma '$.

Consider any $(s,t)$-cut $(S, \overline{S})$ where $S= \{s\} \cup S_1 \cup S_2$ with $S_1 \subseteq U$ and $S_2 \subseteq V$.
Let $\overline{S}_1:=U\backslash S_1$ and $\overline{S}_2 :=V \backslash S_2$.
The capacity of this cut is
$$(1-\mu-\rho)m(m-|S_1|)+e(S_1, \overline{S}_2)+ (1-\mu-\rho)m|S_2|$$
and therefore our aim is to show that
\begin{align}\label{target}
e(S_1, \overline{S}_2) \geq (1-\mu-\rho)m(|S_1|-|S_2|).
\end{align}
If $|S_1| \le (1-\mu-\rho )m$, then 
\begin{align*}
	e(S_1, \overline{S}_2) &  \ge 
\left( ( 1-\mu-\eps) m - |S_2| \right) |S_1| \\
		& = (1-\mu-\rho) m (|S_1|-|S_2|) + (\rho-\eps ) m |S_1| +|S_2|\left( (1-\mu-\rho)m -|S_1| \right)\\
		& \ge  (1-\mu-\rho) m (|S_1|-|S_2|).
\end{align*}
Thus, we may assume that $|S_1| > (1-\mu-\rho )m$.
Note that $|S_1|-|S_2| = |\overline{S}_2|-|\overline{S}_1|$.
Therefore, by a similar argument, we may also assume that $|\overline{S}_2| > (1-\mu-\rho)m$ and so $|S_2| \le (\mu +\rho) m$.
This implies that
\begin{align*}
	e(S_1, \overline{S}_2) & \ge \sum_{x \in S_1} d_{\Gamma} (x) - \sum_{y \in S_2} d_{\Gamma} (y)
 \ge  (1-\mu-\eps ) m |S_1| - (1-\mu+\eps )m |S_2| \\
	& =(1-\mu-\rho) m (|S_1| - |S_2|) + \rho m (|S_1| -|S_2|) - \eps m( |S_1| +|S_2|)  \\
	& > (1-\mu-\rho)m (|S_1|-|S_2|) + (1- 2\mu-2\rho) \rho m^2 - (1+\mu+\rho) \eps m^2\\
	& \ge  (1-\mu-\rho)m (|S_1|-|S_2|).
\end{align*}
(Note that the last inequality follows as $\eps \ll \rho\ll 1$ and $ \mu \le 1/4$.)
So indeed (\ref{target}) is satisfied, as desired.
\endproof


\subsection{Hamilton Cyles in Robust Outexpanders} \label{sec:rob}
Recall that,
given $0<\nu \leq \tau<1$, we say that a digraph $G$ on $n$ vertices is a \emph{robust $(\nu, \tau)$-outexpander},
if for all $S\subseteq V(G)$ with $\tau n\le |S|\le (1-\tau)n$ the number of vertices that have at least $\nu n$
inneighbours in $S$ is at least $|S|+\nu n$. 
The following result was derived in~\cite{3con} as a straightforward consequence of the
result from~\cite{KOT10} that every robust outexpander of linear minimum degree has a Hamilton cycle.%
    \COMMENT{Daniela: reworded}

\begin{thm}\label{expanderthm}
Suppose that $0 < 1/n\ll \gamma \ll \nu \ll \tau\ll\eta<1$. Let~$G$ be a digraph on~$n$ vertices with
$\delta^+(G), \delta^-(G)\ge \eta n$ which is a robust $(\nu,\tau)$-outexpander. 
Let $y_1, \dots, y_p$ be distinct vertices in $V(G)$ with $p \le \gamma n$.
Then~$G$ contains a directed Hamilton cycle visiting $y_1, \dots, y_p$ in this  order.
\end{thm}


\subsection{A Regularity Concept for Sparse Graphs} \label{sec:reg}
We now formulate a concept of $\eps$-superregularity which is suitable for `sparse' graphs.
Let $G$ be a bipartite graph with vertex classes $U$ and $V$, both of size $m$.
Given $A\subseteq U$ and $B\subseteq V$, we write $d(A,B):=e(A,B)/|A||B|$ for the density of $G$
between $A$ and~$B$. Given $0<\eps,d,d^*,c<1$, we say that $G$ is \emph{$(\eps,d,d^*,c)$-superregular} if the following conditions are satisfied:
\begin{itemize}
\item[(Reg1)] Whenever $A\subseteq U$ and $B\subseteq V$ are sets of size at least $\eps m$, then
$d(A,B)=(1\pm \eps)d$.
\item[(Reg2)] For all $u,u'\in V(G)$ we have $|N(u)\cap N(u')|\le c^2m$.%
	\COMMENT{Previously, we had the following statement, which is the same.
For all $u,u'\in U$ we have $|N(u)\cap N(u')|\le c^2m$. Similarly, for all $v,v'\in V$ we have $|N(v)\cap N(v')|\le c^2m$.}
\item[(Reg3)] $\Delta(G)\le cm$.
\item[(Reg4)] $\delta(G)\ge d^*m$.
\end{itemize}
Note that the above definitions also make sense if $G$ is `sparse' in the sense that $d<\eps$ (which will be the case in our proofs).
A bipartite digraph $G=G[U,V]$ is \emph{$(\eps,d,d^*,c)$-superregular} if this holds for
the underlying undirected graph of~$G$.

The following observation follows immediately from the definition.

\begin{prop} \label{superslice6}
Suppose that $0<1/m \ll  d^*,d, \eps, \eps', c  \ll 1$ and $2\eps'\le d^*$.
Let $G$ be an $(\eps,d,d^*,c)$-superregular bipartite graph with vertex classes
$U$ and $V$ of size $m$. 
Let $U' \subseteq U$ and $V' \subseteq V$ with $|U'|=|V'| \ge (1- \eps')m$.
Then $G[U',V']$ is $(2\eps,d, d^*/2 ,2c)$-superregular.%
   \COMMENT{Let $m':= |U'| = |V'|$, so $(1- \eps')m \le m' \le m$.
Since $ 2\eps m'\ge \frac{\eps m'}{1- \eps'} \ge \eps m$, $G[U',V']$ satisfies (Reg1)--(Reg3).
To see (Reg4) note that the degrees in $G'$ are still at least $d^*m-\eps' m \ge d^* m'/2$.}
\end{prop}

The following two simple observations were made in~\cite{Kelly}.
\begin{prop} \label{superslice5}
Suppose that $0<1/m \ll d^*,d, \eps, c\ll 1$. Let $G$ be an $(\eps,d,d^*,c)$-superregular bipartite graph with vertex classes
$U$ and $V$ of size $m$. Suppose that $G'$ is obtained from $G$ by removing
at most $\eps^2 dm$ edges incident to each vertex from $G$.
Then $G'$ is $(2\eps,d,d^*-\eps^2 d,c)$-superregular.
\end{prop}

\begin{lemma}\label{regtoexpander}
Let $0<1/m\ll \nu\ll \tau\ll d\le \eps\ll \mu,\zeta \le 1/2$ and let $G$ be an
$(\eps,d,\zeta d,d/\mu)$-superregular bipartite graph with vertex classes $U$ and $V$ of size $m$. Let $A\subseteq U$ be such that
$\tau m\le |A|\le (1-\tau) m$. Let $B\subseteq V$ be the set of all those vertices in $V$ which have at least $\nu m$ neighbours in $A$.
Then $|B|\ge |A|+\nu m$.
\end{lemma}


\section{Systems and Balanced Extensions} \label{systembalanced}
\subsection{Sketch Proof of Lemma~\ref{almostthm}} \label{sec:sketchzz}
Roughly%
   \COMMENT{Deryk: new para}
speaking, the Hamilton cycles we find will have the following structure:
let $A_1,\dots,A_K\subseteq A$ and $B_1,\dots,B_K\subseteq B$ be the clusters of the $(K,m, \epszero)$-partition $\mathcal{P}$
of $V(G)$ given in Lemma~\ref{almostthm}. So $K$ is odd.
Let $\mathcal{R}_A$ be the complete graph on $A_1,\dots,A_K$ and $\mathcal{R}_B$ be the complete graph on $B_1,\dots,B_K$.
Since $K$ is odd, Walecki's theorem~\cite{lucas} implies that $\mathcal{R}_A$ has a Hamilton decomposition $C_{A,1}, \dots, C_{A,(K-1)/2}$, and similarly  $\mathcal{R}_B$
has a Hamilton decomposition $C_{B,1}, \dots, C_{B,(K-1)/2}$.%
\COMMENT{Andy: added last part of sentence.}
Every Hamilton cycle $C$ we construct in $G$ will have the property that there is a $j$ so that almost all edges of $C[A]$
wind around $C_{A,j}$ and almost all edges of $C[B]$ wind around $C_{B,j}$.
Below, we describe the main ideas involved in the construction of the Hamilton cycles in more detail.

As indicated above, the first idea is that we can reduce the problem of finding the required edge-disjoint Hamilton cycles (and possibly perfect matchings) in~$G$
to that of finding appropriate Hamilton cycles on each of $A$ and $B$ separately.

More precisely, let $\mathcal{J}$ be a set of edge-disjoint exceptional systems as given in Lemma~\ref{almostthm}.%
    \COMMENT{Daniela: reworded}
By deleting some edges if necessary, we may further assume that $\mathcal{J}$ is an edge-decomposition of $G - G[A] - G[B]$.
Thus, in order to prove Lemma~\ref{almostthm}, we have to find $|\mathcal{J}|$ suitable edge-disjoint subgraphs $H_{A,1}, \dots, H_{A,|\mathcal{J}|}$
of $G[A]$ and $|\mathcal{J}|$ suitable edge-disjoint subgraphs $H_{B,1}, \dots, H_{B,|\mathcal{J}|}$ of $G[B]$ such that
$H_s := H_{A,s} + H_{B,s}  + J_s$ are the desired spanning subgraphs of~$G$.
To prove this, for each $J \in \mathcal{J}$, we consider the two corresponding  auxiliary subgraphs $J^*_A$ and $J^*_B$ defined at the beginning of Section~\ref{sec:BES}. 
Thus  $J^*_A$ and $J^*_B$ have the following crucial properties:
\begin{itemize}
\item[($\alpha_1$)] $J^*_A$ and $J^*_B$ are matchings whose vertices are contained in $A$ and $B,$\COMMENT{B\'ela: added comma} respectively;
\item[($\alpha_2$)] the union of any Hamilton cycle $C^*_A$ in $G[A] + J^*_A$ containing $J^*_A$ (in some suitable order)
and any Hamilton cycle $C^*_B$ in $G[B] + J^*_B$ containing $J^*_B$ (in some suitable order) corresponds to either a Hamilton cycle
of $G$ containing $J$ or to the union of two edge-disjoint perfect matchings of $G$ containing $J$.
\end{itemize}
Furthermore, $J$ determines which of the cases in~($\alpha_2$) holds: If $J$ is a Hamilton exceptional system,
then ($\alpha_2$) will give a Hamilton cycle of $G$, while in the case when $J$ is a matching exceptional system, ($\alpha_2$)
will give the union of two edge-disjoint perfect matchings of $G$.
So roughly speaking, this allows us to work with multigraphs $G^*_A := G[A] + \sum_{J \in \mathcal{J}} J^*_A$ and
$G^*_B:= G[B] + \sum_{J \in \mathcal{J} } J^*_B$ rather than $G$ in the two steps.%
	\COMMENT{Allan: added the word multigraphs}
Furthermore, the processes of finding Hamilton cycles in $G^*_A$ and in $G^*_B$ are independent (see Section~\ref{sec:J*2clique} for more details).%
   \COMMENT{Deryk: added brackets} 

By symmetry, it suffices to consider $G^*_A$ in what follows.
The second idea of the proof is that as an intermediate step, we decompose $G_A^*$ into blown-up Hamilton cycles
$G^*_{A,j}$. Roughly speaking, we will then find an approximate Hamilton decomposition of each $G^*_{A,j}$ separately.

More precisely,%
    \COMMENT{Andy: reworded} 
recall that $\mathcal{R}_A$ denotes the complete graph whose vertex set is $\{A_1, \dots, A_K \}$.
As mentioned above, $\mathcal{R}_A$ has a Hamilton decomposition $C_{A,1}, \dots,$ $ C_{A,(K-1)/2}$.%
    \COMMENT{Daniela: changed $C_j$ to $C_{A,j}$ here and below}
We decompose $G[A]$ into edge-disjoint subgraphs $G_{A,1}, \dots,G_{A,{(K-1)/2}}$ such that each $G_{A,j}$ corresponds to the `blow-up' of $C_{A,j}$,
i.e.~$G_{A,j}[U,W] = G[U,W]$ for every edge $UW \in E(C_{A,j})$. (The edges of $G$ lying inside one of the clusters $A_1, \dots, A_K$
are deleted.) We also partition the set $\{J^*_A: J\in\mathcal{J}\}$ into $(K-1)/2$ sets
$\mathcal{J}^*_{A,1}, \dots, \mathcal{J}^*_{A,(K-1)/2}$ of roughly equal size. Set $G^*_{A,j} := G_{A,j} + \mathcal{J}^*_{A,j}$.
Thus in order to prove Lemma~\ref{almostthm}, we need to find $|\mathcal{J}^*_{A,j}|$ edge-disjoint Hamilton cycles in $G^*_{A,j}$
(for each $j\le (K-1)/2$).
Since $G^*_{A,j}$ is still close to being a blow-up of the cycle $C_{A,j}$, finding such Hamilton cycles 
seems feasible.

One complication is that in order to satisfy~($\alpha_2$), we need to ensure that each Hamilton cycle in $G^*_{A,j}$
contains some $J^*_A\in \mathcal{J}^*_{A,j}$ (and it must traverse the edges of $J^*_A$ in some given order).
To achieve this, we will both orient and order the edges of~$J^*_A$.
So we will actually consider an ordered directed matching $J^*_{A, {\rm dir}}$ instead of $J^*_A$. ($J^*_A$ itself will still be undirected and unordered).
We orient the edges of $G_{A,j}$ such that the resulting oriented graph $G_{A,j,{\rm dir}}$ is a blow-up of the directed cycle $C_{A,j}$.

However, $J^*_{A,{\rm dir}}$ may not be `locally balanced with respect to $C_{A,j}$'.
This means that it is impossible to extend $J^*_{A,{\rm dir}}$ into a directed Hamilton cycle using only edges of $G_{A,j,{\rm dir}}$.
For example, suppose that $G_{A,j,{\rm dir}}$ is a blow-up of the directed cycle $A_1A_2 \dots A_{K}$, i.e.~each edge of $G_{A,j,{\rm dir}}$
joins $A_i$ to $A_{i+1}$ for some $1\le i\le K$.
If $J^*_{A,{\rm dir}}$ is non-empty and $V(J^*_{A,{\rm dir}})\subseteq A_1$, then $J^*_{A,{\rm dir}}$
cannot be extended into a directed Hamilton cycle using edges of $G_{A,j,{\rm dir}}$ only.
Therefore, each $J^*_{A,{\rm dir}} $ will first be extended into a `locally balanced path sequence' $PS$.
$PS$ will have the property that it can be extended to a Hamilton cycle
using only edges of $G_{A,j,{\rm dir}}$. We will call the set $\mathcal{BE}_j$ consisting of all such $PS$ for all $J^*_A\in \mathcal{J}^*_{A,j}$
a balanced extension of $\mathcal{J}^*_{A,j}$.%
   \COMMENT{Deryk: changed $\mathcal{J}^*_{j}$ to $\mathcal{J}^*_{A,j}$ twice}
$\mathcal{BE}_j$ will be constructed in
Section~\ref{sec:BE2clique}, using edges from a sparse graph $H'$ on $A$ (which is actually removed from $G[A]$ before defining $G_{A,1}, \dots, G_{A,(K-1)/2}$).

Finally, we find the required directed Hamilton cycles in $G_{A,j,{\rm dir}}+ \mathcal{BE}_j$ in Section~\ref{sec:extendmerge}.
We construct these by first extending the path sequences in $\mathcal{BE}_j$ into (directed) $1$-factors, using edges which wind around the blow-up of $C_{A,j}$.
These are then transformed into Hamilton cycles using a small set of edges set aside earlier (again the set of these edges winds around the blow-up of $C_{A,j}$).

\subsection{Systems and Balanced Extensions} \label{system}
As mentioned above, the proof of Lemma~\ref{almostthm} requires an edge-decomposition and orientation of $G[A]$ and
$G[B]$ into blow-ups of directed cycles as well as `balanced extensions'. These are defined in the current subsection.

Let $k,m \in \mathbb{N}$. Recall that
a \emph{$(k,m)$-equipartition} $\mathcal{Q}$ of a set $V$ of vertices is a partition of $V$ into sets $V_1, \dots, V_k$ such that $|V_i| = m$ for all $i \le k$.
The $V_i$ are called \emph{clusters} of~$\mathcal{Q}$.
$(G, \mathcal Q, C)$ is a \emph{$(k,m, \mu, \eps )$-cyclic system}
 if the following properties hold:
\begin{enumerate}[label={(Sys{\arabic*})}]
	\item $G$ is a digraph and $\mathcal{Q}$ is a $(k,m)$-equipartition of $V(G)$.
	\item $C$ is a directed Hamilton cycle on $\mathcal{Q}$ and $G$ winds around $C$.%
    \COMMENT{Added the "winding around bit" since we need it in the proof of Lemma~\ref{merging}}
Moreover, for every edge $UW$ of $C$,
we have $d^+_G(u,W) = (1- \mu \pm \eps)m $ for every $u \in U$ and $d^-_G(w,U) = (1- \mu \pm \eps)m$ for every $w \in W$.
\end{enumerate}
So roughly speaking, such a cyclic system is a blown-up Hamilton cycle.

Let $\mathcal{Q}$ be a \emph{$(k,m)$-equipartition} of $V$ and let $C$ be a directed Hamilton cycle on~$\mathcal{Q}$.
We say that a digraph $H$ with $V(H)\subseteq V$ is \emph{locally balanced with respect to $C$} if for every edge $UW$
of $C$, the number of edges of $H$ with initial vertex in $U$ equals the number of edges of $H$ with final vertex in~$W$.

Recall that a path sequence is a digraph which is the union of vertex-disjoint directed paths (some of them might be trivial).
Let $M$ be a directed matching.
We say a path sequence $PS$ is a \emph{$V_i$-extension of $M$ with respect to $\mathcal{Q}$} if each edge of $M$ is
contained in a distinct directed path in $PS$ having its final vertex in~$V_i$.%
	\COMMENT{We write $V_i$-extension instead of $i$-extension, because in the bipartite case, we take $\mathcal{Q} = \{A_1, \dots, A_K, B_1, \dots, B_K\}$.}
Let $\mathcal{M} := \{ M_{1}, \dots, M_{q}\}$ be a set of directed matchings. 
A set $\mathcal{BE}$ of path sequences is a \emph{balanced extension of $\mathcal{M}$ with respect to $(\mathcal{Q},C)$
and parameters $(\eps, \ell)$} if $\mathcal{BE}$ satisfies the following properties:
\begin{itemize}
	\item[(BE1)] $\mathcal{BE}$ consists of $q$ path sequences $PS_1, \dots, PS_q$ such that $V(PS_i)\subseteq V$ for
each $i\le q$, each $PS_i$ is locally balanced with respect to $C$ and $PS_1-M_1,\dots,PS_q-M_q$ are edge-disjoint from each other.%
   \COMMENT{Daniela: now have that $PS_1-M_1,\dots,PS_q-M_q$ are edge-disjoint instead of the $PS_s$ themselves}
	\item[(BE2)] Each $PS_s$ is a $V_{i_s}$-extension of $M_s$ with respect to $\mathcal{Q}$ for some $i_s \le k$.
Moreover, for each $i \le k$ there are at most $\ell m/k$  indices $s\le q$ such that $i_s=i$.%
    \COMMENT{will have $k=2K$ in the bip case}
	\item[(BE3)] $|V(PS_s) \cap V_i |\le \eps m$ for all $i \le k $ and $s \le q$. Moreover, for each $i \le k$, there are at most $\ell m/k$ path sequences $PS_s \in \mathcal{BE}$ such that $V(PS_s) \cap V_i \ne \emptyset$.
\end{itemize}

Note that the `moreover part' of (BE3) implies the `moreover part' of (BE2).%
     \COMMENT{Deryk: new sentence}

Given an ordered directed matching $M = \{f_1, \dots, f_{\ell}\}$, we say that a directed cycle $C'$ is
\emph{consistent with $M$} if $C'$ contains $M$ and visits the edges $f_1, \dots ,f_{\ell}$ in this order.
The following observation will be useful:
suppose that $PS$ is a $V_i$-extension of $M$ and let $x_j$ be the final vertex of the path in $PS$ containing $f_j$.
(So $x_1, \dots, x_{\ell}$ are distinct vertices of $V_i$.)
Suppose also that $C'$ is a directed cycle which contains $PS$  and visits $x_1, \dots, x_{\ell}$ in this order.
Then $C'$ is consistent with $M$.

\section[Finding Systems and Balanced Extensions for Two Cliques Case]{Finding Systems and Balanced Extensions for the Two Cliques Case}
Let $G$ be a graph, let $\mathcal{P}$ be a $(K,m, \eps_0)$-partition of $V(G)$ and
let $\mathcal{J}$ be a set of exceptional systems as given by Lemma~\ref{almostthm}.
The aim of this section is to decompose $G[A]+G[B]$ into $(k, m, \mu , \eps )$-cyclic systems and to construct balanced extensions
as described in Section~\ref{sec:sketchzz}. First we need to define $J^*_{A, {\rm dir}}$ and $J^*_{B, {\rm dir}}$ for each exceptional system~$J\in \mathcal{J}$.
Recall from Section~\ref{sec:sketchzz} that these are introduced in order to be able to consider $G[A]$ and $G[B]$ separately
(and thus to be able to ignore the exceptional vertices in $V_0=A_0 \cup B_0$).

\subsection{Defining the Graphs $J^*_{A,{\rm dir}}$ and $J^*_{B, {\rm dir}}$} \label{sec:J*2clique}
We recall the following definition of $J^*$ from Section~\ref{sec:BES}.
Let $A,A_0,B,B_0$ be a partition of a vertex set $V$ on $n$ vertices and let $J$ be an exceptional system with parameter $\eps_0$.
Since each maximal path in $J$ has endpoints in $A \cup B$ and internal vertices in $V_0$, an exceptional system $J$
naturally induces a matching $J^*_{AB}$ on $A \cup B$.
More precisely, if $P_1, \dots ,P_{\ell'}$ are the non-trivial paths in~$J$ and $x_i, y_i$ are the endpoints of $P_i$, then
we define $J^*_{AB} := \{x_iy_i : i  \le \ell'\}$. Thus $e_{J^*_{AB}}(A,B)$ is equal to the number of $AB$-paths in~$J$.
In particular, $e_{J^*_{AB}}(A,B)$ is positive and even if $J$ is a Hamilton exceptional system, while $e_{J^*_{AB}}(A,B)=0$ if $J$
is a matching exceptional system. Without loss of generality we may assume that $x_1y_1, \dots, x_{2\ell}y_{2\ell}$ is
an enumeration of the edges of $J^*_{AB}[A,B]$, where $x_i \in A$ and $y_i \in B$. Define 
$$
J_A^* := J^*_{AB}[A] \cup \{x_{2i-1} x_{2i} :1 \le i \le \ell \} \ \
 \mbox{and} \ \ \ J_B^* := J^*_{AB}[B] \cup \{y_{2i} y_{2i+1} :1 \le i \le \ell \}
$$ 
(with indices considered modulo $2\ell$).
Let $J^* := J_A^* + J_B^*$. So $J^*$ is a matching and $e(J^*) = e(J^*_{AB})$. Moreover, by (EC2), (EC3) and (ES3) we have
\begin{align}
	e(J^*) = e(J^*_{AB}) \le |V_0| + \sqrt{\eps_0} n. \label{ESeqzz}
\end{align}
Recall that the  edges in $J^*$ are called \emph{fictive} edges, and  that if $J_1$ and $J_2$ are two edge-disjoint exceptional systems, then $J^*_1$
and $J^*_2$ may not be edge-disjoint.%
   \COMMENT{Daniela: deleted the following, since it is not needed anymore: However, we will always view fictive edges as being distinct from each other and from the edges
in other graphs. So in particular, whenever $J_1$ and $J_2$ are two exceptional systems, we will view $J^*_1$
and $J^*_2$ as being edge-disjoint.}

Recall that we say that an (undirected) cycle $C$ is \emph{consistent with $J_A^*$} if $C$ contains $J_A^*$ and (there is an orientation of $C$ which)
visits the vertices $x_1, \dots ,x_{2\ell}$ in this order.%
    \COMMENT{We need to prescribe a vertex rather than an edge ordering as it is important in Proposition~\ref{prop:ES}.}
In a similar way we define when a cycle is consistent with $J_B^*$.

As mentioned in Section~\ref{sec:sketchzz}, we will orient and order the edges of $J^*_A$ and $J^*_{B}$ in a suitable way
to obtain $J^*_{A, {\rm dir}}$ and $J^*_{B, {\rm dir}}$.
Accordingly, we will need an oriented version of Proposition~\ref{prop:ES}.
For this,  we first orient the edges of $J^*_{A}$ by orienting the edge $x_{2i-1} x_{2i}$ from $x_{2i-1}$ to $x_{2i}$ for all $i \le \ell$ and the edges of $J^*_{AB}[A]$ arbitrarily. 
Next we order these directed edges as $f_1, \dots, f_{\ell_A}$ such that $f_i = x_{2i-1} x_{2i}$ for all $i \le \ell$, where $\ell_A := e(J_A^*)$.
Define $J_{A, { \rm dir } }^*$ to be the ordered directed matching $\{f_1, \dots, f_{\ell_A} \}$.
Similarly, to define $J_{B, { \rm dir } }^*$, we first orient the edges of $J^*_{B}$ by orienting the edge $y_{2i} y_{2i+1}$
from $y_{2i}$ to $y_{2i+1}$ for all $i \le \ell$ and the edges of $J^*_{AB}[B]$ arbitrarily. 
Next we order these directed edges as $f_1', \dots, f_{\ell_B}'$ such that $f_i' = y_{2i} y_{2i+1}$ for all $i \le \ell$, where $\ell_B := e(J_B^*)$.
Define $J_{B, { \rm dir } }^*$ to be the ordered directed matching $\{f_1', \dots, f_{\ell_B}' \}$.
Note that if $J$ is an $(i,i')$-ES, then $V(J_{A, { \rm dir } }^*) \subseteq A_i$ and $V(J_{B, { \rm dir } }^*) \subseteq B_{i'}$.
Recall from Section~\ref{system} that a directed cycle $C_{A, {\rm dir}}$ is \emph{consistent with} $J_{A, { \rm dir } }^*$
if $C_{A, {\rm dir}}$ contains $J_{A, { \rm dir } }^*$ and visits the edges $f_1, \dots ,f_{\ell_A}$ in this order.%
   \COMMENT{Daniela: reworded} 
The following proposition, which is similar to Proposition~\ref{prop:EPS} follows easily from Proposition~\ref{prop:ES}.

\begin{prop} \label{prop:ES2}
Suppose that $A,A_0,B,B_0$ forms a partition of a vertex set $V$.
Let $J$ be an exceptional system. Let $C_{A, {\rm dir}}$ and $C_{B, {\rm dir}}$ be two directed cycles such that 
\begin{itemize}
	\item $C_{A, {\rm dir}}$ is a directed Hamilton cycle on $A$ that is consistent with $J_{A, {\rm dir}}^*$;
	\item $C_{B, {\rm dir}}$ is a directed Hamilton cycle on $B$ that is consistent with $J_{B, {\rm dir}}^*$.
\end{itemize}
Then the following assertions hold, where $C_A$ and $C_B$ are the undirected cycles obtained from $C_{A, {\rm dir}}$
and $C_{B, {\rm dir}}$ by ignoring the directions of all the edges. 
\begin{itemize}
	\item[\rm (i)] If $J$ is a Hamilton exceptional system, then $C_A+C_B - J^* +J$ is a Hamilton cycle on $V$.
		
	\item[\rm (ii)] If $J$ is a  matching exceptional system, then $C_A+C_B - J^* +J$ is the union of a Hamilton cycle on $A'$ and a Hamilton cycle on $B'$.
In particular, if both $|A'|$ and $|B'|$ are even, then $C_A+C_B - J^* +J$ is the union of two edge-disjoint perfect matchings on $V$.
\end{itemize}
\end{prop}

\subsection{Finding Systems}
In this subsection, we will decompose (and orient) $G[A]$ into cyclic systems
$(G_{A,j,{\rm dir}}, \mathcal{Q}_A, C_{A,j})$, one for each $j\le (K-1)/2$.
Roughly speaking, this corresponds to a decomposition into (oriented) blown-up Hamilton cycles.
We will achieve this by considering a Hamilton decomposition of $\mathcal{R}_A$, where 
$\mathcal{R}_A$ is the complete graph on $\{A_1, \dots, A_K \}$. So each $C_{A,j}$ corresponds to one of the
Hamilton cycles in this Hamilton decomposition. We will split the set $\{J^*_{A,{\rm dir}} : J \in \mathcal{J} \}$
into subsets $\mathcal{J}^*_{A,j}$ and assign $\mathcal{J}^*_{A,j}$ to the $j$th cyclic system. Moreover,
for each $j\le (K-1)/2$,%
   \COMMENT{Daniela replaced $j\le (K-1)/2$ by $j\le (K-1)/2$}
we will also set aside a sparse spanning subgraph $H_{A,j}$%
\COMMENT{Andy:replaced $H_j$ with $H_{A,j}$.}
of $G[A]$, which is removed from $G[A]$
before the decomposition into cyclic systems. $H_{A,j}$%
\COMMENT{Andy:replaced $H_j$ with $H_{A,j}$.}
will be used later on in order to
find a balanced extension of $\mathcal{J}^*_{A,j}$. We proceed similarly for $G[B]$.

\begin{lemma}\label{sysdecom}
Suppose that $0<1/n \ll \eps_0  \ll 1/K \ll \rho    \ll 1$ and $0 \leq \mu \ll 1$,
where $n,K \in \mathbb N$ and $K$ is odd.
Suppose that $G$ is a graph on $n$ vertices and $\mathcal{P}$ is a $(K, m, \eps _0)$-partition of $V(G)$.
Furthermore, suppose that the following conditions hold:%
\begin{itemize}
	\item[{\rm (a)}] $d(v,A_i) = (1 - 4 \mu \pm 4 /K) m $ and $d(w,B_i) = (1 - 4 \mu \pm 4 /K) m $ for all
	$v \in A$, $w \in B$ and $1\leq i \leq K$.
	\item[{\rm (b)}] There is a set $\mathcal J$ which consists of at most $(1/4-\mu - \rho)n$ edge-disjoint exceptional systems with parameter $\eps_0$ in~$G$.
	\item[{\rm (c)}] $\mathcal J$ has a partition into $K^2$ sets $\mathcal J_{i,i'}$ (one for all $1\le i,i'\le K$) such that each
$\mathcal J_{i,i'}$ consists of precisely $|\mathcal J|/{K^2}$ $(i,i')$-ES with respect to~$\cP$.
   \item[{\rm (d)}] If $\mathcal{J}$ contains matching exceptional systems then $|A'|=|B'|$ is even.\end{itemize}%
Then for each $1 \le j \le (K-1)/2$, there is a pair of tuples $(G_{A,j}$,$\mathcal{Q}_{A}$,$C_{A,j}$,$ H_{A,j}$,$\mathcal{J}^*_{A,j})$
and $(G_{B,j}, \mathcal{Q}_{B}, C_{B,j}, H_{B,j}, \mathcal{J}^*_{B,j})$  such that the following assertions hold:
\begin{itemize}
\item[{\rm (a$_1$)}] Each of $C_{A,1}, \dots, C_{A,(K-1)/2}$ is a directed Hamilton cycle on $\mathcal{Q}_A := \{A_1,$ $ \dots, A_K \}$
such that the undirected versions of these Hamilton cycles form a Hamilton decomposition of the complete graph on $\mathcal{Q}_A$.
\item[{\rm (a$_2$)}] $\mathcal{J}^*_{A,1}, \dots, \mathcal{J}^*_{A,(K-1)/2}$ is a partition of $\{J^*_{A,{\rm dir}} : J \in \mathcal{J} \}$.
\item[{\rm (a$_3$)}] Each $\mathcal{J}^*_{A,j}$ has a partition into $K$ sets $\mathcal J^*_{A,j,i}$ (one for each $1\le i\le K$)
such that $|\mathcal J^*_{A,j,i}| \le (1- 4\mu - 3\rho) m/K$ and 
each $J^*_{A, {\rm dir}} \in \mathcal J^*_{A,j,i}$ is an ordered directed matching with $e(J^*_{A, {\rm dir}})\le 5 K \sqrt{\eps_0} m$
and $V(J^*_{A, {\rm dir}})\subseteq A_i$.
\item[{\rm (a$_4$)}] $G_{A,1},\dots, G_{A,(K-1)/2}, H_{A,1},\dots,  H_{A,(K-1)/2}$ are edge-disjoint subgraphs of $G[A]$.
\item[{\rm (a$_5$)}] $H_{A,j}[A_{i},A_{i'}]$ is a $10K\sqrt{\eps_0}m$-regular graph for all $j \le (K-1)/2$ and all $i,i' \le K$ with $i \ne i'$.%
\COMMENT{We have omitted the floor/ceiling on $10K\sqrt{\eps_0}m$.}
\item[{\rm (a$_6$)}] For each $j \le (K-1)/2$, there exists an orientation $G_{A,j,{\rm dir}}$ of $G_{A,j}$ such
that $(G_{A,j,{\rm dir}}, \mathcal{Q}_A, C_{A,j})$ is a $(K,m, 4\mu, 5/K)$-cyclic system.
\item[{\rm (a$_7$)}] Analogous statements to {\rm(a$_1$)--(a$_6$)} hold for $C_{B,j}, \mathcal{J}^*_{B,j},G_{B,j}, H_{B,j}$ for all
$j \le (K-1)/2$, with $\mathcal{Q}_B := \{ B_1, \dots, B_K\}$.
\end{itemize}
\end{lemma}

\begin{proof}
Since $K$ is odd, by Walecki's theorem the complete graph on $\{A_1, \dots, A_K \}$ has a Hamilton decomposition.
(a$_1$) follows by orienting the edges of each of these Hamilton cycles to obtain directed Hamilton cycles $C_{A,1}, \dots, C_{A,(K-1)/2}$.

For each $i,i' \le K$, we partition $\mathcal J_{i,i'}$ into $(K-1)/2$ sets $\mathcal J_{i,i',j}$ (one for each $j\le (K-1)/2$)
whose sizes are as equal as possible.
Note that if $J \in \mathcal J_{i,i',j}$, then $J$ is an $(i,i')$-ES and so $V(J^*_{A, { \rm dir} }) \subseteq A_i$.
Since $\mathcal{P}$ is a $(K,m, \eps_0)$-partition of $V(G)$, $|V_0| \le \eps_0 n $ and $(1- \eps_0) n \le 2Km$.
Hence, 
\begin{align*}
e(J^*_{A, { \rm dir} }) \le e(J^*) \overset{\eqref{ESeqzz}}{\le} |V_0| + \sqrt{\eps_0}n 
\le 5 \sqrt{\eps_0} K m.
\end{align*}
(a$_2$) is satisfied by setting $\mathcal{J}^*_{A,j,i} := \bigcup_{i' \le K} \{J^*_{A, { \rm dir} }: J \in \mathcal J_{i,i',j}\} $
and $\mathcal{J}^*_{A,j} := \bigcup_{i \le K}  \mathcal{J}^*_{A,j,i}$.
Note that 
\begin{align*}
|\mathcal{J}^*_{A,j,i}| & \le \sum_{i' \le K}\left(\frac{2|\mathcal J_{i,i'}|}{K-1}+1 \right) \overset{{\rm (c)}}{=}  \frac{2|\mathcal J|}{K(K-1)}+K  
 \overset{{\rm (b)}}{\le} \frac{(1/2 - 2\mu - 2\rho) n}{K(K-1)} + K \\
& \le  (1-4 \mu - 3 \rho) m/K
\end{align*}
as $2K m  \ge (1- \eps_0)  n $ and $1 /n \ll \eps_0 \ll 1/K \ll \rho$.
Hence (a$_3$) holds.

For $i,i' \le K$ with $i \ne i'$, apply Lemma~\ref{regularsub} with $G[A_i,A_{i'}], 4/K, \rho, 4\mu$ playing the roles of $\Gamma, \eps, \rho, \mu$
to obtain a spanning $(1- 4 \mu - \rho)m$-regular subgraph $H_{i,i'}$ of $G[A_i,A_{i'}]$.
Since $H_{i,i'}$ is a regular bipartite graph and $\epszero \ll 1/K , \rho \ll 1$ and $0 \le \mu \ll 1$,
there exist $(K-1)/2$ edge-disjoint $10K \sqrt{ \epszero } m$-regular spanning subgraphs $H_{i,i',1}, \dots, H_{i,i',(K-1)/2}$ of $H_{i,i'}$.
Set $H_{A,j} := \sum_{1 \le i, i' \le K} H_{i,i',j}$ for each $j \le (K-1)/2$. So (a$_5$) holds.

Define $G_A:=G[A]-(H_{A,1}+\dots+H_{A,(K-1)/2})$.
Note that, as $\eps_0 \ll 1/K$, (a) implies that $d_{G_A}(v,A_i) = (1 - 4 \mu \pm 5 /K) m $ for all $v \in A$ and all $i \le K$.
For each $j \le (K-1)/2$, let $G_{A,j}$ be the graph on $A$ whose edge set is the union of $G_A[A_{i},A_{i'}]$ for each edge $A_iA_{i'} \in E(C_{A,j})$.
Define $G_{A,j, {\rm dir}}$ to be the oriented graph obtained from $G_{A,j}$ by orienting every edge in $G_A[A_{i},A_{i'}]$ from $A_i$ to $A_{i'}$
(for each edge $A_iA_{i'} \in E(C_{A,j})$).
Note that $ ( G_{A,j, {\rm dir}}, \mathcal{Q}_A, C_{A,j})$ is a $(K,m, 4\mu, 5/K) $-cyclic system for each $j \le (K-1)/2$.
Therefore, (a$_4$) and (a$_6$) hold. (a$_7$) can be proved by a similar argument.
\end{proof}

\subsection{Constructing Balanced Extensions} \label{sec:BE2clique}

Let $(G_{A,j}$,$\mathcal{Q}_A$,$C_{A,j}$,$H_{A,j}$,$\mathcal{J}^*_{A,j})$ be one of the $5$-tuples obtained by Lemma~\ref{sysdecom}.
The next lemma will be applied to find a balanced extension of $\mathcal{J}^*_{A,j}$ with respect to $(\mathcal{Q}_A,C_{A,j})$,
using edges of $H_{A,j}$ (after a suitable orientation of these edges).
Consider any $J^*_{A, {\rm dir}} \in \mathcal{J}^*_{A,j,i}$.%
   \COMMENT{Daniela: changed $\mathcal{J}^*_{A,j}$ to $\mathcal{J}^*_{A,j,i}$}
Lemma~\ref{sysdecom}(a$_3$) guarantees that
$V(J^*_{A, {\rm dir}}) \subseteq A_i$, and so $J^*_{A, {\rm dir}}$ is an $A_i$-extension of itself.
Therefore, in order to find a balanced extension of $\mathcal{J}^*_{A,j}$, it is enough to extend each $J^*_{A, {\rm dir}} \in \mathcal{J}^*_{A,j}$
into a locally balanced path sequence by adding a directed matching which is vertex-disjoint from $J^*_{A, {\rm dir}}$ in such a way that~(BE3)
is satisfied as well.

\begin{lemma}\label{balanceextension}
Suppose that $0<1/m \ll \eps  \ll 1$ and%
\COMMENT{Andy: added and} 
that $m,k \in \mathbb N$ with $k \ge 3$.%
\COMMENT{Need $k \ge 3$ to have a cycle. Note the switch from $K$ to $k$ is intentional here, as will have $k=2K$ in the bip section}
Let $\mathcal{Q} =\{V_1, \dots, V_k\}$ be a $(k,m)$-equipartition of a set $V$ of vertices and let $C = V_1 \dots V_k$ be a directed cycle.
Suppose that there exist a set $\mathcal{M}$ and a graph $H$ on $V$ such that the following conditions hold:
\begin{itemize}
\item[{\rm (i)}] $\mathcal{M}$ can be partitioned into $k$ sets $\mathcal M_1,\dots, \mathcal M_k$%
   \COMMENT{Daniela: reworded}
such that $|\mathcal M_i| \le m/k$ and each $M \in \mathcal M_{i}$ is an ordered directed matching with $e(M)\le \eps m$ and $V(M)\subseteq V_i$
(for all $i\le k$).
\item[{\rm (ii)}] $H[V_{i-1},V_{i+1}]$ is a $2 \eps m$-regular graph for all $i\le k$.
\end{itemize}
Then there exist an orientation $H_{{\rm dir}}$ of $H$ and a balanced extension $\mathcal{BE}$ of $\mathcal{M}$ with respect to
$(\mathcal{Q},C)$ and parameters $( 2 \eps , 3 )$ such that each path sequence in $\mathcal{BE}$ is obtained from some
$M\in \mathcal{M}$ by adding edges of~$H_{{\rm dir}}$.
\end{lemma}

\begin{proof}
Fix $i \le k$ and write $\mathcal{M}_i := \{ M_1, \dots, M_{|\mathcal{M}_i|}\}$.
We orient each edge in $H[V_{i-1},V_{i+1}]$ from $V_{i-1}$ to $V_{i+1}$.
By (ii), $H[V_{i-1},V_{i+1}]$ can be decomposed into $2 \eps m$ perfect matchings. 
Each perfect matching can be split into $1/2\eps $ matchings, each containing at least $ \eps m$ edges.%
    \COMMENT{Not sure why we need $1/2\eps $ instead of $1/\eps $ here. Is it because of the rounding? (Allan: Yes, it is due to rounding.)}
Recall from (i) that $|\mathcal{M}_i| \le m/k$ and $e(M_j ) \le \eps m$ for all $M_j  \in \mathcal{M}_i$.
Hence, $H[V_{i-1},V_{i+1}]$ contains $|\mathcal{M}_i|$ edge-disjoint matchings $M'_1, \dots, M'_{|\mathcal{M}_i|}$ with $e(M'_j) = e(M_j)$ for all $j \le |\mathcal{M}_i|$.
Define $PS_j:=M_j \cup M_j'$.
Note that $PS_j$ is locally balanced with respect to $C$.
Also, $PS_j$ is a $V_i$-extension of $M_j$ (as $M_j \in \mathcal{M}_i$ and so $V(M_j) \subseteq V_i$ by~(i)).
Moreover,
\begin{align}\label{eq:new}
	|V(PS_j) \cap V_{i'}|  = \begin{cases}
		|V(M_j)| =2 e(M_j ) \le 2 \eps m & \textrm{if $i' = i$,}\\
		e(M_j ) \le \eps m & \text{if $i' = i+1$ or $i' = i-1$,}\\
		0 & \text{otherwise.}
	\end{cases}
\end{align}
For each $i\le k$, set $\mathcal{PS}_i := \{ PS_1, \dots, PS_{|\mathcal{M}_i|}\}$.
Therefore, $\mathcal{BE}: = \bigcup_{i \le k} \mathcal{PS}_i$ is a balanced extension of $\mathcal{M}$ with respect to $(\mathcal{Q},C)$ and parameters $(2\eps, 3)$.
Indeed, (BE3) follows from (\ref{eq:new}). As remarked after the definition of a balanced extension, this also implies the
`moreover part' of (BE2).%
    \COMMENT{Daniela: last 2 sentences are new. Previously had "For (BE2), note that the number of $V_i$-extensions is $|\mathcal{M}_i| \le m/k$."}
Hence the lemma follows (by orienting the remaining edges of $H$ arbitrarily).
\end{proof}


\section{Constructing Hamilton Cycles via Balanced Extensions}  \label{sec:extendmerge}

Recall that a cyclic system can be viewed as a blow-up of a Hamilton cycle.%
   \COMMENT{Deryk: new sentence}
Given a cyclic system $(G,\mathcal{Q}, C)$ and a balanced extension~$\mathcal{BE}$ of a set $\mathcal{M}$
of ordered directed matchings,
our aim is to extend each path sequence in $\mathcal{BE}$ into a Hamilton cycle using edges of~$G$.
Moreover, each Hamilton cycle has to be consistent with a distinct $M \in \mathcal{M}$. This is achieved by the
following lemma, which is the key step in proving Lemma~\ref{almostthm}.%
   \COMMENT{Deryk: new sentence}

\begin{lemma}\label{merging}
Suppose that $0<1/m \ll \eps_0, \eps' , 1/k \ll 1/\ell , \rho  \leq 1$, that $0 \le \mu , \rho \ll 1$\COMMENT{It is intentional that I have stated two conditions with $\rho$ in. Thought
it was the easiest way to write $\eps_0, \eps' , 1/k \ll 1/\ell , \rho $, and $1/\ell \leq 1$ and
$\rho \ll 1$.}
 and that $m,k, \ell, q\in \mathbb{N}$ with $q \le (1- \mu - \rho) m$.
Let $(G,\mathcal{Q}, C)$ be a $(k,m, \mu, \eps')$-cyclic system and let $\mathcal{M} = \{M_1, \dots, M_q\}$
be a set of $q$ ordered directed matchings.%
	\COMMENT{Allan: $M_1, \dots, M_1$ are no longer edge-disjoint. Reason:
In two cliques case, $J^*_A = M_1$. This has lots of knock on effect on the rest....
The next two sentences also changed.}
Suppose that $\mathcal{BE} = \{ PS_1, \dots, PS_q \}$ is a balanced extension of $\mathcal{M}$ with respect to $(\mathcal{Q},C)$ and
parameters $(\eps_0, \ell)$ such that for each $s \le q$, $M_s \subseteq PS_s$.
Then there exist $q$ Hamilton cycles $C_1, \dots, C_q$ in $G + \mathcal{BE}$ such that for all $s \le q$, $C_s$ contains $PS_s$ and is consistent with $M_s$, and such that $C_1 - PS_1, \dots, C_q - PS_q$ are edge-disjoint subgraphs of $G$.
\end{lemma}

Lemma~\ref{merging} will be used both in the two cliques case (i.e.~to prove Lemma~\ref{almostthm}) and in the bipartite case (i.e.~to prove Lemma~\ref{almostthmbip}).

We now give an outline of the proof of Lemma~\ref{merging}, where for simplicity we assume that
the path sequences in the balanced extension~$\mathcal{BE}$ are edge-disjoint from each other.%
	\COMMENT{Deryk: reworded}
Our first step is to remove a sparse subdigraph $H$ from $G$ (see Lemma~\ref{slicelemmazz}), and set $G':= G - H$.
Next we extend each path sequence in~$\mathcal{BE}$ into a (directed) $1$-factor using edges of~$G'$
such that all these $1$-factors are edge-disjoint from each other (see Lemma~\ref{extend1-factor}). 
Finally, we use edges of $H$ to transform the $1$-factors into Hamilton cycles (see Lemma~\ref{mergecycles2}).

The following lemma enables us to find a suitable sparse subdigraph $H$ of $G$.
Recall that $(\eps, d, d^*, c)$-superregularity was defined in Section~\ref{sec:reg}.%
   \COMMENT{Daniela: statement and proof of Lemma~\ref{slicelemmazz} are slightly different now, since there
was a problem with the application of this lemma in the proof of Lemma~\ref{merging} - CHECK Lemma~\ref{slicelemmazz} and
the proof of Lemma~\ref{merging}!}

\begin{lemma} \label{slicelemmazz}
Suppose that $0<1/m\ll \eps' \ll  \gamma \ll \eps\ll 1$ and $0 \le \mu \ll \eps$.
Let $G$ be a bipartite graph with vertex classes $U$ and $W$ of size $m$ such that 
$d(v) = ( 1- \mu \pm \eps')m$ for all $v \in V(G)$.
Then there is a spanning subgraph $H$ of $G$ which satisfies the following properties: 
\begin{itemize}
\item[{\rm (i)}]  $H$ is $(\eps, 2\gamma , \gamma , 3 \gamma )$-superregular.
\item[{\rm (ii)}] Let $G':=G-H$. Then $d_{G'} (v) = (1 - \mu   \pm 4\gamma)m$ for all $v \in V(G)$.
\end{itemize}
\end{lemma}

\begin{proof}
Note that $\delta(G) \ge (1- \mu - \eps') m \ge (1- \eps^3) m$ as $\eps', \mu \ll \eps \ll 1$. 
Thus, whenever $A \subseteq U$ and $B \subseteq W$ are sets of size at least $\eps m,$\COMMENT{B\'ela: added comma} then
\begin{align}\label{exp1a}
	e_G(A,B) \ge (|B| - \eps^3 m)|A| \ge (1 - \eps^2)|A||B|.
\end{align}
Let $H$ be a random subgraph of $G$ which is obtained by including each edge of $G$ with probability $2\gamma $.
\eqref{exp1a} implies that whenever $A \subseteq U$ and $B \subseteq W$ are sets of size at least $\eps m$ then
\begin{align}\label{exp1}
2\gamma (1- \eps^2)|A||B| \leq \mathbb E (e_{H} (A,B)) \leq 2\gamma |A||B|.
\end{align}
Further, for all $u,u' \in V(H)$,
\begin{align}\label{exp2}
 \mathbb E (|N_{H} (u) \cap N_{H} (u')| ) \leq 4 \gamma ^2 m
\end{align}
and 
\begin{align}\label{exp3}
3\gamma m/2 \leq \mathbb E (\delta (H)),  \mathbb E (\Delta (H) )\leq 2 \gamma m .
\end{align}
Thus, \eqref{exp1}--\eqref{exp3} together with Proposition~\ref{chernoff} imply that, with high probability, $H$
is an $(\eps, 2\gamma , \gamma , 3 \gamma )$-superregular pair.
Since $\Delta (H) \leq 3 \gamma m$ by~(Reg3) and $\eps'\ll \gamma$, $G'$ satisfies~(ii).
\end{proof}

\subsection{Transforming a Balanced Extension into $1$-factors} \label{sec:extend}

The next lemma will be used to extend each locally balanced path sequence $PS$ belonging to a balanced extension $\mathcal{BE}$ into a
(directed) $1$-factor using edges from $G'$.
We will select the edges from $G'$ in such a way that (apart from the path sequences)%
   \COMMENT{Deryk: added brackets}
the $1$-factors obtained are edge-disjoint.

\begin{lemma}\label{extend1-factor}
Suppose that $0<1/m \ll  1/k \le \eps \ll\rho , 1/\ell \leq 1$, that $\rho \ll 1$,
that $0 \le \mu \le 1/4$ and that $q, m,k,\ell \in \mathbb{N}$
with $q \le (1- \mu - \rho) m$.
Let $(G,\mathcal{Q}, C)$ be a $(k,m, \mu, \eps)$-cyclic system, where $C= V_1 \dots V_k$.
Suppose that there exists a set $\mathcal{PS}$ of $q$ path sequences $PS_1,\dots, PS_q$ satisfying the following conditions:
\COMMENT{Allan: had $q$ edge-disjoint path sequences, also changed the last sentence in the statement.}
\begin{itemize}
	\item[\rm (i)] Each $PS_s \in \mathcal{PS}$ is  locally balanced with respect to $C$.
	\item[\rm (ii)] $|V(PS_s) \cap V_i |\le \eps m$ for all $i \le k $ and $s \le q$.
Moreover, for each $i \le k$, there are at most $\ell m/k$ $PS_s$ such that $V(PS_s) \cap V_i \ne \emptyset$.
\end{itemize}
Then there exist $q$ directed $1$-factors $F_1, \dots, F_q$ in $G+ \mathcal{PS}$ such that for all $s \le q$
$PS_s \subseteq F_s$ and $F_1 - PS_1, \dots, F_q - PS_q$ are edge-disjoint subgraphs of $G$.
\end{lemma}

\proof
By changing the values of $\rho$, $\mu$ and $\eps$ slightly,%
   \COMMENT{We do need to make $\eps$ slightly bigger for this as well (in order to ensure that $(G,\mathcal{Q}, C)$
is still a $(k,m, \mu, \eps)$-cyclic system)}
we may assume that $\rho m, \mu m \in \mathbb{N}$.
For each $s \le q$ and each $i \le k$, let $V_i^{s,-}$ (or $V_i^{s,+}$) be the set of vertices in $V_i$ with indegree (or outdegree) one in $PS_s$.
Since each $PS_s$ is locally balanced with respect to $C$, $|V_i^{s,+}| = |V_{i+1}^{s,-}| \le \eps m $ for all $s \le q$ and all $i \le k$
(where the inequality follows from~(ii)).
To prove the lemma, it suffices to show that for each $i \le k$, there exist edge-disjoint directed matchings $M^1_i, \dots, M^q_i$,
so that each $M^s_i$ is a perfect matching in $G[V_i \setminus V_i^{s,+},V_{i+1} \setminus V_{i+1}^{s,-}]$.
The lemma then follows by setting $F_s := PS_s + \sum_{i \le k} M^s_i$ for each $s \le q$.

Fix any $i \le k$.
Without loss of generality (by relabelling the $PS_s$ if necessary) we may assume that there exists an integer $s_0$ such that
$V_i^{s,+} \ne \emptyset$ for all $s \le s_0$ and $V_i^{s,+} = \emptyset $ for all $s_0 < s \le q$.
By~(ii), $s_0 \le \ell m/k$.
Suppose that for some $s$ with $1 \le s \le s_0$ we have already found 
our desired matchings $M^1_i, \dots, M^{s-1}_i$ in $G[V_i,V_{i+1}]$. 
Let 
$$
V'_i : = V_i \setminus V_i^{s,+}, \ \ \ V'_{i+1} : = V_{i+1} \setminus V_{i+1}^{s,-} \ \ \ \mbox{and} \ \ \
G_s := G[V_i',V_{i+1}'] - \sum_{s' < s} M^{s'}_{i}.
$$
Note that each $v \in V_i'$ satisfies
\begin{align} \nonumber
	d^+_{G_s}(v) \ge d^+_{G}(v,V_{i+1}) - ( |V_{i+1}^{s,-}|+s_0) 
 \ge (1- \mu - (2\eps + \ell/k) ) m 
\ge (1- \mu - \sqrt{\eps}) m  
\end{align}
by~(Sys2) and the fact that $ 1/k \le \eps \ll 1/{\ell}$.
Similarly, each $v \in V_{i+1}'$ satisfies $d^-_{G_s}(v) \ge (1- \mu - \sqrt{\eps} ) m$.
Thus $G_s$ contains a perfect matching $M_i^{s}$ (this follows, for example, from Hall's theorem).
So we can find edge-disjoint matchings $M^1_i, \dots, M^{s_0}_i$ in $G[V_i,V_{i+1}]$.

Let $G'$ be the subdigraph of $G[V_i,V_{i+1}]$ obtained by removing all the edges in $M^{1}_{i}, \dots, M^{s_0}_{i}$.
Since $V_i^{s,+} = \emptyset $ for all $s_0 < s \le q$ (and thus also $V_{i+1}^{s,-} = \emptyset$ for all such~$s$),%
	\COMMENT{Allan: had $V_{i}^{s,-}$ before.}
in order to prove the lemma it suffices to find $q-s_0$ edge-disjoint perfect matchings in $G'$.
Each $v \in V_i$ satisfies 
\begin{align*}
d^+_{G'}(v) = d^+_{G}(v,V_{i+1}) \pm  s_0 = d^+_{G} (v,V_{i+1}) \pm \ell m/k = (1- \mu \pm \sqrt{\eps}) m 
\end{align*}
by~(Sys2) and the fact that  $ 1/k \le \eps \ll 1/{\ell}$.
Similarly, each $v \in V_{i+1}$ satisfies $d^-_{G'}(v) = (1- \mu \pm \sqrt{\eps}) m$.
Set $\rho' : = \rho + s_0/m$. 
Note that $\rho' m \in \mathbb{N}$ and $\rho \le \rho' \le \rho + \ell/k \le 2\rho$ as $1/k \ll 1/\ell, \rho$.
Hence, $\eps \ll \rho' \ll 1$. Thus we can apply Lemma~\ref{regularsub} with $G', \rho', \sqrt{\eps}$
playing the roles of $\Gamma,\rho,\eps$ to obtain $(1-\mu - \rho') m $ edge-disjoint perfect matchings in~$G'$.
Since $(1-\mu - \rho') m  =  (1-\mu - \rho) m - s_0 \ge  q- s_0$, there exists $q-s_0$ edge-disjoint
perfect matchings $M_i^{s_0+1}, \dots, M_i^{q}$ in $G'$. This completes the proof of the lemma.
\endproof

\subsection{Merging Cycles to Obtain Hamilton Cycles} \label{sec:merge}

Recall that we have removed a sparse subdigraph $H$ from $G$ and that $G'=G-H$.
Our final step in the proof of Lemma~\ref{merging} is to merge the cycles from each of the $1$-factors $F_s$ returned by
Lemma~\ref{extend1-factor} to obtain edge-disjoint (directed) Hamilton cycles. 
We will apply Lemma~\ref{mergecycles} to merge the cycles of each $F_s$, using the edges in $H$.
However, the Hamilton cycles obtained in this way might not be consistent with the matching $M_s \in \mathcal{M}$
that lies in $PS_s$. Lemma~\ref{ordercycle} is designed to deal with this issue.

Lemma~\ref{mergecycles} was proved in~\cite{Kelly} and was first used to construct approximate Hamilton decompositions in~\cite{OS}.
Roughly speaking, it asserts the following:
suppose that we have a $1$-factor $F$ where most of the edges wind around a cycle $C=V_1\dots V_k$.
Suppose also that we have a digraph $H$ which winds around~$C$. (More precisely, $H$ is the
union of superregular pairs $H[V_i,V_{i+1}]$.)
Then we can transform $F$ into a Hamilton cycle $C'$ by using a few edges of $H$.
The crucial point is that when applying this lemma, the edges in $C'-F$ can be taken from a small number of the
superregular pairs $H[V_i,V_{i+1}]$ (i.e.~the set $J$ in Lemma~\ref{mergecycles} will be very small compared to~$k$).
In this way, we can transform many $1$-factors $F$ into edge-disjoint Hamilton cycles without using any of the pairs $H[V_i,V_{i+1}]$ too often.
This in turn means that we will be able to transform all of our $1$-factors into edge-disjoint Hamilton cycles
by using the edges of a single sparse graph~$H$.

\begin{lemma}\label{mergecycles}
Suppose that $0<1/m\ll d'\ll \eps\ll d\ll \zeta ,1/t\le 1/2$.
Let $V_1,\dots,V_k$ be pairwise disjoint clusters, each of size $m$, and let $C=V_1\dots V_k$ be a directed cycle on these clusters.
Let $H$ be a digraph on $V_1\cup \dots\cup V_k$ and let $J\subseteq E(C)$. For each edge $V_iV_{i+1}\in J$, let $V^1_i\subseteq V_i$
and $V^2_{i+1}\subseteq V_{i+1}$ be such that $|V^1_i|=|V^2_{i+1}|\ge m/100$ and 
 such that $H[V^1_i,V^2_{i+1}]$ is $(\eps,d',\zeta d',td'/d)$-superregular.
Suppose that $F$ is a $1$-regular digraph with $V_1\cup \dots \cup V_k\subseteq V(F)$ such that the following properties hold:
\begin{itemize}
\item[\rm{(i)}]For each edge $V_iV_{i+1}\in J$ the digraph $F[V^1_i,V^2_{i+1}]$ is a perfect matching.
\item[\rm{(ii)}] For each cycle $D$ in $F$ there is some edge $V_iV_{i+1}\in J$ such that $D$ contains a vertex
in $V^1_i$.
\item[\rm{(iii)}] Whenever $V_iV_{i+1}, V_jV_{j+1}\in J$ are such that $J$ avoids all edges in the segment $V_{i+1}CV_j$ of
$C$ from $V_{i+1}$ to $V_j$, then $F$ contains a path $P_{ij}$ joining some vertex $u_{i+1}\in V^2_{i+1}$ to some
vertex $u'_j\in V^1_j$ such that $P_{ij}$ winds around~$C$.
\end{itemize}
Then we can obtain a directed%
    \COMMENT{Daniela: added directed}
cycle on $V(F)$ from $F$ by replacing $F[V^1_i,V^2_{i+1}]$ with a suitable perfect matching
in $H[V^1_i,V^2_{i+1}]$ for each edge $V_iV_{i+1}\in J$.
\end{lemma}

\begin{lemma}\label{ordercycle}
Suppose that $0<1/m\ll \gamma \ll d' \ll \eps\ll d\ll \zeta ,1/t\le 1/2$.
Let $V_1,\dots,V_k$ be pairwise disjoint clusters, each of size $m$, and let $C=V_1\dots V_k$ be a directed cycle on these clusters.
Let $1\le i\le k$ be fixed and let $V^1_i\subseteq V_i$ and $V^2_{i+1}\subseteq V_{i+1}$ be such that $|V^1_i|=|V^2_{i+1}|\ge m/100$. Suppose that 
$H=H[V^1_i,V^2_{i+1}]$ is an $(\eps,d',\zeta d',td'/d)$-superregular bipartite digraph.
Let $X= \{x_1, \dots, x_p\} \subseteq V_{i}^1$ with $|X| \le \gamma m$.
Suppose that $C'$ is a directed cycle with $V_1\cup \dots \cup V_k\subseteq V(C')$ such that $C'[V^1_i,V^2_{i+1}]$ is a perfect matching.
Then we can obtain a directed cycle on $V(C')$ from $C'$ that visits the vertices $x_1, \dots, x_p$ in order by
replacing $C'[V^1_i,V^2_{i+1}]$ with a suitable perfect matching in $H[V_i^1,V_{i+1}^2]$.
\end{lemma}
\proof
Pick $\nu$ and $\tau$ such that $ \gamma \ll \nu \ll \tau\ll d'$.
For every $u \in V^1_{i}$, starting at $u$ we move along 
the cycle $C'$ (but in the opposite direction to the orientation of the edges) and let $f(u)$ be the first vertex on $C'$ in
$V^2_{i+1}$. (Note that $f(u)$ exists since $C'[V^1_i,V^2_{i+1}]$ is a perfect matching.
Moreover,
$f(u) \not = f(v)$ if $u \not = v$.)
Define an auxiliary digraph $A$ on $V^2_{i+1}$ such that $N^+_A(f(u)):=N^+_{H}(u)$.
So $A$ is obtained by identifying each pair $(u,f(u))$ into one vertex with an edge from $(u,f(u))$ to
$(v,f(v))$ if $H$ has an edge from $u$ to $f(v)$. So Lemma~\ref{regtoexpander} applied
with $d'$, $d/t$ playing the roles of $d$, $\mu$ implies
that $A$ is a robust $(\nu,\tau)$-outexpander.
Moreover, $\delta^+(A), \delta^-(A)\ge\zeta d'|V^2_{i+1}|=\zeta d'|A|$ by (Reg4).%
    \COMMENT{Actually $A$ might have loops. But after deleting them we still have a robust $(\nu,\tau)$-outexpander
with $\delta^0(A)\ge \zeta d'|A|-1$. So it's maybe better to gloss over it...}
Thus Theorem~\ref{expanderthm} implies that $A$ has a Hamilton cycle visiting $f(x_1), \dots, f(x_p)$ in order, which clearly corresponds to a perfect matching $M$ in~$H$ with the desired property.
\endproof

The above proof idea is actually quite similar to that for Lemma~\ref{mergecycles} itself.
We now apply Lemmas~\ref{mergecycles} and~\ref{ordercycle} to each $1$-factor $F_s$ given by Lemma~\ref{extend1-factor}
and obtain edge-disjoint Hamilton cycles that are consistent with the~$M_s$.

\begin{lemma}\label{mergecycles2}
Suppose that $0<1/m \ll  \eps_0, 1/k \ll \gamma  \ll \eps   \ll 1$, that $\gamma \ll 1/ \ell \le 1$ and that $q, m,k,\ell \in \mathbb{N}$.%
      \COMMENT{There is actually no condition on $q$ other than that fact that $q \le \ell m$ which is due to balanced extension.}
Let $\mathcal{Q}=\{V_1,\dots,V_k\}$ be a $(k,m)$-equipartition of a vertex set $V$ and let $C=V_1\dots V_k$ be a directed cycle.
Let $\mathcal{M} = \{M_1, \dots, M_q\}$ be a set of ordered directed matchings.
Suppose that $\mathcal{BE}=\{PS_1,\dots,PS_q\}$ is a balanced extension of $\mathcal{M}$ with  respect
to $(\mathcal{Q},C)$ and parameters $(\eps_0, \ell)$.%
    \COMMENT{Daniela: deleted "such that $M_s \subseteq PS_s$ for all $s \le q$" since it follows from (BE2)}
Furthermore, suppose that there exist $1$-regular digraphs $F_1, \dots, F_q $ on $V$ such that for each $s \le q$, $PS_s \subseteq F_s$ and such that
$F_s - PS_s$ winds around~$C$. 
Let $H$ be a digraph on $V$ which is edge-disjoint from each of $F_1 - PS_1, \dots, F_q - PS_q$%
		\COMMENT{Allan: We do not need $F_1-PS_1, \dots, F_s-PS_s$ are edge-disjoint.}
and such that $H[V_{i},V_{i+1}]$ is
$(\eps,2\gamma ,\gamma   , 3 \gamma)$-superregular for all $i\le k$. 
Then there exist $q$ Hamilton cycles $C_1, \dots, C_q$  in
$F_1+ \dots+F_q+H$ such that $C_s$ contains $PS_s$ and is consistent with $M_s$ for all $s \le q$ and such that
$C_1 - F_1, \dots, C_q - F_q$ are edge-disjoint subgraphs of $H$.%
	\COMMENT{Allan: I do mean $C_1 - F_1, \dots, C_q - F_q$ are pairwise edge-disjoint.}
\end{lemma}
\begin{proof}
Recall from (BE2) that for each $s \le q$ there is some $i_s\le k$ such that $PS_s$ is a $V_{i_s}$-extension of $M_s$.
In particular, $M_s\subseteq PS_s$. Let $I_s$ be the set consisting of all $i \le k$ such that $V_i \cap V(PS_s) \ne \emptyset$.
Since $\mathcal{BE}$ is a balanced extension with parameters $(\eps_0, \ell)$, (BE3) implies that%
    \COMMENT{Deryk: deleted $|\{s : i_s = i\}| \le \ell m/k$ in the display below since it is not used}
for every $i \le k$ we have
\begin{align} \label{sbound}
|\{s : i \in I_s\}| \le \ell m/k.
\end{align}
For each $s\le q$ in turn, we are going to show that there exist Hamilton cycles $C_{1}, \dots, C_{s}$ in $F_1 + \dots + F_s+H$ such that 
\begin{itemize}
	\item[\rm(a$_s$)] $PS_{s'}\subseteq C_{s'}$ and $C_{s'}$ is consistent with $M_{s'}$ for all $s' \le s$,%
   \COMMENT{Daniela: added $PS_{s'}\subseteq C_{s'}$}
	\item[\rm(b$_s$)] $E(C_{s'} - F_{s'}) \subseteq \bigcup_{i \in I_{s'}} E(H[V_i,V_{i+1}])  $ for all $s' \le s$,
	\item[\rm(c$_s$)] $C_1 - F_1, \dots, C_{s}-F_s$ are pairwise edge-disjoint.%
	\COMMENT{Allan: had $C_1, \dots, C_s$}
\end{itemize}
So suppose that for some $s$ with $1\le s\le q$ we have already constructed $C_1, \dots, C_{s-1}$. We now construct $C_s$ as follows. 
Let $H_s:=H - \sum_{s' < s}(C_{s'} -F_{s'})$.%
	\COMMENT{Allan: had $H_s:=H-(C_1+\dots+C_{s-1})$}
Define a new constant $d$ such that $\eps \ll d \ll 1$.

Our first task is to apply Lemma~\ref{mergecycles} to $F_s$ to merge all the cycles in $F_s$ into a Hamilton cycle using only edges of $H_s$.
For each $i \in I_s$, let $V_i^{-}$  be the set of vertices in $V_i$ with indegree one in $PS_s$
and let $V_i^{+}$  be the set of vertices in $V_i$ with outdegree one in $PS_s$.
Set $V_i^1 := V_i \setminus V_i^{+}$ and $V_{i+1}^2 := V_{i+1} \setminus V_{i+1}^{-}$.
Since $PS_s$ is locally balanced, $|V_i^{+}| = |V_{i+1}^{-}| \le \eps_0 m $ for all $i \in I_s$ (where the inequality holds by~(BE3)).
By~(b$_{s-1}$) and \eqref{sbound}, $H_s[V_i,V_{i+1}]$ is obtained from $H[V_i,V_{i+1}]$ by removing at most
$|\{s' <s : i \in I_{s'}\}| \le \ell m /k \le \eps^2\gamma m$ edges from each vertex (as $1/k \ll \eps,\gamma, 1/\ell$).
So by Proposition~\ref{superslice5}, $H_s [V_i, V_{i+1}]$ is still $(2\eps,2\gamma ,\gamma/2 , 3 \gamma)$-superregular for each $i \in I_s$.
Recall that $|V_i \setminus V_i^1| = |V_{i+1} \setminus V_{i+1}^2| \le \eps_0 m$.
Hence $H_s [V_i^1, V_{i+1}^2]$ is $(4\eps,2\gamma ,\gamma/4 , 6 \gamma)$-superregular by Proposition~\ref{superslice6}
and thus also $(4\eps,2\gamma ,\gamma/4 , 4\gamma/d)$-superregular.

Let $E_s := \{V_iV_{i+1} : i \in I_s \}$.
Our aim is to apply Lemma~\ref{mergecycles} with $F_s$, $E_s$, $H_s$, $4 \eps$, $2\gamma$, $2$, $1/8$
playing%
    \COMMENT{Need $1/8$ instead of $1/4$ since the density is $2\gamma$}
the roles of $F$, $J$, $H$, $\eps$, $d'$, $t$, $\zeta$.
Our assumption that $F_s- PS_s$ winds around $C$ implies that for each $i\in I_s$, $F_s[V_i^1,V_{i+1}^2]$ is a perfect matching.
So Lemma~\ref{mergecycles}(i) holds. Note that every final vertex of a nontrivial%
    \COMMENT{Daniela: added nontrivial}
path in $PS_s$ must lie in $\bigcup_{i \in I_s} V_i^1$,
implying Lemma~\ref{mergecycles}(ii).%
   \COMMENT{This implies Lemma~\ref{mergecycles}(ii) for cycles $D$ of $F$ which contain some nontrivial path in $PS_s$. The other cycles of $F$
must wind around $C$ and so satisfy Lemma~\ref{mergecycles}(ii) as well.}
Finally, recall that $|V_i^1|, |V_{i+1}^2| \ge (1- \eps_0) m$ for all $i \in I_s$.
Together with our assumption that $F_s - PS_s$ winds around $C$, this easily implies Lemma~\ref{mergecycles}(iii).
So we can apply Lemma~\ref{mergecycles} to obtain a Hamilton cycle $C'_s$ which is constructed from $F_s$ by replacing
$F_s[V^1_i,V^2_{i+1}]$ with a suitable perfect matching in $H_s[V^1_i,V^2_{i+1}]$ for each $i\in I_s$. In particular,
$PS_s\subseteq C'_s$.

Let $H'_s:=H_s - (C'_{s} -F_{s})$.%
	\COMMENT{Allan: had $H'_s:=H_s-C'_{s}$}
Recall that $M_s$ is an ordered directed matching, say $M_s = \{ e_1, \dots, e_r\}$, and that $PS_s$ is a $V_{i_s}$-extension of $M_s$.
For each $j \le r$, let $P_j$ be the path in $PS_s$ containing $e_j$
and let $x_j$ denote the final vertex of $P_j$. Hence $x_1, \dots, x_r$ are distinct and lie in $V_{i_s}^1$.
Together with (BE3) this implies that $r \le \eps_0 m$.
Note that $H'_s [ V_{i_s}^1,V_{i_s+1}^2]$ is obtained from $H_s[ V_{i_s}^1,V_{i_s+1}^2 ]$ by removing a perfect matching, namely $C'_{s}[V_{i_s}^1, V_{i_s+1}^2]$.%
	\COMMENT{Allan: added $C'_{s}[V_{i_s}^1, V_{i_s+1}^2]$ since Lemma~\ref{ordercycle} needs $C'_{s}[V_{i_s}^1, V_{i_s+1}^2]$ is a perfect matching.}
So by Proposition~\ref{superslice5}, $H_s' [V_{i_s}^1, V_{i_s+1}^2]$ is still $(8\eps,2\gamma ,\gamma/8 , 4\gamma/d)$-superregular.
Apply Lemma~\ref{ordercycle} with
$C'_s$, $i_s$, $H_s'[V^1_{i_s}, V^2_{i_s+1}]$, $\eps_0$, $8\eps$, $2\gamma$, $2$, $1/16$ playing the roles of
$C'$, $i$, $H$, $\gamma$, $\eps$, $d'$, $t$, $\zeta$ to obtain a Hamilton cycle $C_s$ which visits $x_1,\dots, x_r$ in this order
and is constructed from $C'_s$ by replacing the perfect matching $C'_s[V_{i_s}^1,V_{i_s+1}^2]$%
    \COMMENT{Daniela: previously had $C'[V_{i_s},V_{i_s+1}]$}
with a suitable perfect matching in $H'_s[V_{i_s}^1,V_{i_s+1}^2]$.%
    \COMMENT{Daniela: previously had $H'_s[V_{i_s},V_{i_s+1}]$}
In particular, $PS_s\subseteq C_s$.

Note that $E(C_s - F_s) \subseteq  \bigcup_{i \in I_{s}} E(H_s[V_i,V_{i+1}])$, so (b$_s$) and (c$_s$) hold.
Since $PS_s \subseteq C_s$ and $x_j$ is the final vertex of $P_j$ and since $e_j\in E(P_j)$, it follows that $C_s$ visits
the edges $e_1, \dots,e_r$ in order. So $C_s$ is consistent with $M_s$, implying (a$_s$). 
\end{proof}

\removelastskip\penalty55\medskip\noindent{\bf Proof of Lemma~\ref{merging}.}%
	\COMMENT{Allan: the proof changed a bit..}
Let $\mathcal{Q}=\{V_1,\dots,V_k\}$. By relabeling the $V_i$ if necessary, we may assume that $C=V_1\dots V_k$.
Define new constants $\gamma$ and $\eps$ such that $\eps_0, \eps',1/k \ll \gamma \ll \eps,\rho , 1 /{\ell} $ and $\mu \ll \eps\ll 1$.%
    \COMMENT{Daniela: hierarchy and some of the constants below changed because of the modifications to Lemma~\ref{slicelemmazz}} 
For each $i \le k$ we apply Lemma~\ref{slicelemmazz} to (the underlying undirected graph of) $G[V_i,V_{i+1}]$ in order to obtain
a spanning subdigraph $H$ of $G$ which satisfies the following properties: 
\begin{itemize}
\item[{\rm (i$'$)}] For each $i \le k$, $H[V_i, V_{i+1}]$ is $(\eps, 2\gamma , \gamma , 3 \gamma )$-superregular.
\item[{\rm (ii$'$)}] Let $G':=G-H$. Then $(G', \mathcal{Q}, C)$ is a $(k,m, \mu,4\gamma)$-cyclic system.
\end{itemize}
Indeed, (ii$'$) follows easily from Lemma~\ref{slicelemmazz}(ii) and the definition of a $(k,m,\mu,4\gamma)$-cyclic system.
Recall that $\mathcal{BE} = \{PS_1, \dots, PS_q\}$ with $M_s \subseteq PS_s$ for all $s \le q$. Our next aim is to
apply Lemma~\ref{extend1-factor} with $G'$, $\mathcal{BE}$, $4\gamma$ playing the roles of $G$, $\mathcal{PS}$, $\eps$
to obtain $1$-factors $F_s$ extending the $PS_s$.%
    \COMMENT{Deryk: last bit is new}
Note that (BE1) and (BE3) imply that conditions~(i) and~(ii) of Lemma~\ref{extend1-factor} hold.
So we can apply Lemma~\ref{extend1-factor} to obtain $q$ (directed) $1$-factors $F_1, \dots, F_q$ in $G'+\mathcal{BE}$%
   \COMMENT{Daniela: replaced on $V$ by in $G'+\mathcal{BE}$}
such that $PS_s \subseteq F_s$ for all $s \le q$ and $F_1 - PS_1, \dots, F_q-PS_q$ are edge-disjoint subgraphs of $G'$.
Recall from (ii$'$) and (Sys2) that $G'$ (and thus also $F_s - PS_s$) winds around~$C$.
So we can apply Lemma~\ref{mergecycles2} to obtain $q$ Hamilton cycles $C_1, \dots, C_q$ in
$F_1+ \dots+F_q+H$ such that $C_s$ contains $PS_s$ and is consistent with $M_s$ for all $s \le q$, and such that $C_1 - F_1, \dots, C_q-F_q$ are edge-disjoint subgraphs of $H$.
Since $H$ and $G'$%
    \COMMENT{Daniela: had $H$ and $H'$ instead of $H$ and $G'$}
are edge-disjoint, $C_1 - PS_1, \dots, C_q-PS_q$ are edge-disjoint subgraphs of $G$.
\endproof

We can now put everything together to prove the approximate decomposition lemma in the two cliques case.
First we apply Lemma~\ref{sysdecom} to obtain cyclic systems and sparse subgraphs $H_{A,j}$ and $H_{B,j}$.
Then we apply Lemma~\ref{balanceextension} to balance out the exceptional systems into balanced extensions.
Next, we apply Lemma~\ref{merging} to $A$ and $B$ separately to extend the balanced extensions into Hamilton cycles.%
    \COMMENT{Deryk: added blabla}

\removelastskip\penalty55\medskip\noindent{\bf Proof of Lemma~\ref{almostthm}. }
Apply Lemma~\ref{sysdecom} to $G, \mathcal{P}$ and $\mathcal{J}$ to obtain  (for each $1 \le j \le (K-1)/2$)  pairs of 
tuples $(G_{A,j}, \mathcal{Q}_{A}, C_{A,j}, H_{A,j}, \mathcal{J}^*_{A,j})$ and $(G_{B,j}, \mathcal{Q}_{B}, C_{B,j},$ $ H_{B,j}, \mathcal{J}^*_{B,j})$
which satisfy (a$_1$)--(a$_7$).
Fix $j \le (K-1)/2$. Write $\mathcal{J}^*_{A,j} = \{J^*_{A, {\rm dir},1}, \dots, $ $J^*_{A, {\rm dir},q}\}$, where%
    \COMMENT{Deryk: defined $q$ and used it below}
\begin{equation}\label{eq:q}
q:=|\mathcal{J}^*_{A,j}|\le (1-4\mu-3\rho)m
\end{equation}
by~(a$_3$).
We now apply Lemma~\ref{balanceextension} with $\mathcal{J}^*_{A,j},\mathcal{Q}_A,C_{A,j}, H_{A,j}, K, 5K \sqrt{\eps_{0}}$ playing
the roles of $\mathcal{M}, \mathcal{Q}, C, H, k,\eps$ to obtain an orientation $H_{A,j,{\rm dir}}$ of $H_{A,j}$ and
a balanced extension $\mathcal{BE}_j$ of $\mathcal{J}^*_{A,j}$ with respect to $(\mathcal{Q}_A, C_{A,j})$ and parameters $(10K \sqrt{\eps_{0}},3)$.
(Note that (a$_3$) and (a$_5$) imply conditions (i) and (ii) of Lemma~\ref{balanceextension}.)
Write $\mathcal{BE}_j :=  \{PS_{1}, \dots, PS_q \}$ such that $J^*_{A, {\rm dir},s} \subseteq PS_s$ for all $s \le q$.
So (BE1) implies that%
   \COMMENT{Daniela: reworded}
$PS_{1} - J^*_{A, {\rm dir},1}, \dots, PS_q - J^*_{A, {\rm dir}, q}$ are edge-disjoint subgraphs of $H_{A,j,{\rm dir}}$.%
	\COMMENT{Allan:added more explanations}
Since $(G_{A,j, {\rm dir}}, \mathcal{Q}_{A}, C_{A,j})$ is a $(K,m, 4 \mu, 5/K)$-cyclic system by (a$_6$), (\ref{eq:q}) implies that%
    \COMMENT{Deryk: added ref to (\ref{eq:q})}
we can apply Lemma~\ref{merging} as follows:
\smallskip
\begin{center}
  \begin{tabular}{ r | c | c | c | c | c | c | c | c | c | c |c}

& $G_{A,j, {\rm dir}}$ & $\mathcal{Q}_{A}$ & $C_{A,j}$ &  $K$ & $\mathcal{J}^*_{A,j}$ & $q$ & $4 \mu$ & $3\rho$ & $10K \sqrt{\eps_{0}}$ & $5/K$ & $3$ \\ \hline
plays role of & $G$ & $\mathcal{Q}$ & $C$ & $k$ &  $\mathcal{M}$ & $q$ & $\mu$  & $\rho$ &  $\eps_0$ & $\eps'$ &  $\ell$ \\
  \end{tabular}
\end{center}
\smallskip
\noindent
In this way we obtain $q$ directed Hamilton cycles $C'_{A,j,1}, \dots, C'_{A,j,q}$
in $G_{A,j, {\rm dir}} + \mathcal{BE}_j$ such that $C'_{A,j,s}$ contains $PS_{s}$ and is consistent with $J^*_{A, {\rm dir},s}$ for all $s \le q$.
Moreover, $C'_{1} - J^*_{A, {\rm dir},1}, \dots, C'_q - J^*_{A, {\rm dir}, q}$ are edge-disjoint subgraphs of $G_{A,j,{\rm dir}} + H_{A,j,{\rm dir}}$.%
	\COMMENT{Allan: added details here, the rest of the proof remains the same.}
Repeat this process for all $j \le (K-1)/2$.

Write $\mathcal{J} = \{ J_1, \dots, J_{|\mathcal{J}|}\}$.
Recall from (a$_2$) that the $\mathcal{J}^*_{A,1}, \dots, \mathcal{J}^*_{A,(K-1)/2}$ partition $\{J_{A,{\rm dir}}^* :J \in \mathcal{J}\}$.
Therefore, we have obtained $|\mathcal{J}|$ directed Hamilton cycles%
    \COMMENT{Daniela: had "edge-disjoint directed Hamilton cycles" before}
$C'_{A,1}, \dots, C'_{A,|\mathcal{J}|}$ on vertex set $A$.
Moreover, by relabelling the $J_s$ if necessary, we may assume that $C'_{A,s}$ is consistent with $(J_s)^*_{A, {\rm dir}}$ for all $s \le |\mathcal{J}|$.
Furthermore, (a$_4$) implies that the undirected versions of $C'_{A,1} - (J_1)^*_{A, {\rm dir}}, \dots, C'_{A,|\mathcal{J}|} - (J_{|\mathcal{J}|})^*_{A, {\rm dir}}$
are edge-disjoint spanning subgraphs of $G[A]$.

Similarly we obtain directed Hamilton cycles%
    \COMMENT{Daniela: had "edge-disjoint directed Hamilton cycles" before}
$C'_{B,1}, \dots, C'_{B,|\mathcal{J}|}$ on vertex set $B$
so that $(J_s)^*_{B, {\rm dir}} \subseteq C'_{B,s}$ for all $s \le |\mathcal{J}|$.
Let $H_s$ be the undirected graph obtained from $C'_{A,s} + C'_{B,s}-J_s^* +J_s$ by ignoring all the orientations of the edges.
Recall that $J_1, \dots, J_{|\mathcal{J}|}$%
\COMMENT{Andy: replaced subscript $s$ with subscript $|\mathcal{J}|$.}
are edge-disjoint exceptional systems and that they are edge-disjoint from the $C'_{A,s} + C'_{B,s}-J_s^*$
by (EC3).%
    \COMMENT{Deryk: added detail}
So $H_1, \dots, H_{|\mathcal{J}|}$ are edge-disjoint spanning subgraphs of $G$.
Finally, Proposition~\ref{prop:ES2} implies that $H_1, \dots, H_{|\mathcal{J}|}$ are indeed as desired in Lemma~\ref{almostthm}.
\endproof


\section{The Bipartite Case}
Roughly speaking, the idea in this case is to reduce the problem of finding the desired edge-disjoint Hamilton cycles in~$G$
to that of finding suitable Hamilton cycles in an almost complete balanced bipartite graph.
This is achieved by considering the graphs $J^*_{\rm dir}$, whose definition we recall in the next subsection.
The main steps are similar to those in the proof of Lemma~\ref{almostthm}
(in fact, we re-use several of the lemmas, in particular Lemma~\ref{merging}).

We will construct the graphs $J^*_{\rm dir}$, which are based on balanced exceptional systems~$J$, in Section~\ref{sec:J*2bip}.
In Section~\ref{systembip} we describe a decomposition of~$G$ into blown-up Hamilton cycles.
We will construct balanced extensions in Section~\ref{sec:BEbip}
(this is more difficult than in the two cliques case).
Finally, we obtain the desired Hamilton cycles using Lemma~\ref{merging} (in the same way as in the two cliques case).

\subsection{Defining the Graphs $J^*_{\rm dir}$ for the Bipartite Case} \label{sec:J*2bip}
In this section we recall a number of definitions from Section~\ref{BESstar}.
Let $\mathcal{P}$ be a $(K,m,\eps)$-partition of a vertex set $V$ and
let $J$ be a balanced exceptional system with respect to~$\mathcal{P}$.
	   \COMMENT{Allan: I have rephrased the definition/construction of $J^*$. $J^*$ is still the same as before.
The construction is similar to the two cliques case.  So check whether things
are ok now. Note we no longer have conditions (BES$'$1), (BES$'$2), (BES$^*$1),(BES$^*$2) but they don't appear anywhere here (and only once in paper 2).}
Since each maximal path in $J$ has endpoints in $A \cup B$ and internal vertices in $V_0$ by (BES1), a balanced exceptional system $J$ naturally induces a matching $J^*_{AB}$ on $A \cup B$.
More precisely, if $P_1, \dots ,P_{\ell'}$ are the non-trivial paths in~$J$ and $x_i, y_i$ are the endpoints of $P_i$, then we define $J^*_{AB} := \{x_iy_i : i  \le \ell'\}$. 
Thus $J^*_{AB}$ is a matching by~(BES1) and $e(J^*_{AB}) \le e(J)$.%
    \COMMENT{Daniela: deleted  by~(BES4)}
Moreover, $J^*_{AB}$ and $E(J)$ cover exactly the same vertices in $A$. 
Similarly, they cover exactly the same vertices in $B$. 
So (BES3) implies that $e(J^*_{AB}[A])=e(J^*_{AB}[B])$.
We can write $E(J^*_{AB}[A])=\{x_1x_2, \dots, x_{2s-1}x_{2s}\}$,
$E(J^*_{AB}[B])=\{y_1y_2, \dots, y_{2s-1}y_{2s}\}$ and $E(J^*_{AB}[A,B])=\{x_{2s+1}y_{2s+1}, \dots, x_{s'}y_{s'}\}$, where $x_i \in A$ and $y_i \in B$.
Define $J^*:= \{ x_iy_i : 1 \le i \le s' \}$.
Note that 
\begin{align}
	e(J^*) =  e(J^*_{AB}) \le e(J). \label{BESeq}
\end{align}
As before, all edges of $J^*$ are called \emph{fictive edges}.%
   \COMMENT{Daniela: deleted "Similarly as in the two cliques case, we regard $J^*$ as being edge-disjoint from the original graph~$G$."}
Recall that an (undirected) cycle $D$ is \emph{consistent with $J^*$} if $D$ contains $J^*$ and (there is an orientation of $D$ which)
visits the vertices $x_1,y_1,x_2,\dots,y_{s'-1},x_{s'},y_{s'}$ in this order.

We will  need a directed version of Proposition~\ref{CES-H}(ii). This directed version immediately follows from  Proposition~\ref{CES-H}(ii)
and is similar to Proposition~\ref{prop:CEPSbiparite}.
For this, define $J^*_{\rm dir}$ to be the ordered directed matching $\{f_1, \dots, f_{s'}\}$ such that $f_i$ is a
directed edge from $x_i$ to $y_i$ for all $i \le s'$. So $J^*_{\rm dir}$ consists only of $AB$-edges.%
    \COMMENT{Deryk: added new sentence}
Similarly to the undirected case, we say that a directed cycle
$D_{\rm dir}$ is \emph{consistent with $J^*_{\rm dir}$}
if $D_{\rm dir}$ contains $J^*_{\rm dir}$ and visits the edges $f_1,\dots,f_{s'}$ in this order.

\begin{prop}\label{CES-H2}
Let $\cP$ be a $(K,m,\eps)$-partition of a vertex set $V$.
Let $G$ be a graph on $V$ and let $J$ be a balanced exceptional system with respect to~$\cP$ such that $J\subseteq G$.
Suppose that $D_{\rm dir}$ is a directed Hamilton cycle on $A \cup B$ such that $D_{\rm dir}$ is consistent with $J^*_{\rm dir}$.
Furthermore, suppose that $D - J^*  \subseteq G$, where $D$ is the cycle obtained from $D_{\rm dir}$ after ignoring the directions of all edges. 
Then $D-J^*+J$ is a Hamilton cycle of $G$.
\end{prop}

\subsection{Finding Systems} \label{systembip}

The following lemma gives a decomposition of an almost complete bipartite graph~$G$ into 
blown-up Hamilton cycles (together with an associated decomposition of exceptional systems).
Its proof is almost the same as that of Lemma~\ref{sysdecom}, so we omit it here.
The only difference is that instead of Walecki's theorem we use a result of Auerbach and Laskar~\cite{auerbach} to decompose the complete 
bipartite graph $K_{K,K}$ into Hamilton cycles, where $K$ is even.%
	\COMMENT{The proof is in an appendix at the end of the file.}
\begin{lemma}\label{sysdecombip}
Suppose that $0<1/n \ll \eps_0  \ll 1/K \ll \rho  \ll 1$ and $0 \leq \mu \ll 1$,
where $n,K \in \mathbb N$ and $K$ is even.
Suppose that $G$ is a graph on $n$ vertices and $\mathcal{P}=\{A_0,A_1,\dots,A_K,B_0,B_1,\dots,B_K\}$ is a $(K, m, \eps _0)$-partition of $V(G)$.
Furthermore, suppose that the following conditions hold:%
\COMMENT{Note that we do not need to assume Lemma~\ref{almostthmbip}(d)}
\begin{itemize}
	\item[{\rm (a)}] $d(v,B_i) = (1 - 4 \mu \pm 4 /K) m $ and $d(w,A_i) = (1 - 4 \mu \pm 4 /K) m $ for all
	$v \in A$, $w \in B$ and $1\leq i \leq K$.
	\item[{\rm (b)}] There is a set $\mathcal J$ which consists of at most $(1/4-\mu - \rho)n$ edge-disjoint  exceptional systems with parameter $\eps_0$ in~$G$.
	\item[{\rm (c)}] $\mathcal J$ has a partition into $K^4$ sets $\mathcal J_{i_1,i_2,i_3,i_4}$ (one for all $1\le \I \le K$) such that each $\mathcal J_{\I}$ consists of precisely $|\mathcal J|/{K^4}$ $\i$-BES with respect to~$\cP$.
\end{itemize}
Then for each $1 \le j \le K/2$, there is a tuple $(G_{j}, \mathcal{Q}, C_{j}, H_{j}, \mathcal{J}_{j})$
such that the following assertions hold, where $\mathcal{Q}:=\{A_1, \dots, A_K,B_1, \dots, B_K \}$:
\begin{itemize}
\item[{\rm (a$_1$)}] Each of $C_{1}, \dots, C_{K/2}$ is a directed Hamilton cycle on $\mathcal{Q}$ 
such that the undirected versions of these cycles form a Hamilton decomposition of the complete bipartite graph
whose vertex classes are $\{A_1, \dots, A_K \}$ and $\{B_1, \dots,$ $ B_K \}$.

\item[{\rm (a$_2$)}] $\mathcal{J}_{1}, \dots, \mathcal{J}_{K/2}$ is a partition of $\mathcal{J}$.

\item[{\rm (a$_3$)}] Each $\mathcal{J}_{j}$ has a partition into $K^4$ sets $\mathcal J_{j,\I}$ (one for all $1\le \I$ $\le K$)
such that $\mathcal J_{j,\I}$ consists of $\i$-BES with respect to~$\cP$ and $|\mathcal J_{j,\I}| \le (1- 4\mu - 3\rho) m/K^4$.

\item[{\rm (a$_4$)}] $G_{1},\dots, G_{K/2}, H_{1},\dots,  H_{K/2}$ are edge-disjoint subgraphs of $G[A,B]$.

\item[{\rm (a$_5$)}] $H_{j}[A_{i},B_{i'}]$ is a $(11K + 248/K) \eps_0 m$-regular graph for all $j \le K/2$ and all $i,i' \le K$.%
\COMMENT{We have omitted the floor/ceiling on $(11K + 248/K) \eps_0 m$.}

\item[{\rm (a$_6$)}] For each $j \le K/2$, there exists an orientation $G_{j,{\rm dir}}$ of $G_{j}$ such that $(G_{j,{\rm dir}}, \mathcal{Q}, C_{j})$ is a $(2K,m, 4\mu, 5/K)$-cyclic system.

\end{itemize}
\end{lemma}

\subsection{Constructing Balanced Extensions} \label{sec:BEbip}
Let $\cP=\{A_0,A_1,\dots,A_K,B_0,$ $B_1,\dots,B_K\}$ be a $(K,m,\eps)$-partition of a vertex set $V$, let $\mathcal{Q} := \{A_1$,$\dots$,$A_K$,$B_1$,$\dots$, $B_K \}$
and let $C= A_1B_1A_2B_2 \dots A_KB_K$ be a directed cycle.
Given a set $\mathcal{J}$ of balanced exceptional systems with respect to $\mathcal{P}$, we write $\mathcal{J}^*_{\rm dir}:=\{J^*_{\rm dir}:J\in\mathcal{J}\}$.
So $\mathcal{J}^*_{\rm dir}$ is a set of ordered directed matchings and thus it makes sense to construct a balanced extension of $\mathcal{J}^*_{\rm dir}$
with respect to $(\mathcal{Q},C)$. (Recall that balanced extensions were defined in Section~\ref{system}.)

Now consider any of the tuples $(G_{j}, \mathcal{Q}, C_{j}, H_{j}, \mathcal{J}_{j})$ guaranteed by Lemma~\ref{sysdecombip}.
We will apply the following lemma to find a balanced extension of $(\mathcal{J}_j)^*_{{\rm dir}}$ with respect to $(\mathcal{Q},C_{j})$, using edges of $H_{j}$
(after a suitable orientation of these edges). So the lemma is a bipartite analogue of Lemma~\ref{balanceextension}.
However, the proof is more involved than in the two cliques case.

\begin{lemma}\label{balanceextensionbip}
Suppose that $0<1/n \ll \eps  \ll 1/K \ll 1$,
where $n,K \in \mathbb N$.
Let $\cP=\{A_0,A_1,\dots,A_K,B_0,B_1,\dots,B_K\}$ be a $(K,m,\eps)$-partition of a set $V$ of $n$ vertices.
Let $\mathcal{Q} := \{A_1,\dots,A_K,B_1,\dots,B_K \}$ and let $C:= A_1B_1A_2B_2 \dots A_KB_K$ be a directed cycle. 
Suppose that there exist a set $\mathcal{J}$ of edge-disjoint balanced exceptional systems with respect to $\mathcal{P}$ and parameter $\eps$ and a graph $H$ such that the following conditions hold:
\begin{itemize}
\item[{\rm (i)}]  $\mathcal{J}$ can be partitioned into $K^4$ sets $\mathcal J_{\I}$ (one for all $1\le \I\le K$) such that $\mathcal J_{\I}$
consists of $\i$-BES with respect to~$\cP$ and $|\mathcal J_{\I}| \le m/K^4$.
\item[{\rm (ii)}] For each $v \in A \cup B$ the number of all those $J \in \mathcal{J}$ for which $v$ is incident to an edge in $J$ is at most $2 \eps n $.
\item[{\rm (iii)}] $H[A_{i},B_{i'}]$ is a $(11K+ 248/K)\eps m$-regular graph for all $i,i' \le K$.
\end{itemize}
Then there exist an orientation $H_{{\rm dir}}$ of $H$ and a balanced extension $\mathcal{BE}$ of $\mathcal{J}^*_{\rm dir}$ with respect to $(\mathcal{Q},C)$
and parameters $( 12 \eps K , 12)$ such that each path sequence in $\mathcal{BE}$ is obtained from some $J^*_{\rm dir}\in \mathcal{J}^*_{\rm dir}$
by adding edges of $H_{{\rm dir}}$.
\end{lemma}

The proof proceeds roughly as follows. Consider any $J \in \mathcal{J}_{\I}$.
We extend $J^*_{\rm dir}$ into a locally balanced path sequence in two steps.
For this, recall that $J^*_{\rm dir}$ consists only of edges from $A_{i_1} \cup A_{i_2}$ to $B_{i_3} \cup B_{i_4}$.%
    \COMMENT{Daniela: reworded}
In the first step, we construct a path sequence $PS$ that is an $A_{i_1}$-extension of $J^*_{\rm dir}$ by adding suitable $B_{i_3}A_{i_1}$- and $B_{i_4}A_{i_1}$-edges
from~$H$ to $J^*_{\rm dir}$. In the second step, we locally balance $PS$ in such a way that (BE1)--(BE3) are satisfied.%
    \COMMENT{Deryk: reworded}

\proof
First we decompose $H$ into $H'$ and $H''$ such that $H'[A_i,B_{i'}]$ is a $11 \eps K m$-regular graph for all $i,i' \le K$ and $H'' := H - H'$.
Hence $H''[A_i,B_{i'}]$ is a $248 \eps m/K$-regular graph for all $i,i' \le K$.

Write $\mathcal{J} := \{ J_1, \dots, J_{|\mathcal{J}|} \}$.
For each $s \le |\mathcal{J}|$, we will extend $J^*_{s,{\rm dir}}:=(J_s)^*_{\rm dir}$ into a path sequence $PS_s$ satisfying the following conditions:
\begin{itemize}
	\item[\rm ($\alpha_s$)] Suppose that $J_s \in \mathcal{J}_{\I}$.
Then $PS_s$ is an $A_{i_1}$-extension of $J^*_{s,{\rm dir}}$ consisting of precisely $e(J^*_{s})$ vertex-disjoint directed paths of length two.
	\item[\rm ($\beta_s$)] $V(PS_s) = V(J^*_{s,{\rm dir}}) \cup A'_s$, where $A'_s \subseteq A_{i_1} \setminus V(J^*_{s,{\rm dir}})$ and $|A'_s| = e(J^*_{s})$. 
	\item[\rm ($\gamma_s$)] $PS_s - J^*_{s,{\rm dir}}$ is a matching of size $e(J^*_{s})$ from $B'_s$ to $A'_s$, where $B'_s := V(J^*_{s,{\rm dir}}) \cap (B_{i_3} \cup B_{i_4})$.
	\item[\rm ($\delta_s$)] Let $M_s$ be the set of undirected edges obtained from $PS_s - J^*_{s,{\rm dir}}$ after ignoring all the orientations.
	Then $M_1, \dots, M_s$ are edge-disjoint matchings in $H'$.	
	\item[\rm ($\eps_s$)] $PS_s$ consists only of	edges from $A_{i_1} \cup A_{i_2}$ to $B_{i_3} \cup B_{i_4}$,  and 
from $B_{i_3} \cup B_{i_4}$ to $A_{i_1}$.
\end{itemize}
Note that ($\beta_s$) and ($\gamma_s$) together imply ($\eps_s$).
Suppose that for some $s$ with $1\le s \le |\mathcal{J}|$ we have already constructed $PS_1, \dots, PS_{s-1}$. We will now construct $PS_s$ as follows.
Let $\I$ be such that $J_s \in \mathcal{J}_{\I}$ and let  $H'_s:=H'-(M_1+ \dots +M_{s-1})$.
(BES4) implies that%
\COMMENT{Allan: added stackrel \eqref{BESeq}}
\begin{align}
e( J^*_{s, {\rm dir}} ) = e( J^*_s ) \stackrel{\eqref{BESeq}}{\le} e(J_s) \le \eps n \le 3 \eps K m 
\quad
\text{and}
\quad
|V(J^*_{s,{\rm dir}}) \cap A_{i_1}| \le 3\eps K m .
\label{J^*}
\end{align}
Consider any $s' < s$. Recall from the definition of $J^*_{s',{\rm dir}}$ that $V(J^*_{s',{\rm dir}})$ is the set of all those
vertices in $A\cup B$ which are covered by edges of $J_{s'}$. Together with ($\beta_{s'}$) and ($\gamma_{s'}$) this implies that
a vertex $v \in B$ is covered by~$M_{s'}$ if and only if $v$ is incident to an edge of $J_{s'}$.
Together with (ii) this in turn implies that for all $v \in B$ we have
\begin{align*}
d_{H'_s}( v , A_{i_1} ) & \ge  d_{H'}( v , A_{i_1} ) - \sum_{s' < s} d_{ M_{s'}}(v) 
\ge  11 K \eps m  - 2\eps n \\
& \ge 11 K \eps m - 5 K \eps m
 \stackrel{\eqref{J^*}}{\ge}  |V(J^*_{s,{\rm dir}}) \cap A_{i_1}|  + e(J^*_{s}).
\end{align*}
Note that $e ( J^*_s )=|V(J^*_{s,{\rm dir}}) \cap (B_{i_3}\cup B_{i_4})|=|B_s'|$. So we can
greedily find a matching $M_s$ of size $e ( J^*_s )$ in $H'_s[A_{i_1}\setminus V(J^*_{s,{\rm dir}}), B_s']$
(which therefore covers all vertices in $B_s'$).
Orient all edges of $M_s$ from $B_s'$ to $A_{i_1}$ and call the resulting directed matching $M_{s, {\rm dir}}$.
Set 
$$
PS_s:= J^*_{s,{\rm dir}} + M_{s, {\rm dir}}.
$$
Note that $PS_s$ consists of precisely $e(J^*_{s})$ directed paths of length two whose final vertices lie in $A_{i_1}$, so ($\alpha_s$)--($\eps_s$) hold by our construction.
This shows that we can obtain path sequences $PS_1, \dots, PS_{|\mathcal{J}|}$ satisfying ($\alpha_s$)--($\eps_s$)
for all $s\le |\mathcal{J}|$.%
   \COMMENT{Daniela: added for all $s\le |\mathcal{J}|$}

The following claim provides us with a `reservoir' of edges which we will use to balance out the edges of each $PS_s$ and thus extend
each $PS_s$ into a path sequence $PS'_s$ which is locally balanced with respect to~$C$.

\medskip

\noindent
{\bf Claim.} 
\emph{$H''$ contains $|\mathcal{J} |$ subgraphs $H''_1, \dots, H''_{|\mathcal{J} |}$ satisfying the following properties
for all $s\le |\mathcal{J} |$ and all $i,i'\le K$:
\begin{itemize}
	\item[\rm (a$_1$)] If $PS_s$ contains an $A_i B_{i'} $-edge, then $H''_s$ contains a matching between $A_{i'}$ and $B_{i}$ of size $30 \eps K m$.
	\item[\rm (a$_2$)] If $PS_s$ contains a $B_i A_{i'} $-edge, then $H''_s$ contains a matching between $A_{i+1}$ and $B_{i'-1}$ of size $30 \eps K  m$.
	\item[\rm (a$_3$)] $H''_1, \dots, H''_{|\mathcal{J} |}$ are edge-disjoint
	and for all $s\le |\mathcal{J} |$ the matchings guaranteed by {\rm (a$_1$)} and {\rm (a$_2$)} are edge-disjoint.
\end{itemize}
}
\noindent
So if $PS_s$ contains both an $A_i B_{i'} $-edge and a $B_{i'-1} A_{i+1} $-edge, then $H''_s$ contains a matching between $A_{i'}$ and $B_{i}$ of size  $60 \eps K m$.
\smallskip

\noindent
To prove the claim, first recall that $H''[A_i,B_{i'}]$ is a $248 \eps m/K$-regular graph for all $i,i' \le K$.
So $H''[A_{i},B_{i'}]$ can be decomposed into $248 \eps m/K$ perfect matchings. 
Each perfect matching can be split into $1/(31\eps K) $ matchings, each of size at least $30 \eps K  m$.
Therefore $H''[A_{i},B_{i'}]$ contains $8m/K^2$ edge-disjoint matchings, each of size at least $30 \eps K  m$.
(i) and ($\eps_s$) together imply that for any $i,i' \le K$, the number of $PS_s$ containing an $A_iB_{i'}$-edge is at most 
$$ \sum_{(\I)\ : \ i \in \{i_1,i_2\}, \ i' \in \{i_3,i_4\}} | \mathcal{J}_{\I} | \le 4 m/K^2.$$
Recall that $H''[A_{i'},B_{i}]$ contains $8m/K^2$ edge-disjoint matchings, each of size at least $30 \eps m$.
Thus we can assign a distinct matching in $H''[A_{i'},B_{i}]$ of size $30 \eps m$ to each $PS_s$ that contains an $A_{i}B_{i'}$-edge. 
Additionally, we can also assign a distinct matching in $H''[A_{i+1},B_{i'-1}]$ of size $30 \eps m$ to each $PS_s$ that contains a $B_{i}A_{i'}$-edge.%
    \COMMENT{Here we use that the number of $PS_s$ containing an $B_{i}A_{i'}$-edge is at most 
$\sum_{(i_2,i_3,i_4)\ : \ i \in \{i_3,i_4\}} | \mathcal{J}_{i',i_2,i_3,i_4} | \le 2 m/K^2.$}
For all $s \le |\mathcal{J}|$, let%
\COMMENT{Andy: set becomes let}
$H_s''$ be the union of all those matchings assigned to $PS_s$.%
    \COMMENT{Deryk: reworded}
Then $H_1'', \dots, H''_{|\mathcal{J} |}$ are as desired in the claim.
\medskip

For each $s\le |\mathcal{J} |$, we will now add suitable edges from $H''_s$ to $PS_s$ in order to obtain a path sequence $PS_s'$ which
is locally balanced with respect to $C$. So fix $s \le |\mathcal{J}|$ and let $e_1, \dots, e_{\ell}$ denote the edges of $PS_s$.
Note that $\ell = 2e(J^*_s) \le 6 K \eps m $ by~($\gamma_s$) and~(\ref{J^*}).%
    \COMMENT{Deryk: had $\ell = 2e(J^*_s) \le 2 \eps n \le 6 K \eps m $}
For each $r \le \ell$, we will find a directed edge $f_r$ satisfying the following conditions: 
\begin{itemize}
	\item[\rm (b$_1$)] If $e_r$ is an $A_i B_{i'} $-edge, then $f_r$ is an $A_{i'}B_{i}$-edge.
	\item[\rm (b$_2$)] If $e_r$ is a $B_i A_{i'} $-edge, then $f_r$ is a $B_{i'-1}A_{i+1}$-edge.
	\item[\rm (b$_3$)] The undirected version of $\{f_1, \dots, f_{\ell}\}$ is a matching in $H''_s$ and vertex-disjoint from $V(PS_s)$.
\end{itemize}
Suppose that for some $r \le \ell$ we have already constructed $f_1, \dots, f_{r-1}$.
Suppose that $e_r$ is an $A_i B_{i'} $-edge. (The argument for the other case is similar.) 
By~(a$_1$), $H''_s[A_{i'},B_{i}]$%
    \COMMENT{Daniela: had $H''_s[A_i,B_{i'}]$}
contains a matching of size $30 K \eps m $.
Note by ($\alpha_s$) and (b$_3$) that 
$$ |V(PS_s \cup \{f_1, \dots, f_{r-1} \})| \le 3 e(J^*_s) + 2(r-1) < 5 \ell \le 30 K \eps m. $$
Hence there exists an edge in $H''_s[A_{i'},B_{i}]$ that is vertex-disjoint from $PS_s \cup \{f_1, \dots,$ $ f_{r-1} \} $.
Orient one such edge from $A_{i'}$ to $B_{i}$ and call it $f_r$.
In this way, we can construct $f_1, \dots, f_{\ell}$ satisfying (b$_1$)--(b$_3$).

Let $PS'_s$ be digraph obtained from $PS_s$ by adding all the edges $f_1, \dots, f_{\ell}$.
Note that $PS'_s$ is a locally balanced path sequence with respect to~$C$.
(Indeed, $PS'_s$ is locally balanced since $\{e_r,f_r\}$ is locally balanced for each $r \le \ell$.)
Let $\I$ be such that $J\in\mathcal{J}_{\I}$. Then the following properties hold:
\begin{itemize}
	\item[(c$_1$)] $PS'_s$ is an $A_{i_1}$-extension of $J^*_{s,{\rm dir}}$.
	\item[(c$_2$)] $|V(PS_s') \cap A_i| , |V(PS_s') \cap B_i| \le 12 \eps K m$ for all $i\le K$.	
	\item[(c$_3$)] If $V(PS_s') \cap A_i \ne \emptyset$, then $i \in \{ \I, i_3+1 ,i_4+1 \}$.
	\item[(c$_4$)] If $V(PS_s') \cap B_i \ne \emptyset$, then $i \in \{ i_1-1, \I\}$.
\end{itemize}
Indeed, (c$_1$) is implied by ($\alpha_s$) and the definition of $PS'_s$.
Since $e(PS'_s) = 2 e(PS_s) = 4 e(J^*_s)$, (c$_2$) holds by~\eqref{J^*}.
Finally, (c$_3$) and (c$_4$) are implied by ($\eps_s$), (b$_1$) and (b$_2$) as $J_s \in \mathcal{J}_{\I}$.
 
Note that $PS'_1-J^*_{1,{\rm dir}}, \dots, PS'_{|\mathcal{J}|}-J^*_{|\mathcal{J}|,{\rm dir}}$ are%
     \COMMENT{Daniela: had $PS'_1, \dots, PS'_{|\mathcal{J}|}$}
pairwise edge-disjoint and let $\mathcal{BE}:=\{PS'_1, \dots, PS'_{|\mathcal{J}|}\}$.
We claim that $\mathcal{BE}$ is a balanced extension of $\mathcal{J}^*_{\rm dir}$ with respect to $(\mathcal{Q},C)$ and
parameters $( 12 \eps K, 12)$.
To see this, recall that $\mathcal{Q}=\{A_1$,$\dots$,$A_K$, $B_1$,$\dots$,$B_K\}$ is a $(2K,m)$-equipartition
of $V':=V\setminus (A_0\cup B_0)$. Clearly, (BE1) holds with $V'$ playing the role of~$V$.
(c$_3$) and (i) imply that for every $i \le K$ there are at most $6m/K$ $PS_s' \in \mathcal{BE}$ such that $V(PS_s') \cap A_i \ne \emptyset$.
A similar statement also holds for each~$B_i$.
So together with (c$_2$), this implies (BE3), where $2K$ plays the role of $k$ in (BE3).%
    \COMMENT{Need parameters $( 12 \eps K, 12)$ since we have $2K$ clusters and so we need that $6m/K=12m/2K$}
As remarked after the definition of a balanced extension, this implies the `moreover part' of (BE2). So (BE2) holds too.%
    \COMMENT{Daniela: previously had "Given $i \le K$, the number of $J$ such that $J \in \mathcal{J}_{i,i_2,i_3,i_4}$ for some $i_2,i_3,i_4$ is at most $m/K$ by~(i).
So (BE2) follows from (c$_1$), where $2K$ plays the role of $k$ in (BE2)."}
Therefore $\mathcal{BE}$ is a balanced extension, so the lemma follows (by orienting the remaining edges of $H$ arbitrarily).
\endproof

\removelastskip\penalty55\medskip\noindent{\bf Proof of Lemma~\ref{almostthmbip}. }
Apply Lemma~\ref{sysdecombip} to obtain tuples $(G_{j}, \mathcal{Q}, C_{j}, H_{j}, \mathcal{J}_{j})$ for all $j \le K/2$
satisfying (a$_1$)--(a$_6$). Fix $j \le K/2$ and write $\mathcal{J}_{j} :=\{J_{j,1}\dots,J_{j,|\mathcal{J}_{j}|}\}$.
Next, apply Lemma~\ref{balanceextensionbip} with $\mathcal{J}_{j},C_{j}, H_{j},  \eps_{0}$ playing the roles
of $\mathcal{J}, C, H, \eps$ to obtain an orientation $H_{j,{\rm dir}}$ of $H_j$ and a balanced extension $\mathcal{BE}_j$ of $\mathcal{J}_{j}$ with respect to
$(\mathcal{Q}, C_{j})$ and parameters $(12 \eps_0 K,12)$.
(Note that (a$_3$) and (a$_5$) imply conditions (i) and (iii) of Lemma~\ref{balanceextensionbip}.
Condition~(ii) follows from Lemma~\ref{almostthmbip}(d).)
So we can write $\mathcal{BE}_j :=\{PS_{j,1}\dots,PS_{j,|\mathcal{J}_{j}|}\}$ such that $(J_{j,s})^*_{\rm dir} \subseteq PS_{j,s}$ for all $s \le |\mathcal{J}_{j}|$.
Each path sequence in $\mathcal{BE}_j$ is obtained from some $(J_{j,s})^*_{\rm dir}$ by adding edges
of $H_{j,{\rm dir}}$.
Since $(G_{j, {\rm dir}}, \mathcal{Q}, C_{j})$ is a $(2K,m, 4 \mu, 5/K)$-cyclic system by Lemma~\ref{sysdecombip}(a$_6$), we can apply Lemma~\ref{merging} 
as follows:
\smallskip
\begin{center}
  \begin{tabular}{ r | c | c | c | c | c | c | c | c | c | c |c}
& $G_{j, {\rm dir}}$ & $\mathcal{Q}$ &  $C_{j}$  &  $2K$ & $\mathcal{J}^*_{j,{\rm dir}}$ & $|\mathcal{J}_j|$ & $4 \mu$ & $3\rho$ & $12\eps_{0}K$ & $5/K$ & $12$ \\ \hline
plays role of & $G$ & $\mathcal{Q}$ &  $C$ & $k$ & $\mathcal{M}$ & $q$ & $\mu$ & $\rho$ &  $\eps_0$ & $\eps'$ & $\ell$ \\
 \end{tabular}
\end{center}
\smallskip
\noindent
This gives us $|\mathcal{J}_{j}|$ directed Hamilton cycles $C'_{j,1}, \dots, C'_{j,|\mathcal{J}_{j}|}$ in $G_{j, {\rm dir}} + \mathcal{BE}_j$
such that each $C'_{j,s}$ contains $PS_{j,s}$ and is consistent with $(J_{j,s})^*_{\rm dir}$.
Moreover, (a$_4$) implies that%
     \COMMENT{Daniela: added (a$_4$) implies that}
$C'_{j,1} - (J_{j,1})^*_{\rm dir}, \dots, C'_{j,|\mathcal{J}_{j}|} - (J_{j,|\mathcal{J}_{j}|})^*_{\rm dir}$ are edge-disjoint subgraphs of $G_{j, {\rm dir}}+ H_{j, {\rm dir}}$.%
	\COMMENT{Allan: more details added.}
Repeat this process for all $j \le K/2$.

Recall from Lemma~\ref{sysdecombip}(a$_2$) that $\mathcal{J}_{1}, \dots, \mathcal{J}_{K/2}$ is a partition of $\mathcal{J}$.
Thus%
     \COMMENT{Daniela: had "Together with (a$_4$) this implies that" instead of "Thus" and "edge-disjoint directed Hamilton cycles" instead of
"directed Hamilton cycles" below}
we have obtained $|\mathcal{J}|$ directed Hamilton cycles 
$C'_{1}, \dots, C'_{|\mathcal{J}|}$ on $A\cup B$ such that each $C'_{s}$ is consistent with $(J_s)^*_{\rm dir}$ for some $J_s\in \mathcal{J}$
(and $J_s\neq J_{s'}$ whenever $s\neq s'$). 
Let $H_s$ be the undirected graph obtained from $C'_{s}-J_s^* +J_s$ by ignoring all the orientations of the edges.
Since $J_1, \dots, J_{|\mathcal{J}|}$ are edge-disjoint exceptional systems, $H_1, \dots, H_{|\mathcal{J}|}$ are edge-disjoint spanning subgraphs of $G$.
Finally, Proposition~\ref{CES-H2} implies that $H_1, \dots, H_{|\mathcal{J}|}$ are indeed as desired in Lemma~\ref{almostthmbip}.
\endproof

\section*{Acknowledgement}

We are grateful to Katherine Staden for drawing Figure~\ref{fig:eps}.

\backmatter


\input{biblio.tex}

\printindex
\end{document}

%% file: biblio.tex

\bibliographystyle{amsalpha}
